\documentclass[english]{smfbook}
\usepackage{sabbah_lisbonne0901}

\begin{document}
\frontmatter
\title
{Introduction to Stokes~structures
\\
\vskip.5\baselineskip\smaller \smaller Lecture Notes (Lisboa, January 2009)
}

\author
{Claude Sabbah}
\address{UMR 7640 du CNRS\\
Centre de Math\'ematiques Laurent Schwartz\\
\'Ecole polytechnique\\
F--91128 Palaiseau cedex\\
France}
\email{sabbah@math.polytechnique.fr}
\urladdr{http://www.math.polytechnique.fr/~sabbah}
\date{\today}
\thanks{This research was supported by the grant ANR-08-BLAN-0317-01 of the Agence nationale de la recherche.}


\begin{abstract}
The purpose of these lectures is to introduce the notion of a Stokes-perverse sheaf as a receptacle for the Riemann-Hilbert correspondence for holonomic $\cD$-modules. They develop the original idea of P\ptbl Deligne in dimension one, and make it enter the frame of perverse sheaves. They also give a first step for a general definition in higher dimension, and make explicit particular cases of the Riemann-Hilbert correspondence, relying on recent results of T\ptbl Mochizuki.
\end{abstract}
\maketitle

\chapterspace{-2}
\tableofcontents
\mainmatter

\chapterspace{-2}
\chapter*{Introduction}

The classical theory of linear differential equation of one complex variable near a singular point distinguishes between a regular and an irregular singularity by checking the vanishing of the irregularity number, which characterizes a regular singularity (Fuchs criterion). The behaviour of the solutions of the equation (moderate growth near the singularity) also characterizes a regular singularity, and this leads to the local Riemann-Hilbert correspondence, characterizing a regular singularity by ``monodromy data''.

On a Riemann surface~$X$, the Riemann-Hilbert correspondence for meromorphic connections with regular singularities on a discrete set~$D$ (first case), or more generally for regular holonomic $\cD$-modules with singularities at~$D$ (second case), induces an equivalence of the corresponding category with the category of ``monodromy data'', which can be presented

$\bbullet$ either quiver-theoretically as the data of local monodromies and connection matrices (first case), together with the so-called canonical and variation morphisms (second case),

$\bbullet$ or sheaf-theoretically as the category of locally constant sheaves of finite dimensional $\CC$-vector spaces on $X^*=X\moins D$ (first case) or perverse sheaves with singularities at~$D$ (second case).

While the first presentation is suited to describing moduli spaces, for instance, the second one is suited to sheaf theoretic operations on such objects. Each of these objects can be defined over subfields $\kk$ of $\CC$, giving rise to a $\kk$-structure on the meromorphic connection with regular singularities, or regular holonomic $\cD$-module.

When the irregularity number is nonzero, finer numerical invariants are introduced, encoded in the Newton polygon of the equation at the singular point. Moreover, such a Riemann-Hilbert correspondence with both aspects also exists. The first one is the most popular, with Stokes data, consisting of Stokes matrices, instead of local monodromy data. An extensive literature exists on this subject, for which classical references are \cite{Wasow65,Sibuya90} and a more recent one is \cite{S-vdP01}. The second aspect, initiated by P\ptbl Deligne \cite{Deligne78b} (case of meromorphic connections) and \cite{Deligne84cc} (holonomic $\cD$-modules), has also been developed by B\ptbl Malgrange \cite{Malgrange83b,Malgrange91} and D\ptbl Babbitt \& V.S\ptbl Varadarajan~\cite{B-V89}. Moreover, the Poincaré duality has been expressed by integrals on ``rapid decay cycles'' by various authors \cite{K-M-M99,B-E04b}.

In higher dimensions, such a dichotomy (regular/irregular singularity) also exists for meromorphic bundles with flat connection (\resp holonomic $\cD$-modules). The work of P\ptbl Deligne \cite{Deligne70} has provided a notion of meromorphic connection with regular singularities and a Riemann-Hilbert correspondence has been obtained by P\ptbl Deligne in such a case, and by M\ptbl Kashiwara on the one hand, and Z\ptbl Mebkhout on the other hand in the case of holonomic $\cD$-modules with regular singularities. The target category for this correspondence is that of $\CC$-perverse sheaves. Moreover, the Fuchs criterion has been generalized by Z\ptbl Mebkhout: the irregularity number is now replaced by the irregularity complex, which is also a perverse sheaf.

When the irregularity perverse sheaf is not zero, it can be refined, giving rise to Newton polygons on strata of a stratification adapted to the characteristic variety of the holonomic $\cD$-module (\cf \cite{L-M99}).

These lectures will be mainly concerned with the second aspect of the the Riemann-Hilbert correspondence for meromorphic connections or holonomic $\cD$-modules, and the main keyword will be the Stokes phenomenon in higher dimension. Their purpose is to develop the original idea of P\ptbl Deligne and B\ptbl Malgrange, and make it enter the frame of perverse sheaves, so that it can be extended to arbitrary dimensions. This has been motivated by recent beautiful results of T\ptbl Mochizuki \cite{Mochizuki08,Mochizuki07b}, who has rediscovered it and shown the powerfulness of this point of view in higher dimension.

This approach is intended to provide a global understanding of the Stokes phenomenon. While in dimension one the polar divisor of a meromorphic connection consists of isolated points and the Stokes phenomenon describes the behaviour of solutions in various sectorial domains around these points, in dimension $\geq2$ the divisor is no more discrete and the sectorial domains extend in some way all along the divisor. Moreover, questions like pull-back and push-forward by holomorphic maps lead to single out the sheaf-theoretic approach to the Stokes phenomenon. Above the usual complex geometry of the underlying complex manifold with its divisor lives a ``wild complex geometry'' governing the Stokes phenomenon.

One of the sought applications of this sheaf-theoretic approach, named Stokes-perverse sheaf, is to answer a question that S\ptbl Bloch asked me some years ago: to define a sheaf-theoretical Fourier transform over $\QQ$ (say) taking into account the Stokes data. Note that the unpublished manuscript \cite{B-B-D-E05} gave an answer to this question (\cf also the recent work \cite{Mochizuki10} of T\ptbl Mochizuki). The need of such an extension to dimension bigger than one also shows up in \cite[p\ptbl 116]{Deligne8406b}\footnote{\label{on-aimerait}Deligne writes: ``On aimerait dire (mais ceci nous obligerait à quitter la dimension~$1$)...''.}.

One of the main problems in the ``perverse'' approach is to understand on which spaces the sheaves are to be defined. In dimension one, Deligne replaces first a Riemann surface by its real oriented blow-up space at the singularities of the meromorphic connection, getting a surface with boundary, and endows the extended local system of horizontal sections of the connection with a ``Stokes filtration'' on the boundary. This is a filtration indexed by an ordered local system. We propose to regard such objects as sheaves on the étale space of the ordered local system (using the notion of étale space as in \cite{Godement64}). In order to obtain a perfect correspondence with holonomic $\cD$-modules, Deligne fills the boundary with discs together with perverse sheaves on them, corresponding to the formal part of the meromorphic connection. The gluing at the boundary between the Stokes-filtered local system and the perverse sheaf is defined through grading the Stokes filtration.

The road is therefore a priori well-paved and the program can be clearly drafted:
\begin{enumerate}
\item
To define the notion of Stokes-constructible sheaf on a manifold, and a $t$\nobreakdash-structure in its derived category, in order to recover the category of Stokes-perverse sheaves on a complex manifold as the heart of this $t$-structure.
\item
To exhibit a Riemann-Hilbert correspondence $\RH$ between holonomic $\cD$\nobreakdash-modules and Stokes-perverse sheaves, and to prove that it is an equivalence of categories.
\item
To define the direct image functor in the derived category of Stokes-constructible sheaves and prove the compatibility of $\RH$ when taking direct images of holonomic $\cD$-modules.
\end{enumerate}

An answer to the latter question would give a way to compute Stokes data of the asymptotic behaviour of integrals of multivalued functions which satisfy themselves a holonomic system of differential equations.

While we realize the first two points of the program in dimension one, by making a little more explicit the contents of \cite{Deligne78b,Deligne84cc}, we do not go to the end in dimension bigger than one, as we only treat the Stokes-perverse counterpart of meromorphic connections, not holonomic $\cD$-modules. The reason is that some new phenomena appear, which were invisible in dimension one.

In order to make them visible, let us consider a complex manifold $X$ endowed with a divisor~$D$. In dimension one, the topological space to be considered is the oriented real blow-up space $\wt X$ of $X$ along~$D$, and meromorphic connections on $X$ with poles on~$D$ are in one-to-one correspondence with local systems on $X\moins D$ whose extension to $\wt X$ is equipped with a Stokes filtration along $\partial\wt X$. If $\dim X\geq2$, in order to remain in the realm of local systems, a simplification of the underlying geometric situation seems unavoidable in general, so that we treat the case of a divisor with normal crossings (with all components smooth), and a generic assumption has also to be made on the connection, called \emph{goodness}. Variants of this genericity condition have occurred in asymptotic analysis (\eg in \cite{Majima84}) or when considering the extension of the Levelt-Turrittin theorem to many variables (\eg in \cite{Malgrange92}). We define the notion of good stratified $\ccI$-covering of $\partial\wt D$. To any good meromorphic connection and to any good Stokes-filtered local system are associated in a natural way such a good stratified $\ccI$-covering, and the categories to be put into Riemann-Hilbert correspondence are those subcategories of objects having an associated stratified $\ccI$-covering contained in a fixed good one.

This approach remains non intrinsic, that is, while the category of meromorphic connections with poles along an arbitrary divisor is well-defined, we are able to define a Stokes-topological counterpart only with the goodness property. This is an obstacle to define intrinsically a category of Stokes-perverse sheaves. This should be overcome together with the fact that such a category should be stable by direct images, as defined in \Chaptersname\ref{chap:Ifil} for pre-$\ccI$-filtrations. Nevertheless, this makes it difficult to use this sheaf-theoretic Stokes theory to obtain certain properties of a meromorphic connection when the polar divisor has arbitrary singularities, or when it has normal crossings but the goodness assumption is not fulfilled. For instance, while the perversity of the irregularity sheaf (a result due to Z\ptbl Mebkhout) is easy in the good case along a divisor with normal crossings, we do not have an analogous proof without these assumptions.

Therefore, our approach still remains non-complete with respect to the program above, but already gives strong evidence of its feasibility.

Compared to the approach of T\ptbl Mochizuki in \cite{Mochizuki07b, Mochizuki08}, which is nicely surveyed in \cite{Mochizuki09}, we regard a Stokes-filtered object as an abstract ``topological'' object, while Mochizuki introduces the Stokes filtration as a filtration of a flat vector bundle. In the recent preprint \cite{Mochizuki10}, T\ptbl Mochizuki has developed the notion of a Betti structure on a holonomic $\cD$-module and proved many functorial properties. Viewing the Betti structure as living inside a pre-existing object (a holonomic $\cD$-module) makes it a little easier to analyze its functorial properties, since the functorial properties of holonomic $\cD$-modules are already understood. On the other hand, this gives a strong evidence of the existence of a category of Stokes-perverse sheaves with good functorial properties.

\Subsubsection*{Contents \chaptersname by \chaptername}

In \Chaptersname \ref{chap:Ifil}, we develop the notion of Stokes filtration in a general framework under the names of pre-$\ccI$-filtration and $\ccI$-filtration, with respect to an ordered sheaf of abelian group $\ccI$. The sheaf $\ccI$ for the Stokes filtration in dimension one consists of polar parts of multivalued meromorphic functions of one variable, as originally introduced by Deligne. Its étale space is Hausdorff, which makes the understanding of a filtration simpler with respect to taking the associated graded sheaf. This \chaptersname may be skipped in a first reading, or may serve as a reference for various notions considered starting from \Chaptersname \ref{chap:Stokesone-pervers}.

Part one, starting at \Chaptersname \ref{chap:Stokesone}, is mainly concerned with dimension one, although \Chaptersname \ref{chap:Laplace} anticipates some results in dimension two, according to the footnote on Page \pageref{on-aimerait}.

In \Chaptersname \ref{chap:Stokesone}, we essentially redo more concretely the same work as in \Chaptersname \ref{chap:Ifil}, in the context of Stokes-filtered local systems on a circle. We prove abelianity of the category in \Chaptersname \ref{chap:abelian}, a fact which follows from the Riemann-Hilbert correspondence, but is proved here directly over any base field $\kk$. In doing so, we introduce the level structure, which was a basic tool in the higher dimensional analogue developed by T\ptbl Mochizuki \cite{Mochizuki08}, and which was previously considered together with the notion of multisummability \cite{B-B-R-S91,M-R92,LR-P97,S-vdP01}.

In \Chaptername s \ref{chap:Stokesone-pervers} and \ref{chap:RH} we develop the notion of a Stokes-perverse sheaf, mainly by following P\ptbl Deligne \cite{Deligne78b,Deligne84cc} and B\ptbl Malgrange \cite[Chap\ptbl IV.3]{Malgrange91}, and prove the Riemann-Hilbert correspondence. We make explicit the behaviour with respect to duality, at least at the level of Stokes-filtered local systems, and the main tools are explained in \Chaptersname \ref{chap:Stokesone-pervers}.

\Chaptersname \ref{chap:holdist} gives two analytic applications of the Riemann-Hilbert correspondence in dimension one. Firstly, the Hermitian dual of a holonomic $\cD$-module (\ie the conjugate module of the module of distribution solutions of the original one) on a Riemann surface is shown to be holonomic. Secondly, the local structure of distributions solutions of a holonomic system is analyzed.

\Chaptersname \ref{chap:Laplace} presents another application, with a hint of the theory in dimension two, by computing the Stokes filtration of the Laplace transform of a regular holonomic $\cD$-module on the affine line. We introduce the topological Laplace transformation, and we make precise the relation with duality, of $\cD$-modules on the one hand, Poincaré-Verdier on the other hand, and their relations. For this \chaptername, we use tools in dimension $2$ which are fully developed in the next \chaptername.

In Part two we start analyzing the Stokes filtration in dimension $\geq2$. \Chaptersname\ref{chap:realbl} defines the real blow-up space along a family of divisors and the relations between various real blow-up spaces. We pay attention to the global existence of these spaces. The basic sheaf on such real blow-up spaces is the sheaf of holomorphic functions with moderate growth along the divisor. It leads to the moderate de~Rham complex of a meromorphic connection. We give some examples of such de~Rham complexes, showing how non-goodness can produce higher dimensional cohomology sheaves.

\Chaptersname\ref{chap:StokesfilteredNCD} takes up \Chaptername s \ref{chap:Stokesone} and \ref{chap:abelian} and introduces the goodness assumption. The construction of the sheaf $\ccI$ is given with some care, to make it global along the divisor.

The first approach to the Riemann-Hilbert correspondence in dimension $\geq2$ is given in \Chaptersname\ref{chap:RHgoodsmooth}, along a smooth divisor. It can be regarded as obtained by putting a (good) parameter in \Chaptersname\ref{chap:RH}. The main new argument is the local constancy of the \emph{Stokes sheaf} (Stokes matrices can be chosen locally constant with respect to the parameter).

\Chaptersname \ref{chap:goodformal} analyzes the formal properties of good meromorphic connections, following T\ptbl Mochizuki \cite{Mochizuki08}. In \Chaptersname\ref{chap:RHgoodnc} we give a proof of the analogue in higher dimension of the Hukuhara-Turrittin theorem, which asymptotically lifts a formal decomposition of the connection. We mainly follow T\ptbl Mochizuki's proof, for which a short account has already been given by M\ptbl Hien in \cite[Appendix]{Hien09}. We then consider the general case of the Riemann-Hilbert correspondence for good meromorphic connections, and we take this opportunity to answer a question of Kashiwara by proving that the Hermitian dual of a holonomic $\cD$-module is holonomic (\cf \Chaptersname\ref{chap:holdist} in dimension one).\enlargethispage{\baselineskip}%

In \Chaptersname\ref{chap:pipes}, we address the question of push-forward and we make explicit a calculation of the Stokes filtration of an exponentially twisted Gauss-Manin system (such a system has already been analyzed by C\ptbl Roucairol \cite{Roucairol06a,Roucairol06b,Roucairol07}). However, the method is dependent on the simple geometric situation, so can hardly be extended directly to the general case, a proof of which has been recently given by T\ptbl Mochizuki.

Lastly, \Chaptername s~\ref{chap:irregnearby} and \ref{chap:nearby} are concerned with the nearby cycle functor. In \Chaptersname\ref{chap:irregnearby} we first recall the definition of the moderate nearby cycle functor for holonomic $\cD$-modules via the Kashiwara-Malgrange $V$-filtration, and we review the definition of the irregular nearby cycle functor, due to Deligne. We give a new proof of the preservation of holonomy in a local analytic setting (the proof of Deligne \cite{Deligne83b} concerns the algebraic setting) when the ambient manifold has dimension two.

In \Chaptersname\ref{chap:nearby}, we give a definition of the nearby cycle functor relative to a holomorphic function for a Stokes-filtered local system, and compare it with the notion of moderate nearby cycles of a holonomic $\cD$-module of \Chaptersname\ref{chap:irregnearby} through the Riemann-Hilbert correspondence. We restrict our study to the case of a meromorphic connection with poles along a divisor with normal crossings and a holomorphic function whose zero set is equal to this divisor.

\subsubsection*{Acknowledgements}
I thank A\ptbl Beilinson and the university of Chicago where part of this work was achieved, and T\ptbl Monteiro Fernandes and O\ptbl Neto in Lisbon (University of Lisbon, CAUL and CMAF) for giving me the opportunity to lecture on it in January 2009, as well as the audience of these lectures (whose content corresponds approximately to the present first seven lectures) for many interesting questions and remarks. Many discussions with S\ptbl Bloch,
H\ptbl Esnault, C\ptbl Hertling, M\ptbl Hien, T\ptbl Mochizuki and G\ptbl Morando have been very stimulating and helpful. I thank especially T\ptbl Mochizuki for letting me know his ongoing work on the subject, which strengthens the approach given here. Discussions with him have always been very enlightening. Two referees have provided many interesting suggestions for improving the manuscript and to correct some errors. Needless to say, this work owes much to P\ptbl Deligne and B\ptbl Malgrange, and to D\ptbl Bertrand who asked me to help him when editing the volume \cite{D-M-R07}, giving me the opportunity of being more familiar with its contents.

\chapter{$\ccI$-filtrations}\label{chap:Ifil}

\begin{sommaire}
This \chaptersname introduces the general framework for the study of the Stokes phenomenon in a sheaf-theoretic way. The underlying topological spaces are étale spaces of sheaves of ordered abelian groups~$\ccI$. The general notion of pre-$\ccI$-filtration is introduced as a convenient abelian category to work in. The notion of~$\ccI$-filtration is first considered when the étale space of~$\ccI$ is Hausdorff. We will soon restrict to $\ccI$-filtrations of locally constant sheaves of $\kk$-vector spaces and we will extend the definition to the case where $\ccI$ satisfies the stratified Hausdorff property. Most of the notions introduced in this \chaptersname will be taken up as a more concrete approach to Stokes filtrations in \Chaptersname \ref{chap:Stokesone}, and this \chaptersname may be skipped in a first reading. It contains nevertheless many guiding principles for \Chaptername s \ref{chap:Stokesone}, \ref{chap:abelian} and \ref{chap:StokesfilteredNCD}.
\end{sommaire}

\subsection{Introduction}
In \cite{Deligne83b}, Deligne introduced the notion of a sheaf filtered by a local system of ordered sets in order to express in a sheaf-theoretic way the Stokes phenomenon of a linear differential equation of one complex variable.

This preliminary lecture clarifies the general setup of a sheaf filtered by a sheaf of ordered abelian groups. The approach considered here consists in replacing the sheaf of ordered abelian groups with the associated ordered set, that is, the étale space associated to this sheaf.

Similarly to the case of filtered objects of an abelian category (filtered vector spaces, say) with indices in an ordered set, the category is not abelian in general, and the question of strictness (strictness of morphisms) soon arises in various questions. Working in an ambient abelian category is thus useful. This is the category of pre-$\ccI$-filtrations.

Various operations can be defined for these filtered objects, but do not lead very far because of the lack of strictness, in general. The notion of a $\ccI$-filtration, that is, a pre-$\ccI$-filtration which is locally graded, will be the right one for Stokes structures and will satisfy the strictness property under suitable conditions.

In this general setup, the operation of grading the filtration is not straightforward, and a new problem arises, namely whether the étale space of $\ccI$ is Hausdorff or not. Although this question does not show up with the Stokes filtration of a differential equation of one complex variable, because $\ccIet$ is obviously Hausdorff in such a case, the extension to higher dimensions needs a framework with non-Hausdorff étale spaces. The notion of a stratified Hausdorff étale space will be sufficient for this purpose. It is introduced in \S\ref{subsec:Ifilstrat}. The support of the graded sheaf in this setting is now a subset of the étale space of the sheaf $\ccI$, which is a ``stratified covering'' of its image in the base space. For a differential equation of one complex variable, it consists of the multivalued polar parts which show up in the exponentials occurring in the fundamental solution of the equation. In higher dimensions, it will take into account the global multivaluedness of such polar parts.

\subsection{\'Etale spaces of sheaves}\label{subsec:espaceetales}
In the following,~$Y$ will be a locally compact and locally connected topological space (\eg $Y=S^1$) and~$\ccI$ a sheaf of abelian groups on~$Y$. Let \index{$I$@$\ccI$, $\ccIet$}$\mu:\ccIet\to Y$ be the \index{etale space@étale space}étale space associated to~$\ccI$ (\cf\cite[\S II.1.2, p\ptbl110]{Godement64}). Recall that, for any \hbox{$y\in Y$}, the fibre $\mu^{-1}(y)$ is equal to the germ $\ccI_y$ with its discrete topology and any germ $\varphi_y\in\ccI_y$ has a fundamental system of open neighbourhoods in~$\ccIet$ consisting of the sets $\{\varphi_z\mid z\in U\}$, where~$U$ is any open neighbourhood of~$y$ in~$Y$ on which the germ~$\varphi_y$ extends as a section $\varphi\in\ccI(U)$. Then $\mu$ is a local homeomorphism, in particular it is an open map and $\ccIet$ is locally connected (because~$Y$ is~so).

At some places in this \chaptername, we will assume that $\ccIet$ is \emphb{Hausdorff}. This means that, for any open set $U\subset Y$ and any $\varphi,\psi\in\Gamma(U,\ccI)$, if $\varphi_y\neq\psi_y$ for some $y\in U$, then $\varphi_z\neq\psi_z$ for any $z$ in some neighbourhood of~$y$, that is, the maximal open subset $U(\varphi,\psi)$ on which $\varphi\equiv\psi$ satisfies $\ov{U(\varphi,\psi)}\not\ni y$. This is equivalent to saying that $U(\varphi,\psi)$ is also closed, that is, is empty or a connected component of~$U$. In such a case, $\ccIet$ is then locally compact.

\skpt
\begin{exemples}\label{exem:separe}\ligne
\begin{enumerate}
\item\label{exem:separe1}
If~$\ccI$ is a locally constant sheaf, then $\mu:\ccIet\to Y$ is a covering map and $\ccIet$ is Hausdorff.
\item\label{exem:separe1b}
If $\ccJ$ is a subsheaf of~$\ccI$ and if $\ccIet$ is Hausdorff, then so is $\ccJet$ (as $\ccJet$ is open in~$\ccIet$).
\item\label{exem:separe2}
Let $X$ be a complex manifold and let~$D$ be a reduced divisor in~$X$. Let $\cO_X(*D)$ be the sheaf of meromorphic functions on $X$ with poles along~$D$ at most. Then if~$D$ is locally irreducible, the sheaf $\ccI=\cO_X(*D)/\cO_X$, regarded as a sheaf on~$D$, has a \emphb{Hausdorff étale space}: indeed, if a germ $\varphi_x\in\cO_{X,x}(*D)$ ($x\in D$), defined on some open set $U\subset X$ such that $(D\moins\Sing(D))\cap U$ is connected, is holomorphic on some nonempty open set of $D\cap U$, it is holomorphic on $(D\moins\Sing(D))\cap U$, and thus all over $D\cap U$ by Hartogs.

\item\label{exem:separe3}
If $f:Y'\to Y$ is continuous and $Y'$ is Hausdorff, then $f^{-1}\ccI$ satisfies the Hausdorff property when~$\ccI$ does so. If moreover $f$ is proper, then the converse holds, that is,~$\ccI$ satisfies the Hausdorff property when~$f^{-1}\ccI$ does so.
\item\label{exem:separe4}
Let $i:Y\hto Y'$ be a closed immersion (with $Y'$ Hausdorff). If the sheaf~$\ccI$ on~$Y$ satisfies the Hausdorff property, then in general $i_*\ccI$ on $Y'$ does not.
\end{enumerate}
\end{exemples}

\begin{exemple}[Geometry in étale spaces]\label{exem:geometale}
Let $Y$ be a complex analytic manifold and let $\ccF$ be a locally free $\cO_Y$-module of finite rank. Let $p:F\to Y$ be the associated holomorphic bundle. If $\mu:\ccF^\et\to Y$ denotes the étale space of $\ccF$, then there is a natural commutative diagram
\[
\xymatrix{
\ccF^\et\ar[rr]^-{\ev}\ar[rd]_\mu&&F\ar[ld]^p\\&Y&}
\]
where the evaluation map $\ev$ associates to any germ $s_y$ of holomorphic section of $p:F\to Y$ at $y$ its value $s_y(y)$. The map $\ev$ is continuous.

We will also deal with the following situation. Assume that we are given a closed subset $\Sigma$ of $\ccF^\et$ such that $\mu:\Sigma\to Y$ is a finite covering. Then $\Sigma$ is naturally equipped with the structure of a complex manifold. The image $\Sigma'=\ev(\Sigma)$ is an analytic subset of $F$, locally equal to the finite union of local holomorphic sections of $p:F\to Y$, the map $p:\Sigma'\to Y$ is finite and the map $\ev:\Sigma\to\Sigma'$ is the normalization.
\end{exemple}

\subsection{\'Etale spaces of sheaves of ordered abelian groups}\label{subsec:espaceetalesord}
Assume now that~$\ccI$ is a \emphb{sheaf of ordered abelian groups}, \ie a sheaf with values in the category of ordered abelian groups. Hence, for every open set~$U$ of~$Y$, $\ccI(U)$ is an ordered abelian group, with order denoted by $\leqU$, and the restriction maps $(\ccI(U),\leqU\nobreak)\to(\ccI(V),\leqV)$, for $V\subset U$, are morphisms of ordered sets. For every $y\in Y$, the germ $\ccI_y$ is then ordered: for $\varphi,\psi\in\ccI_y$, we have \index{$AAAORD$@$\leqU$, $\leU$, $\leqy$, $\ley$}$\varphi\leqy\psi$ iff there exists an open neighbourhood $U\ni y$ such that $\varphi,\psi\in\ccI(U)$ and $\varphi\leqU\psi$. For every open set $V\subset U$ we then have $\varphi_{|V}\leqV\psi_{|V}$.

For $\varphi,\psi\in\Gamma(U,\ccI)$, we then have $\varphi\leqU\psi$ iff $\varphi-\psi\leqU0$. Giving an order compatible with addition on a sheaf of abelian groups is equivalent to giving a subsheaf $\ccI_{\leq0}\subset\ccI$ of abelian groups such that $-\ccI_{\leq0}\cap\ccI_{\leq0}=0$. Then, for $\varphi,\psi\in\Gamma(U,\ccI)$, we have $\varphi\leqU\psi$ iff $\varphi-\psi\in\Gamma(U,\ccI_{\leq0})$.

We also set ${\leU}={}({\leqU}\text{ and }{\neq})$. For every $y\in Y$ and for $\varphi,\psi\in\ccI_y$, we have $\varphi\ley\psi$ iff for any sufficiently small open neighbourhood $V\ni y$ such that $\varphi,\psi\in\ccI(V)$, $\varphi_{|V}\leV\psi_{|V}$, that is, $\varphi_{|V}\leqV\psi_{|V}$ and $\varphi_{|V}\neq\psi_{|V}$. Equivalently, ${\ley}={}({\leqy}\text{ and }{\neq})$.

\begin{notation}\label{not:Ypsiphi}
Let~$U$ be an open subset of~$Y$ and let $\varphi,\psi\in\Gamma(U,\ccI)$. We denote by $U_{\psi\leq\varphi}$ the subset of~$U$ defined by $y\in U_{\psi\leq\varphi}$ iff $\psi_y\leqy\varphi_y$. Then $U_{\psi\leq\varphi}$ is open in~$U$ (hence in~$Y$). The boundary of $U_{\psi\leq\varphi}$ is denoted by \index{$STDIR$@$\St(\psi,\varphi)$}$\St(\psi,\varphi)$, with the convention that $\St(\varphi,\varphi)=\emptyset$.

Let $(\ccIet\times_Y\ccIet)_\leq\subset\ccIet\times_Y\ccIet$ (\resp $(\ccIet\times_Y\nobreak\ccIet)_<$) be the subset consisting of those $(\psi,\varphi)\in\ccI_y\times\ccI_y$ such that $\psi\leqy\varphi$ (\resp $\psi\ley\varphi$) ($y\in Y$). These are open subsets (with the Hausdorff assumption in the $<$ case). We denote by $j_\leq,j_<$ the corresponding inclusions. Let us note that the diagonal inclusion $\delta:\ccIet\hto\ccIet\times_Y\ccIet$ is open, and also closed if we assume that $\ccIet$ is Hausdorff, and $(\ccIet\times_Y\ccIet)_\leq=(\ccIet\times_Y\ccIet)_<\cup\delta(\ccIet)$. The two projections $p_1,p_2:\ccIet\times_Y\ccIet\to\ccIet$ also are local homeomorphisms (both are the étale space corresponding to the sheaf $\mu^{-1}\ccI$).

We regard $U_{\psi\leq\varphi}$ as the pull-back of $(\ccIet\times_Y\ccIet)_\leq$ by the section $(\psi,\varphi):U\to\ccIet\times_Y\ccIet$. Similarly, $U_{\psi<\varphi}$ denotes the subset defined by $y\in U_{\psi<\varphi}$ iff $\psi_y\ley\varphi_y$. If we assume that $\ccIet$ is Hausdorff, $U_{\psi<\varphi}$ is open in~$U$ (because its complementary set in $U_{\psi\leq\varphi}$, which is the pull-back by $(\psi,\varphi)$ of the diagonal, is then closed in $U_{\psi\leq\varphi}$). If~$U$ is connected, then either $\varphi\equiv\psi$ on~$U$ and $U_{\psi<\varphi}=\emptyset$, or $\psi<\varphi$ everywhere on~$U$ and $U_{\psi<\varphi}=U_{\psi\leq\varphi}$.
\end{notation}

\begin{exemple}[Main example in dimension one]\label{exem:Stokes}
Let~$X$ be the open disc centered at~$0$ and radius $1$ in $\CC$, with coordinate $x$. Let~$X^*$ be the punctured disc $X\moins\{0\}$ and let us denote by $\varpi:\wt X\to X$ the \emphb{real oriented blowing-up} of~$X$ at the origin, so that \index{$XWTD$@$\wt X(D)$}$\wt X=S^1\times[0,1)$ (polar coordinates $(x/|x|,|x|$). In the following, we denote by~$S^1$ the boundary $S^1\times\{0\}$ of~$X$ and we forget about~$X$. We still denote by $\varpi$ the constant map $S^1\to\{0\}$.

Set $\cO=\CC\{x\}$, $\cO(*0)=\CC\lpb x\rpb$, and let~$\ccI$ be the constant sheaf $\varpi^{-1}\ccP$ on~$S^1$, with \index{$P$@$\ccP$}$\ccP=\cO(*0)/\cO$ (polar parts of Laurent expansions). For any connected open set~$U$ of~$S^1$ and any $\varphi,\psi\in\ccP=\Gamma(U,\ccI)$, we define $\psi\leqU\varphi$ if $e^{\psi-\varphi}$ has \emphb{moderate growth} on some neighbourhood of~$U$ in $S^1\times[0,1)$, that is, for any compact set $K\subset U$, there exists $C_K>0$ and $N_K\in\NN$ such that, for some open neighbourhood $\nb(K)$ of $K$ in $S^1\times[0,1)$ and for some representatives of $\varphi,\psi$ in $\cO(*0)$, the inequality $|e^{\psi-\varphi}|\leq C_K|x|^{N_K}$ holds on $\nb(K)\cap S^1\times(0,1)$ (this can also be expressed in polar coordinates $x=re^{i\theta}$).

Let us check that this is indeed an order on $\Gamma(U,\ccI)$. The only point to check, due to the additivity property, is that, for $\eta\in\ccP$, $\eta\leqU0$ and $0\leqU\eta$ imply $\eta=0$. If $\eta\neq0$, let us write $\eta=u_n(x)x^{-n}$ with $n\geq1$ and $u_n(0)\neq0$. For a given $\theta\in U$, we have
\bgroup\numstareq
\begin{equation}\label{eq:orderone}\index{$AAAORDth$@$\leqtheta$, $\letheta$}
\eta\leqtheta0\ssi \eta=0\text{ or }\arg u_n(0)-n\theta\in(\pi/2,3\pi/2)\mod2\pi.
\end{equation}
\egroup
It clearly follows that the relation $\leqtheta$ is an order relation on $\ccI_\theta$, and since for $\eta\in\Gamma(U,\ccI)$ we have $\eta\leqU0$ if and only if $\eta\leqtheta0$ for any $\theta\in U$, the relation $\leqU$ is an order relation on $\Gamma(U,\ccI)$.

It clearly follows from \eqref{eq:orderone} that $\psi\leqtheta\varphi$ iff $\reel(\psi-\nobreak\varphi)\leq\nobreak0$ in some neighbourhood of $(\theta,0)\in S^1\times[0,1)$. We also have
\bgroup\numstarstareq
\begin{equation}\label{eq:orderonestrict}\index{$AAAORDth$@$\leqtheta$, $\letheta$}
\eta\letheta0\ssi \eta\neq0\text{ and }\arg u_n(0)-n\theta\in(\pi/2,3\pi/2)\mod2\pi.
\end{equation}
\egroup

Given $\varphi\neq\psi\in\ccP$, there exists at most a finite subset of~$S^1$ where neither $\psi\letheta\varphi$ nor $\varphi\letheta\psi$: if $\varphi-\psi=u_n(x)x^{-n}$ as above, it consists of the $\theta$ such that $\arg u_n(0)-n\theta=\pm\pi/2\bmod2\pi$. It is called the set of \emphb{Stokes directions} of $(\varphi,\psi)$. It is the boundary $\St(\varphi,\psi)$ of $S^1_{\psi\leq\varphi}$.
\end{exemple}

\begin{definitio}[Exhaustivity]\label{def:exhaustI}
We say that a sheaf~$\ccI$ satisfies the \emphb{exhaustivity property} if, for any $y\in Y$ and any finite set $\Phi_y\subset\ccI_y$, there exists $\varphi_y,\psi_y\in\ccI_y$ such that $\psi_y\ley \Phi_y$ and $\Phi_y\leq\varphi_y$.
\end{definitio}

\begin{exemple*}[\ref{exem:Stokes} continued]
Let us remark that the constant sheaf~$\ccI=\varpi^{-1}\ccP$ of Example \ref{exem:Stokes} above satisfies the exhaustivity property: this is seen by choosing $\psi_\theta$ and $\varphi_\theta$ with a pole of order bigger than the orders of the pole of the elements of $\Phi_\theta$, and suitable dominant coefficients.
\end{exemple*}

\subsection{The category of pre-$\ccI$-filtrations}\label{subsec:catpreIfilt}
This category will be the ambient category where objects are defined. Its main advantage is to be abelian. We assume that~$\ccI$ is as in \S\ref{subsec:espaceetalesord}. Recall that we have a commutative diagram of étale maps
\begin{equation}\label{eq:p1p2mu}
\begin{array}{c}
\xymatrix{
\ccIet\times_Y\ccIet\ar[r]^-{p_1}\ar[d]_{p_2}&\ccIet\ar[d]^\mu\\
\ccIet\ar[r]^-\mu&Y
}
\end{array}
\end{equation}
and an open subset $(\ccIet\times_Y\ccIet)_\leq\subset\ccIet\times_Y\ccIet$. We denote by \index{$BETALEQ$@$\beta_\leq$}$\beta_\leq$ the functor on the category of sheaves on $\ccIet\times_Y\ccIet$ which restricts to $(\ccIet\times_Y\ccIet)_\leq$ and extends by zero. In \cite[Prop\ptbl2.3.6]{K-S90}, it is denoted by putting $(\ccIet\times_Y\ccIet)_\leq$ as an index. This is an exact functor, and there is a natural morphism of functors $\beta_\leq\to\id$.

Let $\cF$ be a sheaf of $\kk$-vector spaces on $Y$. We then have $p_1^{-1}\mu^{-1}\cF=p_2^{-1}\mu^{-1}\cF$ and therefore a functorial morphism $\beta_\leq p_1^{-1}\mu^{-1}\cF\to p_2^{-1}\mu^{-1}\cF$. We call $\mu^{-1}\cF$ (together with this morphism) the \index{pre-I-filtration@pre-$\ccI$-filtration!constant --}\emphb{constant pre-$\ccI$-filtration} on $\cF$.

\skpt
\begin{definitio}[Pre-$\ccI$-filtration]\label{def:preIfilt}\ligne
\begin{enumerate}
\item\label{def:preIfilt1}
By a \index{pre-I-filtration@pre-$\ccI$-filtration}\emph{pre-$\ccI$-filtration} we will mean a sheaf \index{$FLEQ$@$\cF_\leq$}$\cF_\leq$ of $\kk$-vector spaces on $\ccIet$ equipped with a morphism $\beta_\leq p_1^{-1}\cF_\leq\to p_2^{-1}\cF_\leq$ such that, when restricted to the diagonal, $\beta_\leq p_1^{-1}\cF_\leq\to p_2^{-1}\cF_\leq$ is the identity.
\item\label{def:preIfilt2}
A morphism $\lambda:\cF_\leq\to\cF'_\leq$ of pre-$\ccI$-filtrations is a morphism of the corresponding sheaves, with the obvious compatibility relation: the following diagram should commute:
\[
\begin{array}{c}
\xymatrix{
\beta_\leq p_1^{-1}\cF_\leq\ar[r]\ar[d]_{\beta_\leq p_1^{-1}\lambda}& p_2^{-1}\cF_\leq\ar[d]^{p_2^{-1}\lambda}\\
\beta_\leq p_1^{-1}\cF'_\leq\ar[r]& p_2^{-1}\cF'_\leq
}
\end{array}
\]
\end{enumerate}
\end{definitio}

We therefore define in this way a category, denoted by \index{$ModkIet$@$\Mod(\kk_{\ccIet,\leq})$}$\Mod(\kk_{\ccIet,\leq})$. It comes with a ``forgetting'' functor to the category $\Mod(\kk_{\ccIet})$ of sheaves of $\kk$-vector spaces on $\ccIet$. We will usually denote an object of both categories by $\cF_\leq$, hoping that this does not lead to any confusion.

Given a sheaf $\cF_\leq$ on $\ccIet$, we denote its germ at $\varphi_y\in\ccI_y$ by $\cF_{\leq\varphi_y}$. The supplementary data of $\beta_\leq p_1^{-1}\cF_\leq\to p_2^{-1}\cF_\leq$ induces, for each pair of germs $\varphi_y,\psi_y\in\ccI_y$ such that $\varphi_y\leqy\psi_y$ a morphism $\cF_{\leq\varphi_y}\to\cF_{\leq\psi_y}$, justifying the name of ``filtration'', although we do not impose injectivity in order to get an abelian category. The condition \ref{def:preIfilt}\eqref{def:preIfilt2} on morphisms is that they are compatible with the ``filtration''. Given a section $\varphi\in\ccI(U)$ on some open set $U\subset Y$, we also denote by $\cF_{\leq\varphi}$ the pull-back $\varphi^{-1}\cF_\leq$ on~$U$. This is a sheaf on~$U$, whose germ at $y\in U$ is the germ $\cF_{\leq\varphi_y}$ of $\cF_\leq$ at $\varphi_y\in\ccIet$.\enlargethispage{\baselineskip}%

\begin{remarque}\label{rem:sheafonccI}
The space $\ccIet$ has an open covering formed by the $\varphi(U)$, where $U$ runs among open subsets of $Y$ and, for each $U$, $\varphi$ belongs to $\Gamma(U,\ccI)$. Giving a sheaf on $\ccIet$ is then equivalent to giving a sheaf $\cF_{\leq\varphi}$ on $U$ for each pair $(U,\varphi)$, together with compatible isomorphisms $(\cF_{\leq\varphi})_{|V}\isom\cF_{\leq(\varphi_{|V})}$ for $V\subset U$. The supplementary data of a morphism $\beta_\leq p_1^{-1}\cF_\leq\to p_2^{-1}\cF_\leq$ is equivalent to giving, for each pair $\varphi,\psi\in\Gamma(U,\ccI)$, of a morphism $\beta_{\varphi\leq\psi}\cF_{\leq\varphi}\to\cF_{\leq\psi}$ (where $\beta_{\varphi\leq\psi}$ is the functor composed of the restriction to $U_{\varphi\leq\psi}$ and the extension by zero), and such morphisms should commute with the isomorphisms above for each $V\subset U$.
\end{remarque}

\begin{definitio}[Twist]\label{def:twist}\index{twist}
Let $\cF_\leq$ be a~pre-$\ccI$-filtration and let $\eta\in\Gamma(Y,\ccI)$ be a global section of~$\ccI$. The \index{pre-I-filtration@pre-$\ccI$-filtration!twisted}\emphb{twisted pre-$\ccI$-filtration} \index{$FLEQETA$@$\cF[\eta]_\leq$}$\cF[\eta]_\leq$ is defined by $\cF[\eta]_{\leq\varphi}=\cF_{\leq\varphi-\eta}$ for any local section~$\varphi$ of~$\ccI$. It is a~pre-$\ccI$-filtration. A morphism $\lambda:\cF_\leq\to\cG_\leq$ induces a morphism $\cF[\eta]_\leq\to\cG[\eta]_\leq$, so that $[\eta]$ is a functor from $\Mod(\kk_{\ccIet,\leq})$ to itself, with inverse functor $[-\eta]$.
\end{definitio}

\begin{lemme}\label{lem:preIfiltab}
The category \index{$ModkIet$@$\Mod(\kk_{\ccIet,\leq})$}$\Mod(\kk_{\ccIet,\leq})$ of pre-$\ccI$-filtrations is abelian and a sequence is exact if and only if the sequence of the underlying sheaves on $\ccIet$ is exact.
\end{lemme}

\begin{proof}
This follows from the fact that $\beta_\leq p_1^{-1}$ and $p_2^{-1}$ are exact functors.
\end{proof}

\begin{remarque}[Tensor product]\label{rem:tensprod}\index{tensor product}
We will not try to define at this level of generality an internal tensor product of the category $\Mod(\kk_{\ccIet,\leq})$. However, we can easily define a tensor product operation $\Mod(\kk_Y)\times\Mod(\kk_{\ccIet,\leq})\to\Mod(\kk_{\ccIet,\leq})$. To $\cF'$ in $\Mod(\kk_Y)$ and $\cF_\leq$ in $\Mod(\kk_{\ccIet,\leq})$, we associate $\mu^{-1}\cF'\otimes_{\kk}\cF_\leq$, where $\mu^{-1}\cF'$ is the constant pre-$\ccI$-filtration on $\cF'$. The associated morphism is obtained as
\begin{multline*}
\beta_\leq p_1^{-1}(\mu^{-1}\cF'\otimes_{\kk}\cF_\leq)=p_1^{-1}\mu^{-1}\cF'\otimes_{\kk}\beta_\leq p_1^{-1}\cF_\leq\\
\to p_1^{-1}\mu^{-1}\cF'\otimes_{\kk}p_2^{-1}\cF_\leq=p_2^{-1}\mu^{-1}\cF'\otimes_{\kk}p_2^{-1}\cF_\leq=p_2^{-1}(\mu^{-1}\cF'\otimes_{\kk}\cF_\leq).
\end{multline*}
\end{remarque}

\subsubsection*{Push-forward}
Our purpose is now to define the index{push-forward (direct image)!of pre-$\ccI$-filtrations}\emph{push-forward of a pre-$\ccI$-filtration} by a continuous map $f:Y'\to Y$. We first consider the case of a cartesian diagram
\begin{equation}\label{eq:diagramf}
\begin{array}{c}
\xymatrix{
\wt\ccIet\ar[r]^-{\wt f}\ar[d]_{\wt\mu}&\ccIet\ar[d]^\mu\\
Y'\ar[r]^-f&Y
}
\end{array}
\end{equation}
with $\wt\ccI\defin f^{-1}\ccI$ equipped with the pull-back order, so that $(\wt\ccIet\times_{Y'}\wt\ccIet)_\leq$ is the pull-back of the open set $(\ccIet\times_Y\ccIet)_\leq\subset \ccIet\times_Y\ccIet$ by $\wt f\times\wt f$. As a consequence,
\begin{equation}\label{eq:wtfbeta}
\begin{split}
(\wt f\times\wt f)_*\circ\wt\beta_\leq&=\beta_\leq\circ(\wt f\times\wt f)_*\\
(\wt f\times\wt f)^{-1}\circ\beta_\leq&=\wt\beta_\leq\circ(\wt f\times\wt f)^{-1}.
\end{split}
\end{equation}

\begin{lemme}\label{lem:adjunct}
There are natural functors
\begin{align*}
\wt f_*:\Mod(\kk_{\wt\ccIet,\leq})&\to\Mod(\kk_{\ccIet,\leq})\\
\wt f^{-1}:\Mod(\kk_{\ccIet,\leq})&\to\Mod(\kk_{\wt\ccIet,\leq})
\end{align*}
and natural morphisms of functors
\begin{align*}
\wt f^{-1}\wt f_*&\to\id\quad\text{in }\Mod(\kk_{\wt\ccIet,\leq})\\
\id&\to\wt f_*\wt f^{-1}\quad\text{in }\Mod(\kk_{\ccIet,\leq}).
\end{align*}
Moreover,
\[
\Hom_{\kk_{\ccIet,\leq}}(\cF_\leq,\wt f_*\wt\cG_\leq)\simeq \Hom_{\kk_{\wt\ccIet,\leq}}(\wt f^{-1}\cF_\leq,\wt\cG_\leq).
\]
\end{lemme}

\begin{proof}
At the level of sheaves the result is standard. One has only to check the compatibility with the morphisms $\beta_\leq p_1^{-1}\to p_2^{-1}$ and $\wt\beta_\leq\wt p_1^{-1}\to\wt p_2^{-1}$. It follows from \eqref{eq:wtfbeta}. For instance, the \emphb{adjunction} morphisms are obtained from the similar ones for $\wt f\times\wt f$.
\end{proof}

\begin{corollaire}\label{cor:inj}
The category $\Mod(\kk_{\ccIet,\leq})$ has enough injectives.
\end{corollaire}

\begin{proof}
Let $Y_{\disc}$ be $Y$ equipped with its discrete topology and $\alpha:Y_{\disc}\to Y$ be the canonical morphism. Let us consider the diagram \eqref{eq:diagramf} for $\alpha$. Since $\wt\mu$ is a local homeomorphism, $\wt{\ccIet}$ is nothing but $\ccIet_{\disc}$ and $\wt\alpha$ is the corresponding canonical morphism. Using the adjunction for $\wt\alpha$, one shows as in \cite[Prop\ptbl2.4.3]{K-S90} that it is enough to check the corollary for $Y_\disc$, that is, stalkwise for $\ccIet$. Assume then that we are given vector spaces $\cF_{\leq\varphi_y}$ for any $\varphi_y\in\ccIet$, together with morphisms $\cF_{\leq\varphi_y}\to\cF_{\leq\psi_y}$ whenever $\varphi_y\leqy\psi_y$. Choose a monomorphism $\cF_{\leq\varphi_y}\hto\cI_{\leq\varphi_y}$ into an injective object in the category of $\kk$-vector spaces, for any $\varphi_y\in\ccIet$. By injectivity of $\cI_{\leq\psi_y}$, the composed morphism $\cF_{\leq\varphi_y}\to\cI_{\leq\psi_y}$ extends to a morphism $\cI_{\leq\varphi_y}\to\cI_{\leq\psi_y}$, providing a monomorphism $\cF_\leq\hto\cI_\leq$ in $\Mod(\kk_{\ccIet_{\disc},\leq})$.
\end{proof}

Arguing as in \cite[Prop\ptbl 2.4.6(vii)]{K-S90}, one checks that injectives are flabby (\ie the underlying sheaf is flabby). Going back to the setting of Lemma \ref{lem:adjunct}, we deduce:

\begin{corollaire}\label{cor:flabby}
The subcategory of $\Mod(\kk_{\wt\ccIet,\leq})$ whose objects have a flabby underlying sheaf is injective with respect to $\wt f_*$ (\cf\cite[Def\ptbl1.8.2]{K-S90}).\qed
\end{corollaire}

Notice also that the construction of Godement (\cite[\S II.4.3]{Godement64}) can be extended to the present setting: using the notation of the proof of Corollary \ref{cor:inj}, the adjunction $\id\to\wt\alpha_*\wt\alpha^{-1}$ gives an injective morphism $\wt\cF_\leq\to\wt\alpha_*\wt\alpha^{-1}\wt\cF_\leq$ and the latter has a flabby underlying sheaf.

We will denote by $D^+(\kk_{\ccIet,\leq})$ (\resp \index{$DbkIet$@$D^\rb(\kk_{\ccIet,\leq})$}$D^\rb(\kk_{\ccIet,\leq})$) the left-bounded (\resp bounded) derived category of $\Mod(\kk_{\ccIet,\leq})$. We have ``forgetting'' functors to $D^+(\kk_{\ccIet})$, \resp $D^\rb(\kk_{\ccIet})$.

\begin{remarque}[Tensor product]\label{rem:tensprodder}\index{tensor product}
According to Remark \ref{rem:tensprod}, we have a tensor product bifuntor $D^\rb(\kk_Y)\times D^\rb(\kk_{\ccIet,\leq})\to D^\rb(\kk_{\ccIet,\leq})$.
\end{remarque}

\begin{corollaire}\label{cor:Rf}
The derived functor $\bR\wt f_*:D^+(\kk_{\wt\ccIet,\leq})\to D^+(\kk_{\ccIet,\leq})$ is well-defined and compatible with the similar functor $D^+(\kk_{\wt\ccIet})\to D^+(\kk_{\ccIet})$, and if $f$ has finite cohomological dimension, so does $\bR\wt f_*:D^\rb(\kk_{\wt\ccIet,\leq})\to D^\rb(\kk_{\ccIet,\leq})$.\qed
\end{corollaire}

In general, given $\ccI'$ on $Y'$,~$\ccI$ on $Y$ and $f:Y'\to Y$, we wish to define the push-forward of an object of $\Mod(\kk_{\ccIpet,\leq})$ as an object of $\Mod(\kk_{\ccIet,\leq})$, and similarly at the level of derived categories. In order to do so, we need to assume the existence of a morphism $q_f:\wt\ccI\defin f^{-1}\ccI\to\ccI'$. We now consider a diagram
\begin{equation}\label{eq:diagramfqf}
\begin{array}{c}
\xymatrix{
\ccIpet\ar[d]_{\mu'}&\wt\ccIet\ar[l]_-{q_f}\ar[r]^-{\wt f}\ar[ld]^(.3){\hspace*{-1mm}\wt\mu}&\ccIet\ar[d]^\mu\\
Y'\ar[rr]^-f&&Y
}
\end{array}
\end{equation}
and we will also make the natural assumption that $q_f$ is compatible with the order, that is, given any two germs $\varphi_y,\psi_y\in\ccI_y$ ($y\in Y$), defining germs $\varphi_y,\psi_y\in(\wt f^{-1}\ccI)_{y'}$ at $y'\in f^{-1}(y)$, we have $\varphi_y\leqy\psi_y\implique q_f(\varphi_y)\leqyp q_f(\psi_y)$ for any such $y'$. This is equivalent to asking that $(\wt\ccIet\times_{Y'}\wt\ccIet)_\leq$ is contained in the pull-back of $(\ccIpet\times_{Y'}\ccIpet)_\leq$ by $q_f\times q_f$. Under this assumption, we get a morphism of functors
\begin{equation}\label{eq:betaq}
\wt\beta_\leq(q_f\times q_f)^{-1}\to(q_f\times q_f)^{-1}\beta'_\leq
\end{equation}
Therefore, for any object $\cF'_\leq$ of $\Mod(\kk_{\ccIpet,\leq})$ with corresponding morphism $\beta'_\leq p^{\prime-1}_1\cF'_\leq\to p^{\prime-1}_2\cF'_\leq$, the sheaf $q_f^{-1}\cF'_\leq$ comes equipped with a morphism $\wt \beta_\leq\wt p^{-1}_1q_f^{-1}\cF'_\leq\to \wt p^{-1}_2q_f^{-1}\cF'_\leq$, according to \eqref{eq:betaq}, and $\wt f_*q_f^{-1}\cF'_\leq$ with a morphism $\beta_\leq p^{-1}_1\wt f_*q_f^{-1}\cF'_\leq\to p^{-1}_2\wt f_*q_f^{-1}\cF'_\leq$, according to \eqref{eq:wtfbeta}.

\begin{definitio}\label{def:imdirpreIfilt}
Under these assumptions, the \index{push-forward (direct image)}push-forward by $f$ is defined by $\wt f_*q_f^{-1}$.
\end{definitio}

The functor $q_f^{-1}$ is exact, while $\wt f_*$ is a priori only left exact. We will now extend the push-forward functor to the derived categories. For the sake of simplicity, we assume that $f$ has finite cohomological dimension.

\begin{definitio}[Push-forward]\label{def:RimdirpreIfilt}
Under these assumptions, the \index{push-forward (direct image)!of pre-$\ccI$-filtrations}push-forward functor \index{$Fsubplus$@$f_+$ (push-forward)}$f_+:D^\rb(\kk_{\ccIpet,\leq})\to D^\rb(\kk_{\ccIet,\leq})$ is defined as the composed functor $\bR \wt f_* q_f^{-1}$.
\end{definitio}

\subsubsection*{The functor $\protect\cHom(\iota^{-1}\protect\cbbullet,\mu^{-1}\protect\cG)$}
This functor will be useful for defining duality in \S\ref{subsec:dualStperv}. We denote by $\iota$ the involution $\varphi\mto-\varphi$ on $\ccIet$ and by $\tau$ the permutation $(\varphi,\psi)\mto(\psi,\varphi)$ on $\ccIet\times_Y\ccIet$. We have $p_2=p_1\circ\tau$. Let~$\cG$ be a sheaf on $Y$. Recall that $\mu^{-1}\cG$ is the constant pre-$\ccI$-filtration on $\cG$. We will define a contravariant functor $\cHom(\iota^{-1}\cbbullet,\mu^{-1}\cG)$ from $\Mod(\kk_{\ccIet,\leq})$ to itself, over the corresponding functor on $\Mod(\kk_{\ccIet})$.

Given an object $\cF_\leq$ of $\Mod(\kk_{\ccIet,\leq})$, we apply $\tau^{-1}$ to the corresponding morphism $\beta_\leq p_1^{-1}\cF_\leq\to p_2^{-1}\cF_\leq$ to get a morphism $\beta_\geq p_2^{-1}\cF_\leq\to p_1^{-1}\cF_\leq$ (with obvious notation~$\beta_\geq$). Applying then $(\iota\times_Y\iota)^{-1}$, we get a morphism $\beta_\leq p_2^{-1}\iota^{-1}\cF_\leq\to p_1^{-1}\iota^{-1}\cF_\leq$. Since $p_2^{-1}\mu^{-1}\cG=p_1^{-1}\mu^{-1}\cG$, we obtain, since $p_1$ and $p_2$ are local homeomorphisms, a morphism
\begin{multline*}
p_1^{-1}\cHom_{\kk}(\iota^{-1}\cF_\leq,\mu^{-1}\cG)=\cHom_{\kk}(p_1^{-1}\iota^{-1}\cF_\leq,p_1^{-1}\mu^{-1}\cG)\\
\to\cHom_{\kk}(\beta_\leq p_2^{-1}\iota^{-1}\cF_\leq,p_2^{-1}\mu^{-1}\cG)=j_{\leq,*}j_\leq^{-1} p_2^{-1}\cHom_{\kk}(\iota^{-1}\cF_\leq,\mu^{-1}\cG),
\end{multline*}
where $j_\leq$ is the open inclusion $(\ccIet\times_Y\ccIet)_\leq\hto\ccIet\times_Y\ccIet$ (\cf \cite[(2.3.18)]{K-S90} for the second equality). Since for any two sheaves $\cA$ and $\cB$ and any open set $Z$, the natural morphism $\Hom(\cA_Z,\cB)\to\Hom(j_Z^{-1}\cA,j_Z^{-1}\cB)=\Hom(\cA,j_{Z,*}j_Z^{-1}\cB)$ is an isomorphism (\cf \loccit), we deduce a morphism
\[
\beta_\leq p_1^{-1}\cHom_{\kk}(\iota^{-1}\cF_\leq,\mu^{-1}\cG)\to p_2^{-1}\cHom_{\kk}(\iota^{-1}\cF_\leq,\mu^{-1}\cG),
\]
as wanted. That it is the identity on the diagonal is clear. The construction is also clearly functorial (in a contravariant way). By Lemma \ref{lem:preIfiltab}, this functor is exact if $\cG$ is injective in $\Mod(\kk_Y)$.

\begin{corollaire}\label{cor:Rhom}
For any object $\cG$ of $D^+(\kk_Y)$ there is a well-defined functor
\[
\bR\cHom(\iota^{-1}\cbbullet,\mu^{-1}\cG):D^{b,\textup{op}}(\kk_{\ccIet,\leq})\to D^+(\kk_{\ccIet,\leq}),
\]
and for any $\varphi\in\Gamma(U,\ccI)$, we have $\varphi^{-1}\bR\cHom(\iota^{-1}\cF_\leq,\mu^{-1}\cG)=\bR\cHom(\cF_{\leq-\varphi},\cG)$.\qed
\end{corollaire}

\subsection{Traces for étale maps}\label{subsec:traces}
Let us recall the definition of the \emphb{trace map} on sheaves of $\kk$-vector spaces, where $\kk$ is a field (more generally, a ring). Let $q:Z\to Z'$ be an étale map between topological spaces. For any $z'\in Z'$, the fibre $q^{-1}(z')$ is discrete with the induced topology. Let $\cF'$ be a sheaf on $Z'$ and let $\wt\cF$ be a subsheaf of $q^{-1}\cF'$. We denote by \index{$TRq$@$\Tr_q$}$\Tr_q(\wt\cF,\cF')$ the subsheaf of $\cF'$ defined by the following condition: for any open subset~$U'$ of $Z'$, $\Gamma(U', \Tr_q(\wt\cF,\cF'))\subset\Gamma(U',\cF')$ consists of those sections $f\in\Gamma(U',\cF')$ such that, for any $z'\in U'$, the germ $f_{z'}$ belongs to $\sum_{z\in q^{-1}(z')}\wt\cF_z$ (where we canonically identify $(q^{-1}\cF')_z$ with $\cF'_{q(z)}$, so that $\wt\cF_z\subset\cF'_{q(z)}$ and the sum is taken in $\cF'_{q(z)}$). This clearly defines a subsheaf of~$\cF'$.

\begin{lemme}\label{lem:trace}
The germ $\Tr_q(\wt\cF,\cF')_{z'}$ is equal to $\sum_{z\in q^{-1}(z')}\wt\cF_z$ (the sum is taken in $\cF'_{z'}$).
\end{lemme}

\begin{proof}
Indeed, the inclusion $\subset$ is clear. For the converse, let $f'_{z'}\in\sum_{z\in q^{-1}(z')}\wt\cF_z$. Then there exists $n\in\NN$, $z_1,\dots,z_n\in q^{-1}(z)$ and $f_{z_i}\in\wt\cF_{z_i}$ ($i=1,\dots,n$) such that $f'_{z'}=\sum_if_{z_i}$. As $q$ is étale, there exist open neighbourhoods $V'_{z'}$ of $z'$ and~$V_i$ of~$z_i$ such that $q_i:V_i\to V'_{z'}$ is an homeomorphism for each $i$ and there exist representatives $f'\in\nobreak\Gamma(V'_{z'},\cF')$ and $f_i\in\Gamma(V_i,\wt\cF)$ such that $f'=\sum_if_i$ (by identifying $(q^{-1}\cF')_{|V_i}$ to $\cF'_{|V'_{z'}}$ via $q_i$). This shows that $f'\in\Gamma(V'_{z'},\Tr_q(\wt\cF,\cF'))$ and therefore $f'_{z'}\in\Tr_q(\wt\cF,\cF')_{z'}$.
\end{proof}

If $Z$ is locally compact, then the functor $q_!$ is defined (\cf \cite[\S2.5]{K-S90}), and $\Tr_q(\wt\cF,\cF')$ is the image of the composed morphism $q_!\wt\cF\to q_!q^{-1}\cF'\to\cF'$.

There is a related notion of \emphb{saturation with respect to an order}, an operation that we denote by \index{$TRLEQ$@$\Tr_\leq$}$\Tr_\leq$. Let $Y,\ccI,\mu$ be as in \S\ref{subsec:espaceetalesord} and let $\cF$ be a sheaf of $\kk$\nobreakdash-vector spaces on~$Y$ (where $\kk$ is a field). If $\cF'$ is a subsheaf of $\mu^{-1}\cF$, the subsheaf $\Tr_\leq(\cF',\mu^{-1}\cF)$ of $\mu^{-1}\cF$ is defined by the following condition: for any open subset~$V$ of $\ccIet$, $\Gamma(V,\Tr_\leq(\cF',\mu^{-1}\cF))$ consists of those sections $f\in\Gamma(V,\mu^{-1}\cF)$ such that, for any $\varphi_y\in V$, the germ of $f$ at $\varphi_y$ belongs to $\sum_{\psi_y\leqy\varphi_y}\cF'_{\psi_y}$.

\begin{lemme}\label{lem:traceleq}
The germ $\Tr_\leq(\cF',\mu^{-1}\cF)_{\varphi_y}$ is equal to $\sum_{\psi_y\leqy\varphi_y}\cF'_{\psi_y}$ (the sum is taken in~$\cF_y$, due to the canonical identification of $(\mu^{-1}\cF)_{\psi_y}$ with $\cF_y$).
\end{lemme}

\begin{proof}
The inclusion $\subset$ is clear. Let $f_{\varphi_y}\in\sum_{\psi_y\leqy\varphi_y}\cF'_{\psi_y}$. There exists a finite subset $\Psi_y\in\mu^{-1}(y)$ such that $\psi_y\leq\varphi_y$ for any $\psi_y\in\Psi_y$ and $f_{\varphi_y}=\sum_{\psi_y\in\Psi_y}f'_{\psi_y}$, with $f'_{\psi_y}\in\cF'_{\psi_y}$, and the sum is taken in $\cF_y$. Choose an open neighbourhood $U$ of~$y$ in $Y$ such that the germs $\varphi_y$ and $\psi_y\in\Psi_y$ are defined on $U$. In particular, $\mu:\psi(U)\to U$ is a homeomorphism for any $\psi$. Up to shrinking $U$, we can also assume that each germ~$f'_{\psi_y}$ is defined on $\psi(U)$ as a section $f'_\psi\in\Gamma(\psi(U),\cF')$. Via $\mu$, we consider $\cF'_{|\psi(U)}$ as a subsheaf of $\cF_{|U}$ and we set $f_\varphi=\sum_{\psi_y\in\Psi_y}f'_\psi\in\Gamma(U,\cF)=\Gamma(\varphi(U),\mu^{-1}(\cF))$.

If $U$ is small enough, for any $z\in U$ and any $\psi_y\in\Psi_y$, we have $\psi_z\leqz\varphi_z$, by the openness of $\leq$ and because $\Psi_y$ is finite. This implies that the germ of $f_\varphi$ at any such~$z$ belongs to $\sum_{\eta_z\leqy\varphi_z}\cF'_{\eta_z}$, as was to be shown.
\end{proof}

\begin{remarque}\label{rem:traceleqfoncteur}
The correspondence $(\cF,\cF')\mto(\cF,\Tr_\leq(\cF',\mu^{-1}\cF))$ from the category of pairs $(\cF,\cF')$ with $\cF'\subset\mu^{-1}\cF$ (and morphisms being morphisms $\lambda:\cF_1\to\cF_2$ such that $\mu^{-1}\lambda$ sends $\cF'_1$ to $\cF'_2$) to the category of pre-$\ccI$-filtered sheaves (see below) is functorial. Indeed, if $\mu^{-1}\lambda(\cF'_1)\subset\cF'_2$, then $\mu^{-1}\lambda\big[\Tr_\leq(\cF'_2,\mu^{-1}\cF_2))\big]\subset\Tr_\leq(\cF'_2,\mu^{-1}\cF_2))$.
\end{remarque}

\subsection{Pre-$\ccI$-filtrations of sheaves}\label{subsec:preIfilt}
We assume that~$\ccI$ is as in \S\ref{subsec:espaceetalesord}. As above, $\cF$ denotes a sheaf of $\kk$-vector spaces on $Y$.

\begin{definitio}[Pre-$\ccI$-filtration of a sheaf]\label{def:pre-I-fil}
We say that a subsheaf $\cF_\leq$ of $\mu^{-1}\cF$ is a \index{pre-I-filtration@pre-$\ccI$-filtration!of a sheaf}\emph{pre-$\ccI$-filtration} of $\cF$ if, for any $y\in Y$ and any $\varphi_y,\psi_y\in\mu^{-1}(y)$, $\psi_y\leqy\varphi_y$ implies $\cF_{\leq\psi_y}\subset\cF_{\leq\varphi_y}$ or, equivalently, if $\cF_\leq=\Tr_\leq(\cF_\leq,\mu^{-1}\cF)$. We will denote by $(\cF,\cF_\bbullet)$ the data of a sheaf and a pre-$\ccI$-filtration on it.
\end{definitio}

In other words, a pre-$\ccI$-filtration of $\cF$ is an object $\cF_\leq$ of $\Mod(\kk_{\ccIet,\leq})$ together with an inclusion into the constant pre-$\ccI$-filtration $\mu^{-1}\cF$: indeed, for any $\varphi_y\leqy\psi_y$, the morphism $\cF_{\leq\varphi_y}\to\cF_{\leq\psi_y}$ given by the pre-$\ccI$-filtration is then an inclusion, since both $\cF_{\leq\varphi_y}\to\cF_y$ and $\cF_{\leq\psi_y}\to\cF_y$ are inclusions.

\begin{definitio}\label{def:morphismpreIfiltered}
Let $(\cF,\cF_\bbullet),(\cF',\cF'_\bbullet)$ be pre-$\ccI$-filtered sheaves and let $f:\cF\to\cF'$ be a morphism of sheaves. We say that it is a \index{morphism!of pre-$\ccI$-filtered sheaves}\emph{morphism of pre-$\ccI$-filtered sheaves} if $\mu^{-1}f$ sends $\cF_\leq$ to $\cF'_\leq$. We say that $f$ is \index{morphism!strict}\index{morphism!strict}\index{strict morphism}\emph{strict} if $\mu^{-1}f(\cF_\leq)=\cF'_\leq\cap\mu^{-1}(f(\cF))$.
\end{definitio}

We therefore get the (a priori non abelian) category of pre-$\ccI$-filtered sheaves of $\kk$-vector spaces. Clearly, $\mu^{-1}f$ induces a morphism $\cF_\leq\to\cF'_\leq$ in the category $\Mod(\kk_{\ccIet,\leq})$. The notion of twist (Definition \ref{def:twist}) applies to pre-$\ccI$-filtered sheaves.

\begin{definitio}[Exhaustivity]\label{def:exhaust}
We say that a pre-$\ccI$-filtration $\cF_\leq$ of $\cF$ is \index{exhaustive!pre-$\ccI$-filtration}\index{pre-I-filtration@pre-$\ccI$-filtration!exhaustive --}\emph{exhaustive} if for any $y\in Y$ there exists $\varphi_y,\psi_y\in\ccI_y$ such that $\cF_{\leq\varphi_y}=\cF_y$ and $\cF_{\leq\psi_y}=0$.
\end{definitio}

We now introduce some operations on pre-$\ccI$-filtered sheaves.

\subsubsection*{Extending the sheaf~$\ccI$}
Let $q:\wt\ccI\to\ccI'$ be a morphism of sheaves of ordered abelian groups on $Y'$. We also denote by $q:\wt\ccIet\to\ccIpet$ the corresponding map between étale spaces, so that we have a commutative diagram of maps:
\begin{equation}\label{eq:diagetale}
\begin{array}{c}
\xymatrix{
\wt\ccIet\ar@<-.5ex>[r]^-{q}\ar[rd]_(.37){\wt\mu}&\ccIpet\ar[d]^{\mu'}\\
&Y'
}
\end{array}
\end{equation}
Let $(\cF',\wt\cF_\bbullet)$ be a pre-$\wt\ccI$-filtered sheaf. We will define a pre-$\ccI'$-filtered sheaf $q_*(\cF',\wt\cF_\bbullet)=(\cF',(q\wt\cF)_\bbullet)$. When we consider the diagram \eqref{eq:diagetale}, we can introduce an ordered trace map as in Lemma \ref{lem:traceleq}.

\begin{lemme}\label{lem:traceqleq}
Let $(\cF',\wt\cF_\bbullet)$ be a pre-$\wt\ccI$-filtered sheaf. Then there is a unique subsheaf \index{$TRLEQQ$@$\Tr_{\leq q}$}$\Tr_{\leq q}(\wt\cF_\leq,\mu^{\prime-1}\cF')$ of $\mu^{\prime-1}\cF'$ defined by the property that, for any $y'\in Y'$ and any $\varphi'_{y'}\in\mu^{\prime-1}(y')$,
\[
\Tr_{\leq q}(\wt\cF_\leq,\mu^{\prime-1}\cF')_{\varphi'_{y'}}=\sum_{\substack{\psi_{y'}\in\wt\ccI_{y'}\\q(\psi_{y'})\leqyp\varphi'_{y'}}}\hspace*{-5mm}\wt\cF_{\leq\psi_{y'}}
\]
(where the sum is taken in $\cF'_{y'}$) is a pre-$\ccI'$-filtration of $\cF'$.
\end{lemme}

\begin{proof}
We note that $\Tr_{\leq q}(\wt\cF_\leq,\mu^{\prime-1}\cF')=\Tr_\leq\big(\Tr_q(\wt\cF_\leq,\mu^{\prime-1}\cF'),\mu^{\prime-1}\cF'\big)$, so the lemma is a consequence of Lemmas \ref{lem:trace} and \ref{lem:traceleq}.
\end{proof}

\begin{definitio}
Given a pre-$\wt\ccI$-filtration $\wt\cF_\leq$ of $\cF'$, we set
$$
(q\wt\cF)_\leq\defin\Tr_{\leq q}(\wt\cF_\leq,\mu^{\prime-1}\cF').
$$
\end{definitio}

\subsubsection*{Restricting the sheaf~$\ccI$}
In the previous situation, let $(\cF',\cF'_\bbullet)$ be a pre-$\ccI'$-filtered sheaf on~$Y'$. Then $\cF'$ becomes pre-$\wt\ccI$-filtered by $q^{-1}\cF'_\leq$.

\subsubsection*{Pull-back}
Let $f:Y'\to Y$ be a continuous map between locally compact spaces. Assume we are sheaves~$\ccI$ on $Y$ and $\ccI'$ on $Y'$, and a morphism $q_f:\wt\ccI\defin f^{-1}\ccI\to\ccI'$ compatible with the order, where $\wt\ccI$ is equipped with the pull-back order (\cf\eqref{eq:diagramf}). Let $(\cF,\cF_\leq)$ be a~pre-$\ccI$-filtered sheaf on $Y$. We will define a pre-$\ccI'$-filtered sheaf $f^+(\cF,\cF_\leq)=(f^{-1}\cF,f^+\cF_\leq)$ on $Y'$.

We have a continuous map $\wt f:\wt\ccIet\to\ccIet$ over $f:Y'\to Y$. Then $f^{-1}(\cF,\cF_\leq)\defin (f^{-1}\cF,\wt f^{-1}\cF_\leq)$ is a pre-$\wt\ccI$-filtered sheaf of $\kk$-vector spaces on $Y'$.

\begin{definitio}[Pull-back]\label{def:pullbackpreI}\index{pull-back (inverse image)!of pre-$\ccI$-filtered sheaves}
The pull-back \index{$Fsupplus$@$f^+$ (pull-back)}$f^+(\cF,\cF_\leq)$ is the sheaf $\cF'=f^{-1}\cF$ equipped with the pre-$\ccI'$-filtration $f^+\cF_\leq\!\defin\!(q_ff^{-1}\cF)_\leq\!=\!\Tr_{\leq q_f}(\wt f{}^{-1}\cF_\leq,\mu'^{-1}f^{-1}\cF)$.
\end{definitio}

\subsubsection*{Push-forward}
The push-forward of pre-$\ccI$-filtrations (Definition \ref{def:RimdirpreIfilt}) is a priori not defined at the level of pre-$\ccI$-filtered sheaves, since the condition that $\cF_\leq\to\mu^{-1}\cF$ is injective may be not satisfied after the push-forward of Definition \ref{def:RimdirpreIfilt}.

\subsubsection*{Grading}
Let us now assume that $\ccIet$ is Hausdorff. The relation~$\ley$ is then also open and we can work with it as with $\leqy$.

\begin{definitio}[Grading a~pre-$\ccI$-filtration (the Hausdorff case)]\label{def:grading}
Assume that $\ccIet$ is Hausdorff. Given a pre-$\ccI$-filtered sheaf $(\cF,\cF_\leq)$, there is, due to the openness of $<$, a unique subsheaf $\cF_<$ of $\cF_\leq$ such that, for any $y\in Y$ and any $\varphi_y\in\ccI_y$, $\cF_{<\varphi_y}=\sum_{\psi_y\ley\varphi_y}\cF_{\leq\psi_y}$. We denote by \index{pre-I-filtration@pre-$\ccI$-filtration!graded --}$\gr\cF$ the quotient sheaf $\cF_\leq/\cF_<$.
\end{definitio}

Notice that $\cF_<$ is also a pre-$\ccI$-filtration of $\cF$. Any morphism $(\cF,\cF_\bbullet)\to(\cF',\cF'_\bbullet)$ sends $\cF_<$ to $\cF'_<$, hence defines a morphism $\gr f:\gr\cF\to\gr\cF'$.

\subsection{$\ccI$-filtrations of sheaves (the Hausdorff case)}\label{subsec:Ifilt}
In order to define the notion of a~$\ccI$-filtration, we will make the assumption that $\ccIet$ is Hausdorff. A more general definition will be given in \S\ref{subsec:Ifilstrat}. This assumption is now implicitly understood. The point is that, if~$\ccIet$ is Hausdorff, it is locally compact (because $Y$ is so), and we make use of the direct image with proper supports $\mu_!$ in the usual way.

\begin{exemple}[Graded~$\ccI$-filtrations]\label{exem:trivIfilt}
Let $\cG$ be a sheaf of $\kk$\nobreakdash-vector spaces on $\ccIet$. We note that, for any $y\in Y$,
\bgroup\numstareq
\begin{equation}\label{eq:trivIfilt}
i^{-1}_y\mu_!\cG=\mu_!i^{-1}_{\mu^{-1}(y)}\cG=\bigoplus_{\varphi_y\in\mu^{-1}(y)}\cG_{\varphi_y}.
\end{equation}
\egroup

Similarly, we have a commutative diagram of étale maps
\[
\xymatrix{
\ccIet\times_Y\ccIet\ar[r]^-{p_1}\ar[d]_{p_2}&\ccIet\ar[d]^\mu\\
\ccIet\ar[r]^-\mu&Y
}
\]
and a base change isomorphism $p_{2,!}p_1^{-1}\cG=\mu^{-1}\mu_!\cG$ (\cf\eg\cite[Prop\ptbl2.5.11]{K-S90}). We denote by $(p_1^{-1}\cG)_\leq$ the extension by~$0$ of the restriction of $p_1^{-1}\cG$ to $(\ccIet\times_Y\ccIet)_\leq$, that is, $(p_1^{-1}\cG)_\leq=j_{\leq,!}j_\leq^{-1}p_1^{-1}\cG$, and by $(\mu_!\cG)_\leq$ the subsheaf $p_{2,!}[(p_1^{-1}\cG)_\leq]$ of $p_{2,!}p_1^{-1}\cG=\mu^{-1}\mu_!\cG$.

For any $\varphi_y\in\mu^{-1}(y)$, one has $(\mu_!\cG)_{\leq\varphi_y}=\bigoplus_{\psi_y\leqy\varphi_y}\cG_{\psi_y}$, and $(\mu_!\cG)_\leq$ is a pre-$\ccI$-filtration on $\mu_!\cG$, that we call the \index{I-filtration@$\ccI$-filtration!graded --}\emph{graded~$\ccI$-filtration on $\mu_!\cG$}.

Note that, if~$\ccI$ satisfies the exhaustivity property (Definition \ref{def:exhaustI}) and if $\mu$ is proper on $\Supp\cG$ then, as the sum in \eqref{eq:trivIfilt} is finite, any graded~$\ccI$-filtration is exhaustive.
\end{exemple}

\begin{definitio}[$\ccI$-filtrations]\label{def:Ifilt}
Let $\cF_\leq$ be a pre-$\ccI$-filtration of $\cF$. We say that it is a~\index{I-filtration@$\ccI$-filtration}$\ccI$-filtration of $\cF$ if, \emph{locally on~$Y$}, $\cF\simeq\mu_!\cG$ and the inclusion $\cF_\leq\subset\mu^{-1}(\cF)$ is isomorphic to the corresponding inclusion for the graded~$\ccI$-filtration on $\mu_!\cG$ for some sheaf $\cG$ (a~priori only locally (w.r.t\ptbl $Y$) defined on~$\ccIet$).
\end{definitio}

The category of~$\ccI$-filtered sheaves is the full subcategory of that of pre-$\ccI$-filtered sheaves whose objects are~sheaves with a $\ccI$-filtration. In other words, a morphism between sheaves underlying~$\ccI$-filtered sheaves is a morphism in the category of~$\ccI$\nobreakdash-filtered sheaves iff it is compatible with order. Note that it is \emph{not} required that~$\mu^{-1}f$ is locally decomposed in the local graded models $\mu_!\cG,\mu_!\cG'$ (\ie $\mu^{-1}f$ is ``filtered'', but possibly not ``graded'').

\begin{exemple}[Grading a graded~$\ccI$-filtration]
For the graded~$\ccI$-filtration on $\mu_!\cG$, we have $(\mu_!\cG)_{<\varphi_y}=\bigoplus_{\psi_y\ley\varphi_y}\cG_{\psi_y}$ and $\gr(\mu_!\cG)=\cG$, so $\cG$ can be recovered from $(\mu_!\cG)_\leq$.
\end{exemple}

Therefore, a~$\ccI$-filtration $(\cF,\cF_\bbullet)$ is, locally (on $Y$), isomorphic to the graded~$\ccI$\nobreakdash-filtra\-tion on $\mu_!\gr\cF$. Notice that the model graded~$\ccI$-filtration $\mu_!\gr\cF$ is now defined globally on~$Y$. In particular, for any $y\in Y$ and any $\varphi_y\in\ccI_y$, we have
\begin{equation}\label{eq:decI}
\begin{split}
\cF_{<\varphi_y,y}&\simeq\bigoplus_{\substack{\psi_y\in\ccI_y\\\psi_y\ley\varphi_y}}\gr_{\psi_y}\cF_y,\\
\cF_{\leq\varphi_y,y}&\simeq\bigoplus_{\substack{\psi_y\in\ccI_y\\\psi_y\leqy\varphi_y}}\gr_{\psi_y}\cF_y=\cF_{<\varphi_y,y}\oplus\gr_{\varphi_y}\cF_y,\\
\cF_y&\simeq\bigoplus_{\psi_y\in\ccI_y}\gr_{\psi_y}\cF_y,
\end{split}
\end{equation}
in a way compatible with the inclusion $\cF_{<\varphi_y,y}\subset\cF_{\leq\varphi_y,y}\subset\cF_y$. Recall also that a morphism may not be graded with respect to such a gradation.

\begin{remarque}\label{rem:exhaust}
If~$\ccI$ satisfies the exhaustivity property (Definition \ref{def:exhaustI}) and if $\mu$ is proper on $\Supp\gr\cF$, then the~$\ccI$-filtration is exhaustive.
\end{remarque}

\begin{lemme}[Stability by pull-back]\label{lem:stabpullback}
In the setting of Definition \ref{def:pullbackpreI}, if $\cF_\leq$ is a $\ccI$-filtration of $\cF$, then $f^+(\cF,\cF_\leq)$ is $\ccI'$-filtered sheaf on $Y'$.
\end{lemme}

\begin{proof}
Since the assertion is local, it is enough to start from a graded $\ccI$-filtration. Let us set $\cF'_\leq=f^+\cF_\leq$ and $\cF_\leq=(\mu_!\cG)_\leq$. Then, for $y'\in Y'$ and $y=f(y')$, we have
\[
\cF'_{\leqyp\varphi'_{y'}}=\sum_{\psi_y\leqyp\varphi'_{y'}}\cF_{\leq\psi_y}=\sum_{\psi_y\leqyp\varphi'_{y'}}\bigoplus_{\varphi_y\leqy\psi_y}\cG_{\varphi_y}=\bigoplus_{\varphi_y\leqyp\varphi'_{y'}}\cG_{\varphi_y}.\qedhere
\]
\end{proof}

\begin{remarque}[co-$\ccI$-filtration]\label{rem:coIfilt}
There is a dual notion of~$\ccI$-filtration, that one can call a \emphb{co-$\ccI$-filtration}. It consists of a pair $(\cF,\cF_<)$, where $\cF_<$ is a pre-$\ccI$-filtration, and an injective morphism $\cF_<\hto\mu^{-1}\cF$, which is locally (on $Y$) isomorphic to a pair $(\mu_!\cG,(\mu_!\cG)_<)$. From a~$\ccI$-filtration $\cF_\leq$, one defines a co-$\ccI$-filtration by using $\cF_<$ of Definition \ref{def:grading}.

Conversely, assume that $\ccI$ satisfies the following property:
\bgroup\numstareq
\begin{equation}\label{eq:coIfilt}
\forall y\in Y,\;\forall\varphi_y,\psi_y,\quad\psi_y\leqy\varphi_y\ssi\forall\eta_y,\;(\varphi_y\ley\eta_y\Rightarrow\psi_y\ley\eta_y).
\end{equation}
\egroup
Given a co-$\ccI$-filtration $\cF_<$ of $\cF$, we define $\cF_\leq$ by the formula
\[
\cF_{\leq\varphi_y}=\bigcap_{\substack{\eta_y\in\ccI_y\\\varphi_y\ley\eta_y}}\cF_{<\eta_y}.
\]
Then $\cF_\leq$ is a~$\ccI$-filtration of $\cF$ (this is checked on $\mu_!\cG$).
\end{remarque}

\subsubsection*{Classification of~$\ccI$-filtered sheaves}
Let $(\cF,\cF_\bbullet)$ and $(\cF',\cF'_\bbullet)$ be two pre-$\ccI$\nobreakdash-filtered sheaves. Then $\cHom_{\kk}(\cF,\cF')$ is equipped with a natural pre-$\ccI$-filtration: for any $\varphi\in\Gamma(U,\ccI$), $\cHom_{\kk}(\cF,\cF')_{\leq\varphi}$ is the subsheaf of $\cHom_{\kk}(\cF,\cF')_{|U}$ such that, for any open set $V\subset\nobreak U$, $f\in\Hom_{\kk}(\cF_{|V},\cF'_{|V})$ belongs to $\Hom_{\kk}(\cF_{|V},\cF'_{|V})_{\leq\varphi}$ iff $f(\cF_{\leq\eta|V})\subset\cF'_{\leq\eta+\varphi|V}$ for any $\eta\in\Gamma(V,\ccI)$. In case $(\cF,\cF_\bbullet)$ and $(\cF',\cF'_\bbullet)$ are~$\ccI$\nobreakdash-filtered sheaves, then so is $\cHom_{\kk}(\cF,\cF')$ since the local decomposition of $\cF,\cF'$ induces a local decomposition of $\cHom_{\kk}(\cF,\cF')$.

We denote by \index{$AUTNEG$@$\cAut_{\kk}^{<0}(\cF)$}$\cAut_{\kk}^{<0}(\cF)$ the subsheaf $\id+\cHom_{\kk}(\cF,\cF)_{<0}$ of $\cHom_{\kk}(\cF,\cF)_{\leq0}$.

\begin{proposition}\label{prop:classifhausdorff}
The set of isomorphism classes of~$\ccI$-filtration $(\cF,\cF_\bbullet)$ with $\gr\cF=\cG$ is equal to $H^1\big(Y,\cAut_{\kk}^{<0}(\mu_!\cG)\big)$, where $\mu_!\cG$ is equipped with the graded $\ccI$\nobreakdash-filtration.
\end{proposition}

\begin{proof}
This is standard (see \eg the proof of Prop\ptbl II.1.2.2, p\ptbl112 in \cite{B-V89}).
\end{proof}

\subsection{$\ccI$-filtered local systems (the Hausdorff case)}\label{subsec:Ifillocsys}\index{I-filtered@$\ccI$-filtered local system}\index{local system!I-filtered@$\ccI$-filtered}
The case where $\cF$ is a \emph{locally constant sheaf of finite dimensional $\kk$\nobreakdash-vector spaces on~$Y$} (that we call a local system) will be the most important for us. Let us assume that $(\cF,\cF_\bbullet)$ is a $\ccI$-filtered locally constant sheaf. As locally constant sheaves are stable by direct summand, it follows from the local isomorphism $\cF\simeq\mu_!\gr\cF$ that $\mu$ is proper on $\Sigma=\Supp\gr\cF$, that~$\Sigma$ is a finite covering of~$Y$, and $\gr\cF$ is a locally constant sheaf on its support $\Sigma$. Moreover, $\mu_*\gr\cF=\mu_!\gr\cF$ is a local system on~$Y$ locally isomorphic to~$\cF$. We call~$\Sigma$ a \index{I-covering@$\ccI$-covering}\emph{$\ccI$-covering} of $Y$.

We note that $\mu:\Sigma\to Y$, being a finite covering, is a local homeomorphism and therefore defines a subsheaf $\Sigma^\sh$ of $\ccI$. This subsheaf is a sheaf of ordered sets (in general not abelian groups) with the ordered induced from that of $\ccI$. Restricting $\cF_\leq$ to $\Sigma$ defines then a $\Sigma^\sh$-filtration of $\cF$. Due to the local grading property of $\cF_\leq$, it is not difficult to check that $\cF_\leq$ is determined by $\cF_{\leq|\Sigma}$ and that both $\cF_\leq$ and $\cF_{\leq|\Sigma}$ have the same graded sheaves. In other words, if we know $\cF_{\leq\varphi_y}\subset\cF_y$ for any $\varphi_y\in\Sigma_y$, we can reconstruct $\cF_{\leq\psi_y}\subset\cF_y$ for any $\psi_y\in\ccI_y$ by setting $\cF_{\leq\psi_y}=\sum_{\varphi_y\leq\psi_y,\,\varphi_y\in\Sigma_y}\cF_{\leq\varphi_y}$, where the sum is taken in $\cF_y$.

Let us make precise Lemma \ref{lem:stabpullback} in the case where $\cF$ is a local system. We consider the setting of Definition \ref{def:pullbackpreI}. Firstly, the pre-$\wt\ccI$-filtration $\wt f^{-1}\cF_\leq$ of $f^{-1}\cF$ is a $\wt\ccI$-filtration and we have $\gr f^{-1}\cF=\wt f^{-1}\gr\cF$, whose support \index{$SZIGMAWT$@$\wt\Sigma$}$\wt\Sigma$ is nothing but $\wt f^{-1}\Sigma$, and is a finite covering of $Y'$. Since $q_f:\wt\ccIet\to\ccIpet$ is étale, being induced by a morphism of sheaves, the image $\Sigma'=q_f(\wt\Sigma)$ is a finite covering of $Y'$ and $q_f:\wt\Sigma\to\Sigma'$ is also a finite covering. The proof of Lemma \ref{lem:stabpullback} shows that $\gr\cF'=q_{f,*}\wt f^{-1}\gr\cF$.

\subsection{$\ccI$-filtered local systems (the stratified Hausdorff case)}\label{subsec:Ifilstrat}\index{I-filtered@$\ccI$-filtered local system}\index{local system!I-filtered@$\ccI$-filtered}
We will now extend the notion of $\ccI$-filtered local system in the case $\ccI$ satisfies a property weaker than Hausdorff, that we call the \index{Hausdorff! stratified}\emphb{stratified Hausdorff property}. We restrict to local systems, as this will be the only case of interest for us, and we will thus avoid general definitions as for the Hausdorff case.

Let $\cY=(Y_\alpha)_{\alpha\in A}$ be a \emphb{stratification} of~$Y$ (that is, a locally finite partition of~$Y$ into locally closed sets, called strata, such that the closure of every stratum is a union of strata). For $\alpha,\beta\in A$, we will denote by $\alpha\preceq\beta$ the relation $Y_\alpha\subset\ov Y_\beta$. In the following, we will always assume that the stratification satisfies the following property:

\enum{.07}{.91}{\refstepcounter{equation}\label{eq:propstrat}\eqref{eq:propstrat}}
{For each $\alpha\in A$, each point $y\in Y_\alpha$ has a fundamental system of open neighbourhoods $\nb(y)$ such that, for any $\beta\succeq\alpha$, $\nb(y)\cap\ov Y_\beta$ is contractible.}

\begin{definitio}[The stratified Hausdorff property]\label{def:stratHausdorff}
Let~$\ccI$ be a sheaf on $Y$ (in the setting of \S\ref{subsec:espaceetales}). We say that~$\ccI$ satisfies the \emph{stratified Hausdorff property} with respect to $\cY$ if for any $Y_\alpha\in\cY$, $\mu^{-1}(Y_\alpha)$ is Hausdorff, that is, the sheaf-theoretic restriction of~$\ccI$ to every $Y_\alpha$ satisfies the Hausdorff property.
\end{definitio}

\pagebreak[2]\skpt
\begin{exemples}\label{exem:separebis}\ligne
\begin{enumerate}
\item\label{exem:separe5}
If~$\ccI$ is a constructible sheaf on~$Y$ with respect to a stratification $\cY$ for which the strata are locally connected, then~$\ccI$ is Hausdorff with respect to $\cY$.
\item\label{exem:separe6}
Let~$X$ be a complex manifold, let~$D$ be a reduced divisor in~$X$ and let $\cY$ be a stratification of~$D$ by complex locally closed submanifolds such that $(X\moins D,\cY)$ is a Whitney stratification of~$X$. Then $\ccI=\cO_X(*D)/\cO_X$ has the Hausdorff property with respect to $\cY$. Indeed, if $x$ belongs to a connected stratum $Y_\alpha$ and if a germ $\varphi_x\in \cO_{X,x}(*D)$, with representative $\varphi$, is such that $\varphi_y$ is holomorphic for $y\in Y_\alpha$ close enough to $x$, then $\varphi$ is holomorphic on each maximal dimensional connected stratum of~$D$ which contains $Y_\alpha$ in its closure, so, by the local product property, $\varphi$ is holomorphic on the smooth part of~$D$ in the neighbourhood of $x$. By Hartogs, $\varphi$ is holomorphic near $x$.
\item\label{exem:separe6b}
The same argument holds if~$X$ is only assumed to be a normal complex space and~$D$ is a locally principal divisor in~$X$.
\end{enumerate}
\end{exemples}

We have seen in \S\ref{subsec:Ifillocsys} that the support of $\gr\cF$ is a finite covering of $Y$ contained in $\ccIet$. We now introduce the notion of stratified $\ccI$-covering (with respect to $\cY$).

\begin{definitio}[stratified $\ccI$-covering]\label{def:Sigma}\index{stratified I-covering@stratified $\ccI$-covering}\index{I-covering@$\ccI$-covering!-- stratified}
Let $\Sigma\subset\ccIet$ be a closed subset. We will say that $\mu_{|\Sigma}:\Sigma\to Y$ is a \emph{stratified $\ccI$-covering} of $Y$ if
\begin{enumerate}
\item
$\mu_{|\Sigma}:\Sigma\to Y$ is a local homeomorphism, that is, $\Sigma$ is the étale space of a subsheaf $\Sigma^\sh$ of $\ccI$,
\item
for each $\alpha\in A$, setting $\Sigma_\alpha=\Sigma\cap\mu^{-1}(Y_\alpha)$, the map $\mu_{|\Sigma_\alpha}:\Sigma_\alpha\to Y_\alpha$ is a finite covering,
\end{enumerate}
\end{definitio}

Let us explain this definition. The first property implies that $\mu_{|\Sigma_\alpha}$ is a local homeomorphism on $Y_\alpha$. Since $\ccIet_{|Y_\alpha}$ is Hausdorff, so is $\Sigma_\alpha$, and the second property is then equivalent to $\mu_{|\Sigma_\alpha}$ being proper.

What does $\Sigma $ look like near a point of $Y_\alpha$? For each $\alpha\in A$, each $y\in Y_\alpha$ has an open neighbourhood $\nb(y)\subset Y$ such that, denoting $\varphi_y^{(i)}$ the points in $\mu_{|\Sigma_\alpha}^{-1}(y)$, the germs $\varphi_y^{(i)}$ are induced by sections $\varphi^{(i)}\in\Gamma(\nb(y),\ccI)$. Then we have $\Sigma\cap\mu^{-1}(\nb(y))=\bigcup_i\varphi^{(i)}(\nb(y))$. We note however that, while the germs $\varphi^{(i)}_y$ (or more generally the germs $\varphi^{(i)}_{y'}$ for $y'\in\nb(y)\cap Y_\alpha$) are distinct, such may not be the case when $y'\not\in Y_\alpha$ and, for $\beta\succeq\alpha$, $\Sigma_\beta\cap\mu^{-1}(\nb(y))$ may have less connected components than $\Sigma_\alpha\cap\mu^{-1}(\nb(y))$.

\begin{definitio}[$\ccI$-filtration (the stratified Hausdorff case)]\label{def:Ifiltstrat}
Let $\ccI$ be stratified Hausdorff with respect to $\cY$, let $\cF$ be a locally constant sheaf of finite dimensional $\kk$-vector spaces on $Y$ and let $\cF_\leq$ be a pre-$\ccI$-filtration of $\cF$. We say that $\cF_\leq$ is a $\ccI$-filtration of $\cF$ if
\begin{enumerate}
\item\label{def:Ifiltstrat1}
for each $\alpha\in A$, the restriction $\cF_{\leq|\mu^{-1}(Y_\alpha)}$ is a $\ccI_{|\mu^{-1}(Y_\alpha)}$-filtration of $\cF_{|Y_\alpha}$ in the sense of Definition \ref{def:Ifilt}, with associated covering $\Sigma_\alpha\subset\ccIet_{|\mu^{-1}(Y_\alpha)}$ and associated local system $\cG_\alpha=\gr\cF_{|\mu^{-1}(Y_\alpha)}$,
\item\label{def:Ifiltstrat2}
the set \index{$SZIGMA$@$\Sigma$, $\Sigma_\alpha$}$\Sigma=\bigcup_\alpha\Sigma_\alpha$ is a stratified $\ccI$-covering of $Y$,
\item\label{def:Ifiltstrat3}
for each $\alpha\in A$ and each $y\in Y_\alpha$, let us set $\mu^{-1}(y)\cap\Sigma_\alpha=\{\varphi_y^{(i)}\mid i\in I_\alpha\}$, and let $\nb(y)$ be a small open neighbourhood on which all $\varphi^{(i)}$ are defined as sections of~$\ccI$; for each $\beta\succeq\alpha$, let $\Sigma_\beta(y,j)$ be the connected components of $\mu^{-1}(\nb(y))\cap \Sigma_\beta=\bigcup_i\varphi^{(i)}(\nb(y)\cap Y_\beta)$; then, for each $j$, $\cG_{\beta|\Sigma_\beta(y,j)}$ is a trivial local system and
\[
\rk\cG_{\beta|\Sigma_\beta(y,j)}=\sum_{i\mid\varphi^{(i)}(\nb(y)\cap Y_\beta)=\Sigma_\beta(y,j)}\rk\cG_{\alpha|\varphi^{(i)}(\nb(y)\cap Y_\alpha)}.
\]
\end{enumerate}
\end{definitio}

The last condition is a gluing condition near a stratum $Y_\alpha$. Since  the stratified space satisfies \eqref{eq:propstrat}, each $y\in Y_\alpha$ has a fundamental system of simply connected neighbourhood such that $\nb(y)\cap Y_\alpha$ is also simply connected. Then the local system~$\cG_\alpha$, restricted to $\varphi^{(i)}(\nb(y)\cap Y_\alpha)$ extends in a unique way to a local system $\cG_\alpha^{(i)}$ on $\varphi^{(i)}(\nb(y))$ (both local systems are trivial, according to our assumptions on $\nb(y)$). The second condition means that $\cG_{\beta|\Sigma_\beta(y,j)}$ is isomorphic to the direct sum of the $\cG_\alpha^{(i)}$ restricted to $\nb(y)\cap Y_\beta$.

As in the Hausdorff case, we also note that the restriction of $\cF_\leq$ to $\Sigma$ defines a $\Sigma^\sh$-filtration of $\cF$, and that $\cF_\leq$ is determined by $\cF_{\leq|\Sigma}$.

\begin{lemme}[Stability by pull-back]\label{lem:stabpullbackstrat}
Let $f:(Y',\cY')\to(Y,\cY)$ be a stratified continuous map, let $\ccI$ be stratified Hausdorff with respect to $\cY$ and let $(\cF,\cF_\leq)$ be a $\ccI$-filtered local system. Assume that we are in the setting of Definition \ref{def:pullbackpreI}. Then $f^+(\cF,\cF_\leq)$ is $\ccI'$-filtered local system with associated $\ccI'$ stratified covering $\Sigma'$ given by $q_f(\wt f^{-1}(\Sigma))$.
\end{lemme}

\begin{proof}
The result on each stratum follows from Lemma \ref{lem:stabpullback} and the details given in \S\ref{subsec:Ifillocsys}. The point is to check the gluing properties \ref{def:Ifiltstrat}\eqref{def:Ifiltstrat2} and \ref{def:Ifiltstrat}\eqref{def:Ifiltstrat3}. Firstly, these properties are easily checked for $\wt f^{-1}\cF_\leq$ as a $\wt\ccI$-filtration, and for $\wt\Sigma=\wt f^{-1}\Sigma$. One then shows that $q_f:\wt\Sigma\to\Sigma'\defin q_f(\Sigma)$ is a stratified covering with respect to the stratification $(\Sigma'_{\alpha'})_{\alpha'\in A'}$. The desired properties follow.
\end{proof}

\subsection{Comments}
The first definition of a $\ccI$-filtered local system appears in \cite{Deligne78b} in the following way (taking the notation of the present notes):

Given the local system $\ccI$ of multi-valued differential forms on the circle~$S^1$, equipped with the order as defined in Example \ref{exem:Stokes}, a $\ccI$-filtration on a local system $\cL$ on~$S^1$ consists of the data of a family of subsheaves $\cL^\varphi$ of $\cL$ indexed by $\ccI$ (be careful that~$\ccI$ is only a local system, but this is enough) satisfying the local grading property like \eqref{eq:decI}.

This definition is taken up in \cite[\S4, Def\ptbl4.1]{Malgrange83b}, and later in \cite[Def\ptbl2.1, p\ptbl57]{Malgrange91} in a similar way, but B\ptbl Malgrange adds between brackets:

``Be careful that the $\cL^\varphi$ are not subsheaves in the usual sense, because they are indexed by a local system and not a set.''

Later, the definition is once more taken up in \cite[p\ptbl74--77]{B-V89}, starting by defining a $\ccI$-graded local system (this is similar to Example \ref{exem:trivIfilt}) and then a $\ccI$-filtered local system by gluing local $\ccI$-graded local systems in a way which preserves the filtration (but possibly not the gradation).

In this \chaptersname we have tried to give a precise intrinsic answer to the question:

\centerline{what \emph{is} a $\ccI$-filtered local system?}

\noindent in order to use it in higher dimensions.

Another approach is due to T\ptbl Mochizuki in \cite[\S2.1]{Mochizuki08}, starting with a family of filtrations indexed by ordered sets, and adding compatibility conditions (corresponding to compatibility with respect to restriction morphisms in the sheaf-theoretic approach of the present notes).

\part{Dimension one}
\chapter{Stokes-filtered local systems in~dimension~one}\label{chap:Stokesone}

\begin{sommaire}
We consider Stokes filtrations on local systems on~$S^1$. We review some of the definitions of the previous \chaptersname in this case and make explicit the supplementary properties coming from this particular case. This \chaptersname can be read independently of \Chaptersname \ref{chap:Ifil}.
\end{sommaire}

\subsection{Introduction}
The notion of a (pre-)Stokes filtration is a special case of the notion of a (pre-)$\ccI$-filtration defined in \Chaptersname \ref{chap:Ifil} with a suitable sheaf $\ccI$ on the topological space $Y=S^1$. We have chosen to present this notion independently of the general results of the previous \chaptername, since many properties are simpler to explain in this case (\cf Proposition \ref{prop:stokeswithout}). Nevertheless, we make precise the relation with the previous \chaptersname when we introduce new definitions. We moreover start with the non-ramified Stokes filtrations, to make easier the manipulation of such objects, and we also call it a $\ccI_1$-filtration, in accordance with the previous \chaptername. The (possibly ramified) Stokes filtrations are introduced in \S\ref{subsec:Stokeswithramif}, where we define the sheaf $\ccI$ with its order. They correspond to the $\ccI$-filtrations of the previous \chaptername.

We also make precise the relation with the approach by Stokes data in the case of Stokes filtrations of simple exponential type.

References are \cite{Deligne78b}, \cite{Malgrange83b}, \cite{B-V89} and \cite[Chap\ptbl IV]{Malgrange91}.

\subsection{Non-ramified Stokes-filtered local systems}\label{subsec:Stokesnon-ram}
Let $\kk$ be a field. In this section, we consider local systems of finite dimensional $\kk$-vector spaces on~$S^1$. Recall (\cf Example \ref{exem:Stokes}) that we consider $S^1$ equipped with the constant sheaf $\ccI_1$ with fibre \index{$P$@$\ccP$}$\ccP=\CC\lpb x\rpb/\CC\{x\}$ consisting of polar parts of Laurent series, and the order depends on the point $e^{i\theta}=x/|x|\in S^1$ as follows. Let $\eta\in\ccP$ and let us set $\eta=u_n(x)x^{-n}$ with $n\geq1$ and $u_n(0)\neq0$ if $\eta\neq0$. Then
\[\index{$AAAORDth$@$\leqtheta$, $\letheta$}\tag{\ref{eq:orderone}}
\eta\leqtheta0\ssi \eta=0\text{ or }\arg u_n(0)-n\theta\in(\pi/2,3\pi/2)\mod2\pi,
\]
and $\eta\letheta0\ssi(\eta\leqtheta0\text{ and }\eta\neq0)$ (\cf \eqref{eq:orderonestrict}). The order is supposed to be compatible with addition, namely, $\varphi\leqtheta\psi\ssi\varphi-\psi\leqtheta0$ and similarly for $\letheta$. Let us rephrase Definition \ref{def:pre-I-fil} in this setting.

\begin{definitio}\label{def:preStokesnonramif}
A \emphb{non-ramified pre-Stokes filtration} on a \index{local system!Stokes-filtered}local system~$\cL$ of finite dimensional $\kk$\nobreakdash-vector spaces on~$S^1$ consists of the data of a family of subsheaves \index{$LLEQETC$@$\cL_{\leq\varphi}$, $\cL_{<\varphi}$, $\gr_\varphi\cL$}$\cL_{\leq\varphi}$ indexed by $\ccP$ such that, for any $\theta\in S^1$, $\varphi\leqtheta\psi\implique\cL_{\leq\varphi,\theta}\subset\cL_{\leq\psi,\theta}$.
\end{definitio}

Let us set, for any $\varphi\in\ccP$ and any $\theta\in S^1$,
\begin{equation}\label{eq:onto<}
\cL_{<\varphi,\theta}=\sum_{\psi\letheta\varphi}\cL_{\leq\psi,\theta}.
\end{equation}
This defines a subsheaf $\cL_{<\varphi}$ of $\cL_{\leq\varphi}$, and we set $\gr_\varphi\cL=\cL_{\leq\varphi}/\cL_{<\varphi}$. Note that The étale space $\ccIet_1$ is Hausdorff (\cf Example \ref{exem:separe}\eqref{exem:separe1}) and the previous definition is in accordance with Definition \ref{def:grading}.

\begin{notation}\label{nota:beta} We rephrase here Notation \ref{not:Ypsiphi} in the present setting. Let $\varphi,\psi\in\ccP$. Recall that we denote by \index{$S1PSIPHI$@$S^1_{\psi\leq\varphi}$, $S^1_{\psi<\varphi}$}$S^1_{\psi\leq\varphi}\subset S^1$ the subset of~$S^1$ consisting of the $\theta$ for which $\psi\leqtheta\varphi$. Similarly, $S^1_{\psi<\varphi}\subset S^1$ is the subset of~$S^1$ consisting of the $\theta$ for which $\psi\letheta\varphi$. Both subsets are a finite union of open intervals. They are equal if $\varphi\neq\psi$. Otherwise, $S^1_{\varphi\leq\varphi}=S^1$ and $S^1_{\varphi<\varphi}=\emptyset$.

Given a sheaf $\cF$ on~$S^1$, we will denote by \index{$BETAPSIPHI$@$\beta_{\psi\leq\varphi}$, $\beta_{\psi<\varphi}$}$\beta_{\psi\leq\varphi}\cF$ the sheaf obtained by restricting~$\cF$ to the open set $S^1_{\psi\leq\varphi}$ and extending it by~$0$ as a sheaf on~$S^1$ (for any open set $Z\subset S^1$, this operation is denoted $\cF_Z$ in \cite{K-S90}). A similar definition holds for $\beta_{\psi<\varphi}\cF$.
\end{notation}

\begin{definitio}[(Graded) Stokes filtration]\label{def:gradedStokes}
Given a finite set $\Phi\subset\ccP$, a \index{local system!Stokes-graded}\index{Stokes filtration!graded} Stokes-graded local system with~\index{$FPHI$@$\Phi$ (exponential factors)}$\Phi$ as set of \emphb{exponential factors} consists of the data of local systems (that we denote by) $\gr_\varphi\cL$ on~$S^1$ ($\varphi\in \Phi$). Then the graded non-ramified Stokes filtration on $\gr\cL\defin\bigoplus_{\psi\in \Phi}\gr_\psi\cL$ is given by
\[
(\gr\cL)_{\leq\varphi}=\bigoplus_{\psi\in \Phi}\beta_{\psi\leq\varphi}\gr_\psi\cL.
\]
We then also have
\[
(\gr\cL)_{<\varphi}=\bigoplus_{\psi\in \Phi}\beta_{\psi<\varphi}\gr_\psi\cL.
\]

A \index{Stokes filtration!non-ramified}\emph{non-ramified $\kk$\nobreakdash-Stokes filtration} on~$\cL$ is a pre-Stokes filtration which is locally on~$S^1$ isomorphic to a graded Stokes filtration. It is denoted by \index{$LBULLET$@$\cL_\bbullet$ (Stokes filtration)}$\cL_\bbullet$.
\end{definitio}

For a \index{Stokes-filtered local@Stokes-filtered local system!non-ramified}Stokes-filtered local system $(\cL,\cL_\bbullet)$, each sheaf $\gr_\varphi\cL$ is a (possibly zero) local system on~$S^1$. By definition, for every $\varphi$ and every $\theta_o\in S^1$, we have on some neighbourhood $\nb(\theta_o)$ of $\theta_o$,
\begin{equation}\label{eq:L<varphi}
\begin{split}
\cL_{<\varphi|\nb(\theta_o)}&\simeq\bigoplus_{\psi\in \Phi}\beta_{\psi<\varphi}\gr_\psi\cL_{|\nb(\theta_o)},\\
\cL_{\leq\varphi|\nb(\theta_o)}&\simeq\cL_{<\varphi|\nb(\theta_o)}\oplus\gr_\varphi\cL_{|\nb(\theta_o)}=\bigoplus_{\psi\in \Phi}\beta_{\psi\leq\varphi}\gr_\psi\cL_{|\nb(\theta_o)},\\
\cL_{|\nb(\theta_o}&\simeq\bigoplus_{\psi\in \Phi}\gr_\psi\cL_{|\nb(\theta_o)}
\end{split}
\end{equation}
in a way compatible with the natural inclusions.

\begin{exo}\label{exo:graded}
Show that the category of Stokes-filtered local systems has direct sums, and that any Stokes-graded local system is the direct sum of Stokes-graded local systems, each of which has exactly one exponential factor.
\end{exo}

One can make more explicit the definition of a non-ramified Stokes-filtered local system.

\begin{proposition}\label{prop:stokeswithout}
Giving a non-ramified Stokes-filtered local system $(\cL,\cL_\bbullet)$ is equivalent to giving, for each $\varphi\in\ccP$, a $\RR$-constructible subsheaf $\cL_{\leq\varphi}\subset\cL$ subject to the following conditions:
\begin{enumerate}
\item
for any $\theta\in S^1$, the germs $\cL_{\leq\varphi,\theta}$ form an exhaustive increasing filtration of~$\cL_\theta$;
\item\label{prop:stokeswithout2}
defining $\cL_{<\varphi}$, and therefore $\gr_\varphi\cL$, from the family $\cL_{\leq\psi}$ as in \eqref{eq:onto<}, the sheaf $\gr_\varphi\cL$ is a local system of finite dimensional $\kk$\nobreakdash-vector spaces on~$S^1$;
\item\label{prop:stokeswithout4}
for any $\theta\in S^1$ and any $\varphi\in\ccP$, $\dim\cL_{\leq\varphi,\theta}=\sum_{\psi\leqtheta\varphi}\dim\gr_\psi\cL_\theta$.
\end{enumerate}
\end{proposition}

We note that when \ref{prop:stokeswithout}\eqref{prop:stokeswithout2} is satisfied, \ref{prop:stokeswithout}\eqref{prop:stokeswithout4} is equivalent to
\begin{enumerate}
\item[(\ref{prop:stokeswithout4}$'$)]
For any $\theta\in S^1$ and any $\varphi\in\ccP$, $\dim\cL_{<\varphi,\theta}=\sum_{\psi\letheta\varphi}\dim\gr_\psi\cL_\theta$.
\end{enumerate}

\begin{proof}[\proofname\ of Proposition \ref{prop:stokeswithout}]
The point is to get the local gradedness property from the dimension property of \ref{prop:stokeswithout}\eqref{prop:stokeswithout4}. Since the local filtrations are exhaustive, the dimension property implies that the local systems $\cL$ and $\bigoplus_\varphi\gr_\varphi\cL$ are locally isomorphic, hence for each $\theta_o$, there exists a finite family $\Phi_{\theta_o}\subset\ccP$ such that $\gr_\varphi\cL_{\theta_o}\neq0\implique\varphi\in\nobreak\Phi_{\theta_o}$. Since $\gr_\varphi\cL$ is a local system, it is zero if and only if it is zero near some~$\theta_o$. We conclude that the set $\Phi_{\theta_o}\subset\ccP$ is independent of $\theta_o$, and we simply denote it by~$\Phi$. We thus have $\gr_\varphi\cL\neq0\implique\varphi\in\Phi$.

\begin{lemme}\label{lem:H1vanishing}
Let $\cF$ be a $\RR$-constructible sheaf of $\kk$\nobreakdash-vector spaces on~$S^1$. For any $\theta_o\in S^1$, let $I$ be an open interval containing $\theta_o$ such that $\cF_{|I\moins\{\theta_o\}}$ is a local system of finite dimensional $\kk$\nobreakdash-vector spaces. Then $H^1(I,\cF)=0$.
\end{lemme}

\begin{proof}
Let $\iota:I\moins\{\theta_o\}\hto I$ be the inclusion. We have an exact sequence $0\to\iota_!\iota^{-1}\cF\to\cF\to\cG\to0$, where $\cG$ is supported at $\theta_o$. It is therefore enough to prove the result for $\iota_!\iota^{-1}\cF$. This reduces to the property that, if $i:(a,b)\hto(a,b]$ is the inclusion (with $a,b\in\RR$, $a<b$), then $H^1((a,b],i_!\kk)=0$, which is clear by Poincaré duality\index{duality!Poincaré}.
\end{proof}

Let us fix $\theta_o\in S^1$. Since $\cL_{<\psi}$ is $\RR$-constructible for any $\psi$, there exists an open interval $\nb(\theta_o)$ of~$S^1$ containing $\theta_o$ such that, for any $\psi\in \Phi$, $\cL_{<\psi}$ is a local system on $\nb(\theta_o)\moins\{\theta_o\}$. Then, for any such $\psi$, $H^1(\nb(\theta_o),\cL_{<\psi})=0$, according to the previous lemma and, as $\gr_\psi\cL$ is a constant local system on $\nb(\theta_o)$, we can lift a basis of global sections of $\gr_\psi\cL_{|\nb(\theta_o)}$ as a family of sections of $\cL_{\leq\psi,|\nb(\theta_o)}$. This defines a morphism $\bigoplus_\psi\gr_\psi\cL_{|\nb(\theta_o)}\to\cL_{|\nb(\theta_o)}$ sending $\bigoplus_{\psi\leqtheta\varphi}\gr_\psi\cL_{\theta}$ to $\cL_{\leq\varphi,\theta}$ for any $\theta\in\nb(\theta_o)$ and any~$\varphi$. Let us show, by induction on $\#\{\psi\in \Phi\mid\psi\leqtheta\varphi\}$, that it sends $\bigoplus_{\psi\leqtheta\varphi}\gr_\psi\cL_{\theta}$ \emph{onto} $\cL_{\leq\varphi,\theta}$ for any $\theta\in\nb(\theta_o)$ and any~$\varphi$: indeed, the assertion is clear by the dimension property if this number is zero; by the inductive assumption and according to \eqref{eq:onto<}, it sends $\bigoplus_{\psi\letheta\varphi}\gr_\psi\cL_{\theta}$ onto $\cL_{<\varphi,\theta}$ for any $\theta\in\nobreak\nb(\theta_o)$; since $\cL_{\leq\varphi}=\cL_{<\varphi}+\text{image}\gr_\varphi\cL$ in $\cL$, the assertion follows. As both spaces $\bigoplus_{\psi\leqtheta\varphi}\gr_\psi\cL_{\theta}$ and $\cL_{\leq\varphi,\theta}$ have the same dimension, due to \ref{prop:stokeswithout}(\ref{prop:stokeswithout4}), this morphism is an isomorphism.
\end{proof}

The finite subset $\Phi\subset\ccP$ such that $\gr_\varphi\cL\neq0\implique\varphi\in\Phi$ is called the \emphb{set of exponential factors} of the non-ramified Stokes filtration. The following proposition is easily checked, showing more precisely exhaustivity.

\begin{proposition}\label{prop:Sfini}
Let $\cL_\bbullet$ be a non-ramified $\kk$\nobreakdash-Stokes filtration on~$\cL$. Then, for any $\theta\in S^1$, and any $\varphi\in\ccP$,
\begin{itemize}
\item
if $\varphi\letheta \Phi$, then $\cL_{\leq\varphi,\theta}=0$,
\item
if $\Phi\letheta\varphi$, then $\cL_{<\varphi,\theta}=\cL_{\leq\varphi,\theta}=\cL_\theta$.\qed
\end{itemize}
\end{proposition}

\begin{remarque}\label{rem:redPhi}
One can also remark that the category of Stokes-filtered local systems with set of exponential factors contained in $\Phi$ is equivalent to the category of $\Phi$-filtered local systems, where we regard $\Phi$ as a constant sheaf on~$S^1$, equipped with the ordered induced by the order of $\ccI_1$ (constant sheaf with fibre $\ccP$).
\end{remarque}

\skpt
\begin{exemples}\label{exem:triviaux}\ligne
\begin{enumerate}
\item\label{exem:twistdual} (Twist)\index{twist}
Let $\eta\in\ccP_x$ and let $(\cL,\cL_\bbullet)$ be a (pre-)Stokes-filtered local system. The twisted local system $(\cL,\cL_\bbullet)[\eta]$ is defined by $\cL[\eta]_{\leq\varphi}=\cL_{\leq\varphi-\eta}$. In the Stokes-filtered case, the set of exponential factors $\Phi[\eta]$ is equal to $\Phi+\eta$. This is analogous to Definition \ref{def:twist}.
\item\label{exem:triviaux1}
The graded Stokes filtration with $\Phi=\{0\}$ (\cf Example \ref{exem:trivIfilt}) on the constant sheaf $\kk_{S^1}$ is defined by $\kk_{S^1,\leq\varphi}\defin\beta_{0\leq\varphi}\kk_{S^1}$ for any $\varphi$, so that $\kk_{S^1,\leq0}=\kk_{S^1}$, $\kk_{S^1,<0}=\nobreak0$, and, for any $\varphi\neq0$, $\kk_{S^1,\leq\varphi}=\kk_{S^1,<\varphi}$ has germ equal to~$\kk_\theta$ at $\theta\in S^1$ iff $0\leqtheta\varphi$, and has germ equal to~$0$ otherwise.

\item\label{exem:triviaux3}
Let $\cL_\bbullet$ be any non-ramified Stokes filtration with set of exponential factors~$\Phi$ reduced to one element $\eta$. According to Proposition \ref{prop:Sfini}, we have $\cL_{<\eta}=0$, and $\cL_{\leq\eta}=\gr_\eta\cL=\cL$ is a local system on~$S^1$. The non-ramified Stokes filtration is then described as in \ref{exem:triviaux}\eqref{exem:triviaux1} above, that is, $\cL_{\leq\varphi}=\beta_{\eta\leq\varphi}\cL$. If we denote by $\cL_\bbullet$ this Stokes filtration of $\cL$ then, using the twist operation \ref{exem:triviaux}\eqref{exem:twistdual}, the twisted Stokes filtration $\cL[-\eta]_\bbullet$ is nothing but the graded Stokes filtration on $\cL$, defined as in \ref{exem:triviaux}\eqref{exem:triviaux1}.
\item
Assume that $\#\Phi=2$ or, equivalently (by twisting, \cf above), that $\Phi=\{0,\varphi_o\}$ with $\varphi_o\neq0$. If the order of the pole of $\varphi_o$ is $n$, then there are $2n$ Stokes directions (\cf Example \ref{exem:Stokes}) dividing the circle in $2n$ intervals. Given such an open interval, then~$0$ and $\varphi_o$ are comparable (in the same way) at any $\theta$ in the interval, and the comparison changes alternatively on the intervals. Assume that $0\letheta\varphi_o$. Then, according to Proposition \ref{prop:Sfini} and to \ref{prop:stokeswithout}\eqref{prop:stokeswithout4}, $\cL_{\leq\varphi_o,\theta}=\cL_\theta$ and $\cL_{<0,\theta}=0$. Moreover, when restricted to the open interval containing $\theta$, $\cL_{\leq0}=\cL_{<\varphi_o}$ is a local system of rank equal to $\rk\gr_0\cL$. On the other intervals, the roles of~$0$ and $\varphi_o$ are exchanged.

Let now $\theta$ be a Stokes direction for $(0,\varphi_o)$. As $\varphi_o$ and~$0$ are not comparable at $\theta$, \ref{prop:stokeswithout}(\ref{prop:stokeswithout4}$'$) implies that $\cL_{<\varphi_o,\theta}=\cL_{<0,\theta}=0$ and, using $\varphi$ such that $0,\varphi_o\letheta\varphi$, we find, by exhaustivity, $\cL_\theta=\cL_{\leq\varphi_o,\theta}\oplus\cL_{\leq0,\theta}$. This decomposition reads as an isomorphism $\cL_\theta\simeq\gr_{\varphi_o}\cL_\theta\oplus\gr_0\cL_\theta$. It extends on a neighbourhood $\nb(\theta)$ of $\theta$ in~$S^1$ (we can take for $\nb(\theta)$ the union of the two adjacent intervals considered above ending at $\theta$) in a unique way as an isomorphism of local systems $\cL_{|\nb(\theta)}\simeq(\gr_{\varphi_o}\cL\oplus\gr_0\cL)_{|\nb(\theta)}$.

In order to end the description, we will show that the equality $\cL_{\leq\varphi}=\cL_{<\varphi}$ for $\varphi\not\in\{0,\varphi_o\}$ can be deduced from the data of the corresponding sheaves for $\varphi\in\{0,\varphi_o\}$. Let us fix $\theta\in S^1$. Assume first $0\letheta\varphi_o$ (and argue similarly if $\varphi_o\letheta0$).
\begin{itemize}
\item
If $\varphi$ is neither comparable to $\varphi_o$ nor to~$0$, then \ref{prop:stokeswithout}\eqref{prop:stokeswithout4} shows that $\cL_{\leq\varphi,\theta}=0$.
\item
If $\varphi$ is comparable to $\varphi_o$ but not to~$0$,
\begin{itemize}
\item
if $\varphi_o\letheta\varphi$, then $\cL_{\leq\varphi,\theta}=\cL_\theta$,
\item
if $\varphi\letheta\varphi_o$, then $\cL_{\leq\varphi,\theta}\subset\cL_{<\varphi_o,\theta}=\cL_{\leq0,\theta}$, hence \ref{prop:stokeswithout}\eqref{prop:stokeswithout4} implies $\cL_{\leq\varphi,\theta}=0$.
\end{itemize}
\item
If $\varphi$ is comparable to~$0$ but not to $\varphi_o$, the result is similar.
\item
If $\varphi$ is comparable to both $\varphi_o$ and~$0$, then,
\begin{itemize}
\item
if $0\letheta\varphi_o\letheta\varphi$, $\cL_{\leq\varphi,\theta}=\cL_\theta$,
\item
if $0\letheta\varphi\letheta\varphi_o$, then $\cL_{\leq\varphi,\theta}=\cL_{\leq0,\theta}$,
\item
if $\varphi\letheta0\letheta\varphi_o$, then $\cL_{\leq\varphi,\theta}=0$.
\end{itemize}
\end{itemize}
If $\varphi_o$ and~$0$ are not comparable at $\theta$, then one argues similarly to determine $\cL_{\leq\varphi,\theta}$.
\end{enumerate}
\end{exemples}

\begin{definitio}
A \emph{morphism} $\lambda:(\cL,\cL_\bbullet)\to(\cL',\cL'_\bbullet)$ of non-ramified $\kk$\nobreakdash-Stokes-filtered local systems is a morphism of local systems $\cL\to\cL'$ on~$S^1$ such that, for any $\varphi\in\ccP$, $\lambda(\cL_{\leq\varphi})\subset\cL'_{\leq\varphi}$. According to \eqref{eq:onto<}, a morphism also satisfies $\lambda(\cL_{<\varphi})\subset\cL'_{<\varphi}$. A morphism $\lambda$ is said to be \index{morphism!strict}\index{strict morphism}\emph{strict} if, for any~$\varphi$, $\lambda(\cL_{\leq\varphi})=\lambda(\cL)\cap\cL'_{\leq\varphi}$.
\end{definitio}

\begin{definitio}\label{def:HomSt}
Given two non-ramified $\kk$\nobreakdash-Stokes-filtered local systems $(\cL,\cL_\bbullet)$ and $(\cL',\cL'_\bbullet)$,
\begin{itemize}
\item
the direct sum $(\cL,\cL_\bbullet)\oplus(\cL',\cL'_\bbullet)$ has local system $\cL\oplus\cL'$ and filtration $(\cL\oplus\nobreak\cL')_{\leq\varphi}=\cL_{\leq\varphi}\oplus\cL'_{\leq\varphi}$,
\item
$\cHom(\cL,\cL')_{\leq\eta}$ is the subsheaf of $\cHom(\cL,\cL')$ consisting of local morphisms $\cL\to\cL'$ sending $\cL_{\leq\varphi}$ into $\cL'_{\leq\varphi+\eta}$ fo any $\varphi$;
\item
in particular, the dual $(\cL,\cL_\bbullet)^\vee$ is defined as $(\cHom(\cL,\kk_{S^1}),\cHom(\cL,\kk_{S^1})_\bbullet)$, where $\kk_{S^1}$ is equipped with the graded Stokes filtration of Example \ref{exem:triviaux}\eqref{exem:triviaux1}.
\item
$(\cL\otimes\cL')_{\leq\eta}=\sum_\varphi\cL_{\leq\varphi}\otimes\cL'_{\leq\eta-\varphi}\subset\cL\otimes\cL'$.
\end{itemize}
\end{definitio}

In particular, a morphism of non-ramified Stokes-filtered local systems is a global section of $\cHom(\cL,\cL')_{\leq0}$.

\begin{proposition}\label{prop:operationsSt}
Given two non-ramified $\kk$\nobreakdash-Stokes filtrations $\cL_\bbullet,\cL'_\bbullet$ of $\cL,\cL'$, $\cL_\bbullet\oplus\nobreak\cL'_\bbullet$, $\cHom(\cL,\cL')_\bbullet$, $(\cL^\vee)_\bbullet$ and $(\cL\otimes\cL')_\bbullet$ are non-ramified $\kk$\nobreakdash-Stokes filtrations of the corresponding local systems and $\cHom(\cL,\cL')_\bbullet\simeq(\cL^{\prime\vee}\otimes\cL)_\bbullet$. Moreover,
\begin{enumerate}
\item\label{prop:operationsSt1}
$\cHom(\cL,\cL')_{<\eta}$ is the subsheaf of $\cHom(\cL,\cL')$ consisting of local morphisms $\cL\to\cL'$ sending $\cL_{\leq\varphi}$ into $\cL'_{<\varphi+\eta}$ fo any $\varphi$;
\item\label{prop:operationsSt2}
$(\cL^\vee)_{\leq\varphi}=(\cL_{<-\varphi})^\perp$ and $(\cL^\vee)_{<\varphi}=(\cL_{\leq-\varphi})^\perp$ for any $\varphi$, so that $\gr_\varphi\cL^\vee=(\gr_{-\varphi}\cL)^\vee$ (here, $(\cL_{<-\varphi})^\perp$, \resp $(\cL_{\leq-\varphi})^\perp$, consists of local morphisms $\cL\to\kk_{S^1}$ sending $\cL_{<-\varphi}$, \resp $\cL_{\leq-\varphi}$, to zero);
\item\label{prop:operationsSt3}
$(\cL\otimes\cL')_{<\eta}=\sum_\varphi\cL_{\leq\varphi}\otimes\cL'_{<\eta-\varphi}=\sum_\varphi\cL_{<\varphi}\otimes\cL'_{\leq\eta-\varphi}$.
\end{enumerate}
\end{proposition}

\begin{proof}
For the first assertion, let us consider the case of $\cHom$ for instance. Using a local decomposition of $\cL,\cL'$ given by the Stokes filtration condition, we find that a section of $\cHom(\cL,\cL')_\theta$ is decomposed as a section of $\bigoplus_{\varphi,\psi}\cHom(\gr_\varphi\cL,\gr_\psi\cL')_\theta$, and that it belongs to $\cHom(\cL,\cL')_{\leq\eta,\theta}$ if and only if its components $(\varphi,\psi)$ are zero whenever $\psi-\varphi\not\leqtheta\eta$. The assertion is then clear, as well as the characterization of $\cHom(\cL,\cL')_{<\eta}$.

As a consequence, a local section of $(\cL^\vee)_{\leq\varphi}$ has to send $\cL_{<-\varphi}$ to zero for any $\varphi$. The converse is also clear by using the local decomposition of $(\cL,\cL_\bbullet)$, as well as the other assertions for $\cL^\vee$.

The assertion on the tensor product is then routine.
\end{proof}

\begin{remarque}\label{rem:expfact}
One easily gets the behaviour of the set of exponential factors with respect to such operations. For instance, the direct sum corresponds to $(\Phi,\Phi')\mto\Phi\cup\Phi'$, the dual to $\Phi\mto-\Phi$ and the tensor product to $(\Phi,\Phi')\mto\Phi+\Phi'$.
\end{remarque}

\subsubsection*{Poincaré-Verdier duality}
For a sheaf~$\cF$ on~$S^1$, its \index{Poincaré-Verdier duality}\index{duality!Poincaré-Verdier}Poincaré-Verdier dual \index{$DUAL$@$\bD$, $\bD'$}$\bD\cF$ is $\bR\cHom_\CC(\cF,\kk_{S^1}[1])$ and we denote by $\bD'\cF=\bR\cHom_\CC(\cF,\kk_{S^1})$ the shifted complex. We clearly have $\bD'\cL=\cHom_\CC(\cL,\kk_{S^1})=:\cL^\vee$.

\begin{lemme}[Poincaré duality]\label{lem:dualS1}
For any $\varphi\in\ccP$, the complexes $\bD'(\cL_{\leq\varphi})$ and $\bD'(\cL/\cL_{\leq\varphi})$ are sheaves. Moreover, the two subsheaves $(\cL^\vee)_{\leq\varphi}$ and $\bD'(\cL/\cL_{<-\varphi})$ of $\cL^\vee$ coincide.
\end{lemme}

\begin{proof}
The assertions are local on~$S^1$, so we can assume that~$(\cL,\cL_\bbullet)$ is split with respect to the Stokes filtration, and therefore that $(\cL,\cL_\bbullet)$ has only one exponential factor $\eta$, that is, $\cL_{\leq\varphi}=\beta_{\eta\leq\varphi}\cL$ (\cf Example \ref{exem:triviaux}\eqref{exem:triviaux3}. Let us denote by $\alpha_{\psi<\varphi}$ the functor which is the composition of the restriction to the open set $S^1_{\psi<\varphi}$ (\cf Notation~\ref{nota:beta}) and the maximal extension to~$S^1$, and similarly with $\leq$. We have an exact sequence
\[
0\to\beta_{\eta<\varphi}\cL\to\cL\to\alpha_{\varphi\leq\eta}\cL\to0
\]
which identifies $\alpha_{\varphi\leq\eta}\cL$ to $\cL/\cL_{<\varphi}$, and a similar one with $\beta_{\eta\leq\varphi}$ and $\alpha_{\varphi<\eta}$. On the other hand, $\bD'(\beta_{\eta\leq\varphi}\cL)=\alpha_{\eta\leq\varphi}\cL^\vee$ and $\bD'(\alpha_{\varphi\leq\eta}\cL)=\beta_{\varphi\leq\eta}\cL^\vee$. The dual of the previous exact sequence, when we replace~$\varphi$ with $-\varphi$, is then
\begin{align*}
&0\to\beta_{-\varphi\leq\eta}\cL^\vee\to\cL^\vee\to\alpha_{\eta<-\varphi}\cL^\vee\to0,\\
\tag*{also written as}
&0\to\beta_{-\eta\leq\varphi}\cL^\vee\to\cL^\vee\to\alpha_{\varphi<-\eta}\cL^\vee\to0,
\end{align*}
showing that $\bD'(\cL/\cL_{<-\varphi})=(\cL^\vee)_{\leq\varphi}$.
\end{proof}

\subsection{Pull-back and push-forward}\label{subsec:pushpull}
Let $f:X'\to X$ be a holomorphic map from the disc $X'$ (with coordinate $x'$) to the disc~$X$ with coordinate $x$. We assume that both discs are small enough so that $f$ is ramified at $x'=0$ only. We now denote by $S^1_{x'}$ and $S^1_x$ the circles of directions in the spaces of polar coordinates $\wt X{}'$ and $\wt X$ respectively. Then $f$ induces $\wt f:S^1_{x'}\to S^1_x$, which is the composition of the multiplication by $N$ (the index of ramification of $f$) and a translation (the argument of $f^{(N)}(0)$). Similarly, $\ccP_x$ and $\ccP_{x'}$ denote the polar parts in the variables $x$ and $x'$ respectively.

\begin{remarque}\label{rem:pullback}
Let $\eta\in\ccP_x$ and set $f^*\eta=\eta\circ f\in\ccP_{x'}$. For any $\theta'\in S^1_{x'}$, set $\theta=\wt f(\theta')$. Then we have
\[
f^*\eta\leqthetap0\ssi\eta\leqtheta0\quad\text{and}\quad f^*\eta\lethetap0\ssi\eta\letheta0.
\]
(This is easily seen using the definition in terms of moderate growth in Example \ref{exem:Stokes}, since $\wt f:\wt X{}'\to\wt X$ is a finite covering.)
\end{remarque}

\begin{definitio}[Pull-back]\label{def:pullback}
Let~$\cL$ be a local system on $S^1_x$ and let $\cL_\bbullet$ be a non-ramified $\kk$\nobreakdash-pre-Stokes filtration of~$\cL$. For any $\varphi'\in\ccP_{x'}$ and any $\theta'\in S^1_{x'}$, let us set
\[
(\wt f^+\cL)_{\leq\varphi',\theta'}\defin\sum_{\substack{\psi\in\ccP_x\\f^*\psi\leqthetap\varphi'}}\cL_{\leq\psi,\wt f(\theta')}\subset \cL_{\wt f(\theta')}=(\wt f^{-1}\cL)_{\theta'}.
\]
Then $(\wt f^+\cL)_\bbullet$ is a non-ramified pre-Stokes filtration on $\wt f^{-1}\cL$, called the \index{pull-back (inverse image)!of Stokes filtrations}\emph{pull-back} of $\cL_\bbullet$ by~$f$. We denote by \index{$Fsupplus$@$f^+$ (pull-back)}$\wt f^+(\cL,\cL_\bbullet)$ the pull-back pre-Stokes-filtered local system, in order to remember that the indexing set has changed (\cf Definition \ref{def:pullbackpreI}).
\end{definitio}

\begin{proposition}[Pull-back]\label{prop:pullback}
The pull-back $\wt f^+\!(\cL,\!\cL_\bbullet)$ has the following properties:
\begin{enumerate}
\item\label{prop:pullback-1}
For any $\varphi'\in\ccP_{x'}$ and any $\theta'\in S^1_{x'}$,
\[
(\wt f^+\cL)_{<\varphi',\theta'}=\sum_{\substack{\psi\in\ccP_x\\f^*\psi\lethetap\varphi'}}\wt f^{-1}(\cL_{\leq\psi,\wt f(\theta')}).
\]
\item\label{prop:pullback0}
For any $\varphi\in\ccP_x$,
\begin{align*}
(\wt f^+\cL)_{\leq f^*\varphi}&=\wt f^{-1}(\cL_{\leq\varphi}),\\
(\wt f^+\cL)_{<f^*\varphi}&=\wt f^{-1}(\cL_{<\varphi})\\
\tag*{and}\gr_{f^*\varphi}(\wt f^+\cL)&=\wt f^{-1}(\gr_\varphi\cL).
\end{align*}
\item\label{prop:pullback0b}
In particular, if $\wt f^+(\cL,\cL_\bbullet$) is a non-ramified Stokes-filtered local system for some $f$, then for any $\varphi\in\ccP_x$, $\gr_\varphi\cL$ is a local system on $S^1_x$.
\item\label{prop:pullback2}
Let $\cL,\cL'$ be two local systems on $S^1_x$ equipped with non-ramified pre-Stokes filtrations and let $\lambda:\cL\to\cL'$ be a morphism of local systems such that, for some~$f$, $\wt f^{-1}\lambda:\wt f^{-1}\cL\to\wt f^{-1}\cL'$ is compatible with the non-ramified pre-Stokes filtrations $(\wt f^+\cL)_\bbullet,(\wt f^+\cL')_\bbullet$. Then~$\lambda$~is compatible with the non-ramified pre-Stokes filtrations $\cL_\bbullet,\cL'_\bbullet$.
\item\label{prop:pullback1}
Assume now that $\cL_\bbullet$ is a non-ramified $\kk$\nobreakdash-Stokes filtration (\ie is locally graded) and let~$\Phi$ be its set of exponential factors. Then $(\wt f^+\cL)_\bbullet$ is a non-ramified $\kk$\nobreakdash-Stokes filtration on $\wt f^{-1}\cL$ and, for any $\varphi'\in\ccP_{x'}$, $\gr_{\varphi'}\wt f^+\cL\neq0\implique\varphi'\in f^*\Phi$.
\item\label{prop:pullback3}
The pull-back of non-ramified Stokes filtrations is compatible with $\cHom$, duality and tensor product.
\end{enumerate}
\end{proposition}

\begin{proof}
By definition,
\[
(\wt f^+\cL)_{<\varphi',\theta'}=\sum_{\substack{\psi'\lethetap\varphi'}}(\wt f^+\cL)_{\leq\psi',\theta'}=\sum_{\substack{\psi'\lethetap\varphi'}}\sum_{\substack{\psi\in\ccP_x\\f^*\psi\leqthetap\psi'}}\cL_{\leq\psi,\wt f(\theta')},
\]
and this is the RHS in \ref{prop:pullback}\eqref{prop:pullback-1}.

The first two lines of \ref{prop:pullback}\eqref{prop:pullback0} are a direct consequence of Remark \ref{rem:pullback}, and the third one is a consequence of the previous ones. Then \ref{prop:pullback}\eqref{prop:pullback0b} follows, as each $\gr_{\varphi'}(\wt f^+\cL)$ is a local system on $S^1_{x'}$.

\ref{prop:pullback}\eqref{prop:pullback2} follows from the first line in \ref{prop:pullback}\eqref{prop:pullback0} and \ref{prop:pullback}\eqref{prop:pullback1} from third line and from the local gradedness condition. Then, \ref{prop:pullback}\eqref{prop:pullback3} is clear.
\end{proof}

\begin{remarque}
In order to make clear the correspondence with the notion introduced in Definition \ref{def:pullbackpreI} and considered in Lemma \ref{lem:stabpullback}, note that the sheaf $\ccI_1$ is the constant sheaf on $S^1_x$ with fibre $\ccP_x$, $\wt f^{-1}\ccI_1$ is the constant sheaf on $S^1_{x'}$ with fibre $\ccP_x$ and $\ccI'_1$ is the constant sheaf on $S^1_{x'}$ with fibre $\ccP_{x'}$. The map $q_f$ is $f^*:\ccP_x\to\ccP_{x'}$.
\end{remarque}

\begin{exo}[Push forward]\label{def:pushforward}
Let $\cL'$ be a local system on $S^1_{x'}$ equipped with a non-ramified pre-Stokes filtration $\cL'_\bbullet$. Show that
\begin{enumerate}
\item
$\wt f_*\cL'$ is naturally equipped with a non-ramified pre-Stokes filtration defined by $(\wt f_*\cL')_{\leq\varphi}=\wt f_*(\cL'_{\leq f^*\varphi})$;

\item
assume moreover that $\cL'_\bbullet$ is a non-ramified Stokes filtration and let $\Phi'\subset\ccP_{x'}$ be its set of exponential factors; if there exists a~finite subset $\Phi\subset\ccP_x$ such that $\Phi'=f^*\Phi$, then the \index{push-forward (direct image)}push-forward pre-Stokes filtration $(\wt f_*\cL')_\bbullet$ is a Stokes filtration.
\end{enumerate}
\end{exo}

\subsection{Stokes filtrations on local systems}\label{subsec:Stokeswithramif}
We now define the general notion of a (possibly ramified) Stokes filtration on a local system $\cL$ on~$S^1$.

Let $d$ be a nonzero integer and let $\rho_d:X_d\to X$ be a holomorphic function from a disc $X_d$ (coordinate~$x'$) to the disc~$X$ (coordinate~$x$). For simplicity, we will assume that the coordinates are chosen so that $\rho_d(x')=x^{\prime d}$.

\begin{definitio}[(Pre-)Stokes filtration]\label{def:Stokeswithramif}
Let~$\cL$ be a local system on $S^1_x$. A~\index{Stokes-filtered local@Stokes-filtered local system}\index{local system!Stokes-filtered}\index{Stokes filtration}\emph{$\kk$\nobreakdash-\hbox{(pre-)}\allowbreak Stokes filtration} (ramified of order $\leq d$) on $\cL$ consists of a non-ramified (pre-)Stokes filtration on $\cL'\defin\rho_d^{-1}\cL$ such that, for any automorphism $\sigma$
\[
\xymatrix{
X_d\ar[dr]_{\rho_d}\ar[rr]^-{\sigma}&&X_d\ar[dl]^{\rho_d}\\&X&
}
\]
and any $\varphi'\in\ccP_{x'}$, we have $\cL'_{\leq\sigma^*\varphi'}=\wt\sigma^{-1}\cL'_{\leq\varphi'}$ in $\cL'=\wt\sigma^{-1}\cL'$. Similarly, a morphism of (pre-)$\kk$\nobreakdash-Stokes-filtered local systems is a morphism of local systems which becomes a morphism of non-ramified Stokes-filtered local systems after ramification.
\end{definitio}

We will make precise the relation with the notion of a $\ccI$-filtration of \Chaptersname\ref{chap:Ifil} by defining first the sheaf $\ccI$.

\subsubsection*{The sheaf $\ccI$ on $S^1$}
Let $d$ be a positive integer. We denote by \index{$ID$@$\ccI_1$, $\ccI_d$, $\ccI_{\bun}$, $\ccI_{\bmd}$}$\ccI_d$ the local system on $S^1_x$ whose fibre at $\theta=0$ is $\ccP_{x'}$ and whose monodromy is given by $\ccP_{x'}\ni\varphi'(x')\mto\varphi'(e^{2\pi i/d}x')$. If we denote by $\wt\rho{}_d:S^1_{x'}\to S^1_x$, $\theta'\mto d\cdot\theta'$ the associated map, the sheaf $\wt\rho{}_d^{-1}\ccI_d$ is the constant sheaf on~$S^1_{x'}$ with fibre $\ccP_{x'}$. In particular, $\ccI_d$ is a sheaf of ordered abelian groups on $S^1_x$, and the constant sheaf \index{$ID$@$\ccI_1$, $\ccI_d$, $\ccI_{\bun}$, $\ccI_{\bmd}$}$\ccI_1$ with fibre $\ccP_x$ is a subsheaf of ordered abelian groups. We will then denote by~\index{$I$@$\ccI$, $\ccIet$}$\ccI$ the sheaf $\bigcup_{d\geq1}\ccI_d$.

\begin{remarque}\label{rem:Euler}
Let us give another description of the sheaf \index{$ID$@$\ccI_1$, $\ccI_d$, $\ccI_{\bun}$, $\ccI_{\bmd}$}$\ccI_d$ which will be useful in higher dimensions. We use the notation of Example \ref{exem:Stokes}. Let us denote by~$j_\partial$ and~$j_{\partial,d}$ the natural inclusions $X^*\hto\wt X$ and $X^*_d\hto\wt X_d$, and by $\wt\rho_d:\wt X_d\to\wt X$ the lifting of $\rho_d$. The natural inclusion $\cO_{X^*}\hto\nobreak\rho_{d,*}\cO_{X_d^*}$ induces an injective morphism $j_{\partial,*}\cO_{X^*}\hto j_{\partial,*}\rho_{d,*}\cO_{X_d^*}=\wt\rho{}_{d,*}j_{\partial,d,*}\cO_{X_d^*}$, that we regard as the inclusion of a subsheaf.

Let us denote by $(j_{\partial,*}\cO_{X^*})^\lb$ the subsheaf of $j_{\partial,*}\cO_{X^*}$ consisting of functions which are locally bounded on~$\wt X$. We have $(j_{\partial,*}\cO_{X^*})^\lb=\wt\rho{}_{d,*}(j_{\partial,d,*}\cO_{X_d^*})^\lb\cap j_{\partial,*}\cO_{X^*}$ since~$\wt\rho{}_d$ is proper.

Let us set $\wt\ccI_1=\varpi^{-1}\cO_X(*0)$, that we consider as a subsheaf of $j_{\partial,*}\cO_{X^*}$. We have $\varpi^{-1}\cO_X=\wt\ccI_1\cap(j_{\partial,*}\cO_{X^*})^\lb$ in $j_{\partial,*}\cO_{X^*}$ (a meromorphic function which is bounded in some sector is bounded everywhere, hence is holomorphic). Therefore, $\ccI_1\defin\varpi^{-1}(\cO_X(*0)/\cO_X)$ is also equal to $\wt\ccI_1/\wt\ccI_1\cap(j_{\partial,*}\cO_{X^*})^\lb$.

Similarly, we can first define $\wt\ccI_d$ as the subsheaf of $\CC$-vector spaces of $j_{\partial,*}\cO_{X^*}$ which is the intersection of $\wt\rho{}_{d,*}\varpi_d^{-1}\cO_{X_d}(*0)$ and $j_{\partial,*}\cO_{X^*}$ in $\wt\rho{}_{d,*}j_{\partial,d,*}\cO_{X_d^*}$. We then set $\ccI_d=\wt\ccI_d/\wt\ccI_d\cap(j_{\partial,*}\cO_{X^*})^\lb$, which is a subsheaf of $j_{\partial,*}\cO_{X^*}/(j_{\partial,*}\cO_{X^*})^\lb$. We have $\ccI_d\subset\ccI_{d'}$ if $d$ divides $d'$.

As we already noticed, $\ccI_1=\varpi^{-1}(\cO_X(*0)/\cO_X)$. More generally, let us show that $\wt\rho{}_d^{-1}\ccI_d=\varpi_d^{-1}(\cO_{X_d}(*0)/\cO_{X_d})$. We will start by showing that $\wt\rho{}_d^{-1}\wt\ccI_d=\varpi_d^{-1}\cO_{X_d}(*0)$.

Let us first note that $\rho_d^{-1}\cO_{X^*}=\cO_{X_d^*}$ since $\rho_d$ is a covering, and $\wt\rho{}_d^{-1}j_{\partial,*}\cO_{X^*}=j_{\partial,d,*}\rho_d^{-1}\cO_{X^*}$ since $\wt\rho{}_d$ is a covering. Hence, $\wt\rho{}_d^{-1}j_{\partial,*}\cO_{X^*}=j_{\partial,d,*}\cO_{X_d^*}$.

It follows that $\wt\rho{}_d^{-1}\wt\ccI_d$ is equal to the intersection of $\wt\rho{}_d^{-1}\wt\rho{}_{d,*}[\varpi_d^{-1}\cO_{X_d}(*0)]$ (since $\wt\ccI_{\wt X_d,1}=\varpi_d^{-1}\cO_{X_d}(*0)$) and $j_{\partial,d,*}\cO_{X_d^*}$ in $\wt\rho{}_d^{-1}\wt\rho{}_{d,*}[j_{\partial,d,*}\cO_{X_d^*}]$. This is $\varpi_d^{-1}\cO_{X_d}(*0)$. Indeed, a germ in $\wt\rho{}_d^{-1}\wt\rho{}_{d,*}[\varpi_d^{-1}\cO_{X_d}(*0)]$ at $\theta'$ consists of a $d$-uple of germs in $\cO_{X_d}(*0)$ at $0$. This $d$-uple belongs to $j_{\partial,d,*}\cO_{X_d^*,\theta'}$ iff the restrictions to $X_d^*$ of the terms of the $d$-uple coincide. Then all the terms of the $d$-uple are equal. The argument for $\ccI_d$ is similar.

Let us express these results in terms of étale spaces. We first note that, since~$\ccI_1$ is a constant sheaf, the étale space $\ccIet_1$ is a trivial covering of $S^1_x$. The previous argument shows that the fibre product $S^1_{x'}\times_{S^1_x}\ccIet_d$ is identified with $\ccIet_{x',1}$, hence is a trivial covering of $S^1_{x'}$. It follows that, since $\wt\rho_d:S^1_{x'}\times_{S^1_x}\ccIet_d\to\ccIet_d$ is a finite covering of degree $d$, that $\ccIet_d\to S^1_x$ is a covering.

The following property will also be useful: there is a one-to-one correspondence between finite sets $\Phi_d$ of $\ccP_{x'}$ and finite coverings \index{$SZIGMAWT$@$\wt\Sigma$}$\wt\Sigma\subset\ccIet_d$. Indeed, given such a $\wt\Sigma$, its pull-back $\wt\Sigma_d$ by $\wt\rho_d$ is a covering of $S^1_{x'}$ contained in the trivial covering $\wt\rho_d^{-1}\ccIet_d$, hence is trivial, and is determined by its fibre $\Phi_d\subset\ccP_{x'}$. Conversely, given such $\Phi_d$, it defines a trivial covering $\wt\Sigma_d$ of $S^1_{x'}$ contained in $\wt\rho_d^{-1}\ccIet_d$. Let $\wt\Sigma$ be its image in $\ccIet_d$. Because the composed map $\wt\Sigma_d\to S^1_{x'}\to S^1_x$ is a covering, so are both maps $\wt\Sigma_d\to\wt\Sigma$ and $\wt\Sigma\to S^1_x$. Moreover, the degree of $\wt\Sigma\to S^1_x$ is equal to that of $\wt\Sigma_d\to S^1_{x'}$, that is, $\#\Phi_d$. Lastly, the pull-back of $\wt\Sigma$ by $\wt\rho_d$ is a covering of $S^1_{x'}$ contained in $S^1_{x'}\times_{S^1_x}\ccIet_d$, hence is a trivial covering, of degree $\#\Phi_d$, and containing $\wt\Sigma_d$, so is equal to $\wt\Sigma_d$.
\end{remarque}

\subsubsection*{Order}
The sheaf $j_{\partial,*}\cO_{X^*}$ is naturally ordered by defining $(j_{\partial,*}\cO_{X^*})_{\leq0}$ as the subsheaf of $j_{\partial,*}\cO_{X^*}$ whose sections have an exponential with moderate growth along~$S^1_x$. Similarly, $j_{\partial,d,*}\cO_{X_d^*}$ is ordered. In this way,~$\wt\ccI$ inherits an order: $\wt\ccI_{\leq0}=\ccI\cap(j_{\partial,*}\cO_{X^*})_{\leq0}$. This order is not altered by adding a local section of $(j_{\partial,*}\cO_{X^*})^\lb$, and thus defines an order on $\ccI$. For each $d$, we also have $\wt\ccI_{d,\leq0}=\wt\rho{}_{d,*}\big((\varpi_d^{-1}\cO_{X_d}(*0)\big)_{\leq0})\cap j_{\partial,*}\cO_{X^*}$ and we also conclude that $\wt\rho{}_d^{-1}(\ccI_{d,\leq0})=\big(\varpi_d^{-1}(\cO_{X_d}(*0)/\cO_{X_d})\big)_{\leq0}$.

\begin{remarque}\label{rem:coIfiltStokes}
The sheaf of ordered abelian groups $\ccI$ satisfies the property \eqref{eq:coIfilt}. The direction $\implique$ is clear. For the other direction, assume that $\psi_\theta\not\leqtheta0$. We will prove that there exists $\eta_\theta$ such that $0\letheta\eta_\theta$ and $\psi_\theta\not\letheta\eta_\theta$. If $\psi_\theta=u_n(x)/x^n$ with $n\in\QQ_+^*$ and $\arg u_n(0)-n\theta\in[-\pi/2,\pi/2]\bmod2\pi$, then we take $\eta\neq0$ having a pole order strictly less than $n$ and a dominant coefficient such that $0\letheta\eta_\theta$. Then the order relation between $0$ and $\psi_\theta$ is the same as the order relation between $0$ and $\psi_\theta-\eta_\theta$.

Note that this argument does not hold on a subsheaf $\ccI_d$ with $d$ fixed.
\end{remarque}

One can rephrase the definition of a (pre-)Stokes filtration by using the terminology of \Chaptersname\ref{chap:Ifil}.

\begin{lemme}
A (pre-)Stokes filtration on $\cL$ is a \hbox{(pre-)\,}$\ccI$-filtration on $\cL$, with~$\ccI$ defined above. It is ramified of order $\leq d$ if the support of $\gr\cL$ is contained in $\ccIet_d$.\qed
\end{lemme}

\skpt
\begin{remarques}\label{rem:Stokeswithramif}\ligne
\begin{enumerate}
\item\label{rem:Stokeswithramif1b}
The condition in Definition \ref{def:Stokeswithramif} can be restated by saying that, for any~$\sigma$, the Stokes-filtered local system $(\cL',\cL'_\bbullet)$ and its pull-back by $\wt\sigma$ coincide (owing to the natural identification $\cL'=\wt\sigma^{-1}\cL'$).

\item\label{rem:Stokeswithramif2}
Given a (possibly ramified) Stokes-filtration on a local system $\cL$, and given a section $\varphi\in\Gamma(U,\ccI)$ on some open set of $S^1$, the subsheaf $\cL_{\leq\varphi}\subset\cL_{|U}$ is well-defined, as well as $\cL_{<\varphi}$, and $\gr_\varphi\cL$ is a local system on $U$. If $\varphi$ is a section of $\ccI$ all over~$S^1$, then it is non-ramified, \ie it is a section of $\ccI_1$, and $\cL_{\leq\varphi},\cL_{<\varphi}$ are subsheaves of $\cL$. From the point of view of Definition \ref{def:Stokeswithramif}, if the non-ramified Stokes filtration exists on $\cL'=\wt\rho_d^{-1}\cL$, one can restrict the set of indices to $\ccP_x\subset\ccP_{x'}$. Then, for $\varphi\in\ccP_x$, $\cL'_{\leq\rho_d^*\varphi}$ is invariant by the automorphisms of~$\cL'$ induced by the automorphisms $\wt\sigma$, hence is the pull-back of a subsheaf $\cL_{\leq\varphi}$ of $\cL$, and similarly for $\cL_{<\varphi}$ and $\gr_\varphi\cL$. This defines a non-ramified pre-Stokes filtration on $\cL$ for which the graded sheaves are local systems (but the dimension property \ref{prop:stokeswithout}\eqref{prop:stokeswithout4} may not be satisfied). Note also that a morphism of Stokes-filtered local systems is compatible with this pre-Stokes filtration. Hence the category of Stokes filtrations on~$\cL$ is a subcategory of the category of non-ramified pre-Stokes filtrations on $\cL$.

Notice however that the non-ramified Stokes-filtered local system $(\wt f^{-1}\cL,(\wt f^{-1}\cL)_\bbullet)$ is not (in general) equal to the pull-back $\wt f^+(\cL,\cL_\bbullet)$ where $\cL_\bbullet$ is this pre-Stokes filtration.

\item
We will still denote by $\cL_\bbullet$ a (possibly ramified) Stokes filtration on $\cL$ and by $(\cL,\cL_\bbullet)$ a (possibly ramified) Stokes-filtered local system, although the previous remark makes it clear that we do not understand $\cL_\bbullet$ as a family of subsheaves of $\cL$ on $S^1$.

\item\label{rem:Stokeswithramif3b}
The ``set of exponential factors of the Stokes-filtered local system'' is now replaced by a subset $\wt\Sigma\subset\ccIet$ such that the projection to \index{$XWTDP$@$\partial\wt X(D)$}$\partial\wt X$ is a finite covering. It corresponds to a finite subset $\Phi_d\subset\ccP_{x'}$ for a suitable ramified covering $\rho_d$ (\cf the last part of Remark \ref{rem:Euler}), which is the set of exponential factors of the non-ramified Stokes filtration of $\wt\rho{}_d^{-1}\cL$.

\item
The category of Stokes-filtered local systems $(\cL,\cL_\bbullet)$ with associated covering contained in $\wt\Sigma$ is equivalent to the category of $\wt\Sigma^\sh$-filtered local systems (\cf Remark \ref{rem:redPhi}).

\item\label{rem:Stokeswithramif4}
Proposition \ref{prop:operationsSt} holds for $\kk$\nobreakdash-Stokes filtrations.

\item\label{rem:Stokeswithramif4b}
Lemma \ref{lem:dualS1} holds for $\kk$\nobreakdash-Stokes filtrations, that is, the family $\bD'(\cL/\cL_{<-\varphi})$ of local subsheaves of $\cL^\vee$ indexed by local sections of~$\ccI$ forms a Stokes filtration of $\cL^\vee$.

\item\label{rem:Stokeswithramif1}
The category of non-ramified $\kk$\nobreakdash-Stokes-filtered local systems on $S^1_x$ is a full subcategory of that of $\kk$\nobreakdash-Stokes-filtered local systems. Indeed, given a non-ramified Stokes-filtered local system on $S^1_x$, one extends it as a ramified Stokes-filtered local system of order~$d$ from $\ccIet_1$ to $\ccIet_d$ using a formula analogous to that of Proposition \ref{prop:pullback}\eqref{prop:pullback-1}.

\item\label{rem:Stokeswithramif5}
If the set $\Phi_d$ of exponential factors of $\wt\rho_d(\cL,\cL_\bbullet)$ takes the form $\rho_d^*\Phi$ for some finite subset $\Phi\subset\ccP_x$ (equivalently, the finite covering $\wt\Sigma$ of $\partial\wt X$ is trivial, \cf \ref{rem:Stokeswithramif}\eqref{rem:Stokeswithramif3b}), then the Stokes filtration is non-ramified.
\end{enumerate}
\end{remarques}

\subsection{Extension of scalars}\index{extension of scalars}
Let $(\cL,\cL_\bbullet)$ be a $\kk$-Stokes-filtered local system and let $\kk'$ be an extension of $\kk$. Then $(\kk'\otimes_{\kk}\nobreak\cL,\kk'\otimes_{\kk}\nobreak\cL_\bbullet)$ is a $\kk'$- Stokes-filtered local system defined over $\kk'$. The following properties are satisfied for any local section $\varphi$ of $\ccI$:
\begin{itemize}
\item
$(\kk'\otimes_{\kk}\cL)_{<\varphi}=\kk'\otimes_{\kk}\cL_{<\varphi}$, and $\gr_\varphi(\kk'\otimes_{\kk}\nobreak\cL)=\kk'\otimes\gr_\varphi\cL$, so the set of exponential factors of $(\kk'\otimes_{\kk}\nobreak\cL,\kk'\otimes_{\kk}\nobreak\cL_\bbullet)$ is equal to that of $(\cL,\cL_\bbullet)$;
\item
$\cL_{\leq\varphi}=(\kk'\otimes_{\kk}\cL_{\leq\varphi})\cap\cL$ in $\kk'\otimes_{\kk}\cL$.
\end{itemize}
In such a case, we say that the $\kk'$-Stokes-filtered local system $(\kk'\otimes_{\kk}\nobreak\cL,\kk'\otimes_{\kk}\nobreak\cL_\bbullet)$ is \emphb{defined over $\kk$}.

Conversely, let now $(\cL,\cL_\bbullet)$ be a $\kk'$-Stokes-filtered local system and let $\wt\Sigma\subset\ccI$ be its covering of exponential factors. We wish to find sufficient conditions to ensure that it comes from a $\kk$-Stokes-filtered local system by extension of scalars.

\begin{proposition}
Assume that the local system $\cL$ is defined over $\kk$, that is, $\cL=\kk'\otimes\cL_{\kk}$ for some $\kk$-local system $\cL_{\kk}$ (regarded as a subsheaf of $\cL$), and that, for any local section $\varphi$ of $\wt\Sigma$,
\[
\cL_{\leq\varphi}=\kk'\otimes_{\kk}(\cL_{\kk}\cap\cL_{\leq\varphi}),
\]
where the intersection is taken in $\cL$. Then $(\cL,\cL_\bbullet)$ is a $\kk'$-Stokes-filtered local system defined over $\kk$.
\end{proposition}

\begin{proof}
It is not difficult to reduce to the non-ramified case, so we will assume below that $\ccI=\ccI_1$ and replace $\wt\Sigma$ by $\Phi\subset\ccP_x$.
We set, for any local section $\psi$ of $\ccP_x$, $\cL_{\kk,\leq\psi}\defin\cL_{\kk}\cap\cL_{\leq\psi}$, so that the condition reads $\cL_{\leq\varphi}=\kk'\otimes_{\kk}\cL_{\kk,\leq\varphi}$ for $\varphi\in\Phi$. This defines a pre-$\ccI$-filtration of $\cL_{\kk}$, and we will show that this is indeed a $\ccI$-filtration.
\begin{enumerate}
\item
We start with a general property of Stokes-filtered local systems. Let $(\cL,\cL_\bbullet)$ and $\Phi$ be as above, and let $\psi\in\ccP_x$. Set $\Psi=\Phi\cup\{\psi\}$ and denote by $\St(\Psi,\Psi)$ the (finite) set of Stokes directions of pairs $\varphi,\eta\in\Psi$. The sheaves $\cL_{\leq\psi}$ and $\cL_{<\psi}$ can be described as triples consisting of their restrictions to the open set $S^1\moins\St(\Psi,\Psi)$, the closed set $\St(\Psi,\Psi)$, and a gluing map from the latter to the restriction to this closed set of the push-forward of the former. We will make this description explicit.

On any connected component $I$ of $S^1\moins\St(\Psi,\Psi)$, the set $\Psi$ is totally ordered, and there exists $\varphi=\varphi(I,\psi)\in\Phi$ such that $\cL_{\leq\psi|I}=\cL_{\leq\varphi|I}$. Similarly, there exists $\eta=\eta(I,\psi)\in\Phi$ such that $\cL_{<\psi|I}=\cL_{<\eta|I}$.

Let us fix $\theta_o\in\St(\Psi,\Psi)$ and denote by $I_1,I_2$ the two connected components of $S^1\moins\St(\Psi,\Psi)$ containing $\theta_o$ in their closure, with corresponding inclusions $j_i:I_i\hto S^1$, $i=1,2$. We also denote by $i_o:\{\theta_o\}\hto S^1$ the closed inclusion and set $\varphi_i\defin\varphi(I_i,\psi)$, $i=1,2$ (\resp $\eta_i\defin\eta(I_i,\psi)$).

We claim that, in the neighbourhood of $\theta_o$ (and more precisely, on $I_1\cup\nobreak I_2\cup\nobreak\{\theta_o\}$), the sheaf $\cL_{\leq\psi}$ is described by the data $\cL_{\leq\psi|I_i}=\cL_{\leq\varphi_i|I_i}$, $i=1,2$, $\cL_{\leq\psi,\theta_o}=i_o^{-1}j_{1,*}\cL_{\leq\varphi_1|I_1}\cap i_o^{-1}j_{2,*}\cL_{\leq\varphi_2|I_2}$, where the intersection is taken in $i_o^{-1}j_{1,*}\cL=i_o^{-1}j_{2,*}\cL=\cL_{\theta_o}$, and the gluing map is the natural inclusion map of the intersection into each of its components. A similar statement holds for $\cL_{<\psi}$. This easily follows from the local description $\cL_{\leq\psi}=\bigoplus_{\varphi\in\Phi}\beta_{\varphi\leq\psi}\gr_\varphi\cL$, and similarly for $\cL_{<\psi}$.

\item
We now claim that $\cL_{\kk,\leq\psi}\defin\cL_{\kk}\cap\cL_{\leq\psi}$ satisfies $\kk'\otimes\cL_{\kk,\leq\psi}=\cL_{\leq\psi}$. This is by assumption on $S^1\moins\St(\Psi,\Psi)$, according to the previous description, and it remains to check this at any $\theta_o\in\St(\Psi,\Psi)$. The previous description gives $i_o^{-1}\cL_{\kk,\leq\psi}=i_o^{-1}j_{1,*}\cL_{\kk,\leq\varphi_1|I_1}\cap i_o^{-1}j_{2,*}\cL_{\kk,\leq\varphi_2|I_2}$ and the result follows easily (by considering a suitable basis of $\cL_{\kk,\theta_o}$ for instance).

\item
Let us now define $\cL_{\kk,<\psi}$ as $\sum_{\psi'\in\ccP_x}\beta_{\psi'<\psi}\cL_{\kk,\psi'}$, as in \eqref{eq:onto<}. Then the previous description also shows that $\cL_{\kk,<\psi}=\cL_{\kk}\cap\cL_{<\psi}$ and that $\cL_{<\psi}=\kk'\otimes_{\kk}\cL_{\kk,<\psi}$.

\item
As a consequence, we obtain that $\gr_\psi\cL=\kk'\otimes_{\kk}\gr_\psi\cL_{\kk}$ for any $\psi$, from which the proposition follows.\qedhere
\end{enumerate}
\end{proof}

\begin{remarque}
The condition considered in the proposition is that considered in~\cite{K-K-P08} in order to define a $\kk$-structure on a Stokes-filtered local system defined over $\kk'$ (\eg $\kk=\QQ$ and $\kk'=\CC$). This proposition shows that there is no difference with the notion of Stokes-filtered $\kk$-local system.
\end{remarque}

\subsection{Stokes-filtered local systems and Stokes data}\label{sec:Stokesfil}
In this section, we make explicit the relationship between Stokes filtrations and the more conventional approach with Stokes data in the simple case of a Stokes-filtered local system of \emphb{exponential type}.

\Subsubsection*{Stokes-filtered local systems of exponential type}

\begin{definitio}[\cf \cite{Malgrange91} and \cite{K-K-P08}]
We say that a \index{local system!Stokes-filtered}Stokes-filtered local system $(\cL,\cL_\bbullet)$ is of \emphb{exponential type} if it is non-ramified and its exponential factors have a pole of order one at most.
\end{definitio}

In such a case, we can replace $\ccP$ with $\CC\cdot x^{-1}$, and thus with $\CC$, and for each $\theta\in S^1$, the partial order $\leqtheta$ on $\CC$ is compatible with addition and satisfies
\[
c\leqtheta0\ssi c=0\text{ or } \arg c-\theta\in(\pi/2,3\pi/2)\mod2\pi.
\]
We will use notation of \S\ref{subsec:Stokesnon-ram} by replacing $\varphi=c/x\in\ccP$ with $c\in\CC$. For each pair $c\neq c'\in\CC$, there are exactly two values of $\theta\bmod2\pi$, say $\theta_{c,c'}$ and $\theta'_{c,c'}$, such that~$c$ and $c'$ are not comparable at $\theta$. We have $\theta'_{c,c'}=\theta_{c,c'}+\pi$. These are the Stokes directions of the pair $(c,c')$. For any $\theta$ in one component of $S^1\moins\{\theta_{c,c'},\theta'_{c,c'}\}$, we have $c\letheta c'$, and the reverse inequality for any $\theta$ in the other component.

\subsubsection*{Stokes data}
These are linear data which provide a description of Stokes-filtered local system. Given a finite set $C\subset \CC$ and given $\theta_o\in S^1$ which is not a Stokes direction of any pair $c\neq c'\in C$, $\theta_o$ defines a total ordering on $C$, that we write $c_1\lethetao c_2\lethetao\cdots\lethetao c_n$.

\begin{definitio}\label{def:catStokesdata}
Let $C$ be a finite subset of $\CC$ totally ordered by $\theta_o$. The category of \index{Stokes data}\emph{Stokes data of type $(C,\theta_o)$} has objects consisting of two families of $\kk$-vector spaces $(G_{c,1},G_{c,2})_{c\in C}$ and a diagram of morphisms
\bgroup\numstareq
\begin{equation}\label{eq:catStokesdata}
\begin{array}{c}
\xymatrix@C=1.5cm{
\bigoplus_{c\in C}G_{c,1}\ar@/^1pc/[r]^-{S}\ar@/_1pc/[r]_-{S'}&\bigoplus_{c\in C}G_{c,2}
}
\end{array}
\end{equation}
\egroup
such that, for the numbering $C=\{c_1,\dots,c_n\}$ given by $\theta_o$,
\begin{enumerate}
\item
$S=(S_{ij})_{i,j=1,\dots, n}$ is block-upper triangular, \ie $S_{ij}:G_{c_i,1}\to G_{c_j,2}$ is zero unless $i\leq j$, and $S_{ii}$ is invertible (so $\dim G_{c_i,1}=\dim G_{c_i,2}$, and $S$ itself is invertible),
\item
$S'=(S'_{ij})_{i,j=1,\dots, n}$ is block-lower triangular, \ie $S'_{ij}:G_{c_i,1}\to G_{c_j,2}$ is zero unless $i\geq j$, and $S'_{ii}$ is invertible (so $S'$ itself is invertible).
\end{enumerate}

A morphism of Stokes data consists of morphisms of $\kk$-vector spaces $\lambda_{c,\ell}:G_{c,\ell}\to G'_{c,\ell}$, $c\in C$, $\ell=1,2$ which are compatible with the diagrams \eqref{eq:catStokesdata}.
\end{definitio}

Fixing bases in the spaces $G_{c,\ell}$, $c\in C$, $\ell=1,2$, allows one to present Stokes data by matrices $(\Sigma,\Sigma')$ where $\Sigma=(\Sigma_{ij})_{i,j=1,\dots, n}$ (\resp $\Sigma'=(\Sigma'_{ij})_{i,j=1,\dots, n}$) is block-lower (\resp -upper) triangular and each $\Sigma_{ii}$ (\resp $\Sigma'_{ii}$) is invertible.

The following is a translation of a classical result (\cf \cite{Malgrange83b} and the references given therein, \cf also \cite{H-S09} for applications):

\begin{proposition}\label{prop:datafiltered}
There is a natural functor from the category of Stokes-filtered local systems with exponential factors contained in $C$ and the category of Stokes data of type $(C,\theta_o)$, which is an equivalence of categories.\qed
\end{proposition}

The proof of this proposition, that we will not reproduce here, mainly uses Theorem \ref{th:abelianwithout} of the next \chaptername, and more precisely Lemma \ref{lem:grandintervalle} to define the functor.

\subsubsection*{Duality}\label{subsec:duality}\index{duality}
Let $(\cL,\cL_\bbullet)$ be a Stokes-filtered local system. Recall (\cf Definition \ref{def:HomSt}) the dual local system $\cL^\vee$ comes equipped with a Stokes filtration $(\cL^\vee)_\bbullet$ defined by
\[
(\cL^\vee)_{\leq c}=(\cL_{<-c})^\perp,
\]
where the orthogonality is relative to duality. In particular, $\gr_c(\cL^\vee)=(\gr_{-c}\cL)^\vee$. Similarly, given Stokes data $((G_{c,1},G_{c,2})_{c\in C},S,S')$ of type $(C,\theta_0)$, let us denote by~$\tS$ the adjoint of $S$ by duality. Define Stokes data $((G_{c,1},G_{c,2})_{c\in C},S,S')^\vee$ of type $(-C,\theta_o)$ by the formula $G^\vee_{-c,i}=(G_{c,i})^\vee$ ($i=1,2$) and $S^\vee=\tS^{-1}$, $S^{\prime\vee}=\tS^{\prime-1}$, so that the diagram \eqref{eq:catStokesdata} becomes
\begin{equation}\tag*{(\protect\ref{eq:catStokesdata})$^\vee$}
\begin{array}{c}
\xymatrix@C=1.5cm{
\bigoplus_{i=1}^r(G_{c_i,1})^\vee\ar@/^1pc/[r]^-{\tS^{-1}} \ar@/_1pc/[r]_-{\tS^{\prime-1}}&\bigoplus_{i=1}^r(G_{c_i,2})^\vee
}
\end{array}
\end{equation}
Then the equivalence of Proposition \ref{prop:datafiltered} is compatible with duality (\cf \cite{H-S09}).

\chapter{Abelianity and strictness}\label{chap:abelian}

\begin{sommaire}
We prove that the category of $\kk$\nobreakdash-Stokes-filtered local systems on~$S^1$ is abelian. The main ingredient, together with vanishing properties of the cohomology, is the introduction of the level structure. Abelianity is also a consequence of the Riemann-Hilbert correspondence considered in \Chaptersname \ref{chap:RH}, but it is instructive to prove it over the base field $\kk$.
\end{sommaire}

\subsection{Introduction}
The purpose of this \chaptersname is to prove the following:

\begin{theoreme}\label{th:abelian}
The category of $\kk$\nobreakdash-Stokes-filtered local systems on~$S^1$ is \index{abelian (category)}abelian and every morphism is strict. Moreover, it is stable by extension in the category of pre-Stokes-filtered sheaves.
\end{theoreme}

Pre-Stokes-filtered sheaves are defined, in the setting of \Chaptersname\ref{chap:Stokesone}, in Definition~\ref{def:preStokesnonramif} (non-ramified case). The possibly ramified case is obtained as in Definition \ref{def:Stokeswithramif}. They correspond to pre-$\ccI$-filtered sheaves on~$S^1$, with~$\ccI$ as in \S\ref{subsec:Stokeswithramif} (this is checked in a way similar to the case of local systems).

Should the morphisms be graded, then the statement of the theorem would be obvious. The global structure of morphisms is quite involved. The idea is to consider morphisms on sufficiently big intervals of $S^1$ (\cf the Comments of \S\ref{subsec:comments3} for the origin of this idea and variants). However, how big the intervals must be depends on the order of the pole of the exponential factors, and a procedure is needed to work on intervals of a fixed length. This procedure is obtained as an induction on the level structure of the Stokes filtration, a notion which is introduced in \S\ref{subsec:level} and corresponds to a fitlration of the indexing set (exponential factors) with respect to the order of the pole.

\subsection{Strictness and abelianity}
Let $\lambda:(\cL,\cL_\bbullet)\to(\cL',\cL'_\bbullet)$ be a morphism of non-ramified pre-Stokes-filtered local systems. The strictness property should be the following:
\begin{equation}\label{eq:strictnon-ramified}
\forall\varphi\in\ccP_x,\quad\lambda(\cL_{\leq\varphi})=\cL_{\leq\varphi}\cap\lambda(\cL_\theta).
\end{equation}
If a morphism $\lambda$ satisfies \eqref{eq:strictnon-ramified}, the two naturally defined pre-Stokes filtrations on $\lambda(\cL)$ coincide, and then the local systems $\ker\lambda$, $\im\lambda$ and $\coker\lambda$ are naturally equipped with non-ramified pre-Stokes filtrations.

However, it is not clear that, for general pre-Stokes filtrations, \eqref{eq:strictnon-ramified} for $\lambda$ implies \eqref{eq:strictnon-ramified} for the pull-back $\wt\rho^+\lambda$ (\cf Definition \ref{def:pullback}) if $\rho$ is some ramification and $\wt\rho$ is the associated covering of $S^1$. Nevertheless, if conversely $\wt\rho^+\lambda$ satisfies \eqref{eq:strictnon-ramified}, then so does~$\lambda$, according to Proposition \ref{prop:pullback}\eqref{prop:pullback0}, and $\ker$, $\im$, $\coker$ in the pre-Stokes-filtered sense are compatible with $\wt\rho^+$. This leads to the following definition for non-ramified or ramified pre-Stokes-filtered local systems.

\begin{definitio}[Strictness]\label{def:strictness}
A morphism~$\lambda:(\cL,\cL_\bbullet)\to(\cL',\cL'_\bbullet)$ between (possibly ramified) pre-Stokes-filtered local systems is said to be \index{morphism!strict}\index{strict morphism}\emph{strict} if $\wt\rho^+\lambda$ satisfies \eqref{eq:strictnon-ramified} for any ramification $\rho$ such that $(\cL,\cL_\bbullet)$ and $(\cL',\cL'_\bbullet)$ are non-ramified pre-Stokes-filtered local systems.
\end{definitio}

\skpt
\begin{remarques}
\begin{enumerate}
\item
The notion of strictness in the definition corresponds to that of Definition \ref{def:morphismpreIfiltered} for a pre-$\ccI$-filtered sheaf, for $\ccI$ as in \S\ref{subsec:Stokeswithramif}, while Condition \eqref{eq:strictnon-ramified} corresponds to strictness for a pre-$\ccI_1$-filtered sheaf.
\item
Let $\lambda:(\cL,\cL_\bbullet)\to(\cL',\cL'_\bbullet)$ be a morphism of non-ramified Stokes-filtered local systems. Assume that, for every $\theta_o\in S^1$, we can find local trivializations \eqref{eq:L<varphi} on some neighbourhood $\nb(\theta_o)$ for $(\cL,\cL_\bbullet)$ and $(\cL',\cL'_\bbullet)$ such that $\lambda:\cL_{\theta_o}\to\cL'_{\theta_o}$ is block-diagonal. Then $\lambda_{\theta_o}$ is block-diagonal in the neighbourhood of~$\theta_o$ and it is strict near~$\theta_o$ (\ie \eqref{eq:strictnon-ramified} is satisfied after any ramification).
\end{enumerate}
\end{remarques}

Theorem \ref{th:abelian} asserts that, if $(\cL,\cL_\bbullet)$ and $(\cL',\cL'_\bbullet)$ are Stokes-filtered local systems, then so are $\ker\lambda$, $\im\lambda$ and $\coker\lambda$ (\ie they also satisfy the local gradedness property). The first part of Theorem \ref{th:abelian}, in the non-ramified case, is a consequence of the following more precise result, proved in \S\ref{pf:abelianwithout}:

\begin{theoreme}\label{th:abelianwithout}
Given two non-ramified Stokes-filtered local systems $(\cL,\cL_\bbullet)$ and $(\cL',\cL'_\bbullet)$, there exist trivializations of them in the neighbourhood of any point of~$S^1$ such that any morphism $\lambda$ between them is diagonal with respect to these local trivializations, hence is strict (Definition \ref{def:strictness}). In particular, such a morphism satisfies \eqref{eq:strictnon-ramified}, and the natural pre-Stokes filtrations on the local systems $\ker\lambda$, $\im\lambda$ and $\coker\lambda$ are Stokes filtrations. Their sets of exponential factors satisfy
\[
\Phi(\ker\lambda)\subset \Phi,\quad \Phi(\coker\lambda)\subset \Phi',\quad \Phi(\im\lambda)\subset \Phi\cap \Phi'.
\]
\end{theoreme}

\begin{corollaire}
Let $\wt\rho$ be a ramification. Then $\ker$, $\im$ and $\coker$ in the category of non-ramified Stokes-filtered local systems are compatible with $\wt\rho^+$.\qed
\end{corollaire}

\begin{proof}[\proofname\ of the first part of Theorem \ref{th:abelian}]
Let $\lambda:(\cL,\cL_\bbullet)\to(\cL',\cL'_\bbullet)$ be a morphism of (possibly ramified) pre-Stokes-filtered local systems. For any $\rho$ such that $\wt\rho^+(\cL,\cL_\bbullet)$ and $\wt\rho^+(\cL',\cL'_\bbullet)$ are non-ramified Stokes-filtered local systems, then $\wt\rho^+\lambda$ satisfies \eqref{eq:strictnon-ramified}, according to Theorem \ref{th:abelianwithout}, and $\ker\wt\rho^+\lambda$, $\im\wt\rho^+\lambda$ and $\coker\wt\rho^+\lambda$ are non-ramified Stokes-filtered local systems, which are compatible with supplementary ramifications.
\end{proof}

\subsection{Level structure of a Stokes-filtered local system}\label{subsec:level}
\let\oldbell\bell
\let\oldbmm\bmm
\let\bell\ell
\let\bmm m
In this paragraph, we work with \emph{non-ramified Stokes-filtered local systems} without mentioning it explicitly. For every $\bell\in\NN$, we define the notion of \index{Stokes filtration!oflevelb@of level $\geq\ell$}\emph{Stokes filtration of \index{level structure}level $\geq\bell$} on~$\cL$, by replacing the set of indices~$\ccP$ by the set $\ccP(\bell)\defin\cO(*0)/x^{-\ell}\cO$. We denote by $[{\cdot}]_\bell$ the map $\ccP\to\ccP(\bell)$. The constant sheaf \index{$IELL$@$\ccI_1(\ell)$, $\ccI_1(\bell)$}$\ccI_1(\bell)$ with fibre $\ccP(\bell)$ is ordered as follows: for every connected open set~$U$ of~$S^1$ and $[\varphi]_\bell,[\psi]_\bell\in\ccP(\bell)$, we have $[\psi]_\bell\leqU[\varphi]_\bell$ if, for some (or any) representatives $\varphi,\psi$ in $\cO(*0)$, $e^{\mx^\bell(\psi-\varphi)}$ has \emphb{moderate growth} in a neighbourhood of~$U$ in~$X$ intersected with~$X^*$. In particular, a Stokes filtration as defined previously has level $\geq0$.

Let us make this more explicit. Let $\eta\in\ccP$, $\eta=u_n(x)x^{-n}$ with either $n=0$ (that is, $\eta=0$) or $n\geq1$ and $u_n(0)\neq0$. Then, in a way analogous to \eqref{eq:orderone} and \eqref{eq:orderonestrict}, we have
\begin{align}\tag{\ref{eq:orderone}$_\ell$}\label{eq:orderonel}
[\eta]_\bell\leqtheta0&\ssi n\leq\bell\text{ or }\arg u_n(0)-n\theta\in(\pi/2,3\pi/2)\bmod2\pi,\\\tag{\ref{eq:orderonestrict}$_\ell$}\label{eq:orderonelstrict}
[\eta]_\bell\letheta0&\ssi u_n(0)\neq0,\ n>\bell\text{ and }\arg u_n(0)-n\theta\in(\pi/2,3\pi/2)\bmod2\pi.
\end{align}

\begin{lemme}\label{lem:IIellorder}
The natural morphism $\ccI_1\to\ccI_1(\bell)$ is compatible with the order.
\end{lemme}

\begin{proof}
According to \eqref{eq:orderone} and \eqref{eq:orderonel}, we have $\eta\leqtheta0\implique[\eta]_\bell\leqtheta0$.
\end{proof}

We will now introduce a reduction procedure with respect to the level. Given a Stokes-filtered local system $(\cL,\cL_\bbullet)$ (of level $\geq0$), we set, using Notation \ref{nota:beta},\index{$LLEQLETC$@$\cL_{\leq[\varphi]_\bell}$, $\cL_{<[\varphi]_\bell}$, $\gr_{[\varphi]_\bell}\cL$}
\[
\cL_{\leq[\varphi]_\bell}=\sum_\psi\beta_{[\psi]_\bell\leq[\varphi]_\bell}\cL_{\leq\psi},
\]
where the sum is taken in~$\cL$. Then
\[
\cL_{<[\varphi]_\bell}\defin\sum_{[\psi]_\bell}\beta_{[\psi]_\bell<[\varphi]_\bell}\cL_{\leq[\psi]_\bell}
=\sum_\psi\beta_{[\psi]_\bell<[\varphi]_\bell}\cL_{\leq\psi}.
\]
Indeed,
\[
\cL_{<[\varphi]_\bell}=\sum_{[\psi]_\bell}\sum_\eta\beta_{[\psi]_\bell<[\varphi]_\bell}
\beta_{[\eta]_\bell\leq[\varphi]_\bell}\cL_{\leq\eta},
\]
and for a fixed $\eta$, the set of $\theta\in S^1$ for which there exists $[\psi]_\bell$ satisfying $[\eta]_\bell\leqtheta[\psi]_\bell\letheta[\varphi]_\bell$ is equal to the set of $\theta$ such that $[\eta]_\bell\letheta[\varphi]_\bell$. So the right-hand term above is written
\[
\sum_\eta\beta_{[\eta]_\bell<[\varphi]_\bell}\cL_{\leq\eta}.
\]

We can also pre-Stokes-filter $\mu^{-1}\gr_{[\varphi]_\bell}\cL$ by setting, for $\psi\in\ccP$,
\[
(\gr_{[\varphi]_\bell}\cL)_{\leq\psi}=(\cL_{\leq\psi}\cap\cL_{\leq[\varphi]_\bell}+\cL_{<[\varphi]_\bell})/\cL_{<[\varphi]_\bell}.
\]

\begin{proposition}\label{prop:level}
Assume $(\cL,\cL_\bbullet)$ is a Stokes-filtered local system (of level $\geq0$) and let~$\Phi$ be the finite set of its exponential factors.
\begin{enumerate}
\item\label{prop:level1}
For each $\bell\in\NN$, $\cL_{\leq[\cbbullet]_\bell}$ defines a Stokes-filtered local system $(\cL,\cL_{[\cbbullet]_\bell})$ of level \hbox{$\geq\bell$} on~$\cL$, $\gr_{[\varphi]_\bell}\cL$ is locally isomorphic to $\bigoplus_{\psi,\,[\psi]_\bell=[\varphi]_\bell}\gr_\psi\cL$, and the set of exponential factors of $(\cL,\cL_{[\cbbullet]_\bell})$ is $\Phi(\bell)\defin\image(\Phi\to\ccP(\bell))$.
\item\label{prop:level2}
For every $[\varphi]_\bell\in \Phi(\bell)$, $(\gr_{[\varphi]_\bell}\cL,(\gr_{[\varphi]_\bell}\cL)_\bbullet)$ is a Stokes-filtered local system and its set of exponential factors is the pull-back of $[\varphi]_\bell$ by $\Phi\to \Phi(\bell)$.
\item\label{prop:level3}
Let us set
\[
\big(\gr_\bell\cL,(\gr_\bell\cL)_\bbullet\big)\defin\bigoplus_{[\psi]_\bell\in \Phi(\bell)}\big(\gr_{[\psi]_\bell}\cL,(\gr_{[\psi]_\bell}\cL)_\bbullet\big).
\]
Then $(\gr_\bell\cL,(\gr_\bell\cL)_\bbullet)$ is a Stokes-filtered local system (of level $\geq0$) which is locally isomorphic to $(\cL,\cL_\bbullet)$.
\end{enumerate}
\end{proposition}

\begin{proof}
All the properties are local, so we can assume that $(\cL,\cL_\bbullet)$ is graded. We have then
\[
\cL_{\leq\psi}=\bigoplus_\eta\beta_{\eta\leq\psi}\gr_\eta\cL,
\]
hence
\[
\cL_{\leq[\varphi]_\bell}=\bigoplus_\eta\Big(\sum_\psi\beta_{[\psi]_\bell\leq[\varphi]_\bell}\beta_{\eta\leq\psi}\gr_\eta\cL\Big),\quad
\cL_{<[\varphi]_\bell}=\bigoplus_\eta\Big(\sum_\psi\beta_{[\psi]_\bell<[\varphi]_\bell}\beta_{\eta\leq\psi}\gr_\eta\cL\Big)
\]
For a fixed $\eta$, the set of $\theta\in S^1$ such that there exists $\psi$ with $\eta\leqtheta\psi$ and $[\psi]_\bell\leqtheta[\varphi]_\bell$ (\resp $[\psi]_\bell\letheta[\varphi]_\bell$) is the set of $\theta$ such that $[\eta]_\bell\leqtheta[\varphi]_\bell$ (\resp $[\eta]_\bell\letheta[\varphi]_\bell$), so
\[
\cL_{\leq[\varphi]_\bell}=\bigoplus_\eta\beta_{[\eta]_\bell\leq[\varphi]_\bell}\gr_\eta\cL,\quad\cL_{<[\varphi]_\bell}=\bigoplus_\eta\beta_{[\eta]_\bell<[\varphi]_\bell}\gr_\eta\cL,
\]
which gives \ref{prop:level}\eqref{prop:level1}.

Let us show \ref{prop:level}\eqref{prop:level2}. All sheaves entering in the definition of $(\gr_{[\varphi]_\bell}\cL)_{\leq\psi}$, $(\gr_{[\varphi]_\bell}\cL)_{<\psi}$ and $\gr_\psi\gr_{[\varphi]_\bell}\cL$ decompose with respect to $\eta$ as above. Given $\eta\in\Phi$ and $\theta\in S^1$, the component $\gr_\eta\cL_\theta$ occurs in $\cL_{\leq\psi,\theta}\cap \cL_{\leq[\varphi]_\bell,\theta}$ if and only if $\eta\leqtheta\psi$ and $[\eta]_\bell\leqtheta[\varphi]_\bell$. Arguing similarly, we finally find that $\gr_\eta\cL_\theta$ occurs in $\gr_\psi\gr_{[\varphi]_\bell}\cL_\theta$ if and only if $\eta=\psi$ and $[\eta]_\ell=[\varphi]_\ell$, that is,
\[
\gr_\psi\gr_{[\varphi]_\bell}\cL_\theta=\begin{cases}
0&\text{if }[\psi]_\bell\neq[\varphi]_\bell,\\
\gr_\psi\cL_\theta&\text{if }[\psi]_\bell=[\varphi]_\bell.
\end{cases}
\]
This concludes the proof of \ref{prop:level}\eqref{prop:level2}. Now, \ref{prop:level}\eqref{prop:level3} is clear.
\end{proof}

\begin{remarque}\label{rem:stepconstr}
Let us explain the meaning and usefulness of this proposition. To a Stokes-filtered local system $(\cL,\cL_\bbullet)$ is associated a partially graded Stokes-filtered local system $(\gr_\bell\cL,(\gr_\bell\cL)_\bbullet)$. Going from $(\cL,\cL_\bbullet)$ to $(\gr_\bell\cL,(\gr_\bell\cL)_\bbullet)$ consists in
\begin{itemize}
\item
considering $(\cL,\cL_\bbullet)$ as indexed by $\ccP(\bell)$ (or $\Phi(\bell)$) and grading it as such,
\item
remembering the $\ccP$-filtration on the graded object, making it a Stokes-filtered local system as well.
\end{itemize}

Conversely, let $(\cG_\bell,\cG_{\bell,\bbullet})$ be a fixed Stokes-filtered local system graded at the level $\bell\geq0$, that is, the associated $\ccP(\bell)$-filtered local system is graded. As a consequence of the last statement of the proposition, an argument similar to that of Proposition \ref{prop:classifhausdorff} implies that the pointed set of isomorphism classes of Stokes-filtered local systems $(\cL,\cL_\bbullet)$ equipped with an isomorphism $f_\bell:(\gr_\bell\cL,(\gr_\bell\cL)_\bbullet)\isom(\cG_\bell,\cG_{\bell,\bbullet})$ is in bijection with the pointed set $H^1\big(S^1,\cAut^{<_{[\cbbullet]_\bell}0}(\cG_\bell)\big)$, where $\cAut^{<_{[\cbbullet]_\bell}0}(\cG_\bell)$ is the sheaf of automorphisms $\lambda$ of $(\cG_\bell,\cG_{\bell,\bbullet})$ such that $\gr_{[\varphi]_\bell}\lambda=\id$ on the local system $\gr_{[\varphi]_\bell}\cG_\bell$ for any $\varphi\in\ccP$ (equivalently, any $[\varphi]_\bell$ in $\ccP(\bell)$).

In particular, one can reconstruct $(\cL,\cL_\bbullet)$ from $\gr\cL$ either in one step, by specifying an element of $H^1(S^1,\cAut^{<0}\gr\cL)$, or step by step with respect to the level, by specifying at each step $\ell$ an element of $H^1\big(S^1,\cAut^{<_{[\cbbullet]_\bell}0}(\gr_\bell\cL)\big)$.
\end{remarque}

\subsection{\proofname\ of Theorem \ref{th:abelianwithout}}\label{pf:abelianwithout}
It will be done by induction, using Corollary \ref{cor:level} below for the inductive assumption on $\gr_\bell\cL$ of Proposition \ref{prop:level}\eqref{prop:level3}.

Let~$\Phi$ be a finite set in~$\ccP$. Assume $\#\Phi\geq2$. We then set $\bmm(\Phi)=\max\{\bmm(\varphi-\psi)\mid\varphi\neq\psi\in \Phi\}$, and we have $\bmm(\Phi)>0$. If $\#\Phi=1$, we set $\bmm(\Phi)=0$. We also set $\bell(\Phi)=\bmm(\Phi)-1$ if $\bmm(\Phi)>0$.

\begin{lemme}
Assume $\#\Phi\geq2$ and fix $\varphi_o\in\Phi$. Then
\[
\bmm(\Phi,\varphi_o)\defin\max\{\bmm(\varphi-\varphi_o)\mid\varphi\in \Phi\}=\bmm(\Phi).
\]
\end{lemme}

\begin{proof}
Let $\eta$ be the sum of common monomials to all $\varphi\in\Phi$. By replacing $\Phi$ with $\Phi-\eta$ we can assume that $\eta=0$. We will show that, in such a case, $\bmm(\Phi)=\max\{\bmm(\varphi)\mid\varphi\in\Phi\}=\max\{\bmm(\varphi-\varphi_o)\mid\varphi\in \Phi\}$. If all $\varphi\in\Phi$ have the same order $k>0$, two of them, say $\varphi_1$ and $\varphi_2$, do not have the same dominant monomial. Then $\bmm(\Phi)=k$. Therefore, the dominant monomial of a given $\varphi_o\in\Phi$ differs from that of~$\varphi_1$ or that of $\varphi_2$, so one of $\varphi_1-\varphi_o$ and $\varphi_2-\varphi_o$ has order $k$, and $\bmm(\Phi,\varphi_o)=k$.

Assume now that $\varphi_1$ has maximal order $k$ and $\varphi_2$ has order $<k$. Then $\bmm(\Phi)=k$. Given $\varphi_o\in\Phi$, then either $\varphi_o$ has order $k$, and then $\varphi_2-\varphi_o$ has order $k$, or $\varphi_o$ has order $<k$, and then $\varphi_1-\varphi_o$ has order $k$, so in any case we have $\bmm(\Phi,\varphi_o)=k$.
\end{proof}

Let us fix $\varphi_o\in \Phi$. In the following, we will work with $\Phi-\varphi_o$. We will assume that $\#\Phi\geq2$, that is, $m(\Phi)>0$ (otherwise the theorem is clear), and we set $\bell=\bell(\Phi)$ as above.

Let $\varphi\in \Phi-\varphi_o$. If $\bmm(\varphi)=\bmm(\Phi)$, then its image in $(\Phi-\varphi_o)(\bell)$ is nonzero. For every $\varphi\in \Phi-\varphi_o$, the subset $(\Phi-\varphi_o)_{[\varphi]_\bell}\defin\{\psi\in \Phi-\varphi_o\mid[\psi]_\bell=[\varphi]_\bell\}$ satisfies $\bmm((\Phi-\varphi_o)_{[\varphi]_\bell})\leq\bell<\bmm$. Indeed, if $[\varphi]_\bell=0$, then any $\psi\in (\Phi-\varphi_o)_{[\varphi]_\bell}$ can be written as $x^{-\bell}u_\psi(x)$ with $u_\psi(x)\in\cO$, and the difference of two such elements is written $x^{-\bell}v(x)$ with $v(x)\in\cO$. On the other hand, if $[\varphi]_\bell\neq0$, $\psi$ is written as $\varphi+x^{-\bell}u_\psi(x)$ and the same argument applies.

\begin{corollaire}[of Prop\ptbl\ref{prop:level}]\label{cor:level}
Let $(\cL,\cL_\bbullet)$ be a Stokes-filtered local system, let $\Phi''$ be a finite subset of~$\ccP$ containing $\Phi(\cL,\cL_\bbullet)$, set $\bmm=\bmm(\Phi'')$, $\bell=\bmm(\Phi'')-1$, and fix~$\varphi_o$ as above.

Then, for every $[\varphi]_{\bell}\in (\Phi''-\varphi_o)(\bell)$, $(\gr_{[\varphi]_{\bell}}\cL[-\varphi_o],(\gr_{[\varphi]_{\bell}}\cL[-\varphi_o])_\bbullet)$ is a Stokes-filtered local system and $\bmm(\gr_{[\varphi]_{\bell}}\cL[-\varphi_o],(\gr_{[\varphi]_{\bell}}\cL[-\varphi_o])_\bbullet)\leq\bell<\bmm$.\qed
\end{corollaire}

\begin{lemme}\label{lem:grandintervalle}
Let $(\cL,\cL_\bbullet)$ be a Stokes-filtered local system and let $\Phi$ be the set of its exponential factors. Let~$I$ be any open interval of~$S^1$ such that, for any $\varphi,\psi\in \Phi$, $\card(I\cap\nobreak\St(\varphi,\psi))\leq1$. Then the decompositions \eqref{eq:L<varphi} hold on~$I$.
\end{lemme}

\begin{proof}
It is enough to show that $H^1(I,\cL_{<\varphi})=\nobreak0$ for any $\varphi\in\Phi$. Indeed, arguing as in the proof of Proposition \ref{prop:stokeswithout}, if we restrict to such a~$I$, we can lift for any~$\varphi$ a basis of global sections of $\gr_\varphi\cL$ as sections of $\cL_{\leq\varphi}$ and get an injective morphism $\gr_\varphi\cL\to\nobreak\cL_{\leq\varphi}$. We therefore obtain a morphism $\bigoplus_{\varphi\in \Phi}\gr_\varphi\cL\to\cL$ sending $\gr_\varphi\cL$ into $\cL_{\leq\varphi}$ for each $\varphi\in\Phi$. This is an isomorphism: indeed, as both sheaves are local systems, it is enough to check this at some $\theta\in I$; considering a splitting of~$\cL$ at $\theta$, the matrix of this morphism is block-triangular with respect to the order at~$\theta$ and the diagonal blocks are equal to the identity.

It remains to show that this morphism induces the splitting for each $\cL_{\leq\eta}$. For every~$\eta$, we have a natural inclusion $\beta_{\varphi\leq\eta}\cL_{\leq\varphi}\hto\cL_{\leq\eta}$, hence the previous isomorphism induces a morphism $\bigoplus_{\varphi\in \Phi}\beta_{\varphi\leq\eta}\gr_\varphi\cL\to\cL_{\leq\eta}$, which is seen to be an isomorphism on stalks.

Let us now show that $H^1(I,\cL_{<\varphi})=\nobreak0$. It will be easier to argue by duality, that is, to show that $H^0_c(I,\bD'(\cL_{<\varphi}))=\nobreak0$ (\cf Lemma \ref{lem:dualS1}, from which we will keep the notation). There is a local decomposition
\begin{equation}\label{eq:L<varphivee}
\bD'(\cL_{<\varphi})_{|\nb(\theta_o)}\simeq\bigoplus_{\psi\in\Phi}\alpha_{\psi<\varphi}(\gr_\psi\cL)^\vee_{|\nb(\theta_o)}.
\end{equation}
Let us set $I=(\theta_0,\theta_{n+1})$ and let us denote by $\theta_1,\dots,\theta_n$ the successive Stokes directions in~$I$. We set $I_i=(\theta_{i-1},\theta_{i+1})$ for $i=1,\dots,n$. According to the first part of the proof and to Lemma \ref{lem:H1vanishing}, the decomposition \eqref{eq:L<varphivee} holds on each~$I_i$ and, on $I_i\cap I_{i+1}$, two decompositions \eqref{eq:L<varphivee} are related by $\lambda^{(i,i+1)}$ whose component $\lambda^{(i,i+1)}_{\psi,\eta}:\alpha_{\psi<\varphi}(\gr_\psi\cL)^\vee_{|I_i\cap I_{i+1}}\to\alpha_{\eta<\varphi}(\gr_\eta\cL)^\vee_{|I_i\cap I_{i+1}}$ is nonzero only if $\psi\leq_{_{I_i\cap I_{i+1}}}\nobreak\eta$, and is equal to $\id$ if $\psi=\eta$.

Assume that $s\in H^0_c(I,\bD'(\cL_{<\varphi}))$ is nonzero. For any $i=1,\dots,n$, we denote by~$s_i$ its restriction to $I_i$ and by $s_{i,\psi}$ its component on $\alpha_{\psi<\varphi}(\gr_\psi\cL)^\vee_{|I_i}$ for the chosen decomposition on $I_i$. Therefore, $s_{i,\psi}$ and $s_{i+1,\psi}$ may differ on $I_i\cap I_{i+1}$. Let us note, however, that the set $\Phi$ is totally ordered by $\leq_{_{I_i\cap I_{i+1}}}$ on $I_i\cap I_{i+1}$, and if we set $\psi_i=\min\{\psi\in \Phi\mid s_{i,\psi}\neq0\text{ on } I_i\cap I_{i+1}\}$ (when defined, \ie if $s_i\not\equiv0$ on $I_i\cap\nobreak I_{i+1}$), then we also have \hbox{$\psi_i=\min\{\psi\in \Phi\mid s_{i+1,\psi}\neq0\text{ on } I_i\cap I_{i+1}\}$}. Indeed, let us denote by $\psi'_i$ the right-hand term.
\begin{itemize}
\item
On the one hand, $s_{i+1,\psi_i}=\sum_{\psi\leq\psi_i}\lambda^{(i,i+1)}_{\psi,\psi_i}(s_{i,\psi})=s_{i,\psi_i}$ on $I_i\cap I_{i+1}$ since $\lambda^{(i,i+1)}_{\psi,\psi_i}=0$ if $\psi<\psi_i$ on $I_i\cap I_{i+1}$, and $\lambda^{(i,i+1)}_{\psi_i,\psi_i}=\id$. Therefore, $\psi'_i\leq\psi_i$ on $I_i\cap I_{i+1}$.
\item
On the other hand, if $\psi'_i<\psi_i$ on $I_i\cap I_{i+1}$, then $s_{i+1,\psi'_i}=\sum_{\psi\leq\psi'_i}\lambda^{(i,i+1)}_{\psi,\psi'_i}(s_{i,\psi})=0$, a contradiction.
\end{itemize}

Since $s$ is compactly supported on $I$, its restrictions to $(\theta_0,\theta_1)$ and to $(\theta_n,\theta_{n+1})$ vanish identically. In the following, we assume that $s_1$ and $s_n$ are not identically zero (otherwise, we just forget these Stokes directions and consider the first and last one for which $s_i$ is not identically zero). Then $\psi_1$ is defined, and we have
\begin{itemize}
\item
$\psi_1<\varphi$ on $I_1\cap I_2$ since $s_{1,\psi_1}\neq0$ on $I_1\cap I_2$,
\item
$\theta_1\in\St(\psi_1,\varphi)$ since $s_{1,\psi_1}=0$ on $(\theta_0,\theta_1)$,
\item
$\varphi<\psi_1$ on $(\theta_0,\theta_1)$ since $\theta_1\in\St(\psi_1,\varphi)$,
\item
and $\psi_1<\varphi$ on $(\theta_1,\theta_{n+1})$, as it is so on $(\theta_1,\theta_2)$ and there is no other element of $\St(\psi_1,\varphi)$ in $(\theta_1,\theta_{n+1})$.
\end{itemize}

It follows that $s_{2,\psi_1}\not\equiv0$ on $I_1\cap I_2$. Moreover, $\theta_2\not\in\St(\psi_1,\varphi)$, so $s_{2,\psi_1}\neq0$ on~$I_2$ (being a section of a local system on this interval). Therefore, $s_2\neq0$ on $I_2\cap I_3$, $\psi_2$ is defined and (by definition) $\psi_2\leq\psi_1$ on $I_2\cap I_3$, and there is a component $s_{2,\psi_2}$ on~$I_2$ which is not identically zero. In conclusion, on $I_2\cap I_3$, the following holds:
\begin{itemize}
\item
$\psi_2\leq\psi_1<\varphi$,
\item
$s_{2,\psi_2}\neq0$.
\end{itemize}
We claim that $\psi_2<\varphi$ on $(\theta_2,\theta_{n+1})$ (and in particular $\theta_3\notin\St(\psi_2,\varphi)$):
\begin{itemize}
\item
this is already proved if $\psi_2=\psi_1$;
\item
assume therefore that $\psi_2<\psi_1$ on $I_2\cap I_3$; then,
\begin{itemize}
\item
if $s_{2,\psi_2}\neq0$ on $I_1\cap I_2$, then $\psi_1<\psi_2$ on $I_1\cap I_2$ (by definition of $\psi_1$) and so $\theta_2\in\St(\psi_1,\psi_2)$, hence there is no other element of $\St(\psi_1,\psi_2)$ belonging to $(\theta_2,\theta_{n+1})$ and therefore also $\psi_2<\varphi$ since $\psi_1<\varphi$ on $(\theta_1,\theta_{n+1})$;
\item
if $s_{2,\psi_2}=0$ on $I_1\cap I_2$, this means that $\theta_2\in\St(\psi_2,\varphi)$, hence there is no other element of $\St(\psi_2,\varphi)$ belonging to $(\theta_2,\theta_{n+1})$, so we keep $\psi_2<\varphi$ on $(\theta_2,\theta_{n+1})$.
\end{itemize}
\end{itemize}

Continuing in the same way, we find that $\psi_n$ is defined, which means that $s_n$ is not identically zero on $(\theta_n,\theta_{n+1})$, a contradiction.
\end{proof}

\begin{remarque}\label{rem:H1leq}
For $I\neq S^1$ as in Lemma \ref{lem:grandintervalle}, we also have $H^1(I,\cL_{\leq\varphi})=0$ for any~$\varphi$. This follows from the exact sequence
\[
0=H^1(I,\cL_{<\varphi})\to H^1(I,\cL_{\leq\varphi})\to H^1(I,\gr_\varphi\cL)=0,
\]
where the last equality holds because $\gr_\varphi\cL$ is a local system.
\end{remarque}

\begin{corollaire}\label{cor:trivI}
In the setting of Corollary \ref{cor:level}, let~$I$ be any open interval of~$S^1$ such that, for any $\varphi,\psi\in \Phi$, $\card(I\cap\nobreak\St(\varphi,\psi))\leq1$. Then
\[
(\cL,\cL_\bbullet)_{|I}\simeq(\gr_\bell\cL,(\gr_\bell\cL)_\bbullet)_{|I}.
\]
\end{corollaire}

\begin{proof}
By Proposition \ref{prop:level}\eqref{prop:level3}, the set of exponential factors of $(\cL,\cL_\bbullet)$ and $(\gr_\bell\cL,(\gr_\bell\cL)_\bbullet)$ coincide, hence $I$ satisfies the assumption of Lemma \ref{lem:grandintervalle} for both $(\cL,\cL_\bbullet)$ and $(\gr_\bell\cL,(\gr_\bell\cL)_\bbullet)$. Therefore, when restricted to $I$, both are isomorphic to the graded Stokes-filtered local system determined by $\gr\cL$ restricted to~$I$.
\end{proof}

\begin{proof}[End of the proof of Theorem \ref{th:abelianwithout}]
Let $\lambda:(\cL,\cL_\bbullet)\to(\cL',\cL'_\bbullet)$ be a morphism of Stokes-filtered local systems. The proof will be done by induction on $\bmm=\bmm(\Phi'')$, with $\Phi''=\Phi\cup \Phi'$. The result is clear if $\bmm=0$ (so $\#\Phi''=1$), as both Stokes-filtered local systems have only one jump.

Assume now $\bmm\neq0$. We fix $\varphi_o\in \Phi''$ as in Corollary \ref{cor:level}. It is enough to prove the assertion for $\lambda:\cL[-\varphi_o]\to\cL'[-\varphi_o]$. We shall assume that $\varphi_o=0$ in order to simplify the notation below. Let~$I$ be any open interval of~$S^1$ of length $\pi/\bmm$ with no Stokes points (of $\Phi''$) as end points. According to the definition of $m$ and~$\ell$, and to Corollary \ref{cor:trivI}, $\lambda_{|I}$ decomposes as $\bigoplus_{[\psi]_\bell,[\psi']_\bell}\lambda_{|I,[\psi]_\bell,[\psi']_\bell}$, with $\lambda_{|I,[\psi]_\bell,[\psi']_\bell}:\gr_{[\psi]_\bell}\cL_{|I}\to\gr_{[\psi']_\bell}\cL'_{|I}$, where $[\psi]_\bell$ runs in $\Phi(\bell)$ and $[\psi']_\bell$ in $\Phi'(\bell)$ (otherwise, it is zero). Each $\lambda_{|I,[\psi]_\bell,[\psi']_\bell}$ is a morphism of (constant) local systems, hence is constant.

Let us show that $\lambda_{|I,[\psi]_\bell,[\psi']_\bell}=0$ if $[\psi]_\bell\neq[\psi']_\bell$. We denote by $\ccP_{[\psi]_\bell}$ the set of $\eta\in\ccP$ such that $[\eta]_\bell=[\psi]_\bell$. According to the proof of Corollary \ref{cor:trivI}, $\lambda_{|I,[\psi]_\bell,[\psi']_\bell}$ itself decomposes as blocks with components $\lambda_{|I,\eta,\eta'}$ where $\eta$ varies in $\ccP_{[\psi]_\bell}$ and $\eta'$ in $\ccP_{[\psi']_\bell}$. We will show that each such block is zero.

As in restriction to $I$,~$\lambda$ is (isomorphic to) a filtered morphism $\gr_\bell\cL\to\gr_\bell\cL'$, it sends $(\gr_\bell\cL)_{\leq\varphi}$ to $(\gr_\bell\cL')_{\leq\varphi}$ for any $\varphi$, that is, after the proof of Corollary \ref{cor:trivI}, it sends $\bigoplus_{[\psi]_\bell}\bigoplus_{\eta\in\ccP_{[\psi]_\bell}}\beta_{\eta\leq\varphi}\gr_\eta\cL$ in $\bigoplus_{[\psi']_\bell}\bigoplus_{\eta'\in\ccP_{[\psi']_\bell}}\beta_{\eta'\leq\varphi}\gr_{\eta'}\cL$, so the block $\lambda_{|I,\eta,\eta'}$ sends $\beta_{\eta\leq\varphi}\gr_\eta\cL$ in $\beta_{\eta'\leq\varphi}\gr_{\eta'}\cL$. Choosing $\varphi=\eta$ shows that $\lambda_{|I,\eta,\eta'}$ sends $\gr_\eta\cL$ in $\beta_{\eta'\leq\eta}\gr_{\eta'}\cL$. Now, if $[\psi]_\bell\neq[\psi']_\bell$, we have $\eta-\eta'=x^{-\bmm}u(x)$ with $u(0)\neq0$ and,~$I$ being of length $\pi/\bmm$, contains at least one Stokes point of $\St(\eta,\eta')$. Therefore $\lambda_{|I,\eta,\eta'}$ vanishes on some non-empty open subset of $I$, hence everywhere.

In conclusion, $\lambda_{|I}$ is diagonal with respect to the $[\cbbullet]_{\bell}$-decomposition, that is, it coincides with the graded morphism $\gr_{[\cbbullet]_\bell}\lambda$, graded with respect to the filtration $\cL_{[\cbbullet]}$ of level $\geq\bell$.

By Proposition \ref{prop:level}\eqref{prop:level2} and Corollary \ref{cor:level}, the inductive assumption can be applied to $\gr_{[\cbbullet]_\bell}\lambda$. By induction we conclude that~$\lambda$ can be diagonalized (with respect to the decomposition in the third line of \eqref{eq:L<varphi}) in the neighbourhood\footnote{In fact, the argument shows that the diagonalization holds on $I$.} of any point of~$I$. As this holds for any such~$I$, this ends the proof of Theorem \ref{th:abelianwithout}.
\end{proof}

\begin{corollaire}[of Lemma \ref{lem:grandintervalle}]
Let $(\cL,\cL_\bbullet)$ be a Stokes-filtered local system. Then, for any $\varphi\in\ccP_x$, $\Gamma(S^1,\cL_{<\varphi})=0$.
\end{corollaire}

\begin{proof}
Let us keep the notation used above, which was introduced at the beginning of \S\ref{pf:abelianwithout}. It is easy to reduce to the case where $(\cL,\cL_\bbullet)$ is non-ramified, that we assume to hold below (\cf Remark \ref{rem:Stokeswithramif}\eqref{rem:Stokeswithramif2} for the definition of $\cL_{<\varphi}$ in the ramified case).

We will argue by decreasing induction on $m(\Phi)$. If $m(\Phi)=0$, so $\Phi=\{\varphi_o\}$, we have $\cL_{<\varphi}=\beta_{\varphi_o<\varphi}\cL$ (\cf Example \ref{exem:triviaux}\eqref{exem:triviaux1}), and the assertion is clear.

Assume now $m(\Phi)>0$, that is, $\#\Phi\geq2$. By twisting it is enough to show $\Gamma(S^1,\cL_{<0})=0$. We set $\bell=\bell(\Phi)$ as above.

There is a natural inclusion $\cL_{\leq0}\hto\cL_{\leq[0]_\bell}$, hence a map $\cL_{\leq0}\to\gr_{[0]_\bell}\cL$ sending $\cL_{\leq0}$ to $(\gr_{[0]_\bell}\cL)_{\leq0}$ and $\cL_{<0}$ to $(\gr_{[0]_\bell}\cL)_{<0}$. It is enough to show that the induced map $\Gamma(S^1,\cL_{<0})\to\Gamma(S^1,(\gr_{[0]_\bell}\cL)_{<0})$ is injective. Indeed, we can apply the inductive assumption to the Stokes-filtered local system $(\gr_{[0]_\bell}\cL,(\gr_{[0]_\bell}\cL)_\bbullet)$ to conclude that $\Gamma(S^1,(\gr_{[0]_\bell}\cL)_{<0})=0$.

Let $I$ be an interval of length $\pi/m(\Phi)$ with no Stokes directions for $\Phi\cup\{0\}$ at its boundary. We will show that $\Gamma(I,\cL_{<0})\to\Gamma(I,(\gr_{[0]_\bell}\cL)_{<0})$ is injective for any such~$I$, which is enough. By Lemma \ref{lem:grandintervalle}, we have $\cL_{<0|I}\simeq\bigoplus_{\varphi\in\Phi}\beta_{\varphi<0}\gr_\varphi\cL_{|I}$, and the map $\cL_{<0|I}\to(\gr_{[0]_\bell}\cL)_{<0|I}$ is the projection to the sum taken on $\varphi\in[0]_\bell$, that is, the $\varphi\in\Phi$ having a pole of order $<m(\Phi)$. But since $I$ contains a Stokes direction for $(\varphi,0)$ for any $\varphi\not\in[0]_\bell$, we have $\Gamma(I,\beta_{\varphi<0}\gr_\varphi\cL_{|I})=0$ for any such $\varphi$. We conclude that $\Gamma(I,\cL_{<0|I})\to\to\Gamma(I,(\gr_{[0]_\bell}\cL)_{<0})$ is isomorphic to the identity map from $\Gamma\big(I,\bigoplus_{\varphi\in[0]_\bell}\beta_{\varphi<0}\gr_\varphi\cL_{|I}\big)$ to itself, by applying the same reasoning to $(\gr_{[0]_\bell}\cL,(\gr_{[0]_\bell}\cL)_\bbullet)$.
\end{proof}

\begin{remarque}
The dimension of the only non-zero cohomology space $H^1(S^1,\cL_{<0})$ is usually called the \emphb{irregularity number} of $(\cL,\cL_\bbullet)$.
\end{remarque}

\begin{corollaire}
Let $f:(\cL,\cL_\bbullet)\to(\cL',\cL'_\bbullet)$ be a morphism of Stokes-filtered local systems on~$S^1$ sending $\cL_{\leq}$ to $\cL'_<$. Then $f$ is equal to zero.
\end{corollaire}

\begin{proof}
Indeed, $f$ is a section on~$S^1$ of $\cHom\big((\cL,\cL_\bbullet),(\cL',\cL'_\bbullet)\big)_{<0}$.
\end{proof}

\subsection{\proofname\ of the stability by extension}
We now prove the second part of Theorem \ref{th:abelian}. We consider an exact sequence
\[
0\to(\cL',\cL'_\bbullet)\to(\cL,\cL_\bbullet)\to(\cL'',\cL''_\bbullet)\to0,
\]
where $(\cL',\cL'_\bbullet),(\cL'',\cL''_\bbullet)$ are Stokes-filtered local systems, but no assumption is made on $(\cL,\cL_\bbullet)$, except that $(\cL,\cL_\bbullet)$ is a pre-Stokes-filtered sheaf in the sense of Definition \ref{def:preStokesnonramif} (or the ramified analogue). Firstly, the extension~$\cL$ of $\cL'$ and $\cL''$ is clearly a local system. It is now a matter of finding a ramification $\rho$ such that the pre-Stokes filtration on $\wt\rho^{-1}\cL$ is non-ramified and is a Stokes filtration. It is therefore enough to assume that $(\cL',\cL'_\bbullet),(\cL'',\cL''_\bbullet)$ are non-ramified.

Let us now restrict to an interval~$I$ as in Lemma \ref{lem:grandintervalle}, with $\Phi=\Phi'\cup\Phi''$. Since $H^1(I,\cL'_{\leq \psi})=0$ for any $\psi\in\ccP$ (\cf Remark \ref{rem:H1leq}), there is a section of $\cL_{\leq\psi}\to\nobreak\cL''_{\leq\psi}$. We will make these sections compatible with respect to $\psi$, that is, coming from a morphism of pre-Stokes-filtered local systems. According to the decomposition of $(\cL'',\cL''_\bbullet)$ on~$I$, we obtain an injective morphism $\gr_\varphi\cL''\hto\cL_{\leq\varphi}$ for any $\varphi\in\nobreak\Phi''$. By the filtration property, it defines for any $\psi\in\ccP$ an injective morphism $\beta_{\varphi\leq\psi}\gr_\varphi\cL''\hto\nobreak\cL_{\leq\psi}$ and then a section $\cL''_{\leq\psi}=\bigoplus_{\varphi\in \Phi''}\beta_{\varphi\leq\psi}\gr_\varphi\cL''\to\cL_{\leq\psi}$ of the projection $\cL_{\leq\psi}\to\cL''_{\leq\psi}$, which is now a morphism of pre-Stokes filtrations. We therefore get an isomorphism $(\cL,\cL_\bbullet)_{|I}\simeq(\cL',\cL'_\bbullet)_{|I}\oplus(\cL'',\cL''_\bbullet)_{|I}$. In particular, $(\cL,\cL_\bbullet)$ is a Stokes filtration when restricted to~$I$.

As this holds for any such~$I$, this proves the assertion.\qed

\subsection{More on the level structure}\label{subsec:morelevel}
It will be useful to understand the level structure in an intrinsic way when the Stokes-filtered local system is possibly ramified. We will therefore use the presentation and the notation of Remark \ref{rem:Euler}.

Let $\ell$ be any non-negative rational number. We define the subsheaf $(\wtj_*\cO_{X^*})(\ell)\subset\wtj_*\cO_{X^*}$ as the subsheaf of germs $\varphi$ such that $x^\ell\varphi$ is a local section of $(\wtj_*\cO_{X^*})^\lb$, for some (or any) local determination of $x^\ell$. We have, for any $\ell\leq\ell'$,
\[
(\wtj_*\cO_{X^*})^\lb=(\wtj_*\cO_{X^*})(0)\subset(\wtj_*\cO_{X^*})(\ell)\subset(\wtj_*\cO_{X^*})(\ell')\subset\wtj_*\cO_{X^*}.
\]
We set $\ccI_1(\ell)=\wt\ccI_1/\wt\ccI_1\cap(\wtj_*\cO_{X^*})(\ell)$. We have a natural morphism $\ccI_1=\ccI_1(0)\to\ccI_1(\ell)$ which is compatible with the order (same proof as for Lemma \ref{lem:IIellorder}). We have a commutative diagram
\[
\xymatrix{
\ccIet_1\ar[r]^-{q_\ell}\ar[dr]_{\mu}&\ccI_1(\ell)^\et\ar[d]^{\mu_\ell}\\
&S^1
}
\]

The map $q_\ell$ is étale and onto. If $\Sigma\subset\ccIet_1$ is a finite covering of~$S^1$ (via $\mu$), then $\Sigma(\ell)\defin q_\ell(\Sigma)$ is a finite covering of~$S^1$ (via $\mu_\ell$), and $q_\ell:\Sigma\to\Sigma(\ell)$ is also a finite covering. Moreover, since $q_\ell:\ccIet_{1|\Sigma(\ell)}\to\Sigma(\ell)$ is étale, it corresponds to a sheaf of ordered abelian groups $\ccI_{1,\Sigma(\ell)}$ on $\Sigma(\ell)$, and $\Sigma$ can be regarded as a finite covering of $\Sigma(\ell)$ contained in $(\ccI_{1,\Sigma(\ell)})^\et$.

\begin{proposition}[Intrinsic version of Prop\ptbl\ref{prop:level}]
Let $(\cL,\cL_\bbullet)$ be a (possibly ramified) Stokes-filtered local system on~$S^1$ (in particular we have a subsheaf $\cL_\leq$ of $\mu^{-1}\cL$) and let $\Sigma\subset\ccIet_1$ be the support of $\gr\cL$. For $\ell\in\QQ_+$,
\begin{enumerate}
\item
the subsheaf $\Tr_{\leq q_\ell}(\cL_\leq,\mu_\ell^{-1}\cL)$ of $\mu_\ell^{-1}\cL$ (\cf Lemma \ref{lem:traceqleq}) defines a pre-$\ccI_1(\ell)$-filtration of $\cL$ which is in fact a $\ccI_1(\ell)$-filtration, that we denote by $(\cL,\cL_{[\cbbullet]_\ell})$;
\item
the sheaf $\gr_{[\cbbullet]_\ell}\cL$ is supported on $\Sigma(\ell)\defin q_\ell(\Sigma)$ and is a local system on $\Sigma(\ell)$;
\item
the sheaf $\cL_\leq$ induces a pre-$\ccI_{1,\Sigma(\ell)}$-filtration on $q_\ell^{-1}\gr_{[\cbbullet]_\ell}\cL$, which is a $\ccI_{1,\Sigma(\ell)}$-filtration denoted by $(\gr_{[\cbbullet]_\ell}\cL,(\gr_{[\cbbullet]_\ell}\cL)_\bbullet)$, and $\gr(\gr_{[\cbbullet]_\ell}\cL)$ is a local system on $\Sigma$; this local system has the same rank as the local system $\gr\cL$.
\end{enumerate}

Conversely, let $\cL_\leq\subset\mu^{-1}\cL$ be a pre-$\ccI_1$-filtration of $\cL$. Let us assume that $\Tr_{\leq q_\ell}(\cL_\leq,\mu_\ell^{-1}\cL)$ is a $\ccI_1(\ell)$-filtration of $\cL$ and let us denote by $\Sigma(\ell)$ the support of $\gr_{[\cbbullet]_\ell}\cL$. Let us also assume that the pre-$\ccI_{1,\Sigma(\ell)}$-filtration induced by $\cL_\leq$ on the local system $\gr_{[\cbbullet]_\ell}\cL$ (on $\Sigma(\ell)$) is a $\ccI_{1,\Sigma(\ell)}$-filtration, and let us denote by $\Sigma$ the support of $\gr(\gr_{[\cbbullet]_\ell}\cL)$. Then $\cL_\leq$ is a $\ccI_1$-filtration of $\cL$ and $\gr\cL$ is supported on $\Sigma$ and has the same rank as $\gr(\gr_{[\cbbullet]_\ell}\cL)$.\qed
\end{proposition}

\subsection{Comments}\label{subsec:comments3}
The notion of level structure has been implicitly present for a long time in the theory of meromorphic differential equations (\cf \cite{Jurkat78,B-J-L79}). It appeared as the ``dévissage Gevrey'' in \cite{Ramis78}. Later, it has been an important tool for the construction of the moduli space of meromorphic connections \cite{B-V89}, as emphasized in \cite{Deligne86b}, and has been strongly related to the multisummation property of solutions of differential equations (\cf\eg\cite{M-R92} and the references given therein). Note also that the basic vanishing property proved during the proof of Lemma \ref{lem:grandintervalle} is already present in \cite{Deligne86b} and explained in the one-slope case in \cite[\S5]{Malgrange83b}, \cf also \cite[Lemma 4.3]{M-R92}. It is related to the Watson lemma (see \loccit). It~is also related to the existence of asymptotic solutions of a differential equation in ``large'' Stokes sectors, see \eg \cite{B-J-L79} and in particular \S B for historical references.

This level structure is not a filtration of the filtered local system. It is the structure induced by a suitable filtration of the indexing sheaf $\ccI$. A visually similar structure appears in knot theory when considering iterated torus knots, for instance the knots obtained from a singularity of an irreducible germ of complex plane curve. The algebraic structure related to the level structure is the ``dévissage'' by Gevrey exponents or the multi-summability property. Similarly, the algebraic structure related to iterated torus knots of singularities of plane curves are the Puiseux approximations.
\let\bell\oldbell
\let\bmm\oldbmm

\chapter{Stokes-perverse~sheaves on Riemann~surfaces}\label{chap:Stokesone-pervers}

\begin{sommaire}
In this \chaptername, we introduce the global notion of Stokes-perverse sheaves on a Riemann surface. The Riemann-Hilbert correspondence of the next \chaptersname will be an equivalence between holonomic $\cD$-modules on the Riemann surface and Stokes-perverse sheaves on it. If $\kk$ is a subfield of $\CC$, this allows one to speak of a $\kk$-structure on a holonomic $\cD$-module when the corresponding Stokes-perverse is defined over $\kk$.
\end{sommaire}

\subsection{Introduction}
Let $X$ be a Riemann surface. Meromorphic connections on holomorphic vector bundles, with poles on a discrete set of points~$D$, form a classic subject going back to the nineteenth century. The modern approach extends the interest to arbitrary holonomic $\cD_X$-modules (\cf \S\ref{subsec:somebasic} in the next \chaptername) and also extends it to arbitrary dimensions. A fundamental result is that the de~Rham complex (or the solution complex) of such a $\cD_X$-module is a perverse constructible complex (also called a perverse sheaf). Moreover, the Riemann-Hilbert correspondence for connections with regular singularities, \resp for regular holonomic $\cD_X$-modules, induces an equivalence between the category of such objects and that of local systems on $X\moins D$, \resp of perverse sheaves. For Riemann surfaces, a modern presentation of this equivalence has been explained in \cite{Malgrange91}.

Perverse sheaves on a Riemann surface are not difficult to understand. For a given perverse sheaf $\cF$, there is a discrete set of points $D$ such that $\cF_{|X\moins D}$ is a local system shifted by one (that is, a complex whose only cohomology sheaf is $\cH^{-1}\cF_{|X\moins D}$ and is a local system on $X\moins D$). The supplementary data to determine the perverse sheaf~$\cF$ are complex vector spaces equipped with an automorphism, one for each point of~$D$, and gluing data around each point of $D$ (\cf\eg\cite{Malgrange91}).

A perverse sheaf has a Jordan-Hölder sequence, whose successive quotients (in the perverse sense) are simple perverse sheaves, that is, either $j_*\cL[1]$ for some irreducible local system $\cL$ on $X\moins D$ (and $j:X\moins D\hto X$ is the open inclusion), or $i_{x,*}\CC$, where $i_x:\{x\}\hto X$ is the closed inclusion. Examples of perverse sheaves are $j_!\cL[1]$, $j_*\cL[1]$, $\bR j_*\cL[1]$, where $\cL$ is a local system on $X\moins D$, and $j_*\cL[1]$ is known as the \emphb{intermediate (or middle) extension} of $\cL[1]$, while $\bR j_*\cL[1]$ is in some sense the maximal one and $j_!\cL[1]$ the minimal one.

In this \chaptername, we construct a category of Stokes-perverse sheaves, which will be the receptacle of the Riemann-Hilbert correspondence for possibly irregular holonomic $\cD_X$-modules, which will be explained in the next \chaptername. It will now be convenient to use the formalism introduced in \Chaptersname \ref{chap:Ifil}. The idea is to replace the data on $(X\moins D,D)$ with data on $(X\moins D,\ccIet)$, where $\ccI$ is defined as in \S\ref{subsec:Stokeswithramif} for each point of~$D$.

The constructions in this \chaptersname are taken from \cite{Deligne78b,Deligne84cc,Malgrange91}.

\subsection{The setting}\label{subsec:settingStPerv}
Let~$X$ be a Riemann surface and let~$D$ be a discrete set of points of~$X$. We set $X^*=X\moins D$. We denote by \index{$XWTD$@$\wt X(D)$}$\varpi:\wt X(D)\to X$ the \index{real blow-up}\index{real blow-up!along a divisor}\emphb{real oriented blowing-up} of~$X$ at each of the points of~$D$. When~$D$ is fixed, we set $\wt X=\wt X(D)$. Locally near a point of~$D$, we are in the situation of Example \ref{exem:Stokes}. We set \index{$XWTDP$@$\partial\wt X(D)$}$\partial\wt X=\varpi^{-1}(D)$ and we denote by $j_\partial:X^*\hto\wt X$ and $i_\partial:\partial\wt X\hto\wt X$ the inclusions. We denote by~$\ccI$ the sheaf which is zero on~$X^*$ and is equal to $\bigcup_d\ccI_{d,x_o}$ on\footnote{The notation $S^1_{x_o}$ should not be confused with the notation $S^1_x$ of \S\ref{subsec:pushpull}.} $S^1_{x_o}\defin\varpi^{-1}(x_o)$ for $x_o\in D$, where $\ccI_{d,x_o}$ is the sheaf introduced in Remark \ref{rem:Euler}. Notice that $\ccIet$ is not Hausdorff, since we consider~$\ccI$ as a sheaf on $\wt X$, not on $\partial\wt X$. Since $\mu:\ccIet\to\wt X$ is a homeomorphism above~$X^*$, we shall denote by $\wtj:X^*=\ccIet_{|X^*}\hto\ccIet$ the open inclusion, by $\wti:\ccIet_{|\partial\wt X}\hto\ccIet$ the closed inclusion, so that $\mu\circ\wtj=j_\partial$ and we have a commutative diagram
\begin{equation}\label{eq:diagIet}
\begin{array}{c}
\xymatrix{
\ccIet_{|\partial\wt X}\ar[d]_{\mu_\partial}\ar@<-.5ex>@{^{ (}->}[r]^-{\wti}&\ccIet\ar[d]^\mu\\
\partial\wt X\ar@<-.5ex>@{^{ (}->}[r]^-{i_\partial}&\wt X
}
\end{array}
\end{equation}

\begin{remarque}[Sheaves on $\ccIet$]\label{rem:sheavesIet}
The sheaf~$\ccI$ does not satisfy the Hausdorff property only because it is zero on $X^*$, but it is Hausdorff when restricted to $\partial\wt X$. It may therefore be convenient to describe a sheaf $\cF_\leq$ on $\ccIet$, up to unique isomorphism inducing the identity on $X^*$ and $\partial\wt X$, as a triple $(\cF^*,\cL_\leq,\nu)$ consisting of a sheaf $\cF^*$ on $X^*$, a sheaf $\cL_\leq$ on $\ccIet_{|\partial\wt X}$, and a morphism $\nu:\cL_\leq\to\wti{}^{-1}\wtj_*\cF^*$. Note that, since $\mu$ is a local homeomorphism, we have $\wti{}^{-1}\wtj_*\cF^*=\mu_\partial^{-1}(i_\partial^{-1}j_{\partial,*}\cF^*)$.

In a similar way, a pre-$\ccI$-filtration $\cF_\leq$ on $\ccIet$ determines two pre-$\ccI_{\partial\wt X}$-filtrations, namely $\cL_\leq=\wti{}^{-1}\cF_\leq$ and $\wti{}^{-1}\wtj_*\cF^*$ (as a constant pre-$\ccI$-filtration, \cf \S\ref{subsec:catpreIfilt}), and the natural morphism $\nu:\cL_\leq\to\wti{}^{-1}\wtj_*\cF^*$ is a morphism of pre-$\ccI_{\partial\wt X}$-filtrations. As a consequence, a pre-$\ccI$-filtration $\cF_\leq$ on $\ccIet$ corresponds in a one-to-one way to a triple $(\cF^*,\cL_\leq,\nu)$, where $\cL_\leq$ is a pre-$\ccI_{\partial\wt X}$-filtration and $\nu$ is a morphism from $\cL_\leq$ to the constant pre-$\ccI_{\partial\wt X}$-filtration $\wti{}^{-1}\wtj_*\cF^*$. Then one can use Remark \ref{rem:sheafonccI} to describe $\cL_\leq$.
\end{remarque}

\subsection{The category of Stokes-$\CC$-constructible sheaves on $\protect\wt X$}\label{subsec:Stcons}
The category of $\St$-$\CC$-constructible sheaves will be a subcategory of the category $\Mod(\kk_{\ccIet,\leq})$ of pre-$\ccI$-filtrations introduced in \S\ref{subsec:catpreIfilt}, and we use the notation introduced there, with~$\ccI$ as above. Recall that $\CC$-constructible means that the stratification giving constructibility is $\CC$-analytic in~$X$. The coefficient field is $\kk$.

Let $\cF^*$ be a $\CC$-constructible sheaf on~$X^*$ with singularity set~$S$. We assume that~$S$ is locally finite on~$X$, so that it does not accumulate on~$D$. Then $\cF^*$ is a locally constant sheaf of finite dimensional $\kk$-vector spaces on $X^*\moins S$, and in particular in the punctured neighbourhood of each point of~$D$. This implies that $\bR j_{\partial,*}\cF^*=j_{\partial,*}\cF^*$ is a local system on $\wt X\moins S$, as well as its restriction to each $S^1_{x_o}$, $x_o\in D$. Similarly, $\bR \wtj_*\cF^*=\wtj_*\cF^*$ is a local system on $\ccIet\moins S$, since the property is local on $\ccIet$ and $\mu$ is a local homeomorphism. Moreover, the adjunction morphism $\mu^{-1}\mu_*\to\id$ induces a morphism $\mu^{-1}j_{\partial,*}\cF^*\to\wtj_*\cF^*$, which is easily seen to be an isomorphism. We will identify both sheaves.

We will set $\cL=i_\partial^{-1}j_{\partial,*}\cF^*$, which is a local system on $\partial\wt X$. We thus have $\mu_\partial^{-1}\cL=\wti{}^{-1}\wtj_*\cF^*$.

\begin{definitio}[$\St$-$\CC$-constructible sheaves]\label{def:CSt}
Let $\cF_\leq$ be a pre-$\ccI$-filtration, \ie an object of $\Mod(\kk_{\ccIet,\leq})$. We say that $\cF_\leq$ is \index{Stokes-$\CC$-constructible sheaf@$\St$-$\CC$-constructible sheaf}$\St$-$\CC$-constructible if
\begin{enumerate}
\item\label{def:CSt1}
$\cF^*\defin\wtj^{-1}\cF_\leq$ is a $\CC$-constructible sheaf on~$X^*$ whose singularity set $S\subset X^*$ is locally finite in~$X$,
\item\label{def:CSt2}
the natural morphism $\nu:\cL_{\leq}\defin\wti{}^{-1}\cF_\leq\to\wti{}^{-1}\wtj_*\cF^*$ is \emph{injective}, \ie the sheaf $\cF_\leq$ does not have a non-zero subsheaf supported in $\ccIet_{|\partial\wt X}$,
\item\label{def:CSt3}
the inclusion $\nu:\cL_\leq\hto\mu_\partial^{-1}\cL$ is a Stokes filtration of $\cL$.
\end{enumerate}

We will denote by \index{$ModkIetstc$@$\Mod_\CSt(\kk_{\ccIet,\leq})$}$\Mod_\CSt(\kk_{\ccIet,\leq})$ the full subcategory of $\Mod(\kk_{\ccIet,\leq})$ whose objects are $\St$-$\CC$-constructible on $\ccIet$.
\end{definitio}

We can also regard $\cF_\leq$ as a pre-$\ccI$-filtration of the sheaf $\cF\defin j_{\partial,*}\cF^*$, in the sense of Definition \ref{def:pre-I-fil}. Note that, for any open set $U\subset\wt X$ and any $\varphi\in\Gamma(U,\ccI)$, $\cF_{\leq\varphi}\defin\varphi^{-1}\cF_\leq$ is a subsheaf of $\cF_{|U}$.

\begin{lemme}\label{lem:CStabelian}
The category of $\St$-$\CC$-constructible sheaves on $\ccIet$ is \index{abelian (category)}abelian and stable by extensions in the category of pre-$\ccI$-filtrations on $\ccIet$.
\end{lemme}

\begin{proof}
When restricted to~$X^*$, the result is well-known for the category of $\CC$\nobreakdash-constructible sheaves in the category of sheaves of $\kk$-vector spaces. It is therefore enough to prove it over $\partial\wt X$. Let $\lambda:\cF_\leq\to\cF'_\leq$ be a morphism of $\St$-$\CC$-constructible sheaves. Its restriction to $\partial\wt X$ induces a morphism in the category of Stokes-filtered local systems on $\partial\wt X$. Note also that
\[
j_{\partial,*}j_\partial^{-1}\ker\lambda=\ker j_{\partial,*}j_\partial^{-1}\lambda\quad\text{and}\quad j_{\partial,*}j_\partial^{-1}\coker\lambda=\coker j_{\partial,*}j_\partial^{-1}\lambda.
\]
According to Theorem \ref{th:abelian} and the previous identification, the kernel and cokernel of $\lambda_{|\partial\wt X}$ inject into $\mu^{-1}\ker[\cL\to\cL']$ and $\mu^{-1}\coker[\cL\to\cL']$, and define Stokes-filtered local systems. Therefore, $\ker\lambda$ and $\coker\lambda$ (taken in $\Mod(\kk_{\ccIet,\leq})$) belong to $\Mod_\CSt(\kk_{\ccIet,\leq})$.

Similarly, the stability by extension is clear on~$X^*$, and an extension of two $\St$-$\CC$-constructible sheaves satisfies \ref{def:CSt}\eqref{def:CSt2} above. Then it satisfies \ref{def:CSt}\eqref{def:CSt3} according to Theorem \ref{th:abelian}.
\end{proof}

Note that, as a consequence, $\coker\lambda$ does not have any nonzero subsheaf supported on $\ccIet_{|\partial\wt X}$, so in particular $\lambda$ is strict with respect to the support condition \ref{def:CSt}\eqref{def:CSt2}.

\begin{remarque}\label{rem:nuprime}
Assume that $\cF_\leq$ is as in Definition \ref{def:CSt}\eqref{def:CSt1}. We have a distinguished triangle in $D^b(\ccIet)$:
\[
\wti^!\cF_\leq\to\wti^{-1}\cF_\leq\to\wti^{-1}\bR\wtj_*\cF^*=\wti^!\bR\wtj_!\cF^*[1]\To{+1}.
\]
Since $\wti^{-1}\cF_\leq$ and $\wti^{-1}\bR\wtj_*\cF^*=\wti^{-1}\wtj_*\cF^*$ are sheaves, this triangle reduces to the exact sequence
\[
0\to\cH^0(\wti^!\cF_\leq)\to\wti^{-1}\cF_\leq\to\wti^{-1}\wtj_*\cF^*\to\cH^1(\wti^!\cF_\leq)\to0
\]
and \ref{def:CSt}\eqref{def:CSt2} is equivalent to $\cH^0(\wti^!\cF_\leq)=0$. If this condition is fulfilled, the morphism $\nu':\mu_\partial^{-1}\cL\to\mu_\partial^{-1}\cL/\cL_\leq$ is equal to $\cH^1$ of the morphism $\wti^!\bR\wtj_!\cF^*\to\wti^!\cF_\leq$ and there is no nonzero other $\cH^k$.
\end{remarque}

\begin{remarque}[$\cHom$]\label{rem:HomSt}
Given two $\St$-$\CC$-constructible sheaves $\cF_\leq,\cF'_\leq$, one can define a $\St$-$\CC$-constructible sheaf $\cHom(\cF,\cF')_\leq$ whose generic part is the constructible sheaf $\cHom(\cF^*,\cF^{\prime*})$, and whose restriction to $\ccIet_{\partial\wt X}$ is $\cHom(\cL,\cL')_\leq$ (\cf Definition \ref{def:HomSt}). Firstly, one remarks that since $\wtj_*\cF^*$ and $\wtj_*\cF^{\prime*}$ are local systems, the natural morphism $\cHom(\wtj_*\cF^*,\wtj_*\cF^{\prime*})\to\wtj_*\cHom(\cF^*,\cF^{\prime*})$ is an isomorphism. The natural injection $\cHom(\cL,\cL')_\leq\hto\cHom(\cL,\cL')=\wti^{-1}\wtj_*\cHom(\cF^*,\cF^{\prime*})$ defines thus the desired $\St$-$\CC$-constructible sheaf.
\end{remarque}

Let $\cF_\leq$ be a Stokes-$\CC$-constructible sheaf on $\wt X$. Since~$\ccI$ is Hausdorff on $\partial\wt X$, the subsheaf $\cL_<$ of $\cL_\leq$ is well-defined, and we can consider the triple $(\cF^*,\cL_<,\nu_{|\cL_<})$, which defines a subsheaf $\cF_{\!\prec}$ of $\cF_\leq$, according to Remark \ref{rem:sheavesIet} (we avoid the notation~$\cF_<$, which has a different meaning over $X^*$).

\begin{definitio}[co-$\St$-$\CC$-constructible sheaves]\label{def:coCSt}
We will say that a pre-$\ccI$-filtration~$\cF_{\!\prec}$ is a \index{co-Stokes-$\CC$-constructible sheaf@co-$\St$-$\CC$-constructible sheaf}co-Stokes-$\CC$-constructible sheaf if it satisfies the properties \ref{def:CSt}\eqref{def:CSt1}, \ref{def:CSt}\eqref{def:CSt2} and, setting $\cL_<=\wti^{-1}\cF_\prec$,
\begin{enumerate}
\item[$(\ref{def:CSt3}')$]
the inclusion $\nu:\cL_<\hto\mu_\partial^{-1}\cL$ is a Stokes co-filtration of $\cL$ (\cf Remark~\ref{rem:coIfilt}).
\end{enumerate}
This defines a full subcategory \index{$ModkIetstco$@$\Mod_\coCSt(\kk_{\ccIet,\leq})$}$\Mod_{\coCSt}(\kk_{\ccIet,\leq})$ of $\Mod(\kk_{\ccIet,\leq})$.
\end{definitio}

We similarly have:

\begin{lemme}\label{lem:coCStabelian}
The category of co-$\St$-$\CC$-constructible sheaves on $\ccIet$ is \index{abelian (category)}abelian and stable by extensions in the category of pre-$\ccI$-filtrations on $\ccIet$.\qed
\end{lemme}

\begin{exo}
From a co-Stokes-$\CC$-constructible sheaf $\cF_{\!\prec}$, reconstruct the Stokes-$\CC$-constructible sheaf $\cF_\leq$ which defines it (this mainly consists in reconstructing $\cL_\leq$ from $\cL_<$, \cf Remark \ref{rem:coIfilt} and Remark \ref{rem:coIfiltStokes}).
\end{exo}

Lastly, there is a simpler subcategory of $\Mod(\kk_{\ccIet,\leq})$ which will also be useful, namely that of \index{graded-$\CC$-constructible sheaf}graded $\CC$-constructible sheaves on $\ccIet$. By definition, such an object takes the form $\wti_*(\mu_{\partial,!}\cG)_\leq$, where $\cG$ is a locally constant sheaf of finite dimensional $\kk$-vector spaces on $\ccIet_{|\partial\wt X}$, on the support of which $\mu_\partial$ is proper. It is easy to check that it is abelian and stable by extensions in $\Mod(\kk_{\ccIet,\leq})$. The following is then clear:

\begin{lemme}\label{lem:Stcofilt}
Let $\cF_\leq$ be a Stokes-$\CC$-constructible sheaf on $\ccIet$ and let $\cF_{\!\prec}$ be the associated co-Stokes-$\CC$-constructible sheaf, together with the natural inclusion $\cF_{\!\prec}\hto\cF_\leq$. Then the quotient (in $\Mod(\kk_{\ccIet,\leq})$) is a graded $\CC$-constructible sheaf on $\ccIet$.\qed
\end{lemme}

\begin{remarque}\label{rem:gr<}
There is no nonzero morphism (in the category $\Mod(\kk_{\ccIet,\leq})$) from a graded $\CC$-constructible sheaf $\cG$ on $\ccIet$ to a (co-)Stokes-$\CC$-constructible sheaf $\cF_\leq$ (or~$\cF_{\!\prec}$). Indeed, presenting $\cF_\leq$ as a triple $(\cF^*,\cL_\leq,\nu)$ and $\cG$ by $(0,\wti^{-1}\cG,0)$, a morphism $\cG\to\cF_\leq$ consists of a morphism $\wti^{-1}\cG\to\cL_\leq$ which is zero when composed with $\nu$. Since $\nu$ is injective, after \ref{def:CSt}\eqref{def:CSt2}, the previous morphism is itself zero. The same argument applies to $\cG\to\cF_{\!\prec}$. In other words, since $\cF_\leq$ (\resp $\cF_{\!\prec}$) does not have any non-zero subsheaf supported on $\ccIet_{|\partial\wt X}$, and since $\cG$ is supported on this subset, any morphism from $\cG$ to $\cF_\leq$ (\resp $\cF_{\!\prec}$) is zero.
\end{remarque}

\Subsection{Derived categories and duality}\label{subsec:dualStperv}
\subsubsection*{Derived categories}
We denote by \index{$DbkIet$@$D^\rb(\kk_{\ccIet,\leq})$}$D^\rb(\kk_{\ccIet,\leq})$ the bounded derived category of pre-$\ccI$-filtrations, and by \index{$DbkIetcst$@$D^\rb_\CSt(\kk_{\ccIet,\leq})$}$D^\rb_\CSt(\kk_{\ccIet,\leq})$ the full subcategory of $D^\rb(\kk_{\ccIet,\leq})$ consisting of objects having $\St$-$\CC$-constructible cohomology. By Lemma \ref{lem:CStabelian}, the category $D^\rb_\CSt(\kk_{\ccIet,\leq})$ is a full triangulated subcategory of $D^\rb(\kk_{\ccIet,\leq})$. We will now define a $t$\nobreakdash-structure on $D^\rb_\CSt(\kk_{\ccIet,\leq})$ which restricts to the usual\footnote{We refer for instance to \cite{Dimca04} for basic results on \index{perverse sheaf}perverse sheaves; recall that the constant sheaf supported at one point is perverse, and a local system shifted by one is perverse, see \eg \cite[Ex\ptbl5.2.23]{Dimca04}. See also \cite{B-B-D81}, \cite[Chap\ptbl10]{K-S90} for more detailed results.} perverse one on~$X^*$.

The subcategory \index{$DbkIetcstp$@$\pD^{b,\leq0}_\CSt(\kk_{\ccIet,\leq})$, $\pD^{b,\geq0}_\CSt(\kk_{\ccIet,\leq})$}$\pD^{b,\leq0}_\CSt(\kk_{\ccIet,\leq})$ consists of objects $\cF_\leq$ satisfying $\cH^j(\cF_\leq)=0$ for $j>-1$ and $\cH^j(i_{x_o}^{-1}\cF_\leq)=0$ for any $x_o\in X^*$ and $j>0$.

Similarly, the subcategory $\pD^{b,\geq0}_\CSt(\kk_{\ccIet,\leq})$ consists of objects $\cF_\leq$ satisfying $\cH^j(\cF_\leq)=0$ for $j<-1$ and $\cH^j(i_{x_o}^!\cF_\leq)=0$ for any $x_o\in X^*$ and $j<0$.

The pair $\big(\pD^{b,\leq0}_\CSt(\kk_{\ccIet,\leq}),\pD^{b,\geq0}_\CSt(\kk_{\ccIet,\leq})\big)$ is a $t$\nobreakdash-structure on $D^\rb_\CSt(\kk_{\ccIet,\leq})$ (this directly follows from the result on~$X^*$).

\begin{definitio}\label{def:Stperv1}
The category \index{$STperv$@$\StPerv_{D}(\kk_{\ccIet,\leq})$}$\StPerv_{D}(\kk_{\ccIet,\leq})$ of perverse sheaves on~$X^*$ with a Stokes structure at $\partial\wt X$ (that we also call \index{Stokes perverse@Stokes-perverse sheaf!on $\wt X$}\emph{Stokes-perverse sheaves} on $\wt X$) is the heart of this $t$\nobreakdash-structure.
\end{definitio}

This is an \index{abelian (category)}abelian category (\cf \cite{B-B-D81}, \cite[\S10.1]{K-S90}).

\begin{remarque}\label{rem:StPerv}
Objects of $\StPerv_{D}(\kk_{\ccIet,\leq})$ behave like sheaves, \ie can be reconstructed, up to isomorphism, from similar objects on each open set of an open covering of~$X$, together with compatible gluing isomorphisms. An object of $\StPerv_{D}(\kk_{\ccIet,\leq})$ is a sheaf shifted by $1$ in the neighbourhood of $\ccIet_{|\partial\wt X}$ and is a perverse sheaf (in the usual sense) on~$X^*$.

Notice also that the objects of $\StPerv_{D}(\kk_{\ccIet,\leq})$ which have no singularity on $X^*$ are sheaves (shifted by one). We will call them \emph{smooth} $\St$-$\CC$-constructible sheaves or smooth Stokes-perverse sheaves (depending on the shifting convention).
\end{remarque}

However, it will be clear below that this presentation is not suitable for defining a duality functor on it. We will use a presentation which also takes into account $\cF_{\!\prec}$ and $\gr\cF$. Anticipating on the Riemann-Hilbert correspondence, $\cF_\leq$ corresponds to horizontal sections of a connection having moderate growth at $\partial\wt X$, while $\cF_{\!\prec}$ corresponds to horizontal sections having rapid decay there, and duality pairs moderate growth with rapid decay.

Note that similar arguments apply to co-Stokes-$\CC$-constructible sheaves, and we get a $t$\nobreakdash-structure $\big(\pD^{b,\leq0}_\coCSt(\kk_{\ccIet,\leq}),\pD^{b,\geq0}_\coCSt(\kk_{\ccIet,\leq})\big)$ on $D^\rb_\coCSt(\kk_{\ccIet,\leq})$.

Lastly, the full subcategory $D^\rb_\grCc(\kk_{\ccIet,\leq})$ of $D^\rb(\kk_{\ccIet,\leq})$ whose objects have graded $\CC$-constructible cohomology is endowed with the $t$\nobreakdash-structure induced from the shifted natural one $(D^{b,\leq-1}(\kk_{\ccIet,\leq}),D^{b,\geq-1}(\kk_{\ccIet,\leq}))$.

\subsubsection*{Duality}
Recall that the \emphb{dualizing complex} on $\wt X$ is $j_{\partial,!}\kk_{X^*}[2]$ (\cf \eg \cite{K-S90}) and (\cf Corollary \ref{cor:Rhom}) we have a functor $\bR\cHom(\iota^{-1}\cbbullet,\mu^{-1}j_{\partial,!}\kk_{X^*}[2])$ from the category $D^{b,\textup{op}}(\kk_{\ccIet,\leq})$ to $D^+(\kk_{\ccIet,\leq})$. We will denote it by \index{$DUAL$@$\bD$, $\bD'$}$\bD$. We will prove the following.

\begin{proposition}\label{prop:DDCSt}\index{duality}
The duality functor $\bD$ induces functors $\bD:D^{b,\textup{op}}_\CSt(\kk_{\ccIet,\leq})\to D^\rb_\coCSt(\kk_{\ccIet,\leq})$ and $\bD:D^{b,\textup{op}}_\coCSt(\kk_{\ccIet,\leq})\to D^\rb_\CSt(\kk_{\ccIet,\leq})$ which are $t$-exact, and there are isomorphisms of functors $\id\simeq\bD\circ\bD$ in both $D^\rb_\CSt(\kk_{\ccIet,\leq})$ and $D^\rb_\coCSt(\kk_{\ccIet,\leq})$.

On the other hand, it induces a functor $\bD:D^{b,\textup{op}}_\grCc(\kk_{\ccIet,\leq})\to D^\rb_\grCc(\kk_{\ccIet,\leq})$ such that $\bD\circ\bD\equiv\id$.
\end{proposition}

Let $\cF_\leq$ be an object of $D^\rb(\kk_{\ccIet,\leq})$. Recall that, for any open set $U\subset\wt X$ and any $\varphi\in\Gamma(U,\ccI)$, $\cF_{\leq\varphi}\defin\varphi^{-1}\cF_\leq$ is an object of $D^\rb(\kk_U)$.

\begin{lemme}\label{lem:Rconst}
If $\cF_\leq$ is a Stokes-$\CC$-constructible sheaf on $\ccIet$, then for any $\varphi\in\Gamma(U,\ccI)$, $\cF_{\leq\varphi}$ is $\RR$-constructible on~$U$. The same result holds for a co-Stokes-$\CC$-constructible sheaf $\cF_{\!\prec}$.
\end{lemme}

\begin{proof}
It suffices to check this on $U\cap\partial\wt X$, where it follows from Proposition \ref{prop:stokeswithout}.
\end{proof}

\begin{lemme}\label{lem:StCdual}
If $\cF_\leq$ is a Stokes-$\CC$-constructible sheaf (\resp $\cF_{\!\prec}$ is a co-Stokes-$\CC$-constructible sheaf) on $\ccIet$, then on $\wt X\moins\Sing\cF^*$, $\cF^\vee_\prec\defin\bR\cHom(\iota^{-1}\cF_\leq,\wtj_!\kk_{X^*})$ (\resp $\cF^\vee_\leq\defin\bR\cHom(\iota^{-1}\cF_{\!\prec},\wtj_!\kk_{X^*})$) has cohomology in degree $0$ at most and is equal to the co-Stokes-$\CC$-constructible sheaf $(\cF^{*,\vee},(\cL^\vee)_<,\nu^{\prime,\vee})$ (\resp the Stokes-$\CC$-constructible sheaf $(\cF^{*,\vee},(\cL^\vee)_\leq,\nu^{\prime,\vee})$), where $\nu^{\prime,\vee}$ is the morphism dual to $\nu'$. Moreover, the dual of $\nu$ is $(\nu^{\prime,\vee})'$
\end{lemme}

\begin{proof}
If $U$ is an open set in $\wt X$ and $\varphi\in\Gamma(U,\ccIet)$, we will set
\[
\bD(\cF_{\leq\varphi})=\bR\cHom_{\kk}(\cF_{\leq\varphi},j_{\partial,!}\kk_{U^*}[2])\quad\text{and}\quad \bD'(\cF_{\leq\varphi})=\bR\cHom_{\kk}(\cF_{\leq\varphi},j_{\partial,!}\kk_{U^*}).
\]
As a consequence of Lemma \ref{lem:Rconst} we have (\cf \cite[Prop\ptbl3.1.13]{K-S90}),
\begin{equation}\label{eq:DDi}
\begin{split}
i_\partial^!\bD(\cF_{\leq\varphi})&=\bR\cHom(i_\partial^{-1}\cF_{\leq\varphi},i_\partial^!j_{\partial,!}\kk[2])\\
&=\bR\cHom(i_\partial^{-1}\cF_{\leq\varphi},\kk[1])=:\bD(i_\partial^{-1}\cF_{\leq\varphi}),
\end{split}
\end{equation}
since $\kk_{\partial\wt X}[1]$ is the dualizing complex on $\partial\wt X$, and by biduality we have $\bD(i_\partial^!\cF_{\leq\varphi})=i_\partial^{-1}\bD(\cF_{\leq\varphi})$ (\cf \cite[Prop\ptbl3.4.3]{K-S90}). This can also be written as
\[
i_\partial^!\bD'(\cF_{\leq\varphi})[1]=\bD'(i_\partial^{-1}\cF_{\leq\varphi})\quad\text{and}\quad i_\partial^{-1}\bD'(\cF_{\leq\varphi})=\bD'(i_\partial^!\cF_{\leq\varphi}[1]).
\]

For the proof of Lemma \ref{lem:StCdual}, we can assume that $\Sing\cF^*=\emptyset$. We will prove the lemma starting from a Stokes $\CC$-constructible sheaf, the co-Stokes case being similar.

For the first part of the statement, it suffices to show that $i_\partial^{-1}\bD'(\cF_{\leq\varphi})$ is a sheaf. We have $i_\partial^{-1}\bD'(\cF_{\leq\varphi})=\bD'(i_\partial^!\cF_{\leq\varphi}[1])=\bD'(\cL_{|U}/\cL_{\leq\varphi})$, \cf Remark \ref{rem:nuprime}, and thus $i_\partial^{-1}\bD'(\cF_{\leq\varphi})$ is a sheaf, equal to $(\cL^\vee)_{<-\varphi}$, according to Lemma \ref{lem:dualS1}. Therefore, $\cF^\vee_\prec=\cHom(\iota^{-1}\cF_\leq,\wtj_!\kk_{X^*})$. Moreover, still using Remark \ref{rem:nuprime}, one checks that the gluing morphism $\nu^\vee$ is the dual of $\nu'$. Since $\nu'$ is onto, $\nu^{\prime,\vee}$ is injective, and $\cF_\prec^\vee$ is equal to $(\cF^{*,\vee},(\cL^\vee)_<,\nu^{\prime,\vee})$, as wanted. A similar argument shows that the dual of~$\nu$ is $(\nu^{\prime,\vee})'$.
\end{proof}

\begin{proof}[\proofname\ of Proposition \ref{prop:DDCSt}]
The first claim is now a direct consequence of Lemma \ref{lem:StCdual}. For any $U$ and any $\varphi\in\Gamma(U,\ccI)$ we have the bi-duality isomorphism after applying the pull-back by $\varphi:U\to\ccIet$, according to the standard result for $\RR$-constructible sheaves and Lemma \ref{lem:Rconst}. This isomorphism is clearly compatible with restrictions of open sets, hence gives $\id\simeq\bD\circ\bD$.

For the graded case, the situation is simpler, as we work with local systems on $\ccIet_{|\partial\wt X}$ extended by zero on $\ccIet$.
\end{proof}

\subsection{The category of Stokes-perverse sheaves on $\protect\wt X$}\label{subsec:Stperv}
In the case of Stokes-filtered local systems on $\partial\wt X$, the Stokes and co-Stokes aspect (namely~$\cL_\leq$ and $\mu^{-1}\cL/\cL_<$) can be deduced one from the other. The same remark applies to (co\nobreakdash-)Stokes-$\CC$-constructible \emph{sheaves}. However, such a correspondence $\text{Stokes}\allowbreak\ssi\allowbreak\text{co-Stokes}$ heavily depends on sheaves operations. In the derived category setting, we will keep them together.

It will be useful to consider the $t$-categories $D^\rb_\CSt(\kk_{\ccIet,\leq})$ etc.\ as subcategories of the same one. We can choose the following one. By a $\CC$-$\RR$-constructible sheaf on $\ccIet$ we will mean a sheaf $\cG$ on $\ccIet$ such that for any open set $U\subset\wt X$ and any $\varphi\in\Gamma(U,\ccI)$, $\varphi^{-1}\cG$ is $\RR$-constructible on $U$ and $\CC$-constructible on $U^*=U\cap X^*$ with locally finite singularity set on $U$. The category $\Mod_\CRc(\kk_{\ccIet,\leq})$ is the full subcategory of $\Mod(\kk_{\ccIet,\leq})$ whose underlying sheaf is $\CC$-$\RR$-constructible. According to Lemma \ref{lem:Rconst}, it contains the full subcategories of (co-/gr-)$\St$-constructible sheaves. It is abelian and stable by extensions in $\Mod(\kk_{\ccIet,\leq})$. The full triangulated subcategory $D^\rb_\CRc(\kk_{\ccIet,\leq})$ of $D^\rb(\kk_{\ccIet,\leq})$ is equipped with a $t$-structure defined as in the beginning of \S\ref{subsec:dualStperv}. This $t$-structure induces the already defined one on the full triangulated subcategories $D^\rb_\CSt(\kk_{\ccIet,\leq})$ etc.\

\begin{definitio}\label{def:StPervtri}
The category \index{$STkiet$@$\St(\kk_{\ccIet,\leq})$}$\St(\kk_{\ccIet,\leq})$ is the full subcategory of the category of distinguished triangles of $D^\rb(\kk_{\ccIet,\leq})$ whose objects consist of distinguished triangles $\cF_{\!\prec}\to\cF_\leq\to\gr\cF\To{+1}$, where $\cF_{\!\prec}$ is an object of $D^\rb_\coCSt(\kk_{\ccIet,\leq})$, $\cF_\leq$ of $D^\rb_\CSt(\kk_{\ccIet,\leq})$, and $\gr\cF$ of $D^\rb_{\grCc}(\kk_{\ccIet,\leq})$. These triangles are also distinguished triangles in $D^\rb_\CRc(\kk_{\ccIet,\leq})$.
\end{definitio}

\begin{proposition}\label{prop:SttstructD}
The category $\St(\kk_{\ccIet,\leq})$ is triangulated, and endowed with a $t$\nobreakdash-structure \index{$STkietstr$@$\St^{\leq0}(\kk_{\ccIet,\leq}),\St^{\geq0}(\kk_{\ccIet,\leq})$}$(\St^{\leq0}(\kk_{\ccIet,\leq}),\St^{\geq0}(\kk_{\ccIet,\leq}))$ defined by the property that $\cF_{\!\prec},\cF_\leq,\gr\cF$ belong to the $\leq0$ (\resp $\geq0$) part of their corresponding categories. There is a \index{duality}duality functor $\bD:\St^{\textup{op}}(\kk_{\ccIet,\leq})\to\St(\kk_{\ccIet,\leq})$ satisfying $\bD\circ\bD\simeq\id$ and which is compatible to the $t$\nobreakdash-structures.
\end{proposition}

\begin{proof}[Sketch of proof]
Any (not necessarily distinguished) triangle $\cF_{\!\prec}\to\cF_\leq\to\gr\cF\To{+1}$ induces a (not necessarily exact) sequence $\cdots\to\pcH^k\cF_{\!\prec}\to\pcH^k\cF_\leq\to\pcH^k\gr\cF\to\cdots$ by taking perverse cohomology of objects of $D^\rb_\CRc(\kk_{\ccIet,\leq})$, and, according to Remark \ref{rem:gr<}, each connecting morphism $\pcH^k\gr\cF\to\pcH^{k+1}\cF_{\!\prec}$ is zero. One then shows that such a triangle is distinguished if and only if each short sequence as before is exact. The axioms of a triangulated category are then checked for $\St(\kk_{\ccIet,\leq})$ by using this property.

In order to check the $t$-structure property (\cf \cite[Def\ptbl10.1.1]{K-S90}), the only non trivial point is to insert a given object of $\St(\kk_{\ccIet,\leq})$ in a triangle with first term in $\St^{\leq0}$ and third term in~$\St^{\geq1}$. Denoting by $\tau^{\leq0}$ and $\tau^{\geq1}$ the perverse truncation functors in $D^\rb_\CRc(\kk_{\ccIet,\leq})$, it is enough to check that each of these functors preserves objects of $\St(\kk_{\ccIet,\leq})$. This follows from the cohomological characterization of the distinguished triangles which are objects of $\St(\kk_{\ccIet,\leq})$.

The assertion on duality follows from Proposition \ref{prop:DDCSt}.
\end{proof}

The heart of this $t$\nobreakdash-structure is an \index{abelian (category)}abelian category (\cf \cite{B-B-D81}, \cite{K-S90}).

\begin{lemme}\label{lem:forget}
The ``forgetting'' functor from $\St(\kk_{\ccIet,\leq})$ to $D^\rb_\CSt(\kk_{\ccIet,\leq})$, sending a triangle $\cF_{\!\prec}\to\cF_\leq\to\gr\cF\To{+1}$ to $\cF_\leq$ (which is by definition compatible with the $t$\nobreakdash-structures) induces an equivalence of abelian categories from the heart of the $t$\nobreakdash-structure of $\St(\kk_{\ccIet,\leq})$ to \index{$STperv$@$\StPerv_{D}(\kk_{\ccIet,\leq})$}$\StPerv_{D}(\kk_{\ccIet,\leq})$.
\end{lemme}

\begin{proof}
For the essential surjectivity, the point is to recover $\cF_{\!\prec}$ and $\gr\cF$ from $\cF_\leq$. As our objects can be obtained by gluing from local data, the question can be reduced to a local one near $\ccIet_{|\partial\wt X}$, where we can apply the arguments for sheaves and define $\cF_{\!\prec}$ and $\gr\cF$ as in Lemma \ref{lem:Stcofilt}.

Similarly, any morphism $\cF_\leq\to\cF'_\leq$ lifts as a morphism between triangles. It remains to show the uniqueness of this lifting. This is a local problem, that we treat for sheaves. The triangle is then an exact sequence, and the morphism $\cF_{\!\prec}\to\cF'_\prec$ coming in a morphism between such exact sequences is simply the one induced by $\cF_\leq\to\cF'_\leq$, so is uniquely determined by it. The same argument holds for $\gr\cF\to\gr\cF'$.
\end{proof}

\begin{corollaire}[Duality]\label{cor:DDStPerv}
The category $\StPerv_{D}(\kk_{\ccIet,\leq})$ is stable by \index{duality}duality.\qed
\end{corollaire}

\subsection{Direct image to~$X$}\label{subsec:imdirStperv}
Given a sheaf $\cF_\leq$ on $\ccIet$, any section $\varphi\in\Gamma(U,\ccI)$ produces a sheaf $\cF_{\leq\varphi}\defin\varphi^{-1}\cF_\leq$ on $U$. We now use more explicitly that~$\ccI$ has a global section $0$, meaning that there is a global section $0:\wt X\to\ccIet$ of $\mu$, making $\wt X$ a closed subset of $\ccIet$. Then $\cF_{\leq0}$ is a sheaf on $\wt X$.

\begin{proposition}\label{prop:StDPervimdir}
Let $\cF_\leq$ be an object of $\StPerv_{D}(\kk_{\ccIet,\leq})$. Then $\bR\varpi_*\cF_{\leq0}$ is perverse (in the usual sense) on~$X$.
\end{proposition}

\begin{proof}
It is a matter of proving that $i_D^{-1}\bR\varpi_*\cF_{\leq0}$ has cohomology in degrees $-1,0$ at most and $i_D^!\bR\varpi_*\cF_{\leq0}$ in degrees $0,1$ at most.

We have $i_D^{-1}\bR\varpi_*\cF_{\leq0}=\bR\Gamma(\partial\wt X,i_\partial^{-1}\cF_{\leq0})$, because $\varpi$ is proper (\cf \cite[Prop\ptbl2.5.11]{K-S90}), so, setting $\cL_{\leq0}=i_\partial^{-1}\cF_{\leq0}[-1]$, the first assertion reduces to $H^k(\partial\wt X,\cL_{\leq0})=0$ for $k\neq0,1$, which is clear since $\cL_{\leq0}$ is a $\RR$-constructible sheaf. Similarly, as we have seen after Lemma \ref{lem:Rconst}, $i_\partial^!(\cF_{\leq0})=i_\partial^!(\cF_{\leq0}[-1])[1]$ is a sheaf, so the second assertion is also satisfied (\cf \cite[Prop\ptbl3.1.9(ii)]{K-S90}).
\end{proof}

\subsection{Stokes-perverse sheaves on $\protect\underline{\protect\wt X}$}\label{subsec:Stpervhat}
We will prove in the next \chaptersname that Stokes-perverse sheaves on $\wt X$ correspond to holonomic $\cD$-modules on~$X$ which are isomorphic to their localization at~$D$, and have at most regular singularities on $X^*$. In order to treat arbitrary holonomic $\cD_X$-modules, we need to introduce supplementary data at~$D$, which are regarded as filling the space $\wt X$ by discs in order to obtain a topological space $\underline{\wt X}$, that we describe now. This construction goes back to \cite{Deligne84cc} and is developed in \cite{Malgrange91}.

For any $x_o\in D$, let $\wh X_{\wh x_o}$ be an open disc with center denoted by $\wh x_o$, and let $\ov{\wh X}_{\wh x_o}$ be the associated closed disc. Denote by $\whj:\wh X_{\wh x_o}\hto\ov{\wh X}_{\wh x_o}$ the open inclusion and $\whi:S^1_{\wh x_o}\hto\ov{\wh X}_{\wh x_o}$ the complementary closed inclusion.

We denote by \index{$XWTU$@$\underline{\wt X}$|finindexnotations}$\underline{\wt X}$ the topological space (homeomorphic to~$X$) obtained by gluing each closed disc $\ov{\wh X}_{\wh x_o}$ to $\wt X$ along their common boundary $S^1_{\wh x_o}\!=\!S^1_{x_o}$ for $x_o\!\in\!D$ (\cf Figure~\ref{fig:0}).
\begin{figure}[htb]
\begin{center}
\includegraphics[scale=.4]{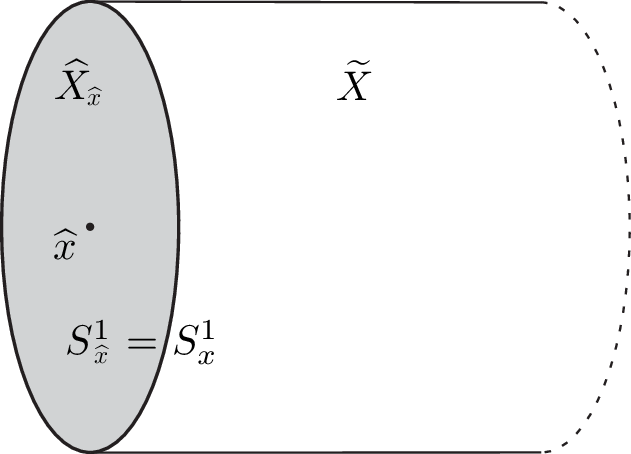}
\end{center}
\caption{The space $\protect\underline{\wt X}$ near $S^1_{x}$ ($x\in D$)}\label{fig:0}
\end{figure}

We first define the category \index{$ModkietleqD$@$\Mod(\kk_{\ccIet,\leq,\wh D})$}$\Mod(\kk_{\ccIet,\leq,\wh D})$ to be the category whose objects are triples $(\cF_\leq,\wh\cF,\wh\nu)$, where $\cF_\leq$ is a pre-$\ccI$-filtration, \ie an object of $\Mod(\kk_{\ccIet,\leq})$, $\wh\cF$ is a sheaf on $\bigsqcup_{x_o\in D}\wh X_{\wh x_o}$, and $\wh\nu$ is a morphism $i_\partial^{-1}\cF_{\leq0}\to\whi^{-1}\whj_*\wh\cF$. The morphisms in this category consist of pairs $(\lambda,\wh\lambda)$ of morphisms in the respective categories which are compatible with $\wh\nu$ in an obvious way. We have ``forgetting'' functors to $\Mod(\kk_{\ccIet,\leq})$ (hence also to $\Mod(\kk_{\ccIet})$) and $\Mod(\kk_{\wh X})$. Then $\Mod(\kk_{\ccIet,\leq,\wh D})$ is an \index{abelian (category)}abelian category. A sequence in $\Mod(\kk_{\ccIet,\leq,\wh D})$ is exact iff the associated sequences in $\Mod(\kk_{\ccIet})$ and $\Mod(\kk_{\wh X})$ are exact. The associated bounded derived category is denoted by $D^\rb(\kk_{\ccIet,\leq,\wh D})$.

\begin{definitio}[$\St$-$\CC$-constructible sheaves on $\underline{\wt X}$]\label{def:CSthat}\index{Stokes-$\CC$-constructible sheaf@$\St$-$\CC$-constructible sheaf}
The objects in the category $\Mod_\CSt(\kk_{\ccIet,\leq,\wh D})$ consist of triples $(\cF_\leq,\wh\cF,\wh\nu)$, where
\begin{enumerate}
\item\label{def:CSthat1}
$\cF_\leq$ is a $\St$-$\CC$-constructible sheaf on $\wt X$ (\cf Definition \ref{def:CSt}), defining a Stokes-filtered local system $(\cL,\cL_\bbullet)$ on each $S^1_{x_o}$, $x_o\in D$,
\item\label{def:CSthat2}
$\wh \cF$ is a $\CC$-constructible sheaf on $\bigsqcup_{x_o\in D}\wh X_{\wh x_o}$ with singularity at $\wh x_o$ ($x_o\in D$) at most,
\item\label{def:CSthat3}
$\wh\nu$ is an isomorphism $\gr_0\cL\isom\whi^{-1}\whj_*\wh\cF$.
\end{enumerate}
Morphisms between such triples consist of pairs $(\lambda,\wh\lambda)$ of morphisms in the respective categories which are compatible with $\wh\nu$.
\end{definitio}

By definition, $\Mod_\CSt(\kk_{\ccIet,\leq,\wh D})$ is a full subcategory of $\Mod(\kk_{\ccIet,\leq,\wh D})$.

\begin{lemme}\label{lem:CSthatabelian}
The category $\Mod_\CSt(\kk_{\ccIet,\leq,\wh D})$ is \index{abelian (category)}abelian and stable by extensions in $\Mod(\kk_{\ccIet,\leq,\wh D})$.
\end{lemme}

\begin{proof}
We apply Lemma \ref{lem:CStabelian} for the $\cF_\leq$-part, a standard result for the $\wh\cF$-part, and the compatibility with $\wh\nu$ follows.
\end{proof}

The definition of the category $\StPerv_D(\kk_{\ccIet,\leq,\wh D})$ now proceeds exactly as in \S\ref{subsec:Stperv}, by only adding the information of the usual $t$\nobreakdash-structure on the various $\wh X_{\wh x_o}$, $x_o\in D$. We leave the details to the interested reader. As above, a \index{Stokes perverse@Stokes-perverse sheaf!on $\underline{\wt X}$}Stokes-perverse sheaf is a sheaf shifted by one away from the singularities in~$X^*$ and of $\wh D$.

\begin{remarque}
The previous presentation makes a difference between singularities~$S$ in $X^*$ and singularities in~$D$. One can avoid this difference, by considering objects which are local systems on $X^*\moins S$ and by replacing $D$ with $D\cup S$. Then ``regular singularities'' are the points of~$D$ where the set of exponential factors of $\cF_\leq$ reduces to $\{0\}$. This point of view is equivalent to the previous one through the functor $\cP$ at those points (see below, Proposition \ref{prop:StDPervhatimdir}).
\end{remarque}

\subsection{Associated perverse sheaf on~$X$}\index{associated perverse sheaf}\index{perverse sheaf!associated --}
We now define a functor $\cP:\StPerv_D(\kk_{\ccIet,\leq,\wh D})\to\Perv(\kk_X)$ between Stokes-perverse sheaves on $\underline{\wt X}$ and perverse sheaves (in the usual sense) on~$X$. This is an extension of the direct image procedure considered in \S\ref{subsec:imdirStperv}. Namely, if $(\cF_\leq,\wh\cF,\wh\nu)$ is a Stokes-constructible sheaf on $\underline{\wt X}$, the triple $(\cF_{\leq0},\wh\cF,\wh\nu)$ allows one to define a sheaf $\underline\cF_{\leq0}$ on $\underline{\wt X}$: indeed, a sheaf on $\underline{\wt X}$ can be determined up to isomorphism by a morphism $\cL_{\leq0}\to\whi^{-1}\whj_*\wh\cF$; we use the composed morphism $\cL_{\leq0}\to\gr_0\cL\To{\wh\nu}\whi^{-1}\whj_*\wh\cF$.

Extending this construction to Stokes-perverse sheaves and taking the direct image by the continuous projection $\underline\varpi:\underline{\wt X}\to X$ which contracts $\ov{\wh X}$ to~$D$ is $\cP(\cF_\leq,\wh\cF,\wh\nu)$. We will however not explicitly use the previous gluing construction in this perverse setting, but an ersatz of it, in order to avoid a precise justification of this gluing. Let us make this more precise.

We first note that the pull-back by the zero section $0:\wt X\hto\ccIet$ defines an exact functor $0^{-1}:\Mod(\kk_{\ccIet,\leq})\to\Mod(\kk_{\wt X})$ and therefore, taking the identity on the $\wh\cF$\nobreakdash-part, an exact functor from $\Mod(\kk_{\ccIet,\leq,\wh D})$ to the category $\Mod(\kk_{\underline{\wt X},\leq0,\wh D})$ consisting of triples $(\cF_{\leq0},\wh\cF,\wh\nu)$.

In the following, we implicitly identify sheaves on~$D$ and sheaves on~$X$ supported on~$D$ via $\bR i_{D,*}$ and we work in $D^\rb_{\Cc}(\kk_X)$. On the one hand, we have a natural composed morphism
\[
\bR\varpi_*\cF_{\leq0}\to\bR\Gamma(\partial\wt X,\wti{}^{-1}\cF_{\leq0})\to\bR\Gamma(\partial\wt X,\gr_0\cF)\underset{\sim}{\To{\wh\nu}}\bR\Gamma(\partial\wh X,\whi{}^{-1}\whj_*\wh\cF).
\]
On the other hand, we have a distinguished triangle
\[
\bR\Gamma(\wh X,\whj_*\wh\cF)\to\bR\Gamma(\partial\wh X,\whi{}^{-1}\whj_*\wh\cF)\to\bR\Gamma(\ov{\wh X},\whj_!\wh\cF)[1]\To{+1}.
\]
We deduce a morphism
\[
\bR\varpi_*\cF_{\leq0}\to\bR\Gamma(\ov{\wh X},\whj_!\wh\cF)[1]=\bR\Gamma_\rc(\wh X,\wh\cF)[1].
\]

\begin{proposition}\label{prop:StDPervhatimdir}
If $(\cF_\leq,\wh\cF,\wh\nu)$ is an object of $\StPerv_D(\kk_{\ccIet,\leq,\wh D})$, then the complex $\bR\Gamma_\rc(\wh X,\wh\cF)[1]$ has cohomology in degrees $-1$ and $0$ at most and the cone of the previous morphism--that we denote by \index{$PP$@$\cP$}$\cP(\cF_\leq,\wh\cF,\wh\nu)[1]$--has perverse cohomology in degree $1$ at most.
\end{proposition}

\begin{corollaire}\label{cor:StDPervhatimdir}
The functor $\cP:(\cF_\leq,\wh\cF,\wh\nu)\mto\pcH^0\cP(\cF_\leq,\wh\cF,\wh\nu)$ is an exact functor from $\StPerv_D(\kk_{\ccIet,\leq,\wh D})$ to $\Perv(\kk_X)$.\qed
\end{corollaire}

\begin{proof}[\proofname\ of Proposition \ref{prop:StDPervhatimdir}]
We consider the complementary inclusions
\[
\wh D\Hto{i_{\wh D}}\wh X\Hfrom{j_{\wh D}}\wh X\moins\{\wh D\}.
\]
We have a distinguished triangle $i_{\wh D}^!\wh\cF\to\wh\cF\to\bR j_{\wh D,*}j_{\wh D}^{-1}\wh\cF\To{+1}$, to which we apply~$\whj_!$ (recall that $\whj$ denotes the inclusion $\wh X_{\wh x_o}\hto\ov{\wh X}_{\wh x_o}$). This has no effect to the first term, which remains $i_{\wh D}^!\wh\cF$ and has cohomology in degrees~$0$ and~$1$ at most by the cosupport condition for $\wh\cF$. We then have
\[
\HH^{-1}(\ov{\wh X},\whj_!\wh\cF)\hto\HH_C^{-1}(\wh X{}^*,j_{\wh D}^{-1}\wh\cF),
\]
where $C$ is the family of supports in $\wh X{}^*$ whose closure in $\wh X$ is compact. We are thus reduced to showing that, if $\wh\cL$ is a local system on $\wh X{}^*$, we have $H^0_C(\wh X,\wh\cL)=0$, which is clear. Therefore, $\bR\Gamma_\rc(\wh X,\wh\cF)$ has cohomology in degrees $0$ and $1$ at most, as wanted.

Let us now show that $\cP=\cP(\cF_\leq,\wh\cF,\wh\nu)$ is perverse on~$X$. Obviously, we only have to check the support and cosupport conditions at~$D$, and we will use Proposition \ref{prop:StDPervimdir}.

For the support condition, we note that $i_D^{-1}\bR\varpi_*\cF_{\leq0}=\bR\Gamma(\partial\wt X,\wti{}^{-1}\cF_{\leq0})$, and $i_D^{-1}\cP[1]$ is the cone of the diagonal morphism below (where $\wh\nu$ is implicitly used):
\begin{equation}\label{eq:Sttri}
\begin{array}{c}
\xymatrix{
&\bR\Gamma(\partial\wt X,\wti{}^{-1}\cF_{\leq0})\ar[d]\ar[dr]&&\\
\bR\Gamma(\wh X,\wh\cF)\ar[r]&\bR\Gamma(\partial\wt X,\gr_0\cF)\ar[r]&\bR\Gamma_\rc(\wh X,\wh\cF)[1]\ar@<-.3ex>[r]^-{+1}&
}
\end{array}
\end{equation}
This gives an exact sequence
\begin{multline}\label{eq:Stsuiteex}
0\to\cH^{-1}(i_D^{-1}\cP)\to \HH^{-1}(\partial\wt X,\wti{}^{-1}\cF_{\leq0})\to\HH^0_\rc(\wh X,\wh\cF)\to\cH^0(i_D^{-1}\cP)\\
\to \HH^0(\partial\wt X,\wti{}^{-1}\cF_{\leq0})\to\HH^1_\rc(\wh X,\wh\cF)\to\cH^1(i_D^{-1}\cP)\to0
\end{multline}
and the support condition reduces to the property that $\HH^0(\partial\wt X,\wti{}^{-1}\cF_{\leq0})\to\HH^1_\rc(\wh X,\wh\cF)$ is onto, because we already know that $\HH^2_\rc(\wh X,\wh\cF)=0$. On the one hand, $\HH^0(\partial\wt X,\wti{}^{-1}\cF_{\leq0})\to\HH^0(\partial\wt X,\gr_0\cF)$ is the morphism $H^1(\partial\wt X,\cL_{\leq0})\to H^1(\partial\wt X,\gr_0\cL)$, which is onto since $\cL_{<0}$ is a sheaf on $\partial\wt X$. On the other hand the distinguished horizontal triangle above gives an exact sequence
\[
\HH^0(\partial\wt X,\gr_0\cF)\to \HH^1_\rc(\wh X,\wh\cF)\to\HH^1(\wh X,\wh\cF),
\]
and $\HH^1(\wh X,\wh\cF)=\cH^1(i_{\wh D}^{-1}\wh\cF)=0$ by the support condition for $\wh\cF$.

Let us now check the cosupport condition. We have $\wti^!\cF_{\leq0}=\cL/\cL_{\leq0}$, as remarked at the end of the proof of Lemma \ref{lem:StCdual}. We now argue as above, by replacing $\bR\Gamma(\partial\wt X,\wti{}^{-1}\cF_{\leq0})$ with $\bR\Gamma(\partial\wt X,\wti{}^!\cF_{\leq0})$ in \eqref{eq:Sttri}, so that \eqref{eq:Stsuiteex} becomes
\begin{multline*}
0\to\cH^{-1}(i_D^!\cP)\to \HH^{-1}(\partial\wt X,\wti{}^!\cF_{\leq0})\to\HH^0_\rc(\wh X,\wh\cF)\to\cH^0(i_D^!\cP)\\
\to \HH^0(\partial\wt X,\wti{}^!\cF_{\leq0})\to\HH^1_\rc(\wh X,\wh\cF)\to\cH^1(i_D^!\cP)\to\HH^1(\partial\wt X,\wti{}^!\cF_{\leq0})\to0,
\end{multline*}
since $\HH^2_\rc(\wh X,\wh\cF)=0$. The cosupport condition means $\cH^{-1}(i_D^!\cP)=0$, which follows from $\HH^{-1}(\partial\wt X,\wti{}^!\cF_{\leq0})=H^{-1}(\partial\wt X,\cL/\cL_{\leq0})=0$.
\end{proof}

\chapter[Riemann-Hilbert correspondence]{The Riemann-Hilbert correspondence for holonomic $\cD$-modules on curves}\label{chap:RH}

\begin{sommaire}
In this \chaptername, we define the Riemann-Hilbert functor on a Riemann surface~$X$ as a functor from the category of holonomic $\cD_X$-modules to that of Stokes-perverse sheaves. It is induced from a functor at the derived category level which is compatible with $t$-structures. Given a discrete set~$D$ in~$X$, we first define the functor from the category of $\cD_X(*D)$-modules which are holonomic and have regular singularities away from~$D$ to that of Stokes-perverse sheaves on $\wt X(D)$, and we show that it is an equivalence. We then extend the correspondence to holonomic $\cD_X$-modules with singularities on $D$, on the one hand, and Stokes-perverse sheaves on $\underline{\wt X}(D)$ on the other hand.
\end{sommaire}

\subsection{Introduction}

In this \chaptername, the base field $\kk$ is $\CC$. Let~$X$ be a Riemann surface and let~$D$ be a discrete set of points in~$X$ as in \S\ref{subsec:settingStPerv} (from which we keep the notation). We also denote by $\varpi:\wt X(D)\to X$ the real oriented blowing-up at all points of $D$.

We refer for instance to \cite{Kashiwara70,Bibi90b,Malgrange91} for basic results on holonomic $\cD$-modules on a Riemann surface. Recall that $\cD_X$ denotes the sheaf of holomorphic differential operators on $X$ with coefficients in $\cO_X$, and that a $\cD_X$-module $\cM$ on a Riemann surface is nothing but a $\cO_X$-module with a holomorphic connection (by a $\cD_X$-module or $\cD_X(*D)$-module, we usually mean a left module). It is \index{holonomic $\cD$-modules}\emph{holonomic} if it is $\cD_X$-coherent and if moreover any local section of $\cM$ is annihilated by some nonzero local section of $\cD_X$. Given a $\cD_X$-module $\cM$, its (holomorphic) de~Rham complex $\DR\cM$ is the complex $\cM\To\nabla\Omega^1_X\otimes_{\cO_X}\cM$.

If $\cM$ is $\cO_X$-coherent, it is then $\cO_X$-locally free, hence is a vector bundle with connection. In such a case, the de~Rham complex $\DR\cM$ has cohomology in degree~$0$ at most and $\cH^0\DR\cM=\ker\nabla=\cM^\nabla$ is the sheaf of horizontal local sections of~$\cM$. It is a locally constant sheaf of finite dimensional $\CC$-vector spaces. Conversely, $\cM$ can be recovered from $\cM^\nabla$ as $\cM\simeq\cO_X\otimes_\CC\cM^\nabla$ with the standard connection on $\cO_X$. This defines a perfect correspondence (equivalence of categories) between holomorphic vector bundles with connection and locally constant sheaves of finite-dimensional $\CC$-vector spaces.

This elementary form of the Riemann-Hilbert correspondence extends to the case where $\cM$ is a holonomic $\cD_X$-module with regular singularities. In such a case, the shifted de~Rham functor \index{$DRAP$@$\pDR\cM$}$\pDR\cM=\DR\cM[1]$ is an equivalence from the category of such $\cD_X$-modules and that of $\CC$-perverse sheaves on $X$.

In this \chaptername, we consider holonomic $\cD_X$-modules without any restriction on their singularities (except that they are located on a fixed set $D$). It is well-known that the de~Rham complex above is not sufficient to recover the whole information of such a module when it has irregular singularities. We therefore extend the target category, which will be the category of Stokes-perverse sheaves as introduced in \Chaptersname\ref{chap:Stokesone-pervers}, and define a modified de~Rham functor with values in this category. Its $\cH^0$ consists of horizontal local sections with moderate growth or rapid decay near the singular set~$D$.

We will prove the Riemann-Hilbert correspondence (equivalence of categories) in this setting, and we will also check the compatibility of the extended de~Rham functor with the duality functor, either for holonomic $\cD_X$-modules, or for Stokes-perverse sheaves (\cf\S\ref{subsec:dualStperv}).

This \chaptersname mainly follows \cite{Deligne78b,Deligne84cc}, \cite{Malgrange83b}, \cite{B-V89} and \cite[\S IV.3]{Malgrange91}.
 
\Subsection{Some basic sheaves}\label{subsec:somebasic}

\subsubsection*{On~$X$}
We denote by $\cO_X$ the sheaf of holomorphic functions on~$X$, by \index{$DAX$@$\cD_X$}$\cD_X$ the sheaf of holomorphic differential operators, by $\cO_X(*D)$ the sheaf of meromorphic functions on~$X$ with poles of arbitrary order on~$D$, and we set \index{$DAXD$@$\cD_X(*D)$}$\cD_X(*D)=\cO_X(*D)\otimes_{\cO_X}\cD_X$.

\subsubsection*{On $\wt X=\wt X(D)$}
Since $\wt X(D)$ is a real manifold with boundary, the notion of a~$C^\infty$ function on $\wt X(D)$ is well-defined. Moreover, working in local polar coordinates $x=re^{i\theta}$, one checks that the operator $\ov x\partial_{\ov x}=\frac12(r\partial_r+i\partial_\theta)$ acts on $\cC^\infty_{\wt X}$ and we can lift the $\ov\partial$ operator on $\cC^\infty_X$ as an operator
\[
\ov\partial:\cC^\infty_{\wt X}\to \cC^\infty_{\wt X}\otimes_{\varpi^{-1}\cO_{\ov X}}\varpi^{-1}\Omega^1_{\ov X}(\log\ov D),
\]
where $\ov X$ denotes the complex conjugate Riemann surface. We will then set
\[
\index{$AXT$@$\cA_{\wt X}$}\cA_{\wt X}=\ker\ov\partial.
\]
This is a $\varpi^{-1}\cD_X$-module, which coincides with $\cO_{X^*}$ on $X^*\defin X\moins D=\wt X\moins\varpi^{-1}(D)$. We set $\cD_{\wt X}=\cA_{\wt X}\otimes_{\varpi^{-1}\cO_X}\nobreak\varpi^{-1}\cD_X$ and we define similarly $\cA_{\wt X}(*D)$ and $\cD_{\wt X}(*D)$ by tensoring with $\varpi^{-1}\cO_X(*D)$.

We have already implicitly defined the sheaf \index{$AXTMOD$@$\cA_{\wt X}^\modD$}$\cA_{\wt X}^\modD$ (\cf Example \ref{exem:Stokes}) of \index{holomorphic functions!with moderate growth ($\cA_{\wt X}^\modD$)}holomorphic functions on $X^*$ having moderate growth along $\partial\wt X$. Let us recall the definition in the present setting. Given an open set $\wt U$ of $\wt X$, a section $f$ of $\cA_{\wt X}^\modD$ on $\wt U$ is a holomorphic function on $U^*\defin\wt U\cap X^*$ such that, for any compact set $K$ in $\wt U$, in the neighbourhood of which~$D$ is defined by $g_K\in\cO_X(K)$, there exists constants $C_K>0$ and $N_K\geq0$ such that $|f|\leq C_K|g_K|^{-N_K}$ on $K$.

Similarly, \index{$AXTRD$@$\cA_{\wt X}^\rdD$}$\cA_{\wt X}^\rdD$ consists of \index{holomorphic functions!with rapid decay ($\cA_{\wt X}^\rdD$)}functions having rapid decay along $\partial\wt X$: change the definition above by asking that, for any $K$ and any $N\geq0$, there exists a constant $C_{K,N}$ such that $|f|\leq C_{K,N}|g_K|^N$ on $K$.

Both sheaves $\cA_{\wt X}^\modD$ and $\cA_{\wt X}^\rdD$ are left $\cD_{\wt X}(*D)$-modules and are $\varpi^{-1}\cO_X(*D)$-flat. It will be convenient to introduce the quotient sheaf $\cA_{\wt X}^\grD\defin\cA_{\wt X}^\modD/\cA_{\wt X}^\rdD$ with its natural $\cD_{\wt X}(*D)$-module structure. It is supported on $\partial\wt X$ and is $\varpi^{-1}\cO_X$-flat (because it has no nonzero $\varpi^{-1}\cO_X$-torsion).

More information on these sheaves can be found for instance in \cite{Malgrange83b,Majima84,Malgrange91,Bibi97}. In particular, $\bR\varpi_*\cA_{\wt X}^\modD=\varpi_*\cA_{\wt X}^\modD=\cO_X(*D)$, which will be reproved in any dimension in Proposition \ref{prop:RpiA}.

\subsubsection*{On $\ccIet$}
Recall that $\ccI$ is the extension by $0$ of the sheaf $\ccI_{\partial\wt X}$. We will use the definition of Remark \ref{rem:Euler} in order to consider $\ccI_{\partial\wt X}$ as a subsheaf of $i_\partial^{-1}j_{\partial,*}\cO_{X^*}/i_\partial^{-1}(j_{\partial,*}\cO_{X^*})^\lb$ on $\partial\wt X$. We will use the following lemma.

\begin{lemme}\label{lem:ephiAmod}
Let $U$ be an open subset of $\partial\wt X$ and let $\varphi\in\Gamma(U,\ccI_{\partial\wt X})$. There exists a unique subsheaf of $(i_\partial^{-1}j_{\partial,*}\cO_{X^*})_{|U}$, denoted by $e^\varphi\cA_U^\modD$ (\resp by $e^\varphi\cA_U^\rdD$) such that, for any $\theta\in U$ and any lifting $\wt\varphi_\theta\in\wt\ccI_\theta$ of the germ $\varphi_\theta\in\ccI_\theta$, the germ $(e^\varphi\cA_U^\modD)_\theta$ (\resp $(e^\varphi\cA_U^\rdD)_\theta$) is equal to $e^{\wt\varphi_\theta}\cA^\modD_\theta\subset(j_{\partial,*}\cO_{X^*})_\theta$ (\resp $e^{\wt\varphi_\theta}\cA^\rdD_\theta\subset(j_{\partial,*}\cO_{X^*})_\theta$).
\end{lemme}

\begin{proof}
It is a matter of checking that $e^{\wt\varphi_\theta}\cA^\modD_\theta$ (\resp $e^{\wt\varphi_\theta}\cA^\rdD_\theta$) does not depend on the choice of the lifting $\wt\varphi_\theta$. This is clear since any two such choices differ by a locally bounded function.
\end{proof}

We will define sheaves \index{$AZIET$@$\cA_{\ccIet}^\modD$, $\cA_{\ccIet}^\modD$, $\cA_{\ccIet}^\grD$}$\cA^\modD_{\ccIet}$ and $\cA^\rdD_{\ccIet}$ on $\ccIet$ by their restrictions to the closed and open subsets above, together with the gluing morphism.
\begin{itemize}
\item
On $\ccIet_{X^*}=X^*$, we set $\cA^\modD_{\ccIet_{X^*}}=\cA^\rdD_{\ccIet_{X^*}}=\cO_{X^*}$.
\item
On $\ccIet_{|\partial\wt X}$, we universally twist the sheaves $\mu_\partial^{-1}\cA^\modD_{\partial\wt X}$ and $\mu_\partial^{-1}\cA^\rdD_{\partial\wt X}$ by $e^\varphi$ as follows. For any open subset $U\subset\partial\wt X$ and any $\varphi\in\Gamma(U,\ccI)$, we set $\varphi^{-1}\cA^\modD_{\ccIet_{|\partial\wt X}}\defin e^\varphi\cA^\modD_U\subset(i_\partial^{-1}j_{\partial,*}\cO_{X^*})_{|U}$ and we define $\varphi^{-1}\cA^\rdD_{\ccIet_{|\partial\wt X}}$ similarly (\cf Lemma \ref{lem:ephiAmod} above; note that on the left-hand side, $\varphi$ is considered as a section of $\mu_\partial:\ccIet_{|\partial\wt X}\to\partial\wt X$, while on the right-hand side, $\varphi$ and $e^\varphi$ are considered as local sections of $i_\partial^{-1}j_{\partial,*}\cO_{X^*}$). One can check that these data correspond to sheaves $\cA^\modD_{\ccIet_{|\partial\wt X}}$ and $\cA^\rdD_{\ccIet_{|\partial\wt X}}$, with $\cA^\rdD_{\ccIet_{|\partial\wt X}}\subset\cA^\modD_{\ccIet_{|\partial\wt X}}\subset\wti{}^{-1}\wtj_*\cO_{X^*}$.
\item
The previous inclusion is used for defining the gluing morphism defining  $\cA^\rdD_{\ccIet}$ and $\cA^\modD_{\ccIet}$ as sheaves on $\ccIet$.
\end{itemize}

We have natural inclusions $\cA^\rdD_{\ccIet}\subset\cA^\modD_{\ccIet}\subset\wtj_*\cO_{X^*}$ of sheaves on $\ccIet$. These are inclusions of sheaves of $\mu^{-1}\cD_{\wt X}(*D)$-modules. Indeed, this is clear on $X^*$, and on $\ccIet_{|\partial\wt X}$ one remarks that, in a local coordinate $x$ used for defining $\ccI$, $\partial_x(e^\varphi\cA^\modD_{\partial\wt X})=e^\varphi(\partial_x+\partial\varphi/\partial x)\cA^\modD_{\partial\wt X}\subset e^\varphi\cA^\modD_{\partial\wt X}$, and similarly for $\cA^\rdD_{\partial\wt X}$.

It will be convenient to introduce the notation $\cA^\grD_{\ccIet}$ for the quotient sheaf $\cA^\modD_{\ccIet}/\cA^\rdD_{\ccIet}$. This sheaf is supported on $\ccIet_{|\partial\wt X}$, and is also equipped with a natural $\mu^{-1}\cD_{\wt X}(*D)$-module structure.

These sheaves are flat over $\mu^{-1}\varpi^{-1}\cO_X$ (\ie they have no torsion).

\begin{lemme}\label{lem:AmodpreIfilt}
The sheaves $\cA^\modD_{\ccIet}$, $\cA^\rdD_{\ccIet}$ are pre-$\ccI$-filtrations of $j_{\partial,*}\cO_{X^*}$, the sheaf $\cA^\grD_{\ccIet}$ is a graded pre-$\ccI$-filtration, and the exact sequence
\[
0\to\cA^\rdD_{\ccIet}\to\cA^\modD_{\ccIet}\to\cA^\grD_{\ccIet}\to0
\]
is an exact sequence in $\Mod(\CC_{\ccIet,\leq})$. Moreover, the pre-$\ccI$-filtration morphisms $\beta_\leq p_1^{-1}\to p_2^{-1}$ are compatible with the $\mu^{-1}\cD_{\wt X}$-action.
\end{lemme}

\begin{proof}
In order to prove the abstract pre-$\ccI$-filtration property, we need to prove that there are natural morphisms $\beta_\leq p_1^{-1}\cA^\modD_{\ccIet}\to p_2^{-1}\cA^\modD_{\ccIet}$, etc.\, compatible with the exact sequence of the lemma and the $\mu^{-1}\cD_{\wt X}$-action. This amounts to defining natural morphisms $\beta_{\varphi\leq\psi}e^\varphi\cA^\modD_{\partial\wt X}\to e^\psi\cA^\modD_{\partial\wt X}$, etc.\, for any open set $U\subset\partial\wt X$ and any pair $\varphi,\psi\in\Gamma(U,\ccI_{\partial\wt X})$. On $U_{\varphi\leq\psi}$, we have $e^\varphi\cA^\modD_{\partial\wt X}=e^\psi e^{\varphi-\psi}\cA^\modD_{\partial\wt X}\subset e^\psi\cA^\modD_{\partial\wt X}$, and this inclusion defines the desired morphism. Notice that it is compatible with the action of $\mu_\partial^{-1}\cD_{\wt X}$ by definition.

Clearly, the natural inclusions of $\cA^\modD_{\ccIet}$ and $\cA^\rdD_{\ccIet}$ in $\wtj_*\cO_{X^*}=\mu^{-1}j_{\partial,*}\cO_{X^*}$ make these sheaves pre-$\ccI$-filtrations of $j_{\partial,*}\cO_{X^*}$.
\end{proof}

\Subsection{The Riemann-Hilbert correspondence for germs}\label{subsec:RHgermsmero}

\Subsubsection*{The Riemann-Hilbert correspondence for germs of meromorphic connections}\index{Riemann-Hilbert correspondence!for germs of meromorphic connections}
We recall here the fundamental results of Deligne \cite{Deligne78b,Deligne84cc} and Malgrange \cite{Malgrange83b}, \cite[\S IV.3]{Malgrange91}. We work here in a local setting, and we denote by $\cO$, $\cD$, etc.\ the germs at $0\in\CC$ of the corresponding sheaves. We denote by~$S^1$ the fibre over~$0$ of the real blow-up space of $\CC$ at $0$, and by $\wt\cO$, $\cA^{\rmod0}$, $\cA^{\rd0}$, $\cA^{\gr_0}$ the restriction to~$S^1$ of the sheaves $j_{\partial,*}\cO_{X^*}$, $\cA^\modD_{\wt X}$, $\cA^\rdD_{\wt X}$ and $\cA^\grD_{\wt X}$ introduced above.

By a meromorphic connection $\cM$ we mean a free $\cO(*0)$-module of finite rank with a connection. We can form various de~Rham complexes, depending on the sheaves of coeffcients.
\begin{enumerate}
\item
\index{$DR$@$\DR\cM$}$\DR\cM=\{\cM\To{\nabla}\Omega^1\otimes\cM\}$ is the ordinary holomorphic de~Rham complex.
\item
\index{$DRWT$@$\wt\DR\cM$}$\wt\DR\cM=\{\wt\cO\otimes\varpi^{-1}\cM\To{\nabla}\wt\cO\otimes\varpi^{-1}(\Omega^1\otimes\cM)\}$ is the multivalued essential holomorphic de~Rham complex.
\item
\index{$DRMOD$@$\DR^{\rmod0}\cM$, $\DR^\modD\cM$}$\DR^{\rmod0}\cM=\{\cA^{\rmod0}\otimes\varpi^{-1}\cM\To{\nabla}\cA^{\rmod0}\otimes\varpi^{-1}(\Omega^1\otimes\cM)\}$ is the moderate de~Rham complex (the tensor products are taken over $\varpi^{-1}\cO$ and the connection is extended in a natural way, by using the $\varpi^{-1}\cD$-structure on $\cA^{\rmod0}$).
\item
\index{$DRRD$@$\DR^{\rd0}\cM$, $\DR^\rdD\cM$}$\DR^{\rd0}\cM=\{\cA^{\rd0}\otimes\varpi^{-1}\cM\To{\nabla}\cA^{\rd0}\otimes\varpi^{-1}(\Omega^1\otimes\cM)\}$ is the rapid-decay de~Rham complex.
\end{enumerate}

Since $\bR\varpi_*\cA^{\rmod0}=\varpi_*\cA^{\rmod0}=\cO(*0)$ as already recalled, the projection formula for the proper map $\varpi$ gives
\[
\DR\cM\simeq\bR\varpi_*\DR^{\rmod0}\cM.
\]
On the other hand, $\wt\DR\cM$ has cohomology in degree $0$ at most, and this cohomology is the locally constant sheaf with the same monodromy as the monodromy of the local system of horizontal sections of $\cM$ away from~$0$. We will denote by $\cL$ the corresponding local system on~$S^1$.

\begin{theoreme}\label{th:H1nul}
For any germ $\cM$ of meromorphic connection, the complexes $\DR^{\rmod0}\cM$, $\DR^{\rd0}\cM$ and $\DR^{\gr_0}\cM$ have cohomology in degree~$0$ at most. The natural morphisms $\DR^{\rd0}\cM\to\DR^{\rmod0}\cM\to\wt\DR\cM$ induce inclusions $\cH^0\DR^{\rd0}\cM\hto\cH^0\DR^{\rmod0}\cM\hto\cH^0\wt\DR\cM$, and $\cH^0\DR^{\gr_0}\cM$ is equal to $\cH^0\DR^{\rmod0}\cM/\cH^0\DR^{\rd0}\cM$.
\end{theoreme}

\begin{proof}[Sketch of the proof]
We will indicate that $\DR^{\rmod0}\cM$ has cohomology in degree $0$ at most. The proof for $\DR^{\rd0}\cM$ and $\DR^{\gr_0}\cM$ is similar, and the remaining part is then straightforward.
\begin{enumerate}
\item\label{proof:varphi}
One first proves the result in the case where \index{$EFHI$@$\cE^\varphi$}$\cM=\cE^\varphi\defin(\cO(*0),d+d\varphi)$ for $\varphi\in\ccP_x$ (\cf \eg \cf \eg \cite{Wasow65} or \cite[App\ptbl1]{Malgrange91}).
\item\label{proof:xalpha}
Using that, for any complex number $\alpha$, the local branches of the function~$x^\alpha$ determine invertible local sections of $\cA_{\wt X}^{\rmod0}$, \eqref{proof:varphi} implies that the result holds for any~$\cM$ of the form $\cE^\varphi\otimes\cR_\alpha$, where $\cR_\alpha$ is the free $\cO[1/x]$-module of rank one with the connection $d+\alpha dx/x$.
\item
Using the well-known structure of free $\cO[1/x]$-modules (of any rank) with a connection having a regular singularity, it follows from \eqref{proof:xalpha} that the result holds for any $\cM$ of the form $\cE^\varphi\otimes\cR_\varphi$, where $\cR_\varphi$ has a connection with regular singularity. We call a direct sum of such modules an \emphb{elementary model}.
\item\label{proof:directsum}
The theorem holds then for any $\cM$ such that, locally on~$S^1$, $\cA_{\wt X}^{\rmod0}\otimes_{\varpi^{-1}\cO}\varpi^{-1}\cM$ is isomorphic to an elementary model.
\item\label{proof:HT}
By the \emphb{Hukuhara-Turrittin theorem} (\cf \eg \cite{Wasow65} or \cite[App\ptbl1]{Malgrange91}), there exists a ramification $\rho:S^1_{x'}\to S^1_x$ such that the pull-back $\rho^+\cM$ satisfies the previous property (more precisely, $\rho^+\cM$ is locally isomorphic to an elementary model after tensoring with the sheaf $\cA_{x'}^{\rmod0}$), hence the result holds for $\rho^+\cM$ for any $\cM$ and a suitable $\rho$ (that we can assume to be of the form $x'\mto x^{\prime d}=x$).
\item\label{proof:rho}
We have $\cA^{\rmod0}_{x'}=\rho^{-1}\cA^{\rmod0}_x$. Then
\[
\cA^{\rmod0}_{x'}\otimes_{\varpi^{\prime-1}\cO}\varpi^{\prime-1}\cM= \rho^{-1}\big(\cA^{\rmod0}_x\otimes_{\varpi^{-1}\cO}\varpi^{-1}\cM\big),
\]
and since $x'\partial_{x'}$ acts like as $d\cdot x\partial_x$, we conclude that $\DR^{\rmod0}(\rho^+\cM)\simeq\wt\rho^{-1}\DR^{\rmod0}(\cM)$. As a consequence, if the theorem holds for $\rho^+\cM$, it holds for $\cM$.\qedhere
\end{enumerate}
\end{proof}

\begin{corollaire}\label{cor:RHStokes}
Let $U$ be any open set of~$S^1$, let $\varphi\!\in\!\Gamma(U,\ccI)$ and set (\cf Lemma \ref{lem:ephiAmod})
\begin{align*}\index{$DRZPHI$@$\DR_{\leq\varphi}$, $\DR_{<\varphi}$, $\DR_{\gr\varphi}$}
\DR_{\leq\varphi}\cM&=\DR(e^\varphi\cA^{\rmod0}\otimes\cM)\\\DR_{<\varphi}\cM&=\DR(e^\varphi\cA^{\rd0}\otimes\cM)\\
\DR_{\gr_\varphi}\cM&=\DR([e^\varphi\cA^{\rmod0}/e^\varphi\cA^{\rd0}]\otimes\cM).
\end{align*}
Then these complexes have cohomology in degree $0$ at most. The natural morphisms between these complexes induce inclusions $\cH^0\DR_{<\varphi}\cM\hto\cH^0\DR_{\leq\varphi}\cM\hto\cH^0\wt\DR\cM$, and $\cH^0\DR_{\gr_\varphi}\cM$ is equal to $\cH^0\DR_{\leq\varphi}\cM/\cH^0\DR_{<\varphi}\cM$.
\end{corollaire}

We will use both notations $\DR_{\leq0}$ and $\DR^{\rmod0}$ (\resp $\DR_{<0}$ and $\DR^{\rd0}$) for the same complex.

\begin{proof}
Assume first that $\varphi$ is not ramified, \ie comes from an element of $\ccP_x$. Termwise multiplication by $e^{-\varphi}$ induces an isomorphism of complexes
\[
\DR_{\leq\varphi}\cM\isom\DR^{\rmod0}(\cE^\varphi\otimes\cM)
\]
and similarly for $\DR_{<\varphi}$ and $\DR_{\gr_\varphi}$. We can apply the previous result to $\cE^\varphi\otimes\cM$ (where we recall that $\cE^\varphi=(\cO(*0),d+d\varphi)$).

If $\varphi$ is ramified, then we first perform a suitable ramification $\rho$ so that $\varphi$ defines a section of $\ccP_{x'}$, and we argue as in \eqref{proof:rho} of the proof of Theorem \ref{th:H1nul} to descend to~$\cM$.
\end{proof}

\begin{proposition}\label{prop:RHStokes}
Let $\rho$ be such that $\cM'=\rho^+\cM$ is locally isomorphic on $S^1_{x'}$ to an elementary model after tensoring with $\cA^{\rmod0}_{x'}$. Then the subsheaves $\cL'_{\leq\varphi'}\defin\cH^0\DR_{\leq\varphi'}\cM'$ considered above ($\varphi'\in\ccP_{x'}$) form a non-ramified Stokes filtration on $\cL'\defin\cH^0\wt\DR\cM$. Moreover, it is a Stokes filtration on $\cL$, for which $\cL_{\leq\varphi}=\cH^0\DR_{\leq\varphi}\cM$, $\cL_{<\varphi}=\cH^0\DR_{<\varphi}\cM$, and $\gr_\varphi\cL=\cH^0\DR_{\gr_\varphi}\cM$, for any local section $\varphi$ of~$\ccI$.
\end{proposition}

\begin{proof}
For the first point, one has to check the local gradedness property for $(\cL',\cL'_\bbullet)$. This is a direct consequence of the Hukuhara-Turrittin mentioned in \eqref{proof:HT} of the proof above. Indeed, this theorem reduces the problem to showing that the pre-Stokes filtration associated to a connection with regular singularity is the graded Stokes filtration with zero as only exponential factor. This reduces to checking when the function $e^{\varphi}\mx^\alpha(\log\mx)^k$ has moderate growth, and this reduces to checking when the function $e^{\varphi}$ has moderate growth: we recover exactly the graded Stokes filtration.

For the second point, we have to check that $\cL'_{\leq\sigma^*\varphi'}=\wt\sigma^{-1}\cL'_{\leq\varphi'}$ (\cf Definition \ref{def:Stokeswithramif}). This follows from
\[
\sigma^+(\cE^{\varphi'}\otimes\rho^+\cM)=\cE^{\sigma^*\varphi'}\otimes\rho^+\cM.\qedhere
\]
\end{proof}

\begin{remarques}\label{rem:infvarphi}
Let us fix $\theta\in\partial\wt X$. Assume that $\cA_{\wt X,\theta}\otimes\cM\simeq\bigoplus_{\varphi\in\Phi}\cA_{\wt X,\theta}\otimes\nobreak(\cE^\varphi\otimes\nobreak\cR_\varphi)$ for some germs of flat meromorphic bundles with a regular singularity, where $\Phi$ is a finite set in $\ccP_x$.
\begin{enumerate}
\item\label{rem:infvarphi1}
Note first that, for $\eta\in\ccP_x$, $\DR^{\rmod0}(\cE^\eta\otimes\cR)_\theta\neq0$ iff $e^{-\eta}\in\cA^{\rmod0}_{\wt X,\theta}$, that is, iff $0\leqtheta\eta$.
\item\label{rem:infvarphi2}
It follows that, for $\cM$ as above, $\DR_{\leq\eta}(\cM)_\theta\simeq\bigoplus_{0\leqtheta\eta+\varphi}\DR^{\rmod0}(\cE^{\varphi+\eta}\otimes\cR_\varphi)_\theta$. In particular, $\DR_{\gr_\eta}(\cM)_\theta\simeq\DR^{\rmod0}(\cR_{-\eta})_\theta$.
\item\label{rem:infvarphi3}
Let us define the filtration $(\cA_{\wt X,\theta}\otimes\cM)_{\leq\eta}$ by $\bigoplus_{\varphi\mid0\leqtheta\eta+\varphi}\cA_{\wt X,\theta}\otimes(\cE^\varphi\otimes\cR_\varphi)$. This is the approach taken in \cite{Mochizuki08,Mochizuki10b}. We then have
\[
\DR_{\leq\eta}(\cM)_\theta\simeq\DR^{\rmod0}\big(\cE^\eta\otimes(\cA_{\wt X,\theta}\otimes\cM)_{\leq\eta}\big).
\]
\end{enumerate}
\end{remarques}

\begin{definitio}[The Riemann-Hilbert functor]\label{def:RHfunctorgerms}
The \index{Riemann-Hilbert functor (RH)}Riemann-Hilbert functor~\index{$RH$@$\RH$}$\RH$ from the category of germs at $0$ of meromorphic connections to the category of Stokes-filtered local systems is the functor which sends an object $\cM$ to $(\cH^0\wt\DR\cM,\cH^0\DR_\leq\cM)$ and to a morphism the associated morphism at the de~Rham level.
\end{definitio}

\begin{theoreme}[Deligne \cite{Deligne78b}, Malgrange \cite{Malgrange83b}]\label{th:RHequivgermmero}
The Riemann-Hilbert functor $\cM\mto(\cH^0\wt\DR\cM,\cH^0\DR_\leq\cM)$ is an equivalence of categories.
\end{theoreme}

\begin{proof}[\proofname\ of the full faithfulness]
We first define a morphism
\[
\cH^0\DR^{\rmod0}\cHom_{\cO_X}(\cM,\cM')\to\cHom(\cL,\cL')_{\leq0}.
\]
The source of the desired morphism consists of local morphisms $\cA^{\rmod0}\otimes\cM\to\cA^{\rmod0}\otimes\cM'$ which are compatible with the connection. These sections also send $e^\varphi\cA^{\rmod0}\otimes\cM$ in $e^\varphi\cA^{\rmod0}\otimes\cM'$ for any section $\varphi$ of $\ccI$, hence, restricting to the horizontal subsheaves, send $\cL_{\leq\varphi}$ in $\cL'_{\leq\varphi}$ for any such $\varphi$. In other words, they define sections of $\cHom(\cL,\cL')_{\leq0}$.

To show that this morphism is an isomorphism, it is enough to argue locally near $\theta_o\in S^1$ and after a ramification, so that one can use the Hukuhara-Turrittin decomposition for $\cHom_{\cO_X}(\cM,\cM')$ coming from that of $\cM$ (indexed by $\Phi$) and of $\cM'$ (indexed by $\Phi'$). In fixed local bases of $\cA^{\rmod0}\otimes\cM$ and $\cA^{\rmod0}\otimes\cM'$ adapted to this decomposition, the matrix of a local section of $\cH^0\DR^{\rmod0}\cHom_{\cO_X}(\cM,\cM')$ has blocks $e^{\varphi-\varphi'}u_{\varphi,\varphi'}$, where $u_{\varphi,\varphi'}$ is a local horizontal section of $\cHom_{\cO_X}(\cR_\varphi,\cR_{\varphi'})$, that is, of $\cHom(\gr_\varphi\cL,\gr_{\varphi'}\cL')$. The condition that it has moderate growth is equivalent to $\varphi\leqthetao\varphi'$, that is, is a section of $\cHom(\gr_\varphi\cL,\gr_{\varphi'}\cL')_{\leq0}$, as wanted.

Let us set $\cN=\cHom_{\cO_X}(\cM,\cM')$ with its natural connection. Then we have $\Hom_\cD(\cM,\cM')=\cH^0\DR\cN$. On the one hand, by the projection formula, $\DR\cN=\bR\varpi_*\DR^{\rmod0}\cN$ and by applying Theorem \ref{th:H1nul} to $\cN$, this is $\bR\varpi_*\cH^0\DR^{\rmod0}\cN$. Therefore, $\Hom_\cD(\cM,\cM')=\varpi_*\cH^0\DR^{\rmod0}\cN$.

As a consequence, $\Hom_\cD(\cM,\cM')=\Gamma\big(S^1,\cHom(\cL,\cL')_{\leq0}\big)$, and the latter term is $\Hom\big((\cL,\cL_\bbullet),(\cL',\cL'_\bbullet)\big)$.
\end{proof}

\begin{proof}[\proofname\ of the essential surjectivity]
We will need:
\begin{lemme}\label{lem:elem}
Any graded non-ramified Stokes-filtered local system on~$S^1$ comes from an elementary model by Riemann-Hilbert.
\end{lemme}

\begin{proof}
Using Exercise \ref{exo:graded} and the twist, one is reduced to showing the lemma for the Stokes-graded local system with only one exponential factor, this one being equal to zero. We have indicated above that the germ of connection with regular singularity corresponding to the local system $\cL$ gives the corresponding Stokes-filtered local system.
\end{proof}

We will first prove the essential surjectivity when $(\cL,\cL_\bbullet)$ is non-ramified. The isomorphism class of $(\cL,\cL_\bbullet)$ is obtained from the Stokes-graded local system $\gr_\bbullet\cL$ by giving a class in $H^1(S^1,\cAut^{<0}(\gr_\bbullet\cL))$.

Let $\cM^\el$ be the elementary model corresponding to the Stokes-graded local system $\gr_\bbullet\cL$ as constructed in Lemma \ref{lem:elem}. Applying the last part of Proposition \ref{prop:RHStokes} to $\cEnd(\cM^\el)$ and the full faithfulness, we find
\[
\cEnd^{\rd0}_{\cD}(\cM^\el)\defin\cH^0\DR_{<0}\cEnd_\cO(\cM^\el)=\cHom(\gr\cL,\gr\cL)_{<0},
\]
and therefore $\cAut^{\rd0}(\cM^\el)\defin\id+\cEnd^{\rd0}_{\cD}(\cM^\el)=\cAut^{<0}(\gr_\bbullet\cL)$.

A cocycle of $\cAut^{<0}(\gr_\bbullet\cL)$ determining $(\cL,\cL_\bbullet)$ (with respect to some covering~$(I_i)$ of~$S^1$) determines therefore a cocycle of $\cAut^{\rd0}(\cM^\el)$. If we fix a $\cO(*0)$-basis of $\cM^\el$, the \emphb{Malgrange-Sibuya theorem} (\cf\cite{Malgrange83b,Malgrange91}) allows one to lift it as a zero cochain $(f_i)$ of $\GL_d(\cA^{\rmod0})$, where $d=\rk\cM$ (the statement of the Malgrange-Sibuya theorem is in fact more precise). This zero cochain is used to twist the connection $\nabla^\el$ of $\cM^\el$ on $\cA^{\rmod0}\otimes\cM^\el$ as follows: on $I_i$ we set $\nabla_i=f_i^{-1}\nabla^\el f_i$. Since $f_if_j^{-1}$ is $\nabla^\el$-flat on $I_i\cap I_j$, the connections $\nabla_i$ glue together as a connection $\nabla$ on the free $\cA^{\rmod0}$-module $\cA^{\rmod0}\otimes\cM^\el$. The matrix of $\nabla$ in the fixed basis of $\cM^\el$ is a global section of $\cEnd((\cA^{\rmod0})^d)$, hence a section of $\End(\cO(*0)^d)$. In other words, $\nabla$ defines a new meromorphic connection on the $\cO(*0)$-module $\cM^\el$, that we now denote $\cM$. By construction, $(\cL,\cL_\bbullet)$ and $(\cH^0\wt\DR\cM,\cH^0\DR_\leq\cM)$ are isomorphic, since they correspond to the same cocycle.

Let us now prove the theorem in the ramified case. We have seen (\cf Definition \ref{def:Stokeswithramif}) that a Stokes filtration on $\cL$ is a non-ramified Stokes filtration on $\wt\rho^{-1}\cL$ which is equivariant with respect to the automorphisms $\wt\sigma$, for a suitable $\rho:X'\to X$.

Similarly, let $\cM$ be a meromorphic connection and let $\cM^*=j_*j^{-1}\cM$, where $j$ denotes the inclusion $\CC\moins\{0\}\hto\CC$ and where we implicitly consider germs at $0$. Then giving $\cM$ as a subspace of $\cM^*$ stable by the connection is equivalent to giving $\cM'$ as a subspace of $\rho^*\cM^*$ compatible with the connection, and such that $\sigma^*\cM'=\cM'$ in $\rho^*\cM^*$, under the natural identification $\sigma^*\rho^*\cM^*=\rho^*\cM^*$.

The essential surjectivity in the ramified case follows then from the functoriality (applied to $\sigma^*$) of the Riemann-Hilbert functor.
\end{proof}

Let $\wh\cO$ be the ring of formal power series in the variable $x$. If $\cM$ is a germ of meromorphic connection, we set $\wh\cM=\wh\cO\otimes_{\cO}\cM$. Recall (\cf \eg \cite[Th\ptbl(1.5) p\ptbl45]{Malgrange91}) that $\wh\cM$ decomposes as \index{$MREG$@$\wh\cM_\reg$, $\wh\cM_\irr$}$\wh\cM_\reg\oplus\wh\cM_\irr$ where the first summand is regular and the second one is purely irregular. Moreover, $\wh\cM_\reg$ is the formalization of a unique regular meromorphic connection $\cM_\reg$ (which is in general not a summand of $\cM$).

\begin{proposition}\label{prop:gr0com}
The diagram of functors
\[
\xymatrix{
\cM\ar@{|->}[r]\ar@{|->}[d]_{(\cH^0\DR^{\rd0},\cH^0\DR^{\rmod0})}&\cM_\reg\ar@{|->}[d]^{\cH^0\DR^{\rmod0}}\\
(\cL_{<0},\cL_{\leq0})\ar@{|->}[r]^-{\gr_0}&\cG
}
\]
commutes.
\end{proposition}

We note that the right vertical functor (to the category of local systems on~$S^1$) is an equivalence. The point in the proposition is that the way from $\cM$ to $\cM_\reg$ by the lower path and the way by the upper horizontal arrow are not directly related, since $\cA^{\rmod0}/\cA^{\rd0}$ is much bigger than $\varpi^{-1}\CC\lpr x\rpr$ (the latter sheaf is equal to $\cA/\cA^{\rd0}$, with $\cA_{\wt X}$ introduced at the beginning of \S\ref{subsec:somebasic}).

\begin{proof}[Sketch of proof]
We will prove the result in the non-ramified case, the ramified case being treated as above. In such a case, $\wh\cM_\irr$ decomposes itself as $\bigoplus_{\varphi\neq0}\cE^\varphi\otimes\wh\cR_\varphi$ (Levelt-Turrittin decomposition) and we can easily define in a functorial way a lifting $\cM_\irr$ of $\wh\cM_\irr$ by lifting each $\wh\cR_\varphi$ occurring in the decomposition of $\wh\cM$. (One can also define such a lifting in the ramified case, but we will not use it.) The functor $\cM\mto\cM_\reg$ factorizes thus through the functor $\cM\mto\cM^\el=\cM_\irr\oplus\cM_\reg\mto\cM_\reg$. By definition, we have
\[
\gr_0(\cL_{<0},\cL_{\leq0})=\cL_{\leq0}/\cL_{<0}=\DR\big((\cA^{\rmod0}/\cA^{\rd0})\otimes_{\varpi^{-1}\cO}\varpi^{-1}\cM\big).
\]
The local isomorphisms
\[
(\cA^{\rmod0}/\cA^{\rd0})\otimes_{\varpi^{-1}\cO}\varpi^{-1}\cM\isom(\cA^{\rmod0}/\cA^{\rd0})\otimes_{\varpi^{-1}\cO}\varpi^{-1}\cM^\el
\]
deduced from the Hukuhara-Turrittin theorem glue together as a global isomorphism since they correspond to the same $\gr_0\cL$, according to Theorem \ref{th:RHequivgermmero} (this would not be the case in general for the isomorphisms with coefficients in $\cA^{\rmod0}$ or $\cA^{\rd0}$, and this assertion is the main point of the proof). It is therefore enough to prove the assertion in the case where $\cM=\cM^\el$, so that $\cM=\cM_{\reg}\oplus\cM_{\irr}$.

If $\varphi\neq0$, then one checks that the natural morphism $\DR^{\rd0}(\cE^\varphi\otimes\cR)\hto\DR^{\rmod0}(\cE^\varphi\otimes\nobreak\cR)$ is a quasi-isomorphism (it is enough to check that it induces an isomorphism on~$\cH^0$, since both $\cH^1$ are zero: this is Theorem \ref{th:H1nul}). If $\varphi=0$, then $\DR^{\rd0}\cM_{\reg}=0$, and $\DR^{\rmod0}\cM_{\reg}$ is the local system defined by $\cM_{\reg}$, hence the assertion.
\end{proof}

\Subsubsection*{The Riemann-Hilbert correspondence for germs of holonomic $\cD$-modules}\index{Riemann-Hilbert correspondence!for germs of holonomic $\cD$-modules}
We remark (\cf \cite[p\ptbl59]{Malgrange91}) that the category of germs of holonomic $\cD$-modules is equivalent to the category of triples $(\cM_*,\wh\cM_\reg,\wh\mu)$, where
\begin{itemize}
\item
$\cM_*$ is a locally free $\cO(*0)$-module with connection,
\item
$\wh\cM_\reg$ is a regular holonomic $\wh\cD$-module,
\item
$\wh\mu$ is a $\wh\cD$-linear isomorphism $(\wh\cM_*)_\reg\simeq\wh\cM_\reg(*0)$.
\end{itemize}

Indeed, to a holonomic $\cD$-module $\cM$, we associate $\cM_*=\cO(*0)\otimes_{\cO}\cM=\cM(*0)$, $\wh\cM=\wh\cO\otimes_{\cO}\cM$ which has a unique decomposition $\wh\cM=\wh\cM_\reg\oplus\wh\cM_\irr$ into regular part and irregular part. Lastly, $\wh\mu$ is uniquely determined by the canonical isomorphism $(\wh\cM)(*0)\simeq\wh{\cM(*0)}$, by restricting to the regular summand of both terms.

On the other hand, one recovers the holonomic $\cD$-module $\cM$ from $(\cM_*,\wh\cM_\reg,\wh\mu)$ as the kernel of the morphism $\cM_*\oplus\wh\cM_\reg\to\wh\cM_\reg(*0)$, $m_*\oplus m_r\mto\wh\mu(\wh{m_*}_r)-\wh{\mathrm{loc}}(\wh m_r)$, where $\wh{\mathrm{loc}}$ is the natural localization morphism $\wh\cM_\reg\to\wh\cM_\reg(*0)$.

The Riemann-Hilbert functor for germs of holonomic $\cD$-modules is now defined as follows: to the holonomic $\cD$-module $\cM=(\cM_*,\wh\cM_\reg,\wh\mu)$ is associated the triple $((\cL,\cL_\bbullet),\wh\cF,\nu)$ consisting of a Stokes-filtered local system $(\cL,\cL_\bbullet)$, a germ of perverse sheaf $\wh\cF$ and an isomorphism $\wh\nu:\gr_0\cL\simeq\psi_x\wh\cF$, where $\psi_x\wh\cF$ is the local system on~$S^1$ defined as in \S\ref{subsec:Stpervhat}. The previous results make it clear that this functor is an equivalence.

\subsection{The Riemann-Hilbert correspondence in the global case}\label{subsec:RHfunctorone}\index{Riemann-Hilbert correspondence!for holonomic $\cD$-modules}
We will say that a coherent $\cD_X(*D)$-module is holonomic if it is holonomic as a $\cD_X$-module. Recall that, in some neighbourhood of $D$ in $X$, holonomic $\cD_X(*D)$-modules are locally free $\cO_X(*D)$-modules of finite rank with connection. On the other hand, any holonomic $\cD_X$-module~$\cM$ gives rise to a holonomic $\cD_X(*D)$-module $\cM(*D)\defin\cD_X(*D)\otimes_{\cD_X}\cM$. We will set $\cM_{\wt X}\defin\cA_{\wt X}\otimes_{\varpi^{-1}\cO_{X}}\varpi^{-1}\cM$ and $\cM_{\ccIet}=\mu^{-1}\cM_{\wt X}$.

Let $\cM^\cbbullet$ be an object of \index{$Dbdx$@$D^\rb(\cD_X)$, $D^\rb_\rh(\cD_X)$}$D^\rb(\cD_X)$. The de~Rham functor $\pDR:D^\rb(\cD_X)\to D^\rb(\CC_X)$ is defined by \index{$DRAP$@$\pDR\cM$}$\pDR(\cM^\cbbullet)=\omega_X\otimes_{\cD_X}^{\bL}\cM^\cbbullet$, where $\omega_X$ is $\Omega^1_X$ regarded as a right $\cD_X$-module. It is usually written, using the de~Rham resolution of $\omega_X$ as a right $\cD_X$-module, as $\Omega^{1+\cbbullet}_X\otimes_{\cO_X}\cM^\cbbullet$, with a suitable differential (there is a shift by one with respect to the notation $\DR$ used in the previous paragraph). We define similarly functors\index{de~Rham complex!moderate --}\index{de~Rham complex!rapid-decay --}
\begin{itemize}
\item
$\pDR^\modD_{\wt X}$, $\pDR^\rdD_{\wt X}$ and $\pDR^\grD_{\wt X}$ from $D^\rb(\cD_{\wt X})$ to $D^\rb(\CC_{\wt X})$,
\item
$\pDR^\modD_{\ccIet}$, $\pDR^\rdD_{\ccIet}$ and $\pDR^\grD_{\ccIet}$ from $D^\rb(\mu^{-1}\cD_{\wt X})$ to $D^\rb(\CC_{\ccIet,\leq})$.
\end{itemize}
These are obtained by replacing $\omega_X$ with $\cA^\modD_{\wt X}\otimes_{\varpi^{-1}\cO_X}\varpi^{-1}\omega_X$, etc.\ That $\pDR^\modD_{\ccIet}$, etc.\ take value in $D^\rb(\CC_{\ccIet,\leq})$ and not only in $D^\rb(\CC_{\ccIet})$ follows from Lemma \ref{lem:AmodpreIfilt}.

\begin{theoreme}\label{th:DRIet}
If $\cM^\cbbullet$ is an object of \index{$Dbdx$@$D^\rb(\cD_X)$, $D^\rb_\rh(\cD_X)$}$D^\rb_{\rh}(\cD_X(*D))$ (\ie has holonomic cohomology), then
\begin{itemize}
\item
$\pDR^\modD_{\ccIet}\cM^\cbbullet$ is an object of $D^\rb_\CSt(\CC_{\ccIet,\leq})$,
\item
$\pDR^\rdD_{\ccIet}\cM^\cbbullet$ is an object of $D^\rb_\coCSt(\CC_{\ccIet,\leq})$,
\item
$\pDR^\grD_{\ccIet}\cM^\cbbullet$ is an object of $D^\rb_\grCc(\CC_{\ccIet,\leq})$,
\end{itemize}
and the functor (called the \index{Riemann-Hilbert functor (RH)}Riemann-Hilbert functor) of triangulated categories
\[
\pDR_{\ccIet}\defin \{\pDR^\rdD_{\ccIet}\to\pDR^\modD_{\ccIet}\to\pDR^\grD_{\ccIet}\To{+1}\}
\]
from $D^\rb_{\rh}(\cD_X(*D))$ to $\St(\CC_{\ccIet,\leq})$ is $t$-exact when $D^\rb_{\rh}(\cD_X(*D))$ is equipped with its natural $t$-structure.
\end{theoreme}

\begin{proof}
One first easily reduces to the case when $\cM$ is a holonomic $\cD_X(*D)$-module. The question is local, and we can assume that~$X$ is a disc with coordinate $x$. We will denote by $\cO$, $\cD$, etc.\ the germs at the origin $0$ (here equal to~$D$) of the corresponding sheaves. The result follows then from Corollary \ref{cor:RHStokes} and Proposition \ref{prop:RHStokes}.
\end{proof}

\subsubsection*{The Riemann-Hilbert correspondence for meromorphic connections}\index{Riemann-Hilbert correspondence!for meromorphic connections}

We restrict here the setting to meromorphic connections with poles on $D$ at most.

\begin{proposition}\label{prop:RHmero}
The functor $\cH^{-1}\pDR_{\ccIet}$ of Theorem \ref{th:DRIet} induces an equivalence between the category of meromorphic connections on~$X$ with poles on~$D$ at most and the category $\StPerv_{D,\textup{smooth}}(\CC_{\ccIet,\leq})$ of Stokes-perverse sheaves which are locally constant on $X^*$.
\end{proposition}

Recall (\cf Remark \ref{rem:StPerv}) that we also consider such objects as $\St$-$\CC$-constructible sheaves shifted by one, so we will mainly work with sheaves and the functor $\DR$ instead of perverse sheaves and the functor $\pDR$. The proof being completely parallel to that of Theorem \ref{th:RHequivgermmero}, we only indicate it.

\begin{lemme}[Compatibility with $\cHom$]\label{lem:comphom}
If $\cM,\cM'$ are meromorphic connections on~$X$ with poles on~$D$ at most, then
\[
\cH^0\DR_{\ccIet}\cHom_{\cO_X}(\cM,\cM')\simeq\cHom(\cH^0\DR_{\ccIet}\cM,\cH^0\DR_{\ccIet}\cM').
\]
\end{lemme}

\begin{proof}
It is similar to the first part of the proof of Theorem \ref{th:RHequivgermmero}.
\end{proof}

\begin{proof}[\proofname\ of Proposition \ref{prop:RHmero}: full faithfulness]
It is analogous to the corresponding proof in Theorem \ref{th:RHequivgermmero}.
\end{proof}

\begin{proof}[\proofname\ of Proposition \ref{prop:RHmero}: essential surjectivity]
Because both categories consist of objects which can be reconstructed by gluing local pieces, and because the full faithfulness proved above, it is enough to prove the local version of the essential surjectivity. This is obtained by the similar statement in Theorem \ref{th:RHequivgermmero}.
\end{proof}

\subsubsection*{The Riemann-Hilbert correspondence for holonomic $\cD_X$-modules}\label{subsec:RHholone}\index{Riemann-Hilbert correspondence!for holonomic $\cD$-modules}
For the sake of simplicity, we consider here the Riemann-Hilbert correspondence for holonomic $\cD_X$-modules, and not the general case of $D^\rb_{\textup{hol}}(\cD_X)$ of bounded complexes with holonomic cohomology. We first define the Riemann-Hilbert functor with values in $\StPerv_D(\CC_{\ccIet,\leq,\wh D})$ (\cf\S\ref{subsec:Stpervhat}). As for germs, we first remark (\cf \cite[p\ptbl59]{Malgrange91}) that the category of holonomic $\cD_X$-modules with singularity set contained in~$D$ is equivalent to the category of triples $(\cM_*,\cM_\reg,\wh\mu)$, where
\begin{itemize}
\item
$\cM_*$ is a locally free $\cO_X(*D)$-module with connection,
\item
$\wh\cM_\reg$ is a regular holonomic $\cD_{\wh D}$-module, where $\cO_{\wh D}$ is the formal completion of $\cO_X$ at~$D$ and $\cD_{\wh D}=\cO_{\wh D}\otimes_{\cO_X}\cD_X$,
\item
$\wh\mu$ is a $\cD_{\wh D}$-linear isomorphism $(\cO_{\wh D}\otimes_{\cO_X}\cM_*)_\reg\simeq\cO_{\wh D}\otimes_{\cO_X}\cM_\reg(*D)$.
\end{itemize}
Indeed, to a holonomic $\cD_X$-module $\cM$, we associate $\cM_*=\cO_X(*D)\otimes_{\cO_X}\cM=\cM(*D)$, $\wh\cM=\cO_{\wh D}\otimes_{\cO_X}\cM$ which has a unique decomposition $\wh\cM=\wh\cM_\reg\oplus\wh\cM_\irr$ into regular part and irregular part. Lastly, $\wh\mu$ is uniquely determined by the canonical isomorphism $(\wh\cM)(*D)\simeq\wh{\cM(*D)}$, by restricting to the regular summand of both terms.

On the other hand, one recovers $\cM$ from data $(\cM_*,\wh\cM_\reg,\wh\mu)$ as the kernel of the morphism $\cM_*\oplus\nobreak\wh\cM_\reg\to\wh\cM_\reg(*D)$, $m_*\oplus m_r\mto\wh\mu((\wh m_*)_r)-\wh{\mathrm{loc}}(\wh m_r)$, where $\wh{\mathrm{loc}}$ is the natural localization morphism $\wh\cM_\reg\to\wh\cM_\reg(*D)$.

A statement similar to that of Proposition \ref{prop:gr0com} holds in this global setting, since it is essentially local at $D$.

The Riemann-Hilbert functor for holonomic $\cD$-modules is now defined as follows: to the holonomic $\cD_X$\nobreakdash-module~$\cM=(\cM_*,\wh\cM_\reg,\wh\mu)$ is associated the triple $(\pDR_{\ccIet}\cM_*,\wh\cF,\wh\nu)$, where $\pDR_{\ccIet}\cM_*$ is the Stokes-perverse sheaf attached to the meromorphic connection $\cM_*\defin\cO_X(*D)\otimes_{\cO_X}\cM$, $\wh\cF=\pDR(\wh\cM_\reg)$, and $\wh\nu$ is the isomorphism defined from $\wh\mu$ and Proposition \ref{prop:gr0com}. At this point, we have proved that this triple is a Stokes-perverse sheaf on $\underline{\wt X}$ (\cf \S\ref{subsec:Stpervhat}) and have given the tools for the proof of

\begin{theoreme}\label{th:RHStperv}
The Riemann-Hilbert functor is an equivalence between the category of holonomic $\cD_X$-modules and the category of Stokes-perverse sheaves on $\underline{\wt X}$. Under this equivalence, the de~Rham functor $\pDR$ corresponds to the ``associated perverse sheaf'' functor $\cP$ (\cf Proposition \ref{prop:StDPervhatimdir}).\qed
\end{theoreme}

\subsection{Compatibility with duality for meromorphic connections}\index{duality}\index{Riemann-Hilbert correspondence!and duality}
We will check the compatibility of the Riemann-Hilbert functor of Definition \ref{def:RHfunctorgerms} with duality, of meromorphic connections on the one hand and of Stokes-filtered local systems on the other hand (\cf Lemma \ref{lem:dualS1}). For the sake of simplicity, we do not go further in checking the analogous compatibility at the level of holonomic $\cD$-modules and Stokes-perverse sheaves. General results in this directions are obtained in \cite{Mochizuki10}.

Let $\cM$ be a germ of meromorphic connection, as in \S\ref{subsec:RHfunctorone} and let $\cM^\vee$ be the dual connection. Let $(\cL,\cL_\bbullet)$ be the Stokes-filtered local system associated to $\cM$ by the Riemann-Hilbert functor $\RH$ of Definition \ref{def:RHfunctorgerms}. On the other hand, let $(\cL,\cL_\bbullet)^\vee$ be the dual Stokes-filtered local system, with underlying local system $\cL^\vee$. We have a natural and functorial identification
\[
\cH^0\wt\DR(\cM^\vee)=\cL^\vee.
\]

\begin{proposition}\label{prop:RHStokesdual}
Under this identification, the Stokes filtration $\cH^0\DR_\leq(\cM^\vee)$ is equal to the dual Stokes filtration $(\cL_\leq)^\vee$.
\end{proposition}

\begin{proof}
The question is local, as it amounts to identifying two Stokes filtrations of the same local system $\cL^\vee$, hence it is enough to consider the case where \hbox{$\cM=\cE^\varphi\otimes\cR$} (we will neglect to check what happens with ramification). We will work on a neighbourhood $X$ of the origin, and its real blow-up space $\wt X$. We will also restrict to checking the compatibility for $\DR_{\leq0},\DR_{<0}$, since the compatibility for $\DR_{\leq\psi},\DR_{<\psi}$ (any $\psi$) directly follows, according to Corollary \ref{cor:RHStokes} and its proof.

The result is clear if $\varphi=0$. If $\varphi\neq0$, we denote by $I_\varphi$ the open set of~$S^1$ where~$e^{-\varphi}$ has moderate growth. On the one hand, $\cH^0\DR^{\rmod0}\big((\cE^\varphi\otimes\cR)^\vee\big)$ is equal to~$\cL^\vee$ on the open set $I_{-\varphi}$ and is zero on the complementary closed subset $\ov{I_\varphi}$ of~$S^1$. On the other hand, $\cL_{<0}=\cL_{\leq0}$ is equal to $\cL$ on $I_\varphi$ and zero on $\ov{I_{-\varphi}}$. Then, $(\cL^\vee)_{\leq0}$, as defined in Proposition \ref{prop:operationsSt} (or Lemma \ref{lem:dualS1}), satisfies the same property as $\cH^0\DR^{\rmod0}\big((\cE^\varphi\otimes\cR)^\vee\big)$ does, and both subsheaves of $\cL^\vee$ are equal.
\end{proof}

\chapterspace{-2}
\chapter[Applications to holonomic distributions]{Applications of the Riemann-Hilbert correspondence to holonomic distributions}\label{chap:holdist}

\begin{sommaire}
To any holonomic $\cD$-module on a Riemann surface $X$ is associated its Hermitian dual, according to M\ptbl Kashiwara. We give a proof that the Hermitian dual is also holonomic. As an application, we make explicit the local expression of a holonomic distribution, that is, a distribution on $X$ (in Schwartz' sense) which is solution to a nonzero holomorphic differential equation on $X$. The conclusion is that working with $C^\infty$ objects hides the Stokes phenomenon.
\end{sommaire}

\subsection{Introduction}

Let $X$ be a Riemann surface. For each open set $U$ of $X$, the space $\Db(U)$ of distributions on $U$ (that is, continuous linear forms on the space of $C^\infty$ functions on $U$ with compact support, equipped with its usual topology) is a left module over the ring $\cD_X(U)$ of holomorphic differential operators on $U$, since any distribution can be differentiated with the usual rules of differentiation. These spaces $\Db(U)$ form a sheaf $\Db_X$ when $U$ varies. Note that $\Db(U)$ is also acted on by anti-holomorphic differential operators (that is, holomorphic differential operators on the conjugate manifold $\ov X$ with structure sheaf $\cO_{\ov X}=\ov{\cO_X}$), and $\Db_X$ is also a left $\cD_{\ov X}$-module. Moreover, the actions of $\cD_X$ and of $\cD_{\ov X}$ commute.

Given a $\cD_X$-module $\cM$, its Hermitian dual is the $\cD_{\ov X}$-module $\cHom_{\cD_X}(\cM,\Db_X)$, where the $\cD_{\ov X}$-action is induced by that on $\Db_X$. The main theorem in this \chaptersname is the result that, if $\cM$ is holonomic, so is its Hermitian dual, and all $\cExt^i_{\cD_X}(\cM,\Db_X)$ vanish identically. From the commutative algebra point of view, this duality is easier to handle than the duality of $\cD$-modules, since it is a $\cHom$ and not a $\cExt^1$.

Owing to the fact that a holonomic $\cD_X$-module is locally generated by one section, this result allows one to give an explicit expansion of holonomic distributions on $X$, that is, distributions which are solution of a nonzero holomorphic differential equation.

The \chaptersname shows an interplay between complex analysis (through the Stokes phenomenon) and real analysis on a Riemann surface. The Stokes phenomenon is hidden, in the $C^\infty$ world, behind the relation between the holomorphic and anti-holomorphic part of a distribution (when this is meaningful).

The notion of Hermitian duality has been introduced by M\ptbl Kashiwara \cite{Kashiwara86}, who analyzed it mainly in the case of regular singularities, and gave also various applications to distributions (see also \cite{Bjork93}). We will revisit the proof of the main result in the case of irregular singularities, as given in \cite[\S II.3]{Bibi97}. The expansion of holonomic distributions given in Theorem \ref{th:distholmod} follows that of \cite{Bibi06a}. We will revisit Hermitian duality in higher dimension in \Chaptersname\ref{chap:RHgoodnc}.\enlargethispage{\baselineskip}%
 
\subsection{The Riemann-Hilbert correspondence for meromorphic connections of Hukuhara-Turrittin type}\index{Riemann-Hilbert correspondence!for meromorphic connections of Hukuhara-Turrittin type}

Let us go back to the local setting of \S\ref{subsec:RHgermsmero}. Let~$\wt\cM$ be a locally free $\cA^{\rmod0}$-module of rank $d<\infty$ on~$S^1$ with a connection $\wt\nabla$. We say that $\wt\cM$ is \index{regular meromorphic connection}\emph{regular} if, locally on~$S^1$, it admits a $\cA^{\rmod0}$-basis with respect to which the matrix of the connection takes the form $Cdx/x$, where $C$ is constant.

Let us start by noting that, locally on~$S^1$, one can fix a choice of the argument of~$x$, so the matrix $x^C$ is well-defined as an element of $\GL_d(\cA^{\rmod0})$, and therefore a regular $\cA^{\rmod0}$-connection of rank $d$ is locally isomorphic to $(\cA^{\rmod0})^d$ with its standard connection. As already mentioned in the proof of Theorem \ref{th:H1nul} (\cf \cite[App\ptbl1]{Malgrange91}), the derivation $\partial_x:\cA^{\rmod0}\to\cA^{\rmod0}$ is onto, and it is clear that its kernel is the constant sheaf. As a consequence, for a regular $\cA^{\rmod0}$-connection $\wt\cM$, the sheaf $\ker\wt\nabla$ is a locally constant sheaf of rank~$d$ on~$S^1$, and the natural morphism $\ker\wt\nabla\otimes_\CC\cA^{\rmod0}\to\wt\cM$ is an isomorphism of $\cA^{\rmod0}$-connections.

We conclude that any regular $\cA^{\rmod0}$-connection $\wt\cM$ takes the form $\wt\cM=\cA^{\rmod0}\otimes_{\varpi^{-1}\cO}\varpi^{-1}\cM$ for some regular $\cO(*0)$-connection $\cM$, which is nothing but the regular meromorphic connection whose associated local system on the punctured disc is the local system isomorphic to $\ker\wt\nabla$. According to the projection formula, we also have $\cM=\varpi_*\wt\cM$.

We now extend this result to more general $\cA^{\rmod0}$-connections. As before, the sheaf $\ccI$ on~$S^1$ is that introduced in \S\ref{subsec:Stokeswithramif}.

\begin{definitio}\label{def:HTtype}
Let $\wt\cM$ be a locally free $\cA^{\rmod0}$-module of finite rank on~$S^1$. We say that $\wt\cM$ is of \emphb{Hukuhara-Turrittin type} if, for any $\theta\in S^1$, there exists a finite set $\Phi_\theta\subset\ccI_\theta$ such that, in some neighbourhood of $\theta$, $\wt\cM$ is isomorphic to the direct sum indexed by $\varphi\in\Phi_\theta$ of $\cA^{\rmod0}$-connections $(\cA^{\rmod0},d+d\varphi)^{d_\varphi}$, for some $d_\varphi\in\NN^*$.
\end{definitio}

As indicated above, the regular part of each summand can be reduced to the trivial connection $d$. On the other hand, to any $\cA^{\rmod0}$-connection one can associate de~Rham complexes $\DR_{\leq\psi}\wt\cM$ and $\DR_{<\psi}\wt\cM$ for any local section $\psi$ of $\ccI$, by the same formulas as in Corollary \ref{cor:RHStokes}. In particular, $\cH^0\DR_{\leq\psi}\wt\cM$ is a subsheaf of the locally constant sheaf $\cH^0\wt\DR\wt\cM$ (horizontal sections with arbitrary growth along~$S^1$) and defines a pre-$\ccI$-filtration of this sheaf.

\begin{proposition}\label{prop:RHAmod}
Assume that $\wt\cM$ is of Hukuhara-Turrittin type. Then $\cL_{\leq}\defin\cH^0\DR_{\leq}\wt\cM$ is a $\ccI$-filtration of $\cL\defin\cH^0\wt\DR\wt\cM$, for which $\cL_<$ coincides with $\cH^0\DR_<\wt\cM$. Moreover, the correspondence $\wt\cM\mto(\cL,\cL_\bbullet)$ is an equivalence between the category of $\cA^{\rmod0}$-connections of Hukuhara-Turrittin type and the category of Stokes-filtered local systems.
\end{proposition}

\begin{proof}
We first note that, near $\theta$, we have at most one non-zero morphism (up to a scalar constant) from $\cA^{\rmod0}\otimes\cE^\varphi=(\cA^{\rmod0},d+d\varphi)$ to $\cA^{\rmod0}\otimes\cE^\psi=(\cA^{\rmod0},d+d\psi)$ as $\cA^{\rmod0}$-connections, which is obtained by sending $1$ to $e^{\varphi-\psi}$, and this morphism exists if and only if $\varphi\leqtheta\psi$. In particular, both $\cA^{\rmod0}$-connections are isomorphic near $\theta$ if and only if $\varphi=\psi$ near $\theta$, hence everywhere. As a consequence, the set $\Phi_\theta$ of Definition \ref{def:HTtype} is locally constant with respect to $\theta$, and we may repeat the proof of Theorem \ref{th:H1nul} and Corollary \ref{cor:RHStokes} for~$\wt\cM$. Similarly, the proof of the full faithfulness in Theorem \ref{th:RHequivgermmero} can be repeated for~$\wt\cM$, since the main argument is local.

The essential surjectivity follows from the similar statement in Theorem \ref{th:RHequivgermmero}.
\end{proof}

\begin{corollaire}\label{cor:RHAmod}
Let $\wt\cM$ be of Hukuhara-Turrittin type. Then $\cM\defin\varpi_*\wt\cM$ is a $\cO(*0)$-connection and $\wt\cM=\cA^{\rmod0}\otimes_{\varpi^{-1}\cO}\varpi^{-1}\cM$.
\end{corollaire}

\begin{proof}
Let $(\cL,\cL_\bbullet)$ be the Stokes-filtered local system associated with $\wt\cM$ by the previous proposition and let $\cN$ be a $\cO(*0)$-connection having $(\cL,\cL_\bbullet)$ as associated Stokes-filtered local system, by Theorem \ref{th:RHequivgermmero}. Let us set $\wt\cN=\cA^{\rmod0}\otimes_{\varpi^{-1}\cO}\varpi^{-1}\cN$. Since $\cN$ is $\cO(*0)$-free, we have $\varpi_*\wt\cN=(\varpi_*\cA^{\rmod0})\otimes_{\cO}\cN=\cN$. The identity morphism $(\cL,\cL_\bbullet)\to(\cL,\cL_\bbullet)$ lifts in a unique way as a morphism $\wt\cM\to\wt\cN$, by the full faithfulness in Proposition \ref{prop:RHAmod}, and the same argument shows that it is an isomorphism. Therefore, $\varpi_*\wt\cM\simeq\cN$.
\end{proof}

\subsection{The Hermitian dual of a holonomic $\cD_X$-module}
We now assume that $X$ is a Riemann surface, and we denote by $\ov X$ the complex conjugate surface (equipped with the structure sheaf $\cO_{\ov X}\defin\ov{\cO_X}$). The sheaf of \index{distribution}distributions \index{$DB$@$\Db_X$}$\Db_X$ on the underlying $C^\infty$ manifold is at the same time a left $\cD_X$ and $\cD_{\ov X}$-module, and both actions commute. If $\cM$ is a left $\cD_X$-module, then \index{$CX$@$C_X\cM$ (Hermitian dual)}$C_X\cM\defin\cHom_{\cD_X}(\cM,\Db_X)$ is equipped with a natural structure of left $\cD_{\ov X}$-module, induced by that on $\Db_X$. This object, called the \emphb{Hermitian dual} of $\cM$, has been introduced by M\ptbl Kashiwara in \cite{Kashiwara86}.

\begin{theoreme}[\cite{Kashiwara86}, \cite{Bibi97}]\label{th:Hermdualone}
Assume that $\cM$ is a holonomic $\cD_X$-module. Then $\cExt^k_{\cD_X}(\cM,\Db_X)=\nobreak0$ for $k>0$ and $C_X\cM$ is $\cD_{\ov X}$-holonomic.
\end{theoreme}

We refer to \cite{Kashiwara86}, \cite{Bjork93} and \cite{Bibi97} for various applications of this result. For instance, the vanishing of $\cExt^1$ implies the solvability in $\Db(\Omega)$ for any linear differential operator with meromorphic coefficients in an open set $\Omega\subset\CC$.

\begin{proof}
Let us recall various reductions used in \cite{Kashiwara86} and \cite{Bibi97}. If $D$ is the singular divisor of $\cM$, it is enough to prove the result for $\cM(*D)$, and one can replace $\Db_X$ with the sheaf \index{$DBMOD$@$\Db_X^\modD$}$\Db_X^\modD$ of \index{distribution!with moderate growth}distributions having moderate growth at $D$ (whose sections on an open set consist of continuous linear forms on the space of $C^\infty$ functions with compact support on this subset and having rapid decay at the points of $D$ in this open set). In such a case, $\cM(*D)$ is a meromorphic bundle with connection, and the statement is that $C_X^\modD(\cM(*D))\defin\cHom_{\cD_X(*D)}(\cM(*D),\Db_X^\modD)$ is an anti-meromorphic bundle with connection, and the corresponding $\cExt^k$ vanish for $k>0$.

One can then work locally near each point of $D$, and prove a similar result on the real blow-up space $\wt X(D)$, by replacing $\cM(*D)$ with $\wt\cM=\cA^\modD\otimes\cM(*D)$ and $\Db_X^\modD$ with the similar sheaf $\Db_{\wt X}^\modD$ on $\wt X$. One can then reduce to the case where $\wt\cM(*D)$ has no ramification. It is then proved in \cite[Prop\ptbl II.3.2.6]{Bibi97} that $\cExt^k_{\cD_{\wt X}}(\wt\cM,\Db_{\wt X}^\modD)=\nobreak0$ for $k>0$ and that $C_{\wt X}^\modD(\wt\cM)$ is of Hukuhara-Turrittin type (see \cite[p\ptbl69]{Bibi97}). One concludes in \loccit by analyzing the Stokes matrices, but this can be avoided by using Corollary \ref{cor:RHAmod} above: we directly conclude that Prop\ptbl II.3.2.6(2) of \loccit is fulfilled.
\end{proof}

\subsection{Asymptotic expansions of holonomic distributions}
We will apply Theorem \ref{th:Hermdualone} to give the general form of a germ of holonomic distribution of one complex variable. We follow \cite{Bibi06a}.

Since the results will be of a local nature, we will denote by $X$ a disc centered at~$0$ in~$\CC$, with coordinate $x$. We denote by $\Db_X^{\rmod0}$ the sheaf on $X$ of distributions on $X\moins\{0\}$ with moderate growth at the origin (\cf above) and by $\Db^{\rmod0}$ its germ at the origin. In particular $\Db^{\rmod0}$ is a left $\cD$ and $\ov\cD$-module. Let $u\in\Db^{\rmod0}$ be holonomic, that is, the $\cD$-module $\cD\cdot u$ generated by $u$ in $\Db^{\rmod0}$ is holonomic. In other words, $u$ is a solution of a non-zero linear differential operator with holomorphic coefficients. Notice that Theorem \ref{th:Hermdualone} implies that $\ov\cD\cdot u$ is also holonomic as a $\ov\cD$-module. Indeed, the inclusion $\cD\cdot u\hto\Db^{\rmod0}$ is a germ of section $\sigma$ of $C_X^{\rmod0}(\cD\cdot u)$ which satisfies $\sigma(u)=u$. For a anti-holomorphic differential operator $\ov P$, we have $(\ov P\sigma)(u)=\ov P u$. There exists $\ov P\neq0$ with $\ov P\sigma=0$, and thus $\ov Pu=0$.

Let $\rho:X'\to X$, $x'\mto x^{\prime d}=x$, be a ramified covering of degree $d\in\NN^*$. Then the pull-back by $\rho$ of a moderate distribution at $0$ is well-defined as a moderate distribution at $0$ on $X'$. If $u$ is holonomic, so is $\rho^*u$ (if $u$ is annihilated by $P\in\cD$, then $\rho^*P$ is well-defined in $\cD_{X',0}(*0)$ and annihilates $\rho^*u$).

\begin{theoreme}\label{th:distholmod}
Let $u$ be a germ at $0$ of a \index{distribution!moderate holonomic --}\index{moderate holonomic distribution}moderate holonomic distribution on~$X$. Then there exist:
\begin{itemize}
\item
an integer $d$, giving rise to a ramified covering $\rho:X'\to X$,
\item
a finite set $\Phi\subset x^{\prime-1}\CC[x^{\prime-1}]$,
\item
for all $\varphi\in\Phi$, a finite set $B_\varphi\in\CC$ and an integer $L_\varphi\in\NN$,
\item
for all $\varphi\in\Phi$, $\beta\in B_\varphi$ and $\ell=0,\dots,L_\varphi$, a function $f_{\varphi,\beta,\ell}\in C^\infty(X')$
\end{itemize}
such that, in $\Dbmo(U')$ and in particular in $C^\infty(U^{\prime*})$ (if $U'$ is a sufficiently small neighbourhood of $0$ in $X'$),
\bgroup\numstareq
\begin{equation}\label{eq:distholmod}
\rho^*u=\sum_{\varphi\in\Phi}\sum_{\beta\in B_\varphi}\sum_{\ell=0}^{L_\varphi}f_{\varphi,\beta,\ell}(x')e^{\varphi-\ov\varphi}|x'|^{2\beta}\Ly^\ell,\quad\text{with }\Ly\defin\module{\log|x'|^2}.
\end{equation}
\egroup
\end{theoreme}

\skpt
\begin{remarques}\label{rem:distholmod}
\begin{enumerate}
\item\label{rem:distholmod1}
Notice that, for $\varphi\in x^{\prime-1}\CC[x^{\prime-1}]$, the function $e^{\varphi-\ov\varphi}$ is a multiplier in $\Dbmo$ (since it is $C^\infty$ away from $0$, with moderate growth, as well as all its derivatives, at the origin). So are the functions $|x'|^{2\beta}$ and $\Ly^\ell$.
\item\label{rem:distholmod2}
One can be more precise concerning the sets $B_\varphi$, in order to ensure a minimality property. For $f\in C^\infty(X)$, the Taylor expansion of $f$ at $0$ expressed with $x,\ov x$ allows us to define a minimal set $E(f)\subset\NN^2$ such that $f=\sum_{(\nu',\nu'')\in E(f)}x^{\nu'}\ov x^{\nu''}\!g_{(\nu',\nu'')}$ with $g_{(\nu',\nu'')}\in C^\infty(X)$ and $g_{(\nu',\nu'')}(0)\neq0$. We will set $E(f)=\emptyset$ if $f$ has rapid decay at~$0$.

It is not a restriction to assume (and we will do so) that any two distinct elements of the index set $B_\varphi$ occurring in \eqref{eq:distholmod} do not differ by an integer, and that each $\beta\in B_\varphi$ is maximal, meaning that the set $\bigcup_\ell E(f_{\varphi,\beta,\ell})$ is contained in $\NN^2$ and in no subset $(m,m)+\NN^2$ with $m\in\NN^*$.
\end{enumerate}
\end{remarques}

We will first assume the existence of an expansion like \eqref{eq:distholmod} and we will make precise in Corollary \ref{cor:asympt} below the $\varphi,\beta$ such that $f_{\varphi,\beta,\ell}\neq0$ for some~$\ell$. We will allow to restrict the neighbourhood~$U$ or $U'$ as needed.

Let $\cM$ denote the $\cD[x^{-1}]$-module generated by $u$ in $\Dbmo$. Then $\cM$ is a free $\cO[x^{-1}]$-module of finite rank equipped with a connection, induced by the action $\partial_x$. Let $\rho:X'\to\nobreak X$ be a ramification such that $\cM'\defin\rho^+\cM$ is formally isomorphic to $\cM^{\prime\el}=\oplus_{\varphi\in\Phi}(\cE^\varphi\otimes\nobreak \cR_\varphi)$. The germ $\rho^+\cM$ is identified to the $\cD'[x^{\prime-1}]$-submodule of $\Dbmo_{X',0}$ generated the moderate distribution $v=\rho^*u$. In particular, $v$ is holonomic if $u$ is so.

\begin{definitio}\label{def:sansram}
We say that the holonomic moderate distribution $u$ has \index{distribution!moderate holonomic -- with no ramification}\emph{no ramification} if one can choose $\rho=\id$.
\end{definitio}

In the following, we will assume that $u$ has no ramification, since the statement of the theorem is given after some ramification. We will however denote by $v$ such a distribution, in order to avoid any confusion. For $\psi\in x^{-1}\CC[x^{-1}]$, let $\cM_\psi$ be the $\cD[x^{-1}]$-module generated by $e^{\ov\psi-\psi} v$ in $\Dbmo$. Note that $e^{\ov\psi-\psi} v$ is also holonomic and that this module is $\cO[x^{-1}]$-locally free of finite rank with a connection.

We can write a differential equation satisfied by $e^{\ov\psi-\psi}v$ in the following way: $[b(x\partial_x)-xQ(x,x\partial_x)]\cdot e^{\ov\psi-\psi} v=0$, for some nonzero polynomial $b(s)\in\CC[s]$. Up to multiplying this equation on the left by a polynomial in $x\partial_x$, we can find a differential equation satisfied by $e^{\ov\psi-\psi} v$ of the following form:
\begin{equation}\label{eq:Bernstein1}
\bigg[\prod_{k=0}^{k(1)}\prod_{\beta\in B'_\psi(v)}\big[-(x\partial_{x}-\beta-k)\big]^{L'_{\psi,\beta}}-xP_{\psi,1}\bigg]\cdot e^{\ov\psi-\psi} v=0,
\end{equation}
for some minimal finite set $B'_\psi(v)\subset\CC$ (in particular no two elements of $B'_\psi(v)$ differ by a nonzero integer), $P_{\psi,1}\in\CC\{x\}\langle x\partial_{x}\rangle$, and for each $\beta\in B'_\psi(v)$, $L'_{\psi,\beta}(v)\in\NN$. Iterating this relation gives a relation of the same kind, for any $j\in\NN^*$:
\begin{equation}\label{eq:Bernsteinj}
\bigg[\prod_{k=0}^{k(j)}\prod_{\beta\in B'_\psi(v)}\big[-(x\partial_{x}-\beta-k)\big]^{L'_{\psi,\beta}}-x^jP_{\psi,j}\bigg]\cdot e^{\ov\psi-\psi} v=0.
\end{equation}

One defines in a conjugate way the objects $L''_{\psi,\beta}$ and $B''_\psi(v)$. One then sets
\begin{equation}\label{eq:Bphiv}
B_\psi(v)=\Big[B'_\psi(v)\cap(B''_\psi(v)-\NN)\Big]\cup\Big[(B'_\psi(v)-\NN)\cap B''_\psi(v)\Big]\subset B'_\psi(v)\cup B''_\psi(v).
\end{equation}
In other words, $\beta\in B_\psi(v)$ iff $\beta\in B'_\psi(v)\cup B''_\psi(v)$ and both $(\beta+\NN)\cap B'_\psi(v)$ and $(\beta+\nobreak\NN)\cap B''_\psi(v)$ are non-empty. For all $\beta\in\CC$, let us set $L_{\psi,\beta}(v)=\min\{L'_{\psi,\beta}(v),L''_{\psi,\beta}(v)\}$.

We denote by $\Phi(v)$ the set of~$\varphi$ for which the component $\cE^\varphi\otimes \cR'_\varphi$ of the formal module associated with $\cD[x^{-1}]\cdot v$ is non-zero. It is also the set of $\varphi$ for which the component $\ov{\cE^{-\varphi}\otimes \cR''_\varphi}$ of the formal module associated with $\ov\cD[1/\ov x]\cdot v$ is non-zero, as a consequence of the proof of Theorem \ref{th:Hermdualone}.

\begin{corollaire}\label{cor:asympt}
Let $v$ be a holonomic moderate distribution with no ramification. Then $v$ has an expansion \eqref{eq:distholmod} in $\Dbmo$, with $\Phi=\Phi(v)$ and $B_\varphi= B_\varphi(v)$. Moreover, if $f_{\varphi,\beta,\ell}\neq0$ and if the point $(k',k'')\in\NN^2$ is in $E(f_{\varphi,\beta,\ell})$, then $\beta+k'\in B'_\varphi(v)+\NN$ and $\beta+k''\in B''_\varphi(v)+\NN$.
\end{corollaire}

\begin{proof}
We assume that the theorem is proved. We will use the \index{Mellin transformation}Mellin transformation to argue on each term of~\eqref{eq:distholmod}. Let $\chi$ be a $C^\infty$ function with compact support contained in an open set where $v$ is defined, and identically equal to $1$ near $0$. We will also denote by $\chi$ the differential form $\chi\tfrac{i}{2\pi}\,dx\wedge d\ov x$. On the other hand, let us choose a distribution $\wt v$ inducing~$v$ on~$X\moins\{0\}$ and let $p$ be its order on the support of $\chi$. Let us first consider the terms in \eqref{eq:distholmod} for which $\varphi=0$.

For all $k',k''\in\NN$, the function $s\mto\langle \wt v,|x|^{2s}x^{-k'}\ov x^{-k''}\chi\rangle$ is defined and holomorphic on the half plane $\{s\mid2\reel s>p+k'+k''\}$. For all $j\geq1$, let us denote by~$Q_j$ the operator such that \eqref{eq:Bernsteinj} (for $\psi=0$) reads $Q_j\cdot v=0$. Then $Q_j\cdot \wt v$ is supported at the origin. It will be convenient below to use the notation $\alpha$ for $-\beta-1$ and to set $A'_\varphi(v)=\{\alpha\mid\beta=-\alpha-1\in B'_\varphi(v)\}$. We deduce then that, on some half plane $\reel s\gg0$, the function
\[
\bigg[\prod_{k=0}^{k(j)}\prod_{\alpha\in A'_0(v)}(s-\alpha-k'+k)^{L'_{0,\alpha}}\bigg]\langle \wt v,|x|^{2s}x^{-k'}\ov x^{-k''}\chi\rangle
\]
coincides with a holomorphic function defined on $\{s\mid2\reel s>p+k'+k''-j\}$. Applying the same argument in a anti-holomorphic way, we obtain that, for each $k',k''\in\nobreak\NN$, the function $s\mto\langle \wt v,|x|^{2s}x^{-k'}\ov x^{-k''}\chi\rangle$ can be extended as a meromorphic function on~$\CC$ with poles contained in $(A'_0(v)+k'-\NN)\cap (A''_0(v)+k''-\NN)$, the order of the pole at $\alpha+\ZZ$ being bounded by $L_{0,\alpha}(v)$. Moreover, this function does not depend on the choice of the lifting~$\wt v$ of~$v$.

\begin{lemme}\label{lem:sanspole}
If $\varphi\neq0$, then for any function $g\in C^\infty(X)$, the Mellin transform of $g(x)e^{\varphi-\ov\varphi}|x|^{2\beta}\Lx^\ell$ is an entire function of the variable~$s$.
\end{lemme}

\begin{proof}
We will show that this Mellin transform is holomorphic on any half plane $\{s\mid\reel s>-q\}$ ($q\in\NN$). In order to do so, for $q$ fixed, we decompose $g$ as the sum of a polynomial in $x,\ov x$ and a remaining term which vanishes with high order at the origin so that the corresponding part of the Mellin transform is holomorphic for $\reel s>-q$. One is thus reduced to showing the lemma when $g$ is a monomial in~$x,\ov x$. One can then find differential equations for $g(x)e^{\varphi-\ov\varphi}|x|^{2\beta}\Lx^\ell$ of the kind \eqref{eq:Bernsteinj} ($j\in\NN^*$) with exponents $L'$ equal to zero, and anti-holomorphic analogues. Denoting by $p$ the order of a distribution lifting this moderate function, we find as above that the Mellin transform is holomorphic on a half plane $2\reel s>p+k'+k''-j$. As this holds for any~$j$, the Mellin transform, when $g$ is a monomial, is entire. The Mellin transform for general $g$ is thus holomorphic on any half plane $\{s\mid\reel s>-q\}$, thus is also entire.
\end{proof}

Let us now compute the Mellin transform of the expansion \eqref{eq:distholmod} for~$v$. Let us first consider the terms of the expansion \eqref{eq:distholmod} for $v$ for which $\varphi=0$. We will use the property that, for all $(\nu',\nu'')\in\ZZ^2$ \emph{not both negative} and any function $g\in C^\infty(X)$ such that $g(0)\neq0$, the poles of the meromorphic function $s\mto\langle g(x)|x|^{2\beta}\Lx^\ell,|x|^{2s}x^{\nu'}\ov x^{\nu''}\chi\rangle$ are contained in $\alpha-\NN$ (with $\alpha=-\beta-1$), and there is a pole at $\alpha$ if and only if $\nu'=0$ and $\nu''=0$, this pole having order $\ell+1$ exactly.

Let $\beta\in B_0$, where $B_0$ satisfies the minimality property of Remark \ref{rem:distholmod}\eqref{rem:distholmod2} and let $E_\beta\subset\NN^2$ be minimal such that $E_\beta+\NN^2=\bigcup_\ell (E(f_{0,\beta,\ell})+\NN^2)$. It follows from the previous remark and from the expansion \eqref{eq:distholmod} that, for each $(k',k'')\in E_\beta$, the function $s\mto\langle \wt v,|x|^{2s}x^{-k'}\ov x^{-k''}\chi\rangle$ has a non trivial pole at $\alpha$; from the first part of the proof we conclude that $\alpha-k'\in A'_0(v)-\NN$ and $\alpha-k''\in A''_0(v)-\NN$, that is, $\beta+k'\in B'_0(v)+\NN$ and $\beta+k''\in B''_0(v)+\NN$. By Remark \ref{rem:distholmod}\eqref{rem:distholmod2} for $\beta\in B_0$, there exists $(k',k'')\in E_\beta$ with $k'=0$ or $k''=0$. As a consequence, we have $\beta\in B_\varphi(v)$ (defined by \eqref{eq:Bphiv}), and the property of the corollary is fulfilled by the elements $(k',k'')$ of $E_\beta$. It is then trivially fulfilled by the elements $(k',k'')$ of all the subsets $E(f_{0,\beta,\ell})$.

The same result holds for the coefficients $f_{\varphi,\beta,\ell}$ by applying the previous argument to the moderate distribution $e^{\ov\varphi-\varphi}v$.
\end{proof}

\begin{proof}[\proofname\ of Theorem \ref{th:distholmod}]
We go back to the setting of the theorem. Recall that we set $\cM=\cD(*0)\cdot u\subset\Dbmo$ and $\Cmo\cM=\cHom_{\cD(*0)}(\cM,\Dbmo)$. There is thus a canonical $\cD\otimes_\CC\ov\cD$-linear pairing
\[
k:\cM\otimes_\CC\Cmo\cM\to\Dbmo,\qquad (m,\mu)\mto\mu(m).
\]
Since $\cM$ is generated by $u$ as a $\cD(*0)$-module, an element $\mu\in\Cmo\cM$ is determined by its value $\mu(u)\in\Dbmo$. Therefore, there exists a section $\bun_u$ of $\Cmo\cM$ such that $\bun_u(u)=u$.

We are thus reduced to proving the theorem in the case where
\[
k:\cM'\otimes_\CC\ov\cM{}''\to\Dbmo
\]
is a sesquilinear pairing between two free $\cO(*0)$-modules of finite rank with connection, $m',m''$ are local sections of $\cM',\cM''$, and $u=k(m',\ov{m''})$.

As we allow cyclic coverings, we can also reduce to the case where both $\cM'$ and~$\cM''$ have a formal decomposition with model $\cM^{\prime\el}$ and $\cM^{\prime\prime\el}$ (when $\cM''=\ov{\Cmo\cM'}$, if this assumption is fulfilled for $\cM'$, it is also fulfilled for $\cM''$, as follows from the proof of Theorem \ref{th:Hermdualone}). We will still denote by $x$ the variable after ramification, and by $X$ the corresponding disc.

We then introduce the real blow-up space $e:\wt X\to X$ at the origin, together with the sheaves $\cA_{\wt X}$ (\cf\S\ref{subsec:somebasic}) and $\Dbmo_{\wt X}$ (sheaf on $\wt X$ of moderate distributions along $e^{-1}(0)=S^1$). Lastly, we set $\wt \cM=\cA_{\wt X}\otimes_{e^{-1}\cO}e^{-1}\cM$ (for $\cM=\cM',\cM''$). This is a left $\cD_{\wt X}$-module which is $\cA_{\wt X}(*0)$-free.

The pairing $k$ can be uniquely extended as a $\cD_{\wt X}\otimes_\CC\cD_{\wt{\ov X}}$-linear pairing
\[
\wt k:\wt \cM'\otimes_\CC\ov{\wt \cM''}\to\Dbmo_{\wt X}
\]
(because $\cM',\cM''$ are $\cO(*0)$-free). One can then work locally on $\wt X$ with $\wt k$ and, according to the Hukuhara-Turrittin theorem, we can replace $\wt\cM'$ et $\wt\cM''$ by their respective elementary models $\bigoplus_\varphi (\cE^\varphi\otimes \cR'_\varphi)$ and $\bigoplus_\varphi (\cE^\varphi\otimes \cR''_\varphi)$.

\begin{lemme}
If $\varphi,\psi\in x^{-1}\CC[x^{-1}]$ are distinct, any sesquilinear pairing $\wt k_{\varphi,\psi}:(\cE^\varphi\otimes \cR'_\varphi)\otimes_\CC(\cE^{-\ov\psi}\otimes \ov{\cR''_{-\psi}})\to\Dbmo_{\wt X}$ takes values in the subsheaf of $C^\infty$ functions with rapid decay.
\end{lemme}

\begin{proof}
Since $e^{\psi-\ov\psi}$ is a multiplier on $\Dbmo_{\wt X}$, we can assume that $\psi=0$ for example. By induction on the rank of $\cR'_\varphi$ and $\cR''_0$, we can reduce to the rank-one case, and since the functions $x^{\alpha}$ and $\ov x^{\beta}$ are also multipliers, we can reduce to the case where $\cR'_\varphi$ and $\cR''_0$ are both equal to $\cO(*0)$. Now, denoting by $\ephi$ the generator $1$ of $\cE^\varphi$, the germ of moderate distribution $\wt u=\wt k(\ephi,1)$ on $\wt X$ satisfies $\ov\partial_{x}\wt u=0$ and $\partial_{x}\wt u=\varphi'(x)\wt u$. It follows that $\wt u_{|X^*}=e^{\varphi}$ with $0\neq\varphi\in x^{-1}\CC[x^{-1}]$. Therefore, $\wt u$ has moderate growth at $\theta\in S^1\ssi\wt u$ has rapid decay at $\theta$.
\end{proof}

In a similar way (using the Jordan normal form for the matrix of the connection on $\cR'_0,\cR''_0$), one checks that the diagonal terms $\wt k_{\varphi,\varphi}(\wt m',\ov{\wt m''})$ decompose as a sum, with coefficients in $\cC^\infty_{\wt X}$, of terms $e^{\varphi-\ov\varphi}x^{\beta'}\ov x^{\beta''}(\log x)^j(\log \ov x)^k$ ($\beta',\beta''\in\CC$, $j,k\in\NN$). One rewrites each term as a sum, with coefficients in $\cC^\infty_{\wt X}$, of terms $|x|^{2\beta}\Lx^\ell$ ($\beta\in\CC$, $\ell\in\NN$).

If $m',m''$ are local sections of $\cM',\cM''$, one uses a partition of the unity on~$\wt X$ to obtain for $e^*k(m',\ov{m''})$ an expansion like in \eqref{eq:distholmod}, with coefficients $\wt f_{\varphi,\beta,\ell}$ in $e_*\cC^\infty_{\wt X}$, up to adding a $C^\infty$ function with rapid decay along $e^{-1}(0)$; such a function can be incorporated in one of the coefficients $\wt f_{\varphi,\beta,\ell}$. We denote then by $\wt B_\varphi$ the set of indices $\beta$ corresponding to $\varphi$. Since $|x|$ is $C^\infty$ on $\wt X$, one can assume that two distinct elements of $\wt B_\varphi$ do not differ by half an integer.

It remains to check that this expansion can be written with coefficients $f_{\varphi,\beta,\ell}$ in $C^\infty(X)$. We will use the Mellin transformation, as in Corollary \ref{cor:asympt}, from which we only take the notation.

We will use polar coordinates $x=r e^{i\theta}$. A function $\wt f\in e_*\cC^\infty_{\wt X}$ has a Taylor expansion $\sum_{m\geq0}\wt f_m(\theta)r^m$, where $\wt f_m(\theta)$ is $C^\infty$ on~$S^1$ and expands as a Fourier series $\sum_n\wt f_{mn}e^{in\theta}$. Such a function can be written as $\sum_{k=-2k_0}^0g_k(x)|x|^k$ with $k_0\in\NN$ and $g_k\in C^\infty(X)$ if and only if
\begin{equation}\label{eq:cns}
\wt f_{m,n}\neq0\implique \frac{m\pm n}{2}\geq -k_0.
\end{equation}
Indeed, one direction is easy, and if \eqref{eq:cns} is fulfilled, let us set $k=\min(m-\nobreak n,m+\nobreak n)$. We have $k\in\ZZ$ and $k\geq-2k_0$. One writes $\wt f_{m,n}e^{in\theta}r^m=\wt f_{m,n}|x|^kx^{\ell'}\ov x^{\ell''}$ with $\ell',\ell''\!\in\!\nobreak\NN$. The part of the Fourier expansion of $\wt f$ corresponding to $k$ fixed gives, up to multiplying by $|x|^k$, according to Borel's lemma, a function $g_k(x)\in C^\infty(X)$. The difference $\wt f-\nobreak\sum_{k=-2k_0}^0g_k(x)|x|^k$ is~$C^\infty$ on~$\wt X$, with rapid decay along~$S^1$. It is thus $C^\infty$ on~$X$, with rapid decay at the origin. One can add it to $g_0$ in order to obtain the desired decomposition of~$\wt f$.

The condition \eqref{eq:cns} can be expressed in terms of Mellin transform. Indeed, one notices that, for all $k',k''\in\hNN$ such that $k'+k''\in\NN$, the Mellin transform $s\mto\langle \wt f,|x|^{2s}x^{-k'}\ov x^{-k''}\chi\rangle$, which is holomorphic for $\reel(s)\gg0$, extends as a meromorphic function on $\CC$ with simple poles at most, and these poles are contained in $\hZZ$. The condition \eqref{eq:cns} is equivalent to the property:\refstepcounter{equation}\label{enum:615}
\enum{.05}{.92}{\eqref{enum:615}}{there exists $k_0\in\NN$ such that, for all $k',k''\in\hNN$ with $k'+k''\in\NN$, the poles of $s\mto\langle \wt f,|x|^{2s}x^{-k'}\ov x^{-k''}\chi\rangle$ are contained in the intersection of the sets $k_0-1+k'-\hNN^*$ and $k_0-1+k''-\hNN^*$.}

Arguing by decreasing induction on $\ell$, that is also on the maximal order of the poles, we conclude that a function $\wt f=\sum_{\ell=0}^L\wt f_\ell\Lx^\ell$ with coefficients in $C^\infty(\wt X)$ can be rewritten as $\sum_{-2k_0\leq k\leq0}\sum_{\ell=0}^Lg_{k,\ell}(x)|x|^k\Lx^\ell$ with $g_{k,\ell}\in C^\infty(X)$ if and only if \eqref{enum:615} is fulfilled by $s\mto\langle \wt f,|x|^{2s}x^{-k'}\ov x^{-k''}\chi\rangle$ (and the poles have order $\leq L+1$).

Now, if $\wt B_0\subset\CC$ is a finite set such that two distinct elements do not differ by half an integer, a function $\wt f=\sum_{\beta\in\wt B_0}\sum_{\ell=0}^L\wt f_{\beta,\ell}|x|^{2\beta}\Lx^\ell$ with coefficients in $C^\infty(\wt X)$ can be rewritten $\sum_{\beta\in B_0}\sum_{\ell=0}^Lf_{\beta,\ell}|x|^{2\beta}\Lx^\ell$ for some subset $B_0$, with $f_{\beta,\ell}\in C^\infty(X)$, if and only if there exists a finite subset $A_0\subset\CC$ such that, for all $k',k''\in\hNN$ with $k'+k''\in\nobreak\NN$, the poles of $s\mto\langle \wt f,|x|^{2s}x^{-k'}\ov x^{-k''}\chi\rangle$ are contained in \hbox{$(A_0+k'-\NN)\cap(A_0+k''-\NN)$}.

Lastly, if $\wt f$ has an expansion of the kind \eqref{eq:distholmod} with coefficients in $C^\infty(\wt X)$, the condition above applied to $\wt f$ is equivalent to the fact that $\wt f$ can be rewriten with coefficients $f_{0,\beta,\ell}\in C^\infty(X)$, according to an obvious analogue of Lemma \ref{lem:sanspole}.

We apply this to $k(m',\ov{m''})$: that the condition on the Mellin transform is fulfilled is seen by using equations like \eqref{eq:Bernstein1} for $m'$ and $m''$, in the same way as in Corollary \ref{cor:asympt} and this gives the result for the coefficients corresponding to $\varphi=0$. If $\varphi\neq0$, one applies the same reasoning to $\cE^{-\varphi}\otimes \cM'$ et $\cE^{\varphi}\otimes \cM''$.
\end{proof}

\subsection{Comments}
The notion of Hermitian dual of a $\cD$-module was introduced by M\ptbl Kashiwara in \cite{Kashiwara86}, where many applications of the property that the Hermitian dual of a holonomic $\cD$-module remains holonomic have been given (see also~\cite{Bjork93}). Kashiwara only treated the case of regular holonomic $\cD$-modules (in arbitrary dimension), and examples of such regular holonomic distributions also appear in \cite{Barlet83}.

The vanishing of $\cExt^1$ in Theorem \ref{th:Hermdualone} is already apparent in \cite[Th\ptbl10.2]{Malgrange74}, if one restricts to the real domain however. The analysis of the $\cHom$ has been done in \cite{Svensson81}, in particular Th\ptbl3.1 which is a real version of Theorem \ref{th:distholmod}, still in the real domain (I~thank J.-E\ptbl Björk for pointing this reference out to me).

\chapterspace{-3}
\chapter[Riemann-Hilbert and Laplace]{Riemann-Hilbert and Laplace on the~affine~line (the regular case)}\label{chap:Laplace}

\begin{sommaire}
The Laplace transform $\Fou M$ of a holonomic $\cD$-module $M$ on the affine line $\Afu$ is also holonomic. If $M$ has only regular singularities (included at infinity), $\Fou M$ provides the simplest example of an irregular singularity (at infinity). We will describe the Stokes-filtered local system attached to $\Fou M$ at infinity in terms of data of $M$. More precisely, we define the topological Laplace transform of the perverse sheaf $\pDR^\an M$ as a perverse sheaf on $\Afuh$ equipped with a Stokes structure at infinity. We make explicit this topological Laplace transform. As a consequence, if $\kk$ is a subfield of $\CC$ and if we have a $\kk$-structure on $\pDR^\an M$, we find a natural $\kk$-structure on $\pDR^\an\Fou M$ which extends to the Stokes filtration at infinity. In other words, the Stokes matrices can be defined over~$\kk$. We end this \chaptersname by analyzing the behaviour of duality by Laplace and topological Laplace transformation, and the relations between them.
\end{sommaire}

\subsection{Introduction}\label{subsec:introRHL}
We denote by \index{$CLT$@$\Clt$}$\Clt$ the \emphb{Weyl algebra} in one variable, consisting of linear differential operators in one variable $t$ with polynomial coefficients. Let $M$ be a holonomic $\Clt$-module (\ie $M$ is a left $\Clt$-module such that each element is annihilated by some nonzero operator in $\Clt$), that we also regard as a sheaf of holonomic modules over the sheaf $\cD_{\Afu}$ of algebraic differential operators on the affine line $\Afu$ (with its Zariski topology). Its \emphb{Laplace transform} \index{$MFOU$@$\Fou M$, $\Fou\ccM$, $\FcM$}$\Fou M$ (also called the Fourier transform) is a holonomic $\Cltau$-module, where $\tau$ is a new variable. Recall that $\Fou M$ can be defined in various equivalent ways. We consider below the Laplace transform with kernel $e^{t\tau}$, and a similar description can be made for the \index{Laplace transform!inverse}inverse Laplace transform, which has kernel $e^{-t\tau}$.\enlargethispage{\baselineskip}%

\begin{enumerate}
\item\label{enum:introRHL1}
The simplest way to define \index{$MFOU$@$\Fou M$, $\Fou\ccM$, $\FcM$}$\Fou M$ is to set $\Fou M=M$ as a $\CC$-vector space and to define the action of $\Cltau$ in such a way that $\tau$ acts as $-\partial_t$ and $\partial_\tau$ as $t$ (this is modeled on the behaviour of the action of differential operators under Fourier transform of temperate distributions).
\item\label{enum:introRHL2}
One can mimic the Laplace integral formula, replacing the integral by the direct image of $\cD$-modules. We consider the diagram
\[
\xymatrix@=.5cm{
&\Afu\times\Afuh\ar[dl]_p\ar[dr]^{\wh p}\\
\Afu&&\Afuh
}
\]
where $t$ is the coordinate on $\Afu$ and $\tau$ that on $\Afuh$. Then $\Fou M=\wh p_+(p^+M\otimes E^{t\tau})$, where \index{$ETTAU$@$E^{t\tau}$}$E^{t\tau}$ is $\CC[t,\tau]$ equipped with the connection $d+d(t\tau)$, and $p^+M$ is $\CC[\tau]\otimes_\CC M$ equipped with its natural connection. Recall also that $\wh p_+$ is the \index{push-forward (direct image)!of $\cD$-modules}\emph{direct image} of $\cD$-modules, which is defined here in a simple way: $\wh p_+(p^+M\otimes E^{t\tau})$ is the complex
\[
0\to(p^+M\otimes E^{t\tau})\To{\partial_t}(p^+M\otimes E^{t\tau})\to0
\]
where the source of $\partial_t$ is in degree $-1$ and the target in degree~$0$. More concretely, this complex is written
\[
0\to \CC[\tau]\otimes_\CC M\To{1\otimes\partial_t+\tau\otimes1}\CC[\tau]\otimes_\CC M\to0
\]
and one checks that the differential is injective, so that the complex is quasi-isomorphic to its cokernel. This complex is in the category of $\CC[\tau]$-modules, and is equipped with an action of $\Cltau$, if $\partial_\tau$ acts as $\partial_\tau\otimes1+1\otimes t$. The map $\CC[\tau]\otimes_\CC M\to M$ sending $\tau^k\otimes m$ to $(-\partial_t)^km$ identifies the cokernel with $\Fou M$ as defined in \eqref{enum:introRHL1}.

\item\label{enum:introRHL3}
It will be useful to work with the analytic topology (not the Zariski topology, as above). In order to do so, one has to consider the projective completion $\PP^1$ of $\Afu$ (\resp $\wh\PP^1$ of $\Afuh$) obtained by adding the point $\infty$ to $\Afu$ (\resp the point $\wh\infty$ to $\Afuh$). In the following, we shall denote by $t'$ (\resp $\tau'$) the coordinate centered at $\infty$ (\resp $\wh\infty$), so that $t'=1/t$ (\resp $\tau'=1/\tau$) on $\Afu\moins\{0\}$ (\resp on $\Afuh\moins\{\wh0\}$). We consider the diagram
\begin{equation}\label{eq:ppproper}
\begin{array}{c}
\xymatrix@=.5cm{
&\PP^1\times\wh\PP^1\ar[dl]_p\ar[dr]^{\wh p}\\
\PP^1&&\wh\PP^1
}
\end{array}
\end{equation}

Let $\ccM$ (\resp \index{$MFOU$@$\Fou M$, $\Fou\ccM$, $\FcM$}$\Fou\ccM$) be the algebraic $\cD_{\PP^1}$-module (\resp $\cD_{\wh\PP^1}$-module) determined by~$M$ (\resp $\Fou M$): it satisfies by definition $\ccM=\cO_{\PP^1}(*\infty)\otimes_{\cO_{\PP^1}}\ccM$ and $M=\Gamma(\PP^1,\ccM)$ (and similarly for $\Fou M$). It is known that $\ccM$ (\resp $\Fou\ccM$) is still holonomic. Let~$\cM$ (\resp $\FcM$) be its analytization. We now write $\cD_{\PP^1}$ (\resp $\cD_{\wh\PP^1}$) instead of $\cD_{\PP^1}^\an$ (\resp $\cD_{\wh\PP^1}^\an$). Notice that, by definition, $\cM=\cO_{\PP^1}(*\infty)\otimes_{\cO_{\PP^1}}\cM$ and \index{$MFOU$@$\Fou M$, $\Fou\ccM$, $\FcM$}$\FcM=\cO_{\wh\PP^1}(*\wh\infty)\otimes_{\cO_{\wh\PP^1}}\FcM$. Applying a similar construction to $E^{t\tau}$ we get $\cE^{t\tau}$ on $\PP^1\times\wh\PP^1$, which is a free $\cO_{\PP^1\times\wh\PP^1}(*(D_\infty\cup D_{\wh\infty}))$-module of rank one, with $D_\infty\cup D_{\wh\infty}=\PP^1\times\wh\PP^1\moins(\Afu\times\Afuh)$. Then we have
\[
\FcM=\cH^0\wh p_+(p^+\cM\otimes\cE^{t\tau}).
\]
\end{enumerate}

If $M$ is a regular holonomic $\Clt$-module (\ie regular at finite distance \emph{and} at infinity), the following is well-known (\cf \eg \cite{Malgrange91} for the results and the definition of the \emphb{vanishing cycle!moderate} functor):
\begin{enumeratea}
\item\label{enum:a}
The Laplace transform $\Fou M$ is holonomic, has a regular singularity at the origin $\tau=0$, no other singularity at finite distance, and possibly irregular at infinity.
\item\label{enum:b}
The formal structure of $\Fou M$ at infinity can be described exactly from the vanishing cycles of $M$ (or of $\DR^\an M$) at its critical points at finite distance. More precisely, denoting by $\wh{\FcM}$ the formalized connection at $\wh\infty$, we have a decomposition $\wh{\FcM}\simeq\bigoplus_c(\cE^{c/\tau'}\otimes\cR_c)$, where the sum is taken over the singular points $c\in\Afu$ of $M$, and $\cR_c$ is a regular formal meromorphic connection corresponding in a one-to-one way to the data of the \index{vanishing cycle}vanishing cycles of the perverse sheaf $\pDR^\an M$ at~$c$. As a consequence, the set of exponential factors of the Stokes filtration of $\Fou M$ at $\wh\infty$ consists of these $c/\tau'$, and the Stokes filtration is non-ramified.
\end{enumeratea}

The purpose of this \chaptersname is to give an explicit formula for the Stokes filtration of $\Fou M$ at infinity, in terms of topological data obtained from $M$. More precisely, let $\cF=\pDR^\an M$ be the analytic de~Rham complex of $M$ (shifted according to the usual perverse convention). The question we address is a formula for the Stokes filtration of~$\Fou M$ at $\tau=\infty$ in terms of $\cF$ only. In other words, we will define a topological Laplace transform of $\cF$ as being a perverse sheaf on $\Afuh$ with a Stokes filtration at infinity, in such a way that the topological Laplace transform of $\DR^\an M$ is $\DR^\an\Fou M$ together with its Stokes filtration at infinity (\ie $\DR^\an\Fou M$ as a Stokes-perverse sheaf on $\wh\PP^1$, \cf Definition \ref{def:Stperv1}).

A consequence of this result is that, if the perverse sheaf $\cF$ is defined over $\QQ$ (say), then the Stokes filtration of $\Fou M$ at infinity is also defined over $\QQ$.

\subsection{Direct image of the moderate de~Rham complex}
Our goal is to obtain the Stokes filtration of $\Fou M$ at infinity as the direct image by $\wh p$ of a sheaf defined over a completion of $\Afu\times\Afuh$ (integral formula).

Let $X\defin\wt\PP^1$ be the real blow-up space of $\PP^1$ at $\infty$ (in order to simplify the notation, we do not use the same notation as in \Chaptersname\ref{chap:Stokesone-pervers}). This space is homeomorphic to a closed disc with boundary $S^1_\infty\defin S^1\times\{\infty\}$. A similar construction can be done starting from $\Afuh$ and its projective completion~$\wh\PP^1$. We get a space $\wh X\defin\wt{\wh\PP^1}$ with boundary $S^1_{\wh\infty}$. We set $X^*=X\moins\{0\}$ and $\wh X{}^*\defin\wh X\moins\{\wh0\}$.

We thus have a diagram
\begin{equation}\label{eq:diagramwhX}
\begin{array}{c}
\xymatrix{
&\PP^1\times\wh X\ar[d]_{\id\times\wh\varpi}\ar@/^1pc/[ddrr]^{\wh q}&&\\
&\PP^1\times\wh\PP^1\ar[dl]_p\ar[dr]^{\wh p}&&\\
\PP^1&&\wh\PP^1&\ar[l]_-{\wh\varpi}\wh X
}
\end{array}
\end{equation}

We denote by $\cA_{\wh X}^{\rmod\wh\infty}$ the sheaf on $\wh X$ of holomorphic functions on $\Afuh$ which have moderate growth along $S^1_{\wh\infty}$.

Similarly we denote by $\cA_{\PP^1\times\wh X}^{\rmod\wh\infty}$ the sheaf on $\PP^1\times\wh X$ of holomorphic functions on $\PP^1\times\Afuh$ which have moderate growth along $\PP^1\times S^1_{\wh\infty}$. We have a natural inclusion
\begin{equation}\label{eq:naturalA}
\wh q^{-1}\cA_{\wh X}^{\rmod\wh\infty}\hto\cA_{\PP^1\times\wh X}^{\rmod\wh\infty}.
\end{equation}

We will denote by $\DR^{\rmod\wh\infty}$ the de~Rham complex of a $\cD$-module with coefficients in such rings.

\begin{lemme}\label{lem:DRmod}
There is a functorial morphism
\bgroup\def\theequation{\ref{lem:DRmod}$_{\leq0}$}\addtocounter{equation}{-1}
\begin{equation}\label{eq:DRmod}
\DR^{\rmod\wh\infty}(\FcM)\to \bR\wh q_*\DR^{\rmod\wh\infty}(p^+\cM\otimes\cE^{t\tau})[1]
\end{equation}
\egroup
and this morphism is injective on the zero-th cohomology sheaves.
\end{lemme}

\begin{proof}
Let us first consider the relative de~Rham complex $\DR_{\wh q}^{\rmod\wh\infty}(p^+\cM\otimes\cE^{t\tau})$ defined as
\[
\cA_{\PP^1\times\wh X}^{\rmod\wh\infty}\otimes(p^+\cM\otimes\cE^{t\tau})\To{\nabla'}\cA_{\PP^1\times\wh X}^{\rmod\wh\infty}\otimes\big[\Omega^1_{\PP^1\times\wh X/\wh X}\otimes(p^+\cM\otimes\cE^{t\tau})\big],
\]
where $\nabla'$ is the relative part of the connection on $(p^+\cM\otimes\cE^{t\tau})$, \ie which differentiates only in the $\PP^1$ direction. Then $\DR^{\rmod\wh\infty}(p^+\cM\otimes\cE^{t\tau})$ is the single complex associated to the double complex
\[
\DR_{\wh q}^{\rmod\wh\infty}(p^+\cM\otimes\cE^{t\tau})\To{\nabla''}\big((\id\otimes\varpi)^{-1}\Omega^1_{\PP^1\times\wh\PP^1/\PP^1}\big)\otimes\DR_{\wh q}^{\rmod\wh\infty}(p^+\cM\otimes\cE^{t\tau}),
\]
where $\nabla''$ is the connection in the $\wh\PP^1$ direction. Hence, $\bR q_*\DR^{\rmod\wh\infty}(p^+\cM\otimes\nobreak\cE^{t\tau})$ is the single complex associated to the double complex
\[
\bR\wh q_*\DR_{\wh q}^{\rmod\wh\infty}(p^+\cM\otimes\cE^{t\tau})\To{\nabla''}\varpi^{-1}\Omega^1_{\wh\PP^1}\otimes\bR\wh q_*\DR_{\wh q}^{\rmod\wh\infty}(p^+\cM\otimes\cE^{t\tau}).
\]
The point in using the relative de~Rham complex is that it is a complex of $(\wh\varpi\circ\wh q)^{-1}\cO_{\wh\PP^1}$-modules and the differential is linear with respect to this sheaf of rings. On the other hand, one knows that $\cA_{\wh X}^{\rmod\wh\infty}$ is flat over $\wh\varpi^{-1}\cO_{\wh\PP^1}$, because it has no $\wh\varpi^{-1}\cO_{\wh\PP^1}$-torsion. So we can apply the projection formula to get (in a functorial way)
\begin{multline*}
\cA_{\wh X}^{\rmod\wh\infty}\ootimes_{\wh\varpi^{-1}\cO_{\wh\PP^1}}\bR\wh q_*(\id\times\wh\varpi)^{-1}\DR_{\wh p}(p^+\cM\otimes\cE^{t\tau})\\
\simeq\bR\wh q_*\Big(\wh q^{-1}\cA_{\wh X}^{\rmod\wh\infty}\ootimes_{\wh q^{-1}\wh\varpi^{-1}\cO_{\wh\PP^1}}(\id\times\wh\varpi)^{-1}\DR_{\wh p}(p^+\cM\otimes\cE^{t\tau})\Big).
\end{multline*}
Using the natural morphism \eqref{eq:naturalA}, we find a morphism
\[
\cA_{\wh X}^{\rmod\wh\infty}\otimes\bR\wh q_*(\id\times\wh\varpi)^{-1}\DR_{\wh p}(p^+\cM\otimes\cE^{t\tau})\to\bR\wh q_*\DR_{\wh q}^{\rmod\wh\infty}(p^+\cM\otimes\cE^{t\tau}),
\]
and therefore, taking the double complex with differential $\nabla''$ and the associated single complex, we find
\begin{equation}\label{eq:intermediaire}
\cA_{\wh X}^{\rmod\wh\infty}\otimes\bR\wh q_*(\id\times\wh\varpi)^{-1}\DR(p^+\cM\otimes\cE^{t\tau})\to\bR\wh q_*\DR^{\rmod\wh\infty}(p^+\cM\otimes\cE^{t\tau}).
\end{equation}
On the other hand, we know (\cf \eg \cite{Malgrange90b}) that the (holomorphic) de~Rham functor has a good behaviour with respect to the direct image of $\cD$-modules, that is, we have a functorial isomorphism
\[
\bR\wh p_*\DR(p^+\cM\otimes\cE^{t\tau})[1]\simeq\DR\big[\wh p_+(p^+\cM\otimes\cE^{t\tau})\big]
\]
where the shift by one comes from the definition of $p_+$ for $\cD$-modules. Because $\wh p$ is proper and the commutative diagram in \eqref{eq:diagramwhX} is cartesian, the base change $\wh\varpi^{-1}\bR\wh p_*\simeq(\id\times\wh\varpi)^{-1}\bR\wh q_*$ shows that the left-hand side of \eqref{eq:intermediaire} is $\DR^{\rmod\wh\infty}(\FcM)[-1]$. Shifting \eqref{eq:intermediaire} by one gives the functorial morphism \eqref{eq:DRmod}.

Notice that, over $\Afuh$, this morphism is an isomorphism, as it amounts to the compatibility of $\DR$ and direct images (\cf the standard references on $\cD$-modules). On the other hand, when restricted to $S^1_{\wh\infty}$, the left-hand term has cohomology in degree $0$ at most (\cf Theorem \ref{th:H1nul}).

Let us show that $\cH^0\eqref{eq:DRmod}$ is injective. It is enough to check this on $S^1_{\wh\infty}$. If $\wtj:\Afuh\hto\wh X$ denotes the inclusion, we have a commutative diagram
\[
\xymatrix{
\cH^0\DR^{\rmod\wh\infty}(\FcM)\ar[r]^-{\cH^0\eqref{eq:DRmod}}\ar@{^{ (}->}[d]&\cH^0\bR\wh q_*\DR^{\rmod\wh\infty}(p^+\cM\otimes\cE^{t\tau})[1]\ar[d]\\
\wtj_*\wtj^{-1}\cH^0\DR^{\rmod\wh\infty}(\FcM)\ar@{=}[r]&\wtj_*\wtj^{-1}\cH^0\bR\wh q_*\DR^{\rmod\wh\infty}(p^+\cM\otimes\cE^{t\tau})[1]
}
\]
and the injectivity of $\cH^0\eqref{eq:DRmod}$ follows.
\end{proof}

\begin{theoreme}\label{th:RHL}
The morphism \eqref{eq:DRmod} is an isomorphism.
\end{theoreme}

This theorem reduces our problem of expressing the Stokes filtration of $\Fou M$ at $\wh\infty$ in terms of $\DR^\an M$ to the question of expressing $\DR^{\rmod\wh\infty}(p^+\cM\otimes\cE^{t\tau})$ in terms of $\DR\cM$. Indeed, applying $\bR\wh p_*$ would then give us the answer for the $\leq0$ part of the Stokes filtration of $\Fou M$, as recalled in \eqref{enum:b} of \S\ref{subsec:introRHL}. The $\leq c/\tau'$ part is obtained by replacing~$t$ with $t-c$ and applying the same argument.

By Lemma \ref{lem:DRmod}, it is enough to prove that the right-hand term of \eqref{eq:DRmod} has cohomology in degree zero at most, and that $\cH^0\eqref{eq:DRmod}$ is onto on $S^1_{\wh\infty}$, or equivalently that the germs of both $\cH^0$ at any $\wh\theta\in S^1_{\wh\infty}$ have the same dimension.

The proof of the theorem will be done by identifying $\DR^{\rmod\wh\infty}(p^+\cM\otimes\nobreak\cE^{t\tau})$ with a complex constructed from $\DR^\an M$ and by computing explicitly its direct image. This computation will be topological. This complex obtained from $\DR^\an M$ will be instrumental for defining the topological Laplace transform. Moreover, the identification will be more easily done on a space obtained by blowing up $\PP^1\times\wh\PP^1$, in order to use asymptotic analysis (see also Remark \ref{rem:blowup} below). So, before proving the theorem, we first do the topological computation and define the topological Laplace transform. (See also Theorem \ref{th:Rpimod} and the references given there for another, more direct proof.)

\begin{remarque}
If we replace the sheaf $\cA_{\wh X}^{\rmod\wh\infty}$ by the sheaf $\cA_{\wh X}^{\rd\wh\infty}$ of functions having rapid decay at infinity, and similarly for $\cA_{\PP^1\times\wh X}^{\rd\wh\infty}$, the same proof as for Lemma \ref{lem:DRmod} gives a morphism for the corresponding rapid decay de~Rham complexes, that we denote by $(\ref{lem:DRmod})_{<0}$. Then Lemma \ref{lem:DRmod} and Theorem \ref{th:RHL} are valid for the rapid-decay complexes, and the proofs are similar.
\end{remarque}

\subsection{Topological spaces}\label{subsec:topspaces}
The behaviour of the function $e^{t\tau}$ near the divisor $D_\infty\cup D_{\wh\infty}=\PP^1\times\wh\PP^1\moins(\Afu\times\Afuh)$, where we set $D_\infty\defin\{\infty\}\times\wh\PP^1$ and $D_{\wh\infty}\defin\PP^1\times\{\wh\infty\}$, is not clearly understood near the two points $(0,\wh\infty)$ and $(\infty,\wh0)$, where it is written in local coordinates as $e^{t/\tau'}$ and $e^{\tau/t'}$ respectively. This leads us to blow up these points.

Let $e:Z\to\PP^1\times\wh\PP^1$ be the complex blowing-up of $(0,\wh\infty)$ and $(\infty,\wh0)$ in $\PP^1\times\wh\PP^1$. Above the chart with coordinates $(t,\tau')$ in $\PP^1\times\wh\PP^1$, we have two charts of $Z$ with respective coordinates denoted by $(t_1,u)$ and $(v',\tau'_1)$, so that~$e$ is given respectively by $t\circ e=t_1$, $\tau'\circ e=t_1u$ and $t\circ e=v'\tau'_1$, $\tau'\circ e=\tau'_1$. The exceptional divisor~$E$ of~$e$ above $(0,\wh\infty)$ is defined respectively by $t_1=0$ and $\tau'_1=0$.

Similarly, above the chart with coordinates $(t',\tau)$ in $\PP^1\times\wh\PP^1$, we have two charts of $Z$ with respective coordinates denoted by $(t'_1,u')$ and $(v,\tau_1)$, so that~$e$ is given respectively by $t'\circ e=t'_1$, $\tau\circ e=t'_1u'$ and $t'\circ e=v\tau_1$, $\tau\circ e=\tau_1$. The exceptional divisor~$F$ of~$e$ above $(\infty,\wh0)$ is defined respectively by $t'_1=0$ and $\tau_1=0$.

Away from~$E\cup F$ (in $Z$) and $\{(0,\wh\infty),(\infty,\wh0)\}$ (in $\PP^1\times\wh\PP^1$),~$e$ is an isomorphism. In particular, we regard the two sets $\Afu\times\Afuh$ and $(\PP^1\moins\{0\})\times(\wh\PP^1\moins\{\wh0\})$ as two open subsets of $Z$ whose union is $Z\moins(E\cup F)=\PP^1\times\wh\PP^1\moins\{(0,\wh\infty),(\infty,\wh0)\}$.

Moreover, the strict transform of $D_\infty$ in $Z$ is a divisor in $Z$ which does not meet~$E$, and meets transversally $F$ and the strict transform of $D_{\wh\infty}$. Similarly, the strict transform of $D_{\wh\infty}$ in $Z$ is a divisor which meets transversally both~$E$ and the strict transform of $D_\infty$, and does not meet $F$. We still denote by $D_\infty,D_{\wh\infty}$ these strict transforms, and we denote by $D$ the normal crossing divisor $D=e^{-1}(D_\infty\cup D_{\wh\infty})=D_\infty\cup D_{\wh\infty}\cup E\cup F$. This is represented on Figure \ref{fig:Z}.
\begin{figure}[htb]
\begin{center}
\includegraphics[scale=.5]{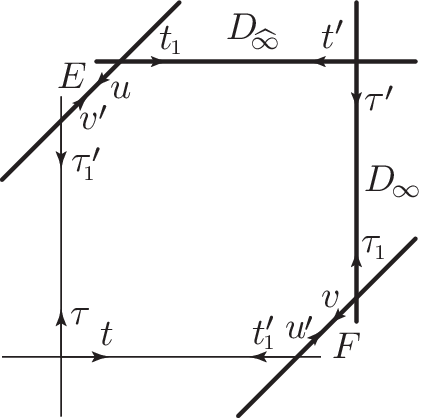}
\end{center}
\vspace*{-\baselineskip}
\caption{The natural divisors on $Z$}\label{fig:Z}
\end{figure}

Let $\wt Z$ be the real blow-up space of $Z$ along the four components $D_\infty,D_{\wh\infty},E,F$ of the normal crossing divisor $D$. Then the map~$e$ lifts as $\wt e:\wt Z\to X\times\wh X$. We have the following commutative diagram of maps:
\begin{equation}\label{eq:ZX}
\begin{array}{c}
\xymatrix@C=1.5cm@R=.7cm{
\wt Z\ar[dd]_{\varpi_Z}\ar[r]^-{\wt e}\ar[rd]_-{\wt\epsilon}&X\times\wh X\ar[d]^{\varpi\times\id}\\
&\PP^1\times\wh X\ar[d]^{\id\times\wh\varpi}\\
Z\ar[r]^-e&\PP^1\times\wh\PP^1
}
\end{array}
\end{equation}
The set $\wt E\defin\varpi_Z^{-1}(E)$ can be described as follows. Over the chart $\Afu_{v'}$ of $E$ with coordinate $v'$, it is equal to $\Afu_{v'}\times S^1$ (coordinates $(v',\arg\tau'_1)$), and over $u=0$ it is equal to $S^1\times S^1$ (coordinates $(\arg u,\arg t_1)$), so it can be identified with $\wt\PP^1_{v'}\times S^1$ through the gluing over $u=0$ defined by the diffeomorphism $(\arg v',\arg\tau'_1)\mto(\arg u=-\arg v',\arg t_1=\arg v'+\arg\tau'_1)$. Similarly, the set $\wt F\defin\varpi_Z^{-1}(F)$ can be described as follows: over the chart $\Afu_{u'}$ of $F$ with coordinate $u'$, it is equal to $\Afu_{u'}\times S^1$ (coordinates $(u',\arg t'_1)$), and over $v=0$ it is equal to $S^1\times S^1$ (coordinates $(\arg v,\arg\tau_1)$), so it can be identified with $\wt\PP^1_{u'}\times S^1$ through the gluing over $v=0$ defined by the diffeomorphism $(\arg u',\arg t'_1)\mto(\arg v=-\arg u',\arg \tau_1=\arg u'+\arg t'_1)$. The complement $\wt Z\moins(\wt E\cup\wt F)$ is identified with $X\times\wh X\moins[(\{0\}\times S^1_{\wh\infty})\cup(S^1_\infty\times\{\wh0\})]$.

We define two nested closed subsets $L''_{\leq0}\subset L''_{<0}$ of $\wt Z$ as the subsets of points in the neighbourhood of which $e^{-t\tau}$ does not have moderate growth (\resp is not exponentially decreasing). More specifically:
\begin{enumerate}
\item
$L''_{\leq0}\subset \varpi_Z^{-1}(D_\infty\cup D_{\wh\infty})$,
\item
over the chart $(t_1,u)$, $L''_{\leq0}=\{\arg u\in[\pi/2,3\pi/2]\bmod2\pi\}\cap\{u=0\}$,
\item
over the chart $(t',\tau')$, $L''_{\leq0}=\{\arg t'+\arg\tau'\in[\pi/2,3\pi/2]\bmod2\pi\}$,
\item
over the chart $(v,\tau_1)$, $L''_{\leq0}=\{\arg v\in[\pi/2,3\pi/2]\bmod2\pi\}\cap\{v=0\}$.
\item
We have $L''_{<0}=L''_{\leq0}\cup\wt E\cup\wt F$.
\end{enumerate}

We denote by $L'_{\leq0}$ (\resp $L'_{<0}$) the complementary open set of $L''_{\leq0}$ (\resp $L''_{<0}$) in~$\wt Z$. So $L'_{\leq0}\supset L'_{<0}\supset\Afu\times\Afuh$ and, away from $u=0$ (\resp away from $v=\nobreak0$), $L'_{\leq0}\cap\wt E$ (\resp $L'_{\leq0}\cap\wt F$) coincides with $\wt E$ (\resp $\wt F$). Moreover, $L'_{<0}\cap\wt E=\emptyset$ and $L'_{<0}\cap\nobreak\wt F=\nobreak\emptyset$.

We have a diagram
\begin{equation}\label{eq:diagabc}
\begin{array}{c}
\xymatrix{
\Afu\times\Afuh\ar@{^{ (}->}[r]^-\alpha\ar[d]_p& L'_{<0}\ar@{^{ (}->}[r]^-\gamma&L'_{\leq0}\ar@{^{ (}->}[r]^-\beta&\wt Z\ar[d]^{\wt{\wh q}}\\
\Afu&&&\wh X
}
\end{array}
\end{equation}

\subsection{Topological Laplace transform}\index{Laplace transform!topological --}
Let $\cF$ be a perverse sheaf on $\Afu$. Unless otherwise stated, we will usually denote by $\cG$ the perverse sheaf $p^{-1}\cF[1]$ on $\Afu\times\Afuh$.

\begin{definitio}\label{def:cFcG}
Using the notation as in Definition \ref{def:coCSt},
\begin{itemize}
\item
for any perverse sheaf $\cG$ on $\Afu\times\Afuh$, we set $(\cFcG)_{\leq0}=\beta_!\bR(\gamma\circ\alpha)_*\cG$ and $(\cFcG)_{\prec0}=(\beta\circ\nobreak\gamma)_!\bR\alpha_*\cG$,
\item
if $\cG=p^{-1}\cF[1]$, we set \index{$FLFC$@$(\FcF)_{\leq0}$, $(\FcF)_{\prec0}$}$(\FcF)_{\leq0}=\bR\wt{\wh q}_*(\cFcG)_{\leq0}$ and \index{$FLFC$@$(\FcF)_{\leq0}$, $(\FcF)_{\prec0}$}$(\FcF)_{\prec0}=\bR\wt{\wh q}_*(\cFcG)_{\prec0}$.
\end{itemize}
\end{definitio}

\begin{proposition}\label{prop:Laplace0}
When restricted to $\Afuh$, $(\FcF)_{\leq0|\Afuh}=(\FcF)_{\prec0,|\Afuh}$ is a perverse sheaf with singularity at $\wh0$ at most, and $i_{\wh0}^{-1}[(\FcF)_{\leq0}]=\bR\Gamma_\rc(\Afu,\cF)[1]$. On the other hand, when restricted to $\wh X^*$, $(\FcF)_{\prec0}[-1]$ and $(\FcF)_{\leq0}[-1]$ are two nested sheaves, and $\gr_0\FcF[-1]\defin(\FcF)_{\leq0}[-1]/(\FcF)_{\prec0}[-1]$ is a local system on $S^1_{\wh\infty}$ isomorphic to $(\phip_t\cF,T)$ when considered as a vector space with monodromy.
\end{proposition}

As usual, $\Gamma_\rc$ denotes the sections with compact support and $\bR\Gamma_\rc$ denotes its derived functor, whose associated cohomology is the cohomology with compact support. Given a perverse sheaf $\cF$ on $\Afu$, the \index{nearby cycle!complex}\emph{nearby cycle} complex \index{$FYPSI$@$\psip_t\cF$, $\phip_t\cF$}$(\psip_t\cF,T)$ and the \index{vanishing cycle!complex}\emph{vanishing cycle} complex \index{$FYPSI$@$\psip_t\cF$, $\phip_t\cF$}$(\phip_t\cF,T)$ (equipped with their monodromy) have cohomology in degree zero at most, and there is a canonical morphism $\mathrm{can}:(\psip_t\cF,T)\to(\phip_t\cF,T)$ whose cone represents the inverse image $i_0^{-1}\cF[-1]$ (\cf\eg\cite{Malgrange91}).

\begin{proof}
Let $C\subset\Afu$ denote the set of singular points of $\cF$, and let us still denote by $C\times\wh\PP^1\subset Z$ the strict transform by $e$ of $C\times\wh\PP^1\subset\PP^1\times\wh\PP^1$. Since $\cF[-1]$ is a local system away from $C$, one checks that $\bR\alpha_*p^{-1}\cF[-1]$ (\resp $\bR(\gamma\circ\alpha)_*p^{-1}\cF[-1]$) is a local system on $L'_{<0}\moins\varpi_Z^{-1}(C\times\wh\PP^1)$ (\resp on $L'_{\leq0}\moins\varpi_Z^{-1}(C\times\wh\PP^1)$). This will justify various restrictions to fibres that we will perform below.

For the first assertion, let us work over $\wt V\defin X\times\Afuh$, and still denote by $\wt e:\wt W\to\wt V$ the restriction of $\wt e$ over this open set.
\begin{figure}[htb]
\centerline{\includegraphics[scale=.5]{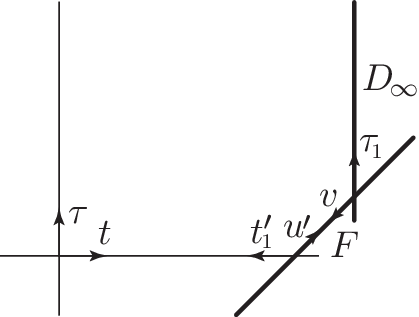}}
\vspace*{-\baselineskip}
\caption{}
\end{figure}
Over $\Afuh{}^*$, $\wt W$ coincides with $X\times\Afuh{}^*$. Then it is known (\cf \eg \cite{Malgrange91} or \cite[\S1b]{Bibi05}) that $(\FcF)_{\leq0|\Afuh{}^*}$ is a smooth perverse sheaf (\ie a~local system shifted by one) with germ at $\tau_o\neq0$ equal to $\HH^0_{\Phi_{\tau_o}}(\Afu,\cF)[1]$, where $\Phi_{\tau_o}$ is the family of closed sets in $\Afu$ whose closure in $\wt\PP^1$ does not cut $L''_{\leq0}\cap(S^1_\infty\times\{\tau_o\})$. Moreover, $L''_{\leq0}$ and $L''_{<0}$ coincide above $\Afuh{}^*$, so $(\FcF)_{\leq0|\Afuh{}^*}=(\FcF)_{\prec0|\Afuh{}^*}$.

On the other hand, we claim that
\[
\bR\wt e_*(\cFcG)_{\leq0}{}_{|S^1_\infty\times\{\wh 0\}}=\bR\wt e_*(\cFcG)_{\prec0}{}_{|S^1_\infty\times\{\wh 0\}}=0.
\]
This will show that
\[
\bR\wt e_*(\cFcG)_{\leq0}{}_{|X\times\{\wh 0\}}=\bR\wt e_*(\cFcG)_{\prec0}{}_{|X\times\{\wh 0\}}\simeq \wtj_!\cF[1],
\]
where $\wtj:\Afu\hto X$ denotes the inclusion, concluding the proof of the first assertion by taking $\bR\Gamma$. Moreover, the natural map $i_{\wh0}^{-1}[(\FcF)_{\leq0}][-1]\to\psip_\tau[(\FcF)_{\leq0}]$ is then induced by the natural map $\bR\Gamma_\rc(\Afu,\cF)\to\bR\Gamma_{\Phi_{\tau_o}}(\Afu,\cF)=\HH^0_{\Phi_{\tau_o}}(\Afu,\cF)$. The cone of this map is the complex $\bR\Gamma_\rc\big(L'_{\leq0}\cap(S^1_\infty\times\{\tau_o\}),\cL[1]\big)$, where~$\cL$ is the local system defined by~$\cF$ on $S^1_\infty$. In particular, it has cohomology in degree~$0$ at most. As a consequence, the complex $\phip_\tau[(\FcF)_{\leq0}]$ has cohomology in degree~$0$ at most, and it follows that $(\FcF)_{\leq0}$ is perverse in the neighbourhood of $\wh0$, since the perversity of any constructible complex $\wh\cF$ near $\tau=0$ is equivalent to both $\psip_\tau\wh\cF$ and $\phip_\tau\wh\cF$ having cohomology in degree $0$ at most.

\begin{remarque}
Developing the proof more carefully at this point, one would obtain an identification of $\phip_\tau[(\FcF)_{\leq0}]$ with $\psip_{t'}\cF$, compatible with monodromies when suitably oriented. The second part of the lemma, that we consider now, is an analogous statement, where the roles of $\Afu$ and $\Afuh$ are exchanged.
\end{remarque}

Let us now prove the claim. On the one hand, $(\cFcG)_{\prec0}$ is zero on $L''_{<0}$ by definition, hence on $\wt F$, so its restriction to $\wt e{}^{-1}(S^1_\infty\times\{\wh0\})$ is zero.

On the other hand, $\wt e{}^{-1}(\theta',\wh0)$ is homeomorphic to the real blow-up space of $F$ at $v=0$ (topologically a closed disc), and $L'_{\leq0}\cap \wt e{}^{-1}(\theta',\wh0)$ is homeomorphic to a closed disc with a closed interval deleted on its boundary. On this set, $\bR(\gamma\circ\alpha)_*\cG[-2]$ is a local system (which is thus constant). Since the cohomology with compact support (due to~$\beta_!$) of such a closed disc with a closed interval deleted on its boundary is identically zero (a particular case of Lemme \ref{lem:pervers-semidisque} below), we get the vanishing of $\bR\wt e_*(\cFcG)_{\leq0,(\theta',\wh0)}$ for any $\theta'\in S^1_\infty$, as claimed.

\smallskip
We now prove the second part of the proposition. Let us compute the fibre of $(\FcF)_{\leq0}$ and $(\FcF)_{\prec0}$ at $(|\tau'_o|=0,\arg\tau'_o=\wh\theta'_o)$. Notice that, above this point that we also denote by $\wh\theta'_o$, we have
\[
L'_{<0,\wh\theta'_o}=X^*\moins\{\arg t\in[\wh\theta'_o+\pi/2,\wh\theta'_o+3\pi/2]\bmod2\pi\},
\]
and this set is equal to $\wt\PP^1$ minus the closure of the half-plane $\reel(te^{-i\wh\theta'_o})\leq0$. Then $(\FcF)_{\prec0,\wh\theta'_o}=\bR\Gamma_\rc(L'_{<0,\wh\theta'_o},\bR\alpha_*\cF)[1]$.

\begin{lemme}\label{lem:pervers-semidisque}
Let $\cP$ be a perverse sheaf with finite singularity set on an open disc $\Delta\subset\CC$ and let $\Delta'$ be the open subset of the closed disc $\ov\Delta$ obtained by deleting from $\partial \ov\Delta$ a nonempty closed interval. Let $\alpha:\Delta\hto \Delta'$ denote the open inclusion. Then $\HH^k_\rc(\Delta',\bR\alpha_*\cP)=0$ if $k\neq0$ and $\dim\HH^0_\rc(\Delta',\bR\alpha_*\cP)$ is the sum of dimensions of the \index{vanishing cycle}vanishing cycles of $\cP$ at points of~$\Delta$ (\ie at the singular points of $\cP$ in~$\Delta$).
\end{lemme}

\begin{proof}
This can be proved as follows: one reduces to the case of a sheaf supported on some point (trivial), and to the case of $\cP=j_*\cL[1]$, where $j$ is the inclusion $\Delta\moins\Sing(\cP)\hto \Delta$ and $\cL$ is a local system on $\Delta\moins\Sing(\cP)$; clearly, there is no $H^0_\rc(\Delta',\alpha_*j_*\cL)$ and, arguing by duality, there is no $H^2_\rc(\Delta',\alpha_*j_*\cL)$. The computation of the dimension of $H^1_\rc(\Delta',\alpha_*j_*\cL)$ is then an exercise.
\end{proof}

Since $L'_{<0,\wh\theta'_o}$ is homeomorphic to such a $\Delta'$, we find that $(\FcF)_{\prec0,\wh\theta'_o}$ has cohomology in degree $-1$ only.

Taking into account the description of $\wt E$ given above, the difference $L'_{\leq0}\moins L'_{<0}=L'_{\leq0}\cap\wt E$ is identified to the set $\wt\PP^1_{v'}\times S^1$ with the subset
\[
\{|v'|=\infty,\,\arg v'\in[\pi/2,3\pi/2]\bmod2\pi\}
\]
deleted. Then $L'_{\leq0,\wh\theta'_o}$ is homeomorphic to the space obtained by gluing $S^1\times[0,\infty]\moins\{\arg t_1\in[\wh\theta'_o+\pi/2,\wh\theta'_o+3\pi/2]\bmod2\pi\}$ (with coordinates $(\arg t_1,|t_1|)$) with \hbox{$\wt\PP^1_{v'}\moins\{|v'|=\infty,\,\arg v'\in[\pi/2,3\pi/2]\bmod2\pi\}$} along $|t_1|=0$, $|v'|=\infty$, by identifying $\arg t_1$ with $\arg v'+\wh\theta'_o$. So $L'_{\leq0,\wh\theta'_o}$ is also homeomorphic to $\Delta'$ like in the lemma.
\begin{figure}[htb]
\centerline{\includegraphics[scale=1]{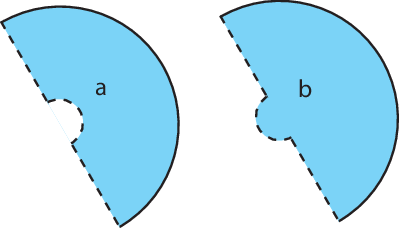}}
\caption{}\label{fig:L'}
\end{figure}

Note that we have an isomorphism $\Afu\times\Afuh{}^*\simeq\Afu_{v'}\times\Afuh{}^*_{\tau'_1}$ given by $(t,\tau)\mto(v'=\nobreak t\tau,\tau'_1=1/\tau)$. It follows that the restriction $\cP$ of $\bR(\gamma\circ\alpha)_*\cG$ to $\Afu_{v'}\times S^1\subset \wt E$ is equal to the pull-back of $\cF$ by the map $(v',\arg\tau'_1)\mto t=v'e^{i\arg\tau'_1}$. Its restriction to $L'_{\leq0,\wh\theta'_o}$ is therefore isomorphic to the pull-back of $\cF$ by the map $\wt\varpi_{\wh\theta'_o}:v'\mto v'e^{i\wh\theta'_o}$. In other words, we can regard the picture of $L'_{\leq0,\wh\theta'_o}$ (Figure \ref{fig:L'}) as the corresponding subset of $\wt\PP^1$ and $\cP$ as the restriction of $\wtj_*\cF$ to this subset. We can apply Lemma \ref{lem:pervers-semidisque} to get that $(\FcF)_{\leq0,\wh\theta'_o}$ has cohomology in degree $-1$ only.

Let us now compute $(\FcF)_{\leq0,\wh\theta'_o}/(\FcF)_{\prec0,\wh\theta'_o}$. By the previous computation, this is
\[
\bR\Gamma_\rc\Big(\wt\PP^1_{v'}\moins\{|v'|=\infty,\,\arg v'\in[\pi/2,3\pi/2]\bmod2\pi\},\wt\varpi_{\wh\theta'_o}^{-1}\cF[1]\Big).
\]
In the coordinate $t$, let~$\Delta$ be a small open disc centered at~$0$ and let $I_{\wh\theta'_o}$ be the closed interval of $\partial\ov\Delta$ defined by $\arg t\in[\wh\theta'_o+\pi/2,\wh\theta'_o+3\pi/2]$. Then the previous complex is the relative cohomology complex $\bR\Gamma(\ov\Delta,I_{\wh\theta'_o};\cF_{|\ov\Delta}[1])$. This complex has cohomology in degree zero at most and $\HH^0(\ov\Delta,I_{\wh\theta'_o};\cF_{|\ov\Delta}[1])\simeq\phip_t\cF$. When $\wh\theta'_o$ runs counterclockwise once around $S^1_{\wh\infty}$, then $I_{\wh\theta'_o}$ also moves counterclockwise once around $\partial\ov\Delta$. Then $\gr_0(\FcF)$ is a local system whose monodromy is identified with the monodromy on $\phip_t\cF$.
\end{proof}

We now extend this construction in order to define $(\FcF)_{\leq c/\tau'}$ for any $c\in\CC$. Let $Z(c)\to\PP^1\times\wh\PP^1$ be the complex blowing-up at the points $(c,\infty)$ and $(\infty,\wh0)$, and let $\wt Z(c)$ be the corresponding real blow-up space (so that $Z(0)$ and $\wt Z(0)$ are respectively equal to~$Z$ and~$\wt Z$ introduced above). We also define $L''_*(c)$ and $L'_*(c)$, where $*$ is for $<0$ or $\leq0$, as we did for $L''_*$ and $L'_*$, and we denote by $\alpha_c$, etc.\ the corresponding maps.

\begin{definitio}
For any $c\in\CC$, we set
\begin{itemize}
\item
$(\cFcG)_{\leq c/\tau'}=\beta_!\bR(\gamma\circ\alpha)_*\cG$ and $(\cFcG)_{\prec c/\tau'}=(\beta\circ\nobreak\gamma)_!\bR\alpha_*\cG$,
\item
$(\FcF)_{\leq c/\tau'}=\bR\wt{\wh q}_*(\cFcG)_{\leq c/\tau'}$ and $(\FcF)_{\prec c/\tau'}=\bR\wt{\wh q}_*(\cFcG)_{\prec c/\tau'}$.
\end{itemize}
\end{definitio}

Clearly, Lemma \ref{prop:Laplace0} also applies similarly to $(\FcF)_{\leq c/\tau'}$ and $(\FcF)_{\prec c/\tau'}$. It is important to notice that the spaces $\wt Z(c)$ ($c\in\CC$) all coincide when restricted over $\PP^1\times\Afuh$. As a consequence, the restrictions of $(\FcF)_{\leq c/\tau'}$ and $(\FcF)_{\prec c/\tau'}$ to $\Afuh$ are the same perverse sheaf, that we denote by $\FcF$. The previous construction defines thus a $\ccI$\nobreakdash-filtration on~$\FcF$, according to Proposition \ref{prop:Laplace0} applied with any $c\in\CC$. Here, we denote by $\ccI$ the sheaf equal to $0$ on $\Afuh$ and to the constant sheaf with fibre $\CC\cdot(1/\tau')$ on $S^1_{\wh\infty}$.

\begin{proposition}\label{prop:StokesF}
The triangle $[(\FcF)_{\prec c/\tau'}\to(\FcF)_{\leq c/\tau'}\to\gr_{c/\tau}\cF\To{+1}]_{c\in\CC}$ defines an object of $\St_{\wh\infty}(\CC_{\ccIet,\leq})$ (\cf Definition \ref{def:StPervtri}).
\end{proposition}

\begin{proof}
It is enough to argue on Stokes filtrations and we will use the definition given in Proposition \ref{prop:stokeswithout}. As indicated above, it will be enough to take the space $\{c/\tau'\mid c\in\CC\}$ as index set. Let us check the filtration property. Let us fix $\wh\theta'_o\in S^1_{\wh\infty}$. From Figure~\ref{fig:L'}, it is clear that if $c$ belongs to $L'_{<0}$, then $(\FcF)_{\leq c/\tau',\wh\theta'_o}\subset(\FcF)_{\prec0,\wh\theta'_o}$. This condition on~$c$ reads $\reel(ce^{-i\wh\theta'_o})<0$, or equivalently $c/\tau'\mathrel{<_{_{\wh\theta'_o}}}0$. We therefore get a pre-Stokes filtration, with jumps at $c/\tau'$ where $c$ is a singular point of~$\cF$. The same argument shows that $(\FcF)_{\prec c/\tau'}$ is the subsheaf defined from $(\FcF)_{\leq}$ by Formula \eqref{eq:onto<}.

Lastly, the dimension property is obtained by the last part of Lemma \ref{lem:pervers-semidisque}.
\end{proof}

\begin{remarque}\label{rem:blowup}
One can ask whether the sheaves $\FcF_{\leq c/\tau'}$ could be defined without using the blowing-up map~$e$ or not. Recall that this blowing-up was used in order to determine in a clear way whether $e^{t/\tau'}$ (\resp $e^{\tau/t'}$) has moderate growth or not near $t=0$, $\tau'=0$ (\resp $t'=0$, $\tau=0$). As we already indicated in the proof of Proposition \ref{prop:Laplace0}, the blowing-up of $(\infty,\wh0)$ is not needed. We introduced it by symmetry, and (mainly) in order to treat duality later. On the other hand, working over $\Afu\times\wh X$, if the sheaves $\cFcG_{\leq c/\tau'}$ were to be defined without blowing up, they should be equal to $\bR\wt e_*\cFcG_{\leq c/\tau'}$, in order that the definition remains consistent. But one can check that $\bR\wt e_*\cFcG_{\leq c/\tau'}$ are not sheaves, but complexes, hence do not enter in the frame of Stokes filtrations of a local system in two variables.
\end{remarque}

\Subsection{\proofname\ of Theorem \ref{th:RHL} and compatibility with Riemann-Hilbert}\label{subsec:proofthRHL}
Let us first indicate the steps of the proof. We anticipate on the notation and results explained in \Chaptersname\ref{chap:realbl}, which we refer to.
\begin{enumerate}
\item
The first step computes $\DR^{\rmod\wh\infty}(p^+\cM\otimes\cE^{t\tau})$ (a complex living on $\PP^1\times\wh X$) from a complex defined on $\wt Z$. We will use the notation of the commutative diagram \eqref{eq:ZX}. We consider the sheaf $\cA_{\wt Z}^\modD$ on $\wt Z$ of holomorphic functions on $\wt Z\moins\partial\wt Z=Z\moins D$ having moderate growth along $\partial\wt Z$. In particular, $\cA_{\wt Z}^\modD{}_{|\wt Z\moins\wt D}=\cO_{Z\moins D}$. Applying Proposition \ref{prop:Rpimod} of the next \chaptersname to $e:Z\to\PP^1\times\wh\PP^1$, and then its variant to $\varpi\times\id$, gives
\begin{equation}\label{eq:imdirepsilon}
\DR^{\rmod\wh\infty}(p^+\cM\otimes\cE^{t\tau})\simeq\bR\wt e_*\DR^\modD\big(e^+(p^+\cM\otimes\cE^{t\tau})\big).
\end{equation}

\begin{remarque}
In fact, \eqref{eq:imdirepsilon} is a statement similar to Theorem \ref{th:RHL}, but is much easier, in particular because of Proposition \ref{prop:Rpimod}, which would not apply in the setting of Theorem \ref{th:RHL}: indeed, $e$ is a proper modification while $\wh p$ is not.
\end{remarque}

As a consequence we get
\begin{equation}\label{eq:qht}
\bR\wh q_*\DR^{\rmod\wh\infty}(p^+\cM\otimes\cE^{t\tau})\simeq\bR\wt{\wh q}_*\DR^\modD\big(e^+(p^+\cM\otimes\cE^{t\tau})\big).
\end{equation}

\item
The second step consists in comparing $\DR^\modD\big(e^+(p^+\cM\otimes\cE^{t\tau})\big)$ with $\beta_!\bR(\gamma\circ\nobreak\alpha)_*p^{-1}(\cF)[-1]$, where $\cF=\DR^\an M[1]$. Both complexes coincide on $\wt Z\moins\partial\wt Z$. We therefore have a natural morphism
\begin{equation}\label{eq:DRmodtop}
\DR^\modD\big(e^+(p^+\cM\otimes\cE^{t\tau})\big)\to\bR\beta_*\bR(\gamma\circ\alpha)_*p^{-1}(\cF)[-1].
\end{equation}
That it factorizes through $\beta_!$...\ would follow from $\delta^{-1}\DR^\modD\big(e^+(p^+\cM\otimes\nobreak\cE^{t\tau})\big)=\nobreak0$ (which will be proved below), where $\delta$ is the closed inclusion $L''_{\leq0}\hto\wt Z$ complementary to $\beta$ (\cf Diagram \eqref{eq:diagabc}). Therefore, proving

\refstepcounter{enumii}\label{enum:loca}\eqref{enum:loca}
that a morphism $\DR^\modD\big(e^+(p^+\cM\otimes\nobreak\cE^{t\tau})\big)\to\beta_!\bR(\gamma\circ\alpha)_*p^{-1}(\cF)[-1]$ exists

\refstepcounter{enumii}\label{enum:locb}\eqref{enum:locb}
and that it is an isomorphism
\par\noindent
are both local statements on $\partial\wt Z$. We will also anticipate on results proved in \Chaptername~\ref{chap:realbl}.

Let $C\subset\Afu$ be the union of the singular set of $M$ and $\{0\}$. We also set $C'=C\moins\{0\}$. We can reduce both local statements to the case where $M$ is supported on $C$ and the case where $M$ is localized along $C$, that is, for each $c\in C$, the multiplication by $t-c$ is invertible on $M$. The first case will be left as an exercise, and we will only consider the second one.\label{page:localisation}

We note that, when expressed in local coordinates adapted to~$D$, the pull-back $e^+\cE^{t\tau}$ satisfies the assumption in Proposition \ref{prop:HkEphinul} (this is one reason for using the complex blowing-up $e$). Therefore, by a simple inductive argument on the rank, Proposition \ref{prop:HkEphinul} applies to $e^+(p^+\cM\otimes\nobreak\cE^{t\tau})$ away from $(p\circ e)^{-1}(C')$ (where the polar divisor of $e^+(p^+\cM\otimes\nobreak\cE^{t\tau})$ contains other components than those of $D$). Both local statements are then clear on such a set, by using that $\cH^0\DR^\modD\cE^{t\tau}$ vanishes where $e^{-t\tau}$ does not have moderate growth. We are thus reduced to considering the situation above a neighbourhood of a point $c\in C'$.

\begin{figure}[htb]
\begin{center}
\includegraphics[scale=.5]{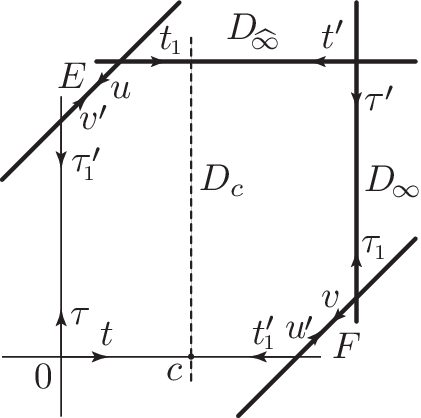}
\end{center}
\end{figure}

We denote by $t_c$ a local coordinate on $\Afu$ centered at $c$ and we set $\tau'=1/\tau$ as above. We will work near the point $(c,\infty)\in Z$ with coordinates $(t_c,\tau')$ (recall that the complex blowing-up $e$ is an isomorphism there, so we identify $Z$ and $\PP^1\times\wh\PP^1$, and we still denote by $Z$ a sufficiently small neighbourhood of $(c,\infty)$). The divisor~$D$ is locally defined by $\tau'=0$, and $p^+\cM\otimes\cE^{t\tau}=\cE^{(c+t_c)/\tau'}\otimes\cR$, where $\cR$ is a free $\cO_Z(*(D\cup D_c))$-module with a regular connection, setting $D_c=\{t_c=0\}$. We denote by $\cL$ the local system on $Z\moins(D\cup D_c)$ determined by $\cR$ (we know that $\cL$ has monodromy equal to identity around $\tau'=0$, but this will not be used in the following). Denoting by $j_c:Z\moins(D\cup D_c)\hto Z\moins D$ the inclusion, we have $p^{-1}\cF=\bR j_{c,*}\cL[1]$. We cannot directly apply Proposition \ref{prop:HkEphinul} as above, because $\cR$ is localized along $D_c$, \ie the multiplication by $t_c$ is invertible on $\cR$, so we will first consider the real blowing-up $\varpi_c:\wt Z_c\to\wt Z$ of $D_c$ in $\wt Z$ (\ie we will work in polar coordinates in both variables~$\tau'$ and~$t_c$).

By Proposition \ref{prop:HkEphinul}, the complex $\DR^{\rmod(D\cup D_c)}(\cE^{(c+t_c)/\tau'}\otimes\cR)$ on $\wt Z_c$ has cohomology in degree $0$ at most, and by (a variant of) Proposition \ref{prop:Rpimod}, we have $\DR^\modD(\cE^{(c+t_c)/\tau'}\otimes\cR)=\bR\varpi_{c,*}\DR^{\rmod(D\cup D_c)}(\cE^{(c+t_c)/\tau'}\otimes\cR)$.

Notice now that the open set in $\wt Z_c$ where $e^{-(c+t_c)/\tau'}$ is exponentially decreasing is nothing but $\varpi^{\prime-1}L'_{\leq0}$ (\cf Diagram \eqref{eq:diagabc}). Therefore, with obvious notation, we have
\[
\beta_!\bR(\gamma\circ\alpha)_*\bR j_{c,*}\cL=\bR\varpi_{c,*}\beta_{c,!}\bR(\gamma_c\circ\alpha_c)_*\cL.
\]
Since both statements \eqref{enum:loca} and \eqref{enum:locb} hold on $\wt Z_c$ by the argument given in the first part of Step two, they hold on $\wt Z$ after applying $\bR\varpi_{c,*}$.

\item
Third step. We conclude from Step two and \eqref{eq:qht} that we have an isomorphism
\[
\bR\wh q_*\DR^{\rmod\wh\infty}(p^+\cM\otimes\cE^{t\tau})[1]\simeq(\FcF)_{\leq0}[-1].
\]
In order to prove the surjectivity of $\cH^0\eqref{eq:DRmod}$ one can restrict to $S^1_{\wh\infty}$, as it holds on $\Afuh$. By Proposition \ref{prop:Laplace0}, we conclude that $\bR\wh q_*\DR^{\rmod\wh\infty}(p^+\cM\otimes\cE^{t\tau})[1]$ has cohomology in degree zero only, and we have already seen that so has the left-hand term of \eqref{eq:DRmod}. We are thus reduced to showing that the fibers of these sheaves have the same dimension at any $\wh\theta\in S^1_{\wh\infty}$. This follows from Proposition \ref{prop:Laplace0} and its extension to any index $c/\tau$, showing that the dimension of the right-hand term of $\cH^0(\eqref{eq:DRmod})_{\wh\theta}$ is the sum of the dimensions of $\phi_{t-c}\cF$, where $c$ varies in the open subset where $0\leqwhtheta c/\tau$. That the dimension of the left-hand term is computed similarly follows from \eqref{enum:b} of \S\ref{subsec:introRHL}.

\item\label{subsec:proofthRHL4}
In conclusion, we have proved Theorem \ref{th:RHL} and, at the same time, the fact that the Stokes filtration of Proposition \ref{prop:StokesF} is the Stokes filtration of $\Fou M$ at infinity, and more precisely that, through the \index{Riemann-Hilbert correspondence!and Laplace}Riemann-Hilbert correspondence given by $\pDR_{\ccIet}$, the image of $\Fou M$ is the triangle of Proposition \ref{prop:StokesF}. In other words, when $M$ has only regular singularities, the Stokes-de~Rham functor (\ie the de~Rham functor enriched with the Stokes filtration at infinity) exchanges the Laplace transform with the topological Laplace transform.\qed
\end{enumerate}

\subsection{Compatibility of Laplace transformation with duality}\index{duality!and Laplace}\index{Laplace transform!and duality}

Let $M$~be a $\Clt$-module of finite type (an object of $\Mod_f(\Clt)$), or more generally an object of $D^\rb_f(\Clt)$ (bounded derived category of $\Clt$-modules with finite-type cohomology). The Laplace transformation considered in \S\ref{subsec:introRHL} is in fact a transformation from $\Mod_f(\Clt)$ (\resp $D^\rb_f(\Clt)$) to $\Mod_f(\Cltau)$ (\resp $D^\rb_f(\Cltau)$), and is not restricted to holonomic objects. Let $\bD$ be the duality functor in these categories. In particular, for a holonomic left $\Clt$-module~$M$, $\bD M$ is the complex having cohomology in degree $0$ only, this cohomology being the holonomic left $\Clt$-module $M^\vee$ naturally associated to the right holonomic $\Clt$-module $\mathrm{Ext}^1_{\Clt}(M,\Clt)$, where the right $\Clt$-module structure comes from the right structure on $\Clt$.

According to \cite{Malgrange91}, there is a canonical isomorphism of functors $\Fou{}\circ\bD(\cbbullet)\isom\iota^+\bD\circ\nobreak\Fou{(\cbbullet)}$, where $\iota$ denotes the involution $\Cltau\isom\Cltau$ sending $\tau$ to $-\tau$ and $\partial_\tau$ to $-\partial_\tau$. This isomorphism can be algebraically described in two ways, depending on the choice of the definition of the Laplace transformation (\cf \S\ref{subsec:introRHL}, \eqref{enum:introRHL1} and \eqref{enum:introRHL2}). Both definitions coincide (\cf \cite[App\ptbl2.4,\,p\ptbl224]{Malgrange91}). Let us recall them.

\begin{enumerate}
\item
Let us choose a resolution $L^\cbbullet$ of $M$ by free left $\Clt$-modules. Then $\Fou L^\cbbullet$ is a resolution of $\Fou M$ by free left $\Cltau$-modules, and clearly, by using such a resolution, $\mathrm{Ext}^1_{\Cltau}(\Fou M,\Cltau)=\iota^+\Fou \mathrm{Ext}^1_{\Clt}(M,\Clt)$ as right $\Cltau$-modules (\cf \cite[Lem\ptbl V.3.6, p\ptbl86]{Malgrange91}, \cite[\S V.2.b]{Bibi00}).\enlargethispage{1.5\baselineskip}%
\item
Let us use the definition $\Fou M=\wh p_+(p^+M\otimes E^{t\tau})$. However, since we will have to use the commutation of direct image with duality, it is necessary to work with proper maps \eqref{eq:ppproper}, and we can work either in the algebraic or the analytic setting. We will choose the latter for future use, and, as in the Introduction of this \chaptername, we denote by $\cM$ the $\cD_{\PP^1}$-module localized at~$\infty$ associated with~$M$. Since $p$ is smooth, we have $p ^+\bD(\cbbullet)=\bD p^+(\cbbullet)$ (\cf \cite[Prop\ptbl VII.9.13]{Borel87}, taking into account the difference in notation). In other words, denoting by $(p^+\cM)^\vee$ the left $\cD_{\PP^1\times\wh\PP^1}$-module associated with $\mathrm{Ext}^2_{\cD_{\PP^1\times\wh\PP^1}}(p^+\cM,\cD_{\PP^1\times\wh\PP^1})$, we have $p^+(\cM^\vee)=(p^+\cM)^\vee$. We now get $p^+(\bD \cM)\otimes \cE^{t\tau}=\bD(p^+\cM)\otimes \cE^{t\tau}=\bD(p^+\cM\otimes \cE^{-t\tau})$. Notice also that $\bD\cM$ may not be localized at $\infty$ (as was $\cM$ by definition), but we have $p^+(\bD \cM)\otimes \cE^{t\tau}=p^+\big((\bD \cM)(*\infty)\big)\otimes \cE^{t\tau}$, since $\cE^{t\tau}$ is localized along $D_\infty$. Therefore, $\wh p_+(p^+(\bD \cM)\otimes \cE^{t\tau})$ is the $\cD_{\wh\PP^1}$-module localized at $\wh\infty$ associated with $\Fou(\bD M)$.

By the relative duality for the proper map $\wh p_+$ (\cf\eg \cite[Prop\ptbl VII.9.6]{Borel87}, \cite{MSaito89b}, \cite[App\ptbl2, Th\ptbl3.3]{Malgrange91}), we have an isomorphism
\[
\wh p_+\bD(p^+\cM\otimes \cE^{-t\tau})\isom\bD\wh p_+(p^+\cM\otimes \cE^{-t\tau}),
\]
from which we get (by applying $\Gamma(\wh\PP^1,\cbbullet)$)
$$
\Fou(\bD M)\isom\iota^+\bD\Fou M.\eqno\qed
$$
\end{enumerate}

The localized module $\CC[\tau,\tau^{-1}]\otimes_{\CC[\tau]}\Fou M$ is a free $\CC[\tau,\tau^{-1}]$-module of finite rank with connection. We denote by $(\cL,\cL_\bbullet)(\Fou M)$ the associated Stokes-filtered local system on~$S^1_{\wh\infty}$.

Applying the Riemann-Hilbert functor of Definition \ref{def:RHfunctorgerms} together with the duality isomorphism of Proposition \ref{prop:RHStokesdual}, we get an isomorphism
\begin{equation}\label{eq:isodualLaplace}
(\cL,\cL_\bbullet)(\Fou\bD M)\isom\iota^{-1}[(\cL,\cL_\bbullet)(\Fou M)]^\vee.
\end{equation}

\Subsection{Compatibility of topological Laplace transformation with Poincaré-Verdier duality}\index{duality!and topological Laplace}

We will define an analogue of the previous isomorphism by working at the topological level. We will compare both constructions in \S\ref{subsec:compFourierdual}. It will be clearer to distinguish between the projection $p:\PP^1\times\wh\PP^1\to\PP^1$ and its restriction over $\Afu$, that we now denote by $p_o:\Afu\times\wh\PP^1\to\Afu$ (or $\Afu\times\Afuh\to\Afu$).

Let $\cF$ be a perverse sheaf of $\kk$-vector spaces on $\Afu$, with singularity set $C$. The pull-back $\cG\defin p_o^{-1}\cF[1]$ on $\Afu\times\Afuh$ is also perverse. Moreover, there is a functorial isomorphism $p_o^{-1}(\bD\cF)[1]\isom(p_o^!\bD\cF)[-1]=\bD(p_o^{-1}\cF[1])$ (\cf\eg\cite[Prop\ptbl3.3.2]{K-S90}) which is compatible with bi-duality.

Let $\iota$ denote the involution $\tau\mto-\tau$ on $\Afuh$. It induces an involution on the spaces $\Afu\times\Afuh$, $\wh X$, $Z$, $X\times\wh X$ and $\wt Z$ in a unique way. We all denote them by $\iota$. We will use the notation $\alpha_\iota$ for the conjugate $\iota\circ\alpha\circ\iota$, etc.

\begin{proposition}\label{prop:dualtopLaplace}
The topological Laplace transformation, from $\kk$-perverse sheaves to Stokes-$\kk$-perverse sheaves, is compatible with duality (up to $\iota$) and bi-duality.
\end{proposition}

Given a $\kk$-perverse sheaf $\cF$ on $\Afu$, we denote by $(\cL,\cL_\bbullet)(\FcF)$ the Stokes-filtered local system $(\FcF_\leq)_{|S^1_{\wh\infty}}$. Then, according to Lemma \ref{lem:StCdual}, the proposition gives an isomorphism\enlargethispage{\baselineskip}%
\begin{equation}\label{eq:isodualtopLaplace}
(\cL,\cL_\bbullet)(\Fou\bD\cF)\isom\iota^{-1}[(\cL,\cL_\bbullet)(\FcF)]^\vee,
\end{equation}
where $(\cL,\cL_\bbullet)^\vee$ is defined by Proposition \ref{prop:operationsSt}\eqref{prop:operationsSt2}.

\begin{proof}[\proofname\ of Proposition \ref{prop:dualtopLaplace}]
We first wish to prove the existence of isomorphisms
\begin{equation}\label{eq:DF}
\Fou(\bD\cF)_{\prec0}\simeq\iota^{-1}\bD(\FcF_{\leq0}),\quad\Fou(\bD\cF)_{\leq0}\simeq\iota^{-1}\bD(\FcF_{\prec0})
\end{equation}
which are exchanged by duality up to bi-duality isomorphisms. Note that we could also write $\iota^{-1}\Fou(\bD\cF)_{\prec0}$ as ${}^{\ov{\mathrm F}}\!(\bD\cF)_{\prec0}$ etc., where $\ov{\mathrm F}$ denotes the inverse Laplace transformation, which has kernel $e^{-t\tau}$.

These isomorphisms are obtained by applying to similar isomorphisms on $\wt Z$ the direct image $\bR\wt{\wh q}_*$, which is known to commute with~$\bD$ in a way compatible with bi-duality since $\wt{\wh q}$ is proper, according to Poincaré-Verdier duality (\cf\eg\cite[Prop\ptbl3.1.10]{K-S90}).

We will thus first work locally on $\wt Z$. Let $\cG$ be a $\kk$-perverse sheaf on $\Afu\times\Afuh$ with singular set $S=p_o^{-1}(C)$. By Poincaré-Verdier duality, we have $\bD(\beta_!\bR(\gamma\circ\alpha)_*\cG)=\bR\beta_*(\gamma\circ\alpha)_!\bD\cG$ in a way compatible with bi-duality (\cf \eqref{eq:diagabc} for the notation). We denote by $\ov\delta$ the closed inclusion $L''_{<0}\hto\wt Z$ and similarly $\ov\delta_\iota:\iota(L''_{<0})\hto\wt Z$.

Similarly, $\bD((\beta\circ\gamma)_!\bR\alpha_*\cG)=\bR(\beta\circ\gamma)_*\alpha_!\bD\cG$ in a way compatible with bi-duality, and we denote by $\delta$ the closed inclusion $L''_{\leq0}\hto\wt Z$ and similarly $\delta_\iota:\iota(L''_{\leq0})\hto\wt Z$.

Both $\cG$ and $\bD\cG$ have the same singular set, and we will later apply the following lemma to $\bD\cG$.

\begin{lemme}\label{lem:dualtopLaplace}
If $\cG$ satisfies the previous assumptions, we have
\begin{itemize}
\item
$\ov\delta_\iota^{-1}\bR\beta_*(\gamma\circ\nobreak\alpha)_!\cG=0$,
\item
$\delta_\iota^{-1}\bR(\beta\circ\gamma)_*\alpha_!\cG=0$.
\end{itemize}
\end{lemme}

\begin{proof}
We start with the first equality. By definition, the cohomology of $\bR\beta_*(\gamma\circ\nobreak\alpha)_!\cG$ is zero on $L'_{\leq0}$, so it is enough to prove that the pull-back of $\bR\beta_*(\gamma\circ\nobreak\alpha)_!\cG$ to $\iota(L''_{<0})\cap L''_{\leq0}$ is zero. Notice that this set does not cut $\wt E\cup\wt F$ by definition of $L''_{\leq0}$. Notice also that, now that we are far from $E$ and $F$, there is no difference between $L''_{<0}$ and $L''_{\leq0}$, so we will forget~$\gamma$ in the notation. We will distinguish three cases:
\begin{enumerate}
\item
Smooth points of $D_\infty\cup D_{\wh\infty}$ not in $\ov S$ (closure of $S$).
\item
The point $(\infty,\wh\infty)$.
\item
The points of $(D_\infty\cup D_{\wh\infty})\cap\ov S$ (they are smooth points of $D_{\wh\infty}$ by assumption).
\end{enumerate}

\begin{enumerate}
\item\label{enum:cassmooth}
Near such a point, $\iota(L''_{\leq0})\cap L''_{\leq0}$ is topologically a product of a similar dimension-one intersection with a disc and $\cG$ is a local system (up to a shift). We will therefore treat the dimension-one analogue. We consider a local system $\cL$ on $(0,\epsilon)\times S^1$, where~$S^1$ has coordinate $e^{i\theta}$, we set $L''_{\leq0}=\{(0,\theta)\mid\theta\in[\pi/2,3\pi/2]\bmod2\pi\}$ and $\iota(L''_{\leq0})=L''_{\leq0}+\pi$. Moreover, $\alpha:(0,\epsilon)\times S^1\hto(0,\epsilon)\times S^1\cup L'_{\leq0}$ and $\beta:(0,\epsilon)\times S^1\cup L'_{\leq0}\hto[0,\epsilon)\times S^1$ are the inclusions. We wish to show that $\bR\beta_*\alpha_!\cL$ is zero at the points with coordinates $(0,\pi/2)$ and $(0,3\pi/2)$. The cohomology of the germ of $\bR\beta_*\alpha_!\cL$ at $(0,\pi/2)$ (say) is the cohomology of an open disc with an open interval $I$ added on its boundary, with coefficient in a sheaf which is constant on the open disc and zero on this open interval. This is also the relative cohomology $H^*(\Delta\cup I,I;\kk^d)$ ($d=\rk\cL$). So this is clearly zero.\enlargethispage{2\baselineskip}%

\item
The pull-back in $\wt Z$ of a neighbourhood of $(\infty,\ov\infty)$ takes the form $[0,\epsilon)^2\times S^1_\infty\times S^1_{\wh\infty}$, with coordinates $(\theta',\wh\theta')$ on $S^1_\infty\times S^1_{\wh\infty}$. The set $L''_{\leq0}$ is defined by $\theta'+\wh\theta'\in[\pi/2,3\pi/2]\bmod2\pi$ and $\iota(L''_{\leq0})$ by $\theta'+\wh\theta'\in[-\pi/2,\pi/2]\bmod2\pi$ (since $\iota$ consists in changing $\wh\theta'$ to $\wh\theta'+\pi$). On the other hand, $\cG$ is a local system (up to a shift) on $(0,\epsilon)^2\times S^1_\infty\times S^1_{\wh\infty}$. Up to taking new coordinates $(\theta'+\wh\theta',\theta'-\wh\theta')$, one is reduced to the same computation as in \eqref{enum:cassmooth}, up to the cartesian product by intervals, which has no effect on the result.

\item
As in the proof on Page \pageref{page:localisation}, it is enough to consider the case where $\cG$ is supported in $\ov S$, which is treated as in \eqref{enum:cassmooth}, and the case where $\cG$ is the maximal extension of a local system (up to a shift) away from $D_c$ ($c\in C'$). In this case, the situation is a product of that considered in \eqref{enum:cassmooth} and that of a local system on an open punctured disc $\Delta^*$. The argument of \eqref{enum:cassmooth} applies here also.
\end{enumerate}

Let us now consider the second equality, for which we argue similarly. The cohomology of $\bR(\beta\circ\gamma)_*\alpha_!\cG$ is zero on $L'_{<0}$, so it is enough to prove that the pull-back of $\bR(\beta\circ\gamma)_*\alpha_!\cG$ to $\iota(L''_{\leq0})\cap L''_{<0}$ is zero. This set does not cut $\wt E\cup\wt F$, so $\gamma=\id$ on this set, and we are reduced to the previous computation.
\end{proof}

\subsubsection*{End of the proof of Proposition \ref{prop:dualtopLaplace}}

Let $\cG$ be as above. As in Definition \ref{def:cFcG}, we set $\cFcG_{\leq0}=\beta_!\bR(\gamma\circ\nobreak\alpha)_*\cG$ and $\cFcG_{\prec0}=(\beta\circ\gamma)_!\bR\alpha_*\cG$. We therefore get two functors $\cFou(\cbbullet)_{\leq0}$ and $\cFou(\cbbullet)_{\prec0}$. We will show that they are compatible with $\iota$-twisted duality.

\begin{lemme}\label{lem:isodualG}
There exist unique isomorphisms
\[
\cFou(\iota^{-1}\bD\cG)_{\prec0}\To{\lambda_{\leq0}}\iota^{-1}\bD(\cFcG_{\leq0}),\quad
\cFou(\iota^{-1}\bD\cG)_{\leq0}\To{\lambda_{\prec0}}\iota^{-1}\bD(\cFcG_{\prec0})
\]
which extend the identity on $\Afu\times\Afuh$. They are functorial, and fit into a bi-duality commutative diagram, and a simlar one by exchanging $\leq0$ and $\prec0$ (and we use the identity $\iota^{-1}\bD=\bD\iota^{-1}$):
\begin{equation}\label{eq:diagbidual}
\begin{array}{c}
\xymatrix@C=2cm{
\iota^{-1}\bD\,\cFou(\iota^{-1}\bD\cbbullet)_{\leq0}&\cFou(\bD\bD\cbbullet)_{\prec0}\ar[l]_-{\lambda_{\leq0}\cdot\iota^{-1}\bD}^-\sim\\
\bD\big[\bD\big(\cFou(\cbbullet)_{\prec0}\big)\big]\ar[u]^{\iota^{-1}\bD(\lambda_{\prec0})}_\wr\\
\cFou(\cbbullet)_{\prec0}\ar[u]_\wr^{\bid\cdot \cFou(\cbbullet)_{\prec0}}\ar@{=}[r]&\cFou(\cbbullet)_{\prec0}\ar[uu]_-{\cFou(\bid)_{\prec0}}^-\wr
}
\end{array}
\end{equation}
\end{lemme}

\begin{proof}
We will use the cartesian square of open inclusions:
\[
\xymatrix@C=2cm{
\iota(L'_{<0})\ar@{^{ (}->}[r]^-{\beta_\iota\circ\gamma_\iota}&\wt Z\\
\Afu\times\Afuh=\iota(L'_{<0})\cap L'_{\leq0}\ar@{^{ (}->}[u]^{\alpha_\iota}\ar@{^{ (}->}[r]^-{\gamma\circ\alpha}&L'_{\leq0}\ar@{_{ (}->}[u]_\beta
}
\]
which implies (\cf \cite[Prop\ptbl2.5.11]{K-S90}) $(\gamma\circ\alpha)_!\alpha_\iota^{-1}=\beta^{-1}(\beta_\iota\circ\gamma_\iota)_!$ and $\alpha_{\iota,!}(\gamma\circ\alpha)^{-1}=(\beta_\iota\circ\gamma_\iota)^{-1}\beta_!$, and, by openness, $\bR(\gamma\circ\alpha)_*\alpha_\iota^{-1}=\beta^{-1}\bR(\beta_\iota\circ\gamma_\iota)_*$ and $(\beta_\iota\circ\gamma_\iota)^{-1}\bR\beta_*=\bR\alpha_{\iota,*}(\gamma\circ\alpha)^{-1}$.

For any $\cG$ as above, the previous remark and the first assertion of Lemma \ref{lem:dualtopLaplace} imply that the following natural adjunction morphism
\[
(\beta_\iota\circ\gamma_\iota)_!\bR\alpha_{\iota,*}\cG=(\beta_\iota\circ\gamma_\iota)_!(\beta_\iota\circ\gamma_\iota)^{-1}\bR\beta_*(\gamma\circ\alpha)_!\cG\to\bR\beta_*(\gamma\circ\alpha)_!\cG
\]
is an isomorphism. We have by adjunction (\cf \cite[(2.6.14) \& Th\ptbl3.15]{K-S90}):
\begin{align*}
\Hom((\beta_\iota\circ\gamma_\iota)_!\bR\alpha_{\iota,*}\cG,\bR\beta_*(\gamma\circ\alpha)_!\cG)&=\Hom(\beta^{-1}(\beta_\iota\circ\gamma_\iota)_!\bR\alpha_{\iota,*}\cG,(\gamma\circ\alpha)_!\cG)\\
&=\Hom((\gamma\circ\alpha)_!\cG,(\gamma\circ\alpha)_!\cG)\\
&=\Hom(\cG,\cG),
\end{align*}
where the last equality follows from $(\gamma\circ\alpha)^!=(\gamma\circ\alpha)^{-1}$ (since $\gamma\circ\alpha$ is open). Therefore, we can also express the previous adjunction morphism as the adjunction
\[
(\beta_\iota\circ\gamma_\iota)_!\bR\alpha_{\iota,*}\cG\to\bR\beta_*\beta^{-1}(\beta_\iota\circ\gamma_\iota)_!\bR\alpha_{\iota,*}\cG=\bR\beta_*(\gamma\circ\alpha)_!\cG.
\]
Note that it is directly seen to be an isomorphism by dualizing the second assertion of Lemma \ref{lem:dualtopLaplace}.

Applying these isomorphisms to $\bD\cG$ instead of $\cG$, the $D^\rb(\wt Z,\kk)$-isomorphism
\[
(\beta_\iota\circ\gamma_\iota)_!\bR\alpha_{\iota,*}\bD\cG\isom\bR\beta_*(\gamma\circ\alpha)_!\bD\cG=\bD(\beta_!\bR(\gamma\circ\alpha)_*\cG)
\]
that we deduce is by definition $\lambda_{\leq0}$. Uniqueness follows from the identity of $\Hom$ above. Now, $\lambda_{\prec0}$ is defined in order that \eqref{eq:diagbidual} commutes. It can be defined in a way similar to~$\lambda_{\leq0}$ by using the two other possible adjunction morphisms, by the same uniqueness argument.
\end{proof}

The proof of \eqref{eq:DF} now follows: on the one hand, $\bD(p_o^{-1}\cF[1])\simeq p_o^{-1}(\bD\cF)[1]$ since $p_o$ is smooth of real relative dimension two, and, on the other hand, Poincaré-Verdier duality can be applied to $\bR\wt{\wh q}_*$ since $\wt{\wh q}_*$ is proper; one concludes by noticing that $\iota^{-1}p_o^{-1}\cF=p_o^{-1}\cF$, since $p_o\circ\iota=p_o$.

The proof of the proposition is obtained by applying the same reasoning to each $c\in\CC$ and $\FcF_{\leq c/\tau'},\FcF_{\prec c/\tau'}$.
\end{proof}

\subsection{Comparison of both duality isomorphisms}\label{subsec:compFourierdual}

The purpose of this section is to show that, through the \index{Riemann-Hilbert correspondence!duality and Laplace}Riemann-Hilbert correspondence $\cF=\pDR^\an M$, the duality isomorphisms \eqref{eq:isodualLaplace} and \eqref{eq:isodualtopLaplace} correspond each other. We will neglect questions of signs, so the correspondence we prove has to be understood ``up to sign''.

Let us make this more precise. By the compatibility of the Riemann-Hilbert correspondence with Laplace and topological Laplace transformations, shown in \S\ref{subsec:proofthRHL}\eqref{subsec:proofthRHL4}, we have a natural isomorphism
\begin{equation}\label{eq:LFDR}
(\cL,\cL_\bbullet)(\Fou\pDR^\an M)\simeq(\cL,\cL_\bbullet)(\Fou M).
\end{equation}
By using the local duality theorem $\bD\,\pDR^\an M\simeq\pDR^\an\bD M$ (\cf\cite{Narvaez04} and the references given therein), \eqref{eq:isodualLaplace} gives rise to an isomorphism
\begin{equation}\label{eq:isodualLaplacetop}
(\cL,\cL_\bbullet)(\Fou\bD\pDR^\an M)\simeq\iota^{-1}[(\cL,\cL_\bbullet)(\Fou\pDR^\an M)]^\vee,
\end{equation}
as the composition
\begin{align*}
(\cL,\cL_\bbullet)(\Fou\bD\pDR^\an M)&\simeq(\cL,\cL_\bbullet)(\Fou\pDR^\an\bD M)\\
&\simeq(\cL,\cL_\bbullet)(\Fou \bD M)&&\text{by \eqref{eq:LFDR} for }\bD M,\\
&\simeq\iota^{-1}[(\cL,\cL_\bbullet)(\Fou M)]^\vee&&\text{by \eqref{eq:isodualLaplace}},\\
&\simeq\iota^{-1}[(\cL,\cL_\bbullet)(\Fou\pDR^\an M)]^\vee&&\text{by \eqref{eq:LFDR}}.
\end{align*}
Since both source and target of \eqref{eq:isodualtopLaplace} and \eqref{eq:isodualLaplacetop} are identical, proving that both isomorphisms coincide amounts to proving that their restriction to the corresponding local systems coincide, through the previously chosen identifications. These identifications will be easier to follow if we restrict to the local systems, that is, to $\tau\in\CC^*$.

From now on, the notation in the diagram \eqref{eq:ppproper} will be understood with $\CC^*_\tau$ replacing $\wh\PP^1$, and we will also consider the restriction $p_o$ of $p$ to $\Afu\times\CC^*_\tau$. We denote by~$\cM$ the $\cD_{\PP^1}$-module localized at $\infty$ corresponding to $M$ and we set $\cF=\pDR^\an M$. With this understood, we have
\[
(\Fou M)^\an=\wh p_+(p^+\cM\otimes\cE^{t\tau}),\quad\FcF=\bR\wh p_*\bR\varpi_*(\beta_!\bR\alpha_*p_o^{-1}\cF).
\]

We now consider the functorial isomorphism
\begin{equation}\label{eq:dualisoD}
\bD(p^+\cM\otimes\cE^{t\tau})\simeq\bD(p^+\cM)\otimes\cE^{-t\tau}\simeq p^+(\bD\cM)\otimes\cE^{-t\tau}.
\end{equation}
Let us denote by $M^\an$ the restriction of $\cM$ to $\Afu$. We have a natural isomorphism of $\cO^\an_{\Afu\times\Afuh}$-modules with connection: $(\cO,d)\isom(\cE^{t\tau},\nabla)$, sending~$1$ to $e^{-t\tau}\cdot 1$. Up to this isomorphism, and its dual, the restriction of \eqref{eq:dualisoD} to $\Afu\times\CC^*_\tau$ is the corresponding isomorphism
\[
\bD(p_o^+M^\an)\simeq p_o^+(\bD M^\an).
\]
On the other hand, the isomorphism associated to \eqref{eq:dualisoD}, obtained by using the local duality theorem:
\begin{equation}\label{eq:dualisoDR}
\bD\,\pDR(p^+\cM\otimes\cE^{t\tau})\isom\pDR\bD(p^+\cM\otimes\cE^{t\tau})\isom\pDR(p^+(\bD\cM)\otimes\cE^{-t\tau})
\end{equation}
restricts similarly to
\begin{equation}\label{eq:dualisoDRrestr}
\bD\pDR(p_o^+M^\an)\isom\pDR\bD(p_o^+M^\an)\simeq\pDR(p_o^+(\bD M^\an)).
\end{equation}

Recall now that \eqref{eq:DRmodtop} (restricted over $\CC^*_\tau$) induces, by applying $\bR\varpi_*$, an isomorphism
\[
\pDR(p^+\cM\otimes\cE^{t\tau})\isom\bR\varpi_*(\beta_!\bR\alpha_*p_o^{-1}\cF).
\]
Hence \eqref{eq:dualisoDR} gives rise to an isomorphism
\begin{equation}\label{eq:dualisoDRtop}
\bD\bR\varpi_*(\beta_!\bR\alpha_*p_o^{-1}\cF[1])\isom\iota^{-1}\bR\varpi_*(\beta_!\bR\alpha_*p_o^{-1}(\bD\cF)[1]),
\end{equation}
by using the local duality $\bD\,\pDR M^\an\simeq\pDR\bD M^\an$.

On the other hand, the isomorphisms of Lemma \ref{lem:isodualG} applied with $\cG=p_o^{-1}\cF[1]$ also give an isomorphism
\begin{equation}\label{eq:dualisotop}
\bD\,\bR\varpi_*(\beta_!\bR\alpha_*p_o^{-1}\cF[1])\isom\iota^{-1}\bR\varpi_*(\beta_!\bR\alpha_*p_o^{-1}(\bD\cF)[1]),
\end{equation}
since here $\gamma=\id$ and there is no distinction between $\cFcG_\leq0$ and $\cFcG_{\prec0}$.

\begin{lemme}\label{lem:dualisoDRtop}
The isomorphisms \eqref{eq:dualisoDRtop} and \eqref{eq:dualisotop} coincide (up to a nonzero constant).
\end{lemme}

We first start by proving:

\begin{lemme}\label{lem:dualisorestr}
The isomorphisms \eqref{eq:dualisoDRtop} and \eqref{eq:dualisotop} are completely determined by their restriction to $\Afu\times\CC^*_\tau$.
\end{lemme}

\begin{proof}
We will use the following notation: $\Delta_\infty$ is a disc in $\PP^1$ centered at $\infty$, $\varpi:\wt\Delta_\infty\to\Delta_\infty$ is the oriented real blowing-up, $\Delta_\infty^*$ is the punctured disc. The open subset $L'_{\leq0}\subset\partial\Delta_\infty\times\CC^*_\tau$ is defined as in \S\ref{subsec:topspaces}. In particular, the fiber of $\varpi:L'_{\leq0}\to\{\infty\}\times\CC^*_\tau$ is an open half-circle in $\varpi^{-1}(\infty,\tau)$. As above, we denote by $\wtj=\beta\circ\alpha$ the corresponding decomposition of the inclusion $\wtj:\Delta_\infty^*\times\CC^*_\tau\hto\wt\Delta_\infty\times\CC^*_\tau$, and we still denote by $i_\partial$ the inclusion $\partial\Delta_\infty\times\CC^*_\tau\hto\Delta_\infty\times\CC^*_\tau$. Let $\cG$ be a locally constant sheaf on $\Delta_\infty^*\times\CC^*_\tau$.

\begin{lemme}\label{lem:automrestr}
With this notation, any automorphism of $\bR\varpi_*\beta_!\bR\alpha_*\cG$ is uniquely determined from its restriction to $\Delta_\infty^*\times\CC^*_\tau$.
\end{lemme}

\begin{proof}
We notice that $\bR\varpi_*\beta_!\bR\alpha_*\cG$ is a sheaf, equal to $\varpi_*\beta_!\alpha_*\cG$: indeed, this is clearly so for $\beta_!\alpha_*\cG$, whose restriction to $\partial\Delta_\infty\times\CC^*_\tau$ is a local system on $L'_{\leq0}$ extended by $0$; since $\varpi$ is proper, it is enough to check that the push-forward of the latter sheaf is also a sheaf, which is clear.

By adjunction, we have
\begin{align*}
\Hom(\bR\varpi_*\beta_!\bR\alpha_*\cG,\bR\varpi_*\beta_!\bR\alpha_*\cG)&=\Hom(\varpi^{-1}\bR\varpi_*\beta_!\bR\alpha_*\cG,\beta_!\bR\alpha_*\cG)\\
&=\Hom(\varpi^{-1}\varpi_*\beta_!\alpha_*\cG,\beta_!\alpha_*\cG).
\end{align*}
On the one hand, $i_\partial^{-1}\varpi^{-1}\varpi_*\beta_!\alpha_*\cG$ is a local system $\cH'$ on $\partial\Delta_\infty\times\CC^*_\tau$, and on the other hand $i_\partial^{-1}\beta_!\alpha_*\cG$ is a local system on $L'_{\leq0}$ extended by $0$, and is a subsheaf of the local system $\cH\defin i_\partial^{-1}\wtj_*\cG$. If $\lambda$ belongs to $\Hom(\varpi^{-1}\varpi_*\beta_!\alpha_*\cG,\beta_!\alpha_*\cG)$, then $i_\partial^{-1}\lambda$ induces a morphism $\cH'\to\cH$ which vanishes on $L''_{\leq0}$, hence is zero. Therefore, $\lambda$ is uniquely determined by its restriction to $\Delta_\infty^*\times\CC^*_\tau$.
\end{proof}

In order to end the proof of Lemma \ref{lem:dualisorestr}, we notice that $\bR\varpi_*(\beta_!\bR\alpha_*p_o^{-1}\cF[1])$ is of the form considered in Lemma \ref{lem:automrestr} in the neighbourhood of $D_\infty$.
\end{proof}

\begin{proof}[\proofname\ of Lemma \ref{lem:dualisoDRtop}]
According to Lemma \ref{lem:dualisorestr}, we are reduced to proving the coincidence on $\Afu\times\CC^*_\tau$. This is given by Lemma \ref{lem:dualDRiminv} below.
\end{proof}

\begin{lemme}\label{lem:dualDRiminv}
The functorial duality isomorphism $\bD\circ p_o^+(\cbbullet)\isom p_o^+\circ\bD(\cbbullet)$, in the category of analytic holonomic modules, is compatible (up to a nonzero constant), via the de~Rham functor, to the functorial isomorphism $\bD\circ p_o^{-1}(\cbbullet)\isom p_o^{-1}\circ\bD(\cbbullet)[2]$ in the category of perverse sheaves.
\end{lemme}

\begin{proof}
See \cite[Cor\ptbl5.6.8]{M-S02}.
\end{proof}

\begin{proposition}
The isomorphisms \eqref{eq:isodualtopLaplace} and \eqref{eq:isodualLaplacetop} coincide (up to a nonzero constant).
\end{proposition}

\begin{proof}
According to the relative duality theorem already mentioned, the isomorphism \eqref{eq:isodualLaplacetop} is also obtained first by applying $\wh p_+$ to \eqref{eq:dualisoD} and then applying $\pDR$. By \cite[Th\ptbl3.13]{MSaito89b}, we can first apply $\pDR$ to \eqref{eq:dualisoD} and then $\bR\wh p_*$. It is thus obtained by applying $\bR\wh p_*$ to \eqref{eq:dualisoDR} or equivalently to \eqref{eq:dualisoDRtop}, and then by using Verdier duality for $\bR\wh p_*$ (\ie commuting $\bD$ and $\bR\wh p_*$).

On the other hand, \eqref{eq:isodualtopLaplace} is obtained by applying $\bR\wh p_*$ to \eqref{eq:dualisotop} and then by using Verdier duality for $\bR\wh p_*$. The conclusion follows then from Lemma \ref{lem:dualisoDRtop}.
\end{proof}

\part{Dimension two and more}
\chapterspace{-3}
\chapter{Real blow-up spaces and moderate~de~Rham~complexes}\label{chap:realbl}

\begin{sommaire}
The purpose of this \chaptersname is to give a global construction of the real blow-up space of a complex manifold along a family of divisors. On this space is defined the sheaf of holomorphic functions with moderate growth, whose basic properties are analyzed. The moderate de~Rham complex of a meromorphic connection is introduced, and its behaviour under the direct image by a proper modification is explained. This \chaptersname ends with an example of a moderate de~Rham complex having cohomology in degree $\geq1$, making a possible definition of Stokes-perverse sheaves more complicated than in dimension one.
\end{sommaire}

\subsection{Introduction}
Given a meromorphic connection on a complex manifold $X$ with poles along a divisor $D$, the asymptotic analysis of the solutions of the corresponding differential equation, \ie the horizontal sections of the connection, in the neighbourhood of $D$ leads us to introduce a space taking into account the multi-sectors where an asymptotic expansion can be looked for. We introduce in this \chaptersname the real blow-up spaces that we will encounter later in this book. Together with these spaces come various extensions of the sheaf of holomorphic functions on $X\moins D$. We analyze their relations with respect to complex blowing-ups, since such blowing-up maps will be an essential tool for simplifying the formal normal form of a meromorphic connection (\cf \Chaptersname\ref{chap:goodformal}).

These constructions extend, in the many-variable case, those already introduced in \Chaptername s \ref{chap:Stokesone-pervers} and \ref{chap:RH}.

We conclude this \chaptersname with an example showing a new phenomenon in dimension $\geq2$. This example will lead to the introduction of the ``goodness'' property in \Chaptersname \ref{chap:StokesfilteredNCD}.

\subsection{Real blow-up}\label{subsec:realblowup}
Recall that the \emphb{real blow-up} space $\wt\CC^\ell$ of $\CC^\ell$ along $t_1,\dots,t_\ell$ is the space of polar coordinates in each variable $t_j$, that is, the product $(S^1\times\RR_+)^\ell$ with coordinates $(e^{i\theta_j},\rho_j)_{j=1,\dots,\ell}$ and $t_j=\rho_je^{i\theta_j}$. The \emphb{real oriented blowing-up} map $\varpi:\wt\CC^\ell\to\CC^\ell$ induces a diffeomorphism $\{\rho_1\cdots\rho_\ell\neq0\}=:(\wt\CC^\ell)^*\isom(\CC^\ell)^*\defin\{t_1\cdots t_\ell\neq0\}$.

\subsubsection*{Real blow-up along a divisor}
Let~$X$ be a reduced complex analytic space (\eg a complex manifold) and let $f:X\to\CC$ be a holomorphic function on~$X$ with zero set $X_0=X_0(f)$. The oriented \index{real blow-up!along a divisor}real blow-up space of~$X$ along $f$, denoted by $\wt X(f)$, is the closure in $X\times S^1$ of the graph of the map $f/|f|:X^*=X\moins X_0\to S^1$. The real blowing-up map $\varpi:\wt X\to X$ is the map induced by the first projection. The inverse image $\varpi^{-1}(X_0)$, that we denote by $\partial\wt X$, is a~priori contained in $X_0\times S^1$.\enlargethispage{\baselineskip}%

\begin{lemme}\label{lem:eclatereel}
We have $\partial\wt X=X_0\times S^1$.
\end{lemme}

\begin{proof}
This is a local question on $X_0$. As $f$ is open, for $x_o\in X_0$ there exists a fundamental system $(U_m)_{m\in\NN}$ of open neighbourhoods of~$x_o$ and a decreasing family~$\Delta_m$ of open discs centered at~$0$ in $\CC$ such that $f:U_m\to\Delta_m$ is onto, as well as $f:U_m^*=U_m\moins X_0\to\Delta_m\moins\{0\}$. It follows that, given any $e^{i\theta_o}\in S^1$, there exists $x_m\in U_m^*$ with $f(x_m)/|f(x_m)|=e^{i\theta_o}$, so $(x_o,e^{i\theta_o})\in\wt X$.
\end{proof}

As a consequence, $\wt X$ is equal to the subset of $X\times S^1$ defined by the (in)equation $fe^{-i\theta}\in\RR_+$, so is a real semi-analytic subset of $X\times S^1$.

Let now~$D$ be a locally principal divisor in~$X$ and let $(U_\alpha)_{\alpha\in A}$ be a locally finite covering of~$X$ by open sets $U_\alpha$ such that in each $U_\alpha$, the divisor~$D$ is defined by a holomorphic function $f^{(\alpha)}$. The data $[U_\alpha,f^{(\alpha)}]_{\alpha\in A}$ allow one to define, by gluing the real blow-up spaces $\wt U_\alpha(f^{(\alpha)})$, a space \index{$XWTD$@$\wt X(D)$}$\wt X(D)$. Set $f^{(\alpha)}=u^{(\alpha,\beta)}f^{(\beta)}$ on $U_\alpha\cap U_\beta$. The gluing map is induced by
\begin{align*}
(U_\alpha\cap U_\beta)\times(\CC^*/\RR_+^*)&\to (U_\alpha\cap U_\beta)\times(\CC^*/\RR_+^*)\\
\big(x,(e^{i\theta})\big)&\mto\big(x,(u^{(\alpha,\beta)}e^{i\theta}\bmod\RR_+^*)\big).
\end{align*}
One checks that the space $\wt X(D)$ does not depend on the choices made (up to a unique homeomorphism compatible with the projection to~$X$).

In a more intrinsic way, let $L(D)$ be the rank-one bundle over~$X$ associated with~$D$ (with associated sheaf $\cO_X(D)$) and let $S^1L(D)$ be the corresponding~$S^1$-bundle. Let us fix a section $f:\cO_X\to\cO_X(D)$. It vanishes exactly along~$D$ and induces a holomorphic map $X^*\defin X\moins D\to L(D)\moins D$, that we compose with the projection $L(D)\moins D\to S^1L(D)$. Then $\wt X(D)$ is the closure in $S^1L(D)$ of the image of~$X^*$ by this map. From Lemma \ref{lem:eclatereel} we deduce that $\partial\wt X(D)=S^1L(D)_{|D}$. If $g=u\cdot f$ with $u\in\Gamma(X,\cO_X^*)$, then both constructions give homeomorphic blow-up spaces. More precisely, denoting by $\varpi_f:\wt X(f)\to X$ the real blowing-up map obtained with the section~$f$, the multiplication by $u$ induces the multiplication by $u/|u|$ on $S^1L(D)$, which sends the subspace $\wt X(f)$ to $\wt X(g)$. Moreover, this is the unique homeomorphism $\wt X(f)\isom\wt X(g)$ making the following diagram
\[
\xymatrix{
\wt X(f)\ar[r]^-\sim\ar[d]_{\varpi_f}&\wt X(g)\ar[d]^{\varpi_g}\\
X\ar@{=}[r]&X
}
\]
commute. This explains the notation and terminology for \emph{the} real blow-up space of~$X$ along~$D$.

\subsubsection*{Real blow-up along a family of divisors}
Let now $(D_j)_{j\in J}$ be a locally finite family of locally principal divisors in~$X$ and let $f_j$ be sections $\cO_X\to\cO_X(D_j)$. The fibre product over~$X$ of the $\wt X(D_j)$ (each defined with $f_j$), when restricted over $X^*_J\defin X\moins\bigcup_jD_j$, is isomorphic to~$X^*_J$. We then define the \index{real blow-up!along a family of divisors}real blow-up \index{$XWTDZJ$@$\wt X(D_{j\in J})$}$\wt X(D_{j\in J})$ as the closure of~$X^*_J$ in this fibre product. If~$J$ is finite, $\wt X(D_{j\in J})$ is the closure in the direct sum bundle $\bigoplus_jS^1L(D_j)$ of the image of the section $(f_j/|f_j|)_{j\in J}$ on~$X^*$. It is defined up to unique homeomorphism compatible with the projection to $X$. We usually regard $X_J^*$ as an open analytic submanifold in $\wt X(D_{j\in J})$.

The closure $\wt X(D_{j\in J})$ of~$X^*$ can be strictly smaller than the fibre product of the $\wt X(D_j)$ (\eg consider $\wt X(D,D)$ for twice the same divisor, and more generally when the $D_j$ have common components). Here is an example when both are the same.

\begin{lemme}\label{lem:ncdrealblup}
Assume that~$X$ and each $D_j$ is smooth and that the family $(D_j)_{j\in J}$ defines a normal crossing divisor $D=\bigcup_jD_j$ in~$X$. Then $\wt X(D_{j\in J})$ is equal to the fibre product (over~$X$) of the $\wt X(D_j)$ for $j\in J$ and we have a natural proper surjective map $\wt X(D_{j\in J})\to \wt X(D)$.
\end{lemme}

\begin{proof}
The first assertion is checked locally. For instance, in the case of two divisors crossing normally, we are reduced to checking that
\[
\big[(S^1\times\RR_+)\times\CC\big]\mathop{\times}\limits_{\CC\times\CC}\big[\CC\times(S^1\times\RR_+)\big]\simeq(S^1\times\RR_+)\times(S^1\times\RR_+).\qedhere
\]
\end{proof}

\begin{corollaire}\label{cor:jncd}
Under the assumptions of Lemma \ref{lem:ncdrealblup}, if $\cF^*$ is a local system on~$X^*$ and $\wtj:X^*\hto\wt X(D_{j\in J})$ denotes the open inclusion, then $\bR\wtj_*\cF^*=\wtj_*\cF^*$ is a local system on $\wt X$.
\end{corollaire}

\begin{proof}
The question is local and, according to the lemma, we can locally regard~$\wtj$ as being the inclusion $(\RR_+^*)^\ell\times(S^1)^\ell\times\CC^{n-\ell}\hto(\RR_+)^\ell\times(S^1)^\ell\times\CC^{n-\ell}$, and we are mainly reduced to consider the inclusion of the open octant $(\RR_+^*)^\ell$ into the closed octant $(\RR_+)^\ell$. Then the assertion is clear.
\end{proof}

\subsubsection*{Morphisms between real blow-up spaces}
For any locally principal divisor~$D$ in~$X$ and any integer $n\geq1$, there is a natural morphism $\cO(D)\to\cO(nD)$, inducing $L(D)\to L(nD)$ and $\wt X(D)\to\wt X(nD)$, which is the identity on~$X^*$.

More generally, let $(D_{j\in J})$ be a finite family of locally principal divisors and let $(n_{ij})$ ($i\in I$, $j\in J$) be a finite family of nonnegative integers. Set $E_i=\sum_jn_{ij}D_j$, so that in particular the support of $(E_{i\in I})$ is contained in the support of $(D_{j\in J})$. Then the identity morphism $\id:X\to X$ lifts as a morphism $\wt X(D_{j\in J})\to\wt X(E_{i\in I})$. Indeed, for any $i\in I$, one has a natural morphism $\bigoplus_jL(D_j)\to L(E_i)$, and taking the direct sums of such morphisms when $i$ varies induces the desired lifting of $\id$.

In particular, if $J'$ is a subset of~$J$, there is a natural projection map between the fibre products, which induces a proper surjective map $\wt X(D_{j\in J})\to \wt X(D_{j\in J'})$.

Similarly, defining locally $D=\bigcup_jD_j$ by the product of the local equations of the~$D_j$, we have a proper surjective map $\wt X(D_{j\in J})\to \wt X(D)$.

Given a morphism $\pi:X\to X'$ and a family $(E'_{i\in I})$ of divisors of~$X'$, let $(E_{i\in I})$ be the pull-back family in~$X$. Then there is a natural morphism $\wt\pi:\wt X(E_{i\in I})\to\wt X{}'(E'_{i\in I})$. In particular, if we are given a family of divisors $(D_{j\in J})$ in~$X$ such that $E_i=\sum_in_{ij}D_j$ for any $j$ ($n_{ij}\in\NN$), we get a natural morphism $\wt\pi:\wt X(D_{j\in J})\to\wt X{}'(E'_{i\in I})$.

For example, if $\pi$ is chosen such that the divisor $E=\pi^*(\sum_jD_j)$ has simple normal crossings, we can choose for $(E'_{i\in I})$ the family of reduced irreducible components of~$E$.

\subsection{The sheaf of functions with moderate growth on the real blow-up space}\label{subsec:modgrowthfunct}

We consider as above a locally finite family $(D_j)_{j\in J}$ of effective divisors in a smooth complex manifold~$X$, and we set $D=\bigcup_j|D_j|$, where $|D_j|$ denotes the support of $D_j$. Let $\cO_X(*D)$ denotes the sheaf of meromorphic functions on~$X$ with poles along~$D$ at most. It can also be defined as the subsheaf of $j_*\cO_{X^*}$ (with $j:X^*=X\moins D\hto X$ the open inclusion) consisting of holomorphic functions having moderate growth along~$D$.

We define a similar sheaf on $\wt X\defin\wt X(D_{j\in J})$, that we denote by \index{holomorphic functions!with moderate growth ($\cA_{\wt X}^\modD$)}\index{$AXTMOD$@$\cA_{\wt X}^\modD$}$\cA_{\wt X}^\modD$: Given an open set $\wt U$ of $\wt X$, a section $f$ of $\cA_{\wt X}^\modD$ on $\wt U$ is a holomorphic function on $U^*\defin\wt U\cap X^*$ such that, for any compact set $K$ in $\wt U$, in the neighbourhood of which~$D$ is defined by $g_K\in\cO_X(K)$, there exists constants $C_K>0$ and $N_K\geq0$ such that $|f|\leq C_K|g_K|^{-N_K}$ on $K$.

\begin{remarque}[Rapid decay]\label{rem:rapiddecay}
We will also use the sheaf \index{$AXTRD$@$\cA_{\wt X}^\rdD$}$\cA_{\wt X}^\rdD$ of \index{holomorphic functions!with rapid decay ($\cA_{\wt X}^\rdD$)}holomorphic functions having rapid decay along $\partial\wt X$: Given an open set $\wt U$ of $\wt X$, a section $f$ of $\cA_{\wt X}^\rdD$ on $\wt U$ is a holomorphic function on $U^*\defin\wt U\cap X^*$ such that, for any compact set~$K$ in~$\wt U$, in the neighbourhood of which~$D$ is defined by $g_K\in\cO_X(K)$, and for any $N\in\NN$, there exists a constant $C_{K,N}>0$ such that $|f|\leq C_K|g_K|^N$ on $K$.
\end{remarque}

\begin{proposition}[Faithful flatness, \cf {\cite[Prop\ptbl2.8]{Bibi93}}]
If $\dim X\leq2$, the sheaves $\cA_{\wt X}^\modD$ and $\cA_{\wt X}^\rdD$ are \index{flatness}flat over $\varpi^{-1}\cO_X(*D)$ or $\varpi^{-1}\cO_X$, \index{flatness!faithfull --}faithfully over $\varpi^{-1}\cO_X(*D)$.
\end{proposition}

\begin{remarque}
T\ptbl Mochizuki has recently shown similar results in higher dimension. More precisely, relying on the basic theorems in \cite[Chap\ptbl VI]{Malgrange66}, he proves the statement for $\cA_{\wt X}^\rdD$ (and more general sheaves defined with rapid decay condition). Concerning moderate growth, the trick is to use, instead of the sheaf of holomorphic functions with moderate growth along~$D$, which is very big, the subsheaf of such functions of the Nilsson class, as in \cite[p\ptbl45]{Deligne84cc} (this sheaf is denoted there by~$\cA_{\log}$) or in \cite[p\ptbl61]{Malgrange91} (this sheaf is denoted there by $\cO^\mathrm{Nils}$). This trick is useful when~$D$ has normal crossings. The flatness of this subsheaf is also a consequence of basic flatness results.
\end{remarque}

\begin{proof}
We will give the proof for $\cA_{\wt X}^\modD$, the proof for $\cA_{\wt X}^\rdD$ being completely similar. Let us fix $x_o\in D$ and $\wt x_o\in\varpi^{-1}(x_o)\subset\partial\wt X$. The case where $\dim X=1$ is clear, because $\cA_{\wt X}^\modD$ has no $\cO_X$-torsion.

\subsubsection*{Faithful flatness}
Notice first that, if flatness is proved, the faithful flatness over $\varpi^{-1}\cO_X(*D)$ is easy: If $\cM_{x_o}$ has finite type over $\cO_{X,x_o}(*D)$ and $\cA_{\wt X,\wt x_o}^\modD\otimes_{\cO_{X,x_o}(*D)}\cM_{x_o}=0$ then, extending locally $\cM_{x_o}$ as a $\cO_X(*D)$-coherent module $\cM$, we obtain that $\cM$ vanishes on some multi-sectorial neighbourhood of $\wt x_o$ away from $\varpi^{-1}(D)$. Being $\cO_X(*D)$-coherent, it vanishes on some neighbourhood of~$x_o$ away from~$D$. It is therefore equal to zero.

\subsubsection*{Flatness}
Proving flatness is a matter of proving that, given $f_1,\dots,f_p\in\cO_{X,x_o}(*D)$, if $a_1,\dots,a_p\in\cA_{\wt X,\wt x_o}^\modD$ (\resp in $\cA_{\wt X,\wt x_o}^\rdD$) are such that $a_1f_1+\cdots+a_pf_p=0$, then $(a_1,\dots,a_p)$ is a linear combination with coefficients in $\cA_{\wt X,\wt x_o}^\modD$ (\resp in $\cA_{\wt X,\wt x_o}^\rdD$) of relations between $f_1,\dots,f_p$ with coefficients in $\cO_{X,x_o}(*D)$. Notice then that it is equivalent to prove flatness over $\cO_X$, because any local equation of~$D$ is invertible in $\cA_{\wt X,\wt x_o}^\modD$ (\resp in $\cA_{\wt X,\wt x_o}^\rdD$), so we will assume below that $f_1,\dots,f_p\in\cO_{X,x_o}$. We argue by induction on $p$, starting with $p=2$.

\subsubsection*{Case where $p=2$}
At this point, we do not need to assume that $\dim X\leq2$. We can then assume that $f_1$ and $f_2$ have no common irreducible component. Let us denote by $g$ a local equation of~$D$ at~$x_o$. Let $\wt U$ be an neighbourhood of $\wt x_o$ in~$\wt X$ such that we have a relation $a_1f_1=a_2f_2$ on $U^*=\wt U\moins\partial\wt X$. By Hartogs, there exists then a holomorphic function $\lambda$ on $U^*$ such that $a_1=\lambda f_2$ and $a_2=\lambda f_1$. We wish to show that $\lambda$ belongs to $\Gamma(\wt U,\cA_{\wt X}^\modD)$.

Let us choose a proper modification $e:Z\to X$, with $Z$ smooth, such that $f_1\circ e\cdot g\circ e$ defines a divisor $E$ with normal crossing in some neighbourhood of $e^{-1}(x_o)$. It is then enough to show that $\lambda\circ e$ has moderate growth along the real blow-up space $\partial\wt Z$ of the irreducible components of $e^{-1}(D)$ in $Z$, in the neighbourhood of $\wt e^{-1}(\wt x_o)$. By compactness of this set, we can work locally near a point $\wt z_o\in\wt e^{-1}(\wt x_o)$. Let us fix local coordinates $z$ at $z_o=\varpi_Z(\wt z_o)\in Z$ adapted to $E$. We denote by $z'$ the coordinates defining $e^{-1}(D)$ and set $z=(z',z'')$, so that $f_1\circ e$ is the monomial $z^{\prime m'}z^{\prime\prime m''}$. The assumption is that, for any compact neighbourhood $K$ of $\wt z_o$ in $\wt Z$, there exists a constant $C_K$ and a negative integer $N_K$ such that, on $K^*$, $|\lambda\circ e|\cdot |z'|^{m'}|z''|^{m''}\leq C_K |z'|^{N_K}$. Notice that such a $K$ can be chosen as the product of a compact polydisc in the variables $z''$ with a compact multi-sector in the variables $z'$. Up to changing $N_K$, this reduces to $|\lambda\circ e|\cdot|z''|^{m''}\leq C_K |z'|^{N_K}$. Fixing $|z''_i|=r''_i>0$ and small enough, and using Cauchy's formula, we obtain $|\lambda\circ e|\leq C'_K |z'|^{N_K}$.

\subsubsection*{Case where $p\geq3$}
We argue by induction on $p$. Assume that we have a relation $a_1f_1+\cdots+a_pf_p=0$ as above.

Firstly, we can reduce to the case where $f_1,\dots,f_{p-1}$ do not have a non trivial common factor $\delta$: we deduce a relation $(a_1f'_1+\cdots+a_{p-1}f'_{p-1})\delta+a_pf_p=0$, and by the $p=2$ case, we deduce a relation $a_1f'_1+\cdots+a_{p-1}f'_{p-1}+a'_pf_p$ of the same kind.

In such a case, since $\dim X=2$, we have $\dim V(f_1,\dots,f_{p-1})=0$, so locally $V(f_1,\dots,f_{p-1})=\{x_o\}$, and there is a relation $\sum_{i=1}^{p-1}h_if_i=1$ in $\cO_{X,x_o}(*D)$. We deduce the relation $\sum_{i=1}^{p-1}(a_i+a_pf_ph_i)f_i=0$, and the inductive step expresses the vector of coefficients $(a_i+a_pf_ph_i)$ in terms of relations between $f_1,\dots,f_{p-1}$ in $\cO_{X,x_o}(*D)$. The conclusion then follows for $f_1,\dots,f_p$, by using the supplementary relation $(h_1f_p)f_1+\cdots+(h_{p-1}f_p)f_{p-1}-f_p=0$.
\end{proof}

If $\pi:X\to X'$ is a proper modification which is an isomorphism $X\moins\nobreak D\isom X'\moins\nobreak E'$, and if $E_i:=\pi^*E'_i=\sum_jn_{ij}D_j$ for each $i$ ($n_{ij}\in\NN$), it induces $\wt\pi:\wt X(D_{j\in J})\to\wt X{}'(E'_{i\in I})$. (With the first assumption, the inclusion $|E|\subset|D|$ is an equality.) Then $\wt\pi_*(\cA_{\wt X}^\modD)=\cA_{\wt X{}'}^{\rmod E'}$. Similarly, if $\varpi:\wt X(D_{j\in J})\to\nobreak X$ is the natural projection, we have $\varpi_*(\cA_{\wt X}^\modD)=\cO_X(*D)$.

\begin{proposition}\label{prop:RpiA}
With the previous assumption, assume moreover that the divisors $\sum_jD_j$ and $\sum_iE'_i$ are normal crossing divisors. Then $\bR\wt\pi_*(\cA_{\wt X}^\modD)=\cA_{\wt X{}'}^{\rmod E'}$, that is, $\bR^k\wt\pi_*(\cA_{\wt X}^\modD)=0$ for $k\geq1$.
\end{proposition}

\begin{proof}
Since $\sum_jD_j$ is a normal crossing divisor, we can apply the Dolbeault-Grothendieck theorem on $\wt X(D)$ (\cf \cite[Prop\ptbl II.1.1.7]{Bibi97}) and get a $c$-soft resolution of $\cA_{\wt X}^\modD$ by the Dolbeault complex of moderate currents $\Db_{\wt X}^{\modD,(0,\cbbullet)}$ on $\wt X$ (which are $(0,\bbullet)$-forms on $\wt X$ with coefficients in the sheaf $\Db_{\wt X}^\modD$ of distributions on $X\moins D$ with moderate growth along~$D$). Therefore, $\bR\wt\pi_*(\cA_{\wt X}^\modD)=\wt\pi_*\Db_{\wt X}^{\modD,(0,\cbbullet)}$.

Recall that, on an open set $\wt U$, $\Db_{\wt X}^\modD(\wt U)$ is dual to the space of $C^\infty$ functions with compact support in $\wt U$ which have rapid decay along $\wt U\cap\partial\wt X$. Since $\pi$ is a modification, for any compact set $K$ in $\wt U$, setting $K'=\wt\pi(K)$, the \index{pull-back (inverse image)!of differential forms}pull-back of forms $\wt\pi^*$ identifies $C^\infty$ forms on $\wt X{}'$ with support in $K'$ and having rapid decay along $\partial\wt X{}'$ with the corresponding forms on $\wt X$ with support in $K$ and rapid decay along $\partial\wt X$, and this identification is compatible with the differential, as well as with $\partial$ and~$\ov\partial$. Dually, the integration along the fibres of $\wt\pi$ of currents identifies the complexes $\wt\pi_*\Db_{\wt X}^{\modD,(0,\cbbullet)}$ with $\Db_{\wt X{}'}^{\rmod E',(0,\cbbullet)}$.

Lastly, the latter complex is a resolution of $\cA_{\wt X{}'}^{\rmod E'}$ by Dolbeault-Grothendieck, since $\sum_iE'_i$ is a normal crossing divisor.
\end{proof}

\begin{remarque}\label{rem:RpiA}
Other variants of this proposition can be obtained with a similar proof. For instance, we have $\bR\varpi_*(\cA_{\wt X}^\modD)=\varpi_*(\cA_{\wt X}^\modD)=\cO_X(*D)$. More generally, with the assumptions in Proposition \ref{prop:RpiA}, let $I_1$ and $J_1$ be subsets of $I$ and~$J$ respectively such that each $E_i$ ($i\in I_1$) is expressed as a linear combination with coefficients in $\NN$ of $(D_{j\in J_1})$. Let $\wt\pi_1:\wt X(D_{j\in J_1})\to\wt X{}'(E'_{i\in I_1})$ be the morphism induced by $\pi$ between the partial real blow-up spaces. Then, with obvious notation, $\bR\wt\pi_*(\cA_{\wt X}^{\rmod D_1}(*D))=\cA_{\wt X{}'}^{\rmod E'_1}(*E')$. A particular case is $I_1=\emptyset$, $J_1=\emptyset$, giving $\bR\pi_*\cO_X(*D)=\cO_{X'}(*E')$.\enlargethispage{2\baselineskip}%
\end{remarque}

\subsection{The moderate de~Rham complex}\label{subsec:modgrowth}
We keep the setting of \S\ref{subsec:modgrowthfunct}. The sheaf $\cA_{\wt X}^\modD$ is stable by derivations of~$X$ (in local coordinates) and there is a natural de~Rham complex on $\wt X(D_{j\in J})$:
\[
\DR^\modD(\cO_X)\defin\{\cA_{\wt X}^\modD\To{d}\cA_{\wt X}^\modD\otimes\varpi^{-1}\Omega^1_X\to\cdots\}
\]
When restricted to~$X^*$, this complex is nothing but the usual holomorphic de~Rham complex.

From now on, we will freely use standard results in the theory of holonomic $\cD_X$-modules, for which we refer to \cite{Bjork79, Borel87,Mebkhout87,Bjork93,Kashiwara03}.

Let $\cM$ be a holonomic $\cD_X$-module which is localized along~$D$, that is, such that $\cM=\cO_X(*D)\otimes_{\cO_X}\cM$. In particular, $\cM$ is also a coherent $\cD_X(*D)$-module. We can also regard $\cM$ as a $\cO_X(*D)$-module equipped with a flat connection $\nabla$. If moreover~$\cM$ is $\cO_X(*D)$-coherent, we call it a meromorphic connection with poles along~$D$ (according to \cite[Prop\ptbl1.1]{Malgrange95}, it is then locally stably free as a $\cO_X(*D)$-module).

We associate with $\cM$ the \index{de~Rham complex!moderate --}\emph{moderate de~Rham complex}
\[
\index{$DRMOD$@$\DR^{\rmod0}\cM$, $\DR^\modD\cM$}\DR^\modD(\cM)\defin\{\cA_{\wt X}^\modD\otimes\varpi^{-1}(\cM)\To{\nabla}\cA_{\wt X}^\modD\otimes\varpi^{-1}(\Omega^1\otimes\cM_X)\ra\cdots\}
\]
which coincides with $\DR(\cM)$ on~$X^*$.

\begin{proposition}\label{prop:Rpimod}
Let $\pi:X\to X'$ be a proper modification between complex manifolds. Assume the following:
\begin{enumerate}
\item
There exist locally finite families of divisors $(E'_{i\in I})$ of $X'$ and $(D_{j\in J})$ of~$X$ such that $E_i\defin\pi^*E'_i=\sum_jn_{ij}D_j$ with $n_{ij}\in \NN$,
\item
$\pi:X\moins D\to X'\moins E'$ is an isomorphism.
\end{enumerate}
Let $\cM$ be a holonomic $\cD_X$-module which is localized along~$D$. If $\dim X\geq3$, assume moreover that $\cM$ is smooth on $X\moins D$, \ie is a meromorphic connection with poles along~$D$ at most. Let $\pi_+\cM$ the \index{push-forward (direct image)!of $\cD$-modules}direct image of $\cM$ (as a $\cD_{X'}(*E)$-module). Then\index{push-forward (direct image)!of moderate de~Rham complexes}
\[
\DR^{\rmod E'}(\pi_+\cM)\simeq\bR\wt\pi_*\DR^\modD(\cM).
\]
\end{proposition}

\begin{remarque}\label{rem:Rpimod}
We have variants corresponding to those in Remark \ref{rem:RpiA}.
\end{remarque}

\subsubsection*{Preliminaries on meromorphic connections and proper modifications}
Let us first recall classical facts concerning direct images of meromorphic connections by a proper modification. The setting is the following. We denote by $\pi:X\to X'$ a proper modification between complex manifolds, and we assume that there are reduced divisors $D\subset X$ and $E'\subset\nobreak X'$ such that $\pi:X\moins D\to X'\moins E'$ is an isomorphism (so that in particular $\pi^{-1}(E')=D$). We do not assume now that~$D$ or $E'$ are normal crossing divisors. We denote by $\cO_X(*D)$ and $\cO_{X'}(*E')$ the corresponding sheaves of meromorphic functions.

\begin{lemme}\label{lem:RkpiO}
We have $\bR\pi_*\cO_X(*D)=\cO_{X'}(*E')$.
\end{lemme}

\begin{proof}
The statement is similar to that of Proposition \ref{prop:RpiA}, but will be proved with less assumptions. If~$D$ and $E'$ have normal crossings, one can adapt the proof of Proposition \ref{prop:RpiA}, according to the Dolbeault-Grothendieck lemma using currents with moderate growth. In general, one can argue differently as follows.

Since $\pi$ is proper, one has $\bR^k\pi_*\cO_X(*D)=\varinjlim_j\bR^k\pi_*\cO_X(jD)$. The left-hand term is equal to $\cO_{X'}(*E')\otimes_{\cO_X'}\bR^k\pi_*\cO_X(*D)$: indeed, this amounts to proving that, if $f'$ is a local equation for $E'$, then the multiplication by $f'$ is invertible on $\bR^k\pi_*\cO_X(*D)$; but it is induced by the multiplication by $f'\circ\pi$ on $\cO_X(*D)$, which is invertible since $f'\circ\pi$ vanishes on $D$ at most (recall that $\pi^{-1}(E')=D$).

Then, $\cO_{X'}(*E')\otimes_{\cO_X'}\varinjlim_j\bR^k\pi_*\cO_X(jD)=\varinjlim_j\cO_{X'}(*E')\otimes_{\cO_X'}\bR^k\pi_*\cO_X(jD)$ (\cf \eg \cite[p\ptbl10]{Godement64}) and the right-hand term is zero if $k\geq 1$, since $\bR^k\pi_*\cO_X(jD)$ is then $\cO_{X'}$-coherent and supported on $E'$.
\end{proof}

For any $k\in\NN$ we have a natural morphism
\[
\Omega^k_{X'}(*E')\to\pi_*\big(\cO_X(*D)\ootimes_{\pi^{-1}\cO_{X'}}\pi^{-1}\Omega^k_{X'}(*E')\big).
\]
It follows from the previous lemma that this morphism is an isomorphism. Indeed, on the one hand, the projection formula and the previous lemma give
\begin{align*}
\bR\pi_*\big(\cO_X(*D)\ootimes^{\bL}_{\pi^{-1}\cO_{X'}}\pi^{-1}\Omega^k_{X'}(*E')\big)&\simeq\bR\pi_*\cO_X(*D)\ootimes^{\bL}_{\cO_{X'}}\Omega^k_{X'}(*E')\\[-5pt]
&\simeq\cO_{X'}(*E')\ootimes^{\bL}_{\cO_{X'}}\Omega^k_{X'}(*E')\simeq\Omega^k_{X'}(*E').
\end{align*}
On the other hand, since $\Omega^k_{X'}$ is $\cO_{X'}$-locally free, one can eliminate the `$\bL$' in the first line above, and conclude
\begin{align*}
\pi_*\big(\cO_X(*D)\ootimes_{\pi^{-1}\cO_{X'}}\pi^{-1}\Omega^k_{X'}(*E')\big)&\simeq\Omega^k_{X'}(*E'),\\
\bR^k\pi_*\big(\cO_X(*D)\ootimes_{\pi^{-1}\cO_{X'}}\pi^{-1}\Omega^k_{X'}(*E')\big)&=0\quad\forall k\geq1.
\end{align*}

The cotangent map $T^*\pi$ is a morphism $\cO_X\otimes_{\pi^{-1}\cO_{X'}}\pi^{-1}\Omega^1_{X'}\to\Omega^1_X$. It induces an isomorphism $\cO_X(*D)\otimes_{\pi^{-1}\cO_{X'}}\pi^{-1}\Omega^1_{X'}(*E')\to\Omega^1_X(*D)$. Applying $\pi_*$ and using the previous remark, we get an isomorphism $\pi_*T^*\pi:\Omega^1_{X'}(*E')\isom\pi_*\Omega^1_X(*D)$, and $\bR^k\pi_*\Omega^1_X(*D)=0$ if $k\geq1$. Since $T^*\pi$ is compatible with differentials, we get a commutative diagram
\begin{equation}\label{eq:dpi_*d}
\begin{array}{c}
\xymatrix{
\cO_{X'}(*E')\ar[d]^\wr\ar[r]^-d&\Omega^1_{X'}(*E')\ar[d]^{\pi_*T^*\pi}_{\wr}\\
\pi_*\cO_X(*D)\ar[r]^-{\pi_*d}&\pi_*\Omega^1_X(*D)
}
\end{array}
\end{equation}
where the upper $d$ is the differential on $X'$ and the lower $d$ is that on~$X$. Arguing similarly for each $\Omega^k$ and the corresponding differentials, we obtain an isomorphism of complexes
\[
\DR\cO_{X'}(*E')\isom\pi_*\DR\cO_X(*D)\simeq\bR\pi_*\DR\cO_X(*D).
\]

Let now $\cM$ be a holonomic $\cD_X$-module which is localized along~$D$ and let us also regard it as a $\cO_X(*D)$-module with a flat connection $\nabla$.

\skpt
\begin{proposition}\label{prop:RkpiM}\ligne
\begin{enumerate}
\item\label{prop:RkpiM1}
The direct image $\pi_+\cM$ of $\cM$ as a $\cD_X$-module, once localized along $E'$, has cohomology in degree $0$ at most and this cohomology is a holonomic $\cD_{X'}$-module localized along $E'$.
\item\label{prop:RkpiM2}
We have $\bR^k\pi_*\cM=0$ for $k\geq1$. Moreover, $\pi_*\cM$ is equal to the $\cO_{X'}(*D')$-module underlying $\pi_+\cM(*E')$, and the connection on the latter is equal to the composition of $\pi_*\nabla:\pi_*\cM\to\pi_*(\Omega^1_X(*D)\otimes\cM)$ with the isomorphism
\bgroup\numstareq
\begin{equation}\label{eq:RkpiM}
\Omega^1_{X'}(*E')\otimes\pi_*\cM\simeq\pi_*(\pi^{-1}\Omega^1_{X'}(*E')\otimes\cM)\isom\pi_*(\Omega^1_X(*D)\otimes\cM)
\end{equation}
\egroup
induced by $\pi_*T^*\pi$.
\end{enumerate}
\end{proposition}

\begin{remarque}\label{rem:imdirDloc}
Let $\Theta_X$ be the sheaf of vector fields on $X$. We also have a morphism $T_*\pi:\pi^*\Theta_{X'}\to\Theta_X$ of locally free $\cO_X$-modules, which induces an isomorphism $T_*\pi:\pi^*\Theta_{X'}(*E')\to\Theta_X(*D)$. We deduce an isomorphism $\pi_*T_*\pi:\Theta_{X'}(*E')\to\pi_*\Theta_X(*D)$. Since the sheaf of localized differential operators $\cD_X(*D)$ is generated by $\cO_X(*D)$ and $\Theta_X(*D)$, this implies that there is a ring isomorphism $\cD_{X'}(*E')\to\pi_*\cD_X(*D)$. The \index{push-forward (direct image)!of $\cD$-modules}direct image functor $\pi_+$ for $\cD$-modules is much simplified in this localized setting. For a holonomic left $\cD_X(*D)$-module $\cM$, we have $\pi_+\cM=\pi_*\cM$ (according to the first assertion of \ref{prop:RkpiM}\eqref{prop:RkpiM2}), where the $\cD_{X'}(*E')$-action is obtained through the $\pi_*\cD_X(*D)$-action.
\end{remarque}

\skpt
\begin{proof}\ligne
\begin{enumerate}
\item
Since $\cM$ has a coherent $\cO_X$-module $\cN$ generating $\cM$ as a $\cD_X$-module (\cf \cite{Malgrange95b}, \cf also \cite[Th\ptbl 3.1]{Malgrange04}), it is known (\cf \cite{Kashiwara76,Malgrange85}) that $\pi_+\cM$ has holonomic cohomology. Moreover, $\cH^k\pi_+\cM$ is supported on $E'$ if $k\geq1$. By flatness of $\cO_{X'}(*E')$ over $\cO_{X'}$, we have $\cH^k(\cO_{X'}(*E')\otimes_{\cO_{X'}}\pi_+\cM)=\cO_{X'}(*E')\otimes_{\cO_{X'}}\cH^k(\pi_+\cM)$. By a theorem of Kashiwara \cite[Prop\ptbl2.9]{Kashiwara78}, $\cO_{X'}(*E')\otimes_{\cO_{X'}}\cH^k(\pi_+\cM)$ is a holonomic $\cD_{X'}$-module and it is localized along $E'$ by definition. If $k\geq1$, it is thus equal to zero.
\item
For the first assertion, the proof is similar to that of Lemma \ref{lem:RkpiO}. We use that $\cM$ has a $\cO_X(*D)$-generating coherent $\cO_X$-submodule: In the case where $\cM$ is a meromorphic connection, this follows from \cite{Malgrange95} (\cf also \cite{Malgrange04}). In the general case, one first uses that $\cM$ has a coherent $\cO_X$-module $\cN$ generating $\cM$ as a $\cD_X$-module (\cf \cite{Malgrange95b}, \cf also \cite[Th\ptbl 3.1]{Malgrange04}). Using Bernstein's theory for a function locally defining~$D$, one then shows that, locally on~$D$, there exists $\ell$ such that $(F_\ell\cD_X)\cN$ generates $\cN$ as a $\cO_X(*D)$-module, where $F_\bbullet\cD_X$ is the filtration by the order. Since $\pi$ is proper, one can find a suitable~$\ell$ valid on the inverse image by~$\pi$ of any compact set in $X'$. In this way, we get the vanishing of $R^k\pi_*\cM$ for $k\geq1$.

Using the isomorphism \eqref{eq:RkpiM} induced by $\pi_*T^*\pi$, we regard $\pi_*\nabla:\pi_*\cM\to\pi_*(\Omega^1_X(*D)\otimes\cM)$ as a morphism $\pi_+\nabla:\pi_*\cM\to\Omega^1_{X'}(*E')\otimes\pi_*\cM$, and \eqref{eq:dpi_*d} shows that it is a connection on $\pi_*\cM$. We denote the resulting object by $\pi_+(\cM,\nabla)$. Checking that it is equal to $\cO_{X'}(*D')\otimes_{\cO_X}\cH^0\pi_+\cM$ (in the $\cD_X$-module sense) is then straightforward, according to Remark \ref{rem:imdirDloc}. By definition, we have a commutative diagram
\begin{equation}\label{eq:nablapi_*nabla}
\begin{array}{c}
\xymatrix{
\pi_*\cM\ar@{=}[d]\ar[r]^-{\pi_+\nabla}&\Omega^1_{X'}(*E')\otimes\pi_*\cM\ar[d]^{\pi_*(T^*\pi\otimes\id)}_{\wr}\\
\pi_*\cM\ar[r]^-{\pi_*\nabla}&\pi_*(\Omega^1_X(*D)\otimes\cM)
}
\end{array}
\qedhere
\end{equation}
\end{enumerate}
\end{proof}

Let us now go in the other direction. Given a holonomic $\cD_{X'}$-module $\cM'$ localized along $E'$, the \index{pull-back (inverse image)!of $\cD$-modules}inverse image $\pi^+\cM'\defin\cD_{X\to X'}\otimes^{\bL}_{\pi^{-1}\cD_{X'}}\pi^{-1}\cM'$ has holonomic cohomology (\cf \cite{Kashiwara78}) and $\cH^k\pi^+\cM'$ is supported on $D=\pi^{-1}(E')$. It follows that $\cH^k(\cO_X(*D)\otimes_{\cO_X}\pi^+\cM')=0$ if $k\neq0$ and $\cH^0(\cO_X(*D)\otimes_{\cO_X}\pi^+\cM')=\cO_X(*D)\otimes_{\cO_X}\cH^0(\pi^+\cM')$ is holonomic and localized along~$D$, and its underlying $\cO_X(*D)$-submodule is $\pi^*\cM'\defin\cO_X(*D)\otimes_{\pi^{-1}\cO_{X'}(*E')}\pi^{-1}\cM'$. Denoting by $\nabla'$ the connection on $\cM'$, denoted by \index{pull-back (inverse image)!of connections}$\pi^+\nabla'$, is defined as $d\otimes\id+T^*\pi(\id\otimes\pi^{-1}\nabla')$, where the second term is the composition of
\[
\id\otimes\nabla':\cO_X(*D)\otimes_{\pi^{-1}\cO_{X'}(*E')}\pi^{-1}\cM'\to\cO_X(*D)\otimes_{\pi^{-1}\cO_{X'}(*E')}\pi^{-1}(\Omega^1_{X'}\otimes\cM')
\]
with
\begin{multline*}
T^*\pi\otimes\id:(\cO_X(*D)\otimes_{\pi^{-1}\cO_{X'}(*E')}\pi^{-1}\Omega^1_{X'}(*E'))\otimes_{\pi^{-1}\cO_{X'}(*E')}\pi^{-1}\cM'\\
\to\Omega^1_X(*D)\otimes_{\pi^{-1}\cO_{X'}(*E')}\pi^{-1}\cM'.
\end{multline*}
We set $\pi^+(\cM',\nabla')=(\pi^*\cM,\pi^+\nabla')$, and this is nothing but $\cO_X(*D)\otimes_{\cO_X}\cH^0(\pi^+\cM')$ as a holonomic $\cD_X$-module localized along~$D$.

The natural morphism of $\cO_X(*D)$-modules
\[
\pi^*\pi_*\cM\defin\cO_X(*D)\otimes_{\pi^{-1}\cO_{X'}(*E')}\pi^{-1}\pi_*\cM\to\cM
\]
induced by the adjunction $\pi^{-1}\pi_*\cM\to\cM$ is compatible with the connections, so that it induces a morphism
\[
\pi^+\pi_+(\cM,\nabla)\to(\cM,\nabla)
\]
which can be regarded as the \emphb{adjunction} morphism at the level of $\cD_X(*D)$-modules. It is therefore an isomorphism, since the kernel and cokernel are localized holonomic $\cD_X$-modules which are supported on~$D$.

Similarly, the adjunction morphism $\id\to\pi_*\pi^{-1}$ together with the projection formula induces an isomorphism
\[
(\cM',\nabla')\isom\pi_+\pi^+(\cM',\nabla').
\]

In conclusion:

\begin{proposition}
The functors $\pi_+$ and $\pi^+$ are quasi-inverse one to the other.
\end{proposition}

\begin{proof}[\proofname\ of Proposition \ref{prop:Rpimod}]
One can extend the previous results by replacing $\cO_X(*D)$ with $\cA_{\wt X}^\modD$ and $\cO_{X'}(*E')$ with $\cA_{\wt X{}'}^{\rmod E'}$. In order to do this, we now assume that~$D$ and $E'$ have normal crossings. By the first assumption in the proposition, we have a morphism $\wt\pi:\wt X\to\wt X{}'$ lifting $\pi$. By the preliminaries above, we can write $\cM=\cO_X(*D)\otimes_{\pi^{-1}\cO_{X'}(*E')}\pi^{-1}\cM'$. Therefore, $\cA_{\wt X}^\modD\otimes_{\varpi^{-1}\cO_X}\cM=\cA_{\wt X}^\modD\otimes_{\wt\pi^{-1}\varpi^{\prime-1}\cO_{X'}}\wt\pi^{-1}\varpi^{\prime-1}\cM'$. Since $\cA_{\wt X}^\modD$ (\resp $\cA_{\wt X{}'}^{\rmod E'}$) is flat over $\varpi^{-1}\cO_X(*D)$ (\resp $\varpi^{\prime-1}\cO_{X'}(*E')$) ($\dim X\leq2$) or since $\cM$ is locally stably free over $\cO_X(*D)$ ($\dim X\geq3$), it follows from Proposition \ref{prop:RpiA} and the projection formula that
\begin{align*}
\bR\wt\pi_*(\cA_{\wt X}^\modD\otimes_{\varpi^{-1}\cO_X(*D)}\cM)&=\bR\wt\pi_*(\cA_{\wt X}^\modD\otimes^{\bL}_{\varpi^{-1}\cO_X(*D)}\cM)\\
&=\cA_{\wt X{}'}^{\rmod E'}\otimes^{\bL}_{\varpi^{\prime-1}\cO_{X'}}\varpi^{\prime-1}\cM'\\
&=\cA_{\wt X{}'}^{\rmod E'}\otimes_{\varpi^{\prime-1}\cO_{X'}}\varpi^{\prime-1}\cM',
\end{align*}
and therefore the latter term is equal to $\wt\pi_*(\cA_{\wt X}^\modD\otimes_{\varpi^{-1}\cO_X}\cM)$. Arguing similarly after tensoring with $\Omega^k$ gives that each term of the complex $\DR^\modD\cM$ is $\wt\pi_*$-acyclic and that $\wt\pi_*\DR^\modD\cM$ and $\DR^{\rmod E'}\cM'$ are isomorphic termwise. Moreover, the connection $\pi_*\nabla$ on $\wt\pi_*(\cA_{\wt X}^\modD\otimes_{\varpi^{-1}\cO_X}\cM)$ coincides, via a diagram similar to \eqref{eq:nablapi_*nabla} to the connection $\pi_+\nabla$ on $\cA_{\wt X{}'}^{\rmod E'}\otimes_{\varpi^{\prime-1}\cO_{X'}}\varpi^{\prime-1}\cM'$. Extending this isomorphism to the de~Rham complexes gives
\[
\DR^{\rmod E'}\cM'\isom\wt\pi_*\DR^\modD\cM\simeq\bR\wt\pi_*\DR^\modD\cM.\qedhere
\]
\end{proof}

\subsection{Examples of moderate de~Rham complexes}
We consider the local setting where $(X,0)$ is a germ of complex manifold,~$D$ is a divisor with normal crossing in~$X$ (defined by $t_1\cdots t_\ell=0$ in some coordinate system $(t_1,\dots,t_n)$ and, setting $L=\{1,\dots,\ell\}$, $\wt X=\wt X(D_{i\in L})$, with $D_i=\{t_i=0\}$. We will give examples of computation of moderate de~Rham complexes $\DR^\modD(\cE^\varphi)$, with $\varphi\in\cO_{X,0}(*D)/\cO_{X,0}$, and $\cE^\varphi\defin(\cO_X(*D), d+d\varphi)$. In the following, we assume that $\varphi\neq0$ in $\cO_{X,0}(*D)/\cO_{X,0}$.

\begin{proposition}\label{prop:HkEphinul}
Assume that there exists $\bmm\in\NN^\ell$ such that $\varphi=t^{-\bmm}u(t)\bmod\cO_{X,0}$, with $u\in\cO_{X,0}$ and $u(0)\neq0$. Then $\DR^\modD(\cE^\varphi)$ has cohomology in degree~$0$ at most.
\end{proposition}

\begin{proof}
This is a direct consequence of theorems in asymptotic analysis due to Majima \cite{Majima84}. See a proof in \cite[Appendix, Th\ptbl A.1]{Hien07}.
\end{proof}

On the other hand, if $u(0)=0$, the result does no longer hold, as shown by the following example.

\begin{exemple}\label{exem:degonecohom}
\emph{Assume that $X=\CC^2$ with coordinates $x,y$, $D=\{y=0\}$ and $\varphi=x^2/y$. Let $\varpi:\wt X(D)\to X$ be the real blowing-up of~$D$ in~$X$. Then $\DR^\modD(\cE^\varphi)$ has a nonzero $\cH^1$.}

\begin{proof}
By Proposition \ref{prop:HkEphinul}, $\cH^1\DR^\modD(\cE^\varphi)$ is supported on $\varpi^{-1}(0,0)\simeq S^1_y$, since~$\varphi$ satisfies the assumption of this proposition away from $x=0$. In order to compute this sheaf, we will use a blowing-up method similar to that used in \Chaptersname\ref{chap:Laplace}, in order to reduce to local computations where Proposition \ref{prop:HkEphinul} applies. In the following, we will fix $\theta_o\in S^1_y$ and we will compute the germ $\cH^1\DR^\modD(\cE^\varphi)$ at $\theta_o$.

Let $e:Y\to X$ be the complex blowing-up of the origin in~$X$. It is covered by two affine charts, $Y_1$ with coordinates $(x,v)$ with $e(x,v)=(x,xv)$ and $Y_2$ with coordinates $(u,y)$ with $e(u,y)=(yu,y)$. On $Y_1\cap Y_2$, we have $v=1/u$, and these are the coordinates on the exceptional divisor $E\simeq\PP^1$. We still denote by~$D$ the strict transform of~$D$, which is defined by $v=0$ in $Y_1$ (and does not meet $Y_2$). We have a natural map $\wt e:\wt Y(D,E)\to\wt X(D)$, and by Proposition \ref{prop:Rpimod}, we have $\DR^\modD(\cE^\varphi)=\bR\wt e_*\DR^{\rmod F}(\cE^{\psi})$, where $F=D\cup E=e^{-1}(D)$ and $\psi=\varphi\circ e$. Restricting to $\arg y=\theta_o$ gives, by proper base change,
\[
\cH^1\DR^\modD(\cE^\varphi)_{\theta_o}=\HH^1\big(\wt e^{-1}(\theta_o),\DR^{\rmod F}(\cE^{\psi})\big).
\]

\subsubsection*{Chart $Y_2$}
We have $\wt Y_2=\wt Y_2(E)$ and, on $Y_2$, $\psi=yu^2=0\bmod\cO_{Y_2}$. Therefore, according to Proposition \ref{prop:HkEphinul}, $\DR^\modE(\cE^{\psi})$ has cohomology in degree~$0$ only on $\wt Y_2$, and $\cH^0\DR^\modE(\cE^{\psi})$ is the constant sheaf $\CC$. Let us note that $\partial\wt Y_2=\Afu\times S^1_y$, so the restriction to $\arg y=\theta_o$ of $\DR^\modE(\cE^{\psi})$ is the constant sheaf $\CC_{\Afu}$.

\subsubsection*{Chart $Y_1$}
We have $\wt Y_1=\wt Y_1(D,E)$ and $\partial\wt Y_{1|E}=S^1_x\times S^1_v\times[0,\infty)$, where $|v|$ runs in $[0,\infty)$. The map $\wt e:\partial\wt Y_{1|E}\to S^1_y$ is the composed map
\begin{align*}
S^1_x\times S^1_v\times[0,\infty)\to S^1_x\times S^1_v&\to S^1_y\\
(\alpha,\beta)&\mto\alpha\beta,
\end{align*}
so, in this chart, $\wt e^{-1}(\theta_o)\simeq S^1\times[0,\infty)$ and we can assume $S^1=S^1_v$ so that we can identify $\wt e^{-1}(\theta_o)$ in $\partial\wt Y$ to a closed disc having $S^1_v$ as boundary. We also have $\psi=x/v$.

\begin{claim*}
The complex $\DR^{\rmod F}(\cE^{\psi})$ has cohomology in degree~$0$ at most on $\wt Y_1(D,E)$.
\end{claim*}

Assuming this claim is proved, it is not difficult to compute $\cH^0\DR^{\rmod F}(\cE^{\psi})_{\theta_o}$. In the chart $Y_2$, this has been done previously, so we are reduced to compute it on~$S^1_v$ (boundary of the closed disc $\wt e^{-1}(\theta_o)$). The $\cH^0$ is zero unless $\arg x-\arg v=\theta_o-2\arg v\in(\pi/2,3\pi/2)\bmod2\pi$.

\subsubsection*{Conclusion}
The complex $\DR^{\rmod F}(\cE^{\psi})_{\theta_o}$ on $\wt e^{-1}(\theta_o)$ has cohomology in degree~$0$ only, and is the constant sheaf on $\Delta\cup I_1\cup I_2$, extended by~$0$ to $\ov\Delta$, where $\Delta$ is an open disc in $\CC$ and $I_1$ and $I_2$ are two opposite (and disjoint) open intervals of length $\pi/2$ on $\partial\Delta$. It is now an exercise to show that $\dim\HH^1\big(\wt e^{-1}(\theta_o),\DR^{\rmod F}(\cE^{\psi})\big)=1$.
\end{proof}

\begin{proof}[\proofname\ of the claim]
Let us sketch it. The question is local in the $(x,v)$-chart, on $S^1_x\times S^1_v$. We blow up the origin in this chart $Y_1$ and get an exceptional divisor $G\simeq\PP^1$ with coordinates $v_1=1/u_1$. On the blow-up space $Z_1$, in the chart $Z_{11}$ with coordinates $(x,v_1)$, the blowing-up map $\epsilon$ is $(x,v_1)\mto(x,v=xv_1)$ and we have $\eta:=\psi\circ\epsilon=x/v=1/v_1$. In the chart $Z_{12}$ with coordinates $(u_1,v)$, the blowing-up map $\epsilon$ is $(u_1,v)\mto(x=u_1v,v)$ and we have $\eta=x/v=u_1$.

On $\wt Z_{12|F}$, $\wt\epsilon$ is the composed map
\begin{align*}
S^1_{u_1}\times S^1_v\times[0,\infty)\to S^1_{u_1}\times S^1_v&\to S^1_x\times S^1_v\\
(\alpha,\beta)&\mto(\alpha\beta,\beta),
\end{align*}
where $|u_1|$ varies in $[0,\infty)$. On this space, $\DR^{\rmod F}(\cE^{\eta})$ is the constant sheaf $\CC$.

Similarly, on $\wt Z_{11|F}$, $\wt\epsilon$ is the composed map
\begin{align*}
S^1_x\times S^1_{v_1}\times[0,\infty)\to S^1_x\times S^1_{v_1}&\to S^1_x\times S^1_v\\
(\alpha,\gamma)&\mto(\alpha,\alpha\gamma),
\end{align*}
and $|v_1|$ varies in $[0,\infty)$. On this space, $\DR^{\rmod F}(\cE^{\eta})$ has cohomology in degree~$0$ at most, after Proposition \ref{prop:HkEphinul}, and is the constant sheaf $\CC$, except on $|v_1|=0$, $\arg v_1\not\in(\pi/2,3\pi/2)\bmod2\pi$, where it is zero.

Let us fix a point $(\alpha^o,\beta^o)\in S^1_x\times S^1_v$. We have $\wt\epsilon{}^{-1}(\alpha^o,\beta^o)\simeq[0,+\infty]$, and $\cH^k\DR^{\rmod F}(\cE^{\psi})_{(\alpha^o,\beta^o)}=H^k([0,+\infty],\cF_{(\alpha^o,\beta^o)})$, where $\cF_{(\alpha^o,\beta^o)}$ is the constant sheaf $\CC$ if $\gamma^o\defin\beta^o\alpha^{o-1}\not\in(\pi/2,3\pi/2)\bmod2\pi$, and the constant sheaf on $(0,+\infty]$ extended by~$0$ at~$0$ otherwise.

In the first case, we have $H^k=0$ for any $k\geq1$, and in the second case we have $H^k=0$ for any $k\geq0$.
\end{proof}
\end{exemple}

\chapter[Stokes-filtered local systems along a NCD]{Stokes-filtered local systems along a~divisor with normal crossings}\label{chap:StokesfilteredNCD}

\begin{sommaire}
We construct the sheaf $\ccI$ to be considered as the index sheaf for Stokes filtrations. This is a sheaf on the real blow-up space of a complex manifold along a family of divisors. We will consider only divisors with normal crossings. The global construction of $\ccI$ needs some care, as the trick of considering a ramified covering cannot be used globally. The important new notion is that of goodness. It is needed to prove abelianity and strictness in this setting, generalizing the results of \Chaptersname\ref{chap:abelian}.
\end{sommaire}

\subsection{Introduction}
After having introduced the real blow-up spaces in the previous \chaptername, we now extend the notion of a Stokes-filtered local system in higher dimension, generalizing the contents of \Chaptername s \ref{chap:Stokesone} and \ref{chap:abelian}. For this purpose, we first define the sheaf $\ccI$ of ordered abelian groups, which will serve as the indexing sheaf for the Stokes filtrations. The general approach of \Chaptersname \ref{chap:Ifil} will now be used, and the sheaf $\ccI$ will be constructed in a way similar to that used in Remark \ref{rem:Euler}. The present \chaptersname was indeed the main motivation for developing \Chaptersname \ref{chap:Ifil}. We will make precise the \emph{global} construction of the sheaf $\ccI$, since the main motivation of this \chaptersname is to be able to work with Stokes-filtrered local systems globally on the real blow-up space $\wt X(D_{j\in J})$.

However, the order of a local section $\varphi$ of $\ccI$ is strongly related to the asymptotic behaviour of $\exp\varphi$ on $\partial\wt X(D_{j\in J})$. This behaviour is in general difficult to analyze, unless $\varphi$ behaves like a monomial with negative exponents. For instance, the asymptotic behaviour of $\exp(x_1/x_2)$ near $x_1=x_2=0$ is not easily analyzed, while that of $\exp(1/x_1x_2)$ is easier to understand. This is why we introduce the notion of \emph{pure monomiality}. Moreover, given a finite subset $\Phi$ of local sections of $\ccI$, the asymptotic comparison of the functions $\exp\varphi$ with $\varphi\in\Phi$ leads to considering the condition ``\emph{good}'', meaning that each nonzero difference $\varphi-\psi$ of local sections of $\Phi$ is purely monomial. The reason for considering differences $\varphi-\psi$ of elements of $\Phi$ is, firstly, that it allows one to totally order the elements of $\Phi$ with respect to the order of the pole and, secondly, that these differences are the exponential factors of the endomorphism of a given Stokes-filtered local system having $\Phi$ as exponential factors, and the corresponding Stokes-filtered local system is essential in classification questions. Of course, pure monomiality and goodness are automatically satisfied in dimension one.

The analogues of the main results of \Chaptername s \ref{chap:Stokesone} and \ref{chap:abelian} are therefore proved assuming goodness. We will find this goodness assumption in various points later on, and the notion of Stokes filtration developed in this text always assumes goodness.

\subsection{The sheaf~$\ccI$ on the real blow-up (smooth divisor case)}\label{subsec:smoothdivcase}\index{sheaf $\ccI$}
Let~$X$ be a smooth complex manifold and let~$D$ be a smooth divisor in~$X$. Let $\varpi:\wt X(D)\to X$ be ``the'' real blow-up space of $X$ along $D$ (associated with the choice of a section $f:\cO_X\to\cO_X(D)$ of~$L(D)$ defining~$D$, \cf \S\ref{subsec:realblowup}). We also denote by $\wti,\wtj$ the inclusions $\partial\wt X(D)\hto\wt X(D)$ and $X^*\hto\wt X(D)$.

In order to construct the sheaf $\ccI$, we will adapt to higher dimensions the construction of Remark \ref{rem:Euler}. We consider the sheaf $\wtj_*\cO_{X^*}$ and its subsheaf $(\wtj_*\cO_{X^*})^\lb$ of locally bounded functions on $\wt X$. We will construct $\ccI$ as a subsheaf of the quotient sheaf $\wtj_*\cO_{X^*}/(\wtj_*\cO_{X^*})^\lb$ (which is supported on $\partial\wt X$). It is the union, over $d\in\NN^*$, of the subsheaves $\ccI_d$ that we define below.

Let us start with \index{$ID$@$\ccI_1$, $\ccI_d$, $\ccI_{\bun}$, $\ccI_{\bmd}$}$\ccI_1$. There is a natural inclusion $\varpi^{-1}\cO_X(*D)\hto\wtj_*\cO_{X^*}$, and $\varpi^{-1}\cO_X=\varpi^{-1}\cO_X(*D)\cap(\wtj_*\cO_{X^*})^\lb$, since a meromorphic function which is bounded in some sector centered on an open set of~$D$ is holomorphic. We then set $\wt\ccI_1=\varpi^{-1}\cO_X(*D)\subset\wtj_*\cO_{X^*}$ and $\ccI_1=\varpi^{-1}(\cO_X(*D)/\cO_X)\subset\wtj_*\cO_{X^*}/(\wtj_*\cO_{X^*})^\lb$.

Locally on~$D$, we can define ramified coverings $\rho_d:X_d\to X$ of order $d$ along~$D$, for any $d$. Let $\wt\rho{}_d:\wt X_d\to\wt X$ be the corresponding covering. The subsheaf $\wt\ccI_d\subset\wtj_*\cO_{X^*}$ of $d$-multivalued meromorphic functions on $\wt X$ is defined as the intersection of the subsheaves $\wtj_*\cO_{X^*}$ and $\wt\rho{}_{d,*}\wt\ccI_{\wt X_d,1}$ of $\wt\rho{}_{d,*}\wtj_{d,*}\cO_{X_d^*}=\wtj_*\rho_{d,*}\cO_{X_d^*}$. As above, we have $\wt\ccI_d\cap(\wtj_*\cO_{X^*})^\lb=\wtj_*\cO_{X^*}\cap\nobreak \wt\rho{}_{d,*}\varpi_d^{-1}\cO_{X_d}$. The sheaf \index{$ID$@$\ccI_1$, $\ccI_d$, $\ccI_{\bun}$, $\ccI_{\bmd}$}$\ccI_d$ is then defined as the quotient sheaf $\wt\ccI_d/\wt\ccI_d\cap\nobreak(\wtj_*\cO_{X^*})^\lb$. This is a subsheaf of $\wtj_*\cO_{X^*}/(\wtj_*\cO_{X^*})^\lb$.

The locally defined subsheaves $\wt\ccI_d$ (and thus $\ccI_d$) glue together as a subsheaf of $\wtj_*\cO_{X^*}$, since the local definition does not depend on the chosen local ramified $d$\nobreakdash-covering. Similarly, $\ccI_d$ exists as a subsheaf of $\wtj_*\cO_{X^*}/(\wtj_*\cO_{X^*})^\lb$ all over~$D$.

\begin{definitio}[Case of a smooth divisor]\label{def:Idivisor}
The sheaf $\wt\ccI$ (\resp \index{$I$@$\ccI$, $\ccIet$}$\ccI$) is the union of the subsheaves~$\wt\ccI_d$ (\resp $\ccI_d$) of $\wtj_*\cO_{X^*}$ (\resp $\wtj_*\cO_{X^*}/(\wtj_*\cO_{X^*})^\lb$) for $d\in\NN^*$.
\end{definitio}

\begin{definitio}[$\ccI$ as a sheaf of ordered abelian groups]\label{def:orderI}
The sheaf $\wti^{-1}\wtj_*\cO_{X^*}$ is naturally ordered by setting $(\wti^{-1}\wtj_*\cO_{X^*})_{\leq0}=\log\cA_{\wt X(D)}^\modD$ (\cf \S\ref{subsec:modgrowth}). In this way,~$\wt\ccI$ inherits an order: $\wt\ccI_{\leq0}=\wt\ccI\cap\log\cA_{\wt X(D)}^\modD$. This order is not altered by adding a local section of $(\wtj_*\cO_{X^*})^\lb$, and thus defines an order on $\ccI$.
\end{definitio}

\begin{lemme}\label{lem:ramifcoverIsmooth}
For any local ramified covering $\rho_d:(X_d,D)\to(X,D)$ of order $d$ along~$D$, $\wt\rho{}_d^{-1}\wt\ccI_d$ is identified with $\varpi_d^{-1}\cO_{X_d}(*D)$ and $\wt\rho{}_d^{-1}\ccI_d$ with $\varpi_d^{-1}(\cO_{X_d}(*D)/\cO_{X_d})$. This identification is compatible with order.
\end{lemme}

\begin{proof}
The proof is completely similar to that given in Remark \ref{rem:Euler}.
\end{proof}

\Subsection{The sheaf~$\ccI$ on the real blow-up (normal crossing case)}\label{subsec:Inormcross}\index{sheaf $\ccI$}

Let us now consider a family $(D_{j\in J})$ of smooth divisors of $X$ whose union~$D$ has only normal crossings, and the corresponding real blowing-up map $\varpi:\wt X(D_{j\in J})\to X$. We will consider multi-integers $\bmd\in(\NN^*)^J$. The definition of the sheaves $\wt\ccI_{\bmd}$ and $\ccI_{\bmd}$ is similar to that in dimension one.

Let us set $\bun=(1,\dots,1)$ ($\#J$ terms) and $\wt\ccI_{\bun}=\varpi^{-1}\cO_X(*D)\subset\wtj_*\cO_{X^*}$. Let us fix $x_o\in D$, let us denote by $D_1,\dots,D_\ell$ the components of~$D$ going through~$x_o$, and set $\wt x_o\in\varpi^{-1}(x_o)\simeq(S^1)^\ell$. Then a local section of $\varpi^{-1}\cO_X(*D)$ near $\wt x_o$ is locally bounded in the neighbourhood of $\wt x_o$ if and only if it is holomorphic in the neighbourhood of~$x_o$. In other words, as in the smooth case, $\varpi^{-1}\cO_X(*D)\cap(\wtj_*\cO_{X^*})^\lb=\varpi^{-1}(\cO_X)$.

We locally define $\wt\ccI_{\bmd}$ near~$x_o$, by using a ramified covering $\rho_{\bmd}$ of $(X,x_o)$ along $(D,x_o)$ of order $\bmd=(d_1,\dots,d_\ell)$, by the formula $\wt\ccI_{\bmd}\defin\wt\rho{}_{\bmd,*}[\varpi_{\bmd,*}\cO_{X_{\bmd}}(*D)]\cap\wtj_*\cO_{X^*}$, and $\ccI_{\bmd}$ by $\ccI_{\bmd}\defin\wt\ccI_{\bmd}/\wt\ccI_{\bmd}\cap(\wtj_*\cO_{X^*})^\lb$.

The locally defined subsheaves $\wt\ccI_{\bmd}$ glue together all over~$D$ as a subsheaf $\wt\ccI_{\bmd}$ of $\wtj_*\cO_{X^*}$. We also set globally \index{$ID$@$\ccI_1$, $\ccI_d$, $\ccI_{\bun}$, $\ccI_{\bmd}$}$\ccI_{\bmd}=\wt\ccI_{\bmd}/\wt\ccI_{\bmd}\cap(\wtj_*\cO_{X^*})^\lb$.

\begin{definitio}\label{def:Igeneral}
The subsheaf~$\wt\ccI\subset\wtj_*\cO_{X^*}$ is the union of the subsheaves $\wt\ccI_{\bmd}$ for $\bmd\in(\NN^*)^J$. The sheaf \index{$I$@$\ccI$, $\ccIet$}$\ccI$ is the subsheaf $\wt\ccI/\wt\ccI\cap(\wtj_*\cO_{X^*})^\lb$ of $\wtj_*\cO_{X^*}/(\wtj_*\cO_{X^*})^\lb$.
\end{definitio}

\begin{definitio}\label{def:Igeneralorder}
The order on~$\wt\ccI$ is given by $\wt\ccI_{\leq0}\defin\wt\ccI\cap\log\cA_{\wt X(D_{j\in J})}^\modD$. It is stable by the addition of an element of $(\wtj_*\cO_{X^*})^\lb$ and defines an order on $\ccI$.
\end{definitio}

\begin{lemme}\label{lem:ramifcoverI}
For any local ramified covering $\rho_{\bmd}:(X_{\bmd},D)\to(X,D)$ of order $\bmd$ along~$(D_{j\in J})$, $\wt\rho{}_{\bmd}^{-1}\wt\ccI_{\bmd}$ is identified with $\varpi_{\bmd}^{-1}\cO_{X_{\bmd}}(*D)$ and $\wt\rho{}_{\bmd}^{-1}\ccI_{\bmd}$ with $\varpi_{\bmd}^{-1}\big(\cO_{X_{\bmd}}(*D)/\cO_{X_{\bmd}}\big)$. These identifications are compatible with order.
\end{lemme}

\begin{proof}
Same proof as in Remark \ref{rem:Euler}.
\end{proof}

\begin{remarque}\label{rem:IstrongHausdorff}
For any subset $I\subset J$, let $D_I$ denote the intersection $\bigcap_{j\in I}D_j$ and set $D_I^{\circ}=D_I\moins\bigcup_{j\in J\moins I}D_j$. Set also $Y_I=\varpi^{-1}(D_I^{\circ})\subset\partial\wt X(D_{j\in J})$. The family $\cY=(Y_I)_{I\subset J}$ is a stratification of $\partial\wt X$ which satisfies the property \eqref{eq:propstrat}. Moreover, the sheaf $\ccI$ is Hausdorff with respect to $\cY$ (this is seen easily locally on~$D$).
\end{remarque}

\subsection{Goodness}\label{subsec:goodness}
The order on sections of~$\ccI$ is best understood for \emph{purely monomial sections} of~$\ccI$. Let us use the following local notation. We consider the case where $X=\Delta^\ell\times\Delta^{n-\ell}$ with base point $0=(0_\ell,0_{n-\ell})$, and $D_i=\{t_i=0\}$ ($i\in L\defin\{1,\dots,\ell\}$) and $D=\bigcup_{i=1}^\ell D_i$. The real blowing-up map $\varpi_L:\wt X(D_{i\in L})=(S^1)^\ell\times[0,1)^\ell\times\Delta^{n-\ell}\to X=\Delta^\ell\times\Delta^{n-\ell}$ is defined by sending $(e^{i\theta_j},\rho_j)$ to $t_j=\nobreak\rho_je^{i\theta_j}$ ($j=1,\dots,\ell$).

\begin{itemize}
\item
In the non-ramified case (\ie we consider sections of $\ccI_{\bun}$), a germ $\eta$ at $\theta\in (S^1)^\ell=(S^1)^\ell\times0_\ell\times0_{n-\ell}\subset(S^1)^\ell\times[0,1)^\ell\times \Delta^{n-\ell}$ of section of $\ccI_{\bun}$ is nothing but a germ at $0\in X$ of section of $\cO_X(*D)/\cO_X$.

\begin{definitio}\label{def:purmonom}
We say that $\eta$ is \emphb{purely monomial} at $0$ if $\eta=0$ or $\eta$ is the class of $t^{-\bmm}u_{\bmm}$, with $\bmm\in\NN^\ell\moins\{0\}$, $u_{\bmm}\in\cO_{X,0}$ and $u_{\bmm}(0)\neq0$. We then set $\bmm=\bmm(\eta)$ with the convention that $\bmm(0)=0$, so that $\bmm(\eta)=0$ iff $\eta=0$.
\end{definitio}

For $\eta\in\cO_{X,0}(*D)/\cO_{X,0}$, written as $\sum_{\bmk\in\ZZ^\ell\times\NN^{n-\ell}}\eta_{\bmk}t^{\bmk}$, the Newton polyhedron \index{$NP$@$\NP$}$\NP(\eta)\subset\RR^\ell\times\RR_+^{n-\ell}$ is the convex hull of $\RR_+^n$ and the octants $\bmk+\RR_+^n$ for which $\eta_{\bmk}\neq0$. Then $\eta$ is purely monomial if and only if $\NP(\eta)$ is an octant $(-\bmm,0_{n-\ell})+\RR_+^n$ with $\bmm\in\NN^\ell$.

If $\eta$ is purely monomial, we have for every $\theta\in (S^1)^\ell$:
\begin{equation}\label{eq:ordern}\index{$AAAORDth$@$\leqtheta$, $\letheta$}
\eta\leqtheta0\ssi \eta=0\text{ or }\arg u_{\bmm}(0)-\ts\sum_jm_j\theta_j\in(\pi/2,3\pi/2)\mod2\pi.
\end{equation}
If $\eta\neq0$ and $\eta$ is purely monomial at $0$, it is purely monomial on some open set $Y=(S^1)^\ell\times V$ (with $V$ an open neighbourhood of $0_{n-\ell}$ in $\Delta^{n-\ell}$, embedded as $0_\ell\times V\subset[0,1)^\ell\times\Delta^{n-\ell}$), and $Y_{\eta\leq0}$ is defined by the inequation $\arg u_{\bmm}(v)-\sum_jm_j\theta_j\in(\pi/2,3\pi/2)\bmod2\pi$, where $v$ varies in the parameter space $V$. For $v$ fixed, it is the inverse image by the fibration map $(S^1)^n\to S^1$, $(e^{i\theta_1},\dots,e^{i\theta_n})\mto e^{i(m_1\theta_1+\cdots+m_n\theta_n)}$, of a set of the kind defined at the end of Example \ref{exem:Stokes}. This set rotates smoothly when~$v$ varies in~$V$. Similarly, the boundary $\St(\eta,0)$ of $Y_{\eta\leq0}$ in $Y$ is the pull-back by the previous map of a subset of $S^1\times V$ which is a finite covering of $V$. It has codimension one in $Y$.
\item
The order on $\ccI_{\bmd}$ is described similarly after a ramification which is cyclic of order~$d_i$ around the component $D_i$ of~$D$.
\end{itemize}

For non-purely monomial elements, we still have the following (\eg in the non-ramified case).

\begin{proposition}\label{prop:suban}
For every $\varphi,\psi\in\Gamma(U,\cO_X(*D)/\cO_X)$, the set $Y_{\psi\leq\varphi}$ is \index{subanalytic}subanalytic in~$Y$ and its boundary \index{$STDIR$@$\St(\psi,\varphi)$}$\St(\varphi,\psi)$ has (real) codimension $\geq1$ in~$Y$.
\end{proposition}

\begin{lemme}\label{lem:averifier}
Let $\eta\in\Gamma(U,\cO_X(*D)/\cO_X)$ and let $x\in D$. Then there exists a projective modification $\epsilon:U'\to U$ of some open neighbourhood $U$ of $x$ in $X$, which is an isomorphism away from~$D$, such that $D'\defin |\epsilon^{-1}(D\cap U)|$ is a reduced divisor with normal crossings and smooth components, and $\eta\circ\epsilon$ is locally purely monomial everywhere on~$D'$.
\end{lemme}

\begin{proof}[\proofname\ of Proposition \ref{prop:suban}]
Set $\eta=\psi-\varphi=t^{-m}u_m$ with $m\in\NN^\ell\moins\{0\}$ and $u_m$ holomorphic. We have $Y_{\psi\leq\varphi}=Y_{\eta\leq0}$. The purely monomial case ($u_m\neq0$ everywhere on $U\cap D$) has been treated above. The statement is local subanalytic on $U$, and for any point of $U$ we replace~$U$ by a subanalytic open neighbourhood of this point, that we still denote by~$X$, on which Lemma \ref{lem:averifier} applies. If we set $D'=\bigcup_jD'_j$, we have natural real analytic proper maps (\cf \S\ref{subsec:realblowup}) $\wt\epsilon:\wt X{}'(D'_{j\in J})\to\wt X$. Let us set $Y'=\wt\epsilon^{-1}(Y)$. Then $Y'_{\eta\circ\epsilon\leq0}=\wt\epsilon^{-1}(Y_{\eta\leq0})$ since $e^\eta$ has moderate growth near $\wt x_o\in Y$ if and only if $e^{\eta\circ\epsilon}$ has moderate growth near any $\wt x'_o\in\wt\epsilon^{-1}(\wt x_o)$, by the properness of $\wt\epsilon$. As a consequence, the set $Y\moins Y_{\eta\leq0}$ is the push-forward by $\wt\epsilon$ of the set $Y'\moins Y'_{\eta\circ\epsilon\leq0}$. Using the purely monomial case considered previously, one shows that $Y'\moins Y'_{\eta\circ\epsilon\leq0}$ is closed and semi-analytic in $Y'$. The subanalyticity of $Y_{\eta\leq0}$ follows then from Hironaka's theorem \cite{Hironaka73b} on the proper images of sub- (or semi-) analytic sets, and stability by complements and closure. The statement for $\St(\varphi,\psi)$ also follows.
\end{proof}

\begin{proof}[Sketch of the proof of Lemma \ref{lem:averifier}]
By using the resolution of singularities in the neighbourhood of $x\in U$, we can find a projective modification $\epsilon_1:U_1\to U$ such that the union of the divisors of zeros and and of the poles of $\eta$ form a divisor with normal crossings and smooth components in $U_1$ (so that, locally in $U_1$ and with suitable coordinates, $\eta\circ\epsilon_1$ takes the form of a monomial with exponents in $\ZZ$). The problem is now reduced to the following question: given a divisor with normal crossings and smooth components $D_1$ in $U_1$, attach to each smooth component an integer (the order of the zero or minus the order of the pole of $\eta\circ\epsilon_1$), so as to write $D_1=D_1^+\cup D_1^0\cup D_1^-$ with respect to the sign of the integer; to any projective modification $\epsilon_2:U_2\to\nobreak U_1$ such that $D_2\defin\epsilon_2^{-1}(D_1)$ remains a divisor with normal crossings and smooth components, one can associate in a natural way a similar decomposition; we then look for the existence of such an $\epsilon_2$ such that $D_2^-$ and $D_2^+$ do not intersect.

We now denote $U_1$ by $X$ and $D_1$ by $D=\bigcup_{j\in J}D_j$. The divisor $D$ is naturally stratified, and we consider minimal (that is, closed) strata. To each such stratum is attached a subset~$L$ of $J$ consisting of indices $j$ for which $D_j$ contains the stratum. Because of the normal crossing condition, the cardinal of this subset is equal to the codimension of the stratum $D_L$. We set $D(L)=\bigcup_{j\in L}D_j$. We will construct the modification corresponding to this stratum with a toric argument. Let us set $\ell=\#L$ and, for each $j\in L$, let us denote by $\cI_j$ the ideal of $D_j$ in $\cO_X$.

As is usual in toric geometry (\cf \cite{Danilov78,Oda86,Fulton93} for instance), we consider the space~$\RR^\ell$ equipped with its natural lattice $N\defin\ZZ^\ell$ and the dual space $(\RR^\vee)^\ell$ equipped with the dual lattice $M$. To each rational cone $\sigma$ in the first octant $(\RR_+)^\ell$ we consider the dual cone $\sigma^\vee\in(\RR^\vee)^\ell$ and its intersection with $M$. This allows us to define a sheaf of subalgebras $\sum_{m\in\sigma^\vee\cap M}\cI_1^{m_1}\cdots\cI_\ell^{m_\ell}$ of $\cO_X(*D(L))$. Locally on~$D_L$, if $D_j$ is defined by $\{x_j=0\}$, this is $\cO_X\otimes_{\CC[x_1,\dots,x_\ell]}\CC[\sigma^\vee\cap M]$, hence this sheaf of subalgebras corresponds to an affine morphism $X_\sigma\to X$. Similarly, to any a fan $\Sigma$ in the first octant $(\RR_+)^\ell$ one associates a morphism $X_\Sigma\to X$, which is a projective modification if the fan completely subdivides the octant. Moreover, if each cone of the fan is strictly simplicial, the space $X_\Sigma$ is smooth and the pull-back of the divisor $D(L)$ has normal crossings with smooth components.

We will now choose the fan $\Sigma$. To each basis vector $e_j$ of $\RR^\ell$ is attached a multiplicity $\nu_j\in\ZZ$, namely the order of $\eta\circ\epsilon_1$ along $D_j$. We consider the trace $H$ on $(\RR_+)^\ell$ of the hyperplane $\{(n_1,\dots,n_\ell)\in\RR^\ell\mid\sum_j\nu_jn_j=0\}$ and we choose a strictly simplicial fan in $(\RR_+)^\ell$ such that $H$ is a union of cones of this fan. For each basis vector $(n_1,\dots,n_\ell)\in\NN^\ell$ of a ray (dimension-one cone) of this fan, the multiplicity of the pull-back of $\eta\circ\epsilon_1$ along the divisor corresponding to this ray is given by $\sum_j\nu_jn_j$. Therefore, for each $\ell$-dimensional cone of the fan, the multiplicities at the rays all have the same sign (or are zero).

The proof of the Lemma now proceeds by decreasing induction on the maximal codimension $\ell$ of closed strata of $D$ which are contained both in $D_+$ and $D_-$. In the space $\RR^J$ we consider the various subspaces~$\RR^L$ ($L\subset J$) corresponding to these closed strata of $D$. We subdivide each octant $(\RR_+)^L$ by a strict simplical fan as above. We also assume that the fans coincide on the common faces of distinct subspaces~$\RR^L$. We denote by $\Sigma$ the fan we obtain in this way. In order to obtain such a $\Sigma$, one can construct a strict simplicial fan completely subdividing $(\RR_+)^J$ which is compatible with the hyperplane $\sum_{j\in J}\nu_j n_j=0$ and with the various octants $(\RR_+)^L$ corresponding to codimension $\ell$ strata of $D$ contained in $D_+\cap D_-$, and then restrict it to the union of these octants $(\RR_+)^L$. We also consider the sheaves $\sum_{m\in\sigma^\vee\cap M_J}\prod_{j\in J}\cI_j^{m_j}\subset\cO_X(*D)$ for $\sigma\in\Sigma$, and get a projective modification $\epsilon_\Sigma:X_\Sigma\to X$. By construction, the maximal codimension of closed strata of $\epsilon_\Sigma^{-1}(D)$ contained in $\epsilon_\Sigma^{-1}(D)_+\cap\epsilon_\Sigma^{-1}(D)_-$ is $\leq\ell-1$.
\end{proof}

For a finite set $\Phi\subset\cO_{X,0}(*D)/\cO_{X,0}$, the notion of pure monomiality is replaced by \emph{goodness} (\cf also Remark \ref{rem:openness}\eqref{rem:openness2} below).

\begin{definitio}[Goodness]\label{def:localgoodness}
We say that a finite subset $\Phi$ of $\cO_{X,0}(*D)/\cO_{X,0}$ is \index{good set of exponential factors}\emph{good} if $\#\Phi=1$ or, for any $\varphi\neq\psi$ in $\Phi$, $\varphi-\psi$ is purely monomial, that is, the Newton polyhedron $\NP(\varphi-\psi)$ is an octant with vertex in $-\NN^\ell\times\{0_{n-\ell}\}$(\cf Definition \ref{def:purmonom}).
\end{definitio}

\begin{remarque}\label{rem:goodness}
Let us give some immediate properties of local goodness (\cf \cite[I.2.1.4]{Bibi97}).
\begin{enumerate}
\item\label{rem:goodness1}
Any subset of a good set is good, and any subset consisting of one element (possibly not purely monomial) is good. If $\Phi$ is good, then its pull-back $\Phi_{\bmd}$ by the ramification $X_{\bmd}\to X$ is good for any $\bmd$. Conversely, if $\Phi_{\bmd}$ is good for some $\bmd$, then $\Phi$ is good.
\item\label{rem:goodness1b}
More generally, let $f:X'\to X$ be a morphism of complex manifolds and let $(D_{j\in J})$ be a family of smooth divisors in $X$ whose union~$D$ is a divisor with normal crossings. We set $D'=f^{-1}(D)$, and we assume that $D'=\bigcup_{j'\in J'}D'_{j'}$ is also a divisor with normal crossings and smooth components $D'_{j'\in J'}$. We will usually denote by $f:(X',D')\to(X,D)$ a mapping satisfying such properties. If $\Phi\subset \cO_{X,x_o}(*D)/\cO_{X,x_o}$ is good, then for any $x'_o\in f^{-1}(x_o)$, the subset $(f^*\Phi)_{x'_o}\subset\cO_{X',x'_o}(*D')/\cO_{X',x'_o}$ is good.
\item\label{rem:goodness2}
Any germ $\varphi_0$ of $\cO_{X,0}(*D)/\cO_{X,0}$ defines in a unique way a germ $\varphi_x\in \cO_{X,x}(*D)/\cO_{X,x}$ for $x\in D\cap U$, $U$ some open neighbourhood of~$0$. Indeed, choose a lifting $\varphi_0^*$ in $\cO_{X,0}(*D)$. It defines in a unique way a section $\varphi^*$ of $\Gamma(U,\cO_U(*D))$ for~$U$ small enough. Its germ at $x\in U\cap D$ is denoted $\varphi_x^*$. Its image in $\cO_{X,x}(*D)/\cO_{X,x}$ is~$\varphi_x$. Given two liftings $\varphi_0^*$ and $\varphi_0^\star$, their difference is in $\cO_{X,0}$. Choose $U$ so that $\varphi^*$ and $(\varphi^*-\varphi^\star)$ are respectively sections of $\Gamma(U,\cO_U(*D))$ and $\Gamma(U,\cO_U)$. Then these two liftings give the same $\varphi_x$.

Similarly, any finite subset $\Phi$ of $\cO_{X,0}(*D)/\cO_{X,0}$ defines in a unique way a finite subset, still denoted by $\Phi$, of $\cO_{X,x}(*D)/\cO_{X,x}$ for any $x$ close enough to~$0$.

Then, if $\Phi$ is is good at~$0$, it is good at any point of~$D$ in some open neighbourhood of~$0$. Note however that a difference $\varphi-\psi$ which is non-zero at $0$ can be zero along some components of $D$, a phenomenon which causes $\ccIet$ to be non-Hausdorff.

\item\label{rem:goodness3}
A subset $\Phi$ is good at~$0$ if and only if for some (or any) $\eta\in \cO_{X,0}(*D)/\cO_{X,0}$ the translated subset $\Phi+\eta$ is good.

\item\label{rem:goodness4}
For a good set $\Phi\subset\cO_{X,0}(*D)/\cO_{X,0}$, and for any \emph{fixed} $\varphi_o\in \Phi$, the subset $\{\bmm(\varphi-\varphi_o)\mid\varphi\in \Phi\}\subset\NN^\ell$ is totally ordered (\ie the Newton polyhedra $\NP(\varphi-\varphi_0)$ form a nested family). Its maximum does not depend on the choice of $\varphi_o\in \Phi$, it is denoted by $\bmm(\Phi)$ and belongs to $\NN^\ell$. We have $\bmm(\Phi)=0$ iff $\#\Phi=1$. We have $\bmm(\Phi+\eta)=\bmm(\Phi)$ for any $\eta\in \cO_{X,0}(*D)/\cO_{X,0}$.

\item\label{rem:goodness5}
Assume $\Phi$ is good. Let us fix $\varphi_o\in\Phi$ and set $\bmm:=\bmm(\Phi)$. The set $\{\bmm(\psi-\varphi_o)\mid\psi\in\Phi\}$ is totally ordered, and we denote by $\bell=\bell_{\varphi_o}$ its submaximum. For any $\varphi\neq\varphi_o$ in $\Phi$ such that $\bmm(\varphi-\varphi_o)=\bmm$, we denote by $[\varphi-\varphi_o]_{\bell}$ the class of $\varphi-\varphi_o$ in $\cO_{X,0}(*D)/\cO_{X,0}(\sum \ell_iD_i)$. For such a $\varphi$, the set
\[
\Phi_{[\varphi-\varphi_o]_{\bell}}\defin\{\psi\in\Phi\mid [\psi-\varphi_o]_{\bell}=[\varphi-\varphi_o]_{\bell}\}\subset \Phi
\]
is good at~$0$, and $\bmm(\Phi_{[\varphi-\varphi_o]_{\bell}})<\bmm=\bmm(\Phi)$. Indeed, for any $\psi\in\Phi_{[\varphi -\varphi_o]_\bell}$, $\psi -\varphi_o $ can be written as $\varphi-\varphi_o+t^{-\bell}u_\psi(t)$ with $u_\psi(t)\in\CC\{t\}$, and the difference of two such elements $\psi,\eta\in\Phi_{[\varphi -\varphi_o]_\bell}$ is also written as $t^{-\bell}v(t)$ with $v(t)\in\CC\{t\}$.
\end{enumerate}
\end{remarque}

Let \index{$SZIGMAWT$@$\wt\Sigma$}$\wt\Sigma_L\subset\ccIet_{|Y_L}$ be a finite covering of $Y_L$ (\cf\S\ref{subsec:Inormcross}). There exists $\bmd$ such that the pull-back $\wt\Sigma_{L,\bmd}$ of $\wt\Sigma_L$ by the ramification $\wt X_{\bmd}\to \wt X$ is a trivial covering of $(S^1)^\ell\times D_L$. Hence there exists a finite set $\Phi_{\bmd}\subset \cO_{X_{\bmd}}(*D)/\cO_{X_{\bmd}}$ such that the restriction of $\wt\Sigma_{L,\bmd}$ over $(S^1)^\ell\times\{0\}$ is equal to $\Phi_{\bmd}\times(S^1)^\ell\times\{0\}$. We say that $\wt\Sigma_L$ is \emph{good} at $0\in D_L$ if the corresponding subset $\Phi_{\bmd}$ is good for some (or any) $\bmd$ making the covering trivial. By the previous remark, if $\wt\Sigma_L$ is good at $0$, it is good in some neighbourhood of $0\in D_L$. Moreover, if $f:X'\to X$ is as in \ref{rem:goodness}\eqref{rem:goodness1b}, if $\wt\Sigma_L$ is good at $0$ then $\wt f^*(\wt\Sigma_L)$ is good at each point of $f^{-1}(0)$.

Let us now consider a stratified $\ccI$-covering \index{$SZIGMAWT$@$\wt\Sigma$}$\wt\Sigma\subset\ccIet$ of $\wt X$ (\cf Definition \ref{def:Sigma}).

\begin{lemme}\label{lem:Sigmagoodouvert}
If $\wt\Sigma_L$ is good at $0\!\in\!D_L$, each $\wt\Sigma_I$ is good on some neighbourhood of~$0$.
\end{lemme}

\begin{proof}
Assume first that $\wt\Sigma_L\to Y_L=(S^1)^L\times\Delta^{n-\ell}$ is trivial and thus $\wt\Sigma_L=\Phi\times Y_L$ for some finite good set $\Phi\subset\cO_{X,0}(*D)/\cO_{X,0}$. Then, if $\Delta^n$ is small enough, we have $\wt\Sigma=\bigcup_{\varphi\in\Phi}\varphi(\wt\Delta^n)$ and the assertion follows from Remark \ref{rem:goodness}\eqref{rem:goodness2}.

In general, one first performs a suitable ramification around the components of $D$ to reduce to the previous case.
\end{proof}

Lastly, let us consider the global setting, where $X$ is a complex manifold and $(D_{j\in J})$ is a family of smooth divisors on~$X$ which intersect normally, and the sheaf of ordered abelian groups $\ccI$ on $\partial\wt X(D_{j\in J})$ is as in Definitions \ref{def:Igeneral} and \ref{def:Igeneralorder}. For any nonempty subset $I$ of $J$, we set $D_I=\bigcap_{i\in I}D_i$ and $D_I^{\circ}=D_I\moins\bigcup_{j\in J\moins I}D_j$. The family $(D_I^{\circ})_{\emptyset\neq I\subset J}$ is a Whitney stratification of $D=\bigcup_{j\in J}D_j$.

\begin{definitio}[Global goodness]\label{def:globalgoodness}
Let us consider a stratified $\ccI$-covering $\wt\Sigma\subset\ccIet$. We say that it is \index{good stratified $\ccI$-covering}\emph{good} if each $\wt\Sigma_I$ is good at each point of $D_I^{\circ}$.
\end{definitio}

\subsection{Stokes filtrations on local systems}\label{subsec:Stokesstr}
As above, $(D_{j\in J})$ is a family of smooth divisors on~$X$ which intersect normally, and the sheaf of ordered abelian groups $\ccI$ on $\partial\wt X(D_{j\in J})$ is as in Definitions \ref{def:Igeneral} and \ref{def:Igeneralorder}.

\begin{definitio}[Stokes-filtered local system]\label{def:Stokesfilt}
Let $\cL$ be a local system of $\kk$-vector spaces on $\partial\wt X(D_{j\in J})$. A~\index{Stokes filtration}\emphb{Stokes filtration} of $\cL$ is a~$\ccI$-filtration of $\cL$, in the sense of Definition \ref{def:Ifiltstrat}. We denote by $(\cL,\cL_\bbullet)$ a \index{Stokes-filtered local@Stokes-filtered local system}\index{local system!Stokes-filtered}Stokes-filtered local system.
\end{definitio}

\begin{remarque}
We will freely extend in the present setting the notation of \Chaptersname \ref{chap:Stokesone} and use some easy properties considered there. In the non-ramified case for instance, the sheaves $\cL_{\leq\varphi}$ are $\RR$-constructible, according to Proposition \ref{prop:suban}.
\end{remarque}

\begin{definitio}[Goodness]\label{def:goodness}
We say that a \index{Stokes-filtered local@Stokes-filtered local system!good}Stokes-filtered local system $(\cL,\cL_\bbullet)$ is \emph{good} if its associated stratified $\ccI$-covering $\wt\Sigma(\cL)\subset\ccIet$, which is the union of the supports of the various $\gr\cL_{|Y_I}$ (with $Y_I=\varpi^{-1}(D_I^{\circ})$), is good.
\end{definitio}

\begin{theoreme}
Let us fix a good stratified $\ccI$-covering $\wt\Sigma\subset\ccIet$ and let $(\cL,\cL_\bbullet)$, $(\cL',\cL'_\bbullet)$ be Stokes-filtered local systems on $\wt X(D_{j\in J})$ whose associated $\ccI$\nobreakdash-stratified coverings $\wt\Sigma(\cL),\wt\Sigma(\cL')$ are contained in $\wt\Sigma$. Let $\lambda:(\cL,\cL_\bbullet)\to(\cL',\cL'_\bbullet)$ be a morphism of local systems which is compatible with the Stokes filtrations. Then~$\lambda$ is \index{morphism!strict}\index{strict morphism}strict.
\end{theoreme}

\begin{corollaire}
If $\wt\Sigma$ is good, the category of Stokes-filtered local systems satisfying $\wt\Sigma(\cL)\subset\wt\Sigma$ is \index{abelian (category)}abelian.
\end{corollaire}

This will be a consequence of the following generalization of Theorem \ref{th:abelianwithout}.

\begin{proposition}\label{prop:abelianwithout}
Assume that $\wt\Sigma(\cL),\wt\Sigma(\cL')\subset\wt\Sigma$ for some good stratified $\ccI$-covering $\wt\Sigma\subset\ccIet$. Let~$\lambda$ be a morphism of local systems compatible with the Stokes filtrations. Then, in the neighbourhood of any point of $\partial\wt X(D_{j\in J})$ there exist gradations of the Stokes filtrations such that the morphism is diagonal with respect to them. In particular, it is \index{morphism!strict}\index{strict morphism}strict, and the natural~$\ccI$-filtrations on the local systems $\ker \lambda$, $\im \lambda$ and $\coker \lambda$ are good Stokes-filtered local systems. Their associated stratified $\ccI$-coverings satisfy
\[
\wt\Sigma(\ker \lambda)\subset \wt\Sigma(\cL),\quad \wt\Sigma(\coker \lambda)\subset \wt\Sigma(\cL'),\quad \wt\Sigma(\im \lambda)\subset \wt\Sigma(\cL)\cap \wt\Sigma(\cL').
\]
\end{proposition}

\subsubsection*{Preliminary reductions}
As in the one-dimensional case, one reduces to the non-ramified case by a suitable $\bmd$-cyclic covering. Moreover, the strictness property is checked on the germs at any point of $\partial\wt X$, so we can work in the local setting of \S\ref{subsec:Inormcross} and restrict the filtered local system to the torus $(S^1)^\ell$. We will moreover forget about the term $\Delta^{n-\ell}$ and assume that $\ell=n$.

In the following, we will give the proof for a~$\ccI$-local system on the torus $(S^1)^n$ in the non-ramified case, that is,~$\ccI$ is the constant sheaf with fibre $\ccP_n=\CC\{t_1,\dots,t_n\}[(t_1\cdots t_n)^{-1}]/\CC\{t_1,\dots,t_n\}$. As this sheaf satisfies the Hausdorff property, many of the arguments used in the proof of Theorem \ref{th:abelianwithout} can be extended in a straightforward way for Proposition \ref{prop:abelianwithout}. Nevertheless, we will give the proof with details, as the goodness condition is new here.

\subsubsection*{Level structure of a Stokes filtration}\index{level structure}
For every $\bell\in\NN^n$, we define the notion of \index{Stokes filtration!oflevelc@of level $\geq\bell$}\emph{Stokes filtration of level $\geq\bell$} on $\cL$, by replacing the set of indices $\ccP_n=\CC[t,t^{-1}]/\CC[t]$ ($t=t_1,\dots,t_n$) by the set $\ccP_n(\bell)\defin\CC[t,t^{-1}]/t^{-\bell}\CC[t]$ (with $t^{-\bell}\defin t_1^{-\bell_1}\cdots t_n^{-\bell_n}$). We denote by $[{\cdot}]_\bell$ the map $\CC[t,t^{-1}]/\CC[t]\to\CC[t,t^{-1}]/t^{-\bell}\CC[t]$. The constant sheaf \index{$IELL$@$\ccI_1(\ell)$, $\ccI_1(\bell)$}$\ccI(\bell)$ is ordered as follows: for every connected open set~$U$ of $(S^1)^n$ and $[\varphi]_\bell,[\psi]_\bell\in\ccP_n(\bell)$, we have $[\psi]_\bell\leqU[\varphi]_\bell$ if, for some (or any) representatives $\varphi,\psi$ in $\CC[t,t^{-1}]$, $e^{|t|^\bell(\psi-\varphi)}$ has moderate growth along~$D$ in a neighbourhood of~$U$ in~$X$ intersected with~$X^*$. In particular, a Stokes filtration as defined previously has level $\geq0$.

\begin{lemme}
The natural morphism $\ccI\to\ccI(\bell)$ is compatible with the order.
\end{lemme}

\begin{proof}
Let~$U$ be a connected open set in $(S^1)^n$ and let $K$ be a compact set in~$U$. Let $\eta\in\ccP_n$ (or a representative of it). We have to show that if $e^\eta$ has moderate growth along~$D$ on $\nb(K)\moins D$, then so does $e^{|t|^\bell\eta}$. Let $\epsilon:X'\to X$ be a projective modification as in the proof of Proposition \ref{prop:suban}, let $\wt\epsilon$ be the associated morphism of real blow-up spaces, and let $K'\subset Y'$ be the inverse image of $K$ in $Y'$. Then $e^\eta$ has moderate growth along~$D$ on $\nb(K)\moins D$ iff $e^{\eta\circ\epsilon}$ has moderate growth along~$D'$ on $\nb(K')\moins D'\simeq\nb(K)\moins D$. It is therefore enough to prove the lemma when $\eta$ is purely monomial, and the result follows from \eqref{eq:ordern}.
\end{proof}

Given a Stokes filtration $(\cL,\cL_\bbullet)$ (of level $\geq0$), we set
\[
\cL_{\leq[\varphi]_\bell}=\sum_\psi\beta_{[\psi]_\bell\leq[\varphi]_\bell}\cL_{\leq\psi},
\]
where the sum is taken in $\cL$. Then
\[
\cL_{<[\varphi]_\bell}\defin\sum_{[\psi]_\bell}\beta_{[\psi]_\bell<[\varphi]_\bell}\cL_{\leq[\psi]_\bell}
=\sum_\psi\beta_{[\psi]_\bell<[\varphi]_\bell}\cL_{\leq\psi}.
\]

We can also pre-$\ccI$-filter $\gr_{[\varphi]_\bell}\cL$ by setting, for $\psi\in\ccP_n$,
\[
(\gr_{[\varphi]_\bell}\cL)_{\leq\psi}=(\cL_{\leq\psi}\cap\cL_{\leq[\varphi]_\bell}+\cL_{<[\varphi]_\bell})/\cL_{<[\varphi]_\bell}.
\]

\begin{proposition}\label{prop:leveln}
Assume $(\cL,\cL_\bbullet)$ is a Stokes filtration (of level $\geq0$) and let $\Phi$ be the finite set of its exponential factors.
\begin{enumerate}
\item\label{prop:leveln1}
For each $\bell\in\NN^n$, $\cL_{\leq[\cbbullet]_\bell}$ defines a Stokes filtration $(\cL,\cL_{[\cbbullet]_\bell})$ of level \hbox{$\geq\bell$} on~$\cL$, $\gr_{[\varphi]_\bell}\cL$ is locally isomorphic to $\bigoplus_{\psi,\,[\psi]_\bell=[\varphi]_\bell}\gr_\psi\cL$, and the set of exponential factors of $(\cL,\cL_{[\cbbullet]_\bell})$ is $\Phi(\bell)\defin\image(\Phi\to\ccP_n(\bell))$.
\item\label{prop:leveln2}
For every $[\varphi]_\bell\in \Phi(\bell)$, $(\gr_{[\varphi]_\bell}\cL,(\gr_{[\varphi]_\bell}\cL)_\bbullet)$ is a Stokes filtration and its set of exponential factors is the pull-back of $[\varphi]_\bell$ by $\Phi\to \Phi(\bell)$.
\item\label{prop:leveln3}
Let us set
\[
\big(\gr_\bell\cL,(\gr_\bell\cL)_\bbullet\big)\defin
\bigoplus_{[\psi]_\ell\in \Phi(\bell)}\big(\gr_{[\psi]_\bell}\cL,(\gr_{[\psi]_\bell}\cL)_\bbullet\big).
\]
Then $(\gr_\bell\cL,(\gr_\bell\cL)_\bbullet)$ is a Stokes-filtered local system (of level $\geq0$) which is locally isomorphic to $(\cL,\cL_\bbullet)$.
\end{enumerate}
\end{proposition}

\begin{proof}
Similar to that of Proposition \ref{prop:level}.
\end{proof}

\begin{remarque}\label{rem:stepconstrbell}
Similarly to Remark \ref{rem:stepconstr}, we note that, as a consequence of the last statement of the proposition, given a fixed Stokes-filtered local system $(\cG_\bell,\cG_{\bell,\bbullet})$ graded at the level $\bell\geq0$, the pointed set of isomorphism classes of Stokes-filtered local systems $(\cL,\cL_\bbullet)$ equipped with an isomorphism $f_\bell:(\gr_\bell\cL,(\gr_\bell\cL)_\bbullet)\isom(\cG_\bell,\cG_{\bell,\bbullet})$ is in bijection with the pointed set $H^1\big((S^1)^\ell,\cAut^{<0}(\cG_\bell,\cG_{\bell,\bbullet})\big)$.
\end{remarque}

\begin{proof}[\proofname\ of Proposition \ref{prop:abelianwithout}]
Let $\Phi$ be a good finite set in $\ccP_n$ such that $\#\Phi\geq2$. As in Remark \ref{rem:goodness}\eqref{rem:goodness4}, let us set $\bmm=\bmm(\Phi)=\max\{\bmm(\varphi-\psi)\mid\varphi\neq\psi\in \Phi\}$ and let us fix $\varphi_o\in \Phi$ for which there exists $\varphi\in \Phi$ such that $\bmm(\varphi-\varphi_o)=\bmm$. The subset $\{\bmm(\varphi-\varphi_o)\mid\varphi\in \Phi\}$ is totally ordered, its maximum is~$\bmm$, and we denote by $\bell$ its submaximum (while $\bmm$ is independent of $\varphi_o$, $\bell$ may depend on the choice $\varphi_o$).

Let $\varphi\in \Phi-\varphi_o$. If $\bmm(\varphi)=\bmm$, then the image of $\varphi$ in $(\Phi-\varphi_o)(\bell)$ is nonzero, otherwise $\varphi$ is zero. For every $\varphi\in \Phi-\varphi_o$, the subset $(\Phi-\varphi_o)_{[\varphi]_\bell}\defin\{\psi\in\nobreak \Phi-\nobreak\varphi_o\mid\nobreak[\psi]_\bell=\nobreak[\varphi]_\bell\}$ is good (any subset of a good set is good) and $\bmm((\Phi-\varphi_o)_{[\varphi]_\bell})\leq\bell<\bmm$ (\cf Remark \ref{rem:goodness}\eqref{rem:goodness5}).

\begin{corollaire}[of Prop\ptbl\ref{prop:leveln}]\label{cor:leveln}
Let $(\cL,\cL_\bbullet)$ be a good Stokes filtration, let $\Phi''$ be a good finite subset of $\ccP_n$ containing $\Phi(\cL,\cL_\bbullet)$, set $\bmm=\bmm(\Phi'')$, fix $\varphi_o$ as above and let $\bell$ be the corresponding submaximum element of $\{\bmm(\varphi-\varphi_o)\mid\varphi\in \Phi''\}$.

Then, for every $[\varphi]_{\bell}\in (\Phi''-\varphi_o)(\bell)$, $(\gr_{[\varphi]_{\bell}}\cL[-\varphi_o],(\gr_{[\varphi]_{\bell}}\cL[-\varphi_o])_\bbullet)$ is a good Stokes filtration and $\bmm_{\max}(\gr_{[\varphi]_{\bell}}\cL[-\varphi_o],(\gr_{[\varphi]_{\bell}}\cL[-\varphi_o])_\bbullet)\leq\bell<\bmm$.\qed
\end{corollaire}

Let us fix $\theta_o\in(S^1)^n$ and $\alpha_1,\dots,\alpha_n\in\NN^*$ such that $\gcd(\alpha_1,\dots,\alpha_n)=1$. The map $\theta\mto(\alpha_1\theta+\theta_{o,1},\dots,\alpha_n\theta+\theta_{o,n})$ embeds~$S^1$ in $(S^1)^n$. In the following, $S^1_{\alphag,\theta_o}$ denotes this circle.

Let $\Phi\subset\ccP_n$ be good finite set. Let us describe the Stokes hypersurfaces $\St(\varphi,\psi)$ with $\varphi\neq\psi\in\Phi$. Since $\varphi-\psi$ is purely monomial, it is written $u_{\bmm}(t)t^{-\bmm}$ with $\bmm=(m_1,\dots,m_n)\in\NN^n\moins\{0\}$ and $u_{\bmm}(0)\neq0$. Then
\[
\St(\varphi,\psi)=\Big\{(\theta_1,\dots,\theta_n)\in(S^1)^n\mid\ts\sum_j m_j\theta_j-\arg u_{\bmm}(0)=\pm\pi/2\bmod2\pi\Big\},
\]
so in particular it is the union of translated subtori of codimension one. As a consequence, the circle $S^1_{\alphag,\theta_o}$ intersects transversally every Stokes hypersurface. We call \emphb{Stokes points} with respect to $\Phi$ the intersection points when $\varphi,\psi$ vary in $\Phi$.

\begin{lemme}\label{lem:grandintervallen}
Let~$I$ be any open interval of $S^1_{\alphag,\theta_o}$ such that, for any $\varphi,\psi\in \Phi$, $\card(I\cap\nobreak\St(\varphi,\psi))\leq1$. Then there exists an open neighbourhood $\nb(I)$ such that the decompositions \eqref{eq:decI} hold on~$\nb(I)$.
\end{lemme}

\begin{proof}
A proof similar to that of Lemma \ref{lem:grandintervalle} gives that $H^1(I,\cL_{<\psi|I})=\nobreak0$ for any~$\psi$. We can then lift for any $\psi\in\Phi$ a basis of global sections of $\gr_\psi\cL_{|I}$ as a family sections of $\cL_{\leq\psi|I}$, which are defined on some $\nb(I)$. The images of these sections in $\gr_\psi\cL_{|\nb(I)}$ restrict to the given basis of $\gr_\psi\cL_{|I}$ and thus form a basis of $\gr_\psi\cL_{|\nb(I)}$ if $\nb(I)$ is simply connected, since $\gr_\psi\cL$ is a locally constant sheaf. We therefore get a section $\gr_\psi\cL_{|\nb(I)}\to\cL_{\leq\psi|\nb(I)}$ of the projection $\cL_{\leq\psi|\nb(I)}\to\gr_\psi\cL_{|\nb(I)}$.

For every $\varphi\in\Phi$, we have a natural inclusion $\beta_{\psi\leq\varphi}\cL_{\leq\psi}\hto\cL_{\leq\varphi}$, and we deduce a morphism $\bigoplus_{\psi\in \Phi}\beta_{\psi\leq\varphi}\gr_\psi\cL_{|\nb(I)}\to\cL_{\leq\varphi|\nb(I)}$, which is seen to be an isomorphism on stalks at points of $I$, hence on a sufficiently small $\nb(I)$, according to the local decomposition \eqref{eq:decI}. The same result holds then for any $\eta\in\ccP_n$ instead of $\varphi$, since $\cL_{\leq\eta}=\sum_{\varphi\in\Phi}\beta_{\varphi\leq\eta}\cL_{\leq\varphi}$ and similarly for the graded pieces.
\end{proof}

\begin{corollaire}\label{cor:trivIn}
In the setting of Corollary \ref{cor:leveln}, let us set $\bmm=(m_1,\dots,m_n)$ and $m=\sum_im_i\alpha_i$. Let~$I$ be any open interval of~$S^1$ of length $\pi/m$ with no Stokes points as boundary points. Then, if $\nb(I)$ is a sufficiently small tubular neighbourhood of~$I$, $(\cL,\cL_\bbullet)_{|\nb(I)}\simeq(\gr_\bell\cL,(\gr_\bell\cL)_\bbullet)_{|\nb(I)}$.
\end{corollaire}

\begin{proof}
By the choice of $\bmm$ and the definition of $m$,~$I$ satisfies the assumption of Lemma \ref{lem:grandintervallen} for both $(\cL,\cL_\bbullet)$ and $(\gr_\bell\cL,(\gr_\bell\cL)_\bbullet)$, hence, when restricted to $\nb(I)$, both are isomorphic to the trivial Stokes filtration determined by $\gr\cL$ restricted to~$\nb(I)$.
\end{proof}

\subsubsection*{End of the proof of Proposition \ref{prop:abelianwithout}}
Let $\lambda:(\cL,\cL_\bbullet)\to(\cL',\cL'_\bbullet)$ be a morphism of Stokes-filtered local systems on $(S^1)^n$ with set of exponential factors contained in $\Phi''$. The proof that~$\lambda$ is strict and that $\ker \lambda$, $\im \lambda$ and $\coker \lambda$ (equipped with the naturally induced pre-$\ccP_n$-filtrations) are Stokes filtrations follows from the local decomposition of the morphism, the proof of which is be done by induction on $\bmm=\bmm(\Phi'')$, with $\Phi''=\Phi\cup \Phi'$. The result is clear if $\bmm=0$ (so $\Phi''=\{0\}$), as both Stokes filtrations have only one jump. The remaining part of the inductive step is completely similar to the end of the proof of Theorem \ref{th:abelianwithout} by working on $\nb(I)$ instead of $I$, and we will not repeat it. We obtain that, for any such $I$, $\lambda_{|\nb(I)}$ is graded, so this ends the proof of the proposition.
\end{proof}

\subsection{Behaviour by pull-back}\label{subsec:pullback}\index{pull-back (inverse image)!of Stokes filtrations}
Let $f:(X',D')\to(X,D)$ be a mapping as in Remark \ref{rem:goodness}\eqref{rem:goodness1b}. According to \S\ref{subsec:realblowup}, there is a natural morphism $\wt f:\wt X{}'(D'_{j'\in J'})\to\wt X(D_{j\in J})$ lifting $f:X'\to X$. There are natural inclusions $\wtj:X\moins D=X^*\hto\wt X(D_{j\in J})$ and $\wtj{}':X'\moins D'=X^{\prime*}\hto\wt X{}'(D'_{j'\in J'})$, and we have $\wt f\circ\wtj{}'=\wtj\circ f$.

Let us describe such a mapping in a local setting: the space $(X,D)$ is the polydisc $\Delta^n$ with coordinates $(x_1,\dots,x_n)$ and $D=\{x_1\cdots x_\ell=0\}$, and similarly for $(X',D')$, and $f(0)=0$ in these coordinates. We have coordinates $(\theta_1,\dots,\theta_\ell,\rho_1,\dots,\rho_\ell,x_{\ell+1},\dots,x_n)$ on $\wt X$, and similarly for $\wt X'$. In these local coordinates, we set $f=(f_1,\dots,f_n)$, with
\begin{equation}\label{eq:formefXD}
f_1(x')=u'_1(x')x^{\prime\kk_1},\dots,f_\ell(x')=u'_\ell(x')x^{\prime\kk_\ell},
\end{equation}
where $u'_j(x')$ are local units, and $\kk_j=(k_{j,1},\dots,k_{j,\ell'})\in\NN^{\ell'}\moins\{0\}$. We also have $f_j(0)=0$ for $j\geq\ell+1$. We note that the stratum $D'_{L'}$ going through the origin in~$X'$ (defined by $x'_1=\cdots=x'_{\ell'}=0$) is sent to the stratum $D_L$ going to the origin in~$X$ (defined by $x_1=\cdots=x_\ell=0$), maybe not submersively.

When restricted to $\varpi^{\prime-1}(D'_{L'})$ defined by the equations $\rho'_{j'}=0$, $j'=1,\dots,\ell'$, the map $\wt f$ takes values in $\varpi^{-1}(D_L)$ and is given by the formula (with $\rho_1=\cdots=\rho_\ell=0$):
\begin{equation}\label{eq:formewtfXD}
(\theta'_1,\dots,\theta'_{\ell'},x'_{\ell'+1},\dots,x'_{n'})\mto
\begin{pmatrix}
\sum k_{1,i}\theta'_i+\arg u'_1(0,x'_{\ell'+1},\dots,x'_{n'})\\
\vdots\\
\sum k_{\ell,i}\theta'_i+\arg u'_\ell(0,x'_{\ell'+1},\dots,x'_{n'})\\
f_{\ell+1}(0,x'_{\ell'+1},\dots,x'_{n'})\\
\vdots\\
f_n(0,x'_{\ell'+1},\dots,x'_{n'})
\end{pmatrix}
\end{equation}

Going back to the global setting, we have a natural morphism $f^*:\wt f{}^{-1}\wtj_*\cO_{X^*}\to\wtj{}'_*\cO_{X^{\prime*}}$, which sends $\wt f{}^{-1}(\wtj_*\cO_{X^*})^\lb$ to $(\wtj{}'_*\cO_{X^{\prime*}})^\lb$. This morphism is compatible with the order: it sends $\wt f{}^{-1}\cA^\modD_{\wt X(D_{j\in J})}$ to $\cA^{\rmod D'}_{\wt X{}'(D'_{j'\in J'})}$.

\medskip
We consider the sheaves $\wt\ccI$ and $\ccI$ on $\wt X$ and $\wt\ccI{}'$ and $\ccI'$ on $\wt X{}'$, relative to the divisors~$D$ and $D'$.

\begin{proposition}\label{prop:f*inj}
The morphism $f^*$ sends $\wt f^{-1}\wt\ccI$ to $\wt\ccI{}'$ and induces a morphism $f^*:\wt f^{-1}\ccI\to\ccI'$, which is compatible with the order. Moreover, if $\wt f:\wt X{}'\to\wt X$ is open, the morphism $f^*$ is injective and strictly compatible with the order.
\end{proposition}

\begin{remarque}\label{rem:wtfopen}
If $\dim X=1$, then $\wt f$ is \emph{open}. Indeed, there are local coordinates in~$X'$ where~$f$ is expressed as a monomial. The assertion is easy to see in this case.
\end{remarque}

\begin{proof}[\proofname\ of Proposition \ref{prop:f*inj}]
Let us prove the first statement. It is clear that $f^*$ sends $\wt f^{-1}\wt\ccI_{\bun}$ to $\wt\ccI{}'_{\bun}$. In general, we note that the assertion is local on $X$ and $X'$ and, given a local ramified covering $\rho_{\bmd}:X_{\bmd}\to X$, there is a commutative diagram
\[
\xymatrix{
X'_{\bmd'}\ar[d]_g\ar[r]^-{\rho'_{\bmd'}}&X'\ar[d]^f\\
X_{\bmd}\ar[r]^-{\rho_{\bmd}}&X
}
\]
for a suitable $\bmd'$ (this is easily seen in local coordinates in $X$ and $X'$ adapted to~$D$ and $D'$). The morphism $\rho_{\bmd'}^{\prime-1}f^*:\rho_{\bmd'}^{\prime-1}f^{-1}\cO_{X^*}\to\rho_{\bmd'}^{\prime-1}\cO_{X^{\prime*}}$ is identified with $g^*:g^{-1}\cO_{X_{\bmd}^*}\to\cO_{X^{\prime*}_{\bmd'}}$ since $\rho_{\bmd}$ and $\rho'_{\bmd'}$ are coverings, and $f^*$ is recovered from $g^*$ as the restriction of $\rho'_{\bmd',*}(g^*)$ to $\cO_{X^{\prime*}}\subset\rho'_{\bmd',*}\cO_{X^{\prime*}_{\bmd'}}$.

As we know that $g^*$ sends $\wt g^{-1}\varpi_{\bmd}^{-1}\cO_{X_{\bmd}}(*D)$ to $\varpi_{\bmd'}^{\prime-1}\cO_{X'_{\bmd'}}(*D')$, we conclude by applying $\rho'_{\bmd',*}$ and intersecting with $\wtj{}'_*\cO_{X^{\prime*}}$ that $f^*$ sends $\wt f^{-1}\wt\ccI_{\bmd}$ to $\wt\ccI{}'_{\bmd'}$, hence in~$\wt\ccI{}'$.

For the injectivity statement, it is enough to prove that, if $\wt f:\wt X{}'\to \wt X$ is open, then $f^*:\wt f{}^{-1}\wtj_*\cO_{X^*}/\wt f{}^{-1}(\wtj_*\cO_{X^*})^\lb\to\wtj{}'_*\cO_{X^{\prime*}}/(\wtj{}'_*\cO_{X^{\prime*}})^\lb$ is injective. Let $y'\in\partial\wt X{}'$ and set $y=\wt f(y')\in\partial\wt X$. Note first that $\partial\wt X{}'=\wt f{}^{-1}(\partial\wt X)$. If $V(y')$ is an open neighbourhood of $y'$ in $\wt X{}'$, then $\wt f(V(y'))$ is an open neighbourhood of $y$ in $\wt X$ by the openness assumption, and we have
\[
\wt f(V(y'))\moins\partial\wt X=\wt f\big(V(y')\moins\wt f^{-1}(\partial\wt X)\big)=\wt f(V(y')\moins\partial\wt X{}').
\]
If $\lambda$ is a local section of $\wt f{}^{-1}\wtj_*\cO_{X^*}$ at $y'$, it is defined on such a $V(y')$. If $f^*\lambda$ is bounded on $V(y')\moins\partial\wt X{}'$, then $\lambda$ is a bounded section of $\cO_{X^*}$ on the open set $\wt f(V(y')\moins\partial\wt X{}')=\wt f(V(y'))\moins\partial\wt X$, hence is a local section of $\wt f{}^{-1}(\wtj_*\cO_{X^*})^\lb$.

Let us show the strictness property. It means that, for any $y'\in \wt f^{-1}(y)$, $\varphi\circ f\leqyp0$ implies $\varphi\leqy0$ if $\varphi\in\ccI_{\wt X,y}$. Let $h$ be a local equation of $D$ and set $h'=h\circ f$. The relation $\varphi\circ f\leqyp0$ means $|e^\varphi\circ f|\leq|h'|^{-N}=|h\circ f|^{-N}$ for some $N\geq0$ on $V(y')\moins\partial\wt X{}'$. We then have $|e^\varphi|\leq|h|^{-N}$ on $\wt f(V(y'))\moins\partial\wt X$, hence $\varphi\leqy0$ according to the openness of $\wt f$.
\end{proof}

In general, the map $f^*$ may not be injective. Indeed, (in the local setting) given $\varphi\in\cO_{X,0}(*D)/\cO_{X,0}$, $f^*\varphi$ may have no poles along $D'$ near $0\in X'$. More precisely, let $\varphi$ be a local section of $\ccI$ at $0\in X$ and let $\wt\Sigma_\varphi\subset\ccIet$ be the image by $\varphi$ of a small open neighbourhood of $0$. Then $f^*$ induces a map from $\wt f^{-1}\wt\Sigma_\varphi$ to $\ccIpet$ whose image is equal to $\wt\Sigma_{f^*\varphi}$.

With a goodness assumption (which is automatically satisfied in dimension one) we recover the injectivity.

\begin{lemme}\label{lem:0implique0}
Assume that $\varphi\in\cO_{X,0}(*D)/\cO_{X,0}$ is purely monomial. Then,
\[
f^*\varphi=\nobreak0\implique\varphi=0.
\]
\end{lemme}

\begin{proof}
If $\varphi\neq0$ let us set $\varphi=U(x)/x^{\bmm}$, where $U(x)$ is a local unit and $\bmm=(m_1,\dots,m_\ell)\in\NN^\ell\moins\{0\}$. Using the notation above for $f$, we have $f^*\varphi=U'(x')/x^{\prime\bmm'}$, where $U'=f^*U/u_1^{\prime m_1}\cdots u_\ell^{\prime m_\ell}$ is a local unit and $\bmm'=m_1\kk_1+\cdots+m_\ell\kk_\ell$. Then $\bmm'\in\NN^{\ell'}\moins\{0\}$, so $f^*\varphi\neq0$.
\end{proof}

\begin{corollaire}
Let $\varphi$ be a local section of $\ccI$ at $0\in X$ which is purely monomial. Then $f^*$ in injective on $\wt f^{-1}\Sigma_\varphi$.\qed
\end{corollaire}

We also recover a property similar to strictness.

\begin{lemme}\label{lem:negimpliqueneg}
With the same assumption as in Lemma \ref{lem:0implique0}, if $f^*\varphi\leq0$ (\resp $f^*\varphi<0$) at $(\theta'_o,0)\in\partial\wt X'_{|D_{L'}}$, then $\varphi\leq0$ (\resp $\varphi<0$) at $(\theta_o,0)=\wt f(\theta'_o,0)$.
\end{lemme}

\begin{proof}
We keep the same notation as above, and it is enough to consider the case $f^*\varphi<0$, according to Lemma \ref{lem:0implique0}. The assumption can then be written~as
\[
\arg U'(0)-\big\langle m_1\kk_1+\cdots+m_\ell\kk_\ell,\theta'_o\big\rangle\in(\pi/2,3\pi/2)\mod2\pi,
\]
where $\langle\,,\rangle$ is the standard scalar product on $\RR^{\ell'}$. Notice now that $\theta_{o,j}=\langle\kk_j,\theta'_o\rangle+\arg u'_j(0)$ for $j=1,\dots,\ell$, so that the previous relation is written as
\[
\arg U(0)-\sum m_j\theta_{o,j}\in(\pi/2,3\pi/2)\mod2\pi,
\]
which precisely means that $\varphi\lethetao0$.
\end{proof}

Let now $(\cL,\cL_\bbullet)$ be a Stokes-filtered local system on $\partial\wt X$. Its \index{pull-back (inverse image)!of Stokes-filtered local systems}pull-back $f^+(\cL,\cL_\bbullet)$ (\cf Definition \ref{def:pullbackpreI}), which is a priori a pre-$\ccI$-filtered local system, is also a Stokes-filtered local system (\cf Lemma \ref{lem:stabpullback}), and its associated stratified $\ccI$-covering is $f^*(\wt f^{-1}\wt\Sigma(\cL))$ (\cf Lemma \ref{lem:stabpullbackstrat}).

\begin{proposition}\label{prop:pullbackgood}
With the previous assumptions on $f$, let us assume that $(\cL,\cL_\bbullet)$ is good. Then $f^+(\cL,\cL_\bbullet)$ is also good.
\end{proposition}

\begin{proof}
According to the previous considerations, it remains to check that, if $\wt\Sigma$ is a good stratified $\ccI$-covering, then $f^*(\wt f^{-1}\wt\Sigma)$ is also good, and this reduces to showing that, if $\varphi\in\cO_{X,0}(*D)/\cO_{X,0}$ is purely monomial, then so is $f^*\varphi$, a property that we already saw in the proof of Lemma \ref{lem:0implique0}.
\end{proof}

\subsection{Partially regular Stokes-filtered local systems}

In the setting of \S\ref{subsec:smoothdivcase}, let $(\cL,\cL_\bbullet)$ be a Stokes-filtered local system on $\partial\wt X(D_{j\in J})$ with associated stratified $\ccI$-covering $\wt\Sigma$ equal to the zero section of $\ccIet_{|\partial\wt X}$. In such a case, we will say that $(\cL,\cL_\bbullet)$ is \index{Stokes-filtered local@Stokes-filtered local system!regular}\emph{regular}. Then $(\cL,\cL_\bbullet)$ is the graded Stokes-filtered local system with jump at $\varphi=0$ only. The category of regular Stokes-filtered local systems is then equivalent to the category of local systems on $\partial\wt X(D_{j\in J})$.

We now consider the case where $(\cL,\cL_\bbullet)$ is \index{Stokes-filtered local@Stokes-filtered local system!partially regular}\emph{partially regular}, that is, there exists a decomposition $J=J'\cup J''$ such that its associated stratified covering $\wt\Sigma$ reduces to the zero section when restricted over $D(J'')\moins D(J')$ (recall that $D(I)=\bigcup_{j\in I}D_j$). We will set $D'=D(J')$ and $D''=D(J'')$ for simplicity. Near each point of $D''\moins D'$ the Stokes-filtered local system is regular. We will now analyze its local behaviour near $D''\cap D'$. We will restrict to a local analysis in the non-ramified case.

According to \S\ref{subsec:realblowup}, the identity map $X\to X$ lifts as a map $\pi:\wt X\defin\wt X(D_{j\in J})\to\wt X{}'\defin\wt X(D_{j\in J'})$ and we have a commutative diagram
\begin{equation}\label{eq:pushdiag}
\begin{array}{c}
\xymatrix{
\ccIet_{\wt X}\ar[d]_\mu&\ar[l]_-q\pi^{-1}\ccIet_{\wt X{}'}\ar[r]^-{\wt\pi}\ar[dl]&\ccIet_{\wt X{}'}\ar[d]^{\mu'}\\
\wt X\ar[rr]^-\pi\ar[dr]_\varpi&&\wt X{}'\ar[dl]^{\varpi'}\\
&X&
}
\end{array}
\end{equation}
The boundary $\partial\wt X$ of $\wt X$ is $\varpi^{-1}(D)$ and $\partial\wt X{}'=\varpi^{\prime-1}(D')$.

We now consider the local setting of \S\ref{subsec:goodness} and we set $\ell=\ell'+\ell''$, $L'=\{1,\dots,\ell'\}$ and $L''=\{\ell'+1,\dots,\ell\}$. If $(\cL,\cL_\bbullet)$ is a non-ramified Stokes-filtered local system with associated $\ccI$-covering equal to $\wt\Sigma$, we assume that $\wt\Sigma_{|D_L}$ is a trivial covering of $Y_L=\varpi^{-1}(D_L)$. Then $\wt\Sigma$ is determined by a finite set $\Phi\subset\cO_{X,0}(*D)/\cO_{X,0}$. The partial regularity property means that the representatives $\varphi$ of the elements of $\Phi$ are holomorphic away from $D'$, that is, have no poles along $D''$, that is also, $\Phi\subset\cO_{X,0}(*D')/\cO_{X,0}$, and $\Phi$ defines a trivial stratified $\ccI_{\wt X'}$-covering $\wt\Sigma'$. The map $q$ induces an homeomorphism of $\pi^{-1}\wt\Sigma'$ onto $\wt\Sigma$.

\begin{proposition}\label{prop:partialreg}
In this local setting, the category of non-ramified Stokes-filtered local systems on $\partial\wt X(D_{j\in L})_{|D_L}$ with associated stratified $\ccI$-covering contained in~$\wt\Sigma$ is equivalent to the category of non-ramified Stokes-filtered local systems on $\partial\wt X(D_{j\in L'})_{|D_L}$ with associated stratified $\ccI$-covering contained in $\wt\Sigma'$, equipped with commuting automorphisms $T_k$ ($k\in L''$).
\end{proposition}

\begin{proof}
Let us first consider the following general setting: $\cF$ is a $\RR$-constructible sheaf on $Z\times(S^1)^k$, where $Z$ is a nice space (\eg a subanalytic subset of $\RR^n$). We denote by $\pi:Z\times(S^1)^k\to Z$ the projection and by $\rho:Z\times\RR^k\to Z\times(S^1)^k$ the map $(z,\theta_1,\dots,\theta_k)\mto(z,e^{i\theta_1},\dots,e^{i\theta_k})$. We will also set $\wt\pi=\rho\circ\pi$.

\begin{lemme}\label{lem:constauto}
The category of $\RR$-constructible sheaves $\cF$ on $Z\times(S^1)^k$ whose restriction to each fibre of $\pi$ is locally constant is naturally equivalent to the category of $\RR$-constructible sheaves on $Z$ equipped with commuting automorphisms $T_1,\dots,T_k$.
\end{lemme}

\begin{proof}
Let $\sigma_i$ ($i=1,\dots,k$) denote the translation by one in the direction of the $i$-th coordinate in $\RR^k$. Then the functor $\rho^{-1}$ induces an equivalence between the category of sheaves $\cF$ on $Z\times(S^1)^k$ and that of sheaves $\cG$ on $Z\times\RR^k$ equipped with isomorphisms $\sigma_i^{-1}\cG\isom\cG$ which commute in a natural way. It induces an equivalence between the corresponding full subcategories of $\RR$-constructible sheaves which are locally constant in the fibres of $\pi$ and $\wt\pi$.

Let $\cG$ be a $\RR$-constructible sheaf on $Z\times\RR^k$. We have a natural (dual) adjunction morphism $\cG\to\wt\pi^!\bR\wt\pi_!\cG=\wt\pi^{-1}\bR\wt\pi_!\cG[k]$ (\cf \cite[Prop\ptbl3.3.2]{K-S90} for the second equality), which is an isomorphism if $\cG$ is locally constant (hence constant) in the fibres of~$\wt\pi$ (\cf \cite[Prop\ptbl2.6.7]{K-S90}). This shows that (via $R^k\wt\pi_!$ and $\wt\pi^{-1}$) the category of $\RR$\nobreakdash-constructible sheaves on $Z\times\RR^k$ which are constant in the fibres of $\wt\pi$ is equivalent to the category of $\RR$-constructible sheaves on $Z$. If now $\cH$ is a $\RR$-constructible sheaf on $Z$ with commuting automorphisms $T_i$ ($i=1,\dots,k$), it produces a sheaf $\cG=\wt\pi^{-1}\cH$ with commuting isomorphisms $\sigma_i^{-1}\cG\simeq\cG$ by composing the natural morphism $\sigma_i^{-1}\wt\pi^{-1}\cH\to\wt\pi^{-1}\cH$ with $T_i$.
\end{proof}

Let us end the proof of the proposition. We know that the first category considered in the proposition is equivalent to that of Stokes-filtered local systems indexed by $\Phi$. Each $\cL_{\leq\varphi}$ is locally constant in the fibres of $\pi$, due to the local grading property of $(\cL,\cL_\bbullet)$. We can therefore apply Lemma \ref{lem:constauto}, since each $\cL_{\leq\varphi}$ is $\RR$-constructible, to get the essential surjectivity. The full faithfulness is obtained in the same way.
\end{proof}

\begin{remarque}
The statement of Proposition \ref{prop:partialreg} does not extend as it is in the ramified case. Indeed, even if $\wt\Sigma$ is regular along $D''$, a ramification may be necessary along $D''$ to trivialize $\wt\Sigma_{|Y_L}$ (\eg a local section of $\wt\Sigma_{|Y_L}$ is written $a(y',y'')/y^{\prime k}$ for ramified coordinates $y',y''$, and $a$ is possibly not of the form $a(y',x'')$).
\end{remarque}

\chapter[The R-H correspondence (case of a smooth divisor)]{The Riemann-Hilbert correspondence for good meromorphic connections (case~of~a~smooth divisor)}\label{chap:RHgoodsmooth}

\begin{sommaire}
This \chaptersname is similar to \Chaptersname\ref{chap:RH}, but we add holomorphic parameters. Moreover, we assume that no jump occurs in the the exponential factors, with respect to the parameters. This is the meaning of the goodness condition in the present setting. We will have to treat the Riemann-Hilbert functor in a more invariant way, and more arguments will be needed in the proof of the main result (equivalence of categories) in order to make it global with respect to the divisor. For the sake of simplicity, we will only consider the case of \emph{germs} of meromorphic connections along a smooth divisor.
\end{sommaire}

\subsection{Introduction}\label{subsec:RHgoodsmoothintro}
How to deform a meromorphic bundle with connection $(\cM_o,\nabla_o)$ on a Riemann surface $X_o$ with poles along a discrete set of points $D_o$, as considered in \Chaptersname\ref{chap:RH}? Classically one looks for \emph{isomonodromy deformations}. Let us forget for a moment the deformation of the Riemann surface itself (\ie its holomorphic structure) and the divisor on it, to focus on the deformation of the connection itself, so that the ambient space of the deformation takes the form $(X,D)=(X_o\times T,D_o\times T)$ for some parameter space $T$, that we usually assume to be a complex manifold and simply connected for a while. If the connection has regular singularities at each point of $D_o$, it is completely determined by its monodromy representation, \ie the associated locally constant sheaf, and a natural condition for the deformation is that the monodromy remains constant along the deformation. In other words, the local systems corresponding to the family of deformed connections form themselves a locally constant sheaf on $X\moins D$ (since from our assumption $\pi_1(X\moins\nobreak D)=\pi_1(X_o\moins\nobreak D_o)$). By the Riemann-Hilbert correspondence for regular meromorphic connections, giving such an isomonodromy deformation of the connection is equivalent to giving a meromorphic bundle with \emph{integrable} connection $(\cM,\nabla)$ having regular singularities along $D$, whose restriction at $t=t_o\in T$ is equal to $(\cM_o,\nabla_o)$.

If we now start from $(\cM_o,\nabla_o)$ with possibly irregular singularities at some points of~$D_o$, the monodromy representation may be a poor invariant to control \hbox{deformations}. An isomonodromy deformation has then to be taken in the sense of an integrable \hbox{deformation} of the connection. However, not any integrable deformation can be \hbox{allowed}. Let us consider the simple case where $X_o$ is a disc and $D_o$ is its origin. Assume also that $(\cM_o,\nabla_o)=\cE^{\varphi_o(x)}\otimes\cR_o$, where $\varphi_o\in\Gamma(X_o,\cO_{X_o}(*D_o))/\Gamma(\cO_{X_o})$ and~$\cR_o$ has a regular singularity at $0\in X_o$. An isomonodromy deformation deforms $\cR_o$ in a constant way, while it can deform $\varphi_o$ as a section $\varphi\in\Gamma(X,\cO_X(*D))/\Gamma(X,\cO_X)$. However, if $\varphi_o=a_k/x^k+\cdots+a_1/x$ with $a_k\in\CC^*$ in a local coordinate $x$ on $X_o$, it is natural to consider deformations $\varphi(x,t)=a_k(t)/x^k+\cdots+a_1(t)/x$ with $a_k(t_o)=a_k$ (that is, $a_\ell(t)\equiv0$ for $\ell\geq k$). 

Conversely, given a meromorphic bundle with integrable connection $(\cM,\nabla)$ on $(X,D)$, it is a natural question to ask whether it is an isomonodromy deformation (in the previous extended meaning) of its restriction to some slice $(X_o,D_o)$, or not. In order to apply the previous criterion, it is necessary to find a Levelt-Turrittin decomposition after formalization along $D$, and possibly after some ramification around $D$, and to check whether the exponential factors $\varphi$ behave in the expected way.

It happens that such a formal Levelt-Turrittin decomposition after ramification exists generically along $D$, as follows from a straightforward adaptation to higher dimensions of the various proofs of the Levelt-Turrittin theorem in dimension one and by using the integrability property of the connection, but there may exist a subset~$S$ of the parameter space $T$ such that, for $t\in S$, the formal the Levelt-Turrittin decomposition after ramification of the restriction $(\cM_t,\nabla_t)$ is not related in a natural way to the formal the Levelt-Turrittin decomposition after ramification of the generic neighbouring restriction $(\cM_{t'},\nabla_{t'})$. The points of the subset $S$ are called \emph{turning points} of $(\cM,\nabla)$ in \cite{Andre07}, and we introduce the notion of a \emph{good formal structure} along $D$ to specify the properties of $(\cM,\nabla)$ which make it an isomonodromy deformation of its restriction to slices.

Eliminating the turning points is a new problem in the theory. A conjecture in \cite{Bibi97} (and proved in some cases there), asserting that this can be performed by a locally finite sequence of complex blowing-ups, was proved independently and very differently by K\ptbl Kedlya \cite{Kedlaya09,Kedlaya10} and T\ptbl Mochizuki \cite{Mochizuki07b,Mochizuki08}. However, the pull-back of the smooth divisor $D$ acquires singularities under this process, that one can assume to be of normal crossing type. The analysis of the corresponding situation will be the subject of \Chaptersname\ref{chap:goodformal}.

In this \chaptername, we restrict ourselves to the case of a smooth divisor in the absence of turning points, and prove a Riemann-Hilbert correspondence, analogous to that of \Chaptersname\ref{chap:RH}, with a parameter. In the approach used in this \chaptername, the new ingredient is the \emph{Stokes sheaf}, which is a locally constant sheaf of (possibly nonabelian) groups.

\medskip
We consider the following setting:
\begin{itemize}
\item
$X$ is a complex manifold and~$D$ is a smooth divisor in~$X$, $X^*\defin X\moins D$,
\item
$\varpi:\wt X\defin\wt X(D)\to X$ is the oriented real blowing-up of~$D$ in~$X$, so that $\varpi$ is a~$S^1$-fibration,
\item
$j:X^*\hto X$ and $\wtj:X^*\hto\wt X$ denote the open inclusions, and $i:D\hto X$ and $\wti:\partial\wt X\hto \wt X$ denote the closed inclusions,
\item
the ordered sheaf~$\ccI$ on $\partial\wt X$ is as in Definitions \ref{def:Idivisor} and \ref{def:orderI}.
\end{itemize}

Since $\ccIet$ is Hausdorff, the notion of $\ccI$-filtration that used in this \chaptersname is the notion introduced in \S\ref{subsec:Ifilt}.

\subsection{Good formal structure of a meromorphic connection}\label{subsec:goodformstructD}
We anticipate here the general definitions of \Chaptersname\ref{chap:goodformal}. Let $\cM$ be a meromorphic connection on~$X$ with poles along~$D$ (\cf\S\ref{subsec:modgrowth}) and let $x_o\in D$. We say that $\cM$ has a \emphb{good formal structure} at~$x_o$ if, after some ramification $\rho_d$ around~$D$ in some neighbourhood of~$x_o$, the pull-back connection $\rho_d^+\cM$ has a \emphb{good formal decomposition}, that is, denoting by~$\wh D_d$ the space $D_d=D$ equipped with the sheaf $\cO_{\wh D_d}\defin\varprojlim_k\cO_{X_d}/\cO_{X_d}(-kD_d)$,
\begin{equation}\label{eq:goodD}
\cO_{\wh D_d}\otimes_{\cO_{X_d}}\rho_d^+\cM\simeq\bigoplus_{\varphi\in \Phi_d}(\cE^\varphi\otimes\wh\cR_\varphi),
\end{equation}
where $\Phi_d$ is a \emph{good} subset of $\cO_{X_d,x_o}(*D_d)/\cO_{X_d,x_o}$ (\cf Definition \ref{def:localgoodness}), $\cE^\varphi=(\cO_{\wh D_d},d+\nobreak d\varphi)$ and $\wh\cR_\varphi$ is a free $\cO_{X_d,x_o}(*D_d)$-module equipped with a connection having regular singularities along $D_d$. We will usually make the abuse of identifying $\Phi_d\subset\cO_{X_d,x_o}(*D_d)/\cO_{X_d,x_o}$ with a set of representatives in $\Gamma(U,\cO_{X_d}(*D_d))/\Gamma(U,\cO_{X_d})$ for some open neighbourhood $U$ of~$x_o$ in $X_d$ (for instance choose $U$ Stein so that $H^1(U,\cO_{X_d})=0$).

\skpt
\begin{remarques}\ligne
\begin{enumerate}
\item
In the neighbourhood of~$x_o$, each $\wh\cR_\varphi$ is obtained, after tensoring with $\cO_{\wh D_d}$, from a meromorphic connection with regular singularity, hence is locally free over $\cO_{\wh D_d}(*D_d)$. As a consequence, if $\cM$ has a good formal structure, it is $\cO_X(*D)$-locally free.
\item
In \cite[Lemma 5.3.1]{Mochizuki08b}, T\ptbl Mochizuki gives a criterion for $\cM$ to have a good formal structure at~$x_o$: choose a local isomorphism $(X,x_o)\simeq (D,x_o)\times(\CC,0)$; if there exists a \emph{good} set $\Phi_d\subset\cO_{X_d,x_o}(*D_d)/\cO_{X_d,x_o}$ defined on some neighbourhood $U\subset D$ of~$x_o$ such that, for any $x\in U$, the set of exponential factors of $\rho_d^+\cM_{|\{x\}\times(\CC,0)}$ is $\Phi_{d|\{x\}\times(\CC,0)}$, then $\cM$ has a good formal structure at~$x_o$.

In \cite[Th\ptbl3.4.1]{Andre07}, Y\ptbl André gives a similar criterion in terms of the Newton polygon of $\cM_{|\{x\}\times(\CC,0)}$ and that of $\cEnd(\cM)_{|\{x\}\times(\CC,0)}$ (so with weaker conditions a priori than in Mochizuki's criterion). Lastly, in \cite{Kedlaya09}, K\ptbl Kedlaya gives another criterion in terms of an irregularity function.

\item
For any given $\cM$ with poles along~$D$, the good formal structure property holds generically on~$D$ (\cf \cite{Malgrange95}). Here, we assume that it holds all over~$D$.
\end{enumerate}
\end{remarques}

We now associate to a germ along~$D$ of meromorphic connection having a good formal structure a $\ccI$-covering $\wt\Sigma\subset\ccIet$ in an intrinsic way. Notice that, due to the goodness assumption, the decomposition \eqref{eq:goodD} is locally unique, as follows from the following lemma.

\begin{lemme}\label{lem:pasdemorphisme}
If $\varphi-\psi$ is nonzero and purely monomial, there is no nonzero morphism $\wh\cR_\psi\to\cE^{\varphi-\psi}\otimes\wh\cR_\varphi$.
\end{lemme}

\begin{proof}
Since $\wh\cR_\psi$ and $\wh\cR_\varphi$ are successive extensions of regular formal meromorphic flat bundles of rank one, one can reduce the proof of the assertion to the case where both have rank one, and then to proving that for such a rank one object $\wh\cR$, $\cE^{\varphi-\psi}\otimes\wh\cR$ has no nonzero section, which can be checked easily by considering the order of the pole of such a section.
\end{proof}

Recall now that we have locally a cartesian square
\[
\xymatrix{
\rho_d^{-1}\ccIet_d\ar[d]_{\mu_d}\ar[r]^{\rho_d^\et}&\ccIet_d\ar[d]^\mu\\
\partial\wt X_d\ar[r]^-{\rho_d}&\partial\wt X
}
\]
and $\rho_d^{-1}\ccI_d=\varpi_d^{-1}\cO_{X_d}(*D_d)/\cO_{X_d}$. The set $\Phi_d$ defines a finite (trivial) covering $\wt\Sigma_d\subset\rho_d^{-1}\ccIet_d$ locally on $D_d$, which is invariant under the Galois group of $\rho_d^\et$, and is therefore equal to $(\rho_d^{\et})^{-1}(\wt\Sigma)$ for some locally defined $\ccI$-covering $\wt\Sigma\subset\ccIet$. The uniqueness statement above implies that $\wt\Sigma_d$ is uniquely determined from $\cM$, and therefore so is $\wt\Sigma$, which thus glues globally as a $\ccI$-covering \index{$SZIGMAWTM$@$\wt\Sigma(\cM)$}$\wt\Sigma(\cM)$ all over~$D$, since the goodness assumption is made all over~$D$. We call $\wt\Sigma(\cM)$ the \emphb{good $\ccI$-covering} associated to $\cM$.

\subsection{The Riemann-Hilbert functor}\label{subsec:RHsmooth}

The sheaf $\cA_{\wt X}^\modD$ is defined in \S\ref{subsec:modgrowthfunct}. Its restriction to $\partial\wt X$ will be denoted by $\cA_{\partial\wt X}^\modD$. We define the sheaf $\cA_{\ccIet}^\modD$ exactly as in \S\ref{subsec:somebasic}. The Riemann-Hilbert functor will then be defined as a functor from the category of germs along~$D$ of meromorphic connections on $X$ with poles on~$D$, that is, germs along~$D$ of coherent $\cO_X(*D)$-modules with a flat connection, to the category of Stokes-filtered local systems on $\partial\wt X$.

Let $\cM$ be a meromorphic connection on~$X$ with poles along~$D$. We define $\DR_{\ccIet}^\modD(\cM)$ as in \S\ref{subsec:RHfunctorone}. The sheaf $\cL_\leq\defin\cH^0\DR_{\ccIet}^\modD(\cM)$ is naturally a subsheaf of $\mu^{-1}\cL$, with $\cL\defin\wti^{-1}\wtj_*\cH^0\DR(\cM_{|X^*})$. At this point, we do not even claim that $\cL_\leq$ is a pre-$\ccI$-filtration of $\cL$. Nevertheless, we have defined a correspondence $\RH$ from the category of germs along~$D$ of meromorphic connections on~$X$ with poles along~$D$ to the category of pairs $(\cL,\cL_\leq)$ consisting of a local system $\cL$ on~$\partial\wt X$ and a subsheaf $\cL_\leq$ of $\mu^{-1}\cL$. It is clear that this correspondence is functorial.

\begin{definitio}\label{def:RHsmooth}
The \index{Riemann-Hilbert functor (RH)}Riemann-Hilbert functor \index{$RH$@$\RH$}$\RH$ is the functor defined above.
\end{definitio}

In order to obtain an equivalence, it is however necessary to have a goodness assumption, that we fix by the choice of a \emph{good} $\ccI$-covering $\wt\Sigma$ of $\partial\wt X$, \ie a closed subset $\wt\Sigma\subset\ccIet$ such that $\mu$ induces a finite covering $\mu:\wt\Sigma\to\partial\wt X$ which is good (\cf Definition \ref{def:globalgoodness} with only one stratum). This choice will be made once and for all in this \chaptername. We now describe the categories involved in the correspondence.

On the one hand, the category of germs along~$D$ of good meromorphic connections with poles along~$D$ and with associated $\ccI$-covering $\wt\Sigma(\cM)$ contained in $\wt\Sigma$, as defined in \S\ref{subsec:goodformstructD} above.

On the other hand, the definition of the category of Stokes-filtered local systems on $\partial\wt X$ with associated stratified $\ccI$-covering contained in $\wt\Sigma$ has been given in Definitions \ref{def:Stokesfilt} and \ref{def:goodness}.

\begin{lemme}\label{lem:goodgoodsmooth}
If $\cM$ has a good formal structure along~$D$, then $\RH(\cM)=(\cL,\cL_\leq)$ is a good Stokes-filtered local system on $\partial\wt X$, that we denote by $(\cL,\cL_\bbullet)$.
\end{lemme}

\begin{proof}
The question is local on $\partial\wt X$, and we easily reduce to the case where $\cM$ has a good formal decomposition. The \index{Hukuhara-Turrittin-Sibuya theorem}Hukuhara-Turrittin theorem with a ``good holomorphic parameter'' is due to Sibuya \cite{Sibuya62,Sibuya74}. It implies that, near any $y\in\partial\wt X$,
\begin{equation}\label{eq:HukTurSib}
\cA_{\partial\wt X}^\modD\otimes\cM\simeq\cA_{\partial\wt X}^\modD\otimes\Big(\bigoplus_{\psi\in\Phi}(\cE^\psi\otimes\cR_\psi)\Big),
\end{equation}
where each $\cR_\psi$ is a locally free $\cO_{X,x}(*D)$-module ($x=\varpi(y)$) with a flat connection having a regular singularity along~$D$. So, by Hukuhara-Turrittin-Sibuya, we can assume that $\cM\simeq\bigoplus_{\psi\in\Phi}(\cE^\psi\otimes\cR_\psi)$, where $\Phi$ is good. Arguing on each summand and twisting by $\cE^{-\psi}$ reduces to the case where $\cM=\cR$ is regular. In such a case, we have to show that $\cL_\leq$ is a pre-$\ccI$-filtration of $\cL$, which moreover is the trivial graded $\ccI$-filtration. Since horizontal sections of $\cR$ have moderate growth in any meromorphic basis of $\cR$, we have, for any $y\in\partial\wt X$, $\cL_{\leq\varphi,y}=\cL_y$ if $\reel(\varphi)\leq0$ in some neighbourhood of~$y$, and $\cL_{\leq\varphi,y}=0$ otherwise. This defines the graded $\ccI$-filtration on~$\cL$ with $\Phi=\{0\}$, according to Example \ref{exem:triviaux}\eqref{exem:triviaux1}.
\end{proof}

\begin{remarque}
Unlike the dimension-one case, the complex $\DR^\modD_{\ccIet}(\cM)$ has cohomology in degrees $\neq0$ even if $\cM=\cR$ has regular singularities along~$D$, as shown by Example \ref{exem:degonecohom}.
\end{remarque}

The main result of this \chaptersname is:

\begin{theoreme}\label{th:RHmerosmooth}\index{Riemann-Hilbert correspondence!for germs of good meromorphic connections}
In the previous setting, the Riemann-Hilbert functor induces an equivalence between the category of good meromorphic connections with poles along~$D$ and associated $\ccI$-covering contained in $\wt\Sigma$, and the category of Stokes-filtered local systems on $\partial\wt X$ with associated $\ccI$-covering contained in $\wt\Sigma$.
\end{theoreme}

\subsection{\proofname\ of the full faithfulness in Theorem \ref{th:RHmerosmooth}}
Let us start with a statement analogous to Theorem \ref{th:H1nul}:\enlargethispage{\baselineskip}%

\begin{lemme}\label{lem:imdirLnegsmooth}
Let $\cM$ be a germ at $x_o\in D$ of meromorphic connection with poles along~$D$. Assume that, after some ramification $\rho_d$ of order $d$ around~$D$, the pull-back connection satisfies \eqref{eq:HukTurSib} near any $y\in\varpi^{-1}(x_o)$, where every $\psi$ entering in the decomposition is purely monomial (\cf Definition \ref{def:purmonom}). Then $\varpi_*\cH^0\DR^\modD(\cM)=\cH^0\DR(\cM)$.
\end{lemme}

\begin{proof}
The complex $\DR(\cM)$ is regarded as a complex on~$D$ (\ie we consider the sheaf restriction to~$D$ of the usual de~Rham complex, since $\cM$ is a germ along~$D$).

We first claim that $\DR^\modD(\cM)$ (\cf \S\ref{subsec:modgrowth}) has cohomology in degree~$0$ at most. This is a local statement on $\partial\wt X$. Assume first $d=1$. By the decomposition \eqref{eq:HukTurSib}, we are reduced to the case where $\cM=\cE^\psi\otimes\cR$, and by an argument similar to that used in \eqref{proof:xalpha} of the proof of Theorem \ref{th:H1nul}, we reduce to the case where $\cM=\cE^\psi$, where~$\psi$ is purely monomial. The assertion is then a consequence of Proposition \ref{prop:HkEphinul}. If $d\geq2$, as $\cM$ is locally stably free (\cf\cite{Malgrange95}), we can apply the projection formula
\[
\wt\rho{}_{d,*}\big(\cA_{\partial\wt X_d}^{\rmod D_d}\otimes\varpi_d^{-1}(\rho_d^+\cM)\big)=(\wt\rho{}_{d,*}\cA_{\partial\wt X_d}^{\rmod D_d})\otimes\varpi^{-1}\cM,
\]
and, as $\wt\rho{}_d$ is a covering or degree $d$, $\wt\rho{}_{d,*}\cA_{\partial\wt X_d}^{\rmod D_d}\simeq(\cA_{\partial\wt X}^\modD)^d$ locally on $\partial\wt X$. This isomorphism is compatible with connections, hence, applying this to the moderate de~Rham complex gives that, locally on $\partial\wt X$, $\DR^\modD(\cM)$ is a direct summand of $\bR\wt\rho{}_{d,*}\DR^{\rmod D_d}(\rho_d^+\cM)$. By the first part of the proof, and since $\wt\rho{}_d$ is finite, the latter term has cohomology in degree zero at most, hence the claim for general $d$.

As indicated in Remark \ref{rem:Rpimod}, we have
\begin{equation}\label{eq:varpiDRmod}
\bR\varpi_*\DR^\modD(\cM)=\DR(\cM)\quad\text{in $D^\rb(X)$}.
\end{equation}
Recall that this uses
\begin{align}
\bR\varpi_*\Big(\cA_{\partial\wt X}^\modD\ootimes_{\varpi^{-1}\cO_X(*D)}\varpi^{-1}\cM\Big)&=\bR\varpi_*\Big(\cA_{\partial\wt X}^\modD\ootimes_{\varpi^{-1}\cO_X(*D)}^{\bL}\varpi^{-1}\cM\Big)\label{eq:Rvarpi1}\\
&=\big(\bR\varpi_*\cA_{\partial\wt X}^\modD\big)\ootimes_{\cO_X(*D)}^{\bL}\cM\label{eq:Rvarpi2}\\
&=\cM,\label{eq:Rvarpi3}
\end{align}
where \eqref{eq:Rvarpi1} holds because $\cM$ is locally stably free, \eqref{eq:Rvarpi2} because of the projection formula, and \eqref{eq:Rvarpi3} because of Proposition \ref{prop:RpiA} and $\cM$ is locally stably free. Extending this to $\DR^\modD(\cM)$ as in the proof of Proposition \ref{prop:Rpimod} gives \eqref{eq:varpiDRmod}.

By the first part of the proof, the left-hand side of \eqref{eq:varpiDRmod} is $\bR\varpi_*\cH^0\DR^\modD(\cM)$. Taking $\cH^0$ on both sides gives the lemma.
\end{proof}

\begin{corollaire}
The Riemann-Hilbert functor in Theorem \ref{th:RHmerosmooth} is fully faithful.
\end{corollaire}

\begin{proof}
The proof is similar to that given in dimension one (\S\ref{subsec:RHholone}), as soon as we show that Lemma \ref{lem:imdirLnegsmooth} applies to $\cHom_{\cO_X}(\cM,\cM')$ and that the analogue of Lemma \ref{lem:comphom} (compatibility with $\cHom$) holds. The point here is that the \hbox{goodness} property for $\cM,\cM'$ does not imply the goodness property for $\cHom(\cM,\cM')$. But clearly, if Hukuhara-Turrittin-Sibuya's theorem applies to $\cM,\cM'$, it applies to $\cHom_{\cO_X}(\cM,\cM')$. Moreover, since $\wt\Sigma(\cM),\wt\Sigma(\cM')\subset\wt\Sigma$ and $\wt\Sigma$ is good, the assumption of Lemma \ref{lem:imdirLnegsmooth} holds for $\cHom_{\cO_X}(\cM,\cM')$. The argument is similar in order to prove the compatibility with $\cHom$.
\end{proof}

\subsection{Elementary and graded equivalences}\label{subsec:elemgraded}
Recall that $\wh D$ denotes the formal completion of~$X$ along~$D$, and $\cO_{\wh D}\defin\varprojlim_k\cO_X/\cO_X(-kD)$. Let $\wh\cM$ be a coherent $\cO_{\wh D}(*D)$-module with a flat connection\footnote{Be careful that the notation $\wh\cM$ will have a different meaning in \Chaptersname\ref{chap:goodformal}.}. We assume that it is \emph{good}, that is, \eqref{eq:goodD} holds for $\wh\cM$ near each point $x_o\in D$, with a good index set $\Phi_d$.

\skpt
\begin{proposition}\label{prop:isotypformel}\ligne
\begin{enumerate}
\item\label{prop:isotypformel1}
Locally at $x_o\in D$, any irreducible good coherent $\cO_{\wh D}(*D)$-module $\wh\cM$ with a flat connection takes the form $\rho_{d,+}(\cE^\varphi\otimes\cR_{\wh D_d})$, for some $d\geq1$, some purely monomial $\varphi\in\cO_{X_d,x_o}(*D_d)/\cO_{X_d,x_o}$, and some free rank-one $\cO_{X_d,x_o}(*D_d)$-module $\cR$ with a connection having regular singularities along~$D_d$.
\item\label{prop:isotypformel2}
Locally at $x_o\in D$, any good coherent $\cO_{\wh D}(*D)$-module $\wh\cM$ with a flat connection has a unique decomposition $\wh\cM=\bigoplus_\alpha\wh\cM_\alpha$, where
\begin{itemize}
\item
each $\wh\cM_\alpha$ is isotypical, that is, takes the form $\rho_{d_\alpha,+}(\cE^{\varphi_\alpha}\otimes\cR_{\alpha|\wh D_{d_\alpha}})$, for some $d_\alpha\geq1$, some purely monomial $\varphi_\alpha\in\cO_{X_{d_\alpha},x_o}(*D_{d_\alpha})/\cO_{X_{d_\alpha},x_o}$, and some free $\cO_{X_{d_\alpha},x_o}(*D_{d_\alpha})$-module $\cR_\alpha$ with a connection having regular singularities along~$D$,
\item
each $\rho_{d_\alpha,+}\cE^{\varphi_\alpha}$ is irreducible,
\item
if $\alpha\neq\beta$, $\rho_{d_\alpha,+}\cE^{\varphi_\alpha}\not\simeq\rho_{d_\beta,+}\cE^{\varphi_\beta}$.
\end{itemize}
\item\label{prop:isotypformel3}
Assume $D$ is connected. Then, any good coherent $\cO_{\wh D}(*D)$-module $\wh\cM$ with a flat connection and associated $\ccI$-covering $\wt\Sigma$ has a unique decomposition $\wh\cM=\bigoplus_{\wt\Sigma_i}\wh\cM_i$, where $\wt\Sigma_i$ runs among the connected components of $\wt\Sigma$ and each $\wh\cM_i$ is globally isotypical, \ie has associated $\ccI$-covering equal to $\wt\Sigma_i$.
\end{enumerate}
\end{proposition}

\begin{proof}
The proof of \eqref{prop:isotypformel1} and \eqref{prop:isotypformel2} is similar to the analogous statements when $\dim X=\nobreak1$, \cf the unpublished preprint \cite{B-B-D-E05}, see also \cite[Prop\ptbl3.1 \& Cor\ptbl3.3]{Bibi07a}. Then~\eqref{prop:isotypformel3} follows by local uniqueness.
\end{proof}

\begin{definitio}[Good elementary meromorphic connections]\label{def:goodelementary}
Let $\cM$ be a coherent $\cO_X(*D)$-module with a flat connection. We say that $\cM$ is a \emph{good elementary} connection if, locally near any~$x_o$ on~$D$ and after some ramification~$\rho_d$ around~$D$, $\rho_d^+\cM$ has a decomposition as $\bigoplus_{\varphi\in\Phi_d}(\cE^\varphi\otimes\cR_\varphi)$ with a good set $\Phi_d$, and where each $\cR_\varphi$ is a free $\cO_{X_d,x_o}(*D_d)$-module of finite rank with connection having regular singularities along~$D$. We then denote it by $\cM^\el$.
\end{definitio}

\begin{proposition}\label{prop:equivgoodformal}
The formalization functor $\cO_{\wh D}\otimes_{\cO_X|D}\cbbullet$ is an equivalence between the category of elementary germs of meromorphic connection along~$D$ with associated $\ccI$-covering contained in $\wt\Sigma$ (full subcategory of that of germs along~$D$ of meromorphic connections) and the category of formal connections along~$D$ with associated $\ccI$-covering contained in $\wt\Sigma$ .
\end{proposition}

\begin{proposition}\label{prop:equivgoodgraded}
The $\RH$ functor induces an equivalence between the category of elementary germs of meromorphic connection along~$D$ with associated $\ccI$-covering contained in $\wt\Sigma$ and the category of graded~$\ccI$-filtered local systems on $\partial\wt X$ with associated $\ccI$-covering contained in $\wt\Sigma$.
\end{proposition}

\begin{proof}[\proofname\ of Propositions \ref{prop:equivgoodformal} and \ref{prop:equivgoodgraded}]
We notice first that, for both propositions, it is enough to prove the equivalence for the corresponding categories of germs at $x_o\in D$ for any $x_o\in D$, to get the equivalence for the categories of germs along~$D$. Indeed, if this is proved, then, given an object in the target category, one can present it as the result of gluing local objects for a suitable open cover of~$D$. By the local essential surjectivity, one can locally lift each of the local objects, and by the local full faithfulness, one can lift in a unique way the gluing isomorphisms, which satisfy the cocycle condition, also by the local full faithfulness. This gives the global essential surjectivity. The global full faithfulness is obtained in a similar way.

Similarly, it is enough to prove both propositions in the case where the covering~$\wt\Sigma$ is trivial: if the equivalence is proved after a cyclic covering around~$D$, then the essential surjectivity is obtained by using the full faithfulness after ramification to lift the Galois action, and the full faithfulness is proved similarly.

In each category, each object is decomposed, and the decomposition is unique (\cf Lemma \ref{lem:pasdemorphisme}). Moreover, each morphism is also decomposed, by the goodness property. One is then reduced to proving the equivalences asserted by both propositions in the case when $\wt\Sigma$ is a covering of degree one and, by twisting, in the regular case, where it is standard.
\end{proof}

\subsection{\proofname\ of the essential surjectivity in Theorem \ref{th:RHmerosmooth}}\label{subsec:proofessentialsurj}
The important arguments have already been explained in \cite[\S II.6]{Bibi00b}. They are an adaptation to the case with a good parameter, of the arguments given in \cite{Malgrange83b,Malgrange83d}. We will not recall all details. Firstly, as already remarked in the proof of Propositions \ref{prop:equivgoodformal} and \ref{prop:equivgoodgraded}, one can reduce to the case of germs at a point of~$D$ and it is enough to consider the case where $\wt\Sigma$ is a trivial covering. This is the setting considered below.

Let $(\cL,\cL_\bbullet)$ be a good Stokes-filtered local system on $\partial\wt X$ with associated $\ccI$\nobreakdash-covering contained in $\wt\Sigma$. Our goal is construct a germ $\cM_D$ of meromorphic connection along~$D$ with $\RH(\cM_D)\simeq(\cL,\cL_\bbullet)$.

\subsubsection*{Step one}
The good Stokes-filtered local system $(\cL,\cL_\bbullet)$ is determined in a unique way, up to isomorphism, by the graded Stokes filtration $\gr\cL$ and an element of $H^1(\partial\wt X,\cAut^{<0}(\mu_!\gr\cL))$ (\cf Proposition \ref{prop:classifhausdorff}). On the other hand, by Proposition \ref{prop:equivgoodgraded}, the graded Stokes filtration $\gr\cL$ is isomorphic to $\RH(\cM^\el)$ for some germ along~$D$ of good elementary meromorphic connection $\cM^\el$.

\begin{lemme}\label{lem:RHEnd}
Let $\cEnd^\modD(\cM^\el)$ be the sheaf of horizontal sections of $\cA_{\partial\wt X}^\modD\otimes\varpi^{-1}\cHom_{\cO_X}(\cM^\el,\cM^\el)$ and let $\cEnd^\rdD(\cM^\el)$ be the subsheaf of sections with coefficients in $\cA_{\partial\wt X}^\rdD$ of holomorphic functions on~$X^*$ having rapid decay along $\partial\wt X$. Then $\cEnd^\modD(\cM^\el)=\cEnd(\mu_!\gr\cL)_{\leq0}$ and $\cEnd^\rdD(\cM^\el)=\cEnd(\mu_!\gr\cL)_{<0}$.
\end{lemme}

\begin{proof}
This is a special case of the compatibility of the Riemann-Hilbert functor with $\cHom$, already used above, and similar to Lemma \ref{lem:comphom}.
\end{proof}

We denote by $\cAut^\rdD(\cM^\el)$ the sheaf $\id+\cEnd^\rdD(\cM^\el)$. Then $\cAut^\rdD(\cM^\el)=\cAut^{<0}(\mu_!\gr\cL)$.

\subsubsection*{Step two}
Consider the presheaf $\cH_D$ on~$D$ such that, for any open set~$U$ of~$D$, $\cH_D(U)$ consists of isomorphism classes of pairs $(\cM,\wh\lambda)$, where $\cM$ is a germ along $U$ of meromorphic connection on~$X$ and $\wh\lambda$ is an isomorphism $\cM_{\wh D}\isom\wh\cM^\el$.

\begin{lemme}\label{lem:HD}
The presheaf $\cH_D$ is a sheaf.
\end{lemme}

\begin{proof}
See \cite[Lemma II.6.2]{Bibi00b}.
\end{proof}

Similarly, let us fix a graded Stokes-filtered local system $\gr\cL^o$ with associated covering contained in $\wt\Sigma$, and let $\cSt_D$ be the presheaf on~$D$ such that, for any open set~$U$ of~$D$, $\cSt_D(U)$ consists of isomorphism classes of pairs $[(\cL,\cL_\bbullet),\wt\lambda]$, where $(\cL,\cL_\bbullet)$ is a Stokes-filtered local system on $\partial\wt X_{|U}$ and $\wt\lambda$ is an isomorphism $\gr\cL\simeq\gr\cL^o_{|U}$ of graded Stokes-filtered local systems. In particular, the associated covering $\wt\Sigma(\cL,\cL_\bbullet)$ is contained in $\wt\Sigma$.

\begin{lemme}\label{lem:cSt}
The presheaf $\cSt_D$ is a sheaf.
\end{lemme}

\begin{proof}
The point is to prove that, given two pairs $[(\cL,\cL_\bbullet),\wt\lambda]$ and $[(\cL',\cL'_\bbullet),\wt\lambda']$, there is at most one isomorphism between them. If there exists one isomorphism, we can assume that $(\cL',\cL'_\bbullet)=(\cL,\cL_\bbullet)$ and we are reduced to proving that an automorphism $\lambda$ of $(\cL,\cL_\bbullet)$ is completely determined by $\gr\lambda$. Arguing as in the proof of Theorem \ref{th:abelianwithout}, locally with respect to~$D$, there exists a local trivialization of $(\cL,\cL_\bbullet)$ such that $\lambda$ is equal to $\gr\lambda$ in this trivializations. It is thus uniquely determined by $\gr\lambda$.
\end{proof}

\begin{remarque}
Lemma \ref{lem:cSt} is not needed for the essential surjectivity, it is given by symmetry with Lemma \ref{lem:HD}. Note that the main argument in Lemma \ref{lem:HD} is the faithful flatness of $\cO_{\wh D}$ over $\cO_{X|D}$. The corresponding argument after applying $\RH$ comes from Theorem \ref{th:abelianwithout}.
\end{remarque}

\subsubsection*{Step three}
Arguing as in \cite{Malgrange83b}, we have a natural morphism of sheaves
\[
\cH_D\to\cSt_D.
\]
where the left-hand side is associated to $\cM^\el$ and the right-hand side to $\RH(\cM^\el)$.

\begin{theoreme}\label{th:isostokes}
This morphism is an isomorphism.
\end{theoreme}

\begin{proof}
See \cite[Th\ptbl II.6.3]{Bibi00b}.
\end{proof}

\begin{remarque}\label{rem:basechangeStokes}
The proof uses the \emphb{Malgrange-Sibuya theorem} on $\varpi^{-1}(x_o)\simeq S^1$, and a base change property for the Stokes sheaf $\cSt_D$ is needed for that purpose (\cf\cite[Prop\ptbl6.9]{Bibi00b}). Note also that it is proved in \loccit that the Stokes sheaf $\cSt_D$ is locally constant on~$D$.
\end{remarque}

\subsubsection*{Step four}
We can now end the proof. Given a good Stokes-filtered local system $(\cL,\cL_\bbullet)$, we construct $\cM^\el$ with $\RH(\cM^\el)\simeq\gr\cL$, according to Proposition \ref{prop:equivgoodgraded}. We are then left with a class in $H^1(\partial\wt X,\cAut^{<0}(\mu_!\gr\cL))$. This defines a global section in $\Gamma(D,\cSt_D)$ (in fact, Lemma \ref{lem:cSt} says that it \emph{is} a global section of $\cSt_D$ on~$D$). By Theorem \ref{th:isostokes}, this is also a class in $\Gamma(D,\cH_D)$, that is, it defines a pair $(\cM,\wh\lambda)$. It is now clear from the construction that $\RH(\cM)\simeq(\cL,\cL_\bbullet)$.\qed

\skpt
\begin{remarques}\ligne
\begin{enumerate}
\item
One can shorten the proof above in two ways: firstly, one can use directly a version of the Malgrange-Sibuya theorem with a parameter; secondly, one can avoid the introduction of the sheaves $\cH_D,\St_D$, and use the full faithfulness of the Riemann-Hilbert functor to glue the locally defined meromorphic connections obtained by applying the local Riemann-Hilbert functor. This will be the strategy in the proof of Theorem \ref{th:RHmeronc}. Nevertheless, it seemed interesting to emphasize these sheaves.
\item
One can deduce from the base change property mentioned in Remark \ref{rem:basechangeStokes} that, if we are given a holomorphic fibration $\pi:X\to D$ (such a fibration locally exists near any point of $D$) with $D$ (smooth and) simply connected, and for a fixed good $\ccI$-covering $\wt\Sigma$ of $\partial\wt X=\varpi^{-1}(D)$, the restriction functor from germs of meromorphic connections along $D$ with associated $\ccI$-covering contained in $\wt\Sigma$ to germs of meromorphic connections on $\pi^{-1}(x_o)$ with associated $\ccI$-covering contained in $\wt\Sigma_{|\varpi^{-1}(x_o)}$ is an equivalence of categories. A similar result holds for Stokes-filtered local systems.

Such a result has been generalized by T\ptbl Mochizuki to the case where $D$ has normal crossings (\cf\cite{Mochizuki10b}).
\item
The arguments of \S\ref{subsec:elemgraded} are useful to prove the analogue of Proposition \ref{prop:gr0com}, that is, the compatibility with the formal Riemann-Hilbert correspondence, which also holds in this case.
\end{enumerate}
\end{remarques}

\chapter{Good meromorphic~connections (formal~theory)}\label{chap:goodformal}

\begin{sommaire}
This \chaptersname is a prelude to the Riemann-Hilbert correspondence in higher dimensions, treated in \Chaptersname\ref{chap:RHgoodnc}. We explain the notion of a good formal structure for germs of meromorphic connections having poles along a divisor with normal crossings.
\end{sommaire}

\subsection{Introduction}
Understanding the notion of a good formal decomposition---which is the good analogue in higher dimensions of the Turrittin-Levelt decomposition in dimension one---leads to solving new questions, compared to the one-variable case or the smooth case away from turning points, when considering the formalization along a variety of codimension $\geq2$. For instance, there may be the question whether the set of exponential factors can be represented by convergent elements; or whether a given formal lattice can be lifted as a convergent lattice of a meromorphic connection. This \chaptersname is therefore dedicated to proving such lifting properties, from the formal neighbourhood of a point to an ordinary neighbourhood, when the meromorphic connection is assumed to have a good formal structure along a normally crossing divisor at the the given point.

The usefulness of the formal goodness property for a meromorphic connection lies in the consequences it provides in asymptotic analysis. This goodness property is strong enough to be used in the proof of the many-variable version of the Hukuhara-Turrittin theorem, that we have already encountered in the case of a smooth divisor (proof of Lemma \ref{lem:goodgoodsmooth}). This theorem will be proved in \S\ref{subsec:HTM}.

\subsection{Preliminary notation}\label{subsec:prelimform}
All along this \chaptername, we will use the following notation, compatible with that of \cite{Mochizuki08}:
\begin{itemize}
\item
$\Delta$ is the open disc centered at~$0$ in $\CC$ and of radius one, and $X=\Delta^n$, with coordinates $t_1,\dots,t_n$.
\item
$D$ is the divisor defined by $\prod_{i=1}^\ell t_i=0$.
\item
For any $I\subset\{1,\dots,\ell\}$, \index{$DI$@$D_I$, $\wh D_I$, $D(I)$}$D_I=\bigcap_{i\in I}\{t_i=0\}$, and \index{$DI$@$D_I$, $\wh D_I$, $D(I)$}$D(I^c)=\bigcup_{\substack{j\leq\ell\\j\not\in I}}D_j$. It will be convenient to set $L=\{1,\dots,\ell\}$, so that $D_L$ is the stratum of lowest dimension of~$D$.
\item
\index{$DI$@$D_I$, $\wh D_I$, $D(I)$}$\wh D_I$ is $D_I$ endowed with the sheaf $\varprojlim_k\cO_X/(t_{i\in I})^k$, that we denote $\cO_{\wh D_I}$ (it is sometimes denoted by $\cO_{\wh{X|D_I}}$). A germ $f\in\cO_{\wh D_I,0}$ is written as $\sum_{\nu\in\mathbb N^I}f_\nu t_I^\nu$ with $f_\nu\in \cO(U\cap D_I)$ for some open neighbourhood $U$ of~$0$ independent of $\nu$.
\item
We will also denote by $\wh0$ the origin endowed with the sheaf $\cO_{\wh0}\defin\CC\lcr t_1,\dots,t_n\rcr$.
\end{itemize}

\begin{exemple}
If $\ell=n=2$, we have $D_1=\{t_1=0\}$, $D_2=\{t_2=0\}$, $D_{12}=\{0\}$, $D(1^c)=D_2$, $D(2^c)=D_1$, $D(\{1,2\}^c)=D(\emptyset)=\emptyset$. Given a germ $f\in\CC\{t_1,t_2\}$, we can consider various formal expansions of $f$:
\begin{itemize}
\item
$f=\sum_{i\in\NN}f_i^{(1)}(t_2)t_1^i$ in $\cO_{\wh D_1}$, where all $f_i^{(1)}(t_2)$ are holomorphic in some common neighbourhood of $t_2=0$,
\item
$f=\sum_{j\in\NN}f_j^{(2)}(t_1)t_2^j$ in $\cO_{\wh D_2}$, where all $f_j^{(2)}(t_1)$ are holomorphic in some common neighbourhood of $t_1=0$,
\item
$f=\sum_{(i,j)\in\NN^2}f_{ij}t_1^it_2^j$, with $f_{ij}\in\CC$.
\end{itemize}
These expansions are of course related in a natural way. We will consider below the case of meromorphic functions with poles on~$D$.
\end{exemple}

Any $f\in\cO_{X,0}(*D)$ has a unique formal expansion $f=\sum_{\nug\in\ZZ^\ell}f_\nug t^\nug$ with $f_\nug\in\cO_{D_L,0}$. When it exists, the minimum (for the natural partial order) of the set $\{\nug\in\nobreak\ZZ^\ell\mid\nobreak f_\nug\neq\nobreak0\}\cup\{0_\ell\}$ is denoted by \index{$ORDL$@$\ord^L$}$\ord^L(f)$. It belongs to $(-\NN)^\ell$ and only depends on the class $f_L$ of $f$ in $\cO_{X,0}(*D)/\cO_{X,0}$. With this definition, we have $\ord^L(f)=0\in(-\NN)^\ell$ iff $f_L=0\in\cO_{X,0}(*D)/\cO_{X,0}$. As in Definition \ref{def:localgoodness}, we then set $\bmm(f_L)=-\ord^L(f)\in\NN^\ell$ for some (or any) lifting $f$ of $f_L$.

Similarly, for $I\subset L$ and $f\in\cO_{X,0}(*D)$, the minimum (for the natural partial order) of the set $(\{\nug\in\ZZ^\ell\mid\nobreak f_\nug\neq\nobreak0\}\moins\ZZ^{I^c})\cup \{0_\ell\}$, \emph{when it exists}, is denoted by $\ord^I(f)$. It belongs to $\big((-\NN)^\ell\moins\ZZ^{I^c}\big)\cup\{0_\ell\}$, and only depends on the class $f_I$ of $f$ in $\cO_{X,0}(*D)/\cO_X(*D(I^c))$. We have $f_I=0$ iff $\ord^I(f)=0_\ell$. We then set $\bmm(f_I)=-\ord^I(f)\in\NN^\ell$ for some (or any) lifting $f$ of $f_I$ in $\cO_{X,0}(*D)$. We will also denote by $\bmm_I\in\NN^I$ the $I$-component of $\bmm\in\NN^\ell$. Then, when $\ord^I(f_I)$ exists, $f_I=0$ iff $\bmm_I=0$.

As a $\cO_{D_L,0}$-module, $\cO_{X,0}(*D)/\cO_{X,0}$ is isomorphic to
\[
\bigoplus_{\emptyset\neq J\subset L}\Big(\cO_{D_J,0}\otimes_\CC(t^{-\bun_J}\CC[t_J^{-1}])\Big).
\]
Under such an identification, we have for any $I\subset L$:
\[
\cO_{X,0}(*D)/\cO_{X,0}(*D(I^c))=\bigoplus_{\substack{\emptyset\neq J\subset L\\ J\cap I\neq\emptyset}}\Big(\cO_{D_J,0}\otimes_\CC(t^{-\bun_J}\CC[t_J^{-1}])\Big).
\]

\begin{lemme}\label{lem:formelconv}
A germ $\wh f\in\cO_{\wh0}(*D)/\cO_{\wh0}$ belongs to $\cO_{X,0}(*D)/\cO_{X,0}$ iff for any $i\in L$, the class $\wh f_i$ of $\wh f$ in $\cO_{\wh0}(*D)/\cO_{\wh0}(*D(i ^c))$ belongs to $\cO_{X,0}(*D)/\cO_{X,0}(*D(i ^c))$.
\end{lemme}

\begin{proof}
We also have a decomposition of $\cO_{\wh0}(*D)/\cO_{\wh0}$ as a $\cO_{\wh0\cap D_L}$-module:
\[
\bigoplus_{\emptyset\neq J\subset L}\Big(\cO_{\wh0\cap D_J}\otimes_\CC(t^{-\bun_J}\CC[t_J^{-1}])\Big).
\]
Then, a germ $\wh f$ belongs to $\cO_{X,0}(*D)/\cO_{X,0}$ iff, for any $J\neq\emptyset$, its $J$-component has coefficients in $\cO_{D_J,0}$. The assumption of the lemma means that, for any $i\in L$ and for any $J\subset L$ with $J\neq\emptyset$ and $J\ni i$, the $J$-component of $\wh f$ has coefficients in $\cO_{D_J,0}$. This is clearly equivalent to the desired statement.
\end{proof}

\Subsection{Good formal decomposition}
\subsubsection*{Case of a smooth divisor}
Let us first revisit the notion of good formal decomposition in the case where~$D$ is smooth (\ie $\#L=1$), as defined by \eqref{eq:goodD}. We will have in mind possible generalizations to the normal crossing case, where formal completion along various strata will be needed. Let $\cM$ be a meromorphic connection on $X$ with poles along $D_1$. We wish to compare the following two properties:
\begin{enumerate}
\item\label{enum:gooddecsmooth1}
$\cM$ has a \emphb{good formal decomposition} along $D_1$ near the origin, that is, there exists a \emph{good set} $\Phi\subset\cO_{X,0}(*D_1)/\cO_{X,0}$ and a decomposition stable by the connection
\[
\cM_{\wh D_1}\defin\cO_{\wh D_1,0}\otimes_{\cO_{X,0}}\cM\simeq\bigoplus_{\varphi\in\Phi}\cM_{|\wh D_1,\varphi},
\]
where, for each $\varphi\in\Phi$, the $\cO_{\wh D_1,0}(*D)$-free module with integrable connection $\cM_{|\wh D_1,\varphi}$ is isomorphic to $(\cE^\varphi\otimes\cR_\varphi)_{\wh D_1}$ for some germ $\cR_\varphi$ of meromorphic connection with regular singularity along $D_1$, and $\cE^\varphi=(\cO_{X,0}(*D_1),d+d\varphi)$.
\item\label{enum:gooddecsmooth2}
The formal germ $\wh\cM=\cM_{\wh0}\defin\cO_{\wh0}\otimes_{\cO_{X,0}}\cM$ has a \emphb{good decomposition} along~ $D_1$, that is, there exists a \emph{good set} $\wh\Phi\subset\cO_{\wh0}(*D_1)/\cO_{\wh0}$ and a decomposition
\[
\wh\cM\simeq\bigoplus_{\wh\varphi\in\wh\Phi}\wh\cM_{\wh\varphi}\qquad\text{with } \wh\cM_{\wh\varphi}\simeq(\cE^{\wh\varphi}\otimes\wh\cR_{\wh\varphi}),
\]
where $\wh\cR_{\wh\varphi}$ is a regular $\cO_{\wh0}(*D_1)$-connection.
\end{enumerate}

\begin{probleme}\label{prob:MMhat}
Let $\cM$ be a meromorphic connection with poles along $D_1$. If the formalization $\cM_{\wh0}$ satisfies \eqref{enum:gooddecsmooth2}, is it true that $\wh\Phi\subset\cO_{X,0}(*D_1)/\cO_{X,0}$ (set then $\Phi=\wh\Phi$) and that $\cM$ satisfies~\eqref{enum:gooddecsmooth1}?
\end{probleme}

The problem reduces in fact, given a local basis $\wh e$ of the $\cO_{\wh0}(*D_1)$-module $\wh\cM$ adapted to the decomposition given by \eqref{enum:gooddecsmooth2}, to finding a local basis $e$ of the $\cO_{X,0}(*D_1)$-module $\cM$ such that the base change $e=\wh e\cdot\wh P$ is given by a matrix in $\GL(\cO_{\wh0})$ (while there is by definition a base change with matrix in $\GL(\cO_{\wh0}(*D_1))$). Notice also that, according to \cite[Prop\ptbl1.2]{Malgrange95}, it is enough to find a basis $e_{\wh D_1}$ of $\cM_{\wh D_1}$ with a similar property with respect to $\wh e$. The solution to the problem in such a case is given by Lemma \ref{lem:I} below, in the particular case where $I=L=\{1\}$.

The conclusion is that, even in the case of a smooth divisor, the two conditions may not be equivalent, and considerations of lattices make them equivalent.

\subsubsection*{General case}
Let $\Phi$ be a finite subset of $\cO_{X,0}(*D)/\cO_{X,0}$. There exists an open neighbourhood $U$ of~$0$ on which each element of $\Phi$ has a representative $\varphi\in\Gamma(U,\cO_X(*D))$ (take $U$ such that $H^1(U,\cO_X)=0$). Let $I$ be a subset of~$L$. Then the restriction of $\varphi$ to $U\moins D(I^c)$ only depends on the class $\varphi_I$ of $\varphi$ in $\cO_U(*D)/\cO_U(*D(I^c))$. We denote by $\Phi_I$ the image of $\Phi$ in $\cO_{X,0}(*D)/\cO_{X,0}(*D(I^c))$.

\begin{definitio}[Good decomposition and good formal decomposition]\label{def:goodformdec}
\begin{enumerate}
\item\label{def:goodformdec1}
Let $\wh\cM$ be a free $\cO_{\wh0}(*D)$-module equipped with a flat connection $\wh\nabla:\wh\cM\to\Omega_{\wh0}^1\otimes\wh\cM$. We say that $\wh\cM$ has a \emphb{good decomposition} if there exist a \emph{good} finite set $\wh\Phi\subset\cO_{\wh0}(*D)/\cO_{\wh0}$ (\cf Definition \ref{def:localgoodness}) and a decomposition
\bgroup\numstareq
\begin{equation}\label{eq:goodformdec*}
\wh\cM=\bigoplus_{\wh\varphi\in\wh\Phi}\wh\cM_{\wh\varphi}\qquad\text{with } \wh\cM_{\wh\varphi}\simeq(\cE^{\wh\varphi}\otimes\wh\cR_{\wh\varphi}),
\end{equation}
\egroup
where $\wh\cR_{\wh\varphi}$ is a free $\cO_{\wh0}(*D)$-module equipped with a flat \emph{regular} connection.
\item\label{def:goodformdec3}
Let $\cM$ be a free $\cO_{X,0}(*D)$-module equipped with a flat connection $\nabla:\cM\to\Omega_{X,0}^1\otimes\cM$. We say that $\cM$ has a \emphb{punctual good formal decomposition} near the origin if, for any $x\in D$ in some neighbourhood (still denoted by) $X$ of the origin, the formal connection $\cO_{\wh x}(*D)\otimes_{\cO_X}\cM$ has a good decomposition.
\item\label{def:goodformdec2}
Let $\cM$ be a free $\cO_{X,0}(*D)$-module equipped with a flat connection $\nabla:\cM\to\Omega_{X,0}^1\otimes\cM$. We say that $\cM$ has a \emphb{good formal decomposition} if there exists a \emph{good} finite set $\Phi\subset\cO_{X,0}(*D)/\cO_{X,0}$ and for any $I\subset L$ and any $\varphi_I\in \Phi_I$ a free $\cO_{\wh D_I,0}(*D)$-module $\IcRhat_{\varphi_I}$ equipped with a flat connection such that, on some neighbourhood $U$ of~$0$ where all objects are defined, $\IcRhat_{\varphi_I|D_I^\circ}$ is \emph{regular}, and there is a decomposition
\bgroup\numstarstareq
\begin{equation}\label{eq:goodformdec**}
\cM_{U\cap\wh D_I^\circ}\defin\cO_{U\cap\wh D_I^\circ}\otimes_{\cO_{U\moins D(I^c)}}\cM_{|U\moins D(I^c)}=\bigoplus_{\varphi_I\in \Phi_I}\cM_{U\cap\wh D_I^\circ,\varphi_I},
\end{equation}
\egroup
where $D_I^\circ\defin D_I\moins D(I^c)$, and $\cM_{U\cap\wh D_I^\circ,\varphi_I}\simeq(\cE^{\varphi_I}\otimes\IcRhat_{\varphi_I})_{|U\cap\wh D_I^\circ}
$.
\end{enumerate}
\end{definitio}

In the previous definition, for a germ $\cM$ of $\cO_{X,0}(*D)$-module with connection, we consider a representative in some neighbourhood of the origin, and a representative of the connection.

\skpt
\begin{remarques}\label{rem:openness}\ligne
\begin{enumerate}
\item\label{rem:openness1}
We note that the property in \ref{def:goodformdec}\eqref{def:goodformdec2} is stronger than the notion introduced in \cite{Bibi97}, as the \emph{same} set $\Phi$ is used for any stratum $D_I^\circ$ whose closure contains the stratum $D_L$ going through the base point~$0$.

Note also that this property is \emph{open}, that is, if $\cM$ has a good formal decomposition then, for any \hbox{$x\!\in\!D\cap U$} ($U$ small enough), $\cM_x$ has a good decomposition with data $\Phi_x\subset\cO_{X,x}(*D)/\cO_{X,x}$ and $\cR_{\varphi,x}$. That $\Phi_x$ is good has been checked in Remark \ref{rem:goodness}\eqref{rem:goodness2}, so $\cM$ satisfies \ref{def:goodformdec}\eqref{def:goodformdec3} near the origin.
\item
On the other hand, in \ref{def:goodformdec}\eqref{def:goodformdec3}, we do not insist on the relation between the various $\wh\Phi_x$. The point which will occupy us later is whether we can choose $\wh\Phi_x\subset\cO_{X,x}(*D)/\cO_{X,x}$ instead of $\wh\Phi_x\subset\cO_{\wh x}(*D)/\cO_{\wh x}$. This is one difference between \hbox{dimension} one and dimension $\geq2$. In dimension two, this has been positively solved in \cite[Prop\ptbl I.2.4.1]{Bibi97}. In general, this will be stated and proved in Theorem \ref{th:gooddec} below.

\item\label{rem:openness3}
These definitions are stable by a twist: for instance, $\wh\cM$ has a formal $\wh\Phi$\nobreakdash-decomposition iff for some $\wh\eta\in\cO_{\wh0}(*D)/\cO_{\wh0}$, $\wh\cE^{\wh\eta}\otimes\wh\cM$ has a $(\wh\Phi+\wh\eta)$-decomposition.
\item\label{rem:openness2}
In \cite{Mochizuki08}, T\ptbl Mochizuki uses a goodness condition which is slightly stronger than ours (\cf Definition \ref{def:localgoodness}). For our purpose, the one we use is enough.
\end{enumerate}
\end{remarques}

\begin{lemme}\label{lem:uniquegooddec}
If $\wh\cM$ has a good decomposition, this decomposition is unique and $\wh\Phi$ is unique. Moreover, this decomposition induces on any $\cO_{\wh0}(*D)$-submodule invariant by the connection a good decomposition.
\end{lemme}

\begin{proof}
For the first point, it is enough to prove that, if $\wh\varphi\neq\wh\psi\in\cO_{\wh0}(*D)/\cO_{\wh0}$, then there is no nonzero morphism $\cE^{\wh\varphi}\otimes\wh\cR_{\wh\varphi}\to\cE^{\wh\psi}\otimes\wh\cR_{\wh\psi}$. This is analogous, in the case of normal crossings, of Lemma \ref{lem:pasdemorphisme}, and  amounts to showing that $\cE^{\wh\psi-\wh\varphi}\otimes\wh\cR$ has no horizontal section, when $\wh\cR$ is regular. We can even assume that~$\wh\cR$ has rank one, since a general $\wh\cR$ is an extension of rank-one regular meromorphic connections. So, we are looking for horizontal sections of $(\cO_{\wh0}(*D),d+d\wh\eta+\omega)$ with $\wh\eta\in\cO_{\wh0}(*D)/\cO_{\wh0}$, $\wh\eta\neq0$, and $\omega=\sum_{i=1}^\ell\omega_idt_i/t_i$, $\omega_i\in\CC$.

Assume that $\wh f\in\cO_{\wh0}(*D)$ is a nonzero horizontal section. Then for any $i\in L$, it satisfies $t_i\partial_{t_i}(\wh f)+t_i\partial_{t_i}(\wh\eta)\cdot \wh f+\omega_i\wh f=0$. Let us consider its Newton polyhedron (in a sense slightly different from that used after Definition \ref{def:purmonom}), which is the convex hull of the union of octants $\nug+\RR_+^\ell$ for which $\wh f_\nug\neq0$ (\cf Notation in \S\ref{subsec:prelimform}). By assumption, there exists $i\in L$ for which the Newton polyhedron of $t_i\partial_{t_i}(\wh\eta)$ is not contained in~$\RR_+^\ell$. Then, on the one hand, the Newton polyhedron of $t_i\partial_{t_i}(\wh\eta)\cdot \wh f$ is \emph{equal} to the Minkowski sum of that of $\wh f$ and that of $t_i\partial_{t_i}(\wh\eta)$, hence is not contained in that of $\wh f$. On the other hand, the Newton polyhedron of $t_i\partial_{t_i}(\wh f)+\omega_i\wh f$ is contained in that of $\wh f$. Therefore, the sum of the two corresponding terms cannot be zero, a contradiction.

For the second point, we argue by induction on $\#\Phi$, the result being clear if $\#\Phi=1$, since regularity is stable by taking submodules. By twisting $\wh\cM$, we can assume that the regular part $\wh\cM_\reg$ of $\wh\cM$ is nonzero, and thus the inductive assumption applies to~$\wh\cM_\irr$. We have a decomposition $\wh\cM=\wh\cM_\reg\oplus\wh\cM_\irr$ Let $\wh\cN$ be a $\wh\cO_{\wh0}(*D)$-submodule of~$\wh\cM$ stable by the connection. Denoting by $\wh\cN^\irr$ (\resp $\wh\cN^\reg$) the image of $\wh\cN$ by the projection $\wh\cM\to\wh\cM_\irr$ (\resp $\wh\cM\to\wh\cM_\reg$), we have $\wh\cN\subset\wh\cN^\reg\oplus\wh\cN^\irr$ and, by the inductive assumption, $\wh\cN^\irr$ decomposes according to the decomposition of $\wh\cM_\irr$ (and $\wh\cN^\reg$ is regular). We can therefore assume that $\wh\cN^\reg=\wh\cM_\reg$ and $\wh\cN^\irr=\wh\cM_\irr$, that is, the two projections of $\wh\cN$ to $\wh\cM_\reg$ and $\wh\cM_\irr$ are onto. We wish to show that, in such a case, $\wh\cN=\wh\cM_\reg\oplus\wh\cM_\irr$ in a way compatible with the connections, and it is enough to show that $\rk\wh\cN=\rk\wh\cM$.

Assume first that $\wh\cM_\irr$ has rank one. If $\rk\wh\cN<\rk\wh\cM$, then $\rk\wh\cN=\rk\wh\cM-1=\rk\wh\cM_\reg$, and thus the surjective map $\wh\cN\to\wh\cM_\reg$ is an isomorphism. It follows that $\wh\cN\to\wh\cM_\irr$ is zero, according to the first part of the lemma, which contradicts the surjectivity assumption.

In general, due to the decomposition of $\wh\cM_\irr$ and the fact that any regular meromorphic connection has a rank-one sub-meromorphic connection, there exists a rank-one formal meromorphic sub-connection $\wh\cM_{\irr,1}\subset\wh\cM_\irr$. Let us denote by $\wh\cN_1$ its inverse image in $\wh\cN$ and by $\wh\cM_{\reg,1}$ the image of $\wh\cN_1$ by the projection to $\wh\cM_\reg$. We have a commutative diagram of exact sequences
\[
\xymatrix{
0\ar[r]&\wh\cN_1\ar[r]\ar[d]&\wh\cN\ar[r]\ar[d]&\wh\cN/\wh\cN_1\ar[r]\ar[d]&0\\
0\ar[r]&\wh\cM_{\reg,1}\ar[r]&\wh\cM_\reg\ar[r]&\wh\cM_\reg/\wh\cM_{\reg,1}\ar[r]&0
}
\]
where the first two vertical morphism are onto, hence so is the third one. But $\wh\cN/\wh\cN_1\simeq\wh\cM_\irr/\wh\cM_{\irr,1}$ is purely irregular, hence the first part of the lemma implies that $\wh\cM_\reg/\wh\cM_{\reg,1}=0$, \ie $\wh\cM_{\reg,1}=\wh\cM_\reg$. Applying the previous case to $\wh\cN_1$ shows that $\rk\wh\cN_1=\rk\wh\cM_\reg+1$, and thus $\rk\wh\cN=\rk\wh\cM_\reg+\rk\wh\cM_\irr$.
\end{proof}

With the previous definition of a (punctual) good formal decomposition, it is natural to set the analogue of Problem \ref{prob:MMhat}, namely, to ask whether, given $\cM$, the existence of a punctual good formal decomposition of $\cM$ (near $0$) comes from a good formal decomposition of $\cM$. This is a tautology if $\dim X=1$.

\begin{theoreme}[T\ptbl Mochizuki \cite{Mochizuki08,Mochizuki10b}]\label{th:gooddec}
Let $\cM$ be a free $\cO_{X,0}(*D)$-module equipped with a flat connection. If $\cM$ has a punctual good formal decomposition near the origin, then~$\cM$ has a good formal decomposition near~$0$, so in particular the sets $\wh\Phi_x\subset \cO_{\wh x}(*D)/\cO_{\wh x}$ ($x\in D$ near $0$) are induced by a single subset $\Phi\subset \cO_{X,0}(*D)/\cO_{X,0}$.
\end{theoreme}

\begin{remarque}
In dimension two, T\ptbl Mochizuki already proved in \cite{Mochizuki08} that the theorem holds with the a priori weaker condition of the existence of a formal good decomposition of~$\cM_{\wh0}$. However, due to known results on the generic existence of a good decomposition (\cf \eg \cite{Malgrange95}), the weaker condition is in fact equivalent to the stronger one.
\end{remarque}

The difficulty for solving Problem \ref{prob:MMhat} (under Condition \ref{def:goodformdec}\eqref{def:goodformdec3} in dimension~$\geq\nobreak3$) comes from the control of the formal coefficients of the polar parts of elements of~$\wh\Phi$, like $a(x_2)/x_1^k$. The original idea in \cite{Mochizuki08} consists in working with lattices (that is, locally free $\cO_X$-modules) instead of $\cO_X(*D)$-modules. The notion of a good lattice will be essential (\cf Definition \ref{def:goodformlattice} below). One of the main achievements in \cite{Mochizuki10b} is to prove the existence of good lattices under the assumption of punctual formal decomposition of $\cM$. This was already done in dimension two in the first version of \cite{Mochizuki08}.

The proof of Theorem \ref{th:gooddec} is done in two steps. In \cite{Mochizuki10b}, T\ptbl Mochizuki proves the existence of a good lattice (\cf Proposition \ref{prop:existgoodlattice}). The second step is given by Theorem \ref{th:goodlattice} below, which follows \cite{Mochizuki08}. We will only explain the second step.

\begin{definitio}[Good meromorphic connection]\label{def:goodmeroconn}
Let $\cM$ be a meromorphic connection with poles along a divisor~$D$ with normal crossings. We say that $\cM$ is a \emphb{good meromorphic connection} if, near any $x_o\in X$, there exists a ramification $\rho_{\bmd}:X_{\bmd}\to X$ around the components of~$D$ going through~$x_o$, such that $\rho_{\bmd}^+\cM$ has a good formal decomposition near $x_o$ (Definition \ref{def:goodformdec}\eqref{def:goodformdec2}).
\end{definitio}

\begin{remarque}
It follows from Theorem \ref{th:gooddec} and from the second part of Lemma \ref{lem:uniquegooddec} that any sub-meromorphic connection of a good meromorphic connection is also good. In particular, such a $\cO_X(*D)$-submodule is locally free.
\end{remarque}

\begin{remarque}[Existence after blowing-up]\label{rem:existblup}
As recalled in \S\ref{subsec:RHgoodsmoothintro}, the basic conjecture \cite[Conj\ptbl I.2.5.1]{Bibi97} asserted that, given any meromorphic bundle with connection on a complex surface, there exists a sequence of point blowing-ups such that the pull-back connection by this proper modification is good along the divisor of its poles. Fortunately, this conjecture has now been settled, in the algebraic setting by T\ptbl Mochizuki \cite{Mochizuki07b} and in general by K\ptbl Kedlaya \cite{Kedlaya09}. The natural extension of this conjecture in higher dimension is also settled in the algebraic case by T\ptbl Mochizuki \cite{Mochizuki08} (\cf also the survey~\cite{Mochizuki09}) and in the local analytic case by K\ptbl Kedlaya \cite{Kedlaya10}.
\end{remarque}

\begin{remarque}[The stratified $\ccI$-covering attached to a good meromorphic connection]\label{rem:I-stratcovmero}\index{stratified I-covering@stratified $\ccI$-covering}
Let $\cM$ be a germ at $0$ of good meromorphic connection with poles on~$D$ at most. Assume first that $\cM$ has a good formal decomposition indexed by a good finite set $\Phi\subset\cO_{X,0}(*D)/\cO_{X,0}$. If $U$ is small neighbourhood of $0$ in~$D$ on which each $\varphi\in\Phi$ is defined, the subset $\wt\Sigma_{|\varpi^{-1}(U)}\defin\bigcup_{\varphi\in\Phi}\varphi(\varpi^{-1}(U))\subset\ccIet_{\bun}$ defines a stratified $\ccI_{\bun}$-covering of $\varpi^{-1}(U)$ relative to the stratification induced by the $(Y_I)_{I\subset L}$ (\cf \S\ref{subsec:Inormcross} for the notation). If $\cM$ is only assumed to be good (Definition \ref{def:goodmeroconn}), \ie has a good formal decomposition after a ramification of order $\bmd$, one defines similarly a subset $\wt\Sigma_{|\varpi^{-1}(U)}$ of $\ccIet_{\bmd}$, hence of $\ccIet$, which is a stratified $\ccI$-covering with respect to the stratification $(Y_I)_{I\subset L}$.

By the uniqueness of the decomposition, this set is intrinsically attached to $\cM$, and therefore is globally defined along~$D$ when $\cM$ is so. We denote it \index{$SZIGMAWTM$@$\wt\Sigma(\cM)$}$\wt\Sigma(\cM)$.
\end{remarque}

\subsection{Good lattices}\label{subsec:goodlattices}
It will be implicit below that the poles of the meromorphic objects are contained in~$D$.

\skpt
\begin{definitio}[Good decomposition and good lattice, T\ptbl Mochizuki \cite{Mochizuki08}]\label{def:goodformlattice}\ligne
\begin{enumerate}
\item\label{def:goodformlattice1}
Let $\wh F$ be a free $\cO_{\wh0}$-module equipped with a flat meromorphic connection $\wh\nabla:\wh F\to\Omega_{\wh0}^1(*D)\otimes\wh F$. We say that $\wh F$ has a \emphb{good decomposition} if there exist a \emph{good} finite set $\wh\Phi\subset\cO_{\wh0}(*D)/\cO_{\wh0}$ and a decomposition
\[
(\wh F,\wh\nabla)=\bigoplus_{\wh\varphi\in \wh\Phi}(\wh F_{\wh\varphi}\wh\nabla)\qquad\text{with }(\wh F_{\wh\varphi}\wh\nabla)\simeq(\wh E^{\wh\varphi}\otimes \wh R_{\wh\varphi},\wh\nabla),
\]
where $\wh R_{\wh\varphi}$ is a free $\cO_{\wh0}$-module equipped with a flat meromorphic connection $\wh\nabla$ such that $\wh\nabla$ is logarithmic and $(\wh E^{\wh\varphi},\wh\nabla)=(\cO_{\wh0},d+d\wh\varphi)$.
\item\label{def:goodformlattice2}
Let $F$ be a free $\cO_{X,0}$-module equipped with a flat meromorphic connection $\nabla:F\to\Omega_{X,0}^1(*D)\otimes F$. We say that $(F,\nabla)$ is a (non-ramified) \emphb{good lattice} if $(F,\nabla)_{\wh0}\defin\cO_{\wh0}\otimes_{\cO_{X,0}}(F,\nabla)$ has a good decomposition indexed by some good finite set~$\wh\Phi$.
\end{enumerate}
\end{definitio}

In order to better understand the notion of a good lattice over $\cO_{X,0}$, we need supplementary notions, using the notation of \S\ref{subsec:prelimform}. Moreover, it will be useful later to distinguish between the pole set of $\Phi$ and that of the connection, and we will introduce a supplementary notation. Let $K\subset L$ a non-empty subset and set $D'=D(K)=\bigcup_{i\in K}D_i$. Similarly, for $I\subset K$, we still denote by $I^c$ the complement $K\moins I$ and we will use the notation $D'(I^c)$ for $\bigcup_{i\in I^c}D_i$. This should not lead to any confusion. In a first reading, one can assume that $K=L$ and set $D'=D$.

\begin{definitio}[$I$-goodness]
Let $I$ be a subset of~$K$. We say that a finite subset~$\Phi_I$ of $\cO_{X,0}(*D')/\cO_{X,0}(*D'(I^c))$ is \emph{good} if $\#\Phi_I=1$ or for any $\varphi_I\neq\psi_I$ in $\Phi_I$, the order $\ord^I(\varphi_I-\psi_I)$ exists (so belongs to $(-\NN)^K\moins\ZZ^{I^c}$) and the corresponding coefficient of $\varphi-\psi$, for some (or any) lifting of $\varphi_I-\psi_I$, does not vanish at~$0$ (in other words, $\varphi_I-\psi_I$ is \emph{purely $I$-monomial}). Setting $\bmm(\varphi_I-\psi_I)=-\ord^I(\varphi_I-\psi_I)$, for any \emph{fixed} $\psi_I\in\Phi_I$, the set $\{\bmm(\varphi_I-\psi_I)\mid\varphi_I\in\Phi_I,\,\varphi_I\neq\psi_I\}\subset\NN^K\moins\ZZ^{I^c}$ (more precisely, $(\NN^K\moins\ZZ^{I^c})\times\{0_{L\moins K}\}$) is totally ordered and its maximum, which does not depend on the choice of $\psi_I\in\Phi_I$, is denoted by $\bmm(\Phi_I)$. If $\#\Phi_I=1$, we set $\bmm(\Phi_I)=0$. We have similar definitions for a subset $\wh\Phi_I\subset\cO_{\wh0}(*D')/\cO_{\wh0}(*D'(I^c))$.
\end{definitio}

\skpt
\begin{definitio}[$\Phi$ and $\Phi_I$-decompositions, T\ptbl Mochizuki \cite{Mochizuki08}]\label{def:Phidecomp}\ligne
\begin{enumerate}
\item\label{def:Phidecomp1}
Let $I\subset K$. Let $(\IFhat,\Inablahat)$ be a free $\cO_{\wh D_I,0}$-module with flat meromorphic connection and let $\Phi_I$ be a finite subset of $\cO_{X,0}(*D')/\cO_{X,0}(*D'(I^c))$. We say that $(\IFhat,\Inablahat)$ has a $\Phi_I$-decomposition along~$D$ if there is a decomposition
\[
(\IFhat,\Inablahat)=\bigoplus_{\varphi_I\in\Phi_I}(\IFhat_{\varphi_I},\Inablahat)\qquad\text{with }(\IFhat_{\varphi_I},\Inablahat)\simeq(\wh E^{\wt\varphi_I}\otimes\IRhat_{\varphi_I},\Inablahat),
\]
with $(\wh E^{\wt\varphi_I},\Inablahat)=(\cO_{\wh D_I,0},d+d\wt\varphi_I)$ for some (or any) lifting $\wt\varphi_I$ of $\varphi_I$ in $\cO_{\wh D_I,0}(*D')$, and $(\IRhat_{\varphi_I},\Inablahat)$ is $I$-logarithmic, that is,
\[
\Inablahat(\IRhat_{\varphi_I})\subset\IRhat_{\varphi_I}\otimes\Big(\Omega^1_{\wh D_I,0}(\log D)+\Omega^1_{\wh D_I,0}(*D'(I^c))\Big),
\]
\ie is logarithmic only partially with respect to $D(I\cup(L\moins K))$, but can have poles of arbitrary order along $D'(I^c)$, although its residue along each $D_i$ ($i\in I\cup\nobreak(L\moins\nobreak K)$) does not have higher order poles. We have a similar definition for a subset $\wh\Phi_I\subset\cO_{\wh D_I,0}(*D')/\cO_{\wh D_I,0}(*D'(I^c))$. We say that $(\IFhat,\Inablahat)$ has a \emph{good $\Phi_I$-decomposition along~$D$} if $\Phi_I$ is good.

\item\label{def:Phidecomp2}
Let $(F,\nabla)$ be a free $\cO_{X,0}$-module with flat meromorphic connection. If there exists $\Phi\in\cO_{X,0}(*D')/\cO_{X,0}$ such that, denoting by $\Phi_I$ the image of $\Phi$ in $\cO_{X,0}(*D')/\cO_{X,0}(*D'(I^c))$, $(F,\nabla)_{|\wh D_I}$ has a good $\Phi_I$-decomposition for any $I\subset K$ in a compatible way with respect to the inclusions $I\subset I'\subset K$, we say that $(F,\nabla)$ has a good formal $\Phi$-decomposition.
\end{enumerate}
\end{definitio}

\begin{exemple}\label{ex:LEhat}
If $I=K$, Definition \ref{def:Phidecomp}\eqref{def:Phidecomp1} reads as follows: $(\KFhat,\Knablahat)$ is a free $\cO_{\wh D_K,0}$-module with flat meromorphic connection which has a decomposition indexed by $\Phi\subset\cO_{X,0}(*D')/\cO_{X,0}$ of the form
\[
(\KFhat,\Knablahat)=\bigoplus_{\varphi\in\Phi}(\KFhat_\varphi,\Knablahat)\qquad\text{with }(\KFhat_\varphi,\Knablahat)\simeq(\hat E^\varphi\otimes\KRhat_\varphi,\Knablahat)
\]
and $(\KRhat,\Knablahat)$ is $\cO_{\wh D_K,0}$-free and logarithmic with poles along $D$. Tensoring $(\KFhat,\Knablahat)$ with $\cO_{\wh0}$ produces $(\wh F,\wh\nabla)$ having a good decomposition as in Definition \ref{def:goodformlattice}\eqref{def:goodformlattice1}.

As a consequence, we see that if $(F,\nabla)$ has a good formal decomposition with good set $\Phi$ of exponential factors (in the sense of Definition \ref{def:Phidecomp}\eqref{def:Phidecomp2}), then $(F,\nabla)$ is a non-ramified good lattice in the sense of Definition \ref{def:goodformlattice}\eqref{def:goodformlattice2}. The converse is the content of Theorem \ref{th:goodlattice} below.
\end{exemple}

\begin{remarque}
We have stability by twist: Fix any $\eta_I\in\cO_{X,0}(*D')/\cO_{X,0}(*D'(I^c))$. Then for any lifting $\eta$ of $\eta_I$, $(\IFhat,\Inablahat+d\eta)$ has a $(\Phi_I+\eta_I)$-decomposition iff $(\IFhat,\Inablahat)$ has a $\Phi_I$-decomposition.

If $(\IFhat,\Inablahat)$ has a $\Phi_I$-decomposition at~$0$, it has such a decomposition on $D_I\cap U$ for some neighbourhood $U$ of~$0$. If the $\Phi_I$-decomposition is good at~$0$, it is good in some neighbourhood of~$0$. The same property holds for a formal $\Phi$-decomposition.
\end{remarque}

\begin{theoreme}[T\ptbl Mochizuki \cite{Mochizuki08}]\label{th:goodlattice}
Let $(F,\nabla)$ be a free $\cO_{X,0}$-module with flat meromorphic connection. If $(F,\nabla)$ is a (non-ramified) good lattice (\cf Definition \ref{def:goodformlattice}\eqref{def:goodformlattice2}) with formal exponential factors~$\wh\Phi\subset\cO_{\wh0}(*D')/\cO_{\wh0}$, then $\wh\Phi$ is equal to a subset $\Phi$ of $\cO_{X,0}(*D')/\cO_{X,0}$ and $(F,\nabla)$ has a good formal $\Phi$-decomposition~at~$0$ (\cf Definition \ref{def:Phidecomp}\eqref{def:Phidecomp2}).
\end{theoreme}

Let us emphasize the following result which will not be proved here.

\begin{proposition}[{\cite[Prop\ptbl2.18]{Mochizuki10b}}]\label{prop:existgoodlattice}
If $\cM$ satisfies the assumption of Theorem \ref{th:gooddec}, then any good lattice of $\cM_{\wh0}$ can be locally lifted as a good lattice of $\cM$ (near the origin).
\end{proposition}

\begin{remarque}\label{rem:lattices}
One can find more properties of good lattices in \cite{Mochizuki08}, in particular good Deligne-Malgrange lattices, which are essential for proving the local and global existence of (possibly ramified) good lattices, extending in this way the result of Malgrange \cite{Malgrange95, Malgrange04}, who shows the existence of a lattice which is generically good.
\end{remarque}

\subsection{\proofname\ of Theorem \ref{th:goodlattice}}
We first generalize to the present setting the classical decomposition with respect to the eigenvalues of the principal part of the connection matrix.

Let $U$ be some open neighbourhood of~$0$ in~$X$ and let $I\subset L$. Set $O_I=\cO(*D_I\cap\nobreak U)\lcr t_I\rcr$.

\begin{lemme}[Decomposition Lemma]\label{lem:decomposition}
Let $\IFhat$ be a free $O_I$-module with a flat connection $\Inablahat:\IFhat\to \IFhat\otimes\Omega^1_{O_I}(*D)$. Let $(\IFhat_{\wh0},\Inablahat_{\wh0})$ denote the formalization of $(\IFhat,\Inabla)$ at the origin. Assume that there exist $\bmm\in\NN^\ell-\{0\}$, $i\in I$ and a $\cO_{\wh0}$-basis of $\IFhat_{\wh0}$ such that, setting $I'=I\moins\{i\}$ and $O_{I'}=O_I/t_iO_I$,
\begin{enumerate}
\item\label{lem:decomposition1}
$m_i>0$,
\item\label{lem:decomposition2}
the matrix of $\Inablahat_{\wh0}$ in the given basis can be written as $t^{-\bmm}\wh\Omega$, where the entries of $\wh\Omega$ are in $\Omega^1_{\wh0}(\log D)$,
\item\label{lem:decomposition3}
if $\wh\Omega^{(i)}$ denotes the component of $\wh\Omega$ on $dt_i/t_i$, then $\wh\Omega^{(i)}(0)=\diag(\wh\Omega_1^{(i)}(0),\wh\Omega_2^{(i)}(0))$ where $\wh\Omega_1^{(i)}(0),\wh\Omega_2^{(i)}(0)$ have no common eigenvalues.
\end{enumerate}
Then, after possibly shrinking $U$, there exists a $\cO_I$-basis of $\IFhat$ such that the matrix of~$\Inablahat$ in the new basis is $t^{-\bmm}(\Omega+t_iB)$, where $\Omega$ has entries in $\Omega^1_{O_{I'}}(\log D)$, $B$ has entries in $\Omega^1_{O_I}(\log D)$, both are block-diagonal as $\wh\Omega^{(i)}(0)$ is, and $\Omega(0)=\wh\Omega(0)$.

Moreover, such a decomposition $(\IFhat,\Inablahat)=(\IFhat_1,\Inablahat)\oplus(\IFhat_2,\Inablahat)$ is unique.
\end{lemme}

\begin{proof}
Set $\tau_j=t_j\partial_{t_j}$ for $j\in L$ and $\tau_j=\partial_{t_j}$ for $j\notin L$ and let us denote by~$\wh\Omega^{(j)}$ the matrix of $\Inablahat_{\wh0,\tau_j}$ in the given basis of $\IFhat_{\wh0}$. Let us first notice that, due to the integrability condition and to the inequality $2\bmm>\bmm$, the matrices $\wh\Omega^{(i)}(0)$ and $\wh\Omega^{(j)}(0)$ commute, hence \eqref{lem:decomposition3} implies that $\wh\Omega(0)=\diag(\wh\Omega_1(0),\wh\Omega_2(0))$.

The first step consists in proving that the assumptions of the lemma hold for the matrix of $\Inablahat$ in any $O_I$-basis of $\IFhat$ which coincides with the given basis of $\IFhat_{\wh0}$ at the origin. Since the matrix of $t^{\bmm}\Inablahat_{\tau_j}$ in some basis of $\IFhat_{\wh0}$ has no pole, it has no pole in any basis of $\IFhat_{\wh0}$. Applying this to a basis of~$\IFhat_{\wh0}$ induced  by a basis of $\IFhat$, this implies that the matrix of~$t^{\bmm}\Inablahat_{\tau_j}$ in any basis of~$\IFhat$ has no pole, so \eqref{lem:decomposition2} holds. If we choose a $O_I$-basis which coincides with the given basis of $\IFhat_{\wh0}$ at the origin, then \eqref{lem:decomposition3} holds.

\subsubsection*{Second step}
It consists in proving that, up to shrinking $U$, there exists a $O_I$-basis of $\IFhat$ which coincides with the given basis when restricted to the origin, such that the matrix of $\Inablahat$ in this basis can be written as $t^{-\bmm}(\Omega+t_iA)$, with $\Omega$ in $\Omega^1_{O_{I'}}(\log D)$, $A\in\Omega^1_{O_I}(\log D)$,  $\Omega=\diag(\Omega_1,\Omega_2)$ and $\Omega_\alpha(0)=\wh\Omega_\alpha(0)$ ($\alpha=1,2$).

The operator $t^{\bmm}\Inablahat_{\tau_i}$ sends $t_i\IFhat$ to itself and its residual action on $\IFhat_{|D_i}$ defines a $O_{I'}$-linear operator $R_i$. By assumption, the characteristic polynomial $\chi_{R_i(0)}(T)$ decomposes as $\chi_1^0(T)\cdot\chi_2^0(T)$ with $a_1(0)\chi_1^0(T)+a_2(0)\chi_2^0(T)=1$ for some $a_1(0),a_2(0)\in\nobreak \CC$. It follows that there exist $a_\alpha\in O_{I'}$ and $\chi_\alpha\in O_{I'}[T]$ ($\alpha=1,2$) such that $\chi_{R_i}(T)=\chi_1(T)\cdot\chi_2(T)$ and $a_1\chi_1(T)+a_2\chi_2(T)=1$, up to shrinking $U$. Then $\IFhat_{|D_i}=\ker\chi_1(R_i)\oplus\ker\chi_2(R_i)$ and the matrix of $R_i$ decomposes into two block in a basis adapted to this decomposition. Moreover, we can realize this so that, at the origin, the matrix of $R_i(0)$ is $\wh\Omega{}^{(i)}(0)$.

For each $j\neq i$, let $\Omega^{(j)}$ be the matrix of $\Inablahat_{\tau_j}$ in the basis that we have obtained. The commutation relation $[\Inablahat_{\tau_i},\Inablahat_{\tau_j}]=0$ and the fact that $2\bmm>\bmm$ imply that $\Omega^{(j)}_{|D_i}$ commutes with $\Omega^{(i)}_{|D_i}$. Therefore the associated endomorphism preserves $\ker\chi_1(R_i)$ and $\ker\chi_2(R_i)$, that is, $\Omega^{(j)}_{|D_i}$ is block-diagonal. Hence, the matrix of $t^{\bmm}\Inablahat_{\tau_j}$ in a basis of $\IFhat$ lifting the previous one takes the desired form.

\subsubsection*{Third step}
Let $\Omega,A$ be as obtained in the second step. Let us start with the component $t^{-\bmm}(\Omega^{(i)}+t_iA^{(i)})$ of $t^{-\bmm}(\Omega+t_iA)$ on $dt_i/t_i$. We wish to find a change of basis $G=\id+\begin{smallpmatrix}0&t_iX\\t_iY&0\end{smallpmatrix}$, where $X,Y$ have entries in $O_I$, such that the matrix of $t^{\bmm}\Inablahat_{\tau_i}$ in the new basis is $\Omega^{(i)}+t_iB^{(i)}$, where $B^{(i)}$ has entries in $O_I$ and is block-diagonal as~$\Omega$. The matrix $G$ has to be a solution of the equation
\[
t^{\bmm}t_i\partial_{t_i}G=G\cdot\begin{pmatrix}\Omega_1^{(i)}+t_iB_1^{(i)}&0\\0&\Omega_2^{(i)}+t_iB_2^{(i)}\end{pmatrix}-\begin{pmatrix}\Omega_1^{(i)}+t_iA_{11}^{(i)}&t_iA_{12}^{(i)}\\t_iA_{21}^{(i)}&\Omega_2^{(i)}+t_iA_{22}^{(i)}\end{pmatrix}\cdot G,
\]
which reduces to
\[
B_1^{(i)}=A_{11}^{(i)}-t_iA_{12}^{(i)}Y,\quad B_2^{(i)}=A_{22}^{(i)}-t_iA_{21}^{(i)}X
\]
and
\[
\Omega_1^{(i)}X-X\Omega_2^{(i)}=A^{(i)}_{12}-t_i\big[A^{(i)}_{11}X-XA^{(i)}_{22}+t_i^2XA^{(i)}_{21}X+t^{\bmm-1_i}(t_i\partial_{t_i}+1)X\big]
\]
and a similar equation for $Y$. The assumption implies that the determinant of the endomorphism $Z\mto(\Omega_1^{(i)}(0)Z-Z\Omega_2^{(i)}(0))$ is not zero. Choose $U$ such that the determinant of $Z\mto(\Omega_1^{(i)}Z-Z\Omega_2^{(i)})$ does not vanish on $U$. We then find a (unique) solution of the equation for $X$ (\resp $Y$) by expanding $X$ (\resp $Y$) with respect to $t_i$. Let us write $t^{-\bmm}(\Omega+t_iB)$ the matrix of $\Inablahat$ after the base change induced by $G$.

The integrability condition is now enough to show that the matrix $B$ is block-diagonal. Indeed, let us denote by
\[
t^{-\bmm}\begin{pmatrix}\Omega_1^{(j)}+t_iB_{11}^{(j)}&t_iB_{12}^{(j)}\\ t_iB_{21}^{(j)}&\Omega_2^{(j)}+t_iB_{22}^{(j)}\end{pmatrix}
\]
the matrix of $\Inablahat_{\tau_j}$. Then the integrability condition $[\Inablahat_{\tau_i},\Inablahat_{\tau_j}]$ reads on the off-diagonal blocks as
\[
\Omega_1^{(i)}B_{12}^{(j)}-B_{12}^{(j)}\Omega_2^{(i)}=t_i\big[B_{12}^{(j)}B^{(i)}_2-B^{(i)}_1B_{12}^{(j)}-t^{\bmm-1_i}(\tau_i+1-m_i)B_{12}^{(j)}\big],
\]
whose unique solution is $B_{12}^{(j)}=0$, as seen by expanding $B_{12}^{(j)}$ with respect to $t_i$, since $Z\mto(\Omega_1^{(i)}Z-Z\Omega_2^{(i)})$ is invertible. Therefore, the matrix of $\Inablahat_{\tau_j}$ is also block-diagonal in the new basis.

For the uniqueness, we are reduced to proving that there is no nonzero morphism between $(\IFhat_1,\Inablahat)$ and $(\IFhat{}'_2,\Inablahat)$ (with obvious notation). The proof is similar.
\end{proof}

\begin{remarque}\label{rem:decomposition}
The previous lemma also holds when one replaces $O_I$ with $\cO_{\wh0}$, by forgetting the first step.
\end{remarque}

\begin{proof}[\proofname\ of Theorem \ref{th:goodlattice}]
We start with:

\begin{lemme}\label{lem:I}
Let $I\subset K$ and let $(\IFhat,\Inablahat)$ be a free $\cO_{\wh D_I,0}$-module with flat meromorphic connection. Assume that
there exists a good finite subset $\wh\Phi\subset\cO_{\wh0}(*D')/\cO_{\wh0}$ such that the formalization $(\IFhat,\Inablahat)_{\wh0}$ at $0$ satisfies:
\bgroup\numstareq
\begin{equation}\label{eq:I*}
(\IFhat,\Inablahat)_{\wh0}=\bigoplus_{\wh\varphi\in\wh\Phi}(\IFhat_{\wh0,\wh\varphi},\Inablahat)\qquad\text{with }(\IFhat_{\wh0,\wh\varphi},\Inablahat)\simeq(\wh E^{\wh\varphi}\otimes\wh R_{\wh\varphi},\wh\nabla)
\end{equation}
\egroup
where
\bgroup\numstarstareq
\begin{equation}\label{eq:I**}
\wh\nabla\wh R_{\wh\varphi}\subset\wh R_{\wh\varphi}\otimes\Omega^1_{\wh0}(\log D)
\end{equation}
\egroup
for any $\wh\varphi\in\wh\Phi$. If $\wh\Phi_I$ denotes the image of $\wh\Phi$ in $\cO_{\wh0}(*D')/\cO_{\wh0}(*D'(I^c))$, then $\wh\Phi_I$ is equal to a subset $\Phi_I$ of $\cO_{\wh D_I,0}(*D')/\cO_{\wh D_I,0}(*D'(I^c))$, and $(\IFhat,\Inablahat)$ has a $\Phi_I$-decomposition.
\end{lemme}

\begin{proof}
The proof is by induction on $\#\wh\Phi_I$. Assume first $\#\wh\Phi_I=1$, that is, $\wh\Phi_I=\{\wh\varphi_I\}$. Then for any lifting $\wh\varphi$ of $\wh\varphi_I$ in $\cO_{\wh0}(*D')/\cO_{\wh0}$ and any $\wh\psi\in\wh\Phi$, $d(\wh\psi-\wh\varphi)$ has poles in $D'(I^c)$ at most, hence we have
\[
(\Inablahat-d\wh\varphi)\IFhat_{\wh0}\subset\IFhat_{\wh0}\otimes\Big(\Omega^1_{\wh0}(\log D)+\Omega^1_{\wh0}(*D'(I^c))\Big).
\]
Writing this in some basis of $\IFhat_{\wh0}$ induced by a $\cO_{\wh D_I,0}$-basis of $\IFhat$ implies that $\wh\varphi_I$ belongs to the quotient $\cO_{\wh D_I,0}(*D')/\cO_{\wh D_I,0}(*D'(I^c))$ and we denote it by $\varphi_I$, hence the first statement. \hbox{After} a twist by $\cE^{-\wt\varphi_I}$ for some lifting $\wt\varphi_I$ of $\varphi_I$, we can then assume that $\wh\Phi_I=0$, and we choose $\wh\varphi=0$. Then
\[
\Inablahat\IFhat_{\wh0}\subset\IFhat_{\wh0}\otimes\Big(\Omega^1_{\wh0}(\log D)+\Omega^1_{\wh0}(*D'(I^c))\Big),
\]
and this implies
\[
\Inablahat\IFhat\subset\IFhat\otimes\Big(\Omega^1_{\wh D_I,0}(\log D)+\Omega^1_{\wh D_I,0}(*D'(I^c))\Big),
\]
which is the $\Phi_I$-decomposition with $\Phi_I=\{0\}$.

Assume now that $\#\wh\Phi_I\geq2$. Then $\#\wh\Phi\geq2$ and $\bmm\defin\bmm(\wh\Phi)\neq0$. Moreover, the $I$-component $\bmm_I$ is nonzero. Let $i\in I$ be such that $m_i>0$ and set $I'=I-\{i\}$.

Assume first that there exists $\wh\varphi\in\wh\Phi$ with $\ord^L(\wh\varphi)=-\bmm$. Set $C\defin(t^{\bmm}\wh\Phi)(0)\subset\CC$. Note that $\#C\geq\nobreak2$ since $\bmm\neq0$. Consider the decomposition coarser than \eqref{eq:I*} indexed by $C$: $(\IFhat,\Inablahat)_{\wh0}=\bigoplus_{c\in C}(\IFhat_{\wh0,c},\Inablahat_c)$. It satisfies the assumptions of the Decomposition Lemma \ref{lem:decomposition}. Applying its existence part, we find a decomposition $(\IFhat,\Inablahat)=\bigoplus_{c\in C}(\IFhat_c,\Inablahat_c)$, so the matrix of $\Inablahat$ takes the form $t^{-\bmm}(\Omega+t_iB)$ where~$B$ is now block-diagonal as $\Omega$. Let us set $\wh\Phi_c=\{\wh\varphi\in\wh\Phi\mid(t^{-\bmm}\wh\varphi)(0)=c\}$. Then $\wh\Phi=\sqcup_c\wh\Phi_c$ and, as $\bmm_I\neq0$, $\wh\Phi_I=\sqcup_c\wh\Phi_{c,I}$. Thus, for every $c\in C$, $\#\wh\Phi_{c,I}<\#\wh\Phi_I$. By the uniqueness in the Decomposition Lemma \ref{lem:decomposition} (in the variant of Remark \ref{rem:decomposition}), we have
\[
(\IFhat_c,\Inablahat_c)_{\wh0}=\bigoplus_{\wh\varphi\in\wh\Phi_c}(\wh E^{\wh\varphi}\otimes\IRhat_{\wh\varphi},\wh\nabla),
\]
so that each $(\IFhat_c,\Inablahat_c)$ satisfies the assumption of the lemma. We conclude by induction on $\#\wh\Phi_I$ to get that $\wh\Phi_I\subset\cO_{\wh D_I,0}(*D')/\cO_{\wh D_I,0}(*D'(I^c))$ and the $\Phi_I$-decomposition of $(\IFhat,\Inablahat)$.

Assume now that there is no $\wh\varphi\in\wh\Phi$ such that $\ord^L(\wh\varphi)=-\bmm$. Then there exists $\wh f\in\cO_{\wh0}(*D')/t^{-\bmm}\cO_{\wh0}$ such that $\ord^L(\wh\varphi-\wh f)\geq-\bmm$ for any $\wh\varphi\in\wh\Phi$ (and equality for some $\wh\varphi\in\wh\Phi$). It is thus enough to prove that $\wh f\in\cO_{\wh D_I,0}(*D')/t^{-\bmm}\cO_{\wh D_I,0}$ and (\cf Lemma \ref{lem:formelconv}) this is equivalent to showing that, for any $i\in K$, the class $\wh f^{(i)}$ of~$\wh f$ in $\cO_{\wh0}(*D')/t_i^{-m_i}\cO_{\wh0}(*D'(i^c))$ belongs to $\cO_{\wh D_I,0}(*D')/t_i^{-m_i}\cO_{\wh D_I,0}(*D'(i^c))$. It will then be denoted by $f^{(i)}$.

Let us fix $i\in K$ and set $\wh f^{(i)}=\sum_{\nu=m_i+1}^{\nu_i}\wh f^{(i)}_\nu t_i^{-\nu}$. Let us prove by induction on $k\in\{0,\dots,\nu_i-m_i-1\}$ that $\wh f^{(i)}_{\nu_i-k}\in\cO_{\wh D_I\cap D_i}(*D')/t_i^{-m_i}\cO_{\wh D_I\cap D_i}(*D'(i^c))$. For any $\wh\varphi\in\wh\Phi$ we can write $\wh\varphi=\wh f+t^{-\bmm}\wh\eta$ with $\wh\eta\in\cO_{\wh0}$. Because of \eqref{eq:I**}, for any $\wh\varphi\in\wh\Phi$ we have $(\Inablahat_{t_i\partial_{t_i}}-t_i\partial\wh\varphi/\partial t_i)\IFhat_{\wh0,\wh\varphi}\subset\IFhat_{\wh0,\wh\varphi}$ and thus, for any $k$ as above,
\[
\Big(t_i^{\nu_i}\Inablahat_{t_i\partial_{t_i}}-\sum_{\mu=0}^k(\mu-\nu_i)\wh f^{(i)}_{\nu_i-\mu}t_i^\mu\Big)\IFhat_{\wh0,\wh\varphi}\subset t_i^{k+1}\cdot\IFhat_{\wh0,\wh\varphi}(*D'(i^c)),
\]
hence also
\[
\Big(t_i^{\nu_i}\Inablahat_{t_i\partial_{t_i}}-\sum_{\mu=0}^k(\mu-\nu_i)\wh f^{(i)}_{\nu_i-\mu}t_i^\mu\Big)\IFhat_{\wh0}\subset t_i^{k+1}\cdot\IFhat_{\wh0}(*D'(i^c)).
\]

This implies (taking $k=0$) that $t_i^{\nu_i}\Inablahat_{t_i\partial_{t_i}}\IFhat\subset\IFhat(*D'(i^c))$, and by induction that the $\cO_{\wh D_I\cap D_i,0}(*D'(i^c))$-linear endomorphism that $\big(t_i^{\nu_i}\Inablahat_{t_i\partial_{t_i}}-\sum_{\mu=0}^{k-1}(\nu_i-\mu)f^{(i)}_{\nu_i-\mu}t_i^\mu\big)$ induced on $t_i^k\cdot\IFhat(*D'(i^c))/t_i^{k+1}\cdot\IFhat(*D'(i^c))$ acts as $(\nu_i-k)\wh f^{(i)}_{\nu_i-k}\id$. It follows that $\wh f^{(i)}_{\nu_i-k}=f^{(i)}_{\nu_i-k}\in\cO_{\wh D_I\cap D_i}(*D')/t_i^{-m_i}\cO_{\wh D_I\cap D_i}(*D'(i^c))$.
\end{proof}

\subsubsection*{End of the proof of Theorem \ref{th:goodlattice}}
The assumption of the theorem implies that \eqref{eq:I*} and \eqref{eq:I**} are satisfied by $(\IFhat,\Inablahat)\defin(F,\nabla)_{\wh D_I}$ for any $I\subset K$, so $\wh\Phi_I\subset\cO_{\wh D_I,0}(*D')/\cO_{\wh D_I,0}(*D'(I^c))$ for any such $I$, and in particular for any $I=\{i\}$. Lemma \ref{lem:formelconv} then implies that $\wh\Phi\subset\cO_{X,0}(*D')/\cO_{X,0}$. The existence of a $\Phi_I$-decomposition in the second part of the theorem follows then from Lemma \ref{lem:I} applied to any $I\subset K$ and the compatibility with respect to the inclusions $I'\supset I$ follows from the uniqueness in Lemma \ref{lem:decomposition}.
\end{proof}

\begin{corollaire}[Good Deligne-Malgrange lattices]\label{cor:goodDM}\index{good lattice!Deligne-Malgrange}
Let $\cM$ be as in Theorem \ref{th:gooddec}. Then there exists a good lattice $(F,\nabla)$ of $\cM$ satisfying the \index{non-resonance condition}non-resonance condition:
\begin{itemize}
\item
for each $i\in L\moins K$, the eigenvalues of the residue $R_i$ of $\nabla$ along $D_i$ do not differ by a nonzero integer.
\end{itemize}
\end{corollaire}

\begin{proof}
Fix $i\in L\moins K$. Let $(F,\nabla)$ be a good lattice of $\cM$. By definition, there exists a basis of $\wh F\defin F_{\wh0}$ in which the component of the matrix of $\wh\nabla$ on $dt_i/t_i$ has no pole. This holds then in any basis of $\wh F$. Choosing a basis induced by a basis of $(F,\nabla)$ implies that $\nabla$ is logarithmic along $D_i$. Its residue $R_i$ is an endomorphism of $F_{|D_i}$ which is horizontal with respect to the connection induced on $F_{|D_i}$. Its eigenvalues are then constant, and coincide with the eigenvalues of the residue $\wh R_i$ of $(\wh F,\wh\nabla)$.

Therefore, if we start from a good lattice $(\wh F,\wh\nabla)$ of $\cM_{\wh0}$ which is non-resonant with respect to $D_i$, that is, such that the eigenvalues of $\wh R_i$ do not differ by a nonzero integer, then the good lattice $(F,\nabla)$ given by Proposition \ref{prop:existgoodlattice} is also non-resonant with respect to $D_i$.
\end{proof}

\begin{remarque}[Very good formal decomposition]\label{rem:verygood}
For the purpose of asymptotic analysis (\cf Theorem \ref{th:HTM} below), it would be more convenient to know that~$\cM$ has a \emphb{very good formal decomposition} near the origin, as defined in \cite{Bibi93}, that is, considering the formal completion $\cO_{\wh D}\defin\varprojlim_k\cO_X/(x_1\cdots x_\ell)^k\cO_X$, that $(\cM_{\wh D},\nabla)$ decomposes in a way similar to \eqref{eq:goodformdec*}. This property is too strong in general, and cannot be achieved after blowing-ups (\cf Remark \ref{rem:existblup}).

Nevertheless, if $\cM$ has a punctual good formal decomposition near the origin, the existence of a good lattice allows one to prove a similar but weaker property. Let us fix $\varphi_o\in\Phi$ and let us set $\Phi-\varphi_o=\bigcup_c(\Phi_c-\varphi_o)$ where $c$ varies in $C\defin[t^{\bmm}(\Phi-\varphi_o)](0)\subset\CC$ and where we set $\bmm=\bmm(\Phi)$. We now choose $K$ to be \emph{minimal} with respect to $\Phi\subset\cO_{X,0}(*D)/\cO_{X,0}$, that is, $m_i>0$ iff $i\in K$. Then for each $I\subset K$, if we define $C_I$ as $C$, but starting from $\Phi_I-\varphi_{o,I}$, we find $C_I=C$, since none of the $c/t^{\bmm}$ becomes zero in $\cO_{X,0}(*D')/\cO_{X,0}(*D'(I^c))$.

\begin{corollaire}[{of Th\ptbl\ref{th:goodlattice}, \cf\cite[\S2.4.3]{Mochizuki08}}]\label{cor:pseudoverygood}
Let $(F,\nabla)$ be a good lattice with associated set $\Phi\subset\cO_{X,0}(*D)/\cO_{X,0}$ (according to Theorem \ref{th:goodlattice}). Assume that~$K\subset L$ is minimal with respect to $\Phi$. Then there exists a decomposition
\bgroup\numstareq
\begin{equation}\label{eq:goodformdecD*}
(F,\nabla)_{\wh D',0}\simeq\bigoplus_{c\in C}(\wh F_c,\wh\nabla_c),
\end{equation}
\egroup
where each $(\wh F_c,\wh\nabla_c)$ is a free $\cO_{\wh D',0}$-module with integrable meromorphic connection such that, for each $I\subset K$, we have compatible isomorphisms
\bgroup\numstarstareq
\begin{equation}\label{eq:goodformdecD**}
\cO_{\wh D_I,0}\otimes(\wh F_c,\wh\nabla_c)=\bigoplus_{\varphi_I\in(\Phi_c)_I}(\IFhat_{\varphi_I},\Inablahat_c)\quad\text{with }(\IFhat_{\varphi_I},\Inablahat_c)\simeq(\wh E^{\wt\varphi_I}\otimes\IRhat_{\varphi_I},\Inablahat),
\end{equation}
\egroup
where the right-hand term is as in Definition \ref{def:Phidecomp}\eqref{def:Phidecomp1}.
\end{corollaire}

\begin{proof}
Recall that the Mayer-Vietoris complex (where we always take $I\subset K$)
\begin{equation}\label{eq:MVO}\tag*{$\MV(\cO):$}
\{0\ra\bigoplus_{\#I=1}\cO_{\wh D_I,0}\ra\bigoplus_{\#I=2}\cO_{\wh D_I,0}\ra\cdots\ra\bigoplus_{\#I=\ell-1}\cO_{\wh D_I,0}\ra\cO_{\wh D_K,0}\ra0\}
\end{equation}
is a resolution of $\cO_{\wh D',0}$. By assumption, for each $I$ one can gather the terms in the $\Phi_I$-decomposition of $(F,\nabla)_{\wh D_I,0}$ so that the decomposition is indexed by $C=C_I$. It follows that there is a decomposition indexed by $C$ of each $\cO_{\wh D_I,0}\otimes F_{\wh D',0}$ compatible with the inclusions $I'\subset I$. As a consequence, for each $c\in C$, there is a Mayer-Vietoris complex $\MV(\wh F_c)$ whose terms are made with the $\IFhat_c$. Moreover, this complex is a direct summand of the Mayer-Vietoris complex $\MV(\cO)\otimes F_{\wh D',0}$. Therefore, $\MV(\wh F_c)$ is a resolution of some free $\cO_{\wh D',0}$-submodule $\wh F_c$ of $F_{\wh D',0}$, and $F_{\wh D',0}$ is the direct sum of these modules. Lastly, the components $(c,c')$ of the matrix of the connection $\nabla$ on $F_{\wh D',0}$ have a vanishing formal expansion along each $D_i$ ($i\in K$) if $c\neq c'$. It follows that they vanish, according to the injectivity of $\cO_{\wh D',0}\to\bigoplus_{\#I=1}\cO_{\wh D_I,0}$.
\end{proof}

\begin{remarque}\label{rem:pseudoverygood}
We keep the assumption of Corollary \ref{cor:pseudoverygood} and we fix $\varphi_o\in\Phi$. Let us denote by $\bell$ the predecessor of $\bmm(\Phi)$ in the totally ordered set $\{\bmm(\varphi-\varphi_o)\mid\varphi\in \Phi\}$. In order to simplify the notation, we will assume that $\varphi_o=0$.

For each $c\in C$ and each pair $\varphi,\psi\in\Phi_c$, we have $\varphi-\psi\in t^{-\bell}\cO_{X,0}$. Let us choose $\varphi_c\in\Phi_c$ for each $c$. Then the formalization at $0$ of $\wh\nabla_c$ satisfies
\[
(\wh\nabla_{c,\wh0}-d\varphi_c\id)(\wh F_{c,\wh0})\subset t^{-\bell}\Omega^1_{\wh 0}(\log D)\otimes \wh F_{c,\wh0},
\]
according to \eqref{eq:goodformdecD**} for $I=K$, and to \eqref{eq:I*} and \eqref{eq:I**}. Therefore, in any $\cO_{\wh0}$-basis of $\wh F_{c,\wh0}$, the matrix of $\wh\nabla_{c,\wh0}-d\varphi_c\id$ has entries in $t^{-\bell}\Omega^1_{\wh 0}(\log D)$. If one chooses a basis induced by a $\cO_{\wh D',0}$-basis of $\wh F_c$, one concludes that the matrix of $\wh\nabla_c-d\varphi_c\id$ in such a basis has entries in $t^{-\bell}\Omega^1_{\wh D',0}(\log D)$, and thus
\numstareq\bgroup
\begin{equation}\label{eq:pseudoverygood*}
\forall c\in C,\quad(\wh\nabla_c-d\varphi_c\id)(\wh F_c)\subset t^{-\bell}\Omega^1_{\wh D',0}(\log D)\otimes \wh F_c.
\end{equation}
In conclusion, in any $\cO_{\wh D',0}$-basis of $\wh F_c$, the matrix of  $\wh\nabla_c$ takes the form $d\varphi_c\id+t^{-\bell}\wh\Omega_c$, where $\wh\Omega_c$ has entries in $\Omega^1_{\wh D',0}(\log D)$.
\egroup
\end{remarque}

Going back to the notion of very good formal decomposition, we obtain:

\begin{corollaire}\label{cor:verygood}
If moreover, $m(\varphi-\psi)_i>0$ for each pair $\varphi\neq\psi\in\Phi$ and each $i\in L$, then $\cM$ has a very good formal decomposition.
\end{corollaire}

In dimension two, this has been proved in \cite[Th\ptbl I.2.2.4]{Bibi97}. This amounts to asking that the stratified $\ccI$-covering $\wt\Sigma(\cM)$ (\cf Remark \ref{rem:I-stratcovmero}) is a true covering of $\partial\wt X(D)$. This result will not be used but can be regarded as the simplest possible behaviour of a flat meromorphic connection at a crossing point of $D$.

\begin{proof}[Sketch of proof of Corollary \ref{cor:verygood}]
The assumption implies that $K=L$. One can argue by induction on $\#\Phi$, since the assumption implies that each~$\Phi_c$ satisfies $m_i(\Phi_c)>0$ for each $i$. The case where $\#\Phi=1$ reduces to proving that a free $\cO_{\wh D,0}$-module with an integrable meromorphic connection such that the connection is logarithmic in a suitable $\cO_{\wh D_L,0}$-basis, has a basis where the connection is logarithmic. This is standard.
\end{proof}
\end{remarque}

\subsection{Comments}
Although some results concerning the notion of a good formal structure have already been considered in dimension two in \cite{Bibi97}, the much more efficient presentation given here is due to T\ptbl Mochizuki \cite[Chap\ptbl2]{Mochizuki08}, \cite{Mochizuki10b}. The advantage of using lattices lies in the decomposition lemma \ref{lem:decomposition}.

\chapter[The Riemann-Hilbert correspondence]{Good meromorphic connections (analytic~theory) and\\ the Riemann-Hilbert correspondence}\label{chap:RHgoodnc}

\begin{sommaire}
This \chaptersname is similar to \Chaptersname\ref{chap:RHgoodsmooth}, but we now assume that~$D$ is a divisor with normal crossings. We start by proving the many-variable version of the Hukuhara-Turrittin theorem, that we have already encountered in the case of a smooth divisor. It will be instrumental for making the link between formal and holomorphic aspects of the theory. The new point in the proof of the Riemann-Hilbert correspondence is the presence of non-Hausdorff étale spaces, and we need to use the level structure to prove the local essential surjectivity of the Riemann-Hilbert functor. As an application of the Riemann-Hilbert correspondence in the good case and of  the fundamental results of K\ptbl Kedlaya  and T\ptbl Mochizuki on the elimination of turning points by complex blowing-ups, we prove a conjecture of M\ptbl Kashiwara asserting that the Hermitian dual of a holonomic $\cD$-module is holonomic, generalizing the original result of M\ptbl Kashiwara for regular holonomic $\cD$-modules to possibly irregular holonomic $\cD$-modules and the result of \Chaptersname\ref{chap:holdist} to higer dimensions.
\end{sommaire}

\subsection{Introduction}
This \chaptersname achieves the main goal of this series of \chaptername s, by proving the Riemann-Hilbert correspondence between good meromorphic connections with poles along a divisor with normal crossings and Stokes-filtered local systems on the corresponding real blow-up space. The correspondence is stated in a global way, by using the notion of good $\ccI$-stratified covering and that of $\ccI$-filtration introduced in \Chaptersname\ref{chap:Ifil}. For a general flat meromorphic connection, this correspondence can be used together with the theorem of K\ptbl Kedlaya and T\ptbl Mochizuki (\cf Remark \ref{rem:existblup}) proving the possibility of eliminating the turning points of any meromorphic connection.

Before proving this correspondence, we will complete the analysis of the formal properties of good meromorphic connections given in \Chaptersname\ref{chap:goodformal} by a generalization in higher dimensions of the Hukuhara-Turrittin theorem. The proof that we give is due to T\ptbl Mochizuki. As in dimension one, this result plays an essential role in the proof of the Riemann-Hilbert correspondence. Let us note however that we will not develop the analogues with precise order of growth, related to Gevrey formal power series. We will neither try to develop the relation with multi-summation, which is hidden behind the level structure of the associated Stokes-filtered local system, as considered in \S\ref{subsec:Stokesstr}.

We give an application of the Riemann-Hilbert correspondence to distributions solutions of holonomic $\cD$-modules, in order to show the powerfulness of this correspondence, and the need of the goodness condition as a new ingredient compared with the proof in dimension one (\cf \Chaptersname\ref{chap:holdist}).

In this \chaptername, we consider germs along~$D$ of meromorphic connections with poles on~$D$ at most, which are locally good in the sense of Definition \ref{def:goodmeroconn} (in particular, we assume the local existence of a good lattice, but we do not care of the global existence of such a lattice, which could be proved by using a good Deligne-Malgrange lattice, \cf \cite{Mochizuki08}).

\subsection{Notation}\label{subsec:notationRHgeq2}
We consider the following setting:
\begin{itemize}
\item
$X$ is a complex manifold and~$D$ is a divisor with normal crossings in~$X$, with smooth components $D_j$ ($j\in J$), and $X^*\defin X\moins D$,
\item
$j:X^*\hto X$ and $\wtj:X^*\hto\wt X(D_{j\in J})$ denote the open inclusions, and $i:D\hto X$ and $\wti:\partial\wt X(D_{j\in J})\hto \wt X(D_{j\in J})$ denote the closed inclusions,
\item
the ordered sheaf~$\ccI$ on $\partial\wt X(D_{j\in J})$ is as in Definitions \ref{def:Igeneral} and \ref{def:Igeneralorder}.
\end{itemize}

In the local setting, we keep the notation of \S\ref{subsec:prelimform}. For short, we will denote by~$\wt X$ the real blow-up space $\wt X(D_{i\in L})$, by $\varpi:\wt X\to X$ the real blowing-up map and by~$\cA_{\wt X}$ the sheaf on~$\wt X$ consisting, locally, of $C^\infty$ functions on~$\wt X$ which are annihilated by $\ov t_i\ov\partial_{t_i}$ ($i\in L$) and $\ov\partial_{t_i}$ ($i\not\in L$). We will set $\cA_{\partial\wt X}=\cA_{\wt X}{}_{|\partial\wt X}$. We refer to \cite{Majima84}, \cite{Bibi93} and \cite[Chap\ptbl2]{Bibi97} for the main properties of this sheaf. Notice that $\cA_{\wt X}$ is a subsheaf of the sheaf $\cA_{\wt X}^\modD$ introduced in \S\ref{subsec:modgrowthfunct}.

Recall also that for each $I\subset L$ one considers the formal completion $\cA_{\wh D_I}=\varprojlim_k\cA_{\wt X}/(t_{i\in I})^k\cA_{\wt X}$ which is a sheaf supported on $\varpi^{-1}D_I$ and is a $\varpi^{-1}\cO_{\wh D_I}$-module via the natural inclusion $\varpi^{-1}\cO_{\wh D_I}\hto\cA_{\wh D_I}$, which is seen as follows. If $\varpi_{D_I}:\wt D_I\to\nobreak D_I$ denotes the real blow-up space of $D_I$ along $D_I\cap D_{j\in I^c}$, then $\cA_{\wh D_I}$ is identified with $\varpi_I^{-1}\cA_{\wt D_I}\lcr t_I\rcr$, where $\varpi_I$ denotes here the projection $\wt X_{|D_I}\to\wt D_I$. The inclusion $\varpi_{D_I}^{-1}\cO_{\wh D_I}=\varpi_{D_I}^{-1}\cO_{D_I}\lcr t_I\rcr\hto\cA_{\wt D_I}\lcr t_I\rcr$ produces the desired inclusion.

One also considers $\cA_{\wh D}$ defined similarly and supported on $\partial\wt X(D)$. There is a Mayer-Vietoris complex $\MV(\cA_{\wt X})$ similar to the complex $\MV(\cO)$ considered in the proof of Corollary \ref{cor:pseudoverygood}, which is a resolution of $\cA_{\wh D}$ and the inclusion $\varpi^{-1}\MV(\cO)\hto\MV(\cA_{\wt X})$ induces an inclusion $\varpi^{-1}\cO_{\wh D}\hto\cA_{\wh D}$ which makes $\cA_{\wh D}$ a $\varpi^{-1}\cO_{\wh D}$-module.

Lastly, let us recall that there is an exact sequence
\begin{equation}\label{eq:Taylor}
0\to\cA_{\partial\wt X}^\rdD\to\cA_{\partial\wt X}\To{T_D}\cA_{\wh D}\to0.
\end{equation}

\subsection{The Malgrange-Sibuya theorem in higher dimension}\label{subsec:proofMalgrange-Sibuya}

We consider the local setting. As in the cases treated before (\S\ref{subsec:RHholone} and \ref{subsec:proofessentialsurj}), an important point in the proof of the Riemann-Hilbert correspondence is the \emphb{Malgrange-Sibuya theorem}. It will be also important for the proof of the higher dimensional  Hukuhara-Turrittin theorem.

\begin{theoreme}[Malgrange-Sibuya in dimension $\geq2$]\label{th:Malgrange-Sibuya}
The image of
\[
H^1\big((S^1)^\ell,\GL_d^\rdD(\cA_{\partial\wt X})\big)\to H^1\big((S^1)^\ell,\GL_d(\cA_{\partial\wt X})\big)
\]
is the identity.
\end{theoreme}

\begin{proof}[\proofname\ of Theorem \ref{th:Malgrange-Sibuya}]
Since the proof of Theorem \ref{th:Malgrange-Sibuya} given in \cite[\S II.1.2]{Bibi97} contains a small mistake\footnote{whose correction is available at \url{http://www.math.polytechnique.fr/~sabbah/sabbah_ast_263_err.pdf}}, we give a detailed proof here.

We denote by $\cE_{\wt X}$ the sheaf of $C^\infty$ functions on $\wt X$ and by $\cE_{|\partial\wt X}$ its sheaf-theoretic restriction to $\partial\wt X$. We can then define the subsheaf $\cE_{|\partial\wt X}^\rdD$ of $C^\infty$ functions with rapid decay and the quotient sheaf is that of $C^\infty$ formal functions along $\partial\wt X$. We note that a local section of $\rM_d(\cE_{|\partial\wt X})$ (matrices of size $d$ with entries in $\cE_{|\partial\wt X}$) is a section of $\GL_d(\cE_{|\partial\wt X})$ if and only if its image in $\rM_d(\cE_{|\partial\wt X}/\cE_{|\partial\wt X}^\rdD)$ belongs to $\GL_d(\cE_{|\partial\wt X}/\cE_{|\partial\wt X}^\rdD)$: indeed, given a matrix in $\rM_d(\cE_{|\partial\wt X})$, the image in $\cE_{|\partial\wt X}/\cE_{|\partial\wt X}^\rdD$ of its determinant is the determinant of its image in $\rM_d(\cE_{|\partial\wt X}/\cE_{|\partial\wt X}^\rdD)$; use now that the values of a local section of $\cE_{|\partial\wt X}$ at the points of $\partial\wt X$ are also the values of its image in $\cE_{|\partial\wt X}/\cE_{|\partial\wt X}^\rdD$.

\begin{lemme}
We have $H^1\big((S^1)^\ell,\GL_d^\rdD(\cE_{|\partial\wt X})\big)=\id$, where we have set as above $\GL_d^\rdD(\cE_{|\partial\wt X})\defin\id+\rM_d(\cE_{|\partial\wt X}^\rdD)$.
\end{lemme}

\begin{proof}
Since a matrix in $\rM_d(\cE_{|\partial\wt X})$ whose image in $\rM_d(\cE_{|\partial\wt X}/\cE_{|\partial\wt X}^\rdD)$ is the identity is invertible, we have an exact sequence of groups
\[
\id\to\GL_d^\rdD(\cE_{|\partial\wt X})\to\GL_d(\cE_{|\partial\wt X})\to\GL_d(\cE_{|\partial\wt X}/\cE_{|\partial\wt X}^\rdD)\to\id.
\]
We first show that $H^0\big((S^1)^\ell,\GL_d(\cE_{|\partial\wt X})\big)\to H^0\big((S^1)^\ell,\GL_d(\cE_{|\partial\wt X}/\cE_{|\partial\wt X}^\rdD)\big)$ is onto. Locally, a section of the right-hand term can be lifted. Using a partition of unity, we lift it globally as a section of $\rM_d(\cE_{|\partial\wt X})$, and by the remark above, it is a section of $\GL_d(\cE_{|\partial\wt X})$.

It remains thus to show that $H^1\big((S^1)^\ell,\GL_d(\cE_{|\partial\wt X})\big)\to H^1\big((S^1)^\ell,\GL_d(\cE_{|\partial\wt X}/\cE_{|\partial\wt X}^\rdD)\big)$ is injective, and since $\cE_{|\partial\wt X}^\rdD\subset t_1\cdots t_\ell\cE_{|\partial\wt X}$, it is enough to show a similar assertion of the restriction map $H^1\big((S^1)^\ell,\GL_d(\cE_{|\partial\wt X})\big)\to H^1\big((S^1)^\ell,\GL_d(\cE_{\partial\wt X})\big)$, where $\cE_{\partial\wt X}=\cE_{|\partial\wt X}/t_1\cdots t_d\cE_{|\partial\wt X}$ is the sheaf of $C^\infty$ functions on $\partial\wt X$.

Using the interpretation of an element of $H^1$ as giving an isomorphism class of vector bundle, we are reduced to showing that, given a $C^\infty$ vector bundle in the neighbourhood of $(S^1)^\ell$ whose restriction to $(S^1)^\ell$ is trivializable, it is trivializable in some (possibly smaller) neighbourhood of $(S^1)^\ell$. For that purpose, it is enough to prove that any global section of the restriction can be lifted to a global section of the original bundle in some neighbourhood of $(S^1)^\ell$, because a lift of a basis of global sections will remain a basis of sections in some neighbourhood of $(S^1)^\ell$. Now, such a lifting property for a global section can be done locally on $(S^1)^\ell$ and glued with a partition of unity.
\end{proof}

Let $\alpha$ be a class in $H^1\big((S^1)^\ell,\GL_d^\rdD(\cA_{\partial\wt X})\big)$ represented by a cocycle $(\alpha_{ij})$ on some open cover $\cU=(U_i)$ of $(S^1)^\ell$. According to the previous lemma, $H^1\big(\cU,\GL_{d}^\rdD(\cE_{|\partial\wt X})\big) =\id$ for any open cover $\cU$ of $(S^1)^\ell$ (\cf \cite[Prop\ptbl II.1.2.1]{B-V89}), and therefore $\alpha_{ij}=\beta_{i}^{-1}\beta_j$, where $\beta_i$ is a section over $U_i$ of $\GL_{d}^\rdD(\cE_{|\partial\wt X})$.

The operator $\ov\partial$ is well-defined on $\cE^\rdD_{|\partial\wt X}$. We set
\[
\gamma_i=\ov\partial \beta_i\cdot\beta_{i}^{-1}.
\]
Then $\gamma_i=\gamma_j$ on $U_i\cap U_j$ and the $\gamma_i$ glue together as a matrix $\gamma$ of $1$-forms with entries in $\cE^\rdD_{X|D}=\varpi_*\cE^\rdD_{|\partial\wt X}$, and of type $(0,1)$. Moreover, the $\gamma_i$ (hence $\gamma$) satisfy
\[
\ov\partial \gamma_i+\gamma_i\wedge\gamma_i=0
\]
because this equality is already satisfied away from $\partial\wt X$. For $\gamma$, this equality is read on $X$.

\begin{lemme}
There exists a neighbourhood of $0\in X$ on which the equation $\ov\partial\varphi=-\varphi\cdot\gamma$ has a solution $\varphi$ which is a section of $\GL_d(\cE_X)$.
\end{lemme}

\begin{proof}
This is Theorem 1 in \cite[Chap\ptbl X]{Malgrange58}.
\end{proof}

Then for each $i$ one has $\ov\partial (\varphi\beta_i)=0$, so that $\varphi\beta_i$ is a section on $U_i$ of $\GL_d(\cA_{\partial\wt X})$ and $\alpha_{ij}=(\varphi\beta_i)^{-1}\cdot(\varphi\beta_j)$, in other words, the image of $\alpha$ in $H^1\big(\cU,\GL_d(\cA_{\partial\wt X})\big)$ is the identity.
\end{proof}

\subsection{The higher dimensional Hukuhara-Turrittin theorem}\label{subsec:HTM}\index{Hukuhara-Turrittin theorem (higher dimension)}
We keep the setting of \S\ref{subsec:prelimform}.

\begin{theoreme}\label{th:HTM}
Let $\cM$ be a meromorphic connection with poles along~$D$. Assume that $\cM$ has a puntual good formal decomposition near the origin (\cf Definition \ref{def:goodformdec}\eqref{def:goodformdec3}) with set of exponential factors $\Phi\subset\cO_{X,0}(*D)/\cO_{X,0}$. Then, for any $\theta_o\in\varpi^{-1}(0)$, the decomposition \eqref{eq:goodformdec*} can be lifted as a decomposition
\[
\cA_{\wt X,\theta_o}\otimes_{\varpi^{-1}\cO_{X,0}}\varpi^{-1}\cM_0\simeq\cA_{\wt X,\theta_o}\otimes_{\varpi^{-1}\cO_{X,0}}\Big(\bigoplus_{\varphi\in\Phi}(\cE^\varphi\otimes\cR_{\varphi})_0\Big),
\]
where each $\cR_\varphi$ is a meromorphic connection with poles along~$D$, and regular singularity along~$D$.
\end{theoreme}

We have implicitly used Theorem \ref{th:gooddec} to ensure that $\wh\Phi\subset\cO_{X,0}(*D)/\cO_{X,0}$. We will also use the existence of a good lattice near the origin (Proposition \ref{prop:existgoodlattice}).

\begin{remarque}
This theorem has already been used in Lemma \ref{lem:goodgoodsmooth} when~$D$ is smooth, referring to Sibuya \cite{Sibuya62,Sibuya74} for its proof. In order to handle the case with normal crossings, an asymptotic theory similar to that developed by H\ptbl Majima \cite{Majima84} is needed. Notice also that previous approaches to this asymptotic theory can be found in \cite{G-S79}. Here, we will use the arguments given in \cite[Chap\ptbl20]{Mochizuki08}.

When $\dim X=2$, this theorem is proved in \cite{Bibi97} (\cf Th\ptbl2.1.1 in \loccit) by using Majima's arguments. However, the proof of \loccit does not seem to extend in arbitrary dimension. Here, we give an alternative proof, due to T\ptbl Mochizuki \cite{Mochizuki08}. The new idea in this proof, compared to that of \cite{Bibi97} in dimension two, is the use of the existence of a good lattice. See also \cite[Appendix]{Hien09} for a similar explanation of this proof.
\end{remarque}

\begin{corollaire}\label{cor:HTM}
Under the assumptions of Theorem \ref{th:HTM}, if moreover each $\varphi\in\nobreak\Phi$ is purely monomial (\cf Definition \ref{def:purmonom}), that is, if $\Phi\cup\{0\}$ is also good, then $\DR^\modD\cM$ is a sheaf.
\end{corollaire}

\begin{proof}
According to Theorem \ref{th:HTM}, it is enough to prove the result for $\cM=\cE^\varphi\otimes\cR_\varphi$, where $\varphi$ is purely monomial, since the statement is local on $\partial\wt X$. The proof is then similar to that indicated for Proposition \ref{prop:HkEphinul}.
\end{proof}

\begin{proof}[\proofname\ of Theorem \ref{th:HTM}]
According to Proposition \ref{prop:existgoodlattice}, there exists a good lattice $(F,\nabla)$ for $\cM$ near the origin. We will therefore prove a statement similar to that of Theorem \ref{th:HTM} where we start with a good lattice $(F,\nabla)$, that we can assume to be non-resonant, according to Corollary \ref{cor:goodDM} (this will be used in the end of the proof of Lemma \ref{lem:grnablanabla}).

The proof is by induction on $\#\Phi$. If $\#\Phi=1$, we can assume that $\Phi=\{0\}$ by twisting. Then the isomorphism of Theorem \ref{th:HTM} already exists for $\cO$-coefficients: indeed, by Proposition \ref{prop:existgoodlattice}, there exists a good lattice of $\cM$ lifting a good lattice of $\cM_{\wh0}$; the latter being logarithmic, so is the former.

We now assume that $\#\Phi\geq2$, so that $\bmm(\Phi)\neq0$. We will use the notation of \S\ref{subsec:goodlattices}. More specifically, we let $K\subset L$ be the minimal subset such that $\Phi\subset\cO_{X,0}(D(K))/\cO_{X,0}$ and we set $D'=D(K)$. We consider the formal decomposition \eqref{eq:goodformdecD*} indexed by the set $C\subset\CC$ of Remark \ref{rem:verygood}, which contains at least two distinct elements, where the connection satisfies \eqref{eq:pseudoverygood*}. The inductive step is given by the following proposition.

\begin{proposition}\label{prop:HTM}
There exist good lattices $(F_c,\nabla_c)_{c\in C}$ near the origin in $X$ such that, setting $\varpi':\wt X'=\wt X(D_{i\in K})\to X$,
\begin{enumerate}
\item
$\cO_{\wh D',0}\otimes_{\cO_{X,0}}(F_c,\nabla_c)=(\wh F_c,\wh\nabla_c)$ (\cf \eqref{eq:goodformdecD*}) for each $c\in C$,
\item
for each $\theta'_o\in\varpi^{\prime-1}(0)$, the decomposition \eqref{eq:goodformdecD*} locally lifts as an isomorphism
\numstareq\bgroup
\begin{equation}\label{eq:HTM*}
\cA_{\wt X',\theta'_o}\otimes(F,\nabla)\simeq\bigoplus_{c\in C}\cA_{\wt X',\theta'_o}\otimes(F_c,\nabla_c).
\end{equation}
\egroup
\end{enumerate}
\end{proposition}

Since $\wt X'$ is dominated by $\wt X$, \eqref{eq:HTM*} lifts to $\wt X$, and the inductive assumption can be applied to get the analogue of Theorem \ref{th:HTM} for good lattices, and hence Theorem \ref{th:HTM} itself.
\end{proof}

\begin{proof}[\proofname\ of Proposition \ref{prop:HTM}]
In order to simplify notation during the proof, we will use the notation $\wt X$ instead of $\wt X'$ and $\theta_o$ instead of $\theta'_o$, since we will not use the original notation $\wt X$ during the proof. Correspondingly, we will set $D=D(K)$ (instead of~$D'$) and denote by $D''$ the remaining components $D(L\moins K)$, since they will play no essential role.

The proof will be achieved in two steps:
\begin{itemize}
\item
we first construct, for each $\theta_o$, a lifting of \eqref{eq:goodformdecD*} for some $(F_c^{\theta_o},\nabla_c^{\theta_o})$ locally defined near $\theta_o$; this uses sectorial asymptotic analysis;
\item
we then glue the various constructions to show that they come from some good lattice $(F_c,\nabla_c)$; in order to do so, one would like to consider an open cover $\cU$ of $\varpi^{-1}(0)$ such that the previous lifting exists on each open set $U_i$, and consider the change of liftings on the intersections $U_i\cap U_j$; if $\#K=1$, so that $\varpi^{-1}(0)=S^1$, one can choose an open cover with empty triple intersections, so that we get in that way a cocycle in $H^1(\varpi^{-1}(0),\GL^\rdD(\cA_{\wt X}))$ and we can use the Malgrange-Sibuya theorem to construct $(F_c,\nabla_c)$; however, if $\#K\geq2$, these changes of liftings do not clearly form a cocycle, since the cocycle condition is not trivial to check; instead, it can be proved that the Stokes \emph{filtration} at the level $\geq\bell$ (where $\bell$ is the predecessor of $\bmm(\Phi)$ as in the proof of Proposition \ref{prop:abelianwithout}) is globally defined, and is locally split; then the graded object of the Stokes filtration is a locally free $\cA_{\wt X}$-module of finite rank, globally defined on $\varpi^{-1}(0)$, which will be proved to be of the form $\cA_{\wt X}\otimes\varpi^{-1}(F_c,\nabla_c)$ for some good lattice $(F_c,\nabla_c)$.
\end{itemize}

\subsubsection*{Step one: existence of $(F_c^{\theta_o},\nabla_c^{\theta_o})$ for each $\theta_o\in\varpi^{-1}(0)$}
Let us consider the decomposition \eqref{eq:goodformdecD*} for $(F,\nabla)$. It induces a decomposition
\[
\cA_{\wh D}\otimes_{\varpi^{-1}\cO_X}(F,\nabla)=\bigoplus_{c\in C}\big(\cA_{\wh D}\otimes_{\varpi^{-1}\cO_{\wh D}}(\wh F_c,\wh\nabla_c)\big).
\]
Let us now fix $\theta_o\in\varpi^{-1}(0)$ and a basis $\wh\bme$ of $\cA_{\wh D,\theta_o}\otimes_{\varpi^{-1}\cO_X}(F,\nabla)$ compatible with this decomposition. If $\epsilong$ is any $\cO_{X,0}$-basis of $F$, it defines a basis $\wh\epsilong$ of $\cA_{\wh D,\theta_o}\otimes_{\varpi^{-1}\cO_X}(F,\nabla)$. Let $\wh P\in\GL_d(\cA_{\wh D,\theta_o})$ be the change of basis, so that $\wh\bme=\wh\epsilon\cdot\wh P$. According to \eqref{eq:Taylor}, there exists $P\in\GL_d(\cA_{\wt X,\theta_o})$ be such that $T_D(P)=\wh P$. Let us set $\bme=\epsilong\cdot P$, so that the basis $\bme$ induces the basis $\wh\bme$ on $\cA_{\wh D,\theta_o}\otimes_{\varpi^{-1}\cO_X}F$. In this basis, the matrix of $\nabla^{\theta_o}$ can be written as
\begin{equation}\label{eq:Omegathetao}
\diag(d\varphi_c\id+t^{-\bell}\Omega_c^{\theta_o})_{c\in C}+\Omega^\rd,
\end{equation}
where, for each $c\in C$, $\varphi_c$ is a fixed element of $\Phi_c$, $\Omega_c^{\theta_o}$ is a $\cA_{\wt X,\theta_o}$-lift of $\wh\Omega_c$ defined in Remark \ref{rem:pseudoverygood}, and $\Omega^\rd$ has entries in $\cA_{\wt X,\theta_o}^\rdD$, according to the exact sequence \eqref{eq:Taylor}. We will show that there is a base change with entries in $\GL^\rdD_d(\cA_{\wt X,\theta_o})$ after which the matrix of $\nabla^{\theta_o}$ is block-diagonal with blocks indexed by $C$. This will be done in four steps:
\begin{itemize}
\item
One first finds a base change in $\GL^\rdD_d(\cA_{\wt X,\theta_o})$ so that, for some $i$ such that $m_i>0$, the component $\Omega^{(i)}$ of $\Omega$ on $dt_i/t_i$ is block-diagonal.
\item
One then shows that the matrix $\Omega$ is then triangular with respect to the order $\leqthetao$.
\item
One then block-diagonalizes all the components $\Omega$ by a base change in $\GL_d(\cO(S\times\nobreak V))$, where $S$ is a small open poly-sector in the variables $t_{i\in K}$ and~$V$ is a neighbourhood of $t_{j\notin K}=0$. If $K=L$, the proof is straightforward. However, if $K\neq L$, one has take care of the residues along the components $D_i$ for $i\in L\moins K$.
\item
By considering the relative differential equation (relative to $t_{i\in K}$) satisfied by the previous base change, one shows that it belongs to $\GL^\rdD_d(\cA_{\wt X,\theta_o})$.
\end{itemize}

Let us index $C$ by $\{1,\dots,r\}$, and write the matrix of $\nabla^{\theta_o}$ in block-diagonal terms $(\Omega_{\alpha\beta})_{\alpha,\beta=1,\dots,r}$, so that $T_D(\Omega_{\alpha\beta})=0$ for $\alpha\neq\beta$, and $T_D(\Omega_{\alpha\alpha})=d\varphi_{c_\alpha}\id+t^{-\bell}\wh\Omega_{c_\alpha}$. Our final goal is to obtain a base change $\id+Q=\id+(Q_{\alpha\beta})_{\alpha,\beta=1,\dots,r}$, where $Q$ has entries in $\cA_{\wt X,\theta_o}^\rdD$, $Q_{\alpha\alpha}=\id$ and $T_D(Q_{\alpha\beta})=0$ for $\alpha\neq\beta$, which splits $\Omega$, that is, such that $dQ+\Omega (\id+Q)=(\id+Q)\Omega'$ with $\Omega'$ block-diagonal and of the form \eqref{eq:Omegathetao}. This amounts to finding $Q$ as above such that, for $\alpha\neq \beta$,
\begin{equation}\label{eq:Q12}
dQ_{\alpha\beta}=-\Omega_{\alpha\beta}+(Q_{\alpha\beta}\Omega_{\beta\beta}-\Omega_{\alpha\alpha}Q_{\alpha\beta})-\sum_{\gamma\neq\alpha,\beta}\Omega_{\alpha\gamma}Q_{\gamma\beta}+\cQ_{\alpha,\beta}(Q,\Omega),
\end{equation}
where $\cQ_{\alpha,\beta}(Q,\Omega)=Q_{\alpha,\beta}\big(\sum_{\gamma\neq\beta}\Omega_{\beta\gamma}Q_{\gamma\beta}\big)$ is a quadratic expresssion in $Q$ with coeffi\-cients having entries in $\cA_{\wt X,\theta_o}^\rdD$. However, we will not solve directly the system \eqref{eq:Q12}.

\subsubsection*{First step}
Since $\bmm-\bell>0$, there exists $i\in K$ such that $m_i-\ell_i>0$. Assume for simplicity that $i=\nobreak1$, and set $t'=(t_2,\dots,t_n)$. Then each $\varphi_{c_\alpha}$ can be written as $t^{-\bmm}(u_\alpha(t')+t_1v_\alpha(t))$ with $u_\alpha(0)=c_\alpha$ and both $u_\alpha,v_\alpha$ holomorphic. Let us set $t_1\partial_{t_1}\varphi_{c_\alpha}=t^{-\bmm}(u_\alpha^{(1)}(t')+\nobreak t_1v_\alpha^{(1)}(t))$, with $u_\alpha^{(1)}(0)=-m_1c_\alpha$. Let us also set $u_{\alpha\beta}^{(1)}(t')=u_\beta^{(1)}(t')-u_\alpha^{(1)}(t')$, so that $u_{\alpha\beta}^{(1)}(0)=m_1(c_\alpha-c_\beta)\neq0$ for $\alpha\neq\beta$. The component of \eqref{eq:Q12} on $dt_1/t_1$ reads,
\begin{equation}
\tag*{$(\ref{eq:Q12})_1$}\label{eq:Q121}
t_1\partial_{t_1}Q_{\neq}=-\Omega_{\neq}^{(1)}+t^{-\bmm}u_{\alpha\beta}^{(1)}(t')Q_{\neq}+t^{-\bmm}t_1L(t)(Q_{\neq})+\cQ_{\neq}(Q_{\neq},\Omega_{\neq}^{(1)}),
\end{equation}
where the index $\neq$ means that we only consider non-diagonal blocks, $L(t)$ is a linear operator with entries in $\cA_{\wt X,\theta_o}$ on the space of matrices like~$Q_{\neq}$, and $\Omega_{\neq}^{(1)}$ has rapid decay along $D$. The existence of a solution with rapid decay to \ref{eq:Q121} is given by \cite[Prop\ptbl20.1.1]{Mochizuki08}. We fix such a solution $Q$ with $Q_{\alpha\alpha}=\id$ for all $\alpha$.

Let us set $\bme'=\bme\cdot (\id+Q)$ and let us denote by $F^{\theta_o}=\bigoplus_{c\in C} F_c^{\theta_o}$ the corresponding decomposition of $F^{\theta_o}$. Then this decomposition is $\nabla^{\theta_o}_{t_1\partial_{t_1}}$-horizontal, and the matrix (that we still denote by) $\Omega$ of $\nabla^{\theta_o}$ in the basis $\bme'$ has the same form as above, with the supplementary property that $\Omega^{(1)}_{\alpha\beta}=0$ for $\alpha\neq\beta$.

\subsubsection*{Second step}
As in the proof of Lemma \ref{lem:decomposition}, set $\tau_j=t_j\partial_{t_j}$ for $j\in L$ and $\tau_j=\partial_{t_j}$ for $j\notin L$ and let us denote by~$\Omega^{(j)}$ the matrix of $\nabla^{\theta_o}_{\tau_j}$ in the basis $\bme'$. The integrability property of~$\nabla^{\theta_o}$ implies that $[\nabla^{\theta_o}_{\tau_1},\nabla^{\theta_o}_{\tau_j}]=0$, which reads, setting $\psi_{\alpha\beta}=\varphi_{c_\beta}-\varphi_{c_\alpha}$ for $\alpha\neq\beta$,
\begin{equation}\label{eq:omegajalphabeta}
t_1\partial_{t_1}\Omega^{(j)}_{\alpha\beta}=\Omega^{(j)}_{\alpha\beta}\Omega^{(1)}_{\beta\beta}-\Omega^{(1)}_{\alpha\alpha}\Omega^{(j)}_{\alpha\beta}=t_1\partial_{t_1}(\psi_{\alpha\beta})\cdot\Omega^{(j)}_{\alpha\beta}+t^{-\bmm}t_1L(t)(\Omega^{(j)}_{\alpha\beta}),
\end{equation}
for some linear operator $L(t)$ as above.

We will now use the order $\leqthetao$ on $t^{-\bmm}\CC$ as defined by \eqref{eq:ordern}, and we forget the factor $t^{-\bmm}$ to simplify the notation.

\begin{lemme}\label{lem:subsheavesFleq}
For each $\eta\in\CC$, the subsheaves
\[
F^{\theta_o}_{\leqthetao \eta}\defin\bigoplus_{\substack{c\in C\\0\leqthetao \eta+c}}F^{\theta_o}_c
\]
are left invariant by the connection $\nabla^{\theta_o}$.
\end{lemme}

\begin{proof}
It is enough to consider those $\eta$ such that $-\eta\in C$. It amounts then to proving that the only solution of the linear equation \eqref{eq:omegajalphabeta} is zero if \hbox{$-c_\alpha\not\lethetao -c_\beta$} or equivalently if $\reel(\psi_{\alpha\beta})\not<0$ on some neighbourhood of $\theta_o$. Such a solution $\Omega^{(j)}_{\alpha\beta}$, if it exists, is defined on a domain \hbox{$S\times V=\{|t_i|<2\rho\mid i\in K\}\times\prod_{i\in K}{}]\theta_o^{(i)}-\epsilon,\theta_o^{(i)}+\epsilon[{}\times V$} for some $\rho,\epsilon>0$, and~some neighbourhood $V$ of $0$ in the variables $t_i$, $i\not\in K$. Note that $c_\beta\not\lethetao c_\alpha$ means that $\arg(c_\beta-c_\alpha)-\sum_{i\in K}m_i\theta_o^{(i)}\in[-\pi/2,\pi/2]\bmod2\pi$, so that we can find a sub-polysector $S'$ with the same radius as $S$ such that, on $S'\times V$, we have $\arg(c_\beta-c_\alpha)-\sum_{i\in K}m_i\theta^{(i)}\in{}]-\pi/2,\pi/2[{}\bmod2\pi$, and then, up to shrinking the radius, $\reel(\psi_{\alpha\beta})>\delta>0$. Restricting to $|t_i|>\rho/2$ for $i\in K$ and $i\neq1$, we find that such an inequality holds on an open set $S_1\times U'$, where $S_1$ is an open sector with respect to $t_1$ and $U'$ is an open set in the variables $t'$. Then, on $S_1\times U'$, we can apply the estimate of \cite[Cor\ptbl20.3.7]{Mochizuki08}, showing that, if nonzero, $\Omega^{(j)}_{\alpha\beta}$ should have exponential growth when $t_1\to0$. This contradicts the rapid decay property.
\end{proof}

\subsubsection*{Third step}
The matrix $\Omega=(\Omega_{\alpha\beta})$ of the connection $\nabla^{\theta_o}$ can now be assumed to have the same form as in \eqref{eq:Omegathetao} together with the property that $\Omega_{\alpha\beta}=0$ if $c_\alpha\not\geqthetao c_\beta$. The index set $\{1,\dots,r\}$ of $C$ can be written as a disjoint union of maximal subsets~$A$ which are totally ordered with respect to $\leqthetao$. We have $\Omega_{\alpha\beta}=0$ if $\alpha$ and $\beta$ do not belong to the same subset $A$. On the other hand, the integrability of $\nabla^{\theta_o}=d+\Omega$ implies the integrability of each graded connection $\gr_\alpha\nabla^{\theta_o}=d+\Omega_{\alpha\alpha}$: the integrability of $d+\Omega$ implies $d\Omega_{\alpha\alpha}+\sum_{\beta}\Omega_{\alpha\beta}\wedge\Omega_{\beta\alpha}=0$; but only one of both $\Omega_{\alpha\beta}$ and $\Omega_{\beta\alpha}$ can be nonzero if $\alpha\neq\beta$.

\begin{lemme}\label{lem:grnablanabla}
There exists a base change $P=(P_{\alpha\beta})\in\GL_d^\rdD(\cA_{\wt X,\theta_o})$ which transforms $d+\Omega$ to $d+\gr\Omega$.
\end{lemme}

\begin{proof}
We search for $P$ which is block-diagonal with respect to the partition by subsets $A$ of $\{1,\dots,r\}$, and we will consider the $A$-block $(P_{\alpha\beta})_{\alpha,\beta\in A}$, where we fix $P_{\alpha\alpha}=\id$ and $P_{\alpha\beta}=0$ for $\alpha<\beta$ (for the sake of simplicity, we will denote by $\leq$ the order on $A$ induced by $\leqthetao$ on $C$). Our final goal will therefore be to look for $P_{\alpha\beta}$ ($\alpha>\beta$) with entries in $\cA_{\wt X,\theta_o}^\rdD$, satisfying
\bgroup\numstareq
\begin{equation}\label{eq:dPalphabeta}
dP_{\alpha\beta}=P_{\alpha\beta}\Omega_{\beta\beta}-\Omega_{\alpha\alpha}P_{\alpha\beta}-\sum_{\beta\leq\gamma<\alpha} \Omega_{\alpha\gamma}P_{\gamma\beta}.
\end{equation}
\egroup
We will search for $(P_{\alpha\beta})_{\alpha,\beta\in A}$ as a product of terms $(\kP_{\alpha\beta})$ for $k\in\NN^*$ so that each $\kP$ satisfies $\kP_{\alpha\alpha}=\id$ and $\kP_{\alpha\beta}=0$ if $\alpha<\beta$ or $\alpha>\beta$ and $\#\{\gamma\in A\mid\beta\leq\gamma<\alpha\}\neq k$. We will assume that, after the succession of base changes $^j\!P$ for $j<k$ the matrix $^k\Omega$ of~$\nabla^{\theta_o}$ is as above and moreover $^k\Omega_{\alpha\beta}=0$ if $\alpha>\beta$ and $\#\{\gamma\in A\mid\beta\leq\gamma<\alpha\}<k$. Then \eqref{eq:dPalphabeta} for $\kP$ reads, if $\#\{\gamma\in A\mid\beta\leq\gamma<\alpha\}=k$:
\begin{equation}\tag*{$(\ref{eq:dPalphabeta})_k$}\label{eq:dkPalphabeta}
d\,\kP_{\alpha\beta}=\kP_{\alpha\beta}{}^k\Omega_{\beta\beta}-{}^k\Omega_{\alpha\alpha}\kP_{\alpha\beta}-{}^k\Omega_{\alpha\beta}.
\end{equation}
We will start by showing the existence of a holomorphic solution of \ref{eq:dkPalphabeta} on a domain $S\times V$ as defined in the proof of Lemma \ref{lem:subsheavesFleq}.

We first consider the case where $K=L$. Let us consider the connection $\nabla_{\alpha,\beta}$ on the space of matrices $R$ of the size of $\kP_{\alpha\beta}$ having matrix $R\mto dR+({}^k\Omega_{\alpha\alpha}R-R{}^k\Omega_{\beta\beta})$. Due to the integrability of $d+{}^k\Omega_{\alpha\alpha}$ and $d+{}^k\Omega_{\beta\beta}$, this connection is integrable. We claim that $\nabla_{\alpha\beta}({}^k\Omega_{\alpha\beta})=0$. Indeed, the integrability of $d+{}^k\Omega$ implies $d\,{}^k\Omega_{\alpha\beta}+\sum_\gamma{}^k\Omega_{\alpha\gamma}\wedge{}^k\Omega_{\gamma\beta}=0$, and the property of~${}^k\Omega$ implies that all terms in the sum are zero except maybe those for $\gamma=\alpha$ and $\gamma=\beta$.

Since $\nabla_{\alpha\beta}$ is integrable on $S\times V$ and since $\nabla_{\alpha\beta}({}^k\Omega_{\alpha\beta})=0$, the equation $\nabla_{\alpha\beta}\kP_{\alpha\beta}=-{}^k\Omega_{\alpha\beta}$ has a holomorphic solution~$\kP_{\alpha\beta}$ in an open domain $S\times V$, as wanted.

We now consider the case where $K\subsetneq L$. We denote by $V''$ the subset of $V$ defined by $t_i=0$ for all $i\in L\moins K$. We consider the connection $\knabla^{\theta_o}$ on the space of block-triangular matrices $(\kP_{\alpha\beta})_{\alpha\beta}$ given by \ref{eq:dkPalphabeta} where we do not assume for the moment that $\kP_{\alpha\alpha}=\id$. It is written as
\[
\knabla^{\theta_o}((\kP_{\alpha\beta}))_{\alpha\beta}=d\,\kP_{\alpha\beta}+{}^k\Omega_{\alpha\alpha}\kP_{\alpha\beta}-\kP_{\alpha\beta}{}^k\Omega_{\beta\beta}\,({}+{}^k\Omega_{\alpha\beta}\kP_{\beta\beta}),
\]
where the last term only exists if $\alpha\neq\beta$. This connection is integrable, has logarithmic poles along $D_i$ for $i\in L\moins K$ (\cf the proof of Corollary \ref{cor:goodDM}) and no other pole in~$S\times V$. The residue endomorphism $R_i$ along $D_i$ is given by
\[
R_i((\kP_{\alpha\beta}))_{\alpha\beta}= {}^kR_{i,\alpha\alpha}\kP_{\alpha\beta}-\kP_{\alpha\beta}{}^kR_{i,\beta\beta}\,({}+{}^kR_{i,\alpha\beta}\kP_{\beta\beta}),
\]
where ${}^kR_{i,\alpha\beta}$ is the residue of ${}^k\Omega_{\alpha\beta}$ along $D_i$. The eigenvalues of $R_i$ are the differences between eigenvalues of ${}^kR_{i,\alpha\alpha}$ and ${}^kR_{i,\beta\beta}$ for $\alpha,\beta$ varying in $A$. Since the original $(F,\nabla)$ satisfies the non-resonance property by assumption, and since the connection with matrix ${}^k\Omega$ has been obtained by holomorphic base changes from the original~$\Omega$, it also satisfies the non-resonance condition. As a consequence, the only integral eigenvalue of $R_i$ is zero.

Since the connections $d+{}^k\Omega$ and $d+\gr{}^k\Omega$ are logarithmic and their residues satisfy the non-resonance condition, there exist local frames in which these connections have constant matrix and the residues also satisfy the non-resonance condition. Then the following holds on $S\times V''$:
\begin{itemize}
\item
The sheaf $\cR\defin\bigcap_{i\in L\moins K}\ker R_i$ is a vector bundle.
\item
It is equipped with the residual integrable connection induced by ${}^k\nabla^{\theta_o}$.
\item
The matrix $(\kP_{\alpha\alpha}=\id_{\alpha\alpha},\kP_{\alpha\beta}=0\;(\alpha\neq\beta))$ is a section of $\cR$.
\item
There exists a horizontal section of $\cR$ (with respect to the residual connection) such that $\kP_{\alpha\alpha}=\id$ for each $\alpha\in A$. In order to get this point, we consider the horizontal section of the residual connection which takes the value $(\kP_{\alpha\alpha}=\nobreak\id,\kP_{\alpha\beta}=\nobreak0\;(\alpha\neq\nobreak\beta))$ at some point of $S\times V''$. The $\kP_{\alpha\alpha}$ component of this section is a horizontal section of the residual connection $d+{}^k\Omega''_{\alpha\alpha}\kP_{\alpha\alpha}-\kP_{\alpha\alpha}{}^k\Omega''_{\alpha\alpha}$. Since $\id_{\alpha\alpha}$ is clearly horizontal, it coincides with $\kP_{\alpha\alpha}$.
\item
Any horizontal section of $\cR$ extends in a unique way as a ${}^k\nabla^{\theta_o}$-horizontal section on $S\times\nb(V'')$: this follows from the computation in a basis where the matrix of the residual connection is constant and the fact that the only integral eigenvalue of $R_i$ is zero.
\item
Arguing as above, the $\kP_{\alpha\alpha}$ component of this section is $\id_{\alpha\alpha}$.
\end{itemize}

In conclusion, we have found a global solution of \ref{eq:dkPalphabeta}, as wanted.

\subsubsection*{Fourth step: End of the proof of Lemma \ref{lem:grnablanabla}}
We now denote by $\Omega'$ the component of $\Omega$ on the $dt_i/t_i$ with $i\in K$, and we denote correspondingly by $\nabla'_{\alpha,\beta}$ the relative differential. Then, for a solution $(\kP_{\alpha\beta}$ of \ref{eq:dkPalphabeta} we have $\nabla'_{\alpha\beta}\kP_{\alpha\beta}=-{}^k\Omega'_{\alpha\beta}$. Let us now fix a domain $S\times V$ small enough so that $\reel(\psi_{\alpha\beta})<\delta<0$ in $S\times V$ if $\alpha\neq\beta$ and $c_\beta\leqthetao c_\alpha$. We can then use \cite[Lem\ptbl20.3.1]{Mochizuki08}, which shows that any holomorphic solution $\kP_{\alpha\beta}$ on $S\times V$ satisfies $\lim_{|t'|\to0}\kP_{\alpha\beta}=0$ and, together with a Cauchy argument, that the same holds for all derivatives of $\kP_{\alpha\beta}$, so that $\kP_{\alpha\beta}$ has entries in $\cA_{\wt X,\theta_o}^\rdD$.
\end{proof}

\subsubsection*{Step two: globalization of the local $(F_c^{\theta_o},\nabla_c^{\theta_o})$}
Let us first start with a uniqueness statement.

\begin{lemme}\label{lem:subsheavesFlequnique}
The subsheaves $F^{\theta_o}_{\leqthetao\eta}$ obtained through $Q$ satisfying \ref{eq:Q121} do not depend on the choice of $Q$.
\end{lemme}

\begin{proof}
Let $P=(P_{\alpha\beta})$ be a base change between two bases in which the matrices $\Omega,\Omega'$ of $\nabla^{\theta_o}$ are of the form considered in Step one, with $\Omega^{(1)}_{\alpha\beta}=\Omega^{\prime(1)}_{\alpha\beta}$ for $\alpha\neq\beta$. Then, for $\alpha\neq\beta$, $P_{\alpha\beta}$ is a solution of
\[
t_1\partial_{t_1}P_{\alpha\beta}=\Omega^{(1)}_{\alpha\alpha}P_{\alpha\beta}-P_{\alpha\beta}\Omega^{\prime(1)}_{\beta\beta}.
\]
This equation has a form similar to that of \eqref{eq:omegajalphabeta} and the same proof as for Lemma \ref{lem:subsheavesFleq} shows that $P_{\alpha\beta}=0$ if $-c_\alpha\not\leqthetao -c_\beta$, that is, $P$ preserves the subsheaves $F^{\theta_o}_{\leqthetao\eta}$.
\end{proof}

The subsheaves $(F^{\theta_o}_{\leqthetao -c},\nabla^{\theta_o})$ as constructed above are uniquely determined, and thus can be glued as subsheaves $(\wt F_{\leq -c},\wt\nabla)$. Setting as usual $\wt F_{<-c}=\sum_{-c'<-c}\wt F_{\leq -c'}$, which is also left invariant by $\wt\nabla$, we consider the quotient sheaf $(\wt F_c,\wt\nabla_c)$ (that we should denote by $\wt F_{-c}$; see Remark \ref{rem:infvarphi}\eqref{rem:infvarphi3} for the relation with the Stokes filtration). By the local computation of Lemma \ref{lem:grnablanabla}, $\wt F_c$ is locally isomorphic to~$F_c^{\theta_o}$, hence is $\cA_{\wt X}$-locally free, and the local isomorphisms induces an isomorphism $\wh\lambda:\nobreak\cA_{\wh D}\otimes\nobreak(\wt F_c,\wt\nabla_c)\simeq\cA_{\wh D}\otimes_{\varpi^{-1}\cO_{\wh D}}\varpi^{-1}(\wh F_c,\wh\nabla_c)$. It is then enough to prove that $\wt F_c\simeq\cA_{\wt X}\otimes_{\varpi^{-1}\cO_X}\varpi^{-1}F_c$ for some locally free $\cO_X$-module $F_c$. Indeed, this would imply that $F_c=\varpi_*\wt F_c$ is equipped with an integrable connection $\nabla_c\defin\varpi_*\wt\nabla_c$, and $\lambda\defin\varpi_*(\wh\lambda)$ would induce an isomorphism $\cO_{\wh D}\otimes(F_c,\nabla_c)\simeq(\wh F_c,\wh\nabla_c)$.

Let $\cU=(U_i)$ be an open cover of $\varpi^{-1}(0)$ with a trivialization of $\wt F_c$ on each~$U_i$. This defines a cocycle $(f_{ij})$ of $\GL_d(\cA_{\wt X})$ with respect to $\cU$. Then $T_D(f_{ij})$ is a cocycle of $\GL_d(\cA_{\wh D})$ with respect to $\cU$ defining the pull-back $\cA_{\wh D}\otimes_{\varpi^{-1}\cO_{\wh D}}\varpi^{-1}\wh F_c$. It follows that $T_D(f_{ij})=\wh g_i^{-1}\wh g_j$ with $\wh g_i\in\Gamma(U_i,\GL_d(\cA_{\wh D}))$. According to the exact sequence \eqref{eq:Taylor}, up to refining the covering, one can lift each $\wh g_i$ as an element $g_i\in\Gamma(U_i,\GL_d(\cA_{\partial\wt X}))$, and the cocycle $h_{ij}\defin g_if_{ij}g_j^{-1}$ satisfies $T_D(h_{ij})=1$. According to Malgrange-Sibuya's theorem \ref{th:Malgrange-Sibuya}, $(h_{ij})$ is a coboundary of $\GL_d(\cA_{\partial\wt X})$, and therefore so is $(f_{ij})$, showing that $\wt F_c$ is globally trivial in the neighbourhood of $\varpi^{-1}(0)$.
\end{proof}

\subsection{The Riemann-Hilbert correspondence}\index{Riemann-Hilbert functor (RH)}
Since the arguments of \S\ref{subsec:RHsmooth} apply exactly in the same way here, Definition \ref{def:RHsmooth} will be used below for the Riemann-Hilbert functor. Similarly to Lemma \ref{lem:goodgoodsmooth} we have:\enlargethispage{\baselineskip}%

\begin{lemme}\label{lem:goodgood}
If $\cM$ is a good meromorphic connection with poles along~$D$ (\cf Definition \ref{def:goodmeroconn}) with associated stratified $\ccI$-covering $\wt\Sigma$, then $\RH(\cM)=(\cL,\cL_\leq)$ is a good Stokes-filtered local system on $\partial\wt X$ (\cf Definition \ref{def:goodness}) with associated stratified $\ccI$-covering $\wt\Sigma$, that we denote by $(\cL,\cL_\bbullet)$.
\end{lemme}

\begin{proof}
Same proof as for Lemma \ref{lem:goodgoodsmooth}, if we use Theorem \ref{th:HTM} instead of the Hukuhara-Turrittin-Sibuya theorem.
\end{proof}

The main result of this \chaptersname is the following theorem, which is the direct generalization of Proposition \ref{prop:RHmero} to higher dimensions, under the goodness assumption.

\begin{theoreme}\label{th:RHmeronc}
Let $\wt\Sigma$ be a good stratified $\ccI$-covering with respect to the (pull-back to $\partial\wt X(D)$ of the) natural stratification of~$D$. The Riemann-Hilbert functor induces an equivalence between the category of germs of good meromorphic connections along~$D$ with stratified $\ccI$-covering (\cf Remark \ref{rem:I-stratcovmero}) contained in $\wt\Sigma$ and the category of good Stokes-filtered $\CC$-local systems on $\partial\wt X$ with stratified $\ccI$-covering (\cf Definition \ref{def:goodness}) contained in $\wt\Sigma$.
\end{theoreme}

\subsubsection*{The Riemann-Hilbert correspondence: local theory}\index{Riemann-Hilbert correspondence!for germs of good meromorphic connections}
We go back to the local setting and the notation of \S\ref{subsec:notationRHgeq2}. In particular, $X=\Delta^\ell\times\Delta^{n-\ell}$ and we may shrink $X$ when necessary. We also set $D=\{t_1\cdots t_\ell=0\}$, and $\varpi:\wt X\defin\wt X(D)\to X$ denotes the real blowing-up of the components of~$D$ (\cf \S\ref{subsec:realblowup}), so that in particular $\varpi^{-1}(0)\simeq (S^1)^\ell$. Let $\Phi\subset\cO_{X,0}(*D)/\cO_{X,0}$ be a good finite set of exponential factors (\cf Definition \ref{def:localgoodness}). We will prove the theorem for germs at $0$ of good meromorphic connections and germs along $\varpi^{-1}(0)$ of good Stokes-filtered local systems.

As in the cases treated before (\S\ref{subsec:RHholone} and \S\ref{subsec:proofessentialsurj}), we will use the corresponding generalization of the Hukuhara-Turrittin theorem, which is Theorem \ref{th:HTM}. This gives the analogue of Lemma \ref{lem:imdirLnegsmooth}. It then remains to show the analogue of Lemma \ref{lem:comphom} (compatibility with $\cHom$) to conclude the proof of the full faithfulness of the Riemann-Hilbert functor. The proof is done similarly, and the goodness condition for $\Phi\cup\Phi'$ makes easy to show the property that (taking notation of the proof of Lemma \ref{lem:comphom}) $e^{\varphi-\varphi'}u_{\varphi,\varphi'}$ has moderate growth in some neighbourhood of $\theta_o\in\varpi^{-1}(0)$ if and only if $\varphi\leqthetao\varphi'$.

For the essential surjectivity, let us consider a Stokes-filtered local system $(\cL,\cL_\bbullet)$ on $\partial\wt X$ whose associated $\ccI$-covering over $\varpi^{-1}(0_\ell\times\Delta^{n-\ell})$ is trivial and contained in $\Phi\times\Delta^{n-\ell}$ (in particular, we consider the non-ramified case).

\begin{proposition}\label{prop:reconstr}
Under these assumptions, there exists a germ at $0\in X$ of good meromorphic connection $\cM$ (in the sense of Definition \ref{def:goodmeroconn}) such that, for any local section $\varphi$ of $\ccI$, $\DR(e^\varphi\cA_{\partial\wt X}^\modD\otimes\varpi^{-1}\cM)\simeq\cL_{\leq\varphi}$ in a way compatible with the filtration.
\end{proposition}

\begin{proof}[\proofname\ of Proposition \ref{prop:reconstr}]
The proof of Proposition \ref{prop:reconstr} will proceed by induction on the pairs $(\ell,\bmm(\Phi))$ (\cf Remark \ref{rem:goodness}\eqref{rem:goodness4} for the definition of $\bmm(\Phi)$) through the level structure, where $\ell$ denotes the codimension of the stratum of $D$ to which belongs the origin.

For each $\ell\geq1$, the case $\bmm=0$ corresponds (up to a twist, \cf Remark \ref{rem:goodness}\eqref{rem:goodness4}), to the case where the Stokes filtration is trivial, and in such a case the regular meromorphic connection associated with the local system $\cL$ fulfills the conditions of Proposition \ref{prop:reconstr}.

We fix $\ell$ and $\bmm=\bmm(\Phi)$ and we assume that Proposition \ref{prop:reconstr} holds for any Stokes-filtered local system $(\cL',\cL'_\bbullet)$ with associated $\ccI$-covering over $\varpi^{-1}(0_\ell\times\Delta^{n-\ell})$ contained in $\Phi'\times\Delta^{n-\ell}$, with a good $\Phi'\subset\cO_{X,0}(*D)/\cO_{X,0}$ satisfying $\bmm(\Phi')<\bmm$ (with respect to the partial order of $\NN^\ell$), and also for any pair $(\ell',\bmm')$ with $\ell'<\ell$, and we will prove it for the pair $(\ell,\bmm)$. We will therefore assume that $\bmm>0$ (in $\NN^\ell$). We also fix some element $\varphi_o\in\Phi$ and we denote by $\bell_{\varphi_o}=\bell=\in\NN^\ell$ the submaximum of $\Phi-\varphi_o$. By twisting we may assume for simplicity that $\varphi_o=0$.

Let $(\cL,\cL_\bbullet)$ be a good Stokes-filtered local system on $(S^1)^\ell$ with set of exponential factors contained in $\Phi$. According to Proposition \ref{prop:leveln} (applied with $\bell=\bell_{\varphi_o}$) and to Remark \ref{rem:goodness}\eqref{rem:goodness5}, this Stokes-filtered local system induces a Stokes-filtered local system $(\cL,\cL_{[\cbbullet]_\bell})$ of level $\geq\bell$, such that each $\big(\gr_{[\varphi]_\bell}\cL,(\gr_{[\varphi]_\bell}\cL)_\bbullet\big)$ is a Stokes-filtered local system to which we can apply the inductive assumption.

If $[\varphi]_\bell\in\Phi(\bell)=\text{image}(\Phi\to\ccP_\ell(\bell))$, it takes the form $ct^{-\bmm}\bmod(t^{-\bell}\CC[t])$ for some $c\in\CC$.

By induction, we get germs of meromorphic connections $\cM_c$ ($c\in\CC$) corresponding to $\big(\gr_{[ct^{-\bmm}]_\bell}\cL,(\gr_{[ct^{-\bmm}]_\bell}\cL)_\bbullet\big)$. Each $\cM_c$ is good (by the inductive assumption) and its set of exponential factors is contained in the set of $\varphi\in\Phi$ of the form $\varphi\equiv ct^{-\bmm}\bmod(t^{-\bell}\CC[t])$.

We will now construct $\cM$ from $\cM_\bell\defin\bigoplus_{c\in\CC}\cM_c$, by considering the Stokes-filtered local system $(\cL,\cL_{[\cbbullet]_\bell})$ of level $\geq\bell$, whose corresponding graded object is the Stokes-filtered local system $(\gr_\bell\cL,(\gr_\bell\cL)_\bbullet)$ graded at the level $\geq\bell$.

\begin{lemme}\label{lem:EndEndmod}
We have a natural isomorphism
\[
\cEnd_\cD^\modD(\cM_\bell)\simeq\cEnd(\gr_\bell\cL)_{\leq0}.
\]
\end{lemme}

\begin{proof}
Since by definition $\cH^0\DR^\modD_{\ccIet}(\cM_\bell)\simeq(\gr_\bell\cL,(\gr_\bell\cL)_\bbullet)$, the assertion is proved as in Lemma \ref{lem:RHEnd}, using the compatibility of the Riemann-Hilbert functor with $\cHom$ mentioned above.
\end{proof}

\Subsubsection*{The case where $m_i>0$ for all $i=1,\dots,\ell$}

\begin{lemme}\label{lem:EndEnd}
In this case, the isomorphism of Lemma \ref{lem:EndEndmod} induces an inclusion $\cEnd(\gr_\bell\cL)_{<_{[\cbbullet]_\bell}0}\subset\cEnd_\cD^\rdD(\cM_\bell)$.
\end{lemme}

\begin{proof}
Recall that the left-hand side consists of those endomorphisms $\lambda$ which satisfy $\gr_{[\varphi]_\bell}\lambda=0$ for each $\varphi$. Note that the condition $\varphi<_{[\cbbullet]_\bell}\psi$ near $\theta\in(S^1)^\ell$ implies that the $\bmm$-dominant part of $\varphi$ and $\psi$ are distinct, and so, by our assumption on $\bmm$, $\varphi-\psi$ has a pole along each component of $D$ and $\varphi<\psi$ near $\theta$, so $e^{\varphi-\psi}$ has rapid decay. This applies to the $(\varphi,\psi)$ components of $\lambda$, expressed as in Lemma \ref{lem:comphom}.
\end{proof}

Let us finish the proof of Proposition \ref{prop:reconstr} in this case. According to Remark \ref{rem:stepconstrbell}, given the Stokes-filtered local system $(\gr_\bell\cL,(\gr_\bell\cL)_\bbullet)$ graded at the level $\geq\bell$, the Stokes-filtered local system $(\cL,\cL_\bbullet)$ determines (and is determined by) a class $\gamma$ in the pointed set $H^1\big((S^1)^\ell,\cAut^{<_{[\cbbullet]_\bell}0}(\gr_\bell\cL)\big)$, hence a class $\gamma$ in $H^1\big((S^1)^\ell,\cAut_\cD^\rdD(\cM_\bell)\big)$, after Lemma \ref{lem:EndEnd}.

The germ $\cM_\bell$ is a free $\cO_{X,0}(*D)$-module with connection $\nabla_\bell$. With respect to some basis of $\cM_\bell$, the class $\gamma$ becomes a class in $H^1\big((S^1)^\ell,\GL_d^\rdD(\cA_{\partial\wt X})\big)$. By the Malgrange-Sibuya theorem \ref{th:Malgrange-Sibuya}, this class is a coboundary of $\GL_d(\cA_{\partial\wt X})$ on $(S^1)^\ell$. Let $(U_i)$ be an open cover of $(S^1)^\ell$ on which this coboundary is defined by sections $g_i\in\nobreak\Gamma\big(U_i,\GL_d(\cA_{\partial\wt X})\big)$ with $\wh g_i=\wh g\in\GL_d(\cO_{\wh D,0})$ for all $i$. On~$U_i$ we twist the connection on $\cM_\bell$ by setting $\nabla_i=g_i^{-1}\nabla_\bell g_i$. Since $g_ig_j^{-1}=\gamma_{ij}$ is $\nabla_\bell$-flat on $U_i\cap U_j$, the~$\nabla_i$ glue together to define a new connection $\nabla$ on the free $\cO_{X,0}(*D)$-module $\cM_\bell$, that we now denote by $(\cM,\nabla)$. By construction, $\cH^0\DR^\modD_{\ccIet}(\cM)$ is the Stokes-filtered local system determined by $(\gr_\bell\cL,(\gr_\bell\cL)_\bbullet)$ and the class $\gamma$, hence is isomorphic to $(\cL,\cL_\bbullet)$.

\subsubsection*{The case where $m_i=0$ for some $i=1,\dots,\ell$}
In this case, $(\cL,\cL_\bbullet)$ is partially regular along $D$ and we shall use the equivalence of Proposition \ref{prop:partialreg}. We set $L=L'\cup L''$ with $m_i=0$ if and only if $i\in L''$, and we have $\Phi\subset \cO_{X,0}(*D(L'))/\cO_{X,0}$. By Proposition \ref{prop:partialreg}, giving $(\cL,\cL_\bbullet)$ is equivalent to giving a Stokes-filtered local system $(\cL',\cL'_\bbullet)$ on $\partial\wt X'$ together with commuting automorphisms $T_k$, $k\in L''$. By induction on $\ell$, there exists a free $\cO_{X,0}(*D(L'))$-module $\cM'$ with flat connection~$\nabla'$ whose associated Stokes-filtered local system is $(\cL',\cL'_\bbullet)$. Moreover, there exist commuting endomorphisms $C_k$ of $(\cL',\cL'_\bbullet)$ such that $\exp(-2\pi iC_k)=T_k$ (represent the local system $\cL'$ as a vector space $V'$ equipped with automorphisms $T'_j$, $j\in L'$; then~$T_k$ are automorphisms of $V'$ which commute with the $T'_j$, and any choice $C_k$ such that $-2\pi iC_k$ is a logarithm of $T_k$ also commutes with $T'_j$ and defines an endomorphism of~$\cL'$; similarly, this endomorphism is filtered with respect to the Stokes filtration~$\cL'_\bbullet$). By~the full faithfulness of the Riemann-Hilbert correspondence, the commuting endomorphisms~$C_k$ of $(\cL',\cL'_\bbullet)$ define commuting endomorphisms $C_k$ of $(\cM',\nabla')$. The free $\cO_{X,0}(*D)$-module $\cM\defin\cM'(*D'')$ can be equipped with the flat connection $\nabla\defin\nabla'+\sum_{k\in L''}C_k\,dt_k/t_k$. Then one checks that the Stokes-filtered local system associated to $(\cM,\nabla)$ on $\partial\wt X$ is isomorphic to $(\cL,\cL_\bbullet)$.
\end{proof}

\begin{proof}[\proofname\ of Theorem \ref{th:RHmeronc} for germs]
It now remains to treat the local reconstruction (Proposition \ref{prop:reconstr}) in the ramified case. The data of a possibly ramified $(\cL,\cL_\bbullet)$ is equivalent to that of a non-ramified one on a covering $(S^1)^\ell_{\bmd}$ which is stable with respect to the Galois action of the covering. The same property holds for germs of meromorphic connections. By the full faithfulness of the Riemann-Hilbert functor, the Galois action on a Stokes-filtered local system is lifted in a unique way as a Galois action on the reconstructed connection after ramification, giving rise to a meromorphic connection before ramification, whose associated Stokes-filtered local system is isomorphic to $(\cL,\cL_\bbullet)$.
\end{proof}

\begin{proof}[\proofname\ of Theorem \ref{th:RHmeronc} in the global setting]
Due to the local full faithfulness of the Riemann-Hilbert functor, one gets at the same time the global full faithfulness and the global essential surjectivity by lifting in a unique way the local gluing morphisms, which remain therefore gluing morphisms (\ie the cocycle condition remains satisfied after lifting).
\end{proof}

\Subsection{Application to Hermitian duality of holonomic $\cD$-modules}
The \index{Riemann-Hilbert correspondence!for germs of good meromorphic connections}Riemann-Hilbert correspondence for good meromorphic connections, as stated in Theorem \ref{th:RHmeronc}, together with the fundamental results of K\ptbl Kedlaya and T\ptbl Mochizuki, allows one to give a complete answer to a question asked by M\ptbl Kashiwara in \cite{Kashiwara86}, namely, to prove that the Hermitian dual $C_X\cM$ of a holonomic $\cD_X$-module is still holonomic. This application has also been considered in \cite{Mochizuki10b}.

Recall the notation of \Chaptersname\ref{chap:holdist}, but now in arbitrary dimension. We now denote by $\cD_X$ the sheaf of holomorphic linear differential operators on a complex manifold~$X$ and by \index{$DB$@$\Db_X$}$\Db_X$ the sheaf of distributions on the underlying $C^\infty$ manifold, which is a left $\cD_X\otimes_\CC\cD_{\ov X}$-module. The \emphb{Hermitian dual} \index{$CX$@$C_X\cM$ (Hermitian dual)}$C_X\cM$ of $\cM$ is the $\cD_X$-module $\cHom_{\cD_{\ov X}}(\ov\cM,\Db_X)$.

\begin{theoreme}
If $\cM$ is holonomic, then so is $C_X\cM$, and $\cExt^k_{\cD_{\ov X}}(\ov\cM,\Db_X)=0$ for $k>0$.
\end{theoreme}

\begin{proof}[Sketch of the proof]
It is done in many steps, and is very similar to that in dimension one (\cf Theorem \ref{th:Hermdualone}), except for the goodness property, which is now essential:
\begin{enumerate}
\item
One first reduces (\cf \cite{Kashiwara86} see also \cite{Bibi97}) to the case where $\cM$ is a meromorphic bundle with connection along a divisor $D$, and to proving that
\begin{itemize}
\item
$C_X^\modD\cM\defin\cHom_{\cD_{\ov X}}(\cM,\Db_X(*D))$ is a meromorphic bundle with connection,
\item
$\cExt^k_{\cD_{\ov X}}(\cM,\Db_X(*D))=0$ for $k>0$.
\end{itemize}
\item
The problem is local on $X$ and one can apply the resolution of singularities in the neighbourhood of a point of $D$ to assume that $D$ has normal crossings. The problem remains local, and one can apply the result of K\ptbl Kedlaya \cite{Kedlaya10} (when $\dim X=\nobreak2$, one refers to \cite{Kedlaya09}; in the algebraic setting and $\dim X=2$, one can use \cite{Mochizuki07b}, and \cite{Mochizuki08} for $\dim X\geq3$, \cf Remark \ref{rem:existblup}) to reduce to the case where $\cM$ is a good meromorphic bundle with connection. This reduction is of course essential. The question remains local, so we can also assume that it has no ramification.
\item
As in dimension one, one reduces to proving a similar result on the real blow-up space $\wt X(D)$, by replacing $\cM$ with $\wt\cM=\cA_{\wt X}^\modD\otimes_{\varpi^{-1}\cO_X}\varpi^{-1}\cM$ and $\Db_X^\modD$ with $\Db_{\wt X}^\modD$. Now, Theorem \ref{th:HTM} asserts that $\wt\cM$ is of good Hukuhara-Turrittin type (a definition analogous to Definition \ref{def:HTtype}, supplemented of the goodness assumption). Recall that an important point here is the existence of a good lattice proved in \cite{Mochizuki10b}, \cf Remark \ref{rem:lattices}.
\item
One now proves as in \cite[Prop\ptbl II.3.2.6]{Bibi97} the vanishing of the $\cExt^k$ for $k>0$, and that $C_{\wt X}^\modD(\wt\cM)$ is a locally free $\cA_{\wt X}^\modD$-module with flat connection of Hukuhara-Turrittin type.
\item
We can now repeat the arguments of Proposition \ref{prop:RHAmod} and Corollary \ref{cor:RHAmod}, by using Theorem \ref{th:RHmeronc} instead of Theorem \ref{th:RHequivgermmero}, to prove that $C_{\wt X}^\modD(\wt\cM)=\cA_{\wt X}^\modD\otimes_{\varpi^{-1}\cO_X}\varpi^{-1}\cN$ for some meromorphic bundle with connection $\cN$, and thus $\cN=C_X^\modD\cM$.\qedhere
\end{enumerate}
\end{proof}

\subsection{Comments}
The proof of Theorem \ref{th:HTM} presented here is due to T\ptbl Mochizuki \cite{Mochizuki08}. In dimension two, a proof was given in \cite{Bibi97}, relying on the work of H\ptbl Majima \cite{Majima84} (\cf also \cite{Bibi93}). The new approach of T\ptbl Mochizuki simplifies and generalizes previous proofs in many ways:
\begin{itemize}
\item
The use of good lattices brings a lot of facilities.
\item
While Majima tried to solve integrable \emph{systems} of quasi-linear differential equations of a certain kind, T\ptbl Mochizuki only uses the solution of a quasi-linear differential equation \ref{eq:Q121} and uses much more the properties of the linear differential system for which an normal form is search.
\item
Searching for a solution of an integrable quasi-linear system was motivated by the idea of finding a solution of \eqref{eq:Q12} in one step. Going step by step with respect to the level structure and using the filtration like in Lemma \ref{lem:subsheavesFleq} takes more into account the Stokes structure intrinsically attached to a flat meromorphic bundle.
\end{itemize}

The other results of this \chaptersname have also been obtained in \cite{Mochizuki08,Mochizuki10b}, although the language of $\ccI$-filtrations is not explicitely used.

\chapterspace{-2}
\addtocontents{toc}{\protect\enlargethispage{-2\protect\baselineskip}}
\chapter{Push-forward of Stokes-filtered~local~systems}\label{chap:pipes}

\begin{sommaire}
In this \chaptername, we consider the direct image functor for Stokes-filtered local systems. We experiment the compatibility of the Riemann-Hilbert correspondence with direct image on a simple but not trivial example. Compared with the computations in \Chaptersname\ref{chap:Laplace}, we go one step further.
\end{sommaire}

\subsection{Introduction}
We have introduced in \Chaptersname \ref{chap:Ifil} the notion of push-forward of a pre-$\ccI$-fitlration. For a Stokes-filtered local system, one expects that, through the Riemann-Hilbert correspondence, it corresponds to the direct image of meromorphic connections considered as $\cD$-modules. In order to simplify the problem, we will consider pairs $(X,D)$ where $D=\bigcup_{j\in J}D_j$ is a divisor with normal crossings and smooth components, and morphisms $\pi$ between such spaces preserve the divisors so that they can be lifted as morphisms $\wt\pi$ between the corresponding real blow-up space. Two different questions arise:
\begin{itemize}
\item
Given a meromorphic connection $\cM$ with poles on a divisor $D$ in $X$, to prove that the direct image by the lifting $\wt\pi:\wt X(D_{j\in J})\to\wt X'(D'_{j'\in J'})$ of a proper morphism $\pi:(X,D)\to(X',D')$ of the complex $\DR^\modD(\cM)$ is equal to $\DR^{\rmod D'}(\pi_+\cM)$, where $\pi_+\cM$ is the direct image complex of $\cM$ as a $\cD_X(*D)$-module. One can also ask the same question for the rapid decay complex. A positive answer would then be an analogue, at the level of real blow-up spaces, of the compatibility of the irregularity complex, as defined by Mebkhout (\cf\cite{Mebkhout89, Mebkhout90, Mebkhout04b}), with respect to proper direct images.
\item
On the other hand, if we are given a good Stokes-filtered local system on $\wt X(D_{j\in J})$, we are tempted to prove that its direct image as a pre-$\cI$-filtered sheaf is also a possibly good Stokes-filtered local system.
\end{itemize}

In the case where $\pi$ is a proper modification, the first question has been answered by Proposition \ref{prop:Rpimod}.

The second question is directly linked to the first one only in the case of good Stokes-filtered local systems, through the Riemann-Hilbert correspondence of \Chaptersname \ref{chap:RHgoodnc}.

We have solved the first question in the special case of Theorem \ref{th:RHL}, by solving the second one first. We will redo a similar exercise in \S\ref{subsec:leakypipes} below by a topological argument, which uses nevertheless an estimate of the exponential factors due to C\ptbl Roucairol.

Recently, T\ptbl Mochizuki has solved the first question directly (\cf\S\ref{subsec:recentadvances}), and this allows one to solve the second one by using the Riemann-Hilbert correspondence, at least for $\kk$-Stokes-filtered local systems when $\kk$ is a subfield of $\CC$. However, some questions for Stokes-filtered local systems remain open.

\subsection{Preliminaries}\label{subsec:settingg}
Let $\pi:X\to\nobreak X'$ be a holomorphic map between complex manifolds $X$ and $X'$. We assume that $X$ and $X'$ are equipped with normal crossing divisors $D$ and $D'$ with smooth components $D_{j\in J}$ and $D'_{j'\in J'}$, and that
\begin{enumerate}
\item\label{enum:g1}
$D=\pi^{-1}(D')$,
\item\label{enum:g2}
$\pi:X\moins D\to X'\moins D'$ is smooth.
\end{enumerate}
Let $\varpi:\wt X(D_{j\in J})\to X$ (\resp $\varpi':\wt X{}'(D'_{j'\in J'})\to X'$) be the real blowing-up of the components $D_{j\in J}$ in $X$ (\resp $D'_{j'\in J'}$ in $X'$). There exists a lifting $\wt \pi:\wt X\to\wt X{}'$ of~$\pi$ (\cf \S\ref{subsec:realblowup}) such that the following diagram commutes:
\[
\xymatrix{
\wt X\ar[d]_{\wt \pi}\ar[r]^-{\varpi}&X\ar[d]^\pi\\
\wt X{}'\ar[r]^-{\varpi'}&X'
}
\]
Notice that $\partial\wt X=\wt \pi^{-1}(\partial\wt X{}')$.

\begin{remarque}[Direct images of local systems]\label{rem:modif}\index{push-forward (direct image)!of local systems}
We set $X^*=X\moins D$, $X^{\prime*}=X'\moins D'$, and we denote by $\wtj$ (\resp $\wtj{}'$) the inclusion $X^*\hto \wt X$ (\resp $X^{\prime*}\hto\wt X{}'$), and by $\wti$ (\resp $\wti{}'$) the inclusion $\partial\wt X\hto \wt X$ (\resp $\partial\wt X'\hto\wt X{}'$). We assume here that $\pi$ is proper.

Let $\cF^*$ be a local system on $X^*$. Then, since $\pi:X^*\to X^{\prime*}$ is smooth and proper, each $\bR^k\pi_*\cF^*$ ($k\geq0$) is a local system on $X^{\prime*}$ and, after Corollary \ref{cor:jncd}, $\cL\defin\wti{}^{-1}\wtj_*\cF^*$ and $\cL^{\prime k}\defin\wti{}^{\prime-1}\wtj{}'_*\bR^k\pi_*\cF^*$ ($k\geq0$) are local systems on $\partial\wt X$ and $\partial\wt X{}'$ respectively. We claim that
\bgroup\numstareq
\begin{equation}\label{eq:imdirtilde}
\forall k\geq0,\quad\bR^k\wt \pi_*\cL=\cL^{\prime k}.
\end{equation}
\egroup
Indeed, we have
\begin{align*}
\bR\wt \pi_*\cL&=\bR\wt \pi_*\wti{}^{-1}\wtj_*\cF^*\\
&=\wti{}^{\prime-1}\bR\wt \pi_*\wtj_*\cF^*\quad\text{($\wt \pi$ proper)}\\
&=\wti{}^{\prime-1}\bR\wt \pi_*\bR\wtj_*\cF^*\quad\text{(Corollary \ref{cor:jncd})}\\
&=\wti{}^{\prime-1}\bR\wtj{}'_*\bR \pi_*\cF^*.
\end{align*}
Using Corollary \ref{cor:jncd} once more, we get \eqref{eq:imdirtilde} by taking the $k$-th cohomology of both terms.

In particular, let us consider the case where, in \eqref{enum:g2} above,~$\pi$ is an isomorphism, that is, $\pi:X\to X'$ is a proper modification (as in Proposition \ref{prop:Rpimod}). We then identify $X\moins D$ with $X'\moins D'$, so that $\wtj{}'=\wt \pi\circ\wtj$, and $\cF^{\prime*}\defin \pi_*\cF^*$ with $\cF^*$. Then \eqref{eq:imdirtilde} reads
\bgroup\numstarstareq
\begin{equation}\label{eq:modif}
\bR\wt \pi_*\cL=\cL',
\end{equation}
\egroup
where $\cL'=\cL^{\prime0}=\wti{}^{\prime-1}\wtj{}'_*\cF^*$.
\end{remarque}

\begin{definitio}[Stokes-filtered local system on $(X,D)$]\label{def:StXD}\index{Stokes-filtered local@Stokes-filtered local system}
Let $(X,D)$ and $\wt X$ be as above and $\ccI_{\wt X}$ as in Definition \ref{def:Igeneral}, and let $\cF_\leq$ be a pre-$\ccI_{\wt X}$-filtration, \ie an object of $\Mod(\kk_{\ccIet,\leq})$. We say that $\cF_\leq$ is a \emph{Stokes-filtered local system on $(X,D)$} if the following holds:
\begin{enumerate}
\item
$\cF^*\defin \wtj^{-1}\cF_\leq$ is a local system of finite dimensional $\kk$-vector spaces on $X^*$,
\item
$\cF_\leq$ has no subsheaf supported on $\partial\wt X$, that is, the natural morphism $\wti{}^{-1}\cF_\leq\to \wti{}^{-1}\wtj_*\cF^*$ is injective,
\item
this inclusion is a Stokes filtration $(\cL,\cL_\bbullet)$ of $\cL\defin \wti{}^{-1}\wtj_*\cF^*$.
\end{enumerate}

We say that $\cF_\leq$ is \emph{good} if the Stokes-filtered local system $(\cL,\cL_\bbullet)$ is good (\cf Definition \ref{def:goodness}).
\end{definitio}

This definition is similar to Definition \ref{def:CSt}, but we do not consider constructible objects $\cF^*$, only local systems. We also denote by the same letters $\wti,\wtj$ the natural inclusions either in $\wt X$ or in $\ccIet$, hoping that this abuse produces no confusion. We will also denote by $(\cF,\cF_\bbullet)$ such a Stokes-filtered local system, in order to emphasize the local system $\cF\defin\wtj_*\cF^*$ on $\wt X$.

\subsection{Adjunction}\index{adjunction}
Let $(\cF',\cF'_\bbullet)$ be a Stokes-filtered local system on $(X',D')$ and let $\pi:(X,D)\to(X',D')$ be a holomorphic map as in \S\ref{subsec:settingg}. The diagram of morphisms (as in \S\ref{subsec:pullback})
\[
\ccI_{\wt X}\From{\pi^*}\wt \pi^{-1}\ccI_{\wt X'}\To{\wt \pi}\ccI_{\wt X'}
\]
allows us to consider the \index{pull-back (inverse image)!of Stokes-filtered local systems}pull-back functor $\pi^+$ as in Lemma \ref{lem:stabpullbackstrat}. It is easy to check that goodness is preserved by pull-back. Below, we will make this property more precise. We will also use the \index{push-forward (direct image)!of pre-$\ccI$-filtrations}direct image functor $\pi_+:D^\rb(\kk_{\ccIet_{\wt X},\leq})\to D^\rb(\kk_{\ccIet_{\wt X{}'},\leq})$ of Definition \ref{def:RimdirpreIfilt}.

\begin{proposition}
Assume that $\pi:(X,D)\to(X',D')$ satisfies \ref{subsec:settingg}\eqref{enum:g1} and \eqref{enum:g2} and is proper, and let $(\cF',\cF'_\bbullet)$ be a good Stokes-filtered local system on $(X',D')$. Then $\pi_+\pi^+\cF'_\leq\simeq\bR\wt \pi_*\kk_{\wt X}\otimes\nobreak\cF'_\leq$ (where the tensor product is taken in the sense of Remark \ref{rem:tensprodder}) in $D^\rb(\kk_{\ccIet,\leq})$. In particular, each $\cH^k(\pi_+\pi^+\cF'_\leq)$ is a good Stokes-filtered local system, isomorphic to $\bR^k\wt \pi_*\kk_{\wt X}\otimes\cF'_\leq$ (where the tensor product is taken in the sense of Remark \ref{rem:tensprod}). If moreover~$\pi$ is a proper modification, then $\pi_+\pi^+\cF'_\leq\simeq\cF'_\leq$.
\end{proposition}

The last assertions are a straightforward consequence of Remark \ref{rem:modif}.

\begin{proof}
Let us consider the diagram \eqref{eq:diagramfqf} with $q_\pi=\pi^*$. The point is to prove that the natural adjunction morphism $\wt \pi^{-1}\cF'_\leq\to q_\pi^{-1}(q_\pi\wt \pi^{-1}\cF')_\leq$ (notation of Definition \ref{def:pullbackpreI}, where we denote by $\wt \pi$ both maps $\wt X\to\wt X'$ and $\wt X\times_{\wt X'}\ccIet\to\ccIet$) is an isomorphism. This is clear on $X^*$, so it is enough to check this on $\partial\wt X$, and since the question is local, we can assume that $\cF'_\leq$ is graded. Let $\wt\Sigma\subset\ccIet$ be the stratified $\ccI$-covering attached to $\cF'_\leq$, let $y\in\partial\wt X$ and set $y'=\wt \pi(y)\in\partial\wt X'$. The subset $\wt\Sigma_{y'}\subset\ccIet_{y'}$ is good, that is, is identified with a good subset $\Phi_{\bmd}$ of $\cO_{X'_{\bmd},0}(*D')/\cO_{X'_{\bmd},0}$ after a local ramification of the germ $(X',0)$ around the germ $(D',0)$. According to the definition of $(q_\pi\wt \pi^{-1}\cF')_\leq$, we are thus reduced to proving that, for any two germs $\varphi_{y'},\psi_{y'}\in\wt\Sigma_{y'}$, we have $\pi^*(\varphi_{y'})\leqy \pi^*(\psi_{y'})$ only if $\varphi_{y'}\leqyp\psi_{y'}$. Since we do not assume that $\wt \pi$ is open, we cannot use the general argument of Proposition \ref{prop:f*inj}, but we can use Lemma \ref{lem:negimpliqueneg}, since $\varphi_{y'}-\psi_{y'}$ is purely monomial when considered in $\cO_{X'_{\bmd},0}(*D)/\cO_{X'_{\bmd},0}$, by the goodness assumption.

It is then enough to remark that $\varphi_{y'}\leqyp\psi_{y'}$ if and only if, for any $y'_{\bmd}\in\partial\wt X'_{\bmd}$ above~$y'$, we have $\varphi_{y'}\leq\psi_{y'}$ at $y'_{\bmd}$, when $\varphi,\psi$ are considered as elements of $\Phi_{\bmd}$.
\end{proof}

\subsection{Recent advances on push-forward and open questions}\label{subsec:recentadvances}

The following result, generalizing Proposition \ref{prop:Rpimod}, has recently been obtained by T\ptbl Mochizuki. We keep the notation of \S\ref{subsec:settingg} and we denote by $\DR^\modD\cM$ the moderate de~Rham complex as defined in \S\ref{subsec:modgrowth}.

\begin{theoreme}\label{th:Rpimod}
Let $\pi:(X,D)\to(X',D')$ be a proper morphism between complex manifolds satisfying \ref{subsec:settingg}\eqref{enum:g1} and \eqref{enum:g2}. Let $\cM$ be a meromorphic connection with poles along~$D$ at most. Let $\pi_+\cM$ the \index{push-forward (direct image)!of $\cD$-modules}direct image of $\cM$ (as a $\cD_X(*D)$-module). Then
\[
\DR^{\rmod D'}(\pi_+\cM)\simeq\bR\wt\pi_*\DR^\modD(\cM).
\]
\end{theoreme}

We have given an independent proof of this theorem in the special case of Theorem \ref{th:RHL}, and we will consider similarly another example in Theorem \ref{th:isomorphismleq0} below. On the other hand, Theorem \ref{th:Rpimod} opens the way, with the help of the Riemann-Hilbert correspondence \ref{th:RHmeronc}, to the analysis of \index{push-forward (direct image)!of good Stokes-filtered local systems}push-forward of good Stokes-filtered local systems.

Assume moreover that $\cM$ is good (\cf Definition \ref{def:goodmeroconn}) and that each $\cH^k\pi_+\cM$ is also good (the latter property is automatically satisfied if $\dim X'=1$, as in the case of the example of \S\ref{subsec:exempushforward} below). Let $(\cF,\cF_\bbullet)$ denote the Stokes-filtered local system associated with $\cM$ through the Riemann-Hilbert functor $\DR_{\ccIet}^\modD$.

\begin{corollaire}\label{cor:degenpi}
Under these assumptions, the natural morphism $\bR^k\wt\pi_*\cF_{\leq0}\to\bR^k\wt\pi_*\cF$ is injective for each $k$.
\end{corollaire}

\begin{proof}[Sketch of proof]
According to Theorem \ref{th:Rpimod}, one identifies the left-hand term with $\DR^{\rmod D'}\cH^k(\pi_+\cM)$, which has cohomology in degree zero at most, according to Corollary \ref{cor:HTM}, and this cohomology is a subsheaf of the right-hand term.
\end{proof}

Up to twisting $\cM$ by $\cE^{\pi^*\varphi'}$ for some section $\varphi'$ of $\cO_{X'}(*D')$, a similar result holds for $\cF_{\leq\pi^*\varphi'}$ over the domain where $\varphi'$ is defined (one can also choose a ramified $\varphi'$). We can regard this result as a kind of \index{Edegeneracy@$E_1$-degeneracy}$E_1$-degeneracy result for the Stokes filtration.

\begin{questions}
We keep the previous assumptions.
\begin{enumerate}
\item
Give a ``topological proof'' of Corollary \ref{cor:degenpi}.
\item
Let $\wt\Sigma$ be the stratified covering associated to the good meromorphic connection~$\cM$ or, equivalently, to its associated Stokes-filtered local system. Considering the direct image diagram \eqref{eq:pushdiag} (in the case $X'\neq X$), prove that the stratified covering associated to each $\cH^k\pi_+\cM$ is contained in $\wt\pi(q^{-1}(\wt\Sigma))$.
\item
Given a good stratified covering $\wt\Sigma$ of $\partial\wt X$, does the condition that $\wt\pi(q^{-1}(\wt\Sigma))$ is a good stratified covering of $\partial\wt X'$ imply that any Stokes-filtered local system $(\cF,\cF_\leq)$ with associated stratified covering contained in $\wt\Sigma$ has a good direct image? Similar question (probably easier) for a good meromorphic connection $\cM$.
\end{enumerate}
\end{questions}

\subsection{An example of push-forward computation}\label{subsec:exempushforward}\index{push-forward (direct image)!computation}
Let $\Delta$ be an open disc with coordinate $t$ and let $\Afu$ be the affine line with coordinate $x$. Let $S\subset\Afu\times\Delta$ be a complex curve with finitely many branches, all distinct from $\Afu\times\{0\}$. We will assume that $\Delta$ is small enough so that any irreducible component of the closure $\ov S$ of $S$ in $\PP^1\times\Delta$ cuts $\PP^1\times\{0\}$ and is smooth away from $\PP^1\times\{0\}$. We denote by $p:\PP^1\times\Delta\to \Delta$ the projection and by~$y$ the coordinate $1/x$ at $\infty$ on $\Afu$. 
\begin{figure}[htb]
\begin{center}
\setlength{\unitlength}{.7mm}
\begin{picture}(40,80)(0,0)
\put(0,0){\vector(1,0){40}}
\put(20,10){\vector(0,1){70}}
\put(0,20){\vector(1,0){40}}
\put(0,70){\vector(1,0){40}}
\put(10,15){\vector(0,-1){10}}
\qbezier(0,80)(15,79)(20,70)
\qbezier(40,80)(25,79)(20,70)
\qbezier(0,60)(20,80)(40,60)
\qbezier(0,50)(20,90)(40,50)
\qbezier(10,35)(15,36)(20,45)
\qbezier(30,35)(25,36)(20,45)
\qbezier(20,50)(25,51)(30,60)
\qbezier(20,50)(25,49)(30,40)
\put(0,2){$\Delta$}
\put(38,2){$t$}
\put(38,22){$t$}
\put(38,72){$t$}
\put(33,44){$\ov S$}
\put(22,78){$x$}
\put(14.5,65.5){$\infty$}
\put(16.5,15.5){$0$}
\put(22,10){$\PP^1$}
\put(6,9){$p$}
\put(19.1,-.5){$\bbullet$}
\put(16.5,-4){$0$}
\end{picture}
\end{center}
\caption{}\label{fig:setting}
\end{figure}

Let $f(x,t)$ be a multivalued solution on $(\AA^1\times\Delta)\moins S$ of a system of differential equations which is holonomic and has regular singularities along $S\cup(\{\infty\}\times\Delta)$. A~natural question\footnote{I cannot resist to quote \cite{Sch-Sh09}: ``\textbf{Definition.} Physics is a part of mathematics devoted to the calculation of \index{integrals}integrals of the form $\int g(x)e^{f(x)}\,dx$. Different branches of physics are distinguished by the range of the variable $x$ and by the names used for $f (x)$, $g(x)$ and for the integral [...].''} is to compute the sectorial asymptotic expansions of \index{integrals}integrals of the form
\begin{equation}\label{eq:intexp}
I(t)=\int_{\gamma_t}f(x,t)e^xdx,
\end{equation}
where $\gamma_t$ is a suitable family of cycles of the fibre $\Afu\times\{t\}$ parametrized by $t$. As noticed in \cite{H-R08}, the computation of the formal expansions of such integrals may be translated in an algebraic problem. In order to state it, we will refer to the literature for the basic notions of holonomic $\cD$-module and regularity (\cf\eg\cite{Borel87,Kashiwara03,Mebkhout87,Mebkhout04b}).

Let $M$ be a holonomic $\cD_\Delta[x]\langle\partial_x\rangle$-module with regular singularities (included along $x=\infty$) whose singular support consists of $S$ and possibly $\Afu\times\{0\}$. Away from $S\cup(\PP^1\times\{0\})$, $M$ is a holomorphic bundle with flat connection. By working in the analytic category with respect to $\PP^1$, we also regard $M$ as a holonomic $\cD_{\PP^1\times\Delta}$-module with regular singularities, and we have $\cO_{\PP^1\times\Delta}(*\infty)\otimes_{\cO_{\PP^1\times\Delta}}M=M$, where~$\infty$ stands for the divisor $\{\infty\}\times\Delta$ in $\PP^1\times\Delta$.

Let us set $\cE^x=(\cO_{\Delta}[x],d+dx)$. The $\cO_\Delta[x]\langle\partial_x\rangle$-module $\cE^x\otimes_{\cO_\Delta[x]}M$ has an irregular singularity along $x=\infty$. The direct image $p_+(\cE^x\otimes_{\cO_\Delta[x]}M)$ is a complex which satisfies $\cH^\ell p_+(\cE^x\otimes M)=0$ if $\ell\neq-1,0$. Moreover, $\cH^{-1}p_+(\cE^x\otimes M)$ is supported at the origin of $\Delta$: Indeed, if $\Delta$ is small enough, the restriction to $\Delta^*$ of $\cH^{-1}p_+(\cE^x\otimes M)$ is a vector bundle, whose fibre at $t=t_o\neq0$ can be computed as $\ker[\nabla_{\partial_x}:(\cE^x\otimes M_{t_o})\to(\cE^x\otimes M_{t_o})]$, where $M_{t_o}=M/(t-t_o)M$. Note that $M_{t_o}$ is a regular holonomic $\CC[x]\langle\partial_x\rangle$-module, and that the kernel is also $\ker[(\nabla_{\partial_x}+\id):M_{t_o}\to M_{t_o}]$. It is well-known that this kernel is $0$ for a regular holonomic $M_{t_o}$ (this can be checked directly on cyclic $\CC[x]\langle\partial_x\rangle$-modules and the general case follows by also considering modules supported on points). Therefore, the only interesting module is
\begin{equation}\label{eq:defN}
N\defin\cH^0p_+(\cE^x\otimes M).
\end{equation}

Information on the Levelt-Turrittin decomposition of $N$ has been given in \cite{Roucairol06a,Roucairol06b,Roucairol07}, leading to a good estimation of the possible form of the asymptotic expansions of the integrals \eqref{eq:intexp}. The next step aims at giving a more precise description of the behaviour of the sectorial asymptotic expansions.

\begin{probleme}\label{prob:imdir}
To compute the Stokes filtration of $N$ at the origin of $\Delta$ in terms of the analytic de~Rham complex $\DR^\an M$.
\end{probleme}

\Subsection{The topological computation of the Stokes filtration: Leaky pipes}\label{subsec:leakypipes}
We keep the notation above. Let $\varpi:\wt\Delta\to \Delta$ be the real blowing-up of the origin and set $S^1=\varpi^{-1}(0)$. On~$\wt\Delta$, we consider the sheaf $\cA_{\wt\Delta}^{\rmod0}$, and similarly, on $\PP^1\times\wt\Delta$, we consider the sheaf $\cA_{\PP^1\times\wt\Delta}^{\rmod0}$, where~$0$ now denotes the divisor $\PP^1\times\{0\}$. We denote by $\wt p:\PP^1\times\wt\Delta\to\wt\Delta$ the projection.

We will consider the moderate de~Rham complex $\DR^{\rmod0}(\cE^x\otimes M)$:
\begin{multline*}
0\to\cA_{\PP^1\times\wt\Delta}^{\rmod0}\otimes_{\varpi^{-1}\cO_{\PP^1\times\Delta}}M\To{\nabla+dx}\cA_{\PP^1\times\wt\Delta}^{\rmod0}\otimes_{\varpi^{-1}\cO_{\PP^1\times\Delta}}(\Omega^1_{\PP^1\times\Delta}\otimes M)\\
\To{\nabla+dx}\cA_{\PP^1\times\wt\Delta}^{\rmod0}\otimes_{\varpi^{-1}\cO_{\PP^1\times\Delta}}(\Omega^2_{\PP^1\times\Delta}\otimes M)\to0.
\end{multline*}
This is a complex on $\PP^1\times\wt\Delta$.

On the one hand, it is known that the complex $\DR^{\rmod0}N$ has cohomology in degree~$0$ at most (\cf Theorem \ref{th:H1nul}). On the other hand, the complex $\DR^{\rmod0}(\cE^x\otimes\nobreak M)$ only depends on the localized module $\cO_{\PP^1\times\Delta}[1/t]\otimes_{\cO_{\PP^1\times\Delta}}\nobreak M$.

\begin{lemme}\label{lem:morphismleq0}
There is a functorial morphism
\bgroup\def\theequation{\ref{lem:morphismleq0}$_{\leq0}$}\addtocounter{equation}{-1}
\begin{equation}\label{eq:morphismleq0}
\cH^0(\DR^{\rmod0}N)\to\cH^1\bR\wt p_*\DR^{\rmod0}(\cE^x\otimes M),
\end{equation}
\egroup
which is injective.
\end{lemme}

\begin{proof}
This is completely similar to Lemma \ref{lem:DRmod}.
\end{proof}

\begin{theoreme}\label{th:isomorphismleq0}
The morphism of Lemma \ref{lem:morphismleq0} is an isomorphism.
\end{theoreme}

The topological proof will be similar to that of Theorem \ref{th:RHL}, in the sense that~it will involve a better topological understanding of the right-hand side in terms of $\DR^\an M$, that is, a solution to Problem \ref{prob:imdir}, but the geometric situation is a little more complicated and uses more complex blowing-ups. This theorem gives a topological computation of the $\leq0$ part of the Stokes filtration attached to $N$. Since we know, by \cite{Roucairol07}, what are the possible exponential factors $\varphi(x)$ of $N$ at the origin, one can perturb the computation below by replacing $\cE^x$ we $\cE^{x-\varphi(x)}$ in order to compute the $\leq\varphi$ part of the Stokes filtration. We will not make precise this latter computation.

\begin{proof}
We can assume that $M=\cO_\Delta[1/t]\otimes_{\cO_\Delta}M$, as the computation of $\DR^{\rmod0}$ only uses the localized module (on $\PP^1\times\Delta$ or on $\Delta$). From the injectivity in Lemma \ref{lem:morphismleq0}, and as the theorem clearly holds away from $\mt=0$, it is enough to check that the germs at $\theta\in S^1=\partial\wt\Delta$ of both terms of \eqref{eq:morphismleq0} have the same dimension. It is then enough to prove the theorem after a ramification with respect to $t$ (coordinate of $\Delta$), so that we are reduced to assuming that, in the neighbourhood of $\PP^1\times\{0\}$, the irreducible components of the singular support $\ov S$ of $M$ are smooth and transverse to $\PP^1\times\{0\}$.

We can also localize $M$ along its singular support $\ov S$. The kernel and cokernel of the localization morphism are supported on $\ov S$, and the desired assertion is easy to check for these modules. We can therefore assume that $M$ is a meromorphic bundle along~$\ov S$ with a flat connection having regular singularities. In particular, $M$ is a locally free $\cO_{\PP^1\times\Delta}(*\ov S)$-module of finite rank (\cf\cite[Prop\ptbl I.1.2.1]{Bibi97}).

Let $e:X\to\PP^1\times\Delta$ be a sequence of \emphb{point blowing-ups} over $(\infty,0)$, with exceptional divisor $E\defin e^{-1}(0,\infty)$, such that the strict transform of~$\ov S$ intersects the pull-back of $(\PP^1\times\{0\})\cup(\{\infty\}\times\Delta)$ only at smooth points of this pull-back. We can choose for~$e$ a sequence of $n$ blowing-ups of the successive intersection points of the exceptional divisor with the strict transform of $\{\infty\}\times\Delta$. We set $D=e^{-1}(\PP^1\times\{0\})$, $D'=e^{-1}[(\PP^1\times\{0\})\cup(\{\infty\}\times\Delta)]$. This is illustrated on Figure \ref{fig:eclatement}.
\begin{figure}[htb]
\begin{center}
\setlength{\unitlength}{.7mm}
\begin{picture}(165,38)(-3,0)
\put(10,30){\line(0,-1){30}}
\put(10,30){\vector(0,-1){15}}
\put(2,10){\line(2,1){16}}
\put(2,4){\line(2,1){16}}
\qbezier(2,8)(10,4)(18,16)

\put(0,15){\line(2,1){40}}
\put(0,15){\vector(2,1){15}}
\put(40,35){\vector(-2,-1){15}}
\put(10,20){\circle*{1.5}}
\put(-5.5,15){\footnotesize $E_1$}
\put(5,15){\footnotesize $v_1$}
\put(15,20){\footnotesize $u_1$}
\put(21.7,19.6){\line(-1,2){5}}
\qbezier(24,22)(16,23)(21,30)

\put(20,35){\line(2,-1){40}}
\put(20,35){\vector(2,-1){15}}
\put(60,15){\vector(-2,1){15}}
\put(30,30){\circle*{1.5}}
\put(14.5,35){\footnotesize $E_2$}
\put(25,25){\footnotesize $v_2$}
\put(30,25){\footnotesize $u_2$}

\put(40,15){\line(2,1){40}}
\put(40,15){\vector(2,1){15}}
\put(80,35){\vector(-2,-1){15}}
\put(50,20){\circle*{1.5}}
\put(34.5,15){\footnotesize $E_3$}
\put(45,23.5){\footnotesize $v_3$}
\put(50,23.5){\footnotesize $u_3$}
\put(59,19){\line(-1,2){5}}
\put(63,20.5){\line(-1,2){5}}
\qbezier(62,20)(54.2,21.6)(57,29.8)

\put(60,35){\line(2,-1){20}}
\put(60,35){\vector(2,-1){15}}
\put(70,30){\circle*{1.5}}
\put(54.5,35){\footnotesize $E_4$}
\put(65,25){\footnotesize $v_4$}
\put(70,25){\footnotesize $u_4$}

\put(85,20){\circle*{1.5}}
\put(90,20){\circle*{1.5}}
\put(95,20){\circle*{1.5}}

\put(105,20){\line(2,1){30}}
\put(135,35){\vector(-2,-1){15}}
\put(97,15){\footnotesize $E_{n-1}$}
\put(115,19.5){\line(-1,2){5}}

\put(115,35){\line(2,-1){40}}
\put(115,35){\vector(2,-1){15}}
\put(155,15){\vector(-2,1){15}}
\put(125,30){\circle*{1.5}}
\put(109,35){\footnotesize $E_n$}
\put(120,25){\footnotesize $v_n$}
\put(125,25){\footnotesize $u_n$}
\put(130,22){\line(1,2){5}}
\put(133,20.5){\line(1,2){5}}

\put(135,20){\line(1,0){30}}
\put(135,20){\vector(1,0){16}}
\put(145,20){\circle*{1.5}}
\put(148,23){\footnotesize $u'_n$}
\put(141,23){\footnotesize $v'_n$}

\thicklines
\put(10,30){\line(0,-1){30}}
\put(0,15){\line(2,1){40}}
\put(20,35){\line(2,-1){40}}
\put(40,15){\line(2,1){40}}
\put(60,35){\line(2,-1){20}}
\put(105,20){\line(2,1){30}}
\put(115,35){\line(2,-1){40}}
\put(10,30){\line(0,-1){30}}
\put(0,15){\line(2,1){40}}
\put(20,35){\line(2,-1){40}}
\put(40,15){\line(2,1){40}}
\put(60,35){\line(2,-1){20}}
\put(105,20){\line(2,1){30}}
\put(115,35){\line(2,-1){40}}
\end{picture}
\end{center}
\caption{The divisor~$D$ is given by the thick lines. The vertical thick line is the strict transform of $\{t=0\}$, the horizontal line is the strict transform of $\{x=\infty\}$ and the other thin lines are the strict transforms of the $\ov S_i$. Each dot is the center of a chart with coordinates $(u_k,v_k)$ ($k=1,\dots,n$) with $t\circ e=u_kv_k$ and $y\circ e= u_k^{k-1}v_k^k=v_k\cdot(t\circ e)^{k-1}$. The chart centered at the last dot has coordinates $(u'_n,v'_n)$ and $t\circ e=u'_n$, $y\circ e=u_n^{\prime n}v'_n=v'_n\cdot(t\circ e)^n$.}\label{fig:eclatement}
\end{figure}

\begin{lemme}
The pull-back connection $e^+(\cE^x\otimes M)$ is good except possibly at the intersection points of the strict transform of $\ov S$ with $D'$.
\end{lemme}

\begin{proof}
Indeed, $e^+M$ has regular singularities along its polar locus, and $e^+\cE^x=\cE^{1/y\circ e}$ is purely monomial (\cf Definition \ref{def:purmonom}). At the intersection points of $\ov S$ with~$D'$, the polar locus $\ov S\cup D'$ is not assumed to be a normal crossing divisor, which explains the restriction in the lemma. Of course, blowing up these points sufficiently enough would lead to a good meromorphic connection, but we try to avoid these supplementary blowing-ups.
\end{proof}

Let $\wt X(D')$ be the real blow-up space of the irreducible components of~$D'$ (this notation is chosen to shorten the notation introduced in \S\ref{subsec:realblowup}, which should be $\wt X(0,\infty,E_{i\in\{1,\dots,n\}})$, where~$0$ (\resp $\infty$) denotes the strict transform of $\PP^1\times\{0\}$ (\resp $\{\infty\}\times\Delta$); this does not correspond to the real blow-up space of the divisor $D'$). We denote by $\cA_{\wt X(D')}^{\rmod D'}$ the corresponding sheaf of functions (\cf \S\ref{subsec:modgrowthfunct}). The morphism~$e$ lifts to a morphism $\wt e:\wt X(D')\to\PP^1\times\wt\Delta$ and we have
\[
\DR^{\rmod0}(\cE^x\otimes M)\simeq\bR\wt e_*\DR^{\rmod D'}_{\wt X(D')}[e^+(\cE^x\otimes M)].
\]
Indeed, this follows from Proposition \ref{prop:Rpimod} and from the isomorphism $e_+e^+(\cE^x\otimes M)=\cE^x\otimes M$, which is a consequence of our assumption that $M$ is localized along $\PP^1\times\{0\}$.

Let us set $\cF_{\leq0}\defin\DR^{\rmod D'}_{\wt X(D')}[e^+(\cE^x\otimes M)]_{|\partial\wt X(D')}$. The proof of the theorem reduces to proving that, for any $\theta\in S^1=\partial\wt\Delta$,
\begin{equation}
\dim\cH^0(\DR^{\rmod0}N)_\theta=\dim\HH^1\big((\wt p\circ\wt e)^{-1}(\theta),\cF_{\leq0}\big).
\end{equation}
Indeed, if we denote by $\cF_{\leq0,\theta}$ the sheaf-theoretic restriction of $\cF_{\leq0}$ to $\wt X(D')_\theta\defin(\wt p\circ\wt e)^{-1}(\theta)$, then, as $\wt p\circ\wt e$ is proper, we get
\[
[\bR^1\wt p_*\DR^{\rmod0}(\cE^x\otimes M)]_\theta\simeq\HH^1(\wt X(D')_\theta,\cF_{\leq0,\theta}).
\]

Let us describe the inverse image by $\wt p\circ\wt e:\wt X(D')\to\wt\Delta$ of $\theta\in\partial\wt\Delta=S^1$. At a crossing point of index $k$ ($k=1,\dots, n$), $\wt X(D')$ is the product $(S^1\times\RR_+)^2$, with coordinates $(\alpha_k,|u_k|,\beta_k,|v_k|)$, and $\arg(t\circ e)=\theta$ is written \hbox{$\alpha_k+\beta_k=\theta$}. At the crossing point with coordinates $(u'_n,v'_n)$, we have coordinates $(\alpha'_n,|u'_n|,\beta'_n,|v'_n|)$ and $\arg(t\circ e)=\theta$ is written $\alpha'_n=\theta$. More globally, $\partial\wt X(D')_\theta\defin(\wt p\circ\wt e)^{-1}(\theta)$ looks like a \emphb{leaky pipe} (\cf Figure \ref{fig:leakypipe}, which has to be seen as lying above Figure \ref{fig:eclatement}), where the punctures (small black dots on the pipe for the visible ones, small circles for the ones which are behind) correspond to the intersection with the strict transforms of the $\ov S_i$.
\begin{figure}[htb]
\begin{center}
\includegraphics[width=.95\textwidth]{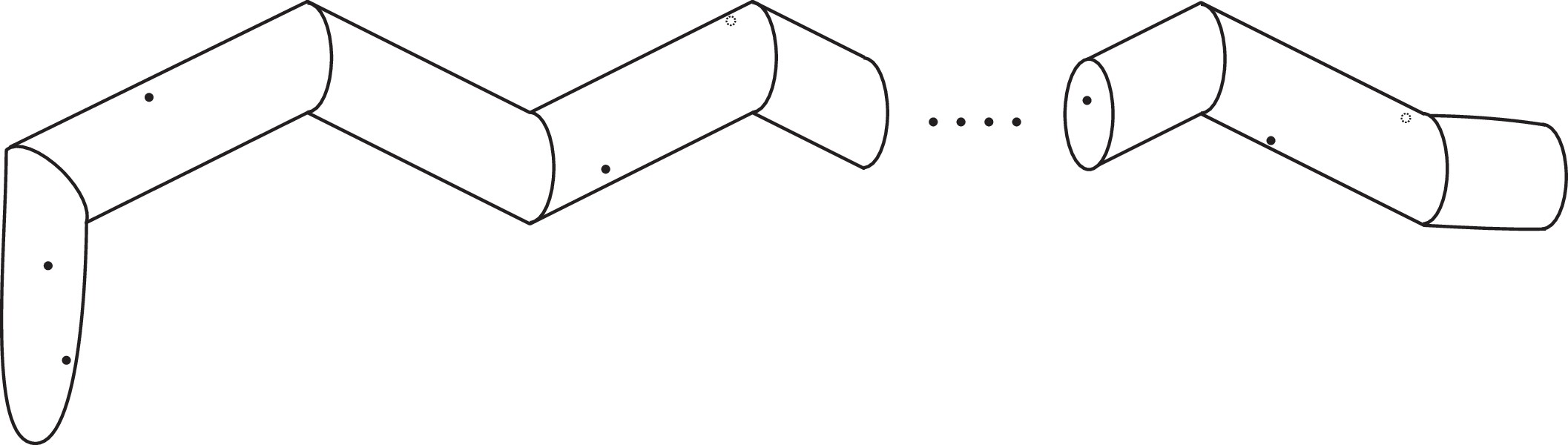}
\end{center}
\vspace*{-1cm}
\caption{}\label{fig:leakypipe}
\end{figure}

\Subsubsection*{First step: hypercohomology of $\cF_{\leq0,\theta}$}

\begin{lemme}
Away from the punctures, the complex $\cF_{\leq0,\theta}$ has cohomology in degree~$0$ at most.
\end{lemme}

\begin{proof}
This is Corollary \ref{cor:HTM}.
\end{proof}

\begin{figure}[htb]
\begin{center}
\includegraphics[width=.9\textwidth]{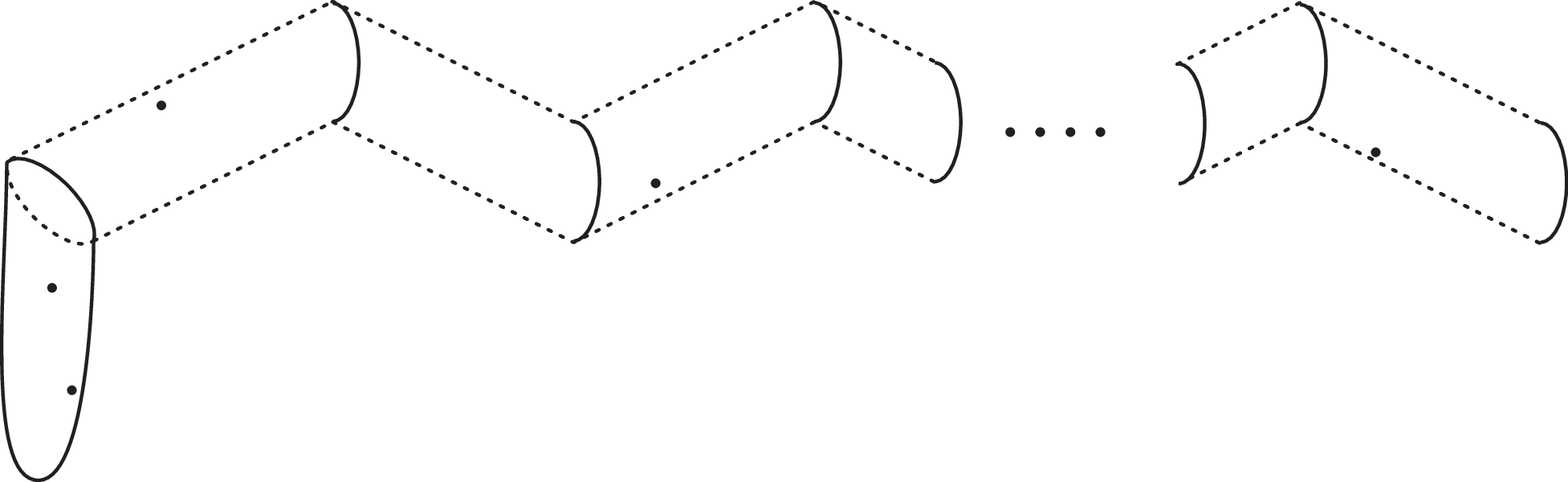}
\end{center}
\vspace*{-1cm}
\caption{}\label{fig:leakyhalfpipeb}
\end{figure}
Considering the growth of the functions $e^{1/u_k^{k-1}v_k^k}$ or $e^{1/u_n^{\prime n}v'_n}$, one obtains that, away from the punctures, the sheaf $\cF_{\leq0,\theta}$ is locally constant on a semi-open \emphb{leaky half-pipe} as in Figure \ref{fig:leakyhalfpipeb}, which is topologically like in Figure \ref{fig:leakyhalfcircleb}. Moreover, $\cF_{\leq0,\theta}$ is extended by~$0$ at the dashed boundary.
\begin{figure}[htb]
\begin{center}
\includegraphics[scale=.7,angle=3]{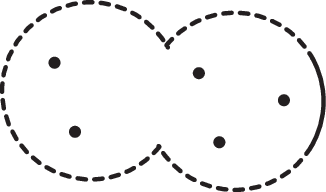}
\end{center}
\vspace*{-.5cm}
\caption{}\label{fig:leakyhalfcircleb}
\end{figure}

Let us denote by $\ov S_i$ ($i\in I$) the components of $\ov S$ which contain the point $(0,\infty)$ and by $\ov S_j$ ($j\in J$) the components which do not. Denoting as above by $y$ the coordinate on $\PP^1$ at~$\infty$ given by $y=1/x$, the local equation of each component~$\ov S_i$ near $(0,\infty)$ takes the form $\mu_i(t)y=t^{q_i}$, with~$\mu_i$ holomorphic and $\mu_i(0)\neq0$. Let us set $\varphi_i(t)=\mu_i(t)/t^{q_i}\bmod\CC\{t\}$. The punctures in the non-vertical part of Figure \ref{fig:leakyhalfpipeb} or in the right part of Figure \ref{fig:leakyhalfcircleb} correspond to the components $\ov S_i$ ($i\in I$) for which $\varphi_i\letheta0$. We denote by $I_\theta\subset I$ the corresponding subset of $I$.

On the other hand, the punctures on the vertical part of Figure \ref{fig:leakyhalfpipeb} or on the left part of Figure \ref{fig:leakyhalfcircleb} correspond to the components $\ov S_j$ ($j\in J$).

\begin{lemme}\label{lem:punctures}
Near the punctures, $\cF_{\leq0,\theta}[1]$ is a \index{perverse sheaf}perverse sheaf. It is zero if the puncture does not belong to the half-pipe of Figure \ref{fig:leakyhalfpipeb}, and the dimension of its \emphb{vanishing cycle} space at the puncture is equal to the number of curves $\ov S_i$ ($i\in I_\theta$ or $i\in J$) going through this puncture, multiplied by the rank of $M$.
\end{lemme}

\begin{proof}
One checks that, blowing up the puncture and then taking the real blow-up space of the components of the new normal crossing divisor, and then restricting to $\arg t=\theta$, amounts to change an neighbourhood of the puncture in Figure \ref{fig:leakyhalfcircleb} (say) with a disc where the the puncture has been replaced by as many punctures as distinct tangent lines of the curves $\ov S_i$ going through the original puncture, and the new sheaf $\wt\cF_{\leq0,\theta}$ remains locally constant away from the new punctures. By an easy induction, one reduces to the case where only one $\ov S_i$ goes through each puncture, and then the result is easy.
\end{proof}

\begin{corollaire}
$\HH^k(\wt X(D')_\theta,\cF_{\leq0,\theta})$ is zero if $k\neq1$ and $\dim\HH^1(\wt X(D'),\cF_{\leq0,\theta})=\rk M\cdot\#(J\cup I_\theta)$.
\end{corollaire}

\begin{proof}
According to Lemma \ref{lem:punctures}, this follows from Lemma \ref{lem:pervers-semidisque}.
\end{proof}

\subsubsection*{End of the proof}
According to \cite{Roucairol06b}, the possible exponential factors of~$N$ (defined by \eqref{eq:defN}) are the $\varphi_i$ with $i\in I$. Denoting by $\phi_i\DR M$ the local system on~$S_i$ of \index{vanishing cycle}vanishing cycles of $\DR M$ along the function $f_i(t,y)=\mu_i(t)y-t^{q_i}$, we have
\[
\dim\cH^0(\DR^{\rd0}N)_\theta=\sum_{i\,\mid\,\varphi_i\letheta0}\rk\phi_i\DR M.
\]
As $M$ is assumed to be a meromorphic bundle, we have $\rk\phi_i\DR M=\rk M$, so the previous formula reads
\[
\dim\cH^0(\DR^{\rd0}N)_\theta=\rk M\cdot\#I_\theta.
\]
We now use
\[
\dim\cH^0(\DR^{\rmod0}N)_\theta=\dim\cH^0(\DR^{\rd0}N)_\theta+\dim\psi_tN,
\]
where $\psi_t^\rmod N$ denote the \index{nearby cycle!moderate}moderate nearby cycles of $N$ (computed with the Kashiwara-Malgrange $V$\nobreakdash-filtration, \cf \cite{Malgrange91}, \cf \Chaptersname \ref{chap:irregnearby}). Arguing as in \cite{Roucairol06b}, we compute them as the direct image of $\psi_t^\rmod (\cE^x\otimes M)$ by $\PP^1\times\nobreak\{0\}\to\nobreak\{0\}$. By the same argument, this is computed as the direct image of $\psi_{t\circ e}^\rmod e^+(\cE^x\otimes M)$. Arguing as in \cite[Lemme III.4.5.10(2)]{Bibi97}, this is supported on the vertical part of~$D$. Still using the argument of \cite[Lemme III.4.5.10]{Bibi97}, this is finally the direct image of $\psi_t^\rmod (M)\otimes\cE^x$, whose dimension is that of the rank of the Fourier transform of $\psi_t^\rmod (M)$, regarded as a $\CC[x]\langle\partial_x\rangle$-module, that is, the dimension of the \index{vanishing cycle}vanishing cycles of $\psi_t^\rmod (M)$ (\cf \cite{Malgrange91} for such a formula, compare also with Proposition \ref{prop:Laplace0}).

Let us compute this dimension at a point $x_o\neq\infty$. Since $M$ is assumed to be a meromorphic connection, one checks that $\dim\phi_{x-x_o}\psi_t^\rmod (M)$ is the number of components of $\ov S$ going through~$x_o$. The sum $\sum_{x_o}\dim\phi_{x-x_o}\psi_t^\rmod (M)$ is then equal to $\#J$, by definition of~$J$.

Summarizing, we get
\[
\dim\psi_t^\rmod N=\rk M\cdot\# J.
\]
This concludes the proof of \eqref{eq:morphismleq0}.
\end{proof}

\chapterspace{-2}
\chapter{Irregular nearby cycles}\label{chap:irregnearby}

\begin{sommaire}
In this \chaptername, we review Deligne's definition of irregular nearby cycles for holonomic $\cD$-modules and recall Deligne's finiteness theorem in the algebraic case. We give a new proof of this theorem when the support of the holonomic $\cD$-module has dimension two, which holds in the complex analytic setting and which makes more precise the non-vanishing nearby cycles.
\end{sommaire}

\subsection{Introduction}
The \emph{moderate nearby cycle functor} (along a hypersurface) for holonomic $\cD$-modules has proved to be a very useful tool, as a replacement for the restriction functor to the hypersurface, since the latter may produce many cohomology $\cD$-modules, and the former is a functor on the category of holonomic $\cD$-module. The compatibility with direct images of $\cD$-modules allows one making explicit computations. This functor has been instrumental, for instance, in \cite{Roucairol06a,Roucairol06b,Roucairol07} for computing the formal decomposition of direct images, as used in \Chaptersname\ref{chap:pipes}.

Unfortunately, this functor may be useless for some holonomic $\cD$-modules and, in order to circumvent this bad behaviour, an enrichment of it has been proposed by P\ptbl Deligne in \cite{Deligne83b}, under the name of the \emph{irregular nearby cycle functor}.

The purpose of this \chaptername, which is a long introduction to \Chaptersname \ref{chap:nearby}, is to recall the definition and the behaviour of the moderate nearby cycle functor for holonomic $\cD$-modules, and that of the irregular nearby cycle functor. We will also give a new proof of the property that the functor $\psi^\Del_f$ preserves holonomy: this is the main result of \cite{Deligne83b}, which holds in arbitrary dimension and in the algebraic setting, while our proof holds in dimension two and in a local analytic setting. In order to do so, we use the reduction theorem of Kedlaya \cite{Kedlaya09} (analytic case) and Mochizuki \cite{Mochizuki07b} (algebraic case).

\Subsection{Moderate nearby cycles of holonomic $\cD$-modules}\label{subsec:moderatenearby}
\subsubsection*{Restriction and de~Rham complex of holonomic $\cD$-modules}

We refer to the book \cite{Kashiwara03} and to the introductory article \cite{Mebkhout04b} for the basic results we need \hbox{concerning} \index{holonomic $\cD$-modules} holonomic $\cD$-modules. Let $X$ be a complex manifold and let $Z$ be a closed analytic subset. We denote by $i_Z:Z\hto X$ the closed inclusion and by $j_Z:X\moins Z\hto X$ the complementary open inclusion. We will denote by $i_{Z,+}i_Z^+$ the functor of local algebraic cohomology, denoted by $\bR\Gamma_{[Z]}$ in \cite{Kashiwara03} and $\bR\text{alg}\Gamma_Z$ in \cite{Mebkhout04b}, and by $j_{Z,+}j_Z^+$ the localization functor denoted by $\bR\Gamma_{[X\moins Z]}$ in \cite{Kashiwara03} and $\bR\cbbullet(*Z)$ in \cite{Mebkhout04b}. There is a distinguished triangle in $D^\rb(\cD_X)$ (\cf \cite[Th\ptbl3.29(2)]{Kashiwara03}):
\[
i_{Z,+}i_Z^+\cM\to \cM\to j_{Z,+}j_Z^+\cM\To{+1}
\]
By a theorem of Kashiwara \cite{Kashiwara78}, these functors preserve the bounded derived category \index{$Dbdx$@$D^\rb(\cD_X)$, $D^\rb_\rh(\cD_X)$}$D^\rb_\rh(\cD_X)$ of complexes of $\cD_X$-modules with holonomic cohomology. If~$Z$ is an hypersurface and $\cM$ is a holonomic $\cD_X$-module, then so is $j_{Z,+}j_Z^+\cM=\cO_X(*Z)\otimes_{\cO_X}\cM$, and $i_{Z,+}i_Z^+\cM$ has holonomic cohomology in degrees $0$ and $1$. We denote by~$\bD$ the duality (contravariant) functor, from $D^\rb_\mathrm{coh}(\cD_X)$ to itself. It also preserves $D^\rb_\rh(\cD_X)$. In particular, $\bD \cM$ is a holonomic $\cD_X$-module if $\cM$ is so. We also have a functorial isomorphism $\id\to\bD\circ\bD$. We then set
\[
i_{Z,+}i_Z^\dag=\bD(i_{Z,+}i_Z^+)\bD,\qquad j_{Z,\dag}j_Z^+=\bD( j_{Z,+}j_Z^+)\bD.
\]
Let $\cO_{\wh Z}$ denote the formalization of $\cO_X$ along $Z$, \ie $\cO_{\wh Z}=\varprojlim_k\cO_X/\cI_Z^k$, where $\cI_Z$ is the ideal of $Z$ in $X$.

\begin{proposition}[{\cf\cite[Cor\ptbl2.7-2]{Mebkhout04b}}]\label{prop:meb}
There is a canonical functorial isomorphism in $D^\rb(\cD_X)$:
$$
\bR\cHom_{\cD_X}(i_{Z,+}i_Z^+\cbbullet,\cO_X)\simeq\bR\cHom_{\cD_X}(\cbbullet,\cO_{\wh Z}).\eqno\qed
$$
\end{proposition}

\begin{corollaire}
There is a canonical functorial isomorphism in $D^\rb_\rh(\cD_X)$:
\[
\DR(i_{Z,+}i_Z^\dag \cM)\simeq \DR(\cO_{\wh Z}\otimes_{\cO_X}\cM).
\]
\end{corollaire}

Recall that the de~Rham functor $\DR:D^\rb_\rh(\cD_X)\to D^b(\CC_X)$ is the functor $\cM\mto\bR\cHom_{\cD_X}(\cO_X,\cM)$, and that for a single holonomic $\cD_X$-module, $\DR\cM$ can be represented as the de~Rham complex $0\to\cM\To{\nabla}\Omega^1_X\otimes\cM\to\cdots$

\begin{proof}
Proposition \ref{prop:meb} applied to $\bD \cM$ gives an isomorphism
\[
\bR\cHom_{\cD_X}(\bD(i_{Z,+}i_Z^\dag \cM),\cO_X)\simeq\bR\cHom_{\cD_X}(\bD \cM,\cO_{\wh Z}).
\]
On the one hand, since $i_{Z,+}i_Z^\dag \cM$ has $\cD_X$-coherent cohomology (being holonomic), we have a canonical isomorphism from the left-hand term to $\bR\cHom_{\cD_X}(\cO_X,i_{Z,+}i_Z^\dag \cM)$, according to \cite[(3.14)]{Kashiwara03} and to the canonical isomorphism $\bD\cO_X\simeq\cO_X$.

On the other hand, if we set $\cD_{\wh Z}=\cO_{\wh Z}\otimes_{\cO_X}\cD_X$ and $\cM_{\wh Z}=\cO_{\wh Z}\otimes_{\cO_X}\cM$, we have, since $\cO_{\wh Z}$ is $\cO_X$-flat,
\[
\bR\cHom_{\cD_X}(\bD \cM,\cO_{\wh Z})\simeq\bR\cHom_{\cD_{\wh Z}}(\bD \cM_{\wh Z},\cO_{\wh Z})
\]
and applying the previous argument in $D^\rb_\mathrm{coh}(\cD_{\wh Z})$, the latter term is isomorphic to $\bR\cHom_{\cD_{\wh Z}}(\cO_{\wh Z},\cM_{\wh Z})$.
\end{proof}

\subsubsection*{Moderate nearby cycles for holonomic $\cD$-modules}
Let $f:X\to S=\CC$ be a holomorphic function and let~$\cM$ be a holonomic $\cD_X$-module. We will recall (\cf \eg \cite{M-S86b,M-M04b}) the construction of the $\cD_X$-module $\psi_f^\rmod \cM$ supported on $D\defin f^{-1}(0)$.

Let $t$ be the coordinate on $S=\CC$. For any nonzero complex number $\lambda$ and any integer $k\geq0$, denote by $\cN_{\lambda,k}$ the rank $k+1$ free $\cO_S(*0)$-module $\bigoplus_{j=0}^k\cO_S(*0)e_{\alpha,j}$, equipped with the $\cD_S$-action defined by $t\partial_t(e_{\alpha,j})=\alpha e_{\alpha,j}+e_{\alpha,j-1}$, for some choice $\alpha\in\nobreak\CC$ such that $\lambda=\exp(2\pi i\alpha)$, with the convention that $e_{\alpha,j}=0$ for $j<0$. Then $\cN_{\lambda,k}$ is equipped with a $\cD_S$-linear \hbox{nilpotent} endomorphism~$\rN$, defined by $\rN e_{\alpha,j}=e_{\alpha,j-1}$, and a natural $\cD_S$-linear inclusion $\iota_k:\cN_{\lambda,k}\hto\cN_{\lambda,k+1}$. There is a canonical isomorphism between the previous data corresponding to $\alpha$ and those corresponding to $\alpha+\ell$ for $\ell\in\ZZ$ by sending $e_{\alpha+\ell,j}$ to $t^\ell e_{\alpha,j}$.

Let us consider the pull-back meromorphic bundle with connection $f^+\cN_{\lambda,k}$ on~$X$. Given a holonomic $\cD_X$-module $\cM$, the $\cD_X$-module $\cM_{\lambda,k}\defin f^+\cN_{\lambda,k}\otimes_{\cO_X}\cM$ remains holonomic, and is equipped with a nilpotent endomorphism $\rN$. The $\cD_X$\nobreakdash-modules $\cH^j(i_{D,+}i_D^\dag\cM_{\lambda,k})$ ($j=0,1$) are also holonomic and supported on $D$. Moreover, the inductive system $[\cH^0(i_{D,+}i_D^\dag\cM_{\lambda,k})]_k$ is locally stationary, and does not depend on the choice of $\alpha$ up to a canonical isomorphism. We denote its limit by \index{nearby cycle!moderate}\index{$FZM$@$\psi_{f,\lambda}^\rmod\cM$}$\psi_{f,\lambda}^\rmod\cM$. It is equipped with a nilpotent endomorphism induced by $\rN$. Lastly, the inductive system $\cH^1(i_{D,+}i_D^\dag\cM_{\lambda,k})$ has limit zero (and the other $\cH^j$ are zero since~$D$ is a divisor).

\begin{remarque}\label{rem:psimodlambdaholom}
It follows from a theorem of Kashiwara (and Bernstein in the algebraic setting) that each $\psi_{f,\lambda}^\rmod\cM$ is $\cD_X$-holonomic.
\end{remarque}

\begin{corollaire}\label{cor:changfonction}
Let $f,h:X\to\CC$ be holomorphic functions and set $g=e^hf$. Then $(\psi_{g,\lambda}^\rmod\cM,\rN)\simeq(\psi_{f,\lambda}^\rmod\cM,\rN)$.
\end{corollaire}

\begin{proof}
It is enough to produce for each $k$ an isomorphism $g^+\cN_{\lambda,k}\isom f^+\cN_{\lambda,k}$ compatible with the inclusions $\iota_k$ and $\rN$. Note that $f^+\cN_{\lambda,k}$ is a free $\cO_X[1/f]$-module with basis $1\otimes e_{\alpha,j}$ ($j=0,\dots,k$) that we denote by $\text{``}f^\alpha(\log f)^k/k!\text{''}$, and similarly for~$g$. The base change with matrix
\[\arraycolsep4.5pt
e^{\alpha h}\cdot
\begin{pmatrix}
1&h&\dots&\dots&h^k/k!\\
0&1&h&\dots&\vdots\\
\vdots&0&\ddots&&\vdots\\
\vdots&\vdots&\ddots&\ddots&h\\
0&\dots&\dots&0&1
\end{pmatrix}
\]
transforms the previous basis for $f$ to that for $g$, and induces an isomorphism of $\cD_X$-modules.
\end{proof}

\begin{corollaire}\label{cor:drpsimod}
With the previous notation, we have in restriction to each compact set $K$ of $D$, an isomorphism (a priori depending on $K$)
\[
\DR\psi_{f,\lambda}^\rmod\cM_{|K}\simeq\varinjlim_k\DR(\cO_{\wh D}\otimes\cM_{\lambda,k})_{|K}.
\]
\end{corollaire}

\begin{proof}
The inductive limit above is meaningful, as it is defined in the category of complexes, but we will not try to give a meaning to the inductive limit of $\DR(i_{D,+}i_D^\dag\cM_{\lambda,k})$. On the other hand, the inductive limits of $\DR\big(\cH^j(i_{D,+}i_D^\dag\cM_{\lambda,k})\big)$ are well-defined and commute with taking cohomology of these complexes.

The distinguished triangle
\[
\cH^0(i_{D,+}i_D^\dag\cM_{\lambda,k})\to i_{D,+}i_D^\dag\cM_{\lambda,k}\to \cH^1(i_{D,+}i_D^\dag\cM_{\lambda,k})[-1]\To{+1}
\]
induces a distinguished triangle in $D^\rb(\CC_X)$:
\[
\DR\cH^0(i_{D,+}i_D^\dag\cM_{\lambda,k})\to\DR i_{D,+}i_D^\dag\cM_{\lambda,k}\to \DR\cH^1(i_{D,+}i_D^\dag\cM_{\lambda,k})[-1]\To{+1}
\]
and, together with the isomorphism $\DR i_{D,+}i_D^\dag\cM_{\lambda,k}\simeq\DR(\cO_{\wh D}\otimes\cM_{\lambda,k})$, we get a morphism
\[
\DR\cH^0(i_{D,+}i_D^\dag\cM_{\lambda,k})\to\DR(\cO_{\wh D}\otimes\cM_{\lambda,k}),
\]
and thus a morphism
\[
\DR\cH^0(i_{D,+}i_D^\dag\cM_{\lambda,k})\to\DR(\cO_{\wh D}\otimes\cM_{\lambda,\infty})=\varinjlim_k\DR(\cO_{\wh D}\otimes\cM_{\lambda,k}).
\]
When restricted to the compact set $K\subset D$, this morphism is an isomorphism for~$k$ big enough: indeed, because the inductive limit of $\cH^1(i_{D,+}i_D^\dag\cM_{\lambda,k})$ is zero, each cohomology sheaf of $\varinjlim_k\DR\cH^1(i_{D,+}i_D^\dag\cM_{\lambda,k})$ is zero; one uses then that $\cH^0(i_{D,+}i_D^\dag\cM_{\lambda,k})\to\cH^0(i_{D,+}i_D^\dag\cM_{\lambda,k+1})$ is an isomorphism when restricted to~$K$, when $k$ is big enough.
\end{proof}

\subsubsection*{Moderate nearby cycles and $V$-filtration}
We recall here the computation of $\psi_{f,\lambda}^\rmod\cM$ by using the $V$-filtration, according to Kashiwara and Malgrange. Let $i_f:X\hto X\times S$ denote the inclusion of the graph of $f$, and let $t$ denote a local coordinate on $S$ in the neighbourhood of the origin. The sheaf of differential operators $\cD_{X\times S}$ is equipped with a decresaing filtration $V^\cbbullet\cD_{X\times S}$ indexed by $\ZZ$, charaterized by the following properties:
\begin{itemize}
\item
$V^k\cD_{X\times S}\cdot V^\ell\cD_{X\times S}\subset V^{k+\ell}\cD_{X\times S}$, with equality if $k,\ell\geq0$;
\item
$V^0\cD_{X\times S}=\cD_{X\times S/S}\langle t\partial_t\rangle$,
\item
$V^k\cD_{X\times S}=\begin{cases}
t^kV^0\cD_{X\times S}&\text{if $k\geq0$},\\
V^{k+1}\cD_{X\times S}+\partial_tV^{k+1}\cD_{X\times S}&\text{if $k\leq-1$}.
\end{cases}$
\end{itemize}
The $\cD_{X\times S}$-module $i_{f,+}\cM$ is holonomic and it admits a unique decreasing \index{Vfiltration@$V$-filtration}filtration $V^\cbbullet(i_{f,+}\cM)$ indexed by $\ZZ$, which is good with respect to the filtration $V^\cbbullet \cD_{X\times S}$, and such that, for each $k\in\ZZ$, the endomorphism induced by $t\partial_t$ on $\gr^k_V(i_{f,+}\cM)$ has a minimal polynomial whose roots have a real part belonging to $[k,k+1)$. Moreover, given a compact subset $K$ of $f^{-1}(0)$, there exists a finite set $A\subset[0,1)\oplus i\RR\subset\CC$ such that, near each point of $K$, the roots are contained in $A+k$.

For $\alpha\in A$, let us set $\lambda=\exp(-2\pi i\alpha)$ and
\[
\psi_{t,\lambda}(i_{f,+}\cM)=\varinjlim_N\ker\big[(t\partial_t-\alpha)^N:\gr^0_V(i_{f,+}\cM)\to\gr^0_V(i_{f,+}\cM)\big].
\]
Then, for each $\lambda$, we have a canonical isomorphism of $\cD_X$-modules $\psi_{t,\lambda}(i_{f,+}\cM)\simeq\psi_{f,\lambda}^\rmod\cM$, such that the action of $\rN$ on the right-hand term correposnds to that of $t\partial_t-\alpha$ on the left-hand term. As a consequence, $\psi_{f,\lambda}^\rmod\cM$ is zero on $K$ except maybe for $\lambda$ in the finite subset $\exp(-2\pi iA)\subset\CC^*$. We will set $\psi_f^\rmod\cM=\bigoplus_{\lambda\in\CC^*}\psi_{f,\lambda}^\rmod\cM$. It is equipped with a semi-simple endomorphism induced by the multiplication by $\lambda$ on $\psi_{f,\lambda}^\rmod\cM$ and a unipotent one, induced by $\exp(-2\pi i\rN)$.

\begin{remarque}\label{rem:psimodholom}
By the finiteness result above and according to Remark \ref{rem:psimodlambdaholom}, $\psi_f^\rmod\cM$ is $\cD_X$-holonomic.
\end{remarque}

Let us notice that $\psi_{f,\lambda}^\rmod(\cM)=\psi_{f,1}^\rmod(f^+\cN_{\lambda,0}\otimes\cM)$. We conclude that, locally on~$D$, there exists a finite number of $\lambda\in\CC^*$ such that $\psi_{f,1}^\rmod(f^+\cN_{\lambda,0}\otimes\cM)\neq0$. One can check the vanishing of $\psi_{f,1}^\rmod$ in the following way.

\begin{proposition}
We have $\psi_{f,1}^\rmod\cM=0$ if and only if $j_{D,\dag}j_D^+\cM\to j_{D,+}j_D^+\cM$ is an isomorphism.
\end{proposition}

\begin{proof}[Sketch of proof]
Since $j_{D,\dag}j_D^+\cM=j_{D,\dag}j_D^+(j_{D,+}j_D^+\cM)$, we can assume that $\cM=j_{D,+}j_D^+\cM$. We will consider the \index{vanishing cycle!moderate}moderate vanishing cycle functor $\phi_{f,1}^\rmod\cM$. Then the \emph{variation morphism} $\phi_{f,1}^\rmod\cM\to\psi_{f,1}^\rmod\cM$ is an isomorphism (\cf \loccit). Hence, $\psi_{f,1}\cM=\nobreak0$ if and only if the \emph{canonical morphism} $\psi_{f,1}^\rmod\cM\to\phi_{f,1}^\rmod\cM$ is an isomorphism, because the composition $\mathrm{can}\circ\mathrm{var}$ is known to be nilpotent. On the other hand, $j_{D,\dag}j_D^+\cM\to\nobreak\cM$ has kernel and cokernel supported on $f=0$, so it is an isomorphism if and only if the natural morphism $\phi_{f,1}^\rmod(j_{D,\dag}j_D^+\cM)\to\phi_{f,1}^\rmod\cM$ is an isomorphism. The proof consists then in identifying the latter morphism to the canonical morphism $\psi_{f,1}^\rmod\cM\to\phi_{f,1}^\rmod\cM$ up to isomorphism, by using that the canonical morphism for $j_{D,\dag}j_D^+\cM$ is an isomorphism.
\end{proof}

\begin{proposition}[Behaviour by powers]\label{prop:psipuissances}
Let $m$ be a nonzero integer. Then for any~$\lambda$ one has a natural isomorphism $(\psi_{f^m,\lambda}^\rmod\cM,\rN)\simeq(\psi_{f,\lambda^m}^\rmod\cM,\rN/m)$. In particular, $\psi_{f^m}^\rmod\cM$ and $\psi_f^\rmod\cM$ have the same support.
\end{proposition}

\begin{proof}
See \cite[Prop\ptbl3.3.13]{Bibi01c} or simply compute $(t^m)^+\cN_{\lambda,k}$.
\end{proof}

\begin{proposition}[Behaviour by ramification]\label{prop:psiramif}
For $q\in\NN^*$, let $\rho_q:\CC\to\CC$ denote the ramification $t_q\mto t=t_q^q$, as well as the induced morphism $X\times\CC\to X\times\CC$. Let~$\cM$ be a holonomic $\cD_X$-module. Then $\psi_{t_q}^\rmod(\rho_q^+(i_{f,+}\cM))$ and $\psi_t^\rmod(i_{f,+}\cM)$ have the same support.
\end{proposition}

\begin{proof}
There is an explicit expression of $\psi_{t_q,\lambda}^\rmod(\rho_q^+(i_{f,+}\cM))$ in terms of various $\psi_{t,\mu}^\rmod(i_{f,+}\cM)$ (\cf \cite[Rem\ptbl2.3.3]{Bibi06b}), which immediately gives the result.
\end{proof}

\begin{proposition}[Behaviour by formalization]\label{prop:psiform}
Denote by $\cO_{\wh{x_o}}$ the formalization of~$\cO_X$ at $x_o\in X$. Then $(\cO_{\wh{x_o}}\otimes_{\cO_{X,x_o}}\psi_{t,\lambda}^\rmod(i_{f,+}\cM,\rN)\simeq\psi_{t,\lambda}^\rmod(i_{f,+}(\cO_{\wh{x_o}}\otimes_{\cO_{X,x_o}}\nobreak\cM),\rN)$.
\end{proposition}

\begin{proof}
The second term is computed via the theory of the $V$-filtration on the ring $\cD_{\wh{x_o}}$ of differential operators with coefficients in $\cO_{\wh{x_o}}$. By uniqueness of the $V$\nobreakdash-filtration, we have $V^k\big(i_{f,+}(\cO_{\wh{x_o}}\otimes_{\cO_{X,x_o}}\cM)\big)=\cO_{\wh{x_o}}\otimes_{\cO_{X,x_o}}V^k(i_{f,+}\nobreak\cM)$ for each $k$ (as the right-hand term is shown to satisfy the characteristic properties of the left-hand term). Since $\cO_{\wh{x_o}}$ is flat over $\cO_{X,x_o}$, this equality extends to the graded objects.
\end{proof}

\begin{proposition}[Behaviour by proper push-forward]\label{prop:psiimdir}\index{push-forward (direct image)!of moderate de~Rham complexes}
Let $\pi:X'\to X$ be a proper morphism and let $\cM'$ be a holonomic $\cD_{X'}$-module. Then there is a functorial isomorphism $\cH^k\pi_+\psi_{f\circ\pi}^\rmod\cM'\isom\psi_f^\rmod\cH^k\pi_+\cM'$ for each $k\in\ZZ$.
\end{proposition}

\begin{proof}
See \eg \cite[Th\ptbl4.8.1 p\ptbl226]{M-S86b}, \cite{L-M95}.
\end{proof}

We will use this proposition in the case where $\pi$ is a proper modification which is an isomorphism out of $f=0$, $\cM$ is a holonomic $\cD_X$-module on which $f$ is invertible, and $\cM'=\pi^+\cM[1/f\circ\pi]$. In such a case, $\pi_+\psi_{f\circ\pi}^\rmod\cM'=\cH^0\pi_+\psi_{f\circ\pi}^\rmod\cM'\simeq\psi_f^\rmod\cM$. More precisely, the following proposition, which is a straightforward consequence of Proposition \ref{prop:psiimdir}, will be important for us.

\begin{proposition}[Compatibility with push-forward by a proper modification]\label{prop:comppush}\nopagebreak
Let $\pi:X'\to X$ be a proper modification which is an isomorphism above $X\moins D$, and let us set $g=f\circ\pi$. Then, for each $j\in\ZZ$, for any holonomic $\cM$ such that $\cM=\cM(*D)$ and for any finite dimensional $\CC\lpb t\rpb$-vector space $\cN$ with connection,
\[
\pi_+\psi_{g,\lambda}^\rmod(g^+\cN\otimes\pi^+\cM)\simeq \psi_{f,\lambda}^\rmod(f^+\cN\otimes\cM).
\]
\end{proposition}

\Subsection{Irregular nearby cycles (after Deligne)}

By a formally irreducible $\CC\lpb t\rpb$-vector space with connection $\cN$ we mean a $\CC\lpb t\rpb$-vector space of the form $\rho_{q,+}(\cE^{\eta^{(q)}}\otimes \cL)$, where
\begin{itemize}
\item
$\rho_q=t_q\mto t=t_q^q$ is a ramification of order $q\geq1$ (here, $t_q$ is a ramified variable of order $q$ with respect to $t$; the notation is taken from \cite{Mochizuki07b}),
\item
$\eta$ is a ramified one-variable polar part, that we write as $\sum_{k\in\QQ_+^*}\eta_k/t^k$, where the sum is finite; for the smallest common denominator $q$ of the indices $k$ for which $\eta_k\neq0$, we set $\eta^{(q)}=\sum\eta_kt_q^{kq}$ (we will say that $\eta^{(q)}$ obtained in this way is \emph{$t$-irreducible}; for $\eta^{(q)}\in t_q^{-1}\CC[t_q^{-1}]$, being $t$-irreducible is equivalent to $\eta^{(q)}(\zeta t_q)\neq\eta(t_q)^{(q)}$ for any $q$th root of unity $\zeta\neq1$),
\item
$\cL$ is a rank-one $\CC\lpb t\rpb$-vector space with a connection having a regular singularity, that is, isomorphic to $\cN_{\lambda,0}$ for some $\lambda\in \CC^*$.
\end{itemize}
By the Levelt-Turrittin theorem in one variable, $\cN$ is formally irreducible if and only if $\CC\lpr t\rpr\otimes_{\CC\lpb t\rpb}\cN$ is irreducible as a $\CC\lpr t\rpr$-vector space with connection.

\begin{definitio}[Irregular nearby cycles]\label{def:irregnearby}
For holonomic $\cD_X$-modules $\cM$, the \index{nearby cycle!irregular}irregular nearby cycle functor \index{$FZZD$@$\psi^\Del_f$}$\cM\mto \psi_f^\Del \cM$ is defined as
\bgroup\numstareq
\begin{equation}\label{eq:irregnearby}
\psi_f^\Del\cM\defin\bigoplus_{\substack{\cN\text{ form.}\\ \text{irred.}}}\psi_f^\rmod(f^+\cN\otimes\cM).
\end{equation}
\egroup
\end{definitio}

Let us note that $\psi_f^\Del \cM$ only depends on the localized module $\cM(*D)$. In dimension one, the theorem of Levelt-Turrittin for $\cM(*D)$ can be restated by saying that giving the formalized module $\CC\lpr t\rpr\otimes_{\CC\lpb t\rpb}\cM(*D)$ is equivalent to giving $\psi_f^\Del\cM$, which is a \emph{finite dimensional} graded vector space (the grading indices being the formally irreducible $\cN$'s) equipped with an automorphism.

One can also use the following expression for $\psi_f^\Del\cM$ by using the notion of $t$\nobreakdash-irreducibility introduced above, \ie if $\rho_{q,+}\cE^{\eta^{(q)}}$ is irreducible. Then,
\begin{equation}\label{eq:Delmod}
\psi_f^\Del\cM=\bigoplus_{\eta^{(q)}\text{ $t$-irred.}}\bigoplus_{\lambda\in\CC^*}\psi_{f,\lambda}^\rmod(f^+\rho_{q,+}\cE^{\eta^{(q)}}\otimes\cM).
\end{equation}

\begin{theoreme}[Deligne \cite{Deligne83b}]\label{th:cyclesprochesirreg}
Assume that $X,f,S$ are algebraic. Then, if $\cM$ is $\cD_X$-holonomic, $\psi_f^\Del\cM$ is holonomic (\ie the sum \eqref{eq:Delmod} is finite).\qed
\end{theoreme}

\Subsection{Another proof of the finiteness theorem in dimension two}
We will revisit Theorem \ref{th:cyclesprochesirreg} from a different perspective, as suggested in \cite[Rem\ptbl2.1.5]{Bibi06b}. Moreover, we will work in the local analytic setting, but only when $\dim\Supp\cM=2$. Notice also that the same proof would be valid in the algebraic setting.

\begin{theoreme}\label{th:finitudelocalirreg}
Let $f:X\to S$ be a holomorphic function and let $\cM$ be a holonomic $\cD_X$-module whose support has dimension two. Then $\psi_f^\Del\cM$ is holonomic.
\end{theoreme}

\begin{proof}
We will work in the neighbourhood of a compact set $K\subset H=f^{-1}(0)$. It is enough to prove the theorem for those $\cM$ such that $\cO_X(*H)\otimes_{\cO_X}\cM=\cM$. Moreover, by a standard ``dévissage'', we can reduce to one of the following two cases:
\begin{enumerate}
\item
$\cM$ is supported on a curve $C$ and $f:C\to S$ is finite,
\item
$\dim X=2$, $\cM$ is a meromorphic bundle with a flat connection, whose poles are contained in a hypersurface (a curve) $D$ containing $H$.
\end{enumerate}

The first case easily reduces to the Levelt-Turrittin theorem, by using the normalization of the curve. We will only consider the second case. According to Proposition \ref{prop:comppush}, one can work on a suitable blow-up space~$X'$ of $X$, obtained by successively blowing up points over points in $K$. Let $e:X'\to X$ be the corresponding projective modification. Then one can replace $\cM$ with $e^+\cM$, $K$ with $e^{-1}(K)$ and~$f$ with $f\circ e$. One can therefore assume that $D$ is a divisor with normal crossing, and thus $H$ also. Moreover, according to Kedlaya's theorem (\cf \cite{Kedlaya09}), one can assume that~$\cM$ has a good formal structure at each point of~$K$. As in the proof of Theorem \ref{th:cyclesprochesirreg}, the point is to prove that, except for a finite number ramified polar parts $\eta$, we have $\psi_{f,\lambda}^\rmod(f^+\rho_{q,+}\cE^{\eta^{(q)}}\otimes\cM)=0$.

We first consider the question in the neighbourhood of each point of $K$, and we distinguish two cases:
\begin{enumeratea}
\item\label{enum:proofa}
a smooth point on $D$,
\item\label{enum:proofb}
a crossing point of $D$.
\end{enumeratea}

According to Proposition \ref{prop:psiform}, we can replace $\cM$ by its formalization (along $D$ in Case \eqref{enum:proofa} and at the given point in Case \eqref{enum:proofb}). Moreover, it is enough to prove the theorem after a fixed ramification around the components of $D$: this follows from Proposition \ref{prop:psiramif} for a ramification around components of $H$, which is enough for Cases \eqref{enum:proofa} and \eqref{enum:proofb} if $H=D$; if $H\subsetneq D$, one considers the supplementary ramification as a finite morphism and one applies Proposition \ref{prop:psiimdir}. We can therefore assume that~$\cM$ has a good formal decomposition at the given point, and, by replacing it with the formal module, that it takes the form $\cE^\omega\otimes\cR$, where $\cR$ has regular singularities along~$D$ and $\omega\in\cO(*D)/\cO$. According to \eqref{eq:Delmod}, it is enough to prove in each case the following statement.

\begin{proposition}\label{prop:finitudelocalirreg}
Given $\omega\in\cO(*D)/\cO$, there exists a finite set $Q\subset\NN^*$ and, for each $q\in Q$, a finite set of $\eta^{(q)}\in\cO_q(*D_q)/\cO_q$ such that $\psi_{f_q}(\cE^{\rho_q^*\omega-f_q^*\eta^{(q)}}\otimes\nobreak\rho_q^+\cR)_0\neq\nobreak0$.
\end{proposition}

Here, we have denoted by $\cO$ the ring $\CC\{x,y\}$, and $D$ is defined by $x=0$ (\resp $xy=\nobreak0$). The function $f$ is a monomial $x^m$ (\resp $x^ay^b$, $a\geq1$, $b\geq0$). The ramification $\rho_q$ is defined by $(x_q,y_q)\mto(x,y)=(x_q^q,y_q)$ (\resp $(x,y)=(x_q^q,y_q^q)$) and $f_q(x,y)=x_q^m$ (\resp $x_q^ay_q^b$).

\subsubsection*{Proof of Proposition \ref{prop:finitudelocalirreg} in Case \eqref{enum:proofa}}
Here, we have $H=D$. Let us start with some preliminary results. We choose local coordinates $x,y$ near a chosen point on the smooth part of $D$ such that $D=\{x=0\}$ and $f(x,y)=x^m$ for some $m\geq1$.

Let $\omega\in\CC\{x,y\}[1/x]/\CC\{x,y\}\moins\{0\}$, that we will usually write as
\begin{equation}\label{eq:omegak}
\frac{\omega_k(y)+\cdots+\omega_1(y)x^{k-1}}{x^k}\quad\text{with}\quad k\geq1,\ \omega_k(y)\not\equiv0,\ \omega_j(y)=\sum_{j'\geq0}\omega_{j,j'}y^{j'}.
\end{equation}

\begin{definitio}\label{def:supppsix}
We say that the point $y=0$ on $D$ is a \emph{singular point} for the pair $(\omega,D)$ if
\bgroup\numstareq
\begin{equation}\label{eq:supppsix*}
d\omega_k/dy(0)=0.
\end{equation}
\egroup
\end{definitio}

\begin{proposition}\label{prop:supppsix}
Let $\cR$ be a nonzero germ of meromorphic connection with regular singularities in the coordinates $x,y$, with poles along $D=\{x=0\}$ at most. Let~$f$ be such that the set $f=0$ is equal to $D$, \ie $f(x,y)=x^m$ for a suitable choice of the coordinates $x,y$ as above. A necessary condition for the germ at $y=0$ of the $\CC\{y\}\langle\partial_y\rangle$-module $\psi_f^\rmod(\cE^{\omega}\otimes\cR)$ to be nonzero is that $y=0$ is a singular point for the pair $(\omega,D)$, and $\omega_k(0)=0$. In particular, the support of $\psi_f^\rmod(\cE^{\omega}\otimes\cR)$ is discrete on $D$.
\end{proposition}

\begin{remarque}
The condition \eqref{eq:supppsix*} (together with $\omega_k(0)=0$) is not sufficient, however, to ensure the non-vanishing of $\psi_f^\rmod(\cE^{\omega}\otimes\cR)$. For example, set $\omega=y(y+\nobreak x)/x^k$ with $k\geq3$, $f(x,y)=x$ and $\cR=\cO_X[1/x]$. Then $\omega$ satisfies \eqref{eq:supppsix*}. However, one can show that $\psi_f^\rmod(\cE^{y(y+x)/x^k})=0$ in the neighbourhood of the origin (by blowing up the origin and by using \cite[Lemma 5.5(1)]{Bibi07a}, see also below, Lemma \ref{lem:basicvanishing}\eqref{lem:basicvanishing1}, together with Proposition \ref{prop:comppush}).
\end{remarque}

\begin{proof}[\proofname\ of Proposition \ref{prop:supppsix}]
By Proposition \ref{prop:psipuissances}, we are reduced to the case $f(x,y)=x$. We will prove that, if the condition \eqref{eq:supppsix*} is not fulfilled or if $\omega_k(0)\neq0$, then $\psi_x^\rmod(\cE^\omega\otimes\cR)$ is zero at $y=0$.

If $\omega_k(0)\neq0$, we apply Lemma \ref{lem:basicvanishing}\eqref{lem:basicvanishing0} below. Otherwise, if moreover $d\omega_k/dy(0)\neq\nobreak0$ then, up to changing the coordinate $y$, we have $\omega=y/x^k$ and we may apply Lemma \ref{lem:basicvanishing}\eqref{lem:basicvanishing1}.
\end{proof}

\skpt
\begin{lemme}\label{lem:basicvanishing}
\begin{enumerate}
\item\label{lem:basicvanishing0}
Let $\cR$ be a meromorphic connection with poles along $x=0$ (and possibly $y=0$) at most and regular singularities. Assume $\lambda(0,0)\neq0$ and $k>0$. Then $\psi_x^\rmod(\cE^{\lambda(x,y)/x^k}\otimes\nobreak\cR)=0$ in the neighbourhood of $y=0$.
\item\label{lem:basicvanishing1}
Let $\cR$ be a meromorphic connection with poles along $x=0$ at most and regular singularities. Assume $k>0$. Then $\psi_x^\rmod(\cE^{y/x^k}\otimes\cR)=0$ in the neighbourhood of $y=0$.
\end{enumerate}
\end{lemme}

\begin{remarque}
The assertion \ref{lem:basicvanishing}\eqref{lem:basicvanishing1} may not hold if we assume that $\cR$ has also poles along $y=0$.
\end{remarque}

\begin{proof}
In both cases, we prove that the $V$-filtration is constant by proving that a system of generators has a constant Bernstein polynomial. Since $\cR$ has regular singularities, it is a successive extension of rank-one objects of the same kind, so one can assume that $\cR$ has rank one.
\begin{enumerate}
\item
By a change of the variable $y$ and by using Corollary \ref{cor:changfonction} to keep $f=x$, we can assume that $\lambda$ is constant. Let $m$ be a local generator of $\cR$ satisfying $(x\partial_x-\alpha)m=0$ and $\partial_ym=0$ (if $\cR$ has no pole along $y=0$) or $(y\partial_y-\beta)m=0$ with $\beta\notin\NN$ (if $\cR$ has poles along $y=0$). Then one checks that $(1\otimes m)\in\cE^{\lambda/x^k}\otimes\cR$ generates $\cE^{\lambda/x^k}\otimes\cR$ as a $\cD$-module, and has constant Bernstein polynomial, as shown by the equation $(1\otimes m)=-x^k(x\partial_x-\alpha)(1\otimes m)/k\lambda$.
\item
The vanishing at any $y\neq0$ follows from \eqref{lem:basicvanishing0}. Let $m$ be  a generator of $\cR$ as above, satisfying $\partial_ym=0$. The equation $(1\otimes m)=x^k\partial_y(1\otimes m)$ shows that $(1\otimes m)$ has constant Bernstein polynomial and also that it generates $\cE^{y/x^k}\otimes\cR$ as a $\cD$-module.
\qedhere
\end{enumerate}
\end{proof}

For a meromorphic function $\omega=(x^\ell\lambda(x,y)+y^m\mu(x,y))/x^k$ we consider the condition
\begin{equation}\label{eq:1klm}
\lambda(0,0)\neq0,\quad\mu(0,0)\neq0,\quad m\geq2,\quad k\geq\ell+1\geq1.
\end{equation}
For a meromorphic function $(x^\ell\lambda(x,y)+y^m\mu(x,y))/x^ky^{k'}$, we also consider the condition
\begin{equation}\label{eq:2klm}
\lambda(0,0)\neq0,\quad\mu(0,0)\neq0,\quad (\ell,m)\in\NN^2\moins\{(0,0)\},\quad k\geq\ell+1\geq1,\quad k'>0.
\end{equation}

The following lemma is very similar to \cite[Lemma 5.5]{Bibi07a}.

\skpt
\begin{lemme}\label{lem:basicvanishingb}
\begin{enumerate}
\item\label{lem:basicvanishing2}
Let $\cR$ be a meromorphic connection with poles along $xy=0$ at most and regular singularities, let $(a,b)\in\NN^2\moins\{0\}$ and set $\omega=(x^\ell\lambda(x,y)+y^m\mu(x,y))/x^ky^{k'}$ and $f(x,y)=x^ay^b$.
\begin{enumerate}
\item\label{lem:basicvanishing2a}
If the numerator of $\omega$ is a unit near $x=y=0$ and $k,k'>0$,
\item\label{lem:basicvanishing2b}
or if $\omega$ satisfies \eqref{eq:2klm},
\end{enumerate}
then $\psi_f^\rmod(\cE^\omega\otimes\cR)=0$ in the neighbourhood of $y=0$.
\item\label{lem:basicvanishing3}
Let $\cR$ be a meromorphic connection with poles along $x=0$ (and possibly $y=0$) at most and regular singularities. Set $\omega=(x^\ell\lambda(x,y)+y^m\mu(x,y))/x^k$. If $\omega$ satisfies \eqref{eq:1klm}, then $\psi_x^\rmod(\cE^\omega\otimes\cR)=0$ in the neighbourhood of $y=0$.
\end{enumerate}
\end{lemme}

\skpt
\begin{proof}
\begin{enumerate}
\item
Away from $y=0$ we can apply \ref{lem:basicvanishing}\eqref{lem:basicvanishing0}, both for \eqref{lem:basicvanishing2a} and \eqref{lem:basicvanishing2b}.

For \eqref{lem:basicvanishing2a} at $y=0$, we apply an argument similar to that of \cite[Lemme III.4.5.10]{Bibi00} (that~$\lambda$ is a unit instead of being constant does not cause much trouble).

Let us now consider \eqref{lem:basicvanishing2b} at $y=0$. We will argue by induction on the pair $(\ell,m)\in\NN^2\moins\nobreak\{(0,0)\}$, the case $(1,0)$ or $(0,1)$ being given by \eqref{lem:basicvanishing2a}. Let $e$ denote the blowing-up at the origin. It is enough to prove the assertion after blowing up the origin, all along the exceptional divisor, for the map $f\circ e$, because we have
\[
\psi_f^\rmod(\cM)=e_+\psi^\rmod_{f\circ e}(e^+\cM)
\]
for any holonomic module localized along $f=0$. The total space of the blowing-up is covered by two charts:
\[
\text{Chart 1:\;} (u,v)\mto (x=u,y=uv),\qquad \text{Chart 2:\;} (u',v')\mto (x=u'v',y=u').
\]
We compute $\Psi\defin\psi_{f\circ e}^\rmod(\cE^{\omega\circ e}\otimes e^*\cR)$ at the origin of each chart:
\[
\omega\circ e=
\begin{cases}
(\text{Chart 1})
\begin{cases}
(\lambda\circ e+u^{m-\ell}v^\ell\mu\circ e)/u^{k-\ell+k'}v^k,&\text{if }\ell\leq m;\;\eqref{lem:basicvanishing2a}\implique \Psi=0,\\
(u^{\ell-m}\lambda\circ e+v^m\mu\circ e)/u^{k-m+k'}v^k,&\hspace*{-2.5mm}\text{if }\ell>m;\;\text{induction}\implique \Psi\!=\!0.
\end{cases}
\\[12pt]
(\text{Chart 2})
\begin{cases}
(v^{\prime\ell}\lambda\circ e+u^{\prime m-\ell}\mu\circ e)/u^{\prime k-\ell+k'}v^{\prime k},&\hspace*{-4mm}\text{if }\ell\leq m;\;\text{induction}\implique \Psi\!=\!0,\\
(\mu\circ e+u^{\prime\ell-m}v^{\prime\ell}\lambda\circ e)/u^{\prime k-m+k'}v^{\prime k},&\text{if }\ell>m;\;\eqref{lem:basicvanishing2a}\implique\Psi=0.
\end{cases}
\end{cases}
\]
It remains to show the vanishing at a general point of the exceptional divisor. Let us work in Chart 1 for instance. The exceptional divisor $u=0$ has coordinate $v$ and we compute $\Psi$ in the neighbourhood of $v_o\neq0$, so that $v$ is a local unit. Using the formulas above for Chart 1, we obtain the vanishing of $\Psi$ at $v_o$ according to Lemma \ref{lem:basicvanishing}\eqref{lem:basicvanishing0} in the case $\ell\leq m$, and to \eqref{lem:basicvanishing2a} in the case $\ell>m$.
\item
Away from $y=0$ we can apply Lemma \ref{lem:basicvanishing}\eqref{lem:basicvanishing0}. At $y=0$, we will argue by induction on $\ell$, the case $\ell=0$ being solved by \eqref{lem:basicvanishing0}. We consider the blowing-up $e$ as above, and argue similarly. Setting $f(x,y)=x$, we compute $\Psi\defin\psi_{f\circ e}^\rmod(\cE^{\omega\circ e}\otimes e^*\cR)$.
\[
\omega\circ e=
\begin{cases}
(\text{Chart 1})
\begin{cases}
(\lambda\circ e+u^{m-\ell}v^\ell\mu\circ e)/u^{k-\ell}&\text{if }\ell\leq m;\quad\ref{lem:basicvanishing}\eqref{lem:basicvanishing0}\implique \Psi=0,\\
(u^{\ell-m}\lambda\circ e+v^m\mu\circ e)/u^{k-m}&\text{if }\ell>m;\quad\text{induction}\implique \Psi=0.
\end{cases}
\\[12pt]
(\text{Chart 2})
\begin{cases}
(v^{\prime\ell}\lambda\circ e+u^{\prime m-\ell}\mu\circ e)/u^{\prime k-\ell}v^{\prime k}&\text{if }\ell\leq m;\quad\eqref{lem:basicvanishing2b}\implique \Psi=0,\\
(\mu\circ e+u^{\prime\ell-m}v^{\prime\ell}\lambda\circ e)/u^{\prime k-m}v^{\prime k}&\text{if }\ell>m;\quad\eqref{lem:basicvanishing2a}\implique\Psi=0.\qedhere
\end{cases}
\end{cases}
\]
\end{enumerate}
\end{proof}

\subsubsection*{End of the proof of Proposition \ref{prop:finitudelocalirreg} in Case \eqref{enum:proofa}}
According to Proposition \ref{prop:psipuissances}, it is enough to prove the theorem when $f(x,y)=x$. Let $h\geq1$ denote the valuation of $\omega_k-\omega_k(0)$ (\cf Notation \eqref{eq:omegak}). It is invariant by pull-back by a ramification $x_q\mto x_q^q$. The proof will be done by induction on $h$, the case $h=1$ corresponding to a non-singular point. Assume that $y=0$ is not a singular point of $\omega$ on $D$. This remains the case for $\rho_q^*\omega$ for any ramification $\rho_q:x_q\mto x=x_q^q$. Then, for any $\eta^{(q)}\in\CC\lpb x_q\rpb$, $y=0$ is not a singular point of $\rho_q^*\omega-f_q^*\eta^{(q)}$, where $f_q(x_q,y)=x_q$. Therefore, according to Proposition \ref{prop:supppsix}, $\psi_{f_q}(\cE^{\rho_q^*\omega-f_q^*\eta^{(q)}}\otimes\nobreak\rho_q^*\cR)_0=\nobreak0$.

Let us now assume that $h\geq2$. Let \index{$NP$@$\NP$}$\NP(\omega)$ be the Newton polygon of $\omega$ (in the given coordinate system), which is by definition the convex hull of the union of the quadrants $(j,v_y(\omega_j(y)))+(\RR_-\times\RR_+)$ in~$\RR^2$, where $v_y$ denotes the valuation (\cf Figure~\ref{fig:NPomega}). Note that $(k,h)$ is a point with nonzero coefficient on the vertical part of the boundary of $\NP(\omega)$ (it is a vertex if $\omega_k(0)=0$). It will be useful to set $\omega'=\omega-\omega_k(0)/x^k$, so that $(k,h)$ is a vertex of $\NP(\omega')$. Note that, as we will have to consider $\rho_q^*\omega-f_q^*\eta^{(q)}$ for any $q$ and $\eta^{(q)}$, we can shift $\omega$ by $f^*(\omega_k(0)/t^k)$ from the beginning, so that it is equivalent to work with $\omega$ or $\omega'$. In the following, we assume for simplicity that $\omega=\omega'$, \ie $\omega_k(0)=0$.

Firstly, the theorem (in Case \eqref{enum:proofa}) holds if $\NP(\omega)$ is a quadrant. In such a case, $\omega=\mu(x,y)y^h/x^k$, where $\mu$ is a local unit, and we can apply Lemma \ref{lem:basicvanishingb}\eqref{lem:basicvanishing3}, which shows that, after any ramification $\rho_q$, $\psi_{f_q}(\cE^{\rho_q^*\omega-f_q^*\eta^{(q)}}\otimes\cR)_0=0$ if $\eta^{(q)}\neq0$ and has a pole of order $<k$. On the other hand, if $\eta^{(q)}\neq0$ has a pole of order $\geq k$, then $y=0$ is not a singular point of $\rho_q^*\omega-f_q^*\eta^{(q)}$, and we apply Proposition \ref{prop:supppsix}.

We now assume that $\NP(\omega)$ is not a quadrant. We will say that the coordinate $y$ is \emph{adapted to $\omega$} if there is no point of the form $(j,h-1)$ on $\partial\NP(\omega)$.

\begin{lemme}
Under the previous assumption on $\omega$, there exists a coordinate $y$ which is adapted to $\omega$.
\end{lemme}

\begin{proof}
Assume that $y$ is not adapted to $\omega$. Let $\gamma_oj-\delta_oj'=\gamma_ok-\delta_oh$ (with $(\gamma_o,\delta_o)=1$) be the equation of the non-vertical edge of $\NP(\omega)$ having $(k,h)$ as a vertex. Since $\NP(\omega)$ is not a quadrant, we have $\gamma_o\neq0$. Then $(j,h-1)$ belongs to this edge for some $j$, and we thus have $\gamma_oj-\delta_o(h-1)=\gamma_ok-\delta_oh$, that is, $\gamma_o(k-j)=\delta_o$. Therefore, $\gamma_o=1$ and $j=k-\delta_o$. Our assumption is then that $\omega_{k-\delta_o,h-1}\neq0$. Let us now set $y=y'-(\omega_{k-\delta_o,h-1}/h)x^{\delta_o}$ and denote by $\NP'(\omega)$ the Newton polygon of $\omega$ in the new variables $(x,y')$. By construction, the point $(k-\delta_o,h-1)$ does not belong anymore to $\partial\NP'(\omega)$. Then,
\begin{itemize}
\item
either $\gamma_oj-\delta_oj'=\gamma_ok-\delta_oh$ remains the equation of the non-vertical side of $\NP'(\omega)$ with vertex $(k,h)$, and $y'$ is adapted to $\omega$,
\item
or the slope $\gamma'_o/\delta'_o$ of the corresponding side strictly decreases, \ie $\gamma'_o/\delta'_o<1/\delta_o$; then, if $y'$ is not adapted to $\omega$, $\gamma'_o=1$ and $\delta'_o\geq\delta_o+1$, so $k-\delta'_o\leq k-\delta_o-1$; since $\omega_{j,j'}\neq0\implique j\geq1$, we must have $k-\delta'_o\geq1$, and this process can continue only a finite number of times.\qedhere
\end{itemize}
\end{proof}

We can therefore assume that the coordinate $y$ is adapted to $\omega$, and still keep $f(x,y)=x$ by using Corollary \ref{cor:changfonction}. Note that adaptedness is preserved by any ramification $\rho_q$. Moreover, if $y$ is adapted to $\omega$, it is also adapted to $\rho_q^*\omega-f_q^*\eta^{(q)}$ for any $\eta^{(q)}\in\CC\lpb t_q\rpb/\CC\{t_q\}$: indeed, either the biggest slope of $\partial\NP(\rho_q^*\omega-f_q^*\eta^{(q)})$ is strictly bigger than that of $\partial\NP(\rho_q^*\omega)$, and the corresponding edge does not contain any point, except its vertices, corresponding to a monomial of $\rho_q^*\omega-f_q^*\eta^{(q)}$, or both slopes are equal, and the addition of $f_q^*\eta^{(q)}$ possibly changes only the point with $j'=0$ on this edge; since $h\geq2$, this does not affect a possible integral point of the form $(j,h-1)$.

We say that $\omega$ is \emph{admissible} (with respect to the given coordinate system) if $\NP(\omega)$ has a vertex $(\ell,0)$ with $\ell>0$ (and $\ell<k$). We call this vertex the \emph{admissibility vertex}. Note that admissibility is preserved by any ramification $\rho_q$. It is straightforward to check that there exists a finite set $S\subset\CC\lpb t\rpb$ such that, for each $q\in\NN^*$ and each $\eta^{(q)}\in\CC\lpb t_q\rpb\moins\rho_q^*S$, $\rho_q^*\omega-f_q^*\eta^{(q)}$ is admissible. In particular, we may assume from the beginning that $\omega$ is admissible (up to changing~$S$). This does not modify $h$, nor the adaptedness of $y$ to $\omega$. The Newton polygon takes the form like in Figure~\ref{fig:NPomega}.
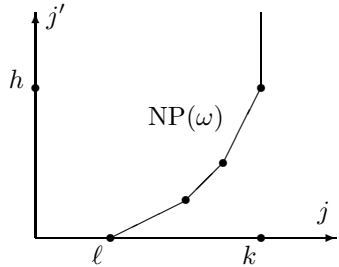
\begin{figure}[htb]
\setlength{\unitlength}{5mm}
\begin{center}
\begin{picture}(8,6)(0,0)
\put(5,2){\line(1,2){1}}
\put(5,2){\line(-1,-1){1}}
\put(5,2){\circle*{.2}}
\put(6,4){\line(0,1){2}}
\put(6,4){\circle*{.2}}
\put(4,1){\line(-2,-1){2}}
\put(0,0){\vector(1,0){8}}
\put(0,0){\vector(0,1){6}}
\put(2,0){\circle*{.2}}
\put(4,1){\circle*{.2}}
\put(6,0){\circle*{.2}}
\put(0,4){\circle*{.2}}
\put(-.7,4){$h$}
\put(.3,5.7){$j'$}
\put(1.5,-.7){$\ell$}
\put(5.5,-.7){$k$}
\put(7.5,.5){$j$}
\put(3,3){$\NP(\omega)$}
\end{picture}
\caption{The Newton polygon of $\omega$\label{fig:NPomega}}
\end{center}
\end{figure}

Let us consider a toric modification $e:(u,v)\mto(x,y)=(u^\alpha v^\gamma,u^\beta v^\delta)$ with $\alpha,\beta,\gamma,\delta\in\NN$, $\alpha\delta-\beta\gamma=-1$, such that $\min_{(j,j')\in\NP(\omega)} (\beta j'-\alpha j)$ and $\min_{(j,j')\in\NP(\omega)} (\delta j'-\gamma j)$ are both achieved on the same vertex $(j_o,j'_o)\neq(\ell,0)$ of $\NP(\omega)$. It follows that $\alpha\neq0$ and $e^{-1}(D)$ has two branches. We also have $j_o,j'_o>0$ and
\begin{equation}\label{eq:toricchart}
e^*\omega=\frac{c(u,v)u^{\beta j'_o}v^{\delta j'_o}}{u^{\alpha j_o}v^{\gamma j_o}}=\frac{c(u,v)}{u^{\alpha j_o-\beta j'_o}v^{\gamma j_o-\delta j'_o}},\quad\text{with }c(0,0)=\omega_{j_o,j'_o}\neq0.
\end{equation}
By assumption of admissibility, we have $\beta j'_o<\alpha j_o$ and $\delta j'_o<\gamma j_o$.

Let $e:X\to\CC^2$ be a toric modification covered by affine charts with coordinates $(u,v)$ as above, such that the slope $\alpha/\beta$ of each side of $\NP(\omega)$ occurs as an exponent $(\alpha,\beta)$ of some chart of $e$. To each chart corresponds then a vertex of $\NP(\omega)$. Moreover, \eqref{eq:toricchart} together with the admissibility condition implies that $e^*\omega$ has a pole along each component of $e^{-1}(D)$. Away from the crossing points, there is thus a finite number of singular points of $e^*\omega$. Let us describe these singular points.
\begin{enumeratea}
\item
Assume $\alpha\neq0$ and, in the chart as in \eqref{eq:toricchart}, assume that we work on the open set $v\neq0$ so that $e^{-1}(D)=\{u=0\}$. If $\alpha/\beta$ is not a slope of an edge of $\NP(\omega)$, then $e^*\omega$ has the form given in \eqref{eq:toricchart} and there is no singular point on this open smooth component of $e^{-1}(D)$.
\item
If on the other hand $\alpha/\beta$ is a slope of an edge of $\NP(\omega)$ (and still assuming $\alpha\neq0$), let us write $\sum_{\alpha j-\beta j'=k_1}\omega_{j,j'}y^{j'}/x^j$ the corresponding part of $\omega$ (at least two coefficients $\omega_{j,j'}$ are nonzero, namely those corresponding to the two vertices $(j_0,j'_0)$ and $(j_1,j'_1)$, $j_0<j_1$, of the edge). Then, in the corresponding chart, $e^*(\omega)$ has a pole of order $k_1$ along $u=0$ and the singular points are the points whose coordinate~$v$ is nonzero and $1/v$ is a multiple root of the polynomial $\sum_{\alpha j-\beta j'=k_1}\omega_{j,j'}w^{\gamma j-\delta j'}$. The height of such a singular point is its multiplicity as a root of this polynomial. The degree of the polynomial (once divided by the maximal power of $w$) is equal to $\gamma(j_1-j_0)-\delta(j'_1-j'_0)$. Since $\alpha(j_1-j_0)=\beta(j'_1-j'_0)$ and $\beta\gamma-\alpha\delta=1$, this degree can be written as $(j'_1-j'_0)/\alpha\leq h/\alpha$.

Let us check that the multiplicity of each root is strictly less than $h$. This is clear if the degree is $<h$. On the other hand, the degree is equal to $h$ if and only if $\NP(\omega)$ has only one edge, with vertices $(\ell,0)$ and $(k,h)$, and $\alpha=1$, $\beta=(k-\ell)/h\in\NN$. The polynomial is written as $\sum_{j'=0}^h\omega_{\ell+\beta j',j'}w^{j'}$. By assumption, the coefficients of~$1$ and~$w^h$ are nonzero, while the coefficient of $w^{h-1}$ is zero by adaptedness. This polynomial has therefore no root of multiplicity $h$.

\item
If $\alpha=0$, we have $\beta=\gamma=1$ and $e^{-1}(D)=\{v=0\}$. If $1/\delta$ is not the slope of the non-horizontal edge of $\NP(\omega)$ with vertex $(\ell,0)$, then $e^*\omega=c(u,v)/v^\ell$ with $c(u,0)\in\CC^*$. Then $e^*\omega$ has no singular point on $\{v=0\}$. On the other hand, if $1/\delta=1/\delta_o$ is the slope of this edge, $u_o\in\CC$ is a singular point of $e^*\omega$ if and only if it is a multiple root of $\sum_{j'\geq0}\omega_{\ell+\delta_oj',j'}u^{j'}$. The degree of this polynomial is $\leq h$, with equality only if $k=\ell+\delta_oh$, and an argument as above shows that the multiplicity of each root is $<h$.
\end{enumeratea}

\medskip
We now come back to the proof of the proposition in Case \eqref{enum:proofa} by induction on~$h$. By taking $\eta\in\CC\lpb t\rpb/\CC\lcr t\rcr$ out of the finite set $S$ introduced above, we can assume that $\omega-f^*\eta$ is admissible for each such $\eta$. Up to replacing $\omega$ with $\omega-\eta_o$ and~$S$ with $S-\eta_o$ for a suitable $\eta_o$, we may assume that $0\notin S$, \ie $\omega$ itself is admissible. Recall also that the adaptedness assumption of the coordinate $y$ with respect to $\omega$, and hence to any $\rho_q^*\omega-f_q^*\eta^{(q)}$, is still active. We can already perform (Proposition \ref{prop:psiramif}) a ramification in order that each side of the Newton polygon of $\omega$ has a slope whose inverse is an integer, or $\infty$. Let us then consider a proper toric modification $e$ adapted to $\NP(\omega)$ as above. Given $q\in\NN^*$, we can lift the ramification $\rho_q:t_q\to t=t_q^q$ in each chart of~$e$ by $\rho_q:(u_q,v_q)\mto(u,v)=(u_q^q,v_q^q)$, and we similarly define the finite map $e_q:(u_q,v_q)\mto(x_q,y)=(u_q^\alpha v_q^\gamma,u_q^{q\beta}v_q^{q\delta})$. This does not affect the multiplicity of the possible singular points away from the center of the chart. In such a chart, $(f\circ e)_q=f_q\circ e_q:(u_q,v_q)\mto u_q^\alpha v_q^\gamma$. (We will modify the definition of $\rho_q,e_q$ if $\alpha=0$, see \eqref{enum:alphanul} below.)

\begin{enumerate}\setcounter{enumi}{-1}
\item\label{enum:crossing}
At each crossing point of $e^{-1}(D)$ (which corresponds to a vertex $(j_o,j'_o)\neq(\ell,0)$ of $\NP(\omega)$), $e^*\omega$ takes the form \eqref{eq:toricchart}. On the other hand, given $\eta=\sum_{i=1}^m\eta_i/t^i$ (with $\eta_m\neq0$), we have $(f\circ e)^{-1}\eta=\sum_{i=1}^m\eta_i/u^{\alpha i}v^{\gamma i}$. Note that $(\alpha j_o-\nobreak\beta j'_o,\gamma j_o-\nobreak\delta j'_o)\neq(\alpha i,\gamma i)$ for any $i\geq1$. We conclude from Lemma \ref{lem:basicvanishingb} (\eqref{lem:basicvanishing2a} if $(\alpha m,\gamma m)$ and $(\alpha j_o-\nobreak\beta j'_o,\gamma j_o-\nobreak\delta j'_o)$ are comparable in $\NN^2$, and \eqref{lem:basicvanishing2b} if they are not comparable) that $\psi_{f\circ e}(e^+(\cE^{\omega-f^*\eta}\otimes\nobreak\cR)_{(0,0)}=0$ if $\eta\neq0$. A similar argument can be used for $\rho_q^*\omega-f_q^*\eta^{(q)}$ for any $q\geq1$.
\item\label{enum:notslope}
Assume that $\alpha\neq0$ and $\alpha/\beta$ is not a slope of an edge of $\NP(\omega)$ at a vertex $(j_o,j'_o)$ with $j'_o\neq0$. Then the coefficient of $1/u^{\alpha j_o-\beta j'_o}$ in $e^*\omega-(f\circ e)^*\eta$ is a Laurent polynomial in the variable $v$ with at most two monomials, namely $\omega_{j_o,j'_o}v^{\delta j'_o-\gamma j_o}$ and $\eta_iv^{-\gamma i}$ for $i$ such that $\alpha i=\alpha j_o-\beta j'_o$. Such a polynomial cannot have a multiple root which is nonzero. Hence $e^*\omega-(f\circ e)^*\eta$ has no singular point on $\{u=0\}\cap\{v\neq0\}$, and $\psi_{f\circ e}(e^+(\cE^{\omega-f^*\eta}\otimes\nobreak\cR)=0$ all along this set. The same property holds for $\rho_q^*\omega-f_q^*\eta^{(q)}$ for any $q\geq1$.
\item\label{enum:slope}
Assume that $\alpha\neq0$ and $\alpha/\beta$ is a slope of an edge of $\NP(\omega)$ at a vertex $(j_o,j'_o)$ with $j'_o\neq0$ (by our assumption, $\beta/\alpha$ is an integer). Let us fix $q\geq1$. For which $\eta^{(q)}$ does the set of singular points of $e_q^*\rho_q^*\omega-(f_q\circ e_q)^*\eta^{(q)}$ differs from that of $e_q^*\rho_q^*\omega$ on $\{u_q=0,v_q\neq0\}$? The order of the pole along $u_q=0$ of $e_q^*\rho_q^*\omega$ is $q(\alpha j_o-\beta j'_o)$ and the corresponding coefficient is $\sum_{(j,j')\mid\alpha j-\beta j'=\alpha j_o-\beta j'_o}\omega_{j,j'}/v_q^{q(\gamma j-\delta j')}$. On the other hand, if we set $\eta^{(q)}=\sum_{i=1}^m\eta^{(q)}_it_q^{-i}$, the dominant term of $(f_q\circ e_q)^*\eta^{(q)}$ is $\eta_mv_q^{-\gamma m}u_q^{-\alpha m}$.
\begin{itemize}
\item
If $\alpha m>q(\alpha j_o-\beta j'_o)$, $e_q^*\rho_q^*\omega-(f_q\circ e_q)^*\eta^{(q)}$ does not have any singular point.
\item
If $\alpha m<q(\alpha j_o-\beta j'_o)$, $e_q^*\rho_q^*\omega-(f_q\circ e_q)^*\eta^{(q)}$ has the same singular points as $e_q^*\rho_q^*\omega$, all of which have multiplicity $<h$ as already proved.
\item
If $\alpha m=q(\alpha j_o-\beta j'_o)$, the coefficient of $1/u_q^{q(\alpha j_o-\beta j'_o)}$ in $e_q^*\rho_q^*\omega-(f_q\circ e_q)^*\eta^{(q)}$ is equal to the Laurent polynomial
\[
P_{q,\eta_m}(w_q)\defin\Big(\sum_{(j,j')\mid\alpha j-\beta j'=\alpha j_o-\beta j'_o}\hspace*{-4mm}\omega_{j,j'}w_q^{q(\gamma j-\delta j')}\Big)-\eta_mw_q^{\gamma m}\qquad (w_q=1/v_q).
\]
Set $m_o=j_o-\beta j'_o/\alpha$ and assume $m_o>0$ (hence $m_o\in\NN^*$, due to our assumption). Then there exists a finite number of $c\in\CC$ such that $P_{1,c}(w)$ has a multiple root. Moreover, the multiple roots of $P_{q,c}(w)=P_{1,c}(w^q)$ are the $q$th powers of those of $P_{1,c}$, and keep the same multiplicity. The set of all multiple roots of all $P_{1,c}$ are therefore the possible singular points if $q=1$, and the possible singular points for $q\geq2$ are obtained from the previous ones by a ramification $w_q\mto w=w_q^q$. Note also that the multiplicity of such a multiple root is $<h$. This is seen as above by using the adaptedness assumption.
\end{itemize}
For each multiple root $w_o$ of $P_{1,c}$, we apply the inductive assumption to $e^*\omega-(f\circ\nobreak e)^*(ct^{-m_o})$ and get a finite set $S'_{c,w_o}$. By translating by $(f\circ e)^*(ct^{-m_o})$, we obtain a finite set $S'_{w_o}$ associated to $e^*\omega$ at $w_o$. When $w_o$ varies in the finite set of all possible singular points in $\CC^*$, we obtain a finite set $S^{\alpha,\beta}$. For every $q\geq1$ and every $\eta^{(q)}\notin \rho_q^*(S\cup S^{\alpha,\beta})$, we have $\psi_{f\circ e}\big(\rho_{q,+}\cE^{-(f_q\circ e_q)^*\eta^{(q)}}\otimes\cE^{e^*\omega}\otimes e^+\cR\big)=0$ all along $v\neq0$.

\item\label{enum:alphanul}
We now consider the case where $\alpha=0$, so that $\beta=\gamma=1$. The case where $1/\delta$ is not the slope of the non-horizontal edge having $(\ell,0)$ as a vertex is treated as in~\eqref{enum:notslope} above, so we only consider the case where it is equal to the slope $1/\delta_o$. In such a chart, we have $f\circ e(u,v)=v$, and we can define the ramification $\rho_q$ on the variable~$v$ only, \ie $\rho_q(u,v_q)=(u,v_q^q)$, and the map $e_q$ is defined by $e_q(u,v_q)=(v_q,uv_q^{q\delta})$. We can then argue exactly as in \eqref{enum:slope} above, with the only difference that the variable $w$, which is equal to $u$, can achieve the value~$0$. We notice that the ramification $\rho_q$ does not affect the multiplicity of the possible singular points all over the chart, including at $u=0$.
\end{enumerate}

Let us now denote by $S^\omega$ the finite set $S\cup S^e$, where $S^e$ is the union of all $S^{\alpha,\beta}$ as above for which $\alpha/\beta$ is a slope of $\NP(\omega)$ (with the previous assumptions that $y$ is adapted to $\omega$, $\omega$ is admissible and all $\beta/\alpha$ are integers). Let us fix $q\in\NN^*$ and $\eta^{(q)}\notin \rho_q^*(S^\omega)$. We will show that $\psi_f((\rho_{q,+}\cE^{\eta^{(q)}})\otimes\cE^\omega\otimes\cR)_0=0$. According to Proposition \ref{prop:psiimdir}, it is enough to check that $\psi_{f\circ e}((e^+\rho_{q,+}\cE^{f_q^*\eta^{(q)}})\otimes\cE^{e^*\omega}\otimes e^+\cR)_{e^{-1}(0)}=0$. This is a local problem on $e^{-1}(0)$. In each chart, we may use a ramification $\rho_q$ at the level of the variables $u,v$ (\resp the variable $v$ if $\alpha=0$). The computation \eqref{enum:crossing}--\eqref{enum:alphanul} above shows that, in a given chart, $\psi_{f\circ e}((\rho_{q,+}e_q^+\cE^{f_q^*\eta^{(q)}})\otimes\cE^{e^*\omega}\otimes e^+\cR)=0$, and this is equivalent to the desired vanishing in this chart.\qed

\subsubsection*{Proof of Proposition \ref{prop:finitudelocalirreg} in Case \eqref{enum:proofb}}
We will need a criterion similar to that of Proposition \ref{prop:supppsix} ensuring the vanishing of~$\psi_f^\rmod$. Let $\omega(x,y)$ be any nonzero germ at the origin of a class (modulo holomorphic functions) of a meromorphic function, and let us denote its Newton polygon by $\NP(\omega)$. By definition, it is the convex hull of the quadrants $(j,j')-\NN^2$ for the pairs $(j,j')$ such that the coefficient $\omega_{j,j'}/x^jy^{j'}$ of $\omega$ is nonzero. Note that $j$ and $j'$ are not both $\leq0$.
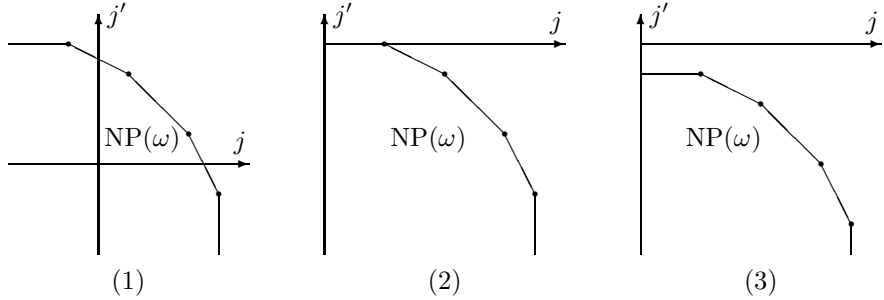
\begin{figure}[htb]
\setlength{\unitlength}{4mm}
\begin{center}
\begin{tabular}{ccccc}
\begin{picture}(8,8)(0,0)
\put(6,4){\line(1,-2){1}}
\put(6,4){\line(-1,1){2}}
\put(7,2){\line(0,-1){2}}
\put(4,6){\line(-2,1){2}}
\put(2,7){\line(-1,0){2}}
\put(0,3){\vector(1,0){8}}
\put(3,0){\vector(0,1){8}}
\put(2,7){\circle*{.2}}
\put(4,6){\circle*{.2}}
\put(6,4){\circle*{.2}}
\put(7,2){\circle*{.2}}
\put(3.3,7.7){$j'$}
\put(7.5,3.5){$j$}
\put(3.2,3.6){$\NP(\omega)$}
\end{picture}
&\hspace*{3mm}&
\begin{picture}(8,8)(0,0)
\put(6,4){\line(1,-2){1}}
\put(6,4){\line(-1,1){2}}
\put(7,2){\line(0,-1){2}}
\put(4,6){\line(-2,1){2}}
\put(2,7){\line(-1,0){2}}
\put(0,7){\vector(1,0){8}}
\put(0,0){\vector(0,1){8}}
\put(2,7){\circle*{.2}}
\put(4,6){\circle*{.2}}
\put(6,4){\circle*{.2}}
\put(7,2){\circle*{.2}}
\put(.3,7.7){$j'$}
\put(7.5,7.5){$j$}
\put(2.2,3.6){$\NP(\omega)$}
\end{picture}
&\hspace*{3mm}&
\begin{picture}(8,8)(0,0)
\put(6,3){\line(1,-2){1}}
\put(6,3){\line(-1,1){2}}
\put(7,1){\line(0,-1){1}}
\put(4,5){\line(-2,1){2}}
\put(2,6){\line(-1,0){2}}
\put(0,7){\vector(1,0){8}}
\put(0,0){\vector(0,1){8}}
\put(2,6){\circle*{.2}}
\put(4,5){\circle*{.2}}
\put(6,3){\circle*{.2}}
\put(7,1){\circle*{.2}}
\put(.3,7.7){$j'$}
\put(7.5,7.5){$j$}
\put(1.5,3.6){$\NP(\omega)$}
\end{picture}\\
(1)&&(2)&&(3)
\end{tabular}
\caption{Examples of Newton polygons of $\omega$\label{fig:NPomega2}}
\end{center}
\end{figure}

We say that $\omega$ is \emph{admissible} if $\partial\NP(\omega)$ is not contained in
\begin{itemize}
\item
a ``closed semi-negative quadrant'', that is, one of the quadrants $(-\NN)\times\NN$ or $\NN\times(-\NN)$ (\eg as in Figure \ref{fig:NPomega2}$(1)$ above, but not (2) or (3)), if $a,b>0$,
\item
the ``open semi-negative quadrant'' $\NN\times(-\NN^*)$ if $b=0$ (\ie Figure \ref{fig:NPomega2}$(3)$ is excluded, but Figure \ref{fig:NPomega2}$(2)$ is accepted).
\end{itemize}
Since we are interested in $\omega$ modulo $f^*\CC\lpb t\rpb$, we can assume from the beginning that~$\omega$ is admissible. More precisely:

\begin{lemme}\label{lem:nonadmissiblefini}
Given $\omega$, there exists a finite number of ramified polar parts $\eta$ such that $\rho_q^*\omega-f_q^*\eta^{(q)}$ is not admissible.\qed
\end{lemme}

\begin{definitio}\label{def:singomegaD}
Assume that $\omega$ is admissible and let $(\alpha,\beta)\in(\NN^*)^2$ be coprime. The subset $\Sing_{\alpha,\beta}(\omega,D)\subset\CC^*$ consists of the nonzero complex numbers $v_o$ such that there exists $(\gamma,\delta)\in\NN^2$ satisfying
\begin{enumerate}
\item
$\alpha\delta-\beta\gamma=\pm1$,
\item
$v_o$ is a singular point of the pair $\big(\omega(v_o^\gamma u^\alpha,v_o^\delta u^\beta),\{u=0\}\big)$.
\end{enumerate}
We set
\[
\Sing(\omega,D)=\bigcup_{(\alpha,\beta)}\Sing_{\alpha,\beta}(\omega,D).
\]
\end{definitio}

\begin{lemme}\label{lem:singfinite}
If $\omega$ is admissible, the set $\Sing(\omega,D)$ is finite.
\end{lemme}

\begin{proof}
If $\mu=\max\{\alpha j+\beta j'\mid(j,j')\in\NP(\omega)\}$ is achieved at a single vertex $(j_o,j'_o)$ of $\partial\NP(\omega)$ then, for each $v_o\in\CC^*$, $\omega(v_o^\gamma u^\alpha,v_o^\delta u^\beta)=c v_o^{-(\gamma j_o+\delta j'_o)}/u^\mu+\text{lower order poles}$, with $c=\omega_{j_o,j'_o}\neq0$, (that $\mu$ is $\geq1$ follows from the assumption of admissibility). Thus $\Sing_{\alpha,\beta}(\omega,D)=\emptyset$ if $(\alpha,\beta)$ is not the direction of an edge of $\partial\NP(\omega)$. On the other hand, each $\Sing_{\alpha,\beta}(\omega,D)$ is finite, because the dependence with respect to $v$ of the coefficient $c(v)$ of $1/u^\mu$ is algebraic.
\end{proof}

Let $e:X'\to X$ be a smooth toric modification of a neighbourhood $X$ of \hbox{$(x=0,y=0)$}. We say that $e$ is \emph{adapted to $\NP(\omega)$} if $X'$ is covered by charts with coordinates $(u,v)$ such that $e(u,v)=(u^\alpha v^\gamma,u^\beta v^\delta)$ with $\alpha\delta-\beta\gamma=\pm1$ and, for each such chart, there is a unique vertex $(j_o,j'_o)$ of $\partial\NP(\omega)$ such that
\begin{align*}
\alpha j_o+\beta j'_o&=\mu\defin\max\{\alpha j+\beta j'\mid(j,j')\in\NP(\omega)\},\\
\gamma j_o+\delta j'_o&=\mu'\defin\max\{\gamma j+ \delta j'\mid(j,j')\in\NP(\omega)\}.
\end{align*}

\begin{lemme}\label{lem:psisomdir}
Assume that $\omega$ is admissible and $e$ is adapted to $\NP(\omega)$. Then $\psi_f(\cE^\omega\otimes\cR)=\bigoplus_{v_o\in\Sing(\omega,D)}\psi_{f\circ e}(e^+(\cE^\omega\otimes\cR))_{v_o}$.
\end{lemme}

\begin{proof}
We first note that all singular points of $(\omega,D)$ appear on some exceptional component of $e$, due to adaptedness. Since $\omega$ is admissible, we have $(\mu,\mu')\in\NN^2\moins\nobreak\{0\}$. At the center of a chart corresponding to a vertex $(j_o,j'_o)$, we have $e^*\omega(u,v)=c(u,v)/u^\mu v^{\mu'}$ with $c(0,0)=\omega_{j_o,j'_o}$ and $f\circ e(u,v)=u^{\alpha a+\beta b}v^{\gamma a+\delta b}$. Therefore, if $\mu,\mu'>0$, $\psi_{f\circ e}^\rmod(\cE^{e^*\omega}\otimes e^+\cR)_0=0$, according to Lemma \ref{lem:basicvanishingb}\eqref{lem:basicvanishing2b}. If $\mu'=0$ and $\mu>0$, then by admissibility we have $j'_o=0$, $\gamma=0$, $\alpha=\delta=1$ and, by admissibility, $b=0$, $f(u,v)=u^a$, hence $\psi_{f\circ e}^\rmod(\cE^{e^*\omega}\otimes e^+\cR)_0=0$ according to Lemma \ref{lem:basicvanishing}\eqref{lem:basicvanishing0}. The case $\mu=0$ and $\mu'>0$ is obtained by inverting the roles of $\alpha,\beta$ and $\gamma,\delta$.

Admissibility implies that $e^*\omega$ has a pole along each irreducible component of $(f\circ e)^{-1}(0)$. From Proposition \ref{prop:supppsix} we obtain that $\psi_{f\circ e}(e^+(\cE^\omega\otimes\cR))$ is supported on $\Sing(\omega,D)$. The result follows from Proposition \ref{prop:comppush}, since the set of singular points of $e^*\omega$ is finite, by Lemma \ref{lem:singfinite}.
\end{proof}

\begin{remarque}\label{rem:usecasea}
Note that the result of Case \eqref{enum:proofa} implies that there exists a finite set of ramified polar parts $\eta$ such that, for any $\eta$ not belonging to this set, and any singular point $v_o\in\Sing(\omega,D)$, $\psi_{f\circ e}\big(e^+[(\rho_{q,+}\cE^{f_q^*\eta^{(q)}})\otimes\cE^\omega\otimes\cR]\big)_{v_o}=0$.
\end{remarque}

The point in proving Case \eqref{enum:proofb} is that the singular points of $\rho_q^*\omega-f_q^*\eta^{(q)}$ may depend on~$\eta$. We will therefore take into account the variation with $\eta$ of this set of singular points.

For any $\omega\neq0$, let us set $\omega'=\omega-\sum_{k\in\QQ_+}\omega_{ka,kb}/x^{ka}y^{kb}$.We say that an edge of $\NP(\omega')$ is \emph{admissible} if the line which supports it cuts the open quadrant $j,j'>0$. When $\eta$ varies among ramified polar parts, the set of $\rho_q^*\omega-f_q^*\eta^{(q)}$ is equal to that of $\rho_q^*\omega'-f_q^*\eta^{(q)}$, and any element of the latter set is admissible if $\eta\neq0$.

Let us denote by $K(\omega')\subset\QQ_+^*$ the finite set of positive rational numbers $k$ such that $k\cdot(a,b)$ belongs to a line containing an (admissible) edge of $\NP(\omega')$ (\cf Figure~\ref{fig:Komega}; recall that $f(x,y)=x^ay^b$).

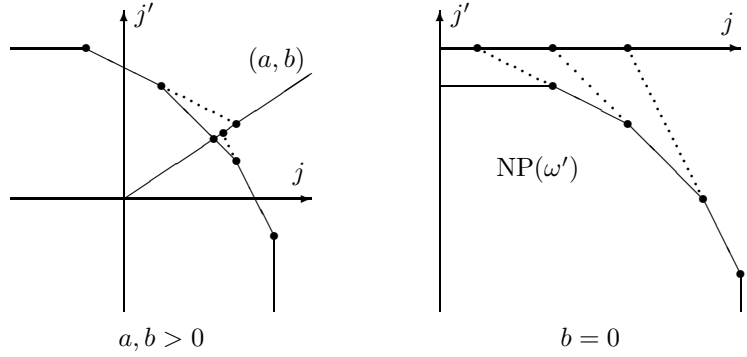
\begin{figure}[htb]
\setlength{\unitlength}{5mm}
\begin{center}
\begin{tabular}{ccc}
\begin{picture}(8,8)(0,0)
\put(6,4){\line(1,-2){1}}
\put(6,4){\line(-1,1){2}}
\put(7,2){\line(0,-1){2}}
\put(4,6){\line(-2,1){2}}
\put(2,7){\line(-1,0){2}}
\put(3,3){\line(3,2){5}}
\put(0,3){\vector(1,0){8}}
\put(3,0){\vector(0,1){8}}
\put(2,7){\circle*{.2}}
\put(4,6){\circle*{.2}}
\put(6,4){\circle*{.2}}
\put(7,2){\circle*{.2}}
\put(6,5){\circle*{.2}}
\put(5.4,4.6){\circle*{.2}}
\put(5.65,4.75){\circle*{.2}}
\multiput(4,6)(.2,-.1){10}{\circle*{.1}}
\multiput(6,4)(-.1,.2){4}{\circle*{.1}}
\put(3.3,7.7){$j'$}
\put(7.5,3.5){$j$}
\put(6.3,6.5){$(a,b)$}
\end{picture}&\hspace*{10mm}&
\begin{picture}(8,8)(0,0)
\put(7,3){\line(1,-2){1}}
\multiput(7,3)(-.1,.2){20}{\circle*{.1}}
\put(7,3){\line(-1,1){2}}
\multiput(5,5)(-.2,.2){10}{\circle*{.1}}
\put(8,1){\line(0,-1){1}}
\put(5,5){\line(-2,1){2}}
\multiput(3,6)(-.2,.1){10}{\circle*{.1}}
\put(3,6){\line(-1,0){3}}
\put(0,7){\vector(1,0){8}}
\put(0,0){\vector(0,1){8}}
\put(3,6){\circle*{.2}}
\put(5,5){\circle*{.2}}
\put(7,3){\circle*{.2}}
\put(8,1){\circle*{.2}}
\put(1,7){\circle*{.2}}
\put(3,7){\circle*{.2}}
\put(5,7){\circle*{.2}}
\put(.3,7.7){$j'$}
\put(7.5,7.5){$j$}
\put(1.5,3.6){$\NP(\omega')$}
\end{picture}\\
$a,b>0$&&$b=0$
\end{tabular}
\end{center}
\caption{Examples of sets $K(\omega)\cdot(a,b)$}\label{fig:Komega}
\end{figure}

\begin{lemme}\label{lem:kogen}
There exists at most a finite number of nonzero ramified polar parts $\eta=\sum_{k\in\QQ_+^*}\eta_kt^k$ such that, if $k_o\defin\max\{k\mid\eta_k\neq0\}$ does not belong to $K(\omega')$, then $\psi_{f_q}(\cE^{\rho_q^*\omega'-f_q^*\eta^{(q)}}\otimes\nobreak\cR)_0\neq\nobreak0$.
\end{lemme}

\begin{proof}
We denote by $\Sing'(\omega',D)$ the set of singular points of $\omega'$ corresponding to admissible edges of $\NP(\omega')$. Let $\eta\neq0$ be as in the lemma and let $q$ be associated with $\eta$ as above. We claim that $\Sing(\rho_q^*\omega'-f_q^*\eta^{(q)},D_q)\subset\Sing'(\rho_q^*\omega',D_q)$. Indeed, let us consider an edge $E$ of $\partial\NP(\rho_q^*\omega'-f_q^*\eta^{(q)})$. By assumption on $k_o$, either $E$ is an admissible edge of $\partial\NP(\rho_q^*\omega')$, and then the corresponding singular points of $\rho_q^*\omega'-f_q^*\eta^{(q)}$ are those of $\rho_q^*\omega'$, or it is not an edge of $\partial\NP(\rho_q^*\omega')$, and then it supports only two monomials of $\rho_q^*\omega'-f_q^*\eta^{(q)}$, one at each vertex, one of them being $k_o\cdot(a,b)$. In the latter case, it produces no singular point of $\rho_q^*\omega'-f_q^*\eta^{(q)}$, since a Laurent polynomial with exactly two terms has only simple roots in $\CC^*$.

If moreover $\eta$ does not belong to the finite set considered in Remark \ref{rem:usecasea}, we have $\psi_{f\circ e}\big(e^+[(\rho_{q,+}\cE^{f_q^*\eta^{(q)}})\otimes\cE^\omega\otimes\cR]\big)_{v_o}=0$ for each $v_o\in\Sing'(\omega',D)$. Applying now Lemma \ref{lem:psisomdir} to $\rho_q^*\omega'-f_q^*\eta^{(q)}$, we obtained the desired vanishing.
\end{proof}

Let now $\eta\neq0$ be such that $k_o=k_o(\eta)\in K(\omega')$.

\begin{lemme}
There exists a finite set $F_{k_o}\subset\CC$ such that
\[
\Sing(\rho_q^*\omega'-f_q^*(\eta_{k_o}/t^{k_o}),D_q)\not\subset\Sing'(\rho_q^*\omega',D_q)\implique \eta_{k_o}\in F_{k_o}.
\]
\end{lemme}

\begin{proof}
It is analogous to the proof given in Part \eqref{enum:slope} of the proof of Case \eqref{enum:proofa}.
\end{proof}

Once this lemma is proved, one may apply the same argument as in Lemma \ref{lem:kogen} if $\eta_{k_o}$ does not belong to $F_{k_o}$.
For each $\eta_{k_o}\in F_{k_o}$, the singular set $\Sing(\rho_q^*\omega'-f_q^*\eta^{(q)})$ does not depend on $\eta=\eta_{k_o}/t^{k_o}+\sum_{k<k_o}\eta_k/t^k$, and one can apply the same argument as in Lemma \ref{lem:kogen}.
\end{proof}

\chapterspace{-2}
\chapter{Nearby cycles of Stokes-filtered~local~systems}\label{chap:nearby}

\begin{sommaire}
In this \chaptername, we define a nearby cycle functor for a good Stokes-filtered local system on $(X,D)$, relative to a holomorphic function whose zero set is contained in the normal crossing divisor $D$. We then show that the Riemann-Hilbert correspondence of \Chaptersname\ref{chap:RHgoodnc} is compatible with taking nearby cycles, either in the sense of irregular nearby cycles for meromorphic flat bundles as defined in \Chaptersname\ref{chap:irregnearby}, or as defined for Stokes-filtered local systems in this \chaptername.
\end{sommaire}

\subsection{Introduction}
The sheaf-theoretic definition of the nearby cycle functor by P\ptbl Deligne in \cite{Deligne73} has led to the definition of the moderate nearby cycle functor for holonomic $\cD$-modules in order that the Riemann-Hilbert correspondence for regular holonomic $\cD$-modules is compatible with both functors. Following P\ptbl Deligne, we have extended the moderate nearby cycle functor to the irregular nearby cycle functor in \Chaptersname\ref{chap:irregnearby}. Going now the way back compared to the case of holonomic $\cD$-modules with regular singularities, we will define a nearby cycle functor for Stokes-filtered local systems and we will prove, in the good case, the compatibility with the irregular nearby cycle functor via the Riemann-Hilbert correspondence of \Chaptersname\ref{chap:RHgoodnc}.

As we will see, the proof of some properties, like the compatibility with proper push-forward, is much easier in the case of holonomic $\cD$-modules, where we can use the strength of the algebraic machinery, and we will give few proofs for Stokes-filtered local systems, where the behaviour of the topology of the real blow-up spaces with respect to complex blowing-ups is difficult to understand in general. This is why we will mainly restrict to dimension two.

\subsection{Nearby cycles along a function (the good case)}\label{subsec:nearbygood}
We denote by~$S$ a neighbourhood of the origin in $\CC$ and we consider a holomorphic function $f:X\to\nobreak S$ such that the divisor $D=f^{-1}(0)$ has normal crossings with smooth irreducible components $D_{j\in J}$.

Let $\varpi:\wt X=\wt X(D_{j\in J})\to X$ be the real blow-up space of $X$ along the components of $D$ and let $\wt S$ be the real blow-up space of $S$ at the origin. There is a lifting $\wt f:\wt X\to \wt S$ of $f$. We have a diagram of sheaves of ordered abelian groups
\[
\ccI_{\wt X}\From{q_f=f^*}\wt f^{-1}\ccI_{\wt S}\To{\wt f}\ccI_{\wt S}.
\]
Because $S$ is one-dimensional, the morphism $f^*$ is injective (\cf Proposition \ref{prop:f*inj}). Let us also notice that any local section $\varphi$ of $\ccI_{\partial\wt S}$ determines a finite covering $\wt\Sigma_\varphi\subset\ccIet_{|\partial\wt S}$ of $\partial\wt S$: indeed, this is clear if $\varphi$ is non-ramified, \ie is a local section of $\ccI_{\partial\wt S,1}$; if it is ramified of order $d$, then one argues by using a ramified covering $\rho_d$ of $(S,0)$. Now, if $\wt\Sigma_\varphi$ is such a covering, then its pull-back $\partial\wt X\times_{\partial\wt S}\wt\Sigma_\varphi$ is a finite covering of $\partial\wt X$, and its image by the inclusion $f^*$ is a finite $\ccI_{\partial\wt X}$-covering of $\partial\wt X$ and we denote it by $\wt\Sigma_{f^*\varphi}$.

Notice that $D\times\partial\wt S$ is the boundary of the real blow-up space of $X$ along $f^{-1}(0)$ (\cf Lemma \ref{lem:eclatereel}) and we have a natural map $(\varpi,\wt f):\partial\wt X\to D\times\partial\wt S$. For $x$ belonging to a stratum of $D$ of codimension $\ell$ in $X$, and for $\theta\in\partial\wt S\simeq S^1$, the fibre $(\varpi,\wt f)^{-1}(x,\theta)$ is a union of a finite number of copies of $(S^1)^{\ell-1}$, since $f$ is locally monomial. The natural map $\wt\Sigma_{f^*\varphi}\to D\times\partial\wt S$ has a similar property.

\begin{definitio}
Let $\wt\Sigma_\varphi\subset\ccIet_{\partial\wt S}$ be the covering associated to a local section $\varphi$ of $\ccI_{\partial\wt S}$, and let $\wt\Sigma\subset\ccIet_{\partial\wt X}$ be a good $\ccIet_{\partial\wt X}$-stratified covering of $\partial\wt X$. We say that the pair $(\wt\Sigma,\varphi)$ is good if $\wt\Sigma\cup \wt\Sigma_{f^*\varphi}$ (which is a stratified covering of $\partial\wt X$) is \emph{good}.
\end{definitio}

Let $(\cL,\cL_\bbullet)$ be a $\ccI_{\partial\wt X}$-filtered local system on $\partial\wt X$. Since $\ccI_{\partial\wt X}$ is not Hausdorff over the crossing points of $D$, the graded sheaf $\gr\cL$ has to be taken on each stratum of~$D$. On the other hand, $\wt f^{-1}\ccI_{\partial\wt S}$ is Hausdorff, since $\ccI_{\partial\wt S}$ is so (\cf Remark \ref{exem:separe}\eqref{exem:separe3}) and, since $q_f^{-1}\cL_\leq$ defines a pre-$\wt f^{-1}\ccI_{\partial\wt S}$-filtration of $\cL$, it is meaningful to consider $\gr\cL$ as a subsheaf on $(\wt f^{-1}\ccI_{\partial\wt S})^\et=\partial\wt X\times_{\partial\wt S}\ccIet_{\partial\wt S}$. We will denote this sheaf as $\gr^f\cL$ in order to avoid any confusion. For any local section $\varphi$ of $\ccI_{\partial\wt S}$, we denote by $\gr^f_{\wt f^{-1}\varphi}\cL$ the restriction of $\gr^f\cL$ to $\partial\wt X\times_{\partial\wt S}\wt\Sigma_\varphi$ and we still denote by $\varpi$ the projection $\partial\wt X\times_{\partial\wt S}\wt\Sigma_\varphi\to D$.

\begin{proposition}\label{prop:grf}
Assume that $(\cL,\cL_\bbullet)$ is a $\ccI_{\partial\wt X}$-filtered local system with associated stratified covering $\wt\Sigma$. Then, for any local section $\varphi$ of $\ccI_{\partial\wt S}$ we have, over each stratum of $D$, a surjective morphism
\[
\gr^f_{\wt f^{-1}\varphi}\cL\to q_f^{-1}\gr_{f^*\varphi}\cL.
\]
Assume moreover that the pair $(\wt\Sigma,\varphi)$ is good. Then the restriction of $\gr^f_{\wt f^{-1}\varphi}\cL$ over each stratum of $D$ is a local system, and $\gr^f_{\wt f^{-1}\varphi}\cL$ is constructible on $\partial\wt X\times_{\partial\wt S}\wt\Sigma_\varphi$ with respect to the stratification obtained by pull-back from that of $D$.
\end{proposition}

\begin{proof}
By definition, we have $(q_f^{-1}\cL_\bbullet)_{\leq\wt f^{-1}\varphi}=\cL_{\leq f^*\varphi}$. Let us fix $y\in\partial\wt X$ and let us set $\theta=\wt f(y)\in\partial\wt S$. The point is to show that $(q_f^{-1}\cL_\bbullet)_{<\wt f^{-1}\varphi,y}\subset\cL_{<f^*\varphi,y}$. Let $\wt\Sigma\subset\ccIet_{\partial\wt X}$ be the stratified $\ccI$-covering attached to $(\cL,\cL_\bbullet)$. On the one hand,
\bgroup\numstareq
\begin{equation}\label{eq:grf*}
(q_f^{-1}\cL_\bbullet)_{<\wt f^{-1}\varphi,y}=\sum_{\psi\letheta\varphi}\cL_{\leq f^*\psi,y}=\sum_{\psi\letheta\varphi}\bigoplus_{\substack{\eta\in\wt\Sigma_y\\\eta\leqy f^*\psi}}\gr_\eta\cL_y=\hspace*{-4mm}\bigoplus_{\substack{\eta\in\wt\Sigma_y\\\exists\psi\letheta\varphi,\,\eta\leqy f^*\psi}}\hspace*{-4mm}\gr_\eta\cL_y.
\end{equation}
\egroup
On the other hand,
\bgroup\numstarstareq
\begin{equation}\label{eq:grf**}
\cL_{<f^*\varphi,y}=\bigoplus_{\substack{\eta\in\wt\Sigma_y\\\eta\ley f^*\varphi}}\gr_\eta\cL_y.
\end{equation}
\egroup

Since $f^*$ is compatible with the order, we have $\psi\letheta\varphi\implique f^*\psi\leqy f^*\varphi$ and, according to the second part of Proposition \ref{prop:f*inj} since $\wt f:\wt X\to \wt S$ is open, we moreover have $f^*\psi\neq f^*\varphi$, that is, $f^*\psi\ley f^*\varphi$. This gives the inclusion $\eqref{eq:grf*}\subset\eqref{eq:grf**}$.

For the second part of the proposition, we will need a lemma.

\begin{lemme}\label{lem:f*le}
Let $\eta\in\ccI_{\partial\wt X,y}$ be such that the associated stratified covering $\wt\Sigma_\eta\cup\nobreak\{0\}\subset\ccIet_{\partial\wt X}$ of $\partial\wt X$ in some neighbourhood of $y$ is good (\ie after some finite ramification around $D$ near $x=\varpi(y)$, $\rho_{\bmd}^*\eta$ is purely monomial). Assume that $\eta\ley0$. Then, setting $\theta=\wt f(y)\in\partial\wt S$, the property
\bgroup\numstareq
\begin{equation}\label{eq:f*le}
\exists\psi\letheta0\in\ccI_{\partial\wt S,\theta},\quad \eta\leqy f^*\psi
\end{equation}
\egroup
holds if and only if $\eta$ has poles along all the local components of $D$ at $x=\varpi(y)$.
\end{lemme}

\begin{proof}
Assume first that $\eta$ has poles along all the local components of $D$. We will prove that, if $\psi\in\ccI_{\partial\wt S,\theta}$ has a pole of order $1/d$ with~$d$ big enough (\ie $\rho_d^{-1}\psi$ has a pole of order one), and if $\psi\letheta0$ (such a $\psi$ clearly exists), then $\eta\leqy f^*\psi$. As in Proposition \ref{prop:f*inj}, one can reduce the statement to the case where~$\eta$ is non-ramified, and so $\eta$ is purely monomial. In local coordinates adapted to $D$, we have $D=\{x_1\cdots x_\ell=0\}$, $f(x_1,\dots,x_n)=x^{\bmk}$ with $\bmk\in(\NN^*)^\ell$, $\eta=u(x_1,\dots,x_n)/x^{\bmm}$ for some $\bmm\in\NN^\ell$ and $u$ is a unit, and our assumption means that $\bmm\in(\NN^*)^\ell$. It is then enough to choose $d$ such that $\bmk<d\bmm$ with respect to the natural partial ordering of $\NN^\ell$. In such a case, $\eta-f^*\psi$ remains purely monomial with the same leading term as $\eta$.

Let us now assume that some $m_j$ is zero, with $j\in\{1,\dots,\ell\}$. Let $\psi$ be such that $\psi\letheta0$. Both $\eta$ and $f^*\psi$ are purely monomial, with leading monomial $x^{-\bmm}$ and $x^{-r\bmk}$ respectively, if $\psi$ has leading monomial $t^{-r}$, $r\in\QQ_+^*$. By our assumption on $\bmm$, there exists a local modification $\epsilon:(X',D')\to(X,D)$ near $x\in D$ such that the leading term of $\epsilon^*(\eta-f^*\psi)$ is $-f^*\psi$. If we had $\eta- f^*\psi\leqy0$ and $\psi\letheta0$, we would also have $-\epsilon^*f^*\psi\leqyp0$ for any $y'\in\wt\epsilon^{-1}(y)$ and $\epsilon^*f^*\psi\leyp0$, a contradiction.
\end{proof}

According to \eqref{eq:grf*} and \eqref{eq:grf**} and to the previous lemma, the kernel of $\gr^f_{\wt f^{-1}\varphi}\cL_y\to q_f^{-1}\gr_{f^*\varphi}\cL_y$ is equal to the sum of $\gr_\eta\cL_y$, where $\eta\in\wt\Sigma_y$ is such that $\eta-f^*\varphi$ has no pole along some irreducible component of $D$ going through $x=\varpi(y)$. This condition does not depend on $y$, but only on the stratum of $D$ which $x$ belongs to. This shows that this kernel is a local system over each stratum of $D$. On the other hand, $\gr_{f^*\varphi}\cL$ is also a local system over each stratum of $D$. Therefore, so is $\gr_{\wt f^{-1}\varphi}^f\cL$.
\end{proof}

\begin{definitio}[Nearby cycles, the good case]\label{def:nearbygood}\index{nearby cycle!of Stokes filtered local systems}
Let $(\cF,\cF_\bbullet)$ be a \emph{good} Stokes-filtered local system on $(X,D)$ (\cf Definition \ref{def:StXD}) with associated stratified $\ccI$-covering $\wt\Sigma\subset\ccIet_{\partial\wt X}$, and let~$\varphi$ be a local section of $\ccI_{\partial\wt S}$ defining a finite covering $\wt\Sigma_\varphi\subset\ccIet_{\partial\wt S}$ of $\partial\wt S\simeq S^1$. Let us assume that the pair $(\wt\Sigma,\varphi)$ is \emph{good}. We then set\index{$FZFW$@$\wt\psi_f^\varphi(\cF,\cF_\bbullet)$}
\begin{align*}
\wt\psi_f^\varphi(\cF,\cF_\bbullet)&=\gr^f_{\wt f^{-1}\varphi}\cL,\\
\index{$FZF$@$\psi_f^\varphi(\cF,\cF_\bbullet)$}\psi_f^\varphi(\cF,\cF_\bbullet)&=\bR(\varpi,\wt f)_*\gr^f_{\wt f^{-1}\varphi}\cL.
\end{align*}
\end{definitio}

We have seen that $\wt\psi_f^\varphi(\cF,\cF_\bbullet)$ is locally constant with respect to the pull-back stratification of $D$. Since the map $(\varpi,\wt f)$ is a topological fibration when restricted above each stratum of $D$ with fibre homeomorphic to a finite number of copies of $(S^1)^{\ell-1}$ when the stratum has codimension $\ell$, it follows that the cohomology sheaves of $\psi_f^\varphi(\cF,\cF_\bbullet)$ are locally constant on $D_I\times\partial\wt S$ for each stratum $D_I$ of $D$. According to Lemma \ref{lem:constauto}, they can be regarded as $\CC$-constructible sheaves on $D$ (constructible with respect to the natural stratification) equipped with an automorphism (the monodromy around $f=0$). We will denote by $(\psi_f^\varphi(\cF,\cF_\bbullet),T)$ the corresponding object of $D^\rb_c(D)$ equipped with its automorphism $T$ (we implicitly extend the equivalence of Lemma \ref{lem:constauto} to the derived category).

More precisely, let us denote by $D^\rb_{D\text{-c}}(D\times\partial\wt S)$ the full subcategory of $D^\rb(D\times\partial\wt S)$ whose objects are constructible with respect to the natural stratification $(D_I\times\partial\wt S)_I$. For each $\lambda\in\CC^*$, we denote by $L_{\lambda^{-1},\infty}$ the local system on $\partial\wt S$ whose fibre is the polynomial ring $\CC[x]$ and the monodromy is the automorphism $\lambda^{-1}\cdot\rU_\infty$, where $\rU_\infty=\exp(2\pi i\rN_\infty)$ and $\rN_\infty:\CC[x]\to\CC[x]$ is defined by $\rN_\infty(x^k)=x^{k-1}$ if $k\geq1$ and $\rN_\infty(1)=0$. We denote similarly the pull-back of $L_{\lambda^{-1},\infty}$ to $D\times\partial\wt S$, and by $p:D\times\partial\wt S\to D$ the projection. We define the functor $\psi_\lambda$ from $D^\rb_{D\text{-c}}(D\times\partial\wt S)$ to $D^\rb(D)$ by
\[\index{$FZA$@$\psi_{f,\lambda}$}
\psi_\lambda(\cG)=\bR p_*(L_{\lambda^{-1},\infty}\otimes_\CC\cG).
\]
This functor takes values in the derived category of bounded constructible complexes on $D$ (constructible with respect to the natural stratification). Moreover, $\psi_\lambda(\cG)$ is equipped functorially with an automorphism $T=\lambda\cdot(\bR p_*(\rU_\infty\otimes\id))$.

\begin{lemme}
The functor $\bigoplus_{\lambda\in\CC^*}\psi_\lambda$ induces an equivalence between $D^\rb_{D\text{-c}}(D\times\partial\wt S)$ and the category $(D^\rb_c(D),T)$ whose objects are pairs of an object of $D^\rb_c(D)$ (constructibility with respect to the natural stratification is understood) and an automorphism $T$ of this object.\qed
\end{lemme}

This lemma is a natural extension to derived categories of Lemma \ref{lem:constauto}, if one uses the finite determination functor of \cite[Lemme 1.5]{Brylinski86}. In particular, for a given object~$\cG$ of $D^\rb_{D\text{-c}}(D\times\partial\wt S)$, all $\psi_\lambda\cG$ but a finite number are isomorphic to zero locally on $D$. In the following, we will have to consider $(\psi_{f,\lambda}^\varphi(\cF,\cF_\bbullet),T)$ for each $\lambda\in\CC^*$.

\subsection{Nearby cycles along a function (dimension two)}\label{subsec:nearbytwo}
We keep the setting of \S\ref{subsec:nearbygood}. Our aim is to define the nearby cycles functors $\wt\psi_f^\varphi$ and $\psi_f^\varphi$ for \emph{good} Stokes-filtered local systems $(\cF,\cF_\bbullet)$ on $(X,D)$, \emph{without assuming that $\varphi$ is good} with respect to the stratified $\ccI$-covering $\wt\Sigma$ associated to $(\cF,\cF_\bbullet)$. We will restrict to the case where~$X$ has dimension two from now on.

The following proposition will be essential to define nearby cycles when the goodness condition on $\varphi$ is not fulfilled.

\begin{proposition}\label{prop:psigood}
Let us keep the assumptions of Definition \ref{def:nearbygood} with $\dim X=2$, and let $\epsilon:(X',D')\to(X,D)$ be a proper modification, where $D'=\epsilon^{-1}(D)$ is a divisor with normal crossings and $\epsilon:X'\moins D'\to X\moins D$ is an isomorphism. Let us set $f'=f\circ\epsilon:(X',D')\to(S,0)$. We have
\bgroup\def\theequation{\ref{prop:psigood}\,\fnsymbol{toto}}\addtocounter{equation}{-2}\refstepcounter{toto}
\begin{align}\label{eq:psigood*}
\wt\psi_f^\varphi(\cF,\cF_\bbullet)&=\bR\wt\epsilon_*\wt\psi_{f'}^\varphi\epsilon^+(\cF,\cF_\bbullet),\\
\psi_f^\varphi(\cF,\cF_\bbullet)&=\bR(\epsilon,\id_{\partial\wt S})_*\psi_{f'}^\varphi\epsilon^+(\cF,\cF_\bbullet).\label{eq:psigood**}\refstepcounter{toto}
\end{align}
\egroup
\end{proposition}

\begin{proof}
We will consider the following diagram:
\[
\xymatrix{
&\wt f^{\prime-1}\ccI_{\wt S}\ar[rrr]^-{\wt\epsilon}\ar[dr]^(.6){\wt\epsilon^{-1}q_f}\ar[dl]_(.6){q_{f'}}&&&\wt f^{-1}\ccI_{\wt S}\ar[rr]^-{\wt f}\ar[d]^{q_f}&&\ccI_{\wt S}\ar[dd]\\
\ccI_{\wt X'}\ar[dr]&&\ar[ll]_-{q_\epsilon}\wt\epsilon^{-1}\ccI_{\wt X}\ar[dl]\ar[rr]&&\ccI_{\wt X}\ar[d]&&\\
&\wt X'\ar[rrr]^-{\wt\epsilon}&&&\wt X\ar[rr]^-{\wt f}&&\wt S
}
\]
where we recall that, for a map $g$, $q_g$ is a notation for $g^*$.

On the one hand, we have by definition $(q_{f'}^{-1}\epsilon^+\cL)_{\leq\wt f^{\prime-1}\varphi}=(\epsilon^+\cL)_{\leq f^{\prime*}\varphi}$. Recall also (\cf Definition \ref{def:pullbackpreI}) that $(\epsilon^+\cL)_\leq=\Tr_{\leq q_\epsilon}(\wt\epsilon^{-1}\cL_\leq)$. Therefore,
\[
(q_{f'}^{-1}\epsilon^+\cL)_{\leq\wt f^{\prime-1}\varphi}=\sum_{\eta,\,\epsilon^*\eta\leq f^{\prime*}\varphi}\wt\epsilon^{-1}\cL_{\leq\eta}.
\]
On the other hand, $\wt\epsilon^{-1}(q_f^{-1}\cL)_{\leq\wt f^{-1}\varphi}=\sum_{\eta\leq f^*\varphi}\wt\epsilon^{-1}\cL_{\leq\eta}$. We thus have a natural inclusion $\wt\epsilon^{-1}(q_f^{-1}\cL)_{\leq\wt f^{-1}\varphi}\subset(q_{f'}^{-1}\epsilon^+\cL)_{\leq\wt f^{\prime-1}\varphi}$ (as subsheaves of the pull-back of the local system $\cL$). We will show that this is an equality by checking this property at each $y'\in\partial\wt X'$. Let us set $y=\wt\epsilon(y')$ and $\theta=\wt f(y)$. A computation similar to \eqref{eq:grf*} gives
\begin{equation}\label{eq:comparaisonpsitildevarphi}
\begin{aligned}
(q_{f'}^{-1}\epsilon^+\cL)_{\leq\wt f^{\prime-1}\varphi,y'}&=\bigoplus_{\substack{\eta\in\wt\Sigma_y\\ \epsilon^*\eta\leqyp f^{\prime*}\varphi}}\gr_\eta\cL_y,\\
\big(\wt\epsilon^{-1}(q_f^{-1}\cL)_{\leq\wt f^{-1}\varphi}\big)_{y'}&=(q_f^{-1}\cL)_{\leq\wt f^{-1}\varphi,y}=\bigoplus_{\substack{\eta\in\wt\Sigma_y\\ \eta\leqy f^*\varphi}}\gr_\eta\cL_y,
\end{aligned}
\end{equation}
so the desired equality is a consequence of Lemma \ref{lem:negimpliqueneg} applied to $\eta-f^*\varphi$ (due to the assumption of goodness, $\eta-f^*\varphi$ is purely monomial) and the map~$\epsilon$.

We also have an inclusion $\wt\epsilon^{-1}(q_f^{-1}\cL)_{<\wt f^{-1}\varphi}\subset(q_{f'}^{-1}\epsilon^+\cL)_{<\wt f^{\prime-1}\varphi}$, by using the same argument as for $\leq$, but we do not claim that it is an equality (the argument for $\leq$ used the goodness property of $\varphi$ with respect to $\wt\Sigma$, a property that one cannot use for $\psi\letheta\varphi$). In any case, we conclude that there is a surjective morphism
\[
\lambda:\wt\epsilon^{-1}\wt\psi_f^\varphi(\cF,\cF_\bbullet)\to\wt\psi_{f'}^\varphi\epsilon^+(\cF,\cF_\bbullet)
\]
and we will compute the kernel in the neighbourhood of $\wt\epsilon^{-1}(y)$ for each $y\in\partial\wt X$. According to Lemma \ref{lem:f*le} and the computation \eqref{eq:comparaisonpsitildevarphi}, the kernel at $y'\in\wt\epsilon^{-1}(y)$ is identified with the sum of $\gr_\eta\cL_y$ for those $\eta\in\wt\Sigma_y$ such that $\epsilon^*(\eta-f^*\varphi)$ has poles along all the components of $D'$ going through $x'=\varpi'(y')$, but $\eta-f^*\varphi$ has no poles along some component of $D$ going through $x=\varpi(y)$. We therefore only need to consider the local case where $D$ has two components $D_1,D_2$ and $\eta-f^*\varphi$ has poles along $D_2$ only. Since $\eta-f^*\varphi$ is purely monomial, $\epsilon^*(\eta-f^*\varphi)$ has no poles exactly along the strict transform $D'_1$ of $D_1$. Therefore, the $\eta$-component of the kernel of $\lambda_{|\wt\epsilon^{-1}(y)}$ is equal to the pull-back of $\gr_\eta\cL_y$ on $\wt\epsilon^{-1}(y)\moins\nobreak\varpi^{\prime-1}(D'_1)$ and is zero on $\wt\epsilon^{-1}(y)\cap \varpi^{\prime-1}(D'_1)$. This holds in some small neighbourhood of $y$. We will denote by $D'_1$ and $D'_{j'\in J'}$ the components of $D'$, so that $2\in J'$.

\begin{lemme}\label{lem:imdirnulle}
Let $\cF$ denote the constant sheaf on $\partial\wt X'\moins\varpi^{\prime-1}(D'_1)$ extended by $0$ on $\varpi^{\prime-1}(D'_1)$. Then $\bR\wt\epsilon_*\cF$ is zero on $\varpi^{-1}(D_1)$.
\end{lemme}

\begin{proof}
It is equivalent to proving that $\bR\epsilon_*\CC_{\wt X'_{|D'_1}}=\CC_{\wt X_{|D_1}}$ since, according to \eqref{eq:modif}, we have $\bR\wt\epsilon_*\CC_{\partial\wt X'}=\CC_{\partial\wt X}$. By definition (\cf Lemma \ref{lem:ncdrealblup}), we have a cartesian square
\[
\xymatrix{
\wt X_{|D_1}=\wt X(D_1,D_2)_{|D_1}\ar[r]\ar[d]&\wt X(D_1)_{|D_1}=\partial\wt X(D_1)\ar[d]\\
\wt X(D_2)_{|D_1}\ar[r]&D_1
}
\]
and $\wt X(D_2)_{|D_1}$ is nothing but the real blow-up space of $D_1$ along $D_1\cap D_2$, that we will denote by $\wt D_1$ for short. We denote similarly by $\wt D'_1$ the real blow-up space of~$D'_1$ along $D'_1\cap D'_{j'\in J'}$, that we identify to $\wt X'(D'_{j'\in J'})_{|D'_1}$. Then the map $\wt\epsilon:\wt X'_{|D'_1}=\wt X'(D'_{j'\in(J'\cup\{1\})})_{|D'_1}\to\wt X_{|D_1}$ factorizes through
\[
\wt{\epsilon_{|D'_1}}\times\id:\wt D'_1\times_{D_1}\partial\wt X(D_1)\to\wt D_1\times_{D_1}\partial\wt X(D_1).
\]
Note that $\epsilon:D'_1\to D_1$ is an isomorphism, hence so is $\wt\epsilon:\wt D'_1\to\wt D_1$. We will show that the map $\wt X'_{|D'_1}\to\wt D'_1\times_{D_1}\partial\wt X(D_1)$ is an isomorphism. Assume that this is proved. Then we have an equality of maps
\begin{equation}\label{eq:eqmaps}
\begin{array}{c}
\xymatrix{
\wt X'_{|D'_1}\ar[d]_{\wt\epsilon}\ar@{=}[r]&\wt D'_1\times_{D_1}\partial\wt X(D_1)\ar[d]^{\wt{\epsilon_{|D'_1}}\times\id}\\
\wt X_{|D_1}\ar@{=}[r]&\wt D_1\times_{D_1}\partial\wt X(D_1)
}
\end{array}
\end{equation}
hence the left-hand map is an isomorphism too, so that $\bR\wt\epsilon_*\CC_{\wt X'_{|D'_1}}=\CC_{\wt X_{|D_1}}$.

To prove the assertion is a local question on $D'_1$. There are local coordinates $(x'_1,x'_2)$ on $X'$ and $(x_1,x_2)$ on $X$ such that $D_1$ is locally defined by $x_1=0$, $D'_1$ by $x'_1=0$, and $\epsilon(x'_1,x'_2)=(x_1^{\prime k}x_2^{\prime\ell},x'_2)$ with $k,\ell>0$. As in \eqref{eq:formefXD}, the map $\wt X'_{|D'_1}\to\wt X_{|D_1}$ is written
\[
(\rho'_1=0,\rho'_2,\theta'_1,\theta'_2)\mto(\rho_1=0,\rho_2=\rho'_2,\theta_1=k\theta'_1+\ell\theta'_2,\theta_2=\theta'_2)
\]
and the assertion means that $\theta'_1$ is uniquely determined from $\theta_1,\theta_2,\rho_2$, which is now clear.
\end{proof}

\subsubsection*{End of the proof of Proposition \ref{prop:psigood}}
Applying Lemma \ref{lem:imdirnulle} together with the projection formula for the proper morphism $\wt\epsilon$ implies that $\bR\epsilon_*$ of this $\eta$-component is zero. Since this holds for any $\eta\in\wt\Sigma_y$, we conclude that we have an isomorphism in the neighbourhood of $y$:
\[
\bR\epsilon_*\lambda:\bR\epsilon_*\wt\epsilon^{-1}\wt\psi_f^\varphi(\cF,\cF_\bbullet)\to\bR\epsilon_*\wt\psi_{f'}^\varphi\epsilon^+(\cF,\cF_\bbullet),
\]
and since $y$ was arbitrary, this is an isomorphism all over $\partial\wt X$. On the other hand, applying once more the projection formula and \eqref{eq:modif} we have $\bR\epsilon_*\wt\epsilon^{-1}\wt\psi_f^\varphi(\cF,\cF_\bbullet)=\wt\psi_f^\varphi(\cF,\cF_\bbullet)$. This ends the proof of \eqref{eq:psigood*}.

By using the commutative diagram
\[
\xymatrix@C=1.2cm{
\partial\wt X'\ar[d]_{(\varpi',\wt f')}\ar[r]^-{\wt\epsilon}&\partial\wt X\ar[d]^{(\varpi,\wt f)}\\
D'\times\partial\wt S\ar[r]_-{\epsilon\times\id}&D\times\partial\wt S
}
\]
one obtains \eqref{eq:psigood**}.
\end{proof}

Going back to nearby cycles, we note that Definition \ref{def:nearbygood} cannot be used in general, since Proposition \ref{prop:psigood} does not hold in general without the goodness property of $\varphi$, and we would expect that the property proved in this proposition to be satisfied by nearby cycles. The idea is then to define nearby cycles $\wt\psi_f^\varphi$ by choosing a modification $\epsilon:(X',D')\to(X,D)$ so that $\varphi$ becomes good with respect to $\epsilon^+(\cF,\cF_\bbullet)$ and take the formulas of Proposition \ref{prop:psigood} as a definition. This is similar to the notion of \emph{good cell} introduced in \cite{Mochizuki10}. The proposition itself is useful to prove that this definition does not depend on the choice of $\epsilon$, provided that the goodness property is fulfilled.

\begin{proposition}\label{prop:goodafterblup}
Let us fix $x\in D$ and let $\eta$ be a local section of $\ccI_{\wt X}$ in some neighbourhood of $y\in\varpi^{-1}(x)$. Then there exists a finite sequence of point blowing-ups $\epsilon:(X',D')\to(X,D)$ with centers projecting to~$x$ such that $\epsilon^*\eta$ (\cf Proposition \ref{prop:f*inj}) is good on $\epsilon^{-1}(U)$, where $U$ is some open neighbourhood of $x$ in $X$.
\end{proposition}

\begin{proof}
If $\eta$ is not ramified, that is, $\eta\in\cO_{X,x}(*D)/\cO_{X,x}$, goodness means pure monomiality, and the assertion is that of Lemma \ref{lem:averifier}, according to the well-known property that any proper modification of a complex surface is dominated by a sequence of point blowing-ups.

In general, by definition of $\ccI$, there exists a local ramification $\rho_{\bmd}:U_{\bmd}\to U$ such that $\rho_{\bmd}^*\eta$ is not ramified, \ie belongs to $\cO_{U_{\bmd},0}(*D)/\cO_{U_{\bmd},0}$. For any automorphism~$\sigma$ of $\rho_{\bmd}$, $\sigma^*\rho_{\bmd}^*\eta$ is also not ramified and has the same polar locus (with the same multiplicities) as $\rho_{\bmd}^*\eta$. The product of all $\sigma^*\rho_{\bmd}^*\eta$ when $\sigma$ varies in the Galois group of $\rho_{\bmd}$ can be written $\rho_{\bmd}^*\xi$, where $\xi$ belongs to $\cO_{U,0}(*D)/\cO_{U,0}$.

Note now that the product of elements of $\cO_{U_{\bmd},0}(*D)/\cO_{U_{\bmd},0}$ having the same polar divisor is purely monomial if and only if each term is purely monomial.

We can apply the first part of the proof to $\xi$ and get a sequence $\epsilon$ of \emphb{point blowing-ups} such that $\epsilon^*\xi$ is purely monomial. Let $x'\in\epsilon^{-1}(x)$ and let $D'=\epsilon^{-1}(D)$. One can choose a local ramification $\rho_{\bmd'}$ in the neighbourhood of $x'$ which dominates $U_{\bmd}$ by a generically finite map $\epsilon'$. We can apply the first property above to $\rho_{\bmd'}^*\epsilon^*\xi=\prod_\sigma\epsilon^{\prime*}\sigma^*\rho_{\bmd}^*\eta$ to conclude that $\epsilon^*\eta$ is purely monomial after the local ramification~$\rho_{\bmd'}$.
\end{proof}

\begin{corollaire}
Let $\wt\Sigma\subset\ccIet$ be a good stratified $\ccI$-covering. Then, for any local section $\varphi$ of $\ccI_{\partial\wt S}$, there exists, over any compact set $K$ of $D$, a finite sequence of point blowing-ups $\epsilon:(X',D')\to(X,D)$ such that $\varphi$ is good with respect to $\epsilon^*\wt\Sigma$ in some neighbourhood of $\epsilon^{-1}(K)$.
\end{corollaire}

\begin{definitio}\label{def:psifvarphi}
Let $f:(X,D)\to(S,0)$ be a proper holomorphic map to a disc~$S$, where $D=f^{-1}(0)$ is a divisor with normal crossings and smooth components. Let $(\cF,\cF_\bbullet)$ be a good Stokes-filtered local system on $(X,D)$ with associated stratified $\ccI$-covering $\wt\Sigma$. For any local section $\varphi$ of $\ccI_{\partial\wt S}$, we \index{nearby cycle!of Stokes filtered local systems}\emph{define}
\bgroup\numstareq
\begin{equation}\label{eq:psifvarphi*}
\wt\psi_f^\varphi(\cF,\cF_\bbullet)\defin\bR\wt\epsilon_*\wt\psi_{f'}^\varphi\epsilon^+(\cF,\cF_\bbullet),
\end{equation}
\egroup
where $f'=f\circ\epsilon$ and $\epsilon:(X',D')\to(X,D)$ is any finite sequence of point blowing-ups such that $\varphi$ is good with respect to $\epsilon^*\wt\Sigma$, and we set, as above,
\[
\psi_f^\varphi(\cF,\cF_\bbullet)=\bR(\varpi,\wt f)_*\wt\psi_f^\varphi(\cF,\cF_\bbullet)
\]
\end{definitio}

\begin{remarque}
That the choice of $\epsilon$ is irrelevant follows from Proposition \ref{prop:psigood}. We could also avoid the properness assumption on $f$, by working on an exhaustive sequence of compact subsets of $D$. Lastly, notice the formula
\bgroup\numstareq
\begin{equation}\label{eq:psifvarphi**}
\psi_f^\varphi(\cF,\cF_\bbullet)=\bR(\epsilon,\id_{\partial\wt S})_*\psi_{f'}^\varphi\epsilon^+(\cF,\cF_\bbullet).
\end{equation}
\egroup
\end{remarque}

\subsection{Comparison}\label{subsec:comparison1}
We now go back to the setting of \S\ref{subsec:nearbygood}. Let $\cM$ be a \emph{good} meromorphic bundle with flat connection on $X$ with poles along a divisor with normal crossings $D$, with associated stratified $\ccI$-covering denoted by $\wt\Sigma$. Let $(\cF,\cF_\bbullet)$ be the associated Stokes-filtered local system on $\wt X(D_{j\in J})$, and let $(\cL,\cL_\bbullet)$ denote its restriction to $\partial\wt X(D_{j\in J})$. Let us also assume that $\wt\Sigma\cup\{0\}$ is good (\ie each local section of $\wt\Sigma$ is purely monomial). Then the complex $\DR^\modD\cM$ has cohomology in degree $0$ at most, and $\cL_{\leq0}\defin\cH^0\DR^\modD\cM$ is a subsheaf of $\cL=\cH^0\wt\DR\cM$ (\cf Corollary \ref{cor:HTM}).

Let us now consider the rapid decay de~Rham complex of $\cM$. This is the complex defined similarly to $\DR^\modD\cM$ (\cf\S\ref{subsec:modgrowth}) by replacing the coefficient sheaf $\cA^\modD_{\wt X}$ with the sheaf $\cA^\rdD_{\wt X}$ as defined in Remark \ref{rem:rapiddecay}. This complex has already been \hbox{considered} in dimension one for the full Riemann-Hilbert correspondence (\cf Theorem~\ref{th:H1nul}), but not in dimension bigger than one because the grading has not been analyzed. Our purpose (Corollary \ref{cor:comparaison} and Remark \ref{rem:comparaison}\eqref{rem:comparaison1}) is to compare this grading process to the nearby cycle functor $\wt\psi_f^0$ applied to the Stokes-filtered local system attached to $\cM$, when $f:X\to\CC$ has zero set equal to~$D$. We will first prove the natural analogue of Theorem \ref{th:H1nul}.

\begin{proposition}\label{prop:H1nuldimplus}
For any germ $\cM$ along $D$ of good meromorphic connection such that $\wt\Sigma\cup\{0\}$ is also good, the complexes $\DR^\modD\cM$, $\DR^\rdD\cM$ and $\DR^\grD\cM$ have cohomology in degree~$0$ at most. The natural morphisms $\DR^\rdD\cM\to\DR^\modD\cM\to\wt\DR\cM$ induce inclusions $\cH^0\DR^\rdD\cM\hto\cH^0\DR^\modD\cM\hto\cH^0\wt\DR\cM$, and $\cH^0\DR^\grD\cM$ is equal to $\cH^0\DR^\modD\cM/\cH^0\DR^\rdD\cM$.
\end{proposition}

\begin{proof}
The question is local. Assume first that $\cM$ has a good decomposition. That $\cH^k(\DR^\modD\cM)=0$ for $k\neq0$ is Corollary \ref{cor:HTM}. For $\DR^\rdD\cM$, the similar assertion is proved with the same arguments (\cf \cite[\S7]{Bibi93} when $\dim X=2$). The remaining part of the proposition follows easily.

In order to treat the ramified case, one uses the same argument as in the proof of Lemma \ref{lem:imdirLnegsmooth}.
\end{proof}

\begin{proposition}
With the same assumptions as in Proposition \ref{prop:H1nuldimplus}, the subsheaf $(q_f^{-1}\cL_\leq)_{<0}$ of $\cL$ is identified with $\cH^0\DR^\rdD\cM$, so that there is a natural isomorphism
\[
\cH^0\DR^\grD\cM\isom\wt\psi_f^0(\cF,\cF_\bbullet).
\]
\end{proposition}

\begin{proof}
Since both sheaves $(q_f^{-1}\cL_\leq)_{<0}$ and $\cH^0\DR^\rdD\cM$ are subsheaves of $\cL_{\leq0}=\cH^0\DR^\modD\cM$, the comparison can be done locally on $\partial\wt X$. Working with $\cA_{\wt X}\otimes\cM$ and using Theorem \ref{th:HTM}, we then reduce to proving the assertion for $\cM=\cE^\eta$, where $\eta$ is a purely monomial local section of $\ccI$. It is not difficult to show that, in such a case,
\[
\cH^0\DR^\rdD\cE^\eta=\begin{cases}
0&\text{if $\eta$ does not have a pole}\\[-5pt]
&\text{along each components of $D$},\\
\cH^0\DR^\modD\cE^\eta&\text{otherwise}.
\end{cases}
\]
According to Lemma \ref{lem:f*le}, $\cH^0\DR^\rdD\cE^\eta$ coincides with the corresponding $(q_f^{-1}\cL_\leq)_{<0}$.
\end{proof}

We now compare the previous construction to that of moderate nearby cycles as recalled in \S\ref{subsec:moderatenearby}.

\begin{corollaire}\label{cor:comparaison}
With the same assumptions as in Proposition \ref{prop:H1nuldimplus}, we have for each $\lambda\in\CC^*$ and on each compact set $K\subset D$ a functorial isomorphism of objects of $D_c^b(D)$ equipped with an automorphism
\[
(\psi_{f,\lambda}^0(\cF,\cF_\bbullet),T)_{|K}\isom(\DR\psi_{f,\lambda}^\rmod\cM,T)_{|K}.
\]
\end{corollaire}

\begin{proof}
Since $\cM$ is $\cO_X(*D)$-locally free, we can apply \cite[Cor\ptbl II.1.1.19, p\ptbl45]{Bibi97} to compute $\bR\varpi_*\DR^\rdD\cM$, and conclude that $\bR\varpi_*\DR^\grD\cM$ is isomorphic to $\DR(\cO_{\wh D}\otimes\cM)$. Let us now replace $\cM$ with $\cM_{\lambda,k}$.

We denote by $L_{\lambda^{-1},k}$ the local system
\[
\ker\big[\partial_t:\cA^{\rmod0}_{\partial\wt S}\otimes\cN_{\lambda,k}\to\cA^{\rmod0}_{\partial\wt S}\otimes\cN_{\lambda,k}\big]
\]
on $\partial\wt S$. This is a rank $k+1$ local system with monodromy $\lambda^{-1}\id+\rN_{k+1}$, where $\rN_{k+1}$ is a Jordan block of size $k+1$.

Similarly, we have
\[
\wt f^{-1}L_{\lambda^{-1},k}=\cH^0\DR^\modD(f^+\cN_{\lambda,k}),
\]
and the $\cH^j$ vanish for $j>0$, according to Proposition \ref{prop:H1nuldimplus}. Moreover, it is immediate to check that $\cH^0\DR^\rdD(f^+\cN_{\lambda,k})=0$, and Proposition \ref{prop:H1nuldimplus} also implies
\[
\wt f^{-1}L_{\lambda^{-1},k}=\cH^0\DR^\grD(f^+\cN_{\lambda,k}).
\]
We conclude that there is a natural morphism of complexes
\[
\wt f^{-1}L_{\lambda^{-1},k}\otimes_\CC\DR^\grD\cM\to\DR^\grD\cM_{\lambda,k},
\]
and one checks by a local computation on $\partial\wt X$, by using Theorem \ref{th:HTM}, that this morphism is a quasi-isomorphism. Taking its inductive limit we get a quasi-isomorphism
\[
\wt f^{-1}L_{\lambda^{-1},\infty}\otimes_\CC\DR^\grD\cM\to\varinjlim_k\DR^\grD\cM_{\lambda,k},
\]
We therefore get an isomorphism
\[
\bR\varpi_*\big(\wt f^{-1}L_{\lambda^{-1},\infty}\otimes_\CC\DR^\grD\cM\big)\to\varinjlim_k\DR(\cO_{\wh D}\otimes\cM_{\lambda,k}).
\]
According to the projection formula, the left-hand term is $\psi_{f,\lambda}^0(\cF,\cF_\bbullet)$ and the right-hand term, restricted to the compact set $K$, is identified with $\DR\psi_{f,\lambda}^\rmod\cM$ (\cf Corollary \ref{cor:drpsimod}). The comparison of monodromies is straightforward.
\end{proof}

\skpt
\begin{remarques}\label{rem:comparaison}\ligne
\begin{enumerate}
\item\label{rem:comparaison1}
If $\dim X=2$, we can define $(\psi_f^0(\cF,\cF_\bbullet),T)$ with the only assumption that $(\cF,\cF_\bbullet)$ is good by the procedure of \S\ref{subsec:nearbytwo}. On the other hand, $\psi_f^\rmod$ commutes with direct images of $\cD$-modules (\cf \eg \cite{M-S86b}), and $\DR$ commutes with direct images by~$\epsilon$ (since $X'$ and $X$ have the same dimension, the shifts in the de~Rham complexes cancel). Therefore, in such a case, we get a comparison isomorphism as in Corollary \ref{cor:comparaison}.
\item\label{rem:comparaison2}
One can extend in a straightforward way the comparison of Corollary \ref{cor:comparaison} to the various $\psi_f^\varphi(\cF,\cF_\bbullet)$ provided the pair $(\wt\Sigma,\varphi)$ is good. If $\varphi$ is not ramified, the right-hand side is replaced with $(\DR\psi_f^\rmod(\cM\otimes\cE^{-f^*\varphi}),T)$. Similarly, in dimension two, one can relax the goodness assumption on $\varphi$ with respect to $\wt\Sigma$.
\end{enumerate}
\end{remarques}

\backmatter
\newcommand{\SortNoop}[1]{}\def\cprime{$'$}
\providecommand{\bysame}{\leavevmode ---\ }
\providecommand{\og}{``}
\providecommand{\fg}{''}
\providecommand{\smfandname}{\&}
\providecommand{\smfedsname}{\'eds.}
\providecommand{\smfedname}{\'ed.}
\providecommand{\smfmastersthesisname}{M\'emoire}
\providecommand{\smfphdthesisname}{Th\`ese}

\renewcommand{\indexname}{Index of notation}
\chapterspace{-2}
\printindex
\end{document}